\documentclass[11pt]{amsart}
 \usepackage{amsmath,amssymb,amsbsy,amsgen, amsthm}
\usepackage{graphicx}  
  \usepackage{xcolor}
   \usepackage[colorlinks, pagebackref, linkcolor=red, citecolor=blue]{hyperref}
   \usepackage{array,esint}
     \usepackage{bigints}
  
\usepackage{fancybox}
 \usepackage{marvosym}
           
\def\rel#1_#2^#3{\mathrel{\mathop{\kern 0pt#1}\limits_{#2}^{#3}}}
  \def\arsinh{\operatorname{arsinh}}
 \def\artanh{\operatorname{artanh}}

   \newtheorem{theorem}{Theorem}
 \newtheorem{lemma}[theorem]{Lemma}

   \newtheorem{criterium}[theorem]{Zero-set criterium}

 \newtheorem{proposition}[theorem]{Proposition}
     \newtheorem{exemple}[theorem]{Example}
  \theoremstyle{definition}

\def\obun#1_#2^#3{\mathrel{\mathop{\kern 0pt#1}\limits_{#2}^{#3}}}

\def\oh{{\scriptstyle{\mathcal O}}}
     \def\Oh{\mathcal O}

\def\co#1{\textcolor{red}{#1}}
\def\gr#1{\textcolor{green}{#1}}
\def\bl#1{\textcolor{blue}{#1}}

\def\bs {\mathbf}
\newcommand{\zit}[1]{(\ref{#1})}
\def\T{ \mathbb T}
\def\R{ \mathbb R}

\def\Q{ \mathbb Q}

\def\D{{ \mathbb D}}
\def\C{{ \mathbb C}}
\def\Z{{ \mathbb Z}}
\def\N{{ \mathbb N}}
\def\e{\varepsilon}

\def\sp {\quad}

\def \n{\raisebox{-2pt}{\rule{1.3pt}{10pt}}}
\def\dis{\displaystyle}
\def\union{\cup}
\def\Union{\bigcup}
\def\inter{\cap}

\def\ov{\overline}

\def\ss{\subseteq}
\def\emp{\emptyset}

\def\buildrel#1_#2^#3{\mathrel{\mathop{\kern 0pt#1}\limits_{#2}^{#3}}}

\def\vs{\vskip}
\def\ssi{\Longleftrightarrow}
\def\imp{\Longrightarrow}

\begin{document}

    \begin{center}

  {\Large {\it Real and Complex Analysis}} \bigskip

{\Large{Solutions to Problems in}}

\vspace{1cm}

{\bf Amer. Math. Monthly\\

 Math. Magazine\\
 
 College Math. J.\\
 
  Elemente der Math\\
  
  Crux Math.\\
  
  EMS Newsletter\\
  
  Math. Gazette
  }
  
\end{center}
  \vspace{3cm}

\centerline{Edited by} \vspace{3cm}

\centerline{ \Large Raymond Mortini}

\vspace{1cm}
 
 \centerline {\small\the\day.\the \month.\the\year} \vspace{0.5cm}

\newpage

In this arxiv-post  I present my solutions (published or not) to Problems that appeared in Amer. Math. Monthly,
 Math. Magazine, Elemente der Mathematik and CRUX, that were mostly done in collaboration with Rudolf Rupp. 
 Some of them (including a few own  proposals which were published) were also done in  cooperation with Rainer Br\"uck, Bikash Chakraborty, Pamela Gorkin, Gerd Herzog, J\'er\^ome No\"el,  Peter Pflug,  Amol Sasane and Robrto Tauraso.
 
 A few of these contributions to ``Recreational Mathematics"  actually  were the base for  interesting  generalizations  that led to some of my  publications (partially co-authored) in research journals  
 (\cite{mpr,  mps, mr2024,  mr24,  mo23,   mr21,   mr20, mo2002}). 
 
 The content will surely be attractive to all undergraduate/graduate students in Mathematics who want to solve   challenging problems in Analysis by calculating explicit values of funny looking integrals, sums and products, by deriving astonishing inequalities  and by solving  functional equations.    It is also a valuable source for teachers in mathematics in preparing exercise sheets for their students. Moreover, I think that it is worth to see in most cases quite different solutions than  those already published in the above listed journals.
 
  My main reason to post this collection of (mainly unpublished solutions),  is to keep also for future generations  an archive for  historians in Mathematics, interested in the work of one of the very few mathematicians with Luxembourgish Nationality.  Without this digital archiving, these contributions to education and science  would for ever disappear in  a few years.

\vspace{1cm}

Some technical remarks: The first items for each journals are still hidden, as the submission deadline has not yet occurred. Regular updates  are planned. Own proposals are presented with  a yellow background.

We obtained permission from the MAA, the EMS, Math. Gazette and CRUX to post a scan of the original statements and to reproduce our published (and non-published) solutions.

\vspace{4cm}
\small\the\day.\the \month.\the\year  \hfill Raymond Mortini, Pontpierre (Luxembourg).

\phantom{\small\the\day.\the \month.\the\year}  \hfill  Prof. emeritus, Univ. Lorraine, Metz,  France

\phantom{\small\the\day.\the \month.\the\year}  \hfill Visiting researcher Univ. Luxembourg, Esch-sur-Alzette


\newpage

\begin{figure}[h!]
 \scalebox{0.49} 
  {\includegraphics{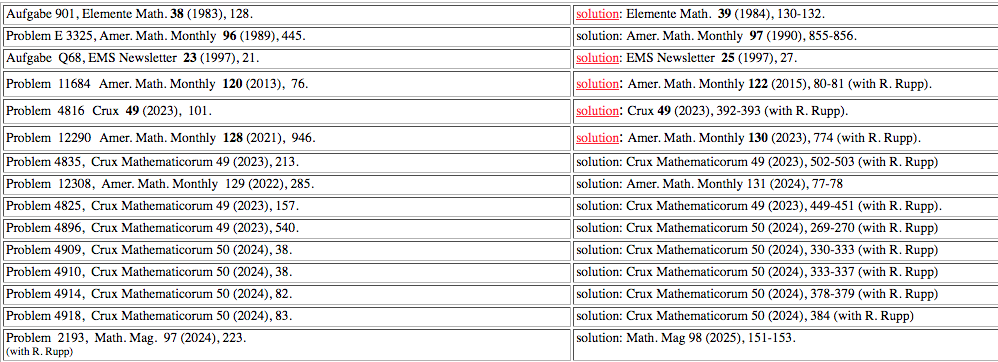}} 
   \end{figure}
   
   {\bf published proposals}
   
   \begin{figure}[h!]
 \scalebox{0.45} 
  {\includegraphics{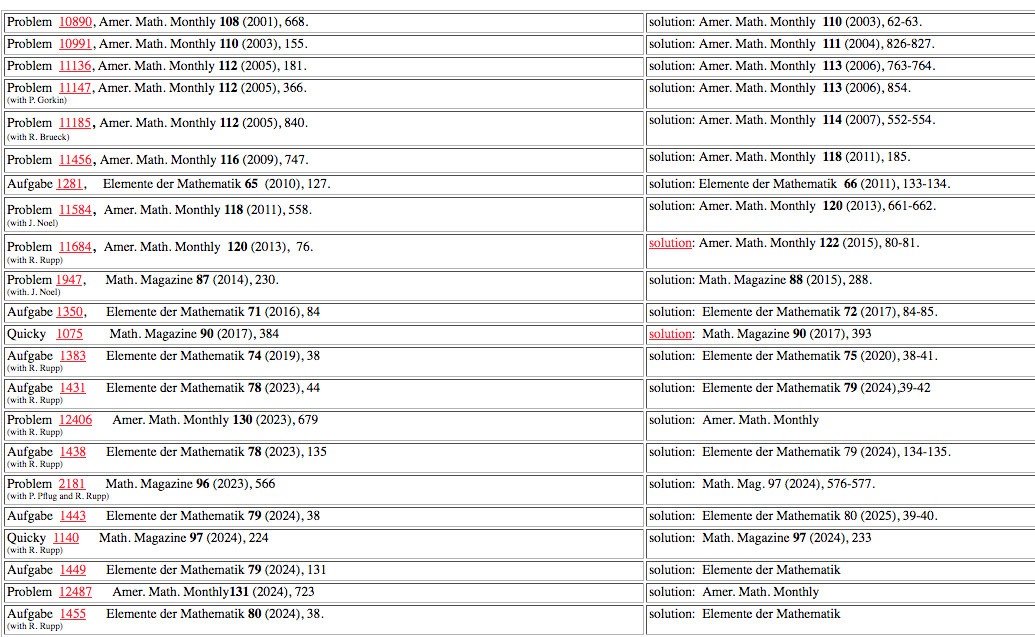}} 
   \end{figure}

   \vspace{1cm}  See  \cite{subsol} for the whole bibliographic  list.

\newpage

.\vspace{-3cm}
\gr{\huge\section{American Math. Monthly}}
\bigskip

\centerline{\co{\copyright Mathematical Association of America, 2025.  }}


\begin{figure}[h!]
 \scalebox{0.45} 
  {\includegraphics{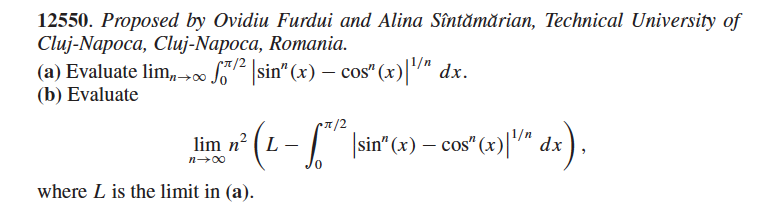}} 
  
 \end{figure}
 

\medskip

\centerline{{\bf Solution to problem 12550  in Amer. Math. Monthly 132 (2025), p. ?}}\medskip 

\centerline{Raymond Mortini  and  Rudolf Rupp }

\bigskip

\centerline{- - - - - - - - - - - - - - - - - - - - - - - - - - - - - - - - - - - - - - - - - - - - - - - - - - - - - -}
  
  \medskip



\newpage

\bigskip

\begin{figure}[h!]
 \scalebox{0.45} 
  {\includegraphics{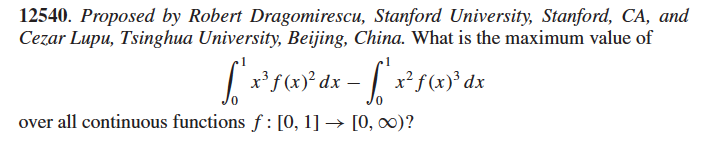}} 
  
 \end{figure}

\medskip

\centerline{{\bf Solution to problem 12540  in Amer. Math. Monthly 132 (2025), p. 592}}\medskip 

\centerline{Raymond Mortini  and  Rudolf Rupp }

\bigskip

\centerline{- - - - - - - - - - - - - - - - - - - - - - - - - - - - - - - - - - - - - - - - - - - - - - - - - - - - - -}
  
  \medskip


We show that for every continuous function $f:[0,1]\to[0,\infty[$, 
$$\ovalbox{$I(f):=\dis \int_0^1 x^3f(x)^2 dx-\int_0^1 x^2 f(x)^3 dx\leq \frac{2}{81}.$}$$
and that for  $f(x)=(2/3)x$, we have $I(f)= 2/81$.\\

First we note that $I(f)=\dis\int_0^1(x^2 f(x)^2(x-f(x))dx.$ Consider for $0\leq x\leq 1$  the auxiliary functions
$$h(y)=x^2y^2(x-y)= -x^2y^3+x^3 y^2.$$
This is  a cubic with double zero $0$ and zero $y=x$. Now 
$h'(y)=x^2y(2x-3y)=0$ if $y=0$ or $y=(2/3)x$. Hence
$$\max_{y\geq 0} h(y) =h\left(\frac{2}{3}x\right)=\frac{4}{27}x^5.$$
Consequently, if we take $y:=f(x)\geq 0$, then
$$I(f)=\int_0^1 h(f(x))dx\leq \int_0^1 \frac{4}{27}x^5 dx= \frac{2}{81}.$$
Now, if $f(x)=(2/3)x$, then, by direct calculation, 
$I(f) =2/81$.

\vspace{1cm}

{\bf Remark.}  We Note that  $\inf \{I(f): f\in C[0,1], f\geq 0\}=-\infty$. In fact, let $f(x)\equiv M$. Then
$$I(f)=\frac{M^2}{4}-\frac{M^3}{3}\to -\infty\quad{\rm as}\quad M\to\infty.$$

\newpage

\begin{figure}[h!]
 \scalebox{0.45} 
  {\includegraphics{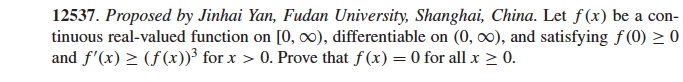}} 
  
 \end{figure}
 

\medskip

\centerline{{\bf Solution to problem 12537  in Amer. Math. Monthly 132 (2025), p. ?}}\medskip 

\centerline{Gerd Herzog, Raymond Mortini  and  Rudolf Rupp }

\bigskip

\centerline{- - - - - - - - - - - - - - - - - - - - - - - - - - - - - - - - - - - - - - - - - - - - - - - - - - - - - -}
  
  \medskip


We first show the hypotheses imply that that $f\geq 0$ on $[0,\infty[$.
If not, let $\xi>0$ be such that $f(\xi)<0$ and let $\eta\geq 0$ be the first zero \footnote{ At this point we used the hypothesis that $f(0)\geq 0$.} of $f$ when moving to the left (note that 
$\eta=0$ may be possible). In particular we have $f<0$ on $I:=]\eta, \xi]$. Now, on $I$, 
$$\frac{f'}{f^3}-1=\frac{f'-f^3}{f^3}\leq 0.$$
For $x\in I$, let
 $$h(x):=-\frac{1}{2} \cdot\frac{1}{f^{2}(x)}-x.$$
 Then $\dis h'=\frac{f'}{f^3}-1\leq 0$, and so $h$ decreases on $I$. But $\lim_{x\to\eta^+}h(x)=-\infty$. This is not possible, as $h(x)>h(\xi)$ for $x\in I$. We conclude that $f\geq 0$ on $[0,\infty[$.\\ 
 
For the rest, we present two solutions.\\  

{\it First solution}

Suppose that for some $x_0$ we have  $f>0$  on  $I=[x_0,b[$, where $b$ is maximal. Then either $f(b)=0$ or 
$b=\infty$ \footnote{ Actually, only the second case occurs: in fact, for $x\in I$,
$f'(x)\geq f(x)^3>0$; hence $f$ is increasing on $I$ and so $f(b)=\lim_{x\to b^-} f(x) \geq f(x_0)>0$. We do not need this observation, though.}.
Consider on $I$ the following function: 
$$h(x):=-\frac{1}{2}f^{-2}(x) +\frac{1}{2} f^{-2}(x_0) -x+x_0.$$
Then $h$ is continuous, differentiable, $h(x_0)=0$, and 
$$h'= \frac{f'}{f^3}-1\geq 0.$$
Hence $h$ is increasing on $I$ and so $h\geq 0$ on $I$. In other words, 
$$f^{-2}(x)-f^{-2}(x_0)\leq -2(x-x_0).$$
Noticing that on $I$,  $f(x_0)^{-2}-2(x-x_0)\geq  f^{-2}(x)>0$, we deduce that \footnote{ Note that 
$u$ is the unique solution to $u'=u^3, u(x_0)=f(x_0)$.}.

$$f(x)\geq \frac{1}{\sqrt{f(x_0)^{-2}+2x_0-2x}}=:u(x), \quad x\in I.$$
Henceforth the maximality of $b$ implies that $b=x_0+\frac{f(x_0)^{-2}}{2}$. 
Thus $f$ can't be a non-trivial  solution to the inequality $f'\geq f^3$ on the whole interval $[0,\infty[$. Consequently, 
as we know that $f\geq 0$,  $f(x)= 0$ for every $x\in [0,\infty[$.\newpage 


 {\it Second solution}\\

Here we use the following standard result (see \cite{wwa}):
 \begin{theorem}
 Suppose that $D\ss\R$ is an interval, that $H:[a,b[\times D \to \R$ is locally Lipschitz, and that $u,v: [a,b[\to D$ are differentiable functions satisfying
 $$u'(x)-H(x, u(x))\leq v'(x)-H(x,v(x)) \quad (x\in [a,b[)$$ 
 $$u(a)\leq v(a).$$
 Then $u(x)\leq v(x)$  for $x\in [a,b[$.
 \end{theorem}
We already know that $f\geq 0$.  
In order to obtain a contradiction, suppose  that $f(x_0)>0$ for some $x_0> 0$. Now we apply the Theorem 
 to $D= \R$, $a=x_0$, $b=\infty$ and $H(x, y)=y^3$, and  
 consider on $[x_0, b[$ the differential equation $u'=u^3$ with $u(x_0)=f(x_0)$. Then
 $0\leq f'(x)-f^3(x)$  implies that on $[x_0,b[$ we have $u(x)\leq f(x)$. But the unique solution $u$ is given by
 $$u(x)=\frac{1}{\sqrt{f(x_0)^{-2}+2x_0-2x}} $$
 with a singularity at $b= x_0+\frac{f(x_0)^{-2}}{2}<\infty$. Thus $f$ can't be a solution to the inequality $f'\geq f^3$ on the whole interval $[0,\infty[$. Consequently
  $f(x)\leq 0$ for every $x\in [0,\infty[$. Putting all together, we have shown that $f\equiv 0$ on $[0,\infty[$.\\


Let us mention  that if $f$ would satisfy $f(0)<0$, then any negative constant or $-e^x$ e.g. would be solutions.


\newpage
 
 \begin{figure}[h!]
 \scalebox{0.45} 
  {\includegraphics{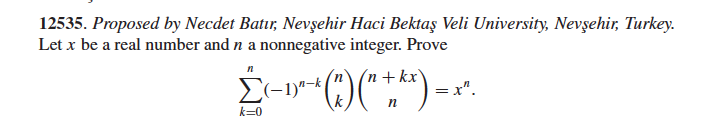}} 
  
 \end{figure}

\medskip

\centerline{{\bf Solution to problem 12535  in Amer. Math. Monthly 132 (2025), p. ?}}\medskip 

\centerline{Raymond Mortini  and  Rudolf Rupp }

\bigskip

\centerline{- - - - - - - - - - - - - - - - - - - - - - - - - - - - - - - - - - - - - - - - - - - - - - - - - - - - - -}
  
  \medskip

 This exercise is a special case of example (5.43) in \cite[p. 190]{GKP} or Example (6.30) in \cite[p. 72]{QG}.
   It is based on Euler's finite difference theorem: if
  $p(u)=\sum_{j=0}^d a_ju^j$ is a polynomial of degree $d$, then
  $$\sum_{k=0}^n (-1)^k \binom{n}{k} p(k)=\begin{cases} 0&\text{if $0\leq d<n$}\\ (-1)^n n! a_n&\text{if $d=n$}.
  \end{cases}
  $$
 (see also \cite[p. 68--69] {QG}). We have just to note that 
 $$p(u):=\binom{n+ux}{n}=\frac{1}{n!}(n+ux)(n-1+ux)\dots(1+ux) $$ 
 is a polynomial of degree $n$ and that $a_n=x^n/n!$.


\newpage
 
 \begin{figure}[h!]
 \scalebox{0.45} 
  {\includegraphics{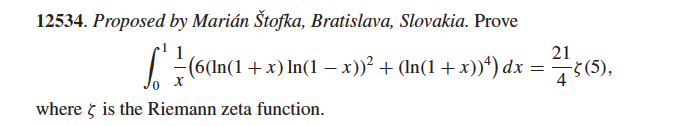}} 
  
 \end{figure}

\medskip

\centerline{{\bf Solution to problem 12534  in Amer. Math. Monthly 132 (2025), p. ?}}\medskip 

\centerline{Raymond Mortini  and  Rudolf Rupp }

\bigskip

\centerline{- - - - - - - - - - - - - - - - - - - - - - - - - - - - - - - - - - - - - - - - - - - - - - - - - - - - - -}
  
  \medskip

We will base the solution on the formula
 $$a^2b^2=\frac{1}{12} \left((a+b)^4+(a-b)^4-2a^4-2b^4\right).$$
 Hence, with $a=\log(1-x)$ and $b=\log(1+x)$, 
{\footnotesize \begin{eqnarray*}
 I&:=&\int_0^1 \frac{1}{x}\left(6 \log^2(1+x)\log^2(1-x)+\log^4(1+x)\right)dx\\
 &=&\frac{1}{2}\int_0^1 \frac{\log^4(1-x^2)}{x}dx
 +\frac{1}{2} \int_0^1 \frac{1}{x}\log^4\left(\frac{1-x}{1+x}\right)dx-\int_0^1\frac{\log^4(1-x)}{x}dx \co{-\int_0^1\frac{\log^4(1+x)}{x}dx+
 \int_0^1\frac{\log^4(1+x)}{x}dx}\\
 &\buildrel=_{x=\sqrt t}^{}& \frac{1}{2}\int_0^1 \frac{\log^4(1-t)}{\sqrt t}\;\frac{1}{2\sqrt t}dt
 +\frac{1}{2} \int_0^1 \frac{1}{x}\log^4\left(\frac{1-x}{1+x}\right)dx-\int_0^1\frac{\log^4(1-x)}{x}dx\\
 &=&-\frac{3}{4}\int_0^1\frac{\log^4(1-x)}{x}dx +\frac{1}{2} \int_0^1 \frac{1}{x}\log^4\left(\frac{1-x}{1+x}\right)dx=:I_1+I_2.
\end{eqnarray*}
}
 Now, due to $\sum\int=\int\sum$ (all terms are positive), and 4 times partial integration,
 \begin{eqnarray*}
 I_1&=&-\frac{3}{4}\int_0^1\frac{\log^4(1-x)}{x}dx\buildrel=_{}^{1-x=t}-\frac{3}{4}\int_0^1\frac{\log^4 t}{1-t}dt=
 -\frac{3}{4}\sum_{k=0}^\infty\int_0^1t^k\log^4 t\; dt\\
 &=&-\frac{3}{4} \sum_{k=0}^\infty \frac{4!}{(k+1)^5}=-18\, \zeta(5).
\end{eqnarray*}
 
 Moreover, by using the transformation $t=(1-x)/(1+x)$, resp. $x= (1-t)/(1+t)$ and $dx=-2/(1+t)^2 dt$,
 
 \begin{eqnarray*}
 I_2&=&\frac{1}{2} \int_0^1 \frac{1}{x}\log^4\left(\frac{1-x}{1+x}\right)dx=\frac{1}{2}\int_0^1\frac{1+t}{1-t}\log^4 t \;\frac{2}{(1+t)^2}dt\\
 &=&\int_0^1\frac{\log^4 t}{1-t^2} dt=\sum_{k=0}^\infty \int_0^1 t^{2k} \log^4 t \,dt=24 \sum_{k=0}^\infty \frac{1}{(2k+1)^5}.
\end{eqnarray*}

 Now 
 $$\zeta(5)=\sum_{n=1}^\infty \frac{1}{n^5}=\sum_{k=0}^\infty \frac{1}{(2k+1)^5}+\sum_{k=1}^\infty\frac{1}{(2k)^5}=
 \sum_{k=0}^\infty \frac{1}{(2k+1)^5}+\frac{\zeta(5)}{32}.
 $$
 Hence
 $$\sum_{k=0}^\infty \frac{1}{(2k+1)^5}=\frac{31}{32}\,\zeta(5).
 $$
 
 Consequently 
 $$I=I_1+I_2= -18 \zeta(5)+ 24 \frac{31}{32}\zeta(5)= \frac{21}{4}\zeta(5).$$
 \bigskip
 
 {\bf Remark}. The value  of the integral  $\dis \int_0^1 \frac{\log^2(1+x)\log^2(1-x)}{x} dx$ 
 is given in \cite[p. 5, p. 78-79]{val}. See also \cite{4785630}, unconvincing.

\newpage

\begin{figure}[h!]
 \scalebox{0.45} 
  {\includegraphics{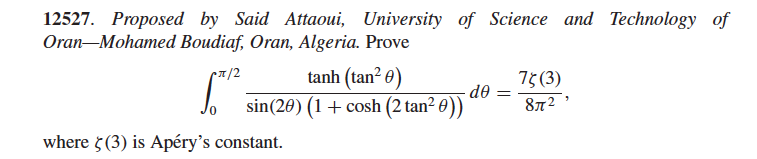}} 
  
 \end{figure}
 

\medskip

\centerline{{\bf Solution to problem 12527  in Amer. Math. Monthly 132 (2025), p. ?}}\medskip 

\centerline{Raymond Mortini  and  Rudof Rupp }

\bigskip

\centerline{- - - - - - - - - - - - - - - - - - - - - - - - - - - - - - - - - - - - - - - - - - - - - - - - - - - - - -}
  
  \medskip


  \begin{eqnarray*}
I:=\int_0^{\pi/2} \frac{\tanh(\tan^2\theta)}{\sin(2\theta)(1+\cosh(2\tan^2\theta))}\;d\theta&\buildrel=_{=2\sin x\cos x}^{\sin(2x)}&
\frac{1}{2}\;\int_0^{\pi/2} \frac{\tanh(\tan^2\theta)}{\tan\theta (1+\cosh(2\tan^2\theta))}\frac{d\theta}{\cos^2\theta}\\
\buildrel=_{}^{x=\tan \theta} \frac{1}{2}\;\int_0^\infty\frac{\tanh(x^2)}{x(1+\cosh(2x^2)} dx&=&
\frac{1}{4}\int_0^\infty \frac{\tanh u}{u(1+\cosh(2u))}du\\
=\frac{1}{8}\int_0^\infty \frac{\tanh u}{u\cosh^2 u}du&=& \frac{1}{8}\int_0^\infty \frac{\sinh u}{u\cosh^3 u}du.
 \end{eqnarray*} 
  The integrand being even, we obtain
  $$I=\frac{1}{16} \int_{-\infty}^\infty \frac{\sinh u}{u\cosh^3 u}du.$$
  Now we apply the residue theorem to the meromorphic function 
  $$f(z)=\dis \frac{1}{16}\; \frac{\frac{\sinh z}{z}}{\cosh^3 z},$$
   whose poles in the upper half plane are $z_n:=i(2n+1)\pi/2$ for $n\in \N$, the order being $3$. Note that $0$ is an artificial/removable singularity. 
  The usual softwares yield 
  $$w_n:={\rm Res}\; (f, z_n)= -\frac{i}{2}\;\frac{1}{(2n+1)^3 \pi^3}.$$
  (For a calculation, see the Addendum  below).
  Hence (using the Remark  1 below)
  $$I=2\pi i\sum_{n=0}^\infty w_n=\frac{1}{\pi^2}\sum_{n=0}^\infty \frac{1}{(2n+1)^3} =\frac{7\zeta(3)}{8\pi^2}
  \sim 0.1065695997\cdots.$$

  {\bf Remark 1} It is easy to see that the integrals $\int_{\Gamma_n} f(z) dz$ over the boundary $\Gamma_n$ of the rectangles 
  $[-n,n]\times [0, 2\pi n]$ tend to $0$ with $n\to\infty$, by using the definition of the hyperbolic functions with the exponential functions.
\\
  
 {\bf Remark 2} Using the substitution $u=\tanh x$, we obtain
 $$ \frac{7}{\pi^2}\zeta(3)=L:=\int_0^\infty \frac{\sinh u}{u\cosh^3 u}du=\int_0^\infty \frac{\tanh u}{u\cosh^2 u}du= \int_0^1 \frac{x}{\artanh x} dx \sim 0.852556797\cdots$$
 The integral $L$ also appears in the solution to Problem 12332 in {\it Amer. Math. Monthly} {\bf 131}
(2024) 270--271 and equals $ \int_0^\infty \frac{\tanh^2 x}{2x^2} dx$, as has been communicated to us by Russ Gordon.\\

{\bf Addendum} 
We calculate the Laurent series at $z_n$: since $\cosh z_n=0$ and $\sinh z_n=i(-1)^n$, we obtain
{\footnotesize \begin{eqnarray*}
16 f(z)&=& \frac{\sinh (z-z_n+z_n)}{(z-z_n+z_n)\cosh^3(z-z_n+z_n)}=\frac{\sinh (z-z_n) \cosh z_n+ \cosh (z-z_n)\sinh z_n}
{ (z-z_n+z_n)  \Big(\cosh(z-z_n) \cosh z_n + \sinh (z-z_n)\sinh z_n\Big)^3}\\
&=&-\frac{\cosh(z-z_n)}{(z-z_n+z_n)\; \sinh^3 (z-z_n)}=-\frac{1}{z-z_n+z_n}\;\frac{\coth(z-z_n)}{\sinh^2(z-z_n)}
=\frac{1}{2 (z-z_n+z_n)} \frac{d}{dz} \coth^2(z-z_n)\\
&=&\frac{1}{2 (z-z_n+z_n)} \frac{d}{dz}\Big( \frac{1}{(z-z_n)^2}+\frac{2}{3}+\frac{1}{15}(z-z_n)^2 -\cdots\Big)\\
&=& \frac{1}{2z_n}\frac{1}{1+ \frac{z-z_n}{z_n}}\Big( -\frac{2}{(z-z_n)^3} +\frac{2}{15}(z-z_n)-\cdots\Big)\\
&\buildrel=_{0<|z-z_n|<|z_n|}^{}&\frac{1}{z_n}\left( 1-\frac{z-z_n}{z_n}+ \frac{(z-z_n)^2}{z_n^2} -\frac{(z-z_n)^3}{z_n^3}+\cdots\right)
\; \left(-\frac{1}{(z-z_n)^3} +\frac{1}{15}(z-z_n)\cdots\right)\\
&=&c_{-3}\frac{1}{(z-z_n)^3}+c_{-2}\frac{1}{(z-z_n)^2}-\frac{1}{z_n^3}\frac{1}{z-z_n}+c_0+\cdots
\end{eqnarray*}
}
Hence 
$${\rm Res}\; (f, z_n)= - \frac{1}{16} \frac{1}{z_n^3}=\frac{1}{2i}\;\frac{1}{(2n+1)^3 \pi^3}.$$

\newpage

\bigskip

\begin{figure}[h!]
 \scalebox{0.45} 
  {\includegraphics{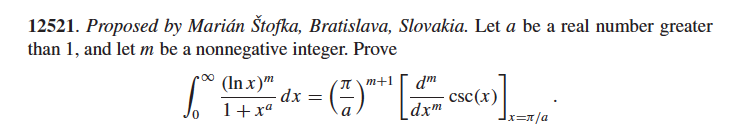}} 
  
 \end{figure}
 

\medskip

\centerline{{\bf Solution to problem 12521  in Amer. Math. Monthly 132 (2025), p. ?}}\medskip 

\centerline{Raymond Mortini  and  Rudof Rupp }

\bigskip

\centerline{- - - - - - - - - - - - - - - - - - - - - - - - - - - - - - - - - - - - - - - - - - - - - - - - - - - - - -}
  
  \medskip


 Fix $a>1$. Consider for $-1< s<a-1$ the function 
  $$F(s):=\int_0^\infty \frac{x^s}{1+x^a}\;dx.$$
    Since the  integral
  $I(s):=\int_0^\infty \frac{ x^s \log s}{1+x^a} dx$ is uniformly convergent with respect to the parameter $s$ \footnote{ This means that
  $\forall\e>0\;\exists  c>0\; \forall    s\in S\; \forall u,v\in [c,\infty[: \left|\int_u^v f_s(x) dx\right|<\e$.},
  we have $d/ds \int =\int d/ds$,  and so $I(s)=F'(s)$. The same reasoning applies to the higher order derivatives. Thus
  $$F^{(m)}(s)=\int_0^\infty \frac{x^s(\log x)^m}{1+x^a}\;dx.$$
  The value we are looking for therefore coincides with $F^{(m)}(0)$. We are going to show that
  $$F(s)=\frac{\pi}{a} \csc \left( \frac{\pi}{a} (s+1)\right),$$
  a function well defined on $]-1,a-1[$, that is a  neighborhood of $s=0$. To this end we use the Euler-Mittag-Leffler representation
  \begin{eqnarray}\label{sini}
\frac{\pi}{\sin (\pi z)}&=& \frac{1}{z}+\sum_{n=1}^\infty (-1)^n\frac{2z}{z^2-n^2}= 
 \frac{1}{z}+ \sum_{n=1}^\infty  (-1)^n\left(\frac{1}{z+n}+\frac{1}{z-n}\right)\\
 &=& \lim_{N\to\infty} \sum_{|n|\leq N} \frac{(-1)^n}{z-n} =\sum_{n=0}^\infty (-1)^n \left(\frac{1}{n+z}+\frac{1}{n+1-z}\right),
  \end{eqnarray}
where the latter identity follows through an index change in the previous finite sum representation
 \footnote{ $\dis\sum_{|n|\leq N} \frac{(-1)^n}{z-n}\buildrel=_{m=-n}^{k=n-1}\sum_{m=0}^N\frac{(-1)^m}{z+m} +
 \sum_{k=1}^{N-1}  \frac{(-1)^{k-1}}{z-(k+1)}+\frac{1}{1-z}$.}.
 Here is now the calculation:
\begin{eqnarray*}
F(s)&=&\int_0^1  \frac{x^s}{1+x^a}\;dx +\int_1^\infty  \frac{x^s}{1+x^a}\;dx\\
&\buildrel=_{x=1/u}^{}&\int_0^1  \frac{x^s}{1+x^a}\;dx + \int_0^1\frac{1}{u^{s+2}} \frac{du}{1+u^{-a}}= \int_0^1  \frac{x^s}{1+x^a}\;dx+
\int_0^1 \frac{u^{a-s-2}}{1+u^a}\;du\\
&\buildrel=_{x^a=t}^{}&\int_0^1 \frac{t^{s/a}+t^{(a-s-2)/a}}{1+t}\;\frac{1}{a}\;t^{-1+1/a}\; dt=
\frac{1}{a}\;\int_0^1\frac{t^{-1+(s+1)/a}+t^{-(s+1)/a}}{1+t}\;dt.
\end{eqnarray*}
Since (via geometric series) 
$$\int_0^1 \frac{t^b}{1+t}dt= \sum_{n=0}^\infty \frac{(-1)^n}{b+n+1},$$
we obtain (with $z=(s+1)/a$ in (\ref{sini}))
\begin{eqnarray*}
F(s)&=&\frac{1}{a} \sum_{n=0}^\infty(-1)^n\left(\frac{1}{n+\frac{s+1}{a}}+\frac{1}{n-\frac{s+1}{a}+1}\right)\\
&=& \frac{\pi}{a} \frac{1}{\sin\left(\frac{\pi}{a}(s+1)\right)}.
\end{eqnarray*}
  Hence
 $$ F^{(m)}(0)= \left(\frac{\pi}{a}\right)^{m+1} \frac{d^m}{(d s)^m}\csc \left(\frac{\pi}{a}(s+1)\right)\Big|_{s=0}=
  \left(\frac{\pi}{a}\right)^{m+1} \frac{d^m}{(d s)^m}\csc (x)\Big|_{x=\pi/a}.$$
  
  \vspace{1cm}
  {\bf Addendum}
  Using the Fa\`a di Bruno formula, or its variant in \cite{moru} or \cite{mo2013}, one may obtain an explicit representation of the value for  $\csc^{(n)}(\pi/a)$: in fact, take $f(x)=1/x$ and $g(x)=\sin x$ and pull in
  
  $$(f\circ g)^{(n)}(z)=
\sum_{j=1}^n f^{(j)}(g(z))\biggl(\sum_{{\mathbf k\in{\N}^j\atop |\mathbf k|=n}}C_\mathbf k^n\, g^{(\mathbf k)}(z)\biggr),\eqno ({\rm Mo}_n)$$

where $\mathbf k=(k_1,k_2,\dots,k_j)\in\N^j$ is an ordered multi-index with $k_1\geq k_2\leq\cdots\geq k_j\geq 1$,
$|\mathbf k|=\sum_{i=1}^j k_i,  g^{(\mathbf k)}=g^{(k_1)}g^{(k_2)}\dots g^{(k_j)}$ and
$\dis C_\mathbf k^n=\frac{1}{~\prod_i[A_\mathbf k(i)!]~}{n\choose k}$. Here $A_\mathbf k(i)$ denotes the cardinal of how often
$i$ appears within the ordered index $\mathbf k$ and ${n\choose \mathbf k}=\frac{n!}{ k_1!k_2!\dots k_j!}$.

This yields an expression of the form $\dis \sum_{j=1}^{n+1} C_j\frac{(\cos x)^{N_j}}{(\sin x)^j}$.

\newpage

\begin{figure}[h!]
 \scalebox{0.45} 
  {\includegraphics{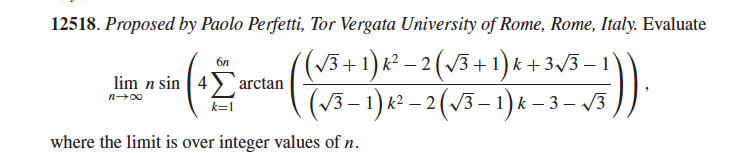}} 
  
 \end{figure}
 

\medskip

\centerline{{\bf Solution to problem 12518  in Amer. Math. Monthly 132 (2025), p. ?}}\medskip 

\centerline{Raymond Mortini  and  Rudof Rupp }

\bigskip

\centerline{- - - - - - - - - - - - - - - - - - - - - - - - - - - - - - - - - - - - - - - - - - - - - - - - - - - - - -}
  
  \medskip


We show that the limit is $4/3$. 
Write

$$u_k:=\frac{k^2-2k +\frac{3\sqrt 3-1}{\sqrt 3+1}}{k^2-2k- \frac{3+\sqrt 3}{\sqrt 3-1}}\; \frac{\sqrt 3+1}{\sqrt 3-1}
=:\frac{r_k}{s_k}v$$

with $$v=\frac{\sqrt 3+1}{\sqrt 3-1} =\tan \left(\frac{5\pi}{12}\right)\sim 3.7320508\cdots$$

Then 
$$u_k=\frac{v(k-1)^2+2}{(k-1)^2-2v}= \left(\frac{\frac{(k-1)^2}{2}-v}{1+v \frac{(k-1)^2}{2}}\right)^{-1}.$$
Note that $u_k>0$ if and only if $k\geq 4$. 
Due to the formula 
$$\mbox{$\dis\arctan x+\arctan y= \arctan \left(\frac{x+y}{1-xy}\right)$  whenever $xy<1$},$$
we have
\begin{eqnarray*}
\arctan (1/u_k)&=& \arctan \frac{(k-1)^2}{2}-\arctan v=\frac{\pi}{2}-\arctan\left(\frac{2}{(k-1)^2}\right)-\arctan v\\
&=&\frac{\pi}{2}-\arctan k+\arctan (k-2) -\arctan v.
\end{eqnarray*}
Thus, for $k\geq 4$,
$$\arctan u_k=\arctan k-\arctan(k-2)+\arctan v,$$
and for $1\leq k\leq 3$,
$$\arctan u_k=-\frac{\pi}{2}-\arctan (1/u_k)=-\pi +\arctan k-\arctan(k-2)+\arctan v.
$$
Hence
\begin{eqnarray*}
4\sum_{k=1}^{6n}  \arctan u_k&=&4\sum_{k=1}^{6n}( \arctan u_k -\arctan v) +24 n \arctan v\\
&=&4 (-3\pi)+4\sum_{k=1}^{6n} (\arctan k-\arctan (k-2)) +10\,\pi n\\
&=&4\sum_{k=1}^{6n}\big(\arctan k-\arctan (k-1)\big) +\big(\arctan (k-1)-\arctan(k-2)\big)+2\pi(5n-6)\\
&=& 4 \big(\arctan (6n) +\arctan(6n-1)\big)+\pi + 2\pi(5n-6).
\end{eqnarray*}
Using that $\arctan x+\arctan (1/x)=\pi/2$ for  $x\geq 0$, we obtain
\begin{eqnarray*}
4\sum_{k=1}^{6n}  \arctan u_k&=&4 \big(\pi -\arctan\frac{1}{6n} -\arctan\frac{1}{6n-1}\big)+\pi + 2\pi(5n-6).
\end{eqnarray*}
Therefore, as $\sin(x+j\pi)=-\sin x$ for each odd $j\in \Z$, and $\arctan x \sim x$ and $\sin x\sim x$ for $x\to 0$,
\begin{eqnarray*}
n\sin \left(4\sum_{k=1}^{6n}  \arctan u_k\right)&=& n \sin\left(4\arctan\frac{1}{6n} +4\arctan\frac{1}{6n-1}\right)
\sim 4/6+4/6=4/3.
\end{eqnarray*}

\newpage


\begin{figure}[h!]
 \scalebox{0.45} 
  {\includegraphics{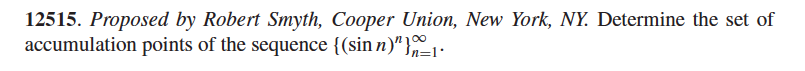}} 
  
 \end{figure}

\medskip

\centerline{{\bf Solution to problem 12515  in Amer. Math. Monthly 132 (2025), p. 181}}\medskip 

\centerline{Raymond Mortini  and  Myriam Ounaies }

\bigskip

 \centerline{- - - - - - - - - - - - - - - - - - - - - - - - - - - - - - - - - - - - - - - - - - - - - - - - - - - - - -}
  
  \medskip
  

  We show the following:\\
  
  \centerline{\ovalbox{The set of cluster points of $(\sin n)^n$ equals $[-1,1]$.} }\bigskip
  
For the proof, we shall  use the following result,  a weaker version apparently  going back due to Kronecker, 
and which was later strengthened e.g. by Khintchine \cite{kh}, Cassels \cite{ca}  and Descombes \cite{des}. See also \cite{raym}.

\begin{theorem}\label{kron2}
For every real number  $\eta$ and every  irrational number $\xi$ there are infinitely many pairs $(p,n)$ of integers with $n>0$ such that
\begin{equation}\label{kro}
  n \xi-p-\eta=\Oh(1/n).
\end{equation}
\end{theorem} 
We emphasize that the non-trivial sign condition on $n$ is very important for us.\\

In order to get acquainted with our method, we  first reprove Kronecker's result that every point in $[-1,1]$ is a 
 cluster point of $(\sin n)$ 
(for the standard proof, see \cite[p. 1878]{moru}.).\\

So let $\lambda\in [-1,1]$. Put $\xi=\frac{1}{2\pi}$, and $\eta=\frac{\arcsin \lambda}{2\pi}$.  
Then take these sequences $(n_k,p_k)$ such that
$$\frac{n_k}{2\pi}-p_k-\frac{\arcsin \lambda}{2\pi}= \Oh(1/n_k),$$
and let  $a_k:=n_k-2\pi p_k -\arcsin \lambda$. Then $a_k\to 0$ and so
\begin{eqnarray*}
\sin n_k&=&\sin(a_k+\arcsin \lambda)= \sin a_k \cos (\arcsin \lambda) +\lambda \cos a_k\\
&\to & 0+\lambda.
\end{eqnarray*}

   {\sl The Solution:}\\

Let $C$ be the set of cluster points of $(\sin n)^n$. Since $C$ is closed,  it is sufficient to prove that every number 
$\lambda\in\; ]-1,1[$, $\lambda\not=0$,  belongs to $C$. \footnote{ The case $\lambda=0$, though is very easy: by Kronecker's  result there exists a subsequence $(n_k)$ with $\sin n_k\to 0$. 
Hence $(\sin n_k)^{n_k}\to 0.$}. Towards  a proof of this, 
let $\sigma:=1$ if $0<\lambda<1$ and $\sigma:=-1$ if $-1<\lambda<0$.
 To prove that $\lambda\in C$, 
we are going  to show in  a moment that  there are sequences of {\it odd} integers $n_k$ and integers $p_k$ 
 tending to $\infty$ such that
\begin{equation}\label{mainsinu}
\sqrt{ n_k}\; \Big|n_k-2\pi p_k-\sigma\pi/2\Big| \to \sqrt{-2\;\log |\lambda|}.
 \end{equation}

First let us see why this gives the result. Put $a_k:=n_k-2\pi p_k - \sigma\pi/2 $.   Then $a_k\to 0$. Moreover,

$$\sin n_k=\sin(a_k+\sigma \pi/2)=\sigma \cos a_k\to \sigma.$$

 Hence
$$n_k(1-\sigma\sin n_k)= n_k(1-\cos a_k)= n_ka_k^2\cdot \frac{1-\cos a_k}{a_k^2}\to 
-2 \log|\lambda|  \cdot\frac{1}{2}$$

Since $\dis \lim_{x\to \sigma}\frac{\log (\sigma x)}{1-\sigma x}=-1$,

\begin{eqnarray*}
n_k\log(\sigma \sin n_k)&=&n_k(1-\sigma \sin n_k)\cdot \frac{\log(\sigma \sin n_k)}{1-\sigma\sin n_k}\\
&\to&- \log(|\lambda|) \cdot (-1) =\log(|\lambda|).
\end{eqnarray*}

Consequently, as $n_k$ is odd and $\sigma^2=1$,  
$$(\sin n_k)^{n_k} =\sigma(\sigma\sin n_k)^{n_k}  \to \lambda.$$

To prove (\ref{mainsinu}), we apply Theorem  \ref{kron2} with
 $\xi=1/\pi$ and $\eta=\sigma/4-1/(2\pi)$. There are
 infinitely many pairs $(P_k,N_k)$ of integers with $N_k>0$ such that
$$ \frac{N_k}{\pi} -P_k-\frac{\sigma}{4}+\frac{1}{2\pi}=\Oh(1/N_k).$$
That is, by multiplying with $2\pi$,
$$S_k:=(2N_k+1)-2\pi P_k-\frac{\sigma \pi}{2}=\Oh(1/N_k).$$

Define  $X_k=X(k,\lambda)$ by
\begin{equation}\label{equ22}
(4X_k+1)^{3/2}=\frac{\sqrt{-2\log |\lambda|}}{|S_k| \sqrt{2N_k+1}}.
\end{equation}
In particular $X_k\to \infty$. Now let
\begin{equation}\label{equ33}
d_k:=4\lfloor X_k \rfloor+1 \sim 4X_k+1,
\end{equation}
where the asymptotics follows from the fact that  $4X_k-3\leq d_k \leq 4X_k+1$. 
Using (\ref{equ22}) and (\ref{equ33}), we have

 $$
\big((2N_k+1)d_k\big)^{1/2}\left| d_k(2N_k+1) -2\pi d_kP_k-d_k\sigma\frac{\pi}{2}\right| =
d_k^{3/2} \sqrt{2N_k+1} \;|S_k|
\sim  \sqrt{-2\log|\lambda|}
$$
 \medskip
 
Next use that
$$d_k \frac{\pi}{2}= \pi\frac{4\lfloor X_k \rfloor +1}{2}= 2\pi \lfloor X_k \rfloor  + \frac{\pi}{2}.$$
\medskip

Hence, if we put
$n_k=d_k(2N_k+1)$ and $p_k=d_kP_k+\sigma \lfloor X_k \rfloor $,  we obtain (\ref{mainsinu}),
that is

 $$\sqrt {n_k} \big|(n_k-2\pi p_k - \sigma \pi/2\big|\to \sqrt{-2\;\log |\lambda|}.$$
 \medskip

  {\bf Remark}  
  This Problem had already appeared in the Monthly \cite{kr}. The displayed solution was rather technical.

  \newpage

\begin{figure}[h!]
 \scalebox{0.45} 
  {\includegraphics{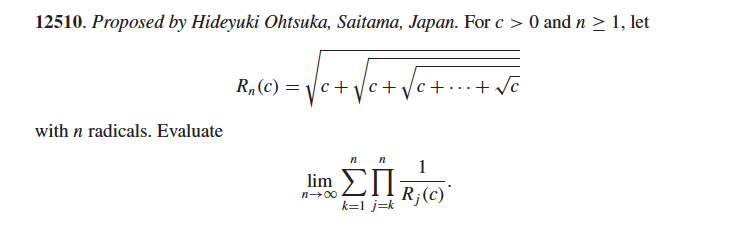}} 
  
 \end{figure}
 

\medskip

\centerline{{\bf Solution to problem 12510  in Amer. Math. Monthly 132 (2025), p. 180}}\medskip 

\centerline{Raymond Mortini  and  Rudolf Rupp }

\bigskip

\centerline{- - - - - - - - - - - - - - - - - - - - - - - - - - - - - - - - - - - - - - - - - - - - - - - - - - - - - -}
  
  \medskip


  Let $R_1(c)=\sqrt c$, $R_2(c)=\sqrt{c+\sqrt c}$, $R_3(c)=\sqrt{c+\sqrt{c+\sqrt c}}$, $\cdots$.
  In general
  $$R_{n+1}(c)=\sqrt{c+R_n(c)}.$$
  
  We prove  that 
  $$\ovalbox{$\dis f(c):=\lim_n \sum_{k=1}^n \prod_{j=k}^n R_j(c)^{-1}=\frac{ \frac{1}{2}+ \sqrt{\frac{1}{4}+c}}{c}.$}$$
  
  {\it Claim 1.} $R_n$ is increasing and $\sqrt c\leq R_n(c)\leq \frac{1}{2}+ \sqrt{\frac{1}{4}+c}$.
  
  The monotony is clear. Now 
  $$0\leq R_{n+1}^2(c)-R_n^2(c)=c+R_n(c)-R_n(c)^2.$$
Put $x:=R_n(c)$. Note that the parabola $c+x-x^2\geq 0$ if and only if
$$\frac{-1+\sqrt{1+4c}}{-2}\leq x\leq \frac{-1-\sqrt{1+4c}}{-2}.$$
Hence 
$$0<\sqrt c \leq R_n(c)\leq \frac{1}{2}+\sqrt{\frac{1}{4}+c}.$$
  
  {\it Claim 2.} $R:=\lim R_n$ exists and $R=\frac{1}{2}+\sqrt{\frac{1}{4}+c}>1$.
  
  In fact, the existence being clear, we must have $R^2=c+R$. Hence $R=\frac{1}{2}+\sqrt{\frac{1}{4}+c}>1$.
  
  {\it Claim 3.}  Let $f_n(c):=\sum_{k=1}^n \prod_{j=k}^n R_j(c)^{-1}$. Then $f(c)=\lim f_n(c)$ exists. This is clear since the product and sum of positive decreasing sequences are decreasing. Let us now calculate $f(c)$.
  
  Since
  $$\frac{ f_n(c)}{R_{n+1}(c)}=\sum_{k=1}^n \prod_{j=k}^{n+1} \frac{1}{R_j(c) }=
  \sum_{k=1}^{\co{n+1}} \prod_{j=k}^{n+1} \frac{1}{R_j(c) }-\frac{1}{R_{n+1}(c)}=f_{n+1}(c)-\frac{1}{R_{n+1}(c)},$$
  we may pass to the limit and obtain
  $$\frac{f(c)}{R}=f(c)-\frac{1}{R}.$$
  In other words, 
  $$f(c)=\frac{1}{R-1}=\frac{ \frac{1}{2}+ \sqrt{\frac{1}{4}+c}}{c}.
  $$

\newpage
\begin{figure}[h!]
 \scalebox{0.45} 
  {\includegraphics{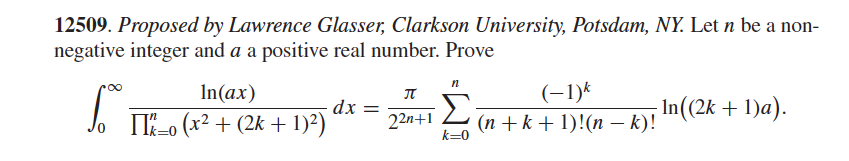}} 
  
 \end{figure}
 

\medskip

\centerline{{\bf Solution to problem 12509  in Amer. Math. Monthly 132 (2025), p. 90}}\medskip 

\centerline{Raymond Mortini  and  Rudolf Rupp }

\bigskip

\centerline{- - - - - - - - - - - - - - - - - - - - - - - - - - - - - - - - - - - - - - - - - - - - - - - - - - - - - -}
  
  \medskip


We first derive the partial fraction: 
  
 \begin{eqnarray}
\frac{1}{\prod_{k=0}^n \left(x^2+(2k+1)^2\right)}&=& \sum_{k=0}^n \frac{r_k^+}{x-(2k+1)i}+ \sum_{k=0}^n \frac{r_k^-}{x+(2k+1)i} \label{11}\\
&=&\sum_{k=0}^n \frac{(r_k^++r_k^-)\; x+ (2k+1)i(r_k^+-r_k^-)}{x^2+(2k+1)^2}.\label{12}
\end{eqnarray}

To calculate the residues $r_k^\pm$, write
$${\prod\limits_{k=0}^{n}{\left({{x}^{2}}+{{(2k+1)}^{2}}\right)}}=\left({x}^{2}+(2j+1)^{2}\right)\cdot
{\prod\limits_{k=0,k\ne j}^{n}{({{x}^{2}}+{{(2k+1)}^{2}})}}
 $$
 multiply equation (\ref{11}) by $x\pm i(2j+1)$, and let $x\to \pm i(2j+1)$ to obtain
 
 \begin{eqnarray*}
 r_{j}^{\pm }&=&\frac{1}{2(2j+1)(\pm i)}\cdot\frac{1}{\prod\limits_{k=0,k\ne j}^{n}{({{(2k+1)}^{2}}-{{(2j+1)}^{2}})}}\\
 &=&\frac{1}{2(2j+1)(\pm i)}\cdot\frac{1}{\prod\limits_{k=0,k\ne j}^{n}{(2k+2j+2)(2k-2j)}} \\
 & =&\mp i\frac{1}{{{2}^{2n+1}}}\frac{1}{\prod\limits_{k=0}^{n}{(k+j+1)\prod\limits_{k=0,k\ne j}^{n}{(k-j)}}}\\
 &=&\mp i\frac{1}{{{2}^{2n+1}}}\frac{1}{(n+j+1)\cdots (j+1)\cdot (-j)\cdots (-1)\cdot 1\cdots (n-j)} \\
&=& \mp i\frac{1}{{{2}^{2n+1}}}\frac{{{(-1)}^{j}}}{(n+j+1)!(n-j)!}.
\end{eqnarray*}

Hence $r_k^++r_k^-=0$ and 
$$r_k^+-r_k^-=-i \frac{1}{2^{2n}}\;\frac{(-1)^k}{(n+k+1)!(n-k)!}.$$

Next we use that
$$I:=\int_0^\infty \frac{\log u}{u^2+1}du=0,$$
since
\begin{eqnarray*}
I&=&\int_0^1 \frac{\log (u)}{u^2+1}du+\int_1^\infty\frac{\log u}{u^2+1}du\\
&\buildrel=_{u=1/x}^{}&\int_0^1 \frac{\log u}{u^2+1}du - \int_0^1\frac{ \log  x}{1+x^2}dx. 
\end{eqnarray*}

Thus

\begin{eqnarray*}
\int\limits_{0}^{\infty }{\frac{\log (ax)}{{{x}^{2}}+{{(2k+1)}^{2}}}}dx&\overset{x=(2k+1)u}=&\,\frac{1}{2k+1}\int\limits_{0}^{\infty }{\frac{\log (a(2k+1))+\log (u)}{{{u}^{2}}+1}}du\\
 &=&\frac{\log (a(2k+1))}{2k+1}\frac{\pi }{2}+0=\frac{\log (a(2k+1))}{2k+1}\frac{\pi }{2}.
\end{eqnarray*}

Hence, by (\ref{12}),
\begin{eqnarray*}
 \int\limits_{0}^{\infty }{\frac{\log (ax)}{\prod\limits_{k=0}^{n}{({{x}^{2}}+{{(2k+1)}^{2}})}}dx}&=&
 \frac{1}{{{2}^{2n}}}\sum\limits_{k=0}^{n}{\frac{{{(-1)}^{k}}(2k+1)}{(n+k+1)!(n-k)!}\int\limits_{0}^{\infty }{\frac{\log (ax)}{{{x}^{2}}+{{(2k+1)}^{2}}}dx}}\\
 &=&\frac{\pi }{{{2}^{2n+1}}}\sum\limits_{k=0}^{n}{{{(-1)}^{k}}\frac{\log (a(2k+1))}{(n+k+1)!(n-k)!}}.
\end{eqnarray*}

\newpage

\begin{figure}[h!]
 \scalebox{0.45} 
  {\includegraphics{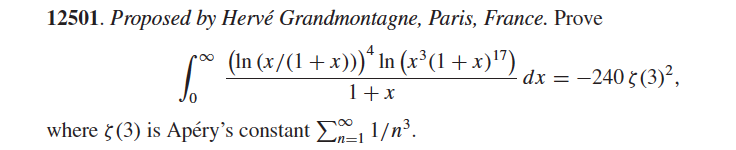}} 
  
 \end{figure}
 

\medskip

\centerline{{\bf Solution to problem 12501  in Amer. Math. Monthly 131 (2024), p. 906}}\medskip 

\centerline{Raymond Mortini  and  Rudolf Rupp }

\bigskip

\centerline{- - - - - - - - - - - - - - - - - - - - - - - - - - - - - - - - - - - - - - - - - - - - - - - - - - - - - -}
  
  \medskip


 Substituting $x/(1+x)=:t$, equivalently $x=t/(1-t)$ with $t\in ]0,1[$ yields
  
 \begin{eqnarray*}
I:= \int_0^\infty \left(\log\frac{x}{1+x}\right)^4 \log\left(x^3(1+x)^{17}\right)\;dx&=&\int_0^1 \frac{\log^4 t
 \left( 3\log\left(\frac{t}{1-t}\right)+17 \log\left(\frac{1}{1-t}\right)\right)}{\frac{1}{1-t}} \; \frac{dt}{(1-t)^2}\\
 &=&\int_0^1 \frac{\log^4 t\Big(3\log t-20\log(1-t)\Big)}{1-t}\;dt\\
 &=& 3\;\int_0^1 \frac{\log^5 t}{1-t}\;dt-20\;\int_0^1 \frac{\log^4 t \;\log(1-t)}{1-t}\;dt\\
 &=:&3I_1-20 I_2.
\end{eqnarray*} 
  
  Now, as the integrands have constant sign, $\int\sum=\sum\int$, and so, using 5-times integration by parts, 
  \begin{eqnarray*}
  I_1&=& \int_0^1\sum_{n=0}^\infty t^n \log^5 t\; dt=\sum_{n=0}^\infty \int_0^1 t^n\,\log^5 t\; dt\\
  &=&-\sum_{n=0}^\infty \frac{5!}{(n+1)^6}= -120 \zeta(6).
\end{eqnarray*}
Let $H_k=1+\frac{1}{2}+\frac{1}{3}+\cdots+\frac{1}{k}$ denote the $k$-th harmonic number.  Using that for $0\leq x<1$, 
$$ \frac{\log(1-x)}{1-x}=-\sum_{k=1}^\infty H_k\,x^k$$ (Cauchy-product), 
we obtain

\begin{eqnarray*}
I_2&=&-\int_0^1 \log^4 t \sum_{k=1}^\infty H_k t^k \;dt=-\sum_{k=1}^\infty \int_0^1 t^k \log^4 t  \;dt\\
&=&-4!\sum_{k=1}^\infty \frac{H_k}{(k+1)^5}.
\end{eqnarray*}
  
 Due to Euler's  formula (see \cite[p. 416]{ba}, \cite{har}), 
 $$
 2\sum_{k=1}^\infty \frac{H_k}{(k+1)^n}= n \zeta(n+1)-\sum_{k=1}^{n-2} \zeta(n-k)\zeta(k+1),\quad (n\geq 2),
 $$
  we conclude that
  \begin{eqnarray*}
  \sum_{k=1}^\infty \frac{H_k}{(k+1)^5}&=&\frac{1}{2}\left( 5\,\zeta(6)- \zeta(4)\zeta(2)-\zeta(3)^2-\zeta(2)\zeta(4)\right)\\
  &=& \frac{5}{2}\zeta(6) -\zeta(4)\zeta(2)-\frac{1}{2}\zeta(3)^2.
\end{eqnarray*}
  Since $\zeta(6)= \frac{\pi^6}{945}$ and $\zeta(4)\zeta(2)=\frac{\pi^4}{90}\; \frac{\pi^2}{6}=\frac{\pi^6}{540}=\frac{945}{540}\zeta(6)
  =\frac{7}{4}\zeta(6)$, we deduce that
  
   \begin{eqnarray*}
   I=3I_1-20I_2&=& -360 \zeta(6)+20  \cdot 24 \Big( \frac{5}{2}\zeta(6) -\zeta(4)\zeta(2)-\frac{1}{2}\zeta(3)^2\Big)\\
   &=&\zeta(6)\Big(-360+1200-840\Big)-240 \zeta(3)^2\\
   &=&-240 \zeta(3)^2.
\end{eqnarray*}

\newpage

\begin{figure}[h!]
 \scalebox{0.45} 
  {\includegraphics{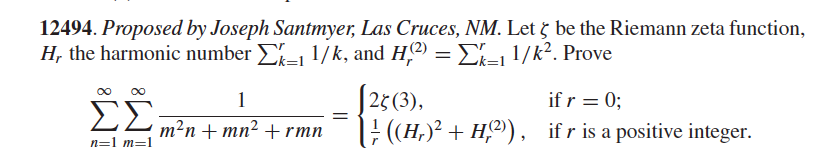}} 
  
 \end{figure}
 

\centerline{{\bf Solution to problem 12494  in Amer. Math. Monthly 131 (2024), p. 815}}\medskip 

\centerline{Raymond Mortini  and  Rudolf Rupp }

\bigskip

\centerline{- - - - - - - - - - - - - - - - - - - - - - - - - - - - - - - - - - - - - - - - - - - - - - - - - - - - - -}
  
  \medskip
  
  Let  $$S:=S(r):=\sum_{n=1}^\infty\sum_{m=1}^\infty \frac{1}{m^2n+mn^2+rmn}.$$\\
  
  We additionally show that for $r\in \N=\{0, 1,2,3,\dots\}$
  $$S(r)=\int_0^{1} x^{r-1} \log^2 (1-x)dx=\begin{cases} 2 \zeta(3)&\text{if $r=0$}\\
 \dis \sum_{j=0}^{r-1} \binom{r-1}{j} (-1)^j \frac{2}{(j+1)^3} &\text{if $r\not=0$}.
  \end{cases}
  $$
 If, more generally,  $r\geq 0$ is a real number, then  this  formula also holds in the form
   $$S(r)=\sum_{j=0}^{\infty}  \binom{r-1}{j} (-1)^j  \frac{2}{(j+1)^3}.$$
   
 \subsection{The calculations}
  
 Let $m,n\in\N^*:=\N\setminus \{0\}$  and $r\in \R, r\geq0$. Then
  
  \begin{eqnarray*}
  \frac{1}{m^2n+mn^2+rmn}&=& \frac{1}{mn}\;\frac{1}{m+n+r}=\frac{1}{m(m+r)} \left(\frac{1}{n}-\frac{1}{n+(m+r)}\right)\\
  &=&\frac{1}{m(m+r)}\;\int_0^1 (x^{n-1}-x^{n+m+r-1})dx\\
  &=&\frac{1}{m(m+r)}\;\int_0^1 x^{n-1}(1-x^{m+r}) dx.
\end{eqnarray*}
  
  Since all terms considered are positive, $\int\sum=\sum\int$. Hence, by partial integration,
  
  \begin{eqnarray}\label{harm-num}
\sum_{n=1}^\infty \frac{1}{m^2n+mn^2+rmn} &=& \frac{1}{m(m+r)}\;\int_0^1 (1-x^{m+r}) \sum_{n=1}^\infty x^{n-1}dx\\\nonumber
&=&  \frac{1}{m(m+r)}\;\int_0^1 \frac{1-x^{m+r}}{1-x} dx\\ \nonumber
&=&-\frac{1}{m} \int_0^1 x^{m+r-1}\log(1-x) dx\nonumber
\end{eqnarray}

Consequently, if $r\geq 0$, $r\in \R$,

\begin{eqnarray}\label{zwisch}
S(r)&=&-\sum_{m=1}^\infty \frac{1}{m}\int_0^1 x^{m+r-1}\log(1-x)dx\nonumber\\
&=&-\int_0^1 x^{r-1} \left(\sum_{m=1}^\infty\frac{x^m}{m}\right)\log(1-x)dx=\int_0^1  x^{r-1}\log^2(1-x) dx\\
&=&\int_0^1 (1-x)^{r-1}\log^2 x\;dx.
\end{eqnarray}

Hence, if $r\in \N^*$,
\begin{eqnarray*}
S(r)&=&\sum_{j=0}^{r-1}  \binom{r-1}{j} (-1)^j  \int_0^1 x^j \log^2 x\;dx\\
&=& \sum_{j=0}^{r-1}  \binom{r-1}{j} (-1)^j  \frac{2}{(j+1)^3}.
\end{eqnarray*}

 where the latter identity is shown using twice partial integration.
 \\
 
 If $r=0$, we have the convergent integral (two singularities), 
 
 \begin{eqnarray*}
 S&=&\int_0^1 (1-x)^{-1}\log^2 x\;dx=\int_0^1 \sum_{j=0}^\infty x^j \log^2 x\;dx= \sum_{j=0}^\infty \int_0^1 x^j \log^2 x\;dx\\
 &=&  \sum_{j=0}^\infty \frac{2}{(j+1)^3}=2 \zeta(3).
\end{eqnarray*}
 
 If $r\in\R$ and $r>0$, 
 \begin{eqnarray*}
S(r)&=&\sum_{j=0}^{\infty}  \binom{r-1}{j} (-1)^j  \int_0^1 x^j \log^2 x\;dx\\
&=& \sum_{j=0}^{\infty}  \binom{r-1}{j} (-1)^j  \frac{2}{(j+1)^3}.
\end{eqnarray*}

 Next, if $r\in \N^*$,  we derive the desired equality  from \ref{zwisch}. We do this using the following trick:  consider the function
 $$f(a):=(1-x)^{a-1}, a>0.$$
 Then
 $$\frac{d^2}{da^2}(1-x)^{a-1}=\frac{d}{da}(1-x)^{a-1}\log(1-x)=(1-x)^{a-1}\log^2(1-x).$$
 In particular, 
 $$\log^2(1-x)=\frac{d^2}{da^2}(1-x)^{a-1}\Big|_{a=1}.$$
  
  For $a>0$, let
  $$G(a):=\int_0^1  x^{r-1} \frac{d^2}{da^2}(1-x)^{a-1}dx.$$

The continuity of  $G$ on $]0,\infty[$ and $(x,a)\mapsto  x^{r-1} \frac{d^2}{da^2}(1-x)^{a-1}$ on $]0,1[\times ]0,\infty[$  yields that 
 $$\lim_a\int=\int\lim_a.$$
  Hence

 \begin{eqnarray*}
 G(1)&=& \lim_{a\to 1}  \int_0^1  x^{r-1} \frac{d^2}{da^2}(1-x)^{a-1}dx=   \int_0^1  x^{r-1} \lim_{a\to 1}\frac{d^2}{da^2}(1-x)^{a-1}dx\\
 &=& \int_0^1  x^{r-1} \log^2(1-x) dx.
 \end{eqnarray*}

  Now, in view of the definition of the Eulerian beta-function, $B(r,a)$,  and the fact that $ \frac{d^2}{da^2}\int= \int\frac{d^2}{da^2}$
 
 \begin{eqnarray*}
G(a)&= & \frac{d^2}{da^2} \int_0^1  x^{r-1}(1-x)^{a-1}dx\\
 &=&\frac{d^2}{da^2} B(r,a)= \frac{d^2}{da^2} \frac{\Gamma(r)\Gamma(a)}{\Gamma(a+r)}\\
 &=&\frac{d^2}{da^2} \left(\frac{(r-1)!}{a(a+1)\cdots(a+r-1)}\right).
\end{eqnarray*}
For $a>0$, put
  $$L(a):= \log \frac{1}{ \prod_{j=0}^{r-1}(a+j)} =-\sum_{j=0}^{r-1} \log (a+j).$$
 
  Then 
  $$F(a):=\frac{B(r,a)}{(r-1)!}= e^{L(a)}.$$
  Hence
\begin{eqnarray*}
G(1)=(r-1)!F''(1)=(r-1)!F(1)\big((L'(1))^2+L''(1)\big)=\frac{(r-1)!}{r!} \big((H_r)^2+H_r^{(2)}\big).
\end{eqnarray*}
  
  \vspace{1cm}
  
{\bf Remark}.~~
Using that for $r\in \N^*$
\begin{eqnarray*}
 \frac{1}{m(m+r)}\;\int_0^1 \frac{1-x^{m+r}}{1-x} dx
&=& \frac{1}{m(m+r)}\; \int_0^1\sum_{k=0}^{m+r-1} x^k dx\\
&=&  \frac{H_{m+r}}{m(m+r)},
\end{eqnarray*}
  
  we also have that,  in view of (\ref{harm-num}),  $\dis S(r)=\sum_{m=1}^\infty \frac{H_{m+r}}{m(m+r)}$.

\newpage

\begin{figure}[h!]
 \scalebox{0.45} 
  {\includegraphics{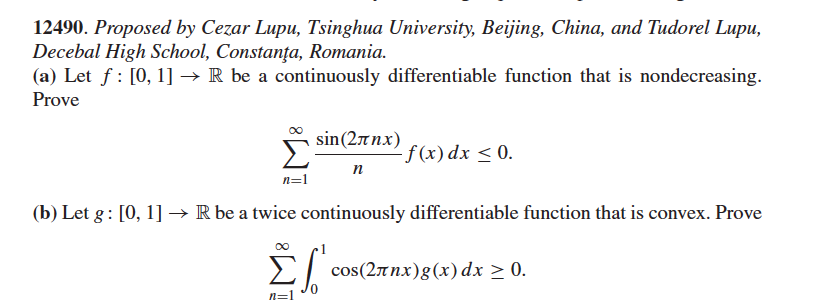}} 
  
 \end{figure}

\centerline{{\bf Solution to problem 12490  in Amer. Math. Monthly 131 (2024), p. 814}}\medskip 

\centerline{Raymond Mortini  and  Rudolf Rupp }

\bigskip

\centerline{- - - - - - - - - - - - - - - - - - - - - - - - - - - - - - - - - - - - - - - - - - - - - - - - - - - - - -}
  
  \medskip


 a1) We prove that $\sum\int$ is well defined and that
  $$\sum_{n=1}^\infty \int_0^1 \frac{\sin(2\pi nx)}{n}\,f(x)\;dx\leq 0.$$
  
 As $f'\geq 0$,  partial integration yields
  
  \begin{eqnarray*}
 I_n:= \int_0^1 \frac{\sin(2\pi nx)}{n}\,f(x)\;dx&=&-\frac{\cos(2\pi nx)}{2\pi n^2} f(x)\Bigg|^1_0 + \int_0^1 \frac{\cos(2\pi nx)}{2\pi n^2} f'(x)dx\\
  &\leq &\frac{f(0)-f(1)}{2\pi n^2} +\frac{1}{2\pi n^2}\int_0^1 f'(x)dx=0.
  \end{eqnarray*}

Since $|I_n|\leq  \frac{c}{n^2}$ the series converges $S:=\sum I_n$ and so $S\leq 0$.\\

a2)  We prove that $\int\sum$ is well defined and that
  $$\int_0^1\left(\sum_{n=1}^\infty \frac{\sin(2\pi nx)}{n}\,f(x)\;\right)dx\leq 0.$$
  
  A standard exercise in Fourier analysis tells us 
  $$\mbox{$\dis S(x):=\sum_{n=1}^\infty\frac{\sin(2\pi nx)}{n}=\frac{\pi-2\pi x}{2}$ for $0<x<1$.}
  $$
  
Since $f'\geq 0$, 
\begin{eqnarray*}
\int_0^1 S(x) f(x)dx&= &\frac{\pi}{2}\left(\int_0^1  f(x)dx-2\int_0^1  x f(x) dx\right)\\
&=& \frac{\pi}{2}\left( x f(x)\Big|_0^1 -\int_0^1 x f'(x) dx- x^2 f(x)\Big|_0^1 +\int_0^1 x^2f'(x) dx\right)\\
 &=&\frac{\pi}{2}\left( (x-x^2) f(x)\Big|_0^1 +\int_0^1 (x^2-x) f'(x) dx\right)\\
 &=&\frac{\pi}{2}\int_0^1 (x^2-x) f'(x) dx\leq 0.
\end{eqnarray*}

b) Since $g''\geq 0$,
\begin{eqnarray*}
J_n&=& \int_0^1\cos(2\pi nx) g(x)dx=-\frac{1}{2\pi n} \int_0^1 \sin(2\pi nx) g'(x)dx\\
&=&\frac{1}{4\pi^2 n^2} \cos(2\pi n x) g'(x)\Big|_0^1 -\frac{1}{4\pi^2 n^2}\int_0^1 \underbrace{\cos(2\pi nx)}_{\leq 1} g''(x) dx\\
&\geq& \frac{g'(1)-g'(0)}{4\pi^2 n^2}  -\frac{g'(1)-g'(0)}{4\pi^2 n^2} =0.
\end{eqnarray*}

Since $|J_n|\leq  \frac{c}{n^2}$, the series  $S^*:=\sum J_n$ converges and so $S^*\geq 0$.

\bigskip\bigskip

  - - - - - - - - - - - - - - - - - - - - - - - - - - - - - - - - - - - - \bigskip

\underline{Second proof of b)}\\

Extend $g$ 1-periodically to $\R$. 
The Fourier series for $g$ is given by
$$g(x)=\frac{a_0}{2} +\sum_{n=1}^\infty \left(a_n \cos(2\pi nx) +b_n\sin(2\pi nx)\right)$$ 
with
$a_n= 2 \int_0^1 g(x)\cos(2\pi nx)dx$ and $b_n= 2 \int_0^1 g(x)\sin(2\pi nx)dx$.
Since $g$ is smooth, Dirichlet's rule yields
\begin{equation}\label{diri}
\frac{g(0)+g(1)}{2}= \frac{a_0}{2} +\sum_{n=1}^\infty a_n= \int_0^1 g(x)dx+ 2\sum_{n=1}^\infty  \int_0^1 \cos(2\pi nx) g(x) dx.
\end{equation}

Now $\frac{g(0)+g(1)}{2}\geq \int_0^1 g(x)dx$ since the convexity of $g$ implies
$g(x)\leq g(0)+x(g(1)-g(0))$ and so
$$\int_0^1 g(x)dx\leq g(0)\cdot 1 +\frac{1}{2}(g(1)-g(0))=\frac{g(0)+g(1)}{2}.$$

  Hence, in view of (\ref{diri}), $\dis \sum_{n=1}^\infty  \int_0^1 \cos(2\pi nx) g(x) dx\geq 0$.\\

\underline{Second proof of a)}\\

This follows from b): Given $f$ in a), put $g(x):=\int_0^x f(t)dt$.  Then $g''\geq 0$; hence $g$ is convex. Now, by partial integration, 
\begin{eqnarray*}
\int_0^1 g(x)\cos(2\pi nx)dx&=& g(x)\frac{\sin(2\pi nx)}{2\pi n}\Big|^1_0-\int_0^1 \frac{\sin(2\pi nx)}{2\pi n}g'(x)dx\\
&=&-\int_0^1  \frac{\sin(2\pi nx)}{2\pi n} f(x)dx.
\end{eqnarray*}
Thus b) implies a).

\newpage

\pagecolor{yellow}
\begin{figure}[h!]
 
  \scalebox{0.45} 
  {\includegraphics{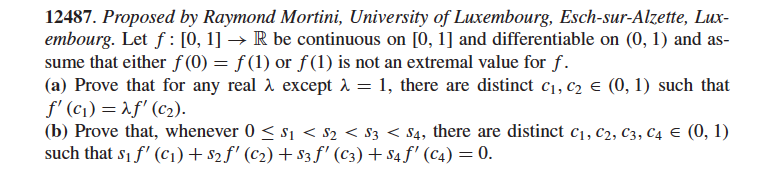}} 
  
 \end{figure}
 

\centerline{{\bf Solution to problem 12487  in Amer. Math. Monthly 131 (2024), p. 723}}\medskip 

\centerline{Raymond Mortini }

\bigskip

\centerline{- - - - - - - - - - - - - - - - - - - - - - - - - - - - - - - - - - - - - - - - - - - - - - - - - - - - - -}
  
  \medskip

 (a) Suppose that $f$ is not constant. Due to the hypotheses,  $f$ takes at least one of its extrema  inside $]0,1[$. By Rolle's theorem, there is $x_0\in ]0,1[$ for which $f'(x_0)=0$. Now $f'$ must also take negative as well as positive values, since otherwise $f$ is decreasing, respectively increasing, on $[a,b]$,   contradicting the assumption that  $f(0)=f(1)$ or that $f(1)$ is not an extremal value.  Say $f'(x_1)<0$ and $f'(x_2)>0$.  By the intermediate value property for the derivative (or Darboux property, see \cite[Theorem 5.12, p.108]{rud})), $f'$ takes every value $\eta$ with $f'(x_1)<\eta<f'(x_2)$.  In particular, every
      value $\eta$ with small modulus is taken, say $|\eta|\leq \e$. Now let $r_1\not=r_2$ be taken so that $|r_j|\leq\e$ and $r_1/r_2=\lambda$
   (this is possible: for instance,  if $|\lambda| >1 $, choose $r_1=\e$ and $r_2= \e/\lambda$, and if $|\lambda|<1$, choose $r_1=\lambda\e$ and $r_2=\e$). Now let $c_j\in ]0,1[$ be such that $f'(c_j)=r_j$. Then $c_1\not=c_2$ and 
   $$f'(c_1)/f'(c_2)=r_1/r_2=\lambda.$$

   (b) Let $0<s_1<s_2<s_3<s_4$.    
   By i)  there are $c_1,c_2\in ]0,1[$ with 
   $$f'(c_1)/f'(c_2)=\lambda:=-s_2/s_1=r_1/r_2$$  and
   $c_3,c_4\in ]a,b[$ with 
   $$f'(c_3)/f'(c_4)=\lambda:=-s_4/s_3=r_3/r_4.$$ 
   Since the modulus of $r_j$ can be taken to be arbitrarily small, the proof above  guarantees that all the $c_j$ are distinct.
   We conclude that $\sum_{j=1}^4 s_j f'(c_j)=0$.\\
 
 If  $s_1=0$, we choose $c_3,c_4$ as above, then $c_2:=x_0$ and $c_1$ arbitrary, but different from $c_2,c_3,c_4$. 
   
 \bigskip
 
{\bf Remark}  This was motivated by   problem 4779 Crux Math. 48 (8) 2022, 484.

\newpage

\nopagecolor
\begin{figure}[h!]
 
  \scalebox{0.45} 
  {\includegraphics{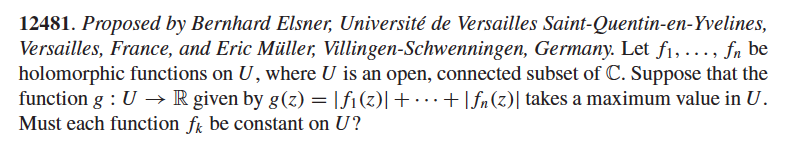}} 
  
 \end{figure}
 

\centerline{{\bf Solution to problem 12481  in Amer. Math. Monthly 131 (2024), p. 630}}\medskip 

\centerline{Raymond Mortini, Peter Pflug  and Rudolf Rupp }

\bigskip

\centerline{- - - - - - - - - - - - - - - - - - - - - - - - - - - - - - - - - - - - - - - - - - - - - - - - - - - - - -}
  
  \medskip

The answer is yes, and this is an easy  classical exercise in a complex analysis course; see \cite[p. 168]{bur},
 \cite[p. 353]{burck}, \cite[p. 161]{BC},  \cite[p. 164, Ex. 300/302]{PS}, \cite[p. 41]{T}, \cite{289114}. 
Since the modulus of a holomorphic function is subharmonic, it satisfies the maximum principle. Thus, if the maximum value 
of $f\in H(U)$ is taken inside the domain $U$, then $f$ is constant. This establishes the problem for $n=1$. Now if $n\geq 2$, we use that the finite sum $u$ of subharmonic functions is subharmonic again, a fact best seen by using the definition via the mean-value inequality
$$u(x_0)\leq \frac{1}{2\pi}\int_0^{2\pi} u(x_0+re^{i\theta})d\theta$$
for all closed disks $D(x_0,r)\ss U$. 
Thus $g=\sum_{j=1}^n |f_j|$ is subharmonic, and so is constant $c$ under the assumption of the problem.
 Next, fix $j_0\in \{1,\dots,n\}$. Then
 $$|f_{j_0}|= c-\underbrace{\sum_{j=1\atop j\not=j_0}^n|f_j|}_{:=v}.$$
 The subharmonicity of $v$ implies that $-v$ is superharmonic. Hence, by adding the constant $c$, $c-v$ is superharmonic, too. 
 Consequently, $|f_{j_0}|$ is superharmonic, as well as subharmonic. In other words, $|f_{j_0}|$ is harmonic. Any holomorphic function $f$ in  a domain $U$ whose modulus is harmonic, is constant though. In fact, let $D$ be a closed disk in $U$ such that $f$ has no zeros in $D$. Then $f$ admits a holomorphic square-root $q$ in $D$; that is $q^2=f$ (\cite[p. 816]{moru}). Now $|q|^2$ is harmonic by assumption. 
 Since $\Delta |q|^2= 4|q'|^2$ (where $\Delta$ is the Laplacian \cite[p. 222]{moru}), we get that $q'\equiv 0$ in $D$, and so $q$ is constant in $D$. Consequently, $f=q^2$ is constant in $D$. By the uniqueness theorem for holomorphic functions
  \cite[p. 347]{burck}, $f$ is constant in the
connected open set $U$.\\
 
 {\bf Remark 1.} A more direct, but not so elegant  way to prove that the harmonicity of $|f|$ implies constancy, is purely computational and uses the $\ov{\partial}$-calculus, that is the Wirtinger derivatives  $\partial u=u_z$ and $\ov{\partial} u=u_{\ov z}$
 (see \cite[sect. 4]{moru}): 
 
 Let $s(z):=|f(z)|$, $f\not\equiv 0$. Then  $s^2=f\ov f$ and $\partial s^2= 2 s \partial s= f' \ov f$. Hence, outside the discrete zero  set of $f$, 
 $\partial s= \frac{f' \ov f}{2 s}$, and so
 $$\Delta s=4\ov{\partial}\partial s= \frac{4}{2}\;\frac{f' \ov{f' }\;s-f' \ov f \frac{\ov{f'} f}{2s}}{s^2}\buildrel=_{}^{!}0 \iff 
 |f'|^2 \left(2s^2-|f|^2\right)=0\iff |f'|^2 |f|^2=0\iff f'=0.$$
 This could of course also be established via the "real"-method  by calculating $s_{xx}+s_{yy}$ with $s(z)\sim s(x,y)$, $z=x+iy$ and 
  $s=\sqrt{u^2+v^2}$, where $f=u+iv$, and applying the Cauchy-Riemann equations in their `real' form, instead of the shorter form $f_{\ov z}=0$ above.\\

  {\bf Remark 2.}  An analysis of the proof shows that one obtains the same conclusion replacing $\sum_{j=1}^n|f_j|$ by
  $\sum_{j=1}^n |f_j|^{\alpha_j}$, where $\alpha_j>0$. Just note the following two points: 
  
  1) if $f\not\equiv 0$ is holomorphic, 
  then $u:=\alpha\log|f|$ is subharmonic, and so its left composition $e^u$ with the exponential function yields the subharmonicity of $|f|^{\alpha}$ (see \cite[Chap.1, \S 6]{ga} or \cite[p.44]{ran}).
  
  2) The square root $q$ of $f$ above is replaced by an $\alpha/2$-root of $f$, so that  $q^2= f^\alpha$.

\newpage

\begin{figure}[h!]
 
  \scalebox{0.45} 
  {\includegraphics{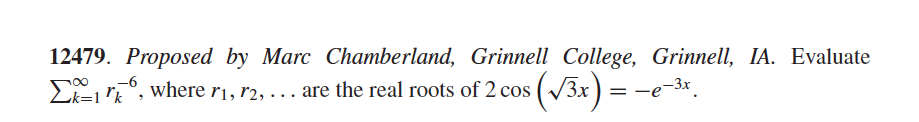}} 
  
 \end{figure}
 

\centerline{{\bf Solution to problem 12479  in Amer. Math. Monthly 131 (2024), p. 630}}\medskip 

\centerline{Raymond Mortini, Peter Pflug  and Rudolf Rupp }

\bigskip

\centerline{- - - - - - - - - - - - - - - - - - - - - - - - - - - - - - - - - - - - - - - - - - - - - - - - - - - - - -}
  
  \medskip

 For technical reasons, we use the new enumeration $r_0, r_1, r_2,\dots$ to  denote the real roots of the equation
 $2\cos(\sqrt 3 x)=-e^{-3x}$.  We show that 
 $$\ovalbox{$S:=\dis\sum_{n=0}^\infty r_n^{-6}=\frac{8}{5}$}. $$ \medskip
 
Recall the  Bourbaki notation $\N:=\{0,1,2,\dots\}$ and $\N^*:=\N\setminus \{0\}$.
 
 Let $$g(z):=2\cos(\sqrt 3z)+e^{-3z},$$
  and, by multiplication with $e^{3z}$,
$$f(z):=2e^{3z} \cos(\sqrt 3 z)+1.$$
 
 Note that $f$ also writes as  
 $$
 f(z)=e^{(3+i\sqrt 3)z}+e^{(3-i\sqrt 3)z}+1.
 $$
 Depending on which is best adapted to the computations, we shall use in the sequel all these variants.  
 Note that the functions $f$, $g$ have the same zero sets.
\\

{\bf Property 1}  Let 
$$\alpha:=e^{i2\pi/3}=-\frac{1}{2}+i\frac{\sqrt 3}{2}.$$
Then
$$
f(\alpha z)=e^{(-3+i\sqrt 3)z} f(z).
$$
In fact, by using that
$$\left(-\frac{1}{2}+i\frac{\sqrt 3}{2}\right)\;\left(3+i\sqrt 3\right)= -3+\sqrt 3 i$$
and
$$\left(-\frac{1}{2}+i\frac{\sqrt 3}{2}\right)\;\left(3-i\sqrt 3\right)=2i\sqrt 3,$$
we obtain
\begin{eqnarray*}
f(\alpha z)&=& e^{(-3+\sqrt 3i)z}+e^{2i\sqrt 3 z}+1\\
&=&e^{\sqrt 3iz}\left(e^{-3z}+e^{i\sqrt 3z}\right)+1\\
&=&e^{\sqrt 3iz}\left(e^{-3z}+e^{i\sqrt 3z}+e^{-i\sqrt 3z}\right)-1+1\\
&=&e^{\sqrt 3iz}\left(e^{-3z}+2\cos(\sqrt 3z)\right)\\
&=&e^{\sqrt 3iz}e^{-3z} f(z).
\end{eqnarray*}

As an immediate consequence we have the following useful property:\\

{\bf Property 2}  {\sl $f(z)=0$ if and only $f(\alpha z)=0$.}\\

{\bf Property 3} {\sl The function $f$ has infinitely  many positive zeros and no negative one. More precisely, for every $n\in \N$  each interval
$$I_n:=\left[\frac{2n\pi}{\sqrt 3},  \frac{2(n+1)\pi}{\sqrt 3}\right]$$ 
contains exactly two positive zeros of $f$, one contained in the left part of the interval 
$\left[\frac{\frac{\pi}{2}+2n\pi}{\sqrt 3},\frac{\frac{3\pi}{2}+2n\pi}{\sqrt 3}\right]$, and one in the right part. }
\begin{proof}

$\bullet$ Let us first show that $g$ (hence $f$) has no negative zeros.  In fact if $-\frac{\pi}{2\sqrt 3}\leq x\leq 0$, then $\cos(\sqrt 3x)\geq 0$, and so $g(x)=e^{-3x}+2\cos(\sqrt 3x)>0 $. If $x<-\frac{\pi}{2\sqrt 3}$, then $e^{-3x}> e^{\sqrt 3 \frac{\pi}{2}}\sim 15.1909\cdots>2$, and again $g(x)>0$.

$\bullet$ Now we deal with the positive solutions to $g(x)=0$; or equivalently $f(x)= 2e^{3x}\cos(\sqrt 3 x)+1=0$.
Consider  the points
$$a_n=\frac{2n\pi}{\sqrt 3}, \quad b_n =\frac{2n\pi +\frac{\pi}{2}}{\sqrt 3},\quad c_n=\frac{2n\pi +\pi}{\sqrt 3},
\quad d_n=\frac{2n\pi +\frac{3\pi}{2}}{\sqrt 3},
\quad e_n=\frac{2n\pi +2\pi}{\sqrt 3}.
$$

To achieve our goal,  we study for $x\geq 0$ the variation of $f(x)=2e^{3x}\cos(\sqrt 3 x)+1$. Note that
$$f'(x)=2e^{3x}\big(3\cos (\sqrt 3 x)-\sqrt 3 \sin(\sqrt 3 x)\big).$$
In particular,  $f'(x)=0\iff \sqrt 3=\tan \sqrt 3 x\iff x=x_n:= \frac{\arctan \sqrt 3 +n\pi }{\sqrt 3}=\frac{\frac{\pi}{3}+n\pi}{\sqrt 3}$, $n\in \N$.
Thus $f$ is increasing on $[\frac{2n\pi}{\sqrt 3}, \frac{\frac{\pi}{3}+2n\pi}{\sqrt 3}]$, decreasing on 
$[ \frac{\frac{\pi}{3}+2n\pi}{\sqrt 3},\frac{\frac{4\pi}{3}+2n\pi}{\sqrt 3}]$ and increasing on
$[\frac{\frac{4\pi}{3}+2n\pi}{\sqrt 3}, \frac{2\pi +2n\pi}{\sqrt 3}]$.

Now 
$$f(a_n)=2e^{3a_n}+1, \quad f(b_n)=1,\quad f(c_n)=-2e^{3c_n}+1<0, \quad f(d_n)=1,\quad f(e_n)=2e^{3e_n}+1.$$ 
Since $ a_n<\frac{\frac{\pi}{3}+2n\pi}{\sqrt 3}<b_n< \frac{\pi+2n\pi}{\sqrt 3}=c_n< \frac{\frac{4\pi}{3}+2n\pi}{\sqrt 3}<
\frac{\frac{3\pi}{2}+2n\pi}{\sqrt 3}=d_n$, we see that there are exactly two zeros on $[a_n,e_n]$, namely one between $b_n$ and $c_n$ and one between $c_n$ and $d_n$  \footnote{Using only the information $f(a_n)>0, f(c_n)<0$ and $f(e_n)>0$  (so without
 studying the variation of $f$), allows us  to conclude that $f$ has at least two zeros in $I_n=[a_n,e_n]$. Using Rouch\'es theorem  and the additional information that in the right half-plane the complex zeros of $\cos(\sqrt 3 z)$, all simple,   are exactly the real zeros $\frac{\frac{\pi}{2}+2n\pi}{\sqrt 3}$ 
 and $\frac{\frac{3\pi}{2}+2n\pi}{\sqrt 3}$, gives us the possibility to conclude that $f$ has no other zeros in the strips 
 $\{z\in \C: {\rm Re}\; z\in I_n\}$ (see Property 6).}.
\\

Let us denote the zeros by  $r_n$, in increasing order ($n\in \N$).  
 Note that $r_n$ is eventually close to the zero 
 $$\ovalbox{$\dis s_n:=\frac{\frac{\pi}{2}+n\pi}{\sqrt 3}$}$$
  of
  $2\cos(\sqrt 3x)$, since $e^{-3x}$ tends rapidly to $0$ as $x\to\infty$, and that for $n\in \N$,
  \begin{equation}\label{hauptungl}
s_{2n}<r_{2n}<r_{2n+1}<s_{2n+1}<s_{2n+2}<r_{2n+2}
\end{equation}
(because the cosinus values at $r_n$ must be negative).
\end{proof}

Here are some numerical examples:\\

$r_0\sim 0.924906\dots$,\quad\quad\;\; $s_0=\frac{\pi}{2 \sqrt 3}\sim 0.9068996\dots$, \quad\; $|s_0-r_0|\sim 0.0180064 <0.125=1/8 $,
   
    $r_1\sim 2.720616677\dots$, \quad  $s_1=\frac{3\pi}{2 \sqrt 3}\sim 2.7206990\dots$,\quad\;\; $|r_1-s_1|\sim 0.00008223\cdots$,
   
    $r_2\sim 4.53449876\dots$,   \quad\; $s_2=\frac{5\pi}{2 \sqrt 3}\sim 4.53449841\dots$, \quad $|r_2-s_2|\sim0.00000035\cdots$.\\
  
$|r_1-r_0|\sim 1.7957107\cdots$, \quad $|r_2-r_1|\sim 1.81388208\cdots$\\

The items \bl{(1)--(4)} below are nice additional  properties of the real zeros of $f$. We do not need these, though.
\\

{\bf Property 4} 

\begin{enumerate}
\item[(1)] \bl{$|r_n-s_n|<\frac{1}{16\sqrt 3}<\frac{1}{8}$ for all $n\in \N$.}
\item[(2)] \bl{$|r_n-s_n|\to 0$.}
\item[(3)] \bl{$|r_n-r_{n-1}|>1$ for all $n\in \N^*$.}
\item[(4)]\bl{$|r_n-r_{n-1}|\to \frac{\pi}{\sqrt 3}\sim 1.813799\cdots $.}
\item[(5)] $r_n^3\sim c n^3$, for some constant $c$.
\end{enumerate}
Noticing that $\{r_{2n},r_{2n+1}\}\ss I_n$, one can easily show  that $|r_{2n+1}-r_{2n}|$ increases to
  $\frac{\pi}{\sqrt 3}$ and $|r_{2n}-r_{2n-1}|$ decreases to $\frac{\pi}{\sqrt 3}$.

\begin{proof}

Recall that $f(z)=2e^{3z} \cos(\sqrt 3 z)+1$. 

(1) For $n=2m$ even, it is sufficient to prove that  $s_n<r_n<s_n+\frac{1}{16\sqrt 3}=:p_n$,
where $s_n,r_n\in I_m$.  Since $f(s_n)=1$ and $f(r_n)=0$, and $f$ is decreasing on $[ \frac{\frac{\pi}{3}+2m\pi}{\sqrt 3},\frac{\frac{4\pi}{3}+2m\pi}{\sqrt 3}]\;\supseteq\; [s_n,p_n]$, this proof is done by showing that $f(p_n)<0$.  Now, noticing that the cosinus term is negative, 

\begin{eqnarray*}
f(p_n)&=& 2e^{\sqrt 3(\frac{\pi}{2}+n\pi+\frac{1}{16})}\cos\left(\frac{\pi}{2}+\frac{1}{16}\right)+1\\
 &\leq& 2e^{\sqrt 3(\frac{\pi}{2}+\frac{1}{16})}\cos\left(\frac{\pi}{2}+\frac{1}{16}\right)+1\sim-1.114587\dots.
 \end{eqnarray*}
 For $n=2m+1$ odd, it is proved in the same way  that $q_n:=s_n-\frac{1}{16\sqrt 3}<r_n< s_n$, where $s_n,r_n\in I_m$.
   In fact, $f$ is increasing on
$[\frac{\frac{4\pi}{3}+2m\pi}{\sqrt 3}, \frac{2\pi +2m\pi}{\sqrt 3}]\;\supseteq\; [q_n, s_n]$.  So,  by noticing that the cosinus term is negative,
\begin{eqnarray*}
f(q_n)&=&2e^{\sqrt 3(\frac{3\pi}{2}+(n-1)\pi-\frac{1}{16})}\cos\left(\frac{3\pi}{2}-\frac{1}{16}\right)+1\\
& \leq&  2e^{\sqrt 3(\frac{3\pi}{2}-\frac{1}{16})}\cos\left(\frac{3\pi}{2}-\frac{1}{16}\right)+1\sim -391.97708\cdots
\end{eqnarray*}

 (2) This works in the same way as above; just replace $1/16$ be an number $\e>0$ arbitrary close to $0$.  Then the cosinus term, still negative, may be very small. But the power $n$ in the exponential factor can be made sufficiently big (depending on $\e$), so that $f(p_n)$, resp. $f(q_n)$ is strictly negative.\\

 (3) Let $n$ be even. Then
 $$|r_n-r_{n-1}|= r_n-r_{n-1}\geq s_n-s_{n-1}= \frac{\pi}{\sqrt 3} \sim 1.813799\dots >1.$$
 If $n$ is odd, then by (1)
 $$|r_n-r_{n-1}|= r_n-r_{n-1}\geq( s_n-\frac{1}{16\sqrt 3})-(s_{n-1}+\frac{1}{16\sqrt 3})=\frac{\pi}{\sqrt 3}-\frac{1}{8\sqrt 3}\sim
 1.7416305\cdots >1.
 $$
 
(4)  Due to (2)
  $$r_n-r_{n-1}= (r_n-s_n)+(s_n-s_{n-1})+(s_{n-1}-r_{n-1}) \to 0+\frac{\pi}{\sqrt 3}+0.$$


(5) Using  (\ref{hauptungl}),  we obtain that
\begin{equation}\label{mainin}
\frac{\frac{\pi}{2}+2n\pi}{\sqrt 3}<r_{2n}<r_{2n+1}<\frac{\frac{3\pi}{2}+2n\pi }{\sqrt 3}.
\end{equation}
Hence $\dis \frac{r_n}{n}\to \frac{\pi}{\sqrt 3}$ and so $r_n^3\sim  c n^3$.

\end{proof}

Let $Z:=\{r_n,  \alpha r_n, \alpha^2 r_n\}$.  Combining the properties (1) and (2), we see that $Z$ is a symmetric
 zero-set for $f$, resp.  $g$. \\

{\bf Property 5}  {\it The elements of $Z$ are simple zeros for $f$ (hence for  $g$)} \footnote{Also note here, that Rouch\'e's theorem (applied as in the proof of Property 6) automatically yields this fact, too. We preferred though a straightforward elementary proof.}.
\\

\begin{proof}

Suppose, contrariwise,  that  for some real $z$,
\begin{equation}\label{(i)}
\quad f(z)=2e^{3z} \cos(\sqrt 3 z)+1=0,
\end{equation}
and   
\begin{equation}\label{(ii)}
 f'(z)=3\cdot 2e^{3z} \cos \sqrt 3 z-2\cdot e^{3z} \sqrt 3 \sin\sqrt 3 z=0.
 \end{equation}
Plugging equality  (\ref{(i)})  into (\ref{(ii)}) yields
\begin{equation}\label{(iii)}
 \quad 0=f'(z)=-3-2\sqrt 3 e^{3z}\sin \sqrt 3 z.
 \end{equation}
Hence
$$1=\sin^2(\sqrt 3 z)+\cos^2(\sqrt 3z)\buildrel=_{(\ref{(i)})}^{(\ref{(iii)})} \frac{3}{4}e^{-6z}+\frac{1}{4}e^{-6z}=e^{-6z}.$$
Therefore $z=0$. But $f(0)=3$, a contradiction. We conclude that, due to Property 2, all elements in $Z$ are simple zeros for $f$.
\end{proof}


{\bf Property 6} {\sl The exact zero-set of $f$ coincides with $Z$.} \\

\begin{proof}

 Numerically it can  be shown that  for $k\in \N^*$
$$\frac{1}{2\pi i}\int_{|z|=2k\pi/\sqrt 3 } \frac{f'}{f} dz=6 k,$$
which yields the assertion, as we already know that  within the annuli 
$$\left\{z\in \C:\frac{2n\pi}{\sqrt 3} \leq |z|\leq   \frac{2(n+1)\pi}{\sqrt 3} \right\}$$
there are 2 real zeros and their rotations by $\alpha$ and $\alpha^2$. \\

For our genuine proof, 
we  first restrict the calculations to the case where ${\rm Re}\, z\geq 0$. The other case will be deduced at the very end of the proof.\\

$\bullet$ We know that $|\cos z|^2= (\cos x)^2 +(\sinh y)^2$, where $z=x+iy$. In fact,
\begin{eqnarray*}
4|\cos z|^2&=& |e^{i(x+iy)}+e^{-i(x+iy)}|^2=|e^{-y}e^{ix}+e^ye^{-ix}|^2\\
&=&e^{-2y}+e^{2y} +2 {\rm Re}( e^{2ix})=2 \cosh (2y)+2\cos (2x)\\
&=&2(\cosh^2 y+\sinh^2 y)+ 2(\cos^2 x-\sin^2 x)\\
&=& 2( 1+2\sinh^2 y)+2(\cos^2 x-(1-\cos^2 x))\\
&=& 4(\sinh^2y+\cos^2 x).
\end{eqnarray*}

$\bullet$  Recall that $g(z)=e^{-3z}+2\cos(\sqrt 3 z)$. We now show that $g\not=0$ on $[0, \infty[\times [1,\infty[$.  In fact,

$$|g(z)|\geq2 |\cos(\sqrt 3z)|-|e^{-3z}| = 2\sqrt{\cos^2 (\sqrt 3 x)+\sinh^2 (\sqrt 3 y)}-e^{-3x}\geq  2\sinh \sqrt 3 -1>0.$$
The same proof also shows that $g\not=0$ on $[0, \infty[\times \,]-\infty ,-1]$.

Next we use Rouch\'e's theorem for the rectangles $R_k:=\left[\frac{2k\pi}{\sqrt 3}, \frac{2(k+1)\pi}{\sqrt 3}\right]\times [-1,1]$, $k\in \N$.
If $x=\frac{2k\pi}{\sqrt 3}$ or  $x=\frac{2(k+1)\pi}{\sqrt 3}$, then 
\begin{eqnarray*}
|g(z)- 2\cos(\sqrt 3 z)|&=& e^{-3x}\leq 1<2  |\cos \sqrt 3 iy|=2(\sqrt {1+ \sinh^2(\sqrt 3 y)}=2 |\cos(\sqrt 3 z)| \\
&\leq& |g(z)|+2 |\cos(\sqrt 3 z)|.
\end{eqnarray*}

If $|y|=1$, then 
\begin{eqnarray*}
|g(z)- 2\cos(\sqrt 3 z)|&=& e^{-3x}\leq 1< 2\sinh \sqrt 3\leq 2\sqrt{\cos^2 ({\sqrt 3}x)+\sinh^2 (\sqrt 3})\leq 2|\cos(\sqrt 3 z)|\\
&\leq&|g(z)|+2|\cos(\sqrt 3 z)|.
\end{eqnarray*}
By Rouch\'e's Theorem \cite[p. 852]{moru}, $g$ and  $\cos(\sqrt 3z)$ have the same number of zeros on $R_k$, namely 2. 
As we already know that $g$ has two real zeros in these intervals $\left[\frac{2k\pi}{\sqrt 3}, \frac{2(k+1)\pi}{\sqrt 3}\right]$, we are done.
\\

$\bullet$  Having only the real zeros $r_k$ of $f$ (equivalently $g$)  in the right-half plane, the symmetry $f(z)=0\iff f(e^{i2\pi/3}z)=0$ now implies that no other zeros are in the left half-plane excepted the rotations $e^{i2\pi/3} r_k$ and $e^{i4\pi/3} r_k$ as any
  $\zeta:=re^{i\theta}$ with
  $\pi/2\leq \theta\leq3\pi/2$ would yield that  at least one of the points $e^{i4\pi/3}\zeta $ or $e^{i2\pi/3}\zeta$ belongs 
  to the right-half plane. 
  This finishes the proof of Property 6.
  \end{proof}

We are finally ready to derive the value 
for  the desired sum $S$.
The main idea is to apply the residue theorem for the function
$$F(z):= \frac{1}{z^6}\;\frac{f'(z)}{f(z)},$$
 and the formula ${\rm Res} \left(h \frac{f'}{f}, p\right)= h(p)$, where $p$ is a simple pole of $f$ and $h$ is holomorphic in a neighborhood of $p$.

Observe that \footnote{ Obtained e.g. with wolframalpha.}
$$\frac{f'(z)}{f(z)}= 2-4z^2-\frac{24}{5}z^5-\frac{212}{35}z^8-\frac{14736}{1925}z^{11}-\frac{1694832}{175175}z^{14}-\cdots$$
Thus, 

$${\rm Res} \left(\frac{1}{z^6}\frac{f'}{f},z=0\right) =-\frac{24}{5}.$$
Note that  the zeros $z_n=\alpha r_n$, resp.  $z_n=\alpha^2 r_n$ have the property that  $z_n^6=r_n$. 

Recall from (\ref{mainin}) that  for $n\in \N$ the numbers $r_{2n}, r_{2n+1}$ belong to 
$$\left[\frac{2n\pi+\frac{\pi}{2}}{\sqrt 3},  \frac{2n\pi+\frac{3\pi}{2}}{\sqrt 3}\right].$$ 
For $N\in \N$, let  $C_N$  be the circle centered at the origin and with radius $R_N:=\frac{2\pi (N+1)}{\sqrt 3 }$.   Then the associated disk contains 
$r_0,r_1,\dots, r_{2N+1}$ and

\begin{enumerate}
\item[(1)] $R_N>N$,
\item[(2)] $R_N^3-r_{2N+1}^3\geq 1$.
\end{enumerate} 
Note that (2) is a consequence to $R_N-r_{2N+1}\geq \pi/2$.

Due to the residue theorem 
\begin{eqnarray*}
L:=\lim_{N\to\infty}\frac{1}{2\pi i }\int_{C_N}  \frac{1}{z^6}\frac{f'(z)}{f(z)} dz&=&{\rm Res} \left( \frac{1}{z^6}\frac{f'}{f},0\right) +
\sum_{\xi\in Z} {\rm Res} \left( \frac{1}{z^6}\frac{f'}{f}, \xi\right)\\
&=&-\frac{24}{5}+ 3 \sum_n \frac{1}{r_n^6}.
\end{eqnarray*}

We claim that $L=0$, from which we deduce that 
$$S=\sum_n \frac{1}{r_n^6}=\frac{8}{5}.$$

 To prove this claim,  we use that by Property (4), $\sum \frac{1}{r_n^3}$ converges. Hence the infinite product
$$p(z):=\prod_{n=0}^\infty\left(1- \left(\frac{z}{r_n}\right)^3\right)$$
converges locally uniformly on $\C$. According to Weierstrass's factorization theorem,
$f=pq$, where $q$ is  a zero-free entire  function.
Note that
$$\frac{f'}{f}=\frac{p'(z)}{p(z)}+ \frac{q'(z)}{q(z)}$$ and that
$$\int_{C_N}  \frac{1}{z^6}\frac{f'(z)}{f(z)} dz= \int_{C_N}  \frac{1}{z^6}\frac{p'(z)}{p(z)} dz +0.$$
 
Now

$$\frac{p'(z)}{p(z)}=-3\sum_{n=0}^\infty \frac{z^2}{r_n^3-z^3},$$
which is a locally uniformly convergent series on $\C\setminus Z$ (note that $r_n^3\sim c\cdot n^3$).
 Hence $\int\sum=\sum\int$ and so,
by Cauchy's integral theorem,
$$2\pi i J_N:=\int_{C_N}\frac{1}{z^6}\frac{p'(z)}{p(z)}dz =-3\sum_{n=0}^{2N+1} \int_{C_N}\frac{z^{-4}}{r_n^3-z^3} dz.$$
Now 
$$\left| \frac{z^{-4}}{r_n^3-z^3}\right| \leq \frac{R_N^{-4}}{R_N^3-r_n^3}\leq R_N^{-4}.$$
Thus
$$0\leq 2\pi J_N\leq 3 (2N+2) R_N^{-4} 2\pi R_N\leq C\cdot N^{-2}.$$
We conclude that $L=\lim_{N\to\infty} J_N=0$. \hfill $\square$ \\

  \bigskip

Related problems are given in \cite[p. 279] {ray} and \cite{LJ} ($\cos x \cosh x+1=0$) and  \cite{li} ($\tan x=kx$) \footnote{ Communicated to us by A. Sasane.} .

{\sl \underline{Second, but more complicated proof of Property 6}}\\

Note that $g(z)=e^{-3z}+2\cos(\sqrt 3z)=0$ implies that 
$$4|\cos^2 (\sqrt 3 z)|= e^{-6x},$$ 
and so
$$
4 \sinh^2 (\sqrt 3 y)= e^{-6x} -4\cos^2 (\sqrt 3 x).
$$
In particular 
\begin{equation}\label{sinhcos}
4\cos^2 (\sqrt 3 x)\leq e^{-6x}.
\end{equation}

  If $k\in \N$, equality in (\ref{sinhcos}) holds if $x=r_k$ and if
$x=\frac{\frac{\pi}{2} +k \pi}{\sqrt 3}=s_k$,  then trivially $0=4\cos^2 (\sqrt 3 x)\leq e^{-6x}$.

In the figure below, we display one "branch" of the curves $y_j$, given by 
\begin{equation} \label{winzig}
\mbox{$\dis y_2= \frac{1}{\sqrt 3}{\rm arsinh}\sqrt{\frac{1}{4}e^{-6x} -\cos^2 (\sqrt 3 x)}$ and $y_1=2\cos(\sqrt 3 x)+e^{-3x}$}.
\end{equation}
 They are very tiny, as the coordinates show. To recapitulate, if $z$ is a zero of $g$ with non-negative real part, then it belongs to these tiny arcs.  
 
 \begin{figure}[h!]
 
  \scalebox{0.45} 
  {\includegraphics{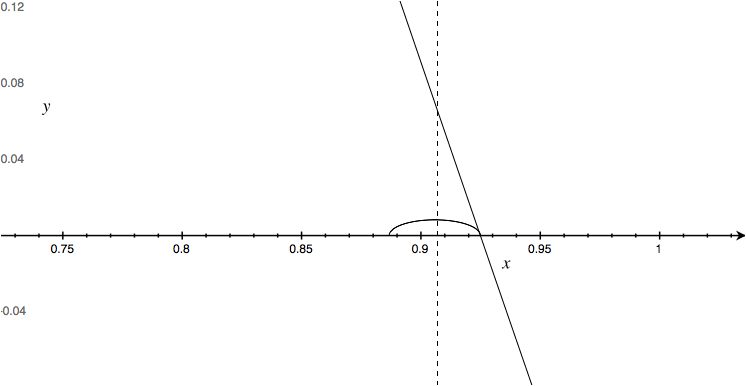}} 
  \caption {The curves $y_2= \frac{1}{\sqrt 3}{\rm arsinh}\sqrt{\frac{1}{4}e^{-6x} -\cos^2 (\sqrt 3 x)}$ and $y_1=2\cos(\sqrt 3 x)+e^{-3x}$ and $s_0=\frac{\pi}{2\sqrt 3}$}
  \end{figure}
  
 $\bullet$  To show that in the right half-plane $g$ has only the zeros $r_k$, it remains to study the behaviour of $g(z)=2\cos(\sqrt 3z)+e^{-3z}$  on  the disks $D(r_k,1/8)$. A major difficulty will be to show  that the union of these disks contains the graph of the curve $y_2$. A tool will be the following result from complex analysis:
  
\begin{criterium}\cite[p. 365]{moru}\label{holi}
Let $\Phi$ be bounded by $1$ and holomorphic in $\D$ and suppose that $\Phi(0)=0$ as well as $|\Phi'(0)|\geq \delta> 0$. 
Then $\Phi$ has no zeros on $\{0<|z|<\delta\}$.
\end{criterium}

  We first note that for $z=x+iy$ with $x\geq 0$ and $|y|\leq 1$, $g$ is bounded by $8$ :
 \begin{eqnarray*}
|g(z)|&\leq& 2|\cos(\sqrt 3 z)|+ e^{-3x}= 2 (\sqrt{\cos^2 (\sqrt 3 x)+\sinh^2 (\sqrt 3 y)})+1\leq 2 (\sqrt{1+ \sinh^2\sqrt 3})+1\\
&=& 2 \cosh(\sqrt 3) +1\leq  e^{\sqrt 3}+2\leq 8.
 \end{eqnarray*} 
 Moreover  $|g'(r_k)|\geq 1 $, since
 $$g'(z)=-2\sqrt 3 \sin(\sqrt 3 z)-3e^{-3z},$$
 and so
\begin{eqnarray*}
|g'(r_k)|&=&|-2\sqrt 3 \sin(\sqrt 3 r_k)+ 6 \cos(\sqrt 3 r_k)|\to  2 \sqrt 3\sim 3.4641\cdots.
\end{eqnarray*}
respectively 
$$|g'(r_k)|=|(-1)^{k}\sqrt 3\sqrt{4-e^{-6r_k}}-3 e^{-3r_k}|$$

   \begin{figure}[h!]
 
  \scalebox{0.45} 
  {\includegraphics{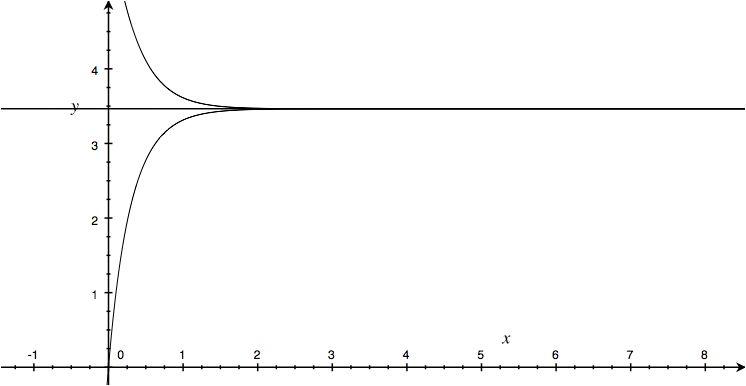}} 
  \caption {The curves $y=|\pm \sqrt 3\sqrt{4-e^{-6x}}-3 e^{-3x}|$}
  \end{figure}
 
 The value of the lower branch at $x=1/6$ is $\sim 1.481371\cdots$ and the branch is increasing.  So $|g'(r_k)|\geq 1$ for $r$ even.
 The upper branch is always bigger than $3$. So, $|g'(r_k)|\geq 1$, too  if  $r$ is odd. 
 \medskip

 We conclude from the zero set criterium \ref{holi} that $g\not=0$ on $D(r_k, 1/8)$, excepted at $r_k$.\\
 
 $\bullet$ The final part which remains is now the proof that the graph of the (different branches) of the curve $y_2$ is contained in the union of the disks $D(r_k, 1/8)$.
 
 We first determine the positive zeros of $y_2$. These are given by $e^{-6x} -4\cos^2 (\sqrt 3 x)=0$. Of course the $r_k$, 
the positive zeros of $2\cos (\sqrt 3 x)- e^{-3x}$ are zeros of $y_2$, too. The other ones are given by the zeros of the function
$$h^*(x)= 2e^{3x}\cos(\sqrt 3 x)-1=h(x)-2.$$
This works as in the proof oh Property 3; note that $(h^*)'=h'$.

$$h^*(a_n)=2e^{3a_n}-1 >0, \quad h(b_n)=-1,\quad h(c_n)=-2e^{3c_n}-1<0, \quad h(d_n)=-1,\quad h(e_n)=2e^{3e_n}-1>0.$$ 
Since $ a_n<\frac{\frac{\pi}{3}+2n\pi}{\sqrt 3}<b_n< \frac{\pi+2n\pi}{\sqrt 3}=c_n< \frac{\frac{4\pi}{3}+2n\pi}{\sqrt 3}<
\frac{\frac{3\pi}{2}+2n\pi}{\sqrt 3}=d_n$, and $h^*$ is increasing on $[\frac{2n\pi}{\sqrt 3}, \frac{\frac{\pi}{3}+2n\pi}{\sqrt 3}]$, decreasing on 
$[ \frac{\frac{\pi}{3}+2n\pi}{\sqrt 3},\frac{\frac{4\pi}{3}+2n\pi}{\sqrt 3}]$ and increasing on
$[\frac{\frac{4\pi}{3}+2n\pi}{\sqrt 3}, \frac{2\pi +2n\pi}{\sqrt 3}]$,
we see that there are exactly two zeros on $[a_n,e_n]$, namely one between 
$\frac{\frac{\pi}{3}+2n\pi}{\sqrt 3}$ and $b_n=\frac{\frac{\pi}{2}+2n\pi}{\sqrt 3}$ and one between 
$d_n=\frac{\frac{3\pi}{2}+2n\pi}{\sqrt 3}$
and $e_n=\frac{2\pi+2n\pi}{\sqrt 3}$.  We enumerate these in an increasing order, say $r_n^*$, and
$$ \mbox{$r^*_{2n}< s_{2n} < r_{2n}$ as well as $r_{2n+1}<s_{2n+1}<r^*_{2n+1}.$}$$
As in the proof of Property 5 (1) we see that
$$|r_n^*- s_n|< \frac{1}{16\sqrt 3 }.$$
In fact,

 For $n=2m$ even, it is sufficient to prove that  $p_n^*:=s_n-\frac{1}{16\sqrt 3}<r^*_n<s_n$,
where $s_n,r^*_n\in I_m=\left[\frac{2\pi m}{\sqrt 3},\frac{2\pi m+2\pi}{\sqrt 3}\right]$.  Since $h(s_n)=-1$ and $h(r^*_n)=0$, and $h$ is decreasing on $[ \frac{\frac{\pi}{3}+2m\pi}{\sqrt 3},\frac{\frac{4\pi}{3}+2m\pi}{\sqrt 3}]\;\supseteq\; [p_n^*,s_n]$, this is done by showing that $h^*(p_n)>0$.  Now, noticing that the cosinus term is positive, 
\begin{eqnarray*}
h^*(p_n^*)&=& 2e^{\sqrt 3(\frac{\pi}{2}+n\pi-\frac{1}{16})}\cos\left(\frac{\pi}{2}-\frac{1}{16}\right)-1\\
 &\geq& 2e^{\sqrt 3(\frac{\pi}{2}-\frac{1}{16})}\cos\left(\frac{\pi}{2}-\frac{1}{16}\right)-1\sim 0.7029349\dots >0.
 \end{eqnarray*}
 For $n=2m+1$ odd, it is proved in the same way  that $s_n< r_n^*<q_n^*:=s_n+\frac{1}{16\sqrt 3}$, where $s_n,r^*_n\in I_m$.
   In fact, $h^*$ is increasing on
$[\frac{\frac{4\pi}{3}+2m\pi}{\sqrt 3}, \frac{2\pi +2m\pi}{\sqrt 3}]\;\supseteq\; [s_n,q_n^* ]$. 
 So,  by noticing that the cosinus term is positive,
\begin{eqnarray*}
h^*(q^*_n)&=&2e^{\sqrt 3(\frac{3\pi}{2}+(n-1)\pi+\frac{1}{16})}\cos\left(\frac{3\pi}{2}+\frac{1}{16}\right)-1\\
& \geq&  2e^{\sqrt 3(\frac{3\pi}{2}+\frac{1}{16})}\cos\left(\frac{3\pi}{2}+\frac{1}{16}\right)-1\sim 486.971814\cdots
\end{eqnarray*}

Next we estimate  the local maxima of $y_2$. Since $r^*_n\geq r_0^*> \pi/3$, 
$$y_2\leq \frac{1}{\sqrt 3} {\rm arsinh} \sqrt{ \frac{1}{4} e^{-6 \pi/3}}=\frac{1}{\sqrt 3} {\rm arsinh}  \frac{e^{-\pi}}{2}
\sim 0.01247381\cdots<0.02:=\rho
 $$
 
 Now the rectangle 
 $$R:=\left[r_0-\frac{1}{8\sqrt 3}, r_0+\frac{1}{8\sqrt 3 }\right]\times [0, \rho]\ss D(r_0, \frac{1}{8}),$$
 because the vertex $(\frac{1}{8\sqrt 3} , \rho)$ satisfies 
 $$\left(\frac{1}{8\sqrt 3}\right)^2 + \rho^2\sim 0.00568\cdots< 0.01562=\frac{1}{64}.$$
It is easily shown that $\{(x,y_2(x)): r_{2n}^*\leq x\leq r_{2n}\}$ and  $\{(x,y_2(x)): r_{2n+1}\leq x\leq r_{2n+1}^*\}$ are contained in the rectangle $R$.  \hfill $\square$

$\bullet$  Having only the real zeros $r_k$ of $f$ in the right-half plane, the symmetry $f(z)=0\iff f(e^{i2\pi/3}z)=0$ now implies that no other zeros are in the left half-plane excepted the rotations $e^{i2\pi/3} r_k$ and $e^{i4\pi/3} r_k$ as any
  $\zeta:=re^{i\theta}$ with
  $\pi/2\leq \theta\leq3\pi/2$ would yield that  at least one of the points $e^{i4\pi/3}\zeta $ or $e^{i2\pi/3}\zeta$ belongs 
  to the right-half plane. 
  This finishes the proof of Property 6.\hfill $\square$
  
A detailed analysis  of the zeros of the functions $e^{az}+e^{-az}+1$ appears in \cite{mpr}. That paper is based on the preceding methods and text.

\newpage

\begin{figure}[h!]
 
  \scalebox{0.45} 
  {\includegraphics{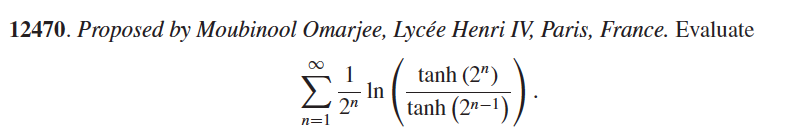}} 
  
 \end{figure}
 

\centerline{{\bf Solution to problem 12470  in Amer. Math. Monthly 131 (2024), p. 536}}\medskip 

\centerline{Raymond Mortini and Rudolf Rupp }

\bigskip

\centerline{- - - - - - - - - - - - - - - - - - - - - - - - - - - - - - - - - - - - - - - - - - - - - - - - - - - - - -}
  
  \medskip

  We show that
  $$\ovalbox{$\dis S:=\sum_{n=1}^\infty\frac{1}{2^n} \log\left(\frac{\tanh 2^n}{\tanh 2^{n-1}}\right)=\log (e^2+1)-2\sim 0.1269280110\cdots$}.$$
  \bigskip
  
  One has to transform this into a telescoping series. 
  
  \begin{eqnarray*}
  \frac{\tanh 2^n}{\tanh 2^{n-1}}&=& \frac{\sinh 2^n}{\sinh 2^{n-1}}\; \frac{\cosh 2^{n-1}}{\cosh 2^n}=
  \frac{2\; \sinh 2^{n-1}\cosh 2^{n-1}}{\sinh 2^{n-1}}\; \frac{\cosh 2^{n-1}}{\cosh 2^n}\\
  &=&2 \frac{(\cosh 2^{n-1})^2}{\cosh 2^n}.
\end{eqnarray*}

  Hence
  \begin{eqnarray*}
\frac{1}{2^n} \log\left(\frac{\tanh 2^n}{\tanh 2^{n-1}}\right)&=&\frac{1}{2^n}\Big( \log 2 +2\log(\cosh 2^{n-1})-\log\cosh 2^n\Big)\\
&=&\frac{ \log 2}{2^n} + \frac{1}{2^{n-1}}\log(\cosh 2^{n-1})-\frac{1}{2^n}\log(\cosh 2^{n}).
 \end{eqnarray*} 
    Note that $\e_n:=\frac{1}{2^n}\log(\cosh 2^{n})\co{\to 1}$ since (by using l'Hospital's rule).
    $$ \lim_{x\to\infty}\frac{\log( e^x+e^{-x})}{x}=1.
    $$
    
 Consequently the series below converges and 
 \begin{eqnarray*}
S&=&\sum_{n=1}^\infty \frac{\log 2}{2^n} +\sum_{n=1}^\infty  \left(\frac{1}{2^{n-1}}\log(\cosh 2^{n-1})-\frac{1}{2^n}\log(\cosh 2^{n})\right)\\
&=&1\cdot \log 2  +\log\cosh 1-\lim_n \e_n=\log 2+\log\left(\frac{e^1+e^{-1}}{2}\right)-1=\log (e^2+1)-2.
\end{eqnarray*}

\newpage

\begin{figure}[h!]
 
  \scalebox{0.45} 
  {\includegraphics{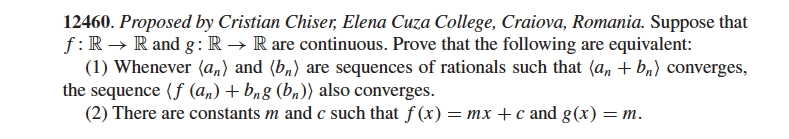}} 
  
 \end{figure}
 

\centerline{{\bf Solution to problem 12460  in Amer. Math. Monthly 131 (2024), 354}} \medskip 

\centerline{Raymond Mortini and Rudolf Rupp }

\bigskip

\centerline{- - - - - - - - - - - - - - - - - - - - - - - - - - - - - - - - - - - - - - - - - - - - - - - - - - - - - -}
  
  \medskip

  (2) $\imp$ (1): If $f(x)=mx+c$ and $g(x)=m$, then trivially 
  $$S(n):=f(a_n)+b_ng(b_n)=ma_n+c+b_n m= m(a_n+b_n)+c$$
  and so $S(n)$ converges whenever $a_n+b_n$ converges.
  
  (1)$\imp$ (2):  Let $\N:=\{0,1,2,\dots\}$. Fix $a, q\in \Q$. 
  Consider the sequences 
    $$\mbox{$(a_n)_{n\in \N}=(a+q , a, \co{a+q, a},a+q, a,\dots)$ \quad
    and \quad $(b_n)_{n\in \N}=(-q,0,\co{-q,0},\dots)$}.$$
   
Then $a_n+b_n=a$ for all $n$. Moreover
    $$s_n:=f(a_n)+b_n g(b_n)=\begin{cases} f(a+q)-q g(-q)& \text{if $n$ is even}\\
    f(a)&\text{if $n$ is odd}.
    \end{cases}.
    $$
    As by assumption $(s_n)$ converges, we deduce that 
    \begin{equation}\label{assu1}
    f(a+q)-q g(-q)=f(a)
\end{equation}
  
  Next consider the sequences $(a_n)_{n\in \N}=(q,0,q,0,\dots)$ and $(b_n)_{n\in \N}=(-q,0,-q,0, \dots)$.
  Then $a_n+b_n=0$ and
  $$r_n:=f(a_n)+b_n g(b_n)=\begin{cases} f(q)-q g(-q)& \text{if $n$ is even}\\
    f(0)&\text{if $n$ is odd}.\end{cases}$$
  Since by assumption also $(r_n)$ converges, we have
  \begin{equation}\label{assu2}
  f(q)-q g(-q)=f(0).
\end{equation}
  Now (\ref{assu1})--(\ref{assu2}) yields that for every $q\in \Q$ and $a\in \Q$
  $$
  f(a+q)-f(q)=f(a)-f(0).
  $$
  Since $f$ is assumed to be continuous,  $f(a+x)-f(a)= f(x)-f(0)$ for every $x\in \R$. This implies that $f$ is an affine function. In fact,
  let $h(x)=f(x)-f(0)$. Then $h(a+x)=h(x)+h(a)$, that is, $h$ is a continuous additive function. By a classical result due to Cauchy, $h$ is linear; that is $h(x)=mx$ for some $m\in \R$. Consequently $f(x)=mx+f(0)$.  Let $c:=f(0)$.
  Then the condition on $f$ and $g$ has the form
  $$S_n=f(a_n)+b_ng(b_n)=m a_n+c +b_n g(b_n)= c+ m(a_n+b_n)+ b_n(g(b_n)-m).$$

 Let $p,q\in \Q$ and consider the sequences 
 $(b_n)_{n\in \N}=(p,q,p,q,\dots)$ and $(a_n)_{n\in \N}=(-p,-q,-p,-q,\dots)$.
 Then  $a_n+b_n=0$ for all $n$ and, by assumption,
 $$S_n=\begin{cases}
 c+p(g(p)-m)&\text{if $n$ is even}\\ c+q(g(q)-m)&\text{if $n$ is odd}
 \end{cases}
 $$  
 converges. Hence the function $x (g(x)-m)$ must be constant on $\Q$, hence on $\R$ (due to continuity).  Consequently $g(x)=m$ for every $x\in \R$.

\newpage

  \begin{figure}[h!]
 
  \scalebox{0.45} 
  {\includegraphics{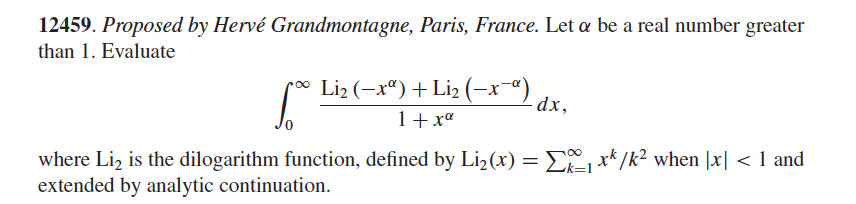}} 
  
 \end{figure}
 

\centerline{{\bf Solution to problem 12459  in Amer. Math. Monthly 131 (2024), 354}}\medskip 

\centerline{Raymond Mortini and Rudolf Rupp }

\bigskip

\centerline{- - - - - - - - - - - - - - - - - - - - - - - - - - - - - - - - - - - - - - - - - - - - - - - - - - - - - -}
  
  \medskip

We prove that for $a>1$
 $$\ovalbox{$\dis
 I(a):=\int_0^\infty\frac{{\rm Li}_2(-x^a)+{\rm Li}_2(-x^{-a})}{1+x^a}\;dx =
 \frac{\pi^3}{3a}  \left( \frac{\sin^2(\pi/a)-3}{\sin^3 (\pi/a)}\right).
$ }
 $$
 
 \medskip

 To start with, we use the known formula \cite{max}
 $${\rm Li}_2(z)+{\rm Li}_2(\frac{1}{z})=-\frac{\pi^2}{6}-\frac{1}{2}\log^2(-z),\quad z\in \C\setminus [0,\infty[, $$
 for $z=-x^a$. So the integral to be computed is
 $$I(a)=-\frac{\pi^2}{6}\int_0^\infty \frac{1}{1+x^a}dx - \frac{1}{2}\int_0^\infty \frac{\log^2 (x^a)}{1+x^a}\;dx.$$

 The change of variable $x^a\mapsto e^{-t}$ now yields
 $$I(a)=-\frac{\pi^2}{6a} \int_{-\infty} ^\infty \frac{e^{-{\frac{t}{a}}}}{1+e^{-t}}dt -
  \frac{1}{2a}\int_{-\infty} ^\infty  \frac{t^2e^{-{\frac{t}{a}}}}{1+e^{-t}}dt.
  $$
 We solve this with the help of the residue theorem.  So, for $m=0$ or $m=2$,  let 
 $$f_m(z):= z^m \frac{e^{-{\frac{z}{a}}}}{1+e^{-z}}.$$
 In order the obtained   series converge and the path-integrals tend to $0$ when "blowing up" the contours,  we consider for $0<r<1$ the auxiliary
 functions 
 $$u_r(z):=r^{-iz/2\pi}:=e^{-i \frac{\log r}{2\pi} z}$$
(which converge locally uniformly to $1$ as $r\to 1$) and
 $$F_{m,r}(z):= f_m(z) u_r(z),$$
 and calculate the integral
\begin{equation}\label{paraint}
J_r(a):= -\frac{\pi^2}{6a} \int_{-\infty} ^\infty \frac{e^{-\frac{t}{a}} r^{-i t/2\pi}}{1+e^{-t}}dt -
  \frac{1}{2a}\int_{-\infty} ^\infty  \frac{t^2e^{-\frac{t}{a}}r^{-i t/2\pi}}{1+e^{-t}}dt.
\end{equation}
 As $|u_r(t)|\leq 1$,  we deduce from Lebesgue's dominated convergence theorem that 
 $$\lim_{r\to 1} J_r(a)=I(a).$$
 
 Note that  $F_{m,r}$ is meromorphic in $\C$ with simple poles at $z_n=i\pi(1 + 2n)$ for $n\in\Z$.  We integrate $F_{m,r}$ over  the  
 positively oriented boundary $\Gamma_N=\gamma_1+\gamma_2+\gamma_3$ of the rectangles 
 $[-2N\pi, 2N\pi]\times[0, 2N\pi]$, where $N\in \N^*$.  Let $s_N:=2\pi N$. Then
 
 $$\begin{matrix}
\gamma_1(t)&=&s_N (1+it),&  0\leq t\leq 1,\\ 
 \gamma^{[-1]}_2(t)&=& s_N(t+i ),& -1\leq t\leq 1,\\
  \gamma_3^{[-1]}(t)&=&s_N(-1+ it),&  0\leq t\leq 1.
 \end{matrix}
 $$
 
By the residue theorem
 $$\int_{\Gamma_N} F_{m,r}(z)dz=2\pi i \sum_{n=0}^{N-1}{\rm Res}\;(F_{m,r}, z_n).$$
 Now $F_{m,r}=g/h$ and so 
 $${\rm Res}\;(F_{m,r}, z_n)= \frac{g(z_n)}{h'(z_n)}.
 $$
 Moreover, $\dis\int_{\Gamma_N} F_{m,r}(z) dz\to 0$ as $N\to\infty$. To see this we have to consider three cases:
 
 Since on $\gamma_j$  for $j=1,3$,  we have $|u_r(\gamma_j(t))|\leq e^{t N\log r}\leq 1$, we see that
\begin{eqnarray*}
\left|\int_{\gamma_1} F_{m,r}(z)dz\right| &\leq &\int_0^1 \frac{|s_N(1+it)|^m \;|e^{-s_N(1+it)/a}|\,|u_r(s_N(1+it))|}
{|1+e^{-s_N(1+it)}|} s_N dt\\
&\leq& \frac{C N^{m+1} e^{-s_N/a}}{1-e^{-s_N}}\to 0\;{\rm as}\; N\to\infty.
\end{eqnarray*}

\begin{eqnarray*}
\left|\int_{\gamma_3} F_{m,r}(z)dz\right| &\leq &\int_0^1 \frac{|s_N(-1+it)|^m \;|e^{-s_N(-1+it)/a}| \,|u_r(s_N(-1+it))|
}{|1+e^{-s_N(-1+it)}|} s_N dt\\
&\leq& \frac{C N^{m+1} e^{s_N/a}}{e^{s_N}-1} \cdot \co{\frac{e^{-s_N}}{e^{-s_N}}}   \\
&=&\frac{CN^{m+1} e^{-s_N\left(1-\frac{1}{a}\right)}}{1-e^{-s_N}}\to 0\;{\rm as}\; N\to\infty.
\end{eqnarray*}

Next we observe that on $\gamma_2$ we have
$$|u(\gamma_2(t))|=\left|e^{-i N \log r( t+i)}\right|=e^{N\log r},
$$
and that 
$$N^{m+1} e^{N\log r}= e^{(m+1)\log N  +N\log r}\to 0 \;{\rm as}\; N\to\infty.$$
Hence
\begin{eqnarray*}
\left|\int_{\gamma_2} F_{m,r}(z)dz\right| &\leq &\int_{-1}^1 \frac{|s_N(i+t)|^m |e^{-s_N (t+i)/a}| \,|u_r(s_N(i+t))|
} {|1+ e^{-s_N(t+i)}|} s_Ndt\\
&\leq&C N^{m+1} e^{N\log r} \int_{-1}^1  \frac{e^{-s_N t/a}}{1+e^{-t s_N}}dt\\
&\leq&CN^{m+1} e^{N\log r}\left(\int_0^1  \underbrace{e^{-s_Nt/a}}_{\leq 1} dt +\int_{-1}^0 \underbrace{e^{s_Nt\left(1-\frac{1}{a}\right)}}_{\leq 1} dt\right)\\
&&\to 0\cdot  0\;{\rm as}\; N\to\infty,
\end{eqnarray*}
where the property $\lim\int=\int\lim$ is used (Lebesgue's dominated convergence theorem: the integrands are bounded
 (in moduli) by $1$ and converge  to $0$ on the associated open intervals).\\

By letting $N\to\infty$, we conclude that

\begin{eqnarray*}
J_r(a)=2\pi i  \left( -\frac{\pi^2}{6a} \, \sum_{n=0}^\infty{\rm Res}\;(F_{0,r}, z_n) 
  -\frac{1}{2a}  \,\sum_{n=0}^\infty{\rm Res}\;(F_{2,r}, z_n)       \right).
\end{eqnarray*}

That the series converge will be clear  in a moment. To this end, we need to calculate the residua. Note that
$$u_r(z_n)=e^{-i\frac{\log r}{2\pi} i\pi(1+2n)}=  (\sqrt r)^{1+2n}.
$$
Hence
\begin{eqnarray*}
{\rm Res}\;(F_{m,r}, z_n) &=&z_n^m \;\frac{e^{-z_n/a}}{-e^{-z_n}} u_r(z_n)= -(i\pi(1+2n))^m\frac{e^{-i\pi(1+2n)/a}}{e^{-i\pi(1+2n)}}  u_r(z_n)\\
&=& (\sqrt r)^{1+2n}\;(i\pi(1+2n))^m e^{-i\pi(1+2n)/a}. 
\end{eqnarray*}

Let $\zeta:=e^{-i\pi/a}$. Then
\begin{eqnarray*}
J_r(a)
&=&2\pi i  \left( -\frac{\pi^2}{6a}   \sum_{n=0}^\infty ( \sqrt  r\zeta)^{2n+1}  
-\frac{1}{2a}  \,\sum_{n=0}^\infty (i\pi(1+2n))^2 (\sqrt r\zeta)^{2n+1} \right)\\
&=&2\pi i  \left( -\frac{\pi^2}{6a}   \sum_{n=0}^\infty  (\sqrt r \zeta)^{2n+1}  
+\frac{\pi^2}{2a}  \,\sum_{n=0}^\infty (2n+1)^2 (\sqrt r\zeta)^{2n+1}\right).
\end{eqnarray*}

 It is straightforward to check  the following result:
 
 \begin{equation}\label{powersu}
 \sum_{n=0}^\infty (2n+1)^2 z^{2n+1} =z\frac{z^4+6z^2+1}{(1-z^2)^3}=:S(z).
\end{equation}

Hence 
$$J_r(a)=2\pi i\left( -\frac{\pi^2}{6a} \frac{\sqrt r \zeta}{1-r \zeta^2} +\frac{\pi^2}{2a} S(\sqrt r \zeta)\right),$$
and so 
$$I=\lim_{r\to 1} J_r(a)=2\pi i\left( -\frac{\pi^2}{6a} \frac{ \zeta}{1- \zeta^2} +\frac{\pi^2}{2a} S( \zeta)\right).
$$
  Note that $\zeta\; \ov{\zeta}=1$, but $\zeta\not=1$. A short calculation yields
 $$S(\zeta)= \frac{\zeta^2+6+\ov{\zeta}^2}{(\ov{\zeta}-\zeta)^3}.
 $$
 Consequently, by using that $\zeta^2+6+\ov{\zeta}^2= (\zeta-\ov{\zeta})^2 +8$,
 \begin{eqnarray*}
I&=&2\pi i  \left( -\frac{\pi^2}{6a}  \frac{\zeta}{1-\zeta^2}+\frac{\pi^2}{2a}  \,\frac{\zeta^2+6+\ov{\zeta}^2}{(\ov{\zeta}-\zeta)^3}\right)\\
&=&2\pi i  \left( -\frac{\pi^2}{6a}  \frac{1}{\ov{\zeta}-\zeta}+\frac{\pi^2}{2a}  \,\frac{(\zeta-\ov{\zeta})^2+8}{(\ov{\zeta}-\zeta)^3}\right)\\
&=&2\pi i  \left( -\frac{\pi^2}{6a}  \frac{(\ov{\zeta}-\zeta)^2}{(\ov{\zeta}-\zeta)^3}+\frac{\pi^2}{2a}  
\,\frac{(\zeta-\ov{\zeta})^2+8}{(\ov{\zeta}-\zeta)^3}\right)\\
&=&-\frac{\pi^3}{3a} \left(i \; \frac{(-4)\sin^2(\pi/a)}{(2i)^3\sin^3 (\pi/a)} -i\;\frac{(-12) \sin^2(\pi/a) +24}{(2i)^3\sin^3 (\pi/a)}\right)\\
&=&-\frac{\pi^3}{3a}  \left( \frac{3-\sin^2(\pi/a)}{\sin^3 (\pi/a)}\right).
\end{eqnarray*}

{\bf Remark} A more classical way is to compute
$$I(a)=-\frac{\pi^2}{6}\int_0^\infty \frac{1}{1+x^a}dx - \frac{1}{2}\int_0^\infty \frac{\log^2 (x^a)}{1+x^a}\;dx$$
"directly" without the change of variable by applying the residue theorem to the functions
$\frac{1}{1+z^a}$, $\frac{\log z}{1+z^a}$ and $\frac{\log^2 z}{1+z^a}$ for the standard branches of the power and logarithm and the boundary of the sectors
$$\{z\in \C: |z|<R,  0\leq \arg z\leq \frac{2\pi}{a}\},$$
which contains one simple pole.

\newpage

   \begin{figure}[h!]
 
  \scalebox{0.5} 
  {\includegraphics{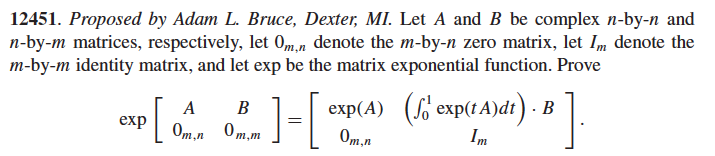}} 
  
 \end{figure}
 

\centerline{\bf Solution to problem 12451  in Amer. Math. Monthly 131 (2024)} \medskip 

\centerline{Raymond Mortini and Rudolf Rupp }

\bigskip

\centerline{- - - - - - - - - - - - - - - - - - - - - - - - - - - - - - - - - - - - - - - - - - - - - - - - - - - - - -}
  
  \medskip

As usual,  $M^0=I_s$ where $M$ is  a square $s\times s$ matrix.
  Via induction
  $$\left(\begin{matrix}
  A& B\\
  0&0
  \end{matrix}
  \right)^k
  = \left(\begin{matrix}
  A^k& A^{k-1}B\\
  0&0
  \end{matrix}
  \right).
 $$
 Hence
 \begin{eqnarray*}
\exp\left(\begin{matrix}
  A& B\\
  0&0
  \end{matrix}
  \right)
  &=&\sum_{k=0}^\infty \frac{1}{k!} \left(\begin{matrix}
  A& B\\
  0&0
  \end{matrix}
  \right)^k
 =I_{n+m} +\sum_{k=1}^\infty \frac{1}{k!} \left(\begin{matrix}
  A^k& A^{k-1}B\\
  0&0
  \end{matrix}
  \right) \\
  &=& \left(\begin{matrix}
 \exp A& \sum_{k=1}^\infty  \frac{1}{k!}A^{k-1}B\\
  0&I_{m}
  \end{matrix}
  \right).
\end{eqnarray*}
 But 
 $$\int_0^1 \exp(At)dt= \sum_{j=0}^\infty A^j\;\int_0^1 \frac{t^j}{j!}=\sum_{j=0}^\infty \frac{1}{(j+1)!} A^j.$$
 Hence
 $$ 
\exp\left(\begin{matrix}
  A& B\\
  0&0
  \end{matrix} \right)
 =\left(\begin{matrix}
  \exp A& \left(\int_0^1 \exp(At)dt\right)B\\
  0&I_m
  \end{matrix}\right).
  $$

\newpage

   \begin{figure}[h!]
 
  \scalebox{0.45} 
  {\includegraphics{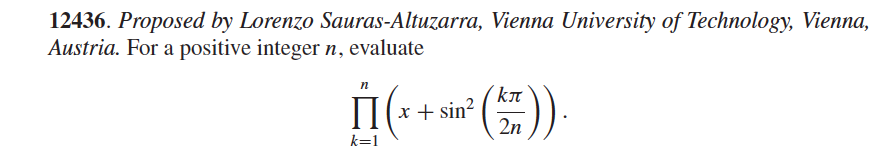}} 
  
 \end{figure}
 
\nopagecolor

\centerline{\bf Solution to problem 12436  in Amer. Math. Monthly 130 (2023)} \medskip 

\centerline{Raymond Mortini and Rudolf Rupp }

\bigskip

\centerline{- - - - - - - - - - - - - - - - - - - - - - - - - - - - - - - - - - - - - - - - - - - - - - - - - - - - - -}
  
  \medskip

 We show that
$$\ovalbox{$P(x):=\dis \prod_{k=1}^n \left( x+\sin^2\left(\frac{k\pi}{2n}\right)\right)=2^{-2n+2}(x+1) U_{n-1}(2x+1)$}\;,$$

 where  $U_0=1$ and 
 $$U_n(x)=\sum_{k=0}^{\lfloor\frac{n}{2}\rfloor} (-1)^k {n-k\choose k} (2x)^{n-2k}\buildrel=_{}^{n\geq 1}
 2^n\prod_{k=1}^n \left(x-\cos\left(\frac{k\pi}{n+1}\right)\right)$$
 is the Chebyshev polynomial of the second kind.
 
 This is very easy, though.
 \begin{eqnarray*}
 P(x)&=&\prod_{k=1}^n \left( x+\co{\sin^2\left(\frac{k\pi}{2n}\right)}\right)= \prod_{k=1}^n \left( x+\co{\frac{1}{2}
 \left(1-\cos\left(\frac{k\pi}{n}\right)\right)}\right)\\
 &=&2^{-n}\prod_{k=1}^n \left(2x+1-\cos\left(\frac{k\pi}{n}\right)\right)=2^{-n} \prod_{k=1}^{n-1} \left(2x+1-\cos\left(\frac{k\pi}{n}\right)\right)\; (2x+1-\cos(\pi))\\
 &=&2^{-2n+2}(x+1) U_{n-1}(2x+1).
\end{eqnarray*}

 \newpage
 
   \begin{figure}[h!]
 
  \scalebox{0.45} 
  {\includegraphics{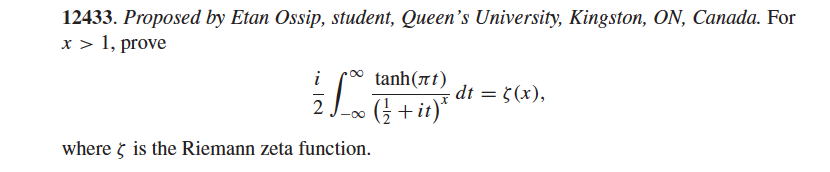}} 
  
 \end{figure}

 \centerline{\bf Solution to problem 12433  in Amer. Math. Monthly 130 (2023), ?}\medskip 
 
\centerline{Raymond Mortini and Rudolf Rupp }

\bigskip

\centerline{- - - - - - - - - - - - - - - - - - - - - - - - - - - - - - - - - - - - - - - - - - - - - - - - - - - - - -}
  
  \medskip

 Let $G:=\{z\in \C: {\rm Im} z<0.5\}$ be the  shifted lower half-plane, and  
 let $\log z=\log|z|+i\arg z$ with $-\pi<\arg z<\pi$  be the standard holomorphic branch of the logarithm. Since  for $z\in G$  
we have ${\rm Re}\;(0.5+iz)>0$,  the function 
$$(0.5 +iz)^x=e^{x\log (0.5+iz)}$$
 is well defined and holomorphic in $G$.  Consequently,  the  function
 $$f(z):= \frac{\tanh (\pi z)}{(0.5 +iz)^x},$$
  is meromorphic in $G$ with simple poles at $z_k:=-i(0.5 +k)\in G$, where $k\in \N=\{0,1,2,\dots\}$.
  We apply now  the residue theorem to $f$. To this end,  we integrate for $N\geq 1$  the function $f$ along  the positively oriented
  boundary $\Gamma_N$ of the rectangles $R_N:=[-N,N]\times [0,-N]$ and conclude that
  $$
  \int_{\Gamma_N} f(z) dz= 2\pi i \sum_{k=0}^\infty n(\Gamma_N,z_k){\rm Res}(f, z_k),
  $$
  where $n(\Gamma,z)$ denotes the number of times the point $z$ is  surrounded by $\Gamma$. Observe that at most a finite number of terms in this sum are not equal to $0$ as 
  $$n(\Gamma_N, z_k)=\begin{cases} 1 &\text{if $k=0,1,\dots, N$}\\ 0& \text{if $k>N$}.
  \end{cases}
  $$
  Let us calculate the residue now. We use the formula ${\rm Res}\;( \frac{g}{h}, a)=  \frac{g(a)}{ h'(a)}$, whenever $a$ is a simple zero of $h$. That is, when we choose $\dis g(z)=\frac{\sinh (\pi z)}{(0.5+iz)^x}$ and 
  $\dis h(z)= \cosh (\pi z)$, 
  $${\rm Res}(f, z_k)= \frac{ \sinh (\pi z)}{(1+k)^x} \frac{1}{ \pi \sinh(\pi z)}\Big|_{z=-i(0.5+k)}= \frac{1}{\pi}\;\frac{1}{(1+k)^x}.
  $$

 It remains to show that the integral along the three parts $\Gamma_N^j$ of $\Gamma_N$ that are contained in the lower half plane 
 ${\rm Im}\; z<0$ tends to zero. First note that
 $$\tanh (\pi z)= \frac{e^{2\pi z}-1}{e^{2\pi z}+1}.$$
 
 i) Let $z(t)=-N-it$, where $0\leq t\leq N$. Then for $n\geq N_0$,
 
 \begin{eqnarray*}
 |\tanh z(t)|&=& \left|\frac{e^{-2\pi N}e^{-2i \pi t}-1}{e^{-2\pi N}e^{-2i\pi t}-1}\right|\leq \frac{1+e^{-2\pi N}}{1-e^{-2\pi N}}\leq 2.
 \end{eqnarray*}
Moreover,

\begin{eqnarray*}
|0.5+iz(t)|^x&=& |0.5-iN+t|^x\geq N^x.
\end{eqnarray*}
 
 ii) Let $z(t)=N-it$, where $0\leq t\leq N$. Then for $n\geq N_0$,
 \begin{eqnarray*}
  |\tanh z(t)|&=& \left|\frac{e^{2\pi N}e^{-2i\pi t}-1}{e^{2\pi N}e^{-2i\pi t}+1}\right|\leq \frac{e^{2\pi N}+1}{e^{2\pi N}-1}\leq 2.
  \end{eqnarray*}
  Moreover
  
  \begin{eqnarray*}
  |0.5+iz(t)|^x&=&|0.5 +iN+t|^x\geq  N^x.
\end{eqnarray*}
  
  iii) Let $z(t)=t -iN$ where $-N\leq t\leq N$. Then
  \begin{eqnarray*}
   |\tanh z(t)|&=&\left|\frac{e^{2\pi t} e^{-2i\pi N}-1}{e^{2\pi t} e^{-2i\pi N}+1}\right|= \frac{e^{2\pi t}-1}{e^{2\pi t}+1}\leq 1.
\end{eqnarray*}

Moreover
\begin{eqnarray*}
|0.5+iz(t)|^x&=&|0.5+it+N|^x\geq (N+0.5)^x\geq N^x.
\end{eqnarray*}

Since $x>1$, we conclude that for $N\geq N_0$
\begin{eqnarray*}
\left|\sum_{j=1}^3\int_{\Gamma_N^j} f(z)dz\right|&\leq&\sum_{j=1}^2\int_0^N |f(z_j(t))| dt +\int_{-N}^N|f(z_3(t))| dt\\
&\leq&2 \cdot  \frac{2N}{N^x} +  2N\frac{1}{N^x}=\frac{6}{N^{x-1}}\to 0 \; \text{as $N\to \infty$.}
\end{eqnarray*}

We conclude that 
\begin{eqnarray*}
\frac{i}{2}\int_{-\infty}^\infty \frac{\tanh (\pi t)}{(0.5+it)^x}\;dt=-\frac{i}{2} 2\pi i \sum_{k=0}^\infty \frac{1}{\pi }\; \frac{1}{(1+k)^x}=\zeta(x).
\end{eqnarray*}
 Note that the minus sign comes from the fact that the upper boundary of the rectangle $R_N$ is run through from the right to the left.
   \begin{figure}[h!]
 
  \scalebox{0.45} 
  {\includegraphics{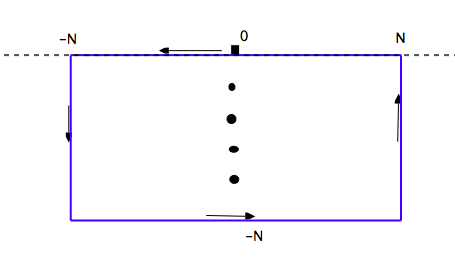}} 
  
 \end{figure}

\newpage

\begin{figure}[h!]
 
  \scalebox{0.5} 
  {\includegraphics{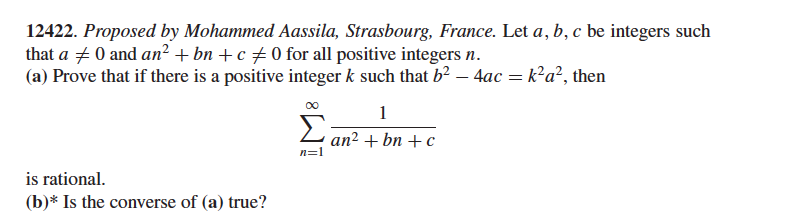}} 
  
 \end{figure}
 \nopagecolor

\centerline{\bf Solution to problem 12422  in Amer. Math. Monthly 130 (2023), 862} \medskip 

\centerline{Raymond Mortini and Rudolf Rupp }

\bigskip

\centerline{- - - - - - - - - - - - - - - - - - - - - - - - - - - - - - - - - - - - - - - - - - - - - - - - - - - - - -}
  
  \medskip

We solve (a).  Put
  $$R:=\sum_{n=1}^\infty \frac{1}{an^2+bn+c}.$$
   Let $r_1$ and $r_2$ be the  zeros of the polynomial $p(x)= ax^2+bx+c$. Suppose that $a\not=0$ and 
  $b^2-4ac=k^2a^2$ for some $k\in\{1,2,3,\dots\}$.
  Then
  $$\mbox{$\dis r_1= \frac{-b-ka}{2a}$ and $r_2=\dis \frac{-b+ka}{2a} $}$$
  and $r_1-r_2= -k$. As an example we mention $a=1$, $b=5$ and $c=4$, $k=3$, $r_1=-4$,  $r_2=-1$.
  
 Now the partial fraction decomposition of $1/p(n)$ reads as
 \begin{eqnarray*}
\frac{1}{an^2+bn+c}&= & \frac{1}{a(n-r_1)(n-r_2)}= \frac{1}{a(r_1-r_2)}\left(\frac{1}{n-r_1} -\frac{1}{n-r_2}\right).
 \end{eqnarray*} 
 Hence, for $n\geq n_0>1$ and $n_0$ chosen so that $n-r_j-1> 0$, 
 \begin{eqnarray*}
 S:=\sum_{n=n_0}^\infty \frac{1}{an^2+bn+c}&=&\frac{1}{a(r_2-r_1)}\sum_{n=n_0}^\infty \int_0^1 (x^{n-r_2-1}- x^{n-r_1-1})dx\\
 &=& \frac{1}{ak}\sum_{n=n_0}^\infty\int_0^1 x^{n-1} ( x^{-r_2}-x^{-r_1}) dx
\buildrel=_{}^{(1)} \frac{1}{ak} \int_0^1 ( x^{-r_2}-x^{-r_1}) \sum_{n=n_0}^\infty x^{n-1} dx\\
&=& \frac{1}{ak} \int_0^1  x^{n_0-1} \frac{ x^{-r_2}-x^{-r_1}}{1-x}\; dx
= \frac{1}{ak} \int_0^1 x^{n_0-1} x^{-r_2} \frac{1-x^{r_2-r_1}}{1-x}\; dx\\
&=&\frac{1}{ak} \int_0^1 x^{n_0-r_2-1} \frac{1-x^k}{1-x}\;dx
= \frac{1}{ak} \int_0^1\sum_{j=0}^{k-1} x^{n_0-1-r_2+j}\;dx\\
&= &\frac{1}{ak} \sum_{j=0}^{k-1} \frac{1}{n_0-r_2+j}.
\end{eqnarray*}
Hence $S$ is rational, and therefore $R$ is rational, too. 
 
 Note that in (1) the interchanging $\int\sum=\sum\int$ is possible, since $x^{n-1} ( x^{-r_2}-x^{-r_1}) $ has constant sign.

\newpage

 \begin{figure}[h!]
 
  \scalebox{0.45} 
  {\includegraphics{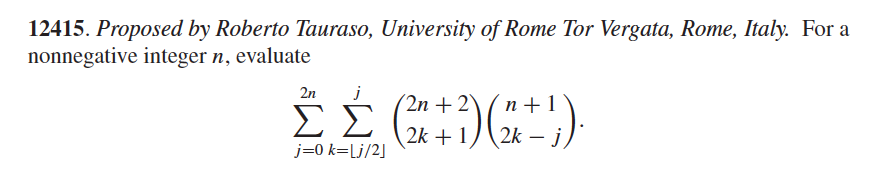}} 
  
 \end{figure}

\centerline{\bf Solution to problem 12415  in Amer. Math. Monthly 130 (2023), 765}\medskip 

\centerline{Raymond Mortini and Rudolf Rupp }

\bigskip

\centerline{- - - - - - - - - - - - - - - - - - - - - - - - - - - - - - - - - - - - - - - - - - - - - - - - - - - - - -}
  
  \medskip

Let  $n\in \N=\{0,1,2,\dots\}$ and  let 
$$ S_n:=\sum_{j=0}^{2n}\sum_{k=\lfloor j/2\rfloor}^j {2n+2\choose 2k+1}\;{n+1\choose 2k-j}.$$
We show that
$$\ovalbox{$S_n=2^{3n+1}.$}$$

First we interchange the two summations.

  \begin{figure}[h!]
 
  \scalebox{0.4} 
  {\includegraphics{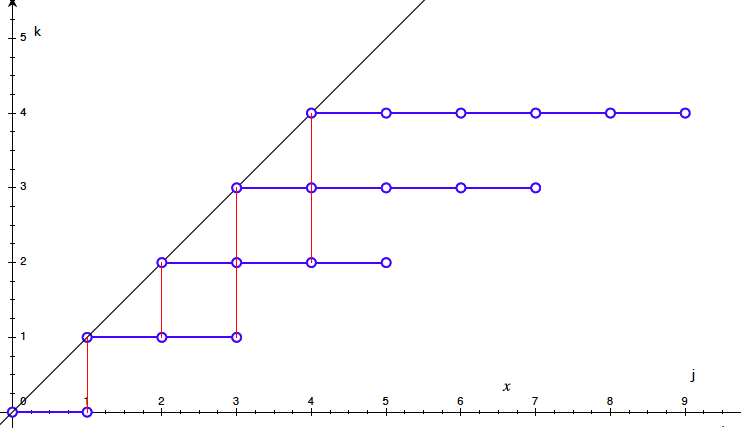}} 
  \caption{\label{gitt} \bl{$k\leq j\leq 2k+1$, $k=0,1,2,3,4$}, or \co{$ \lfloor j/2\rfloor\leq k\leq  j$ for $j=0,1,2,3,4$}}
  
 \end{figure}
 
\begin{eqnarray*}
S_n&=&\sum_{k=0}^{2n}\sum_{j=k}^{2k+1} {2n+2\choose 2k+1}\;{n+1\choose 2k-j}=
\sum_{k=0}^{2n}\left[{2n+2\choose 2k+1}\sum_{j=k}^{2k}{n+1\choose 2k-j}\right]\\
&=&\sum_{k=0}^{n}\left[{2n+2\choose 2k+1}\sum_{m=0}^{k}{n+1\choose m}\right]
\end{eqnarray*}

  \begin{figure}[h!]
 
  \scalebox{0.35} 
  {\includegraphics{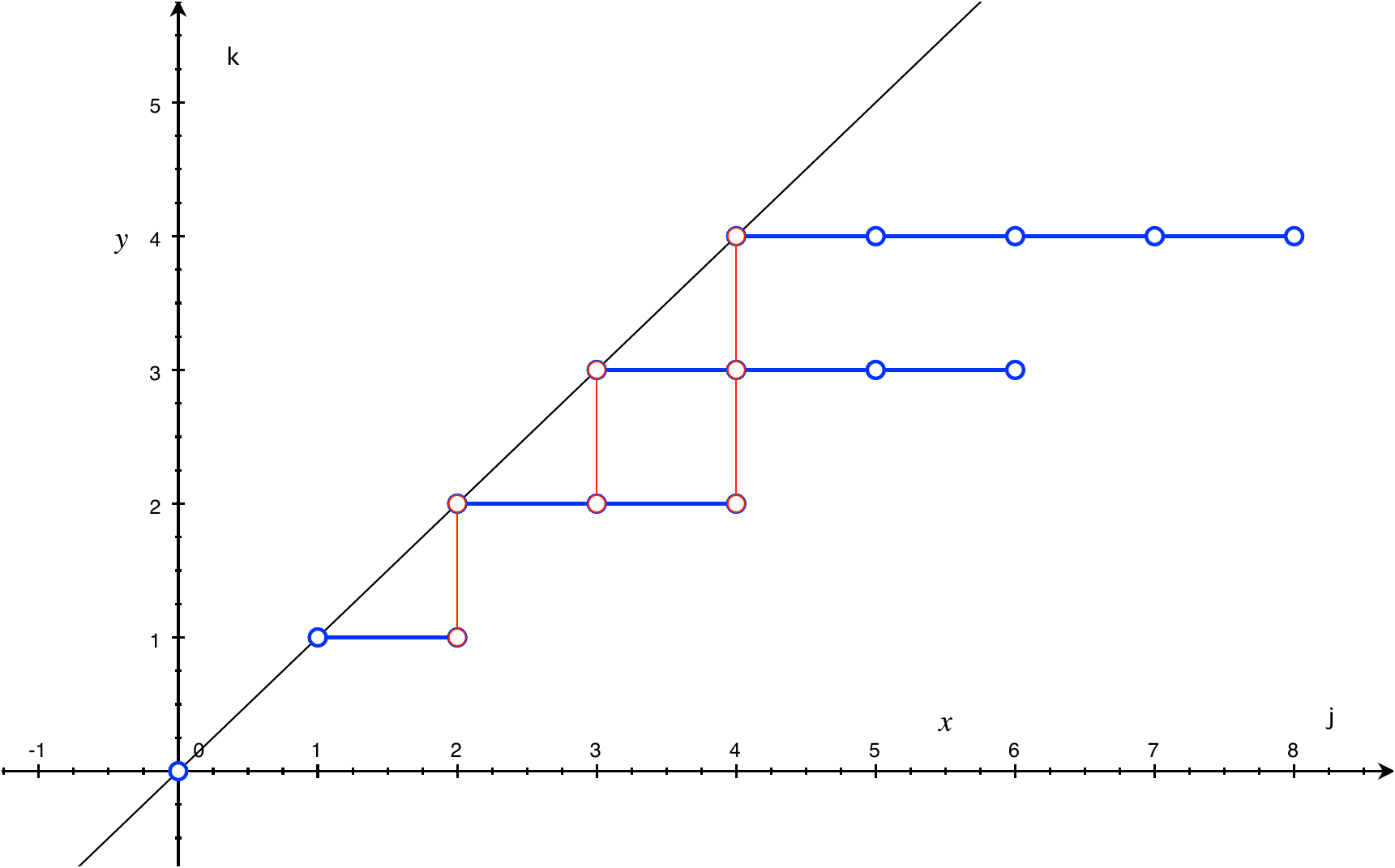}} 
  \caption{\label{gitt2} \bl{$k\leq j\leq 2k$, $k=0,1,2,3,4$}, or \co{$ \lceil j/2\rceil\leq k\leq  j$ for $j=0,1,2,3,4$}}
  
\end{figure}

\begin{figure}[ht!]
  
\vspace{4cm}
\begin{minipage}[b]{3cm}
 \begin{picture}(10,0)
 \hspace{-6cm}
     {\scalebox{0.25} {\includegraphics{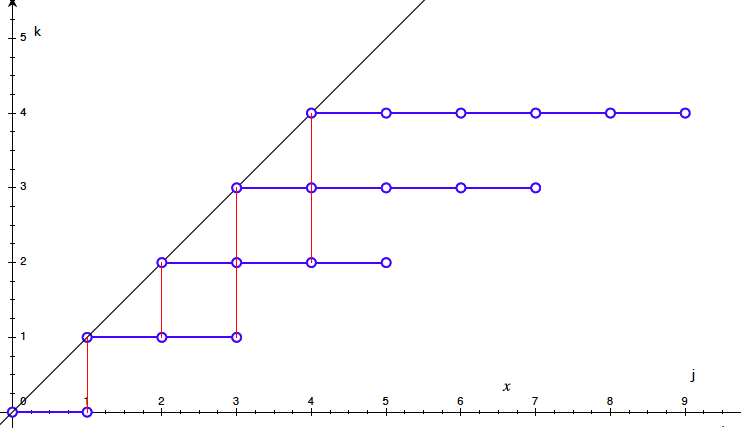}} }
\end{picture}
\end{minipage}

\begin{minipage}[b]{3cm}
 \begin{picture}(10,0)
 \hspace{4cm}
     {\scalebox{0.2} {\includegraphics{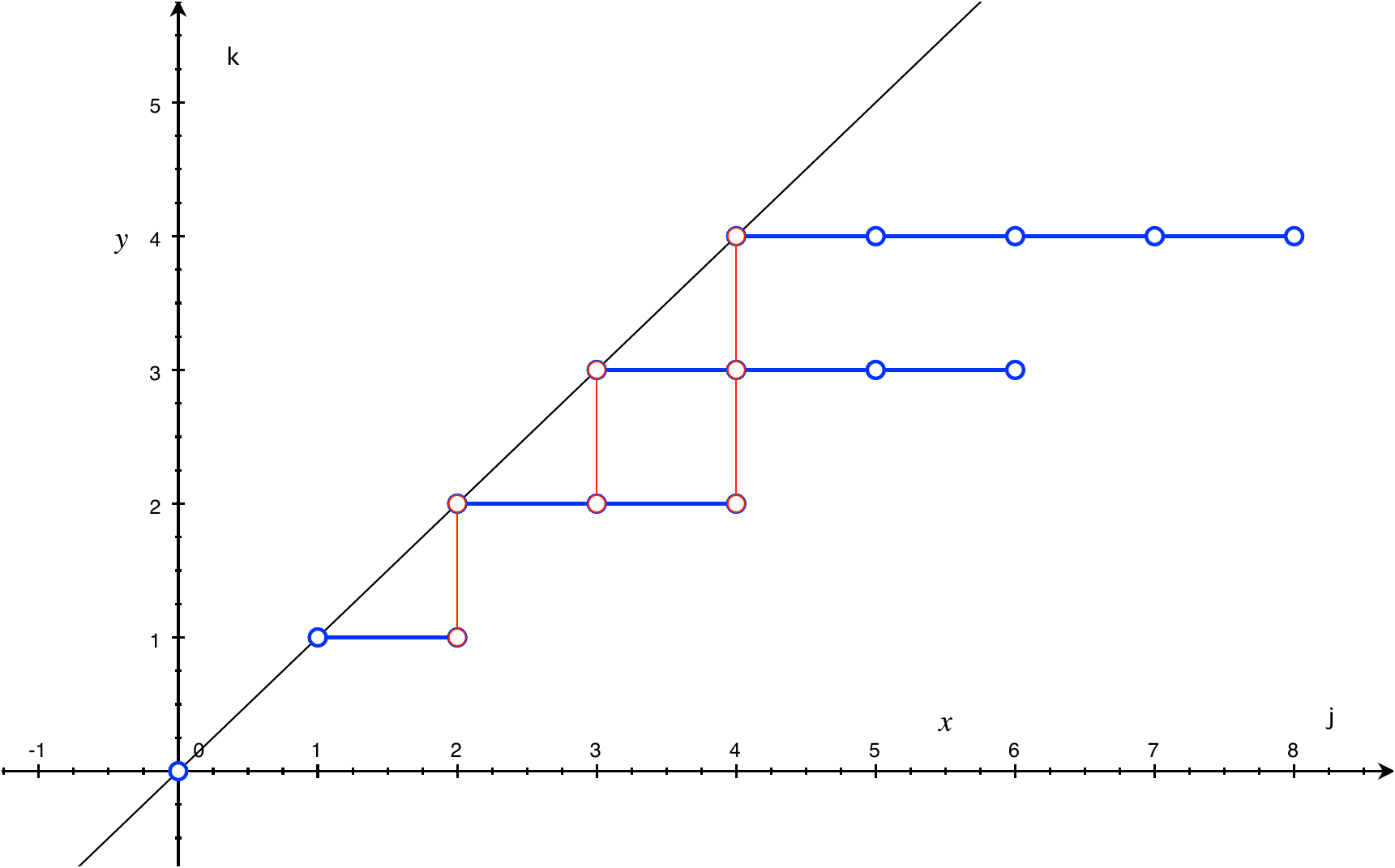}} }
\end{picture}
\end{minipage}
  
 \end{figure}

It is well known that $\dis\sum_{k=0}^n {2n+2\choose 2k+1}=2^{2n+1}$. In fact
$$2^{2n+2}=\sum_{k=0}^{2n+2}{2n+2\choose k}\quad\text{and}\quad 0=(1+ (-1))^{2n+2}= \sum_{k=0}^{2n+2}(-1)^k{2n+2\choose k}.$$
Substraction yields that
$$
2^{2n+2}= 2\sum_{k=0\atop k\; {\rm odd} }^{2n+2}{2n+2\choose k}=2\sum_{m=0}^n  {2n+2\choose 2m+1}.
$$

Also, 
\begin{eqnarray*}
\dis 2^{n+1}&=&(1+1)^{n+1}= \sum_{j=0}^{n+1} {n+1\choose j} =\sum_{j=0}^k {n+1\choose j} + \sum_{j=k+1}^{n+1} {n+1\choose j} \\
&=&\sum_{j=0}^k {n+1\choose j} + \sum_{j=k+1}^{n+1} {n+1\choose n+1-j}\\
&\buildrel=_{i=n+1-j}^{}&\sum_{j=0}^k {n+1\choose j} + \sum_{i=0}^{n-k}{n+1\choose i}.
\end{eqnarray*}
Hence
\begin{eqnarray*}
S_n&=&\sum_{k=0}^{n}{2n+2\choose 2k+1}\sum_{m=0}^k{n+1\choose m}\buildrel=_{k=n-j}^{}
\sum_{j=0}^n{2n+2\choose 2n-2j+1} \sum_{i=0}^{n-j}{n+1\choose i}\\
&=&\sum_{j=0}^n{2n+2\choose (2n+2)-(2n-2j+1)} \sum_{i=0}^{n-j}{n+1\choose i}\\
&=&\sum_{j=0}^n{2n+2\choose 1+2j} \sum_{i=0}^{n-j}{n+1\choose i}\buildrel=_{j\to k}^{}
\sum_{k=0}^n{2n+2\choose 2k+1} \sum_{i=0}^{n-k}{n+1\choose i}.
\end{eqnarray*}

Addition yields
\begin{eqnarray*}
2S_n&=&\sum_{k=0}^{n}{2n+2\choose 2k+1}\sum_{j=0}^k{n+1\choose j}+
\sum_{k=0}^n{2n+2\choose 2k+1} \sum_{i=0}^{n-k}{n+1\choose i}\\
&=&\sum_{k=0}^{n}{2n+2\choose 2k+1}\left(\sum_{j=0}^k{n+1\choose j}+ \sum_{i=0}^{n-k}{n+1\choose i}\right)\\
&=&2^{2n+1}\cdot 2^{n+1}=2^{3n+2}.
\end{eqnarray*}
Hence $S_n=2^{3n+1}$.

{\bf Remarks}

(1)Note that 
$$S_n\buildrel=_{j=k-m}^{}\sum_{k=0}^{n}\left[{2n+2\choose 2k+1}\sum_{j=0}^k{n+1\choose k-j}\right]$$
This has the form  $\dis \sum_{k=0}^\infty a_k\sum_{j=0}^k b_{k-j}$, which is  a little bit different from the Cauchy product
$$\left(\sum_{k=0}^\infty a_k\right)\; \left(\sum_{k=0}^\infty b_k\right)= \sum_{k=0}^\infty\sum_{j=0}^k a_jb_{k-j}.$$

(2) Replacing  $\lfloor j/2\rfloor$ by $\lceil j/2\rceil$ yields the same result
 $$R_n:=\sum_{j=0}^{2n}\sum_{k=\lceil j/2\rceil}^j {2n+2\choose 2k+1}\;{n+1\choose 2k-j}=2^{3n+1}$$
 (see below),
alhough the associated index-grid is different (see figure \ref{gitt} and \ref{gitt2}).

Just note that
\begin{eqnarray*}
R_n&=&\sum_{k=0}^{2n}\sum_{j=k}^{2k} {2n+2\choose 2k+1}\;{n+1\choose 2k-j}=
\sum_{k=0}^{2n}\sum_{j=k}^{2k+1} {2n+2\choose 2k+1}\;{n+1\choose 2k-j}=S_n
\end{eqnarray*}

 
\newpage

 \begin{figure}[h!]
 
  \scalebox{0.5} 
  {\includegraphics{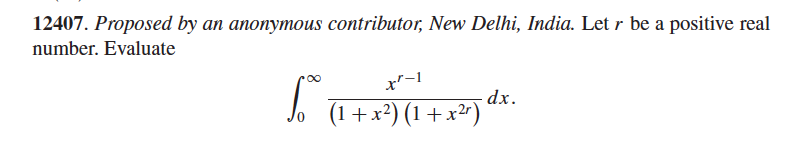}}
\end{figure}


\centerline{\bf Solution to problem 12407  in Amer. Math. Monthly 130 (2023), **}\medskip 

\centerline{Raymond Mortini and Rudolf Rupp }

\bigskip

\centerline{- - - - - - - - - - - - - - - - - - - - - - - - - - - - - - - - - - - - - - - - - - - - - - - - - - - - - -}
  
  \medskip

\bigskip

Given $r>0$, 
let $$I(r):=\int_0^\infty \frac{x^{r-1}}{(1+x^2)(1+ x^{2r})}\;dx.$$
We show that 
$$\ovalbox{$\dis I(r)=\frac{\pi}{4r}$}.$$
First it is clear that the integral converges since at $\infty$ we have that the integrand  $f_r(x)$ is similar to $1/x^{r+3}$ and at $0$
$f_r(x)$ is similar to $x^{r-1}$, where $r-1>-1$.
We make the change of the variable $x\to 1/y$. Then
\begin{eqnarray*}
I(r)&=& \int_0^\infty \frac{y^{1-r}}{(1+y^{-2})(1+y^{-2r})}\frac{dy}{y^2}=\int_0^\infty \frac{y^{1+r}}{(1+y^2)(1+y^{2r})}\;dy\\
&=&\int_0^\infty\frac{(y^2+1-1)y^{r-1}}{(1+y^2)(1+y^{2r})}\;dy=\int_0^\infty \frac{y^{r-1}}{1+y^{2r}}\;dy-I(r).
\end{eqnarray*}
Hence
$$2I(r)=\frac{1}{r}\arctan(y^r)\Big|^\infty_0=\frac{\pi}{2r},$$
from which we deduce that $I(r)=\pi/(4r)$.

\newpage

\begin{figure}[h!]
 
  \scalebox{0.6} 
  {\includegraphics{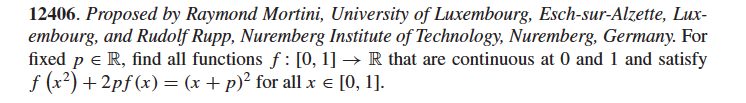}}
\end{figure}

\bigskip

\pagecolor{yellow}

\centerline{\bf Solution to problem 12406  in Amer. Math. Monthly 130 (2023), 679} \medskip 

\centerline{Raymond Mortini and Rudolf Rupp }

\bigskip

\centerline{- - - - - - - - - - - - - - - - - - - - - - - - - - - - - - - - - - - - - - - - - - - - - - - - - - - - - -}
  
  \medskip

\bigskip
Let $p\in\R$.  Consider the functional equation
\begin{equation}\label{fgl}
f(x^2)+2p f(x)=(x+p)^2.
\end{equation}

We claim that all solutions of (\ref{fgl}) on $[0,1]$ and continuous at $\{0,1\}$ are actually continuous on $[0,1]$ and are given by 
 $$f(x)= x+ \frac{p^2}{1+2p}$$
 whenever $p\not=-1/2$.\\
 
  $\bullet$ If $p=-1/2$, then $f(x^2)-f(x)=(x-\frac{1}{2})^2$ has  no solution  on $[0,1]$ (independently of being continuous or not) since for $x=1$, we would get  $0=f(1)-f(1)=1/4$.

$\bullet$ If $p=0$,  then $f(x^2)=x^2$ implies that on $[0,1]$ one has $f(x)=x$.

$\bullet$ Let $p\not=-1/2$.  We first determine the polynomial solutions. So let $q$ be  a polynomial solving (\ref{fgl}).
 Then the degree of $q$ is at most $1$. Say $q(x)=ax+b$. 
 Pulling into the functional equation yields
 $$ax^2+b+2p(ax+b)=x^2+2px +p^2$$
 or equivalently
 $$(a-1)x^2+2p(a-1)x +b(1+2p)-p^2=0.$$
 Hence $a=1$ and $b=\frac{p^2}{1+2p}$. 
 
 It is straighforward to check that $q(x)=x+ \frac{p^2}{1+2p}$ is indeed a solution to (\ref{fgl}). 
  We conclude that all polynomial solutions are given by the linear function $q$ above.\\
 
Next we determine the general solution   (\ref{fgl}). So let $f$ be  a solution on $[0,1]$  continuous at $0,1$.
 Now put $h(x):=f(x)-q(x)$. Then $h$ satisfies on $[0,1]$ the functional equation (of Schroeder type)
\begin{equation}\label{fglred}
h(x^2)=-2p h(x).
\end{equation}
Of course this implies that $h(0)=0$. 

i)  Let $p<-1/2$ or $p\geq 1/2$. Via induction
$$h(x^{2^n})=(-2p)^n h(x).$$
Since $h$ is continuous at $0$,  and $x^{2^n}\to 0$ for $0<x<1$,  $h(0)=0$, and  $|2p|^n\to \infty$ respectively 
$(-2p)^n=(-1)^n$ if $p=1/2$,  we
 deduce that $h(x)=0$ for $0<x<1$, too.

 ii) If $0<|p|<1/2$,  we rewrite (\ref{fglred}) as
 \begin{equation}\label{fglred2}
 h(\sqrt x)=-\frac{1}{2p} h(x).
\end{equation}
 Via induction
 $$
 h(x^{1/2^n})=\left(-\frac{1}{2p}\right)^{n} h(x).
 $$
 Since $h$ is assumed to be continuous at $1$ and $h(1)=0$ by (\ref{fglred}), $\left(\frac{1}{2p}\right)^{n}\to\infty$ implies that  $$0= h(1)=\infty \cdot h(x),$$
  and so $h(x)=0$ for $x>0$.
 \medskip
 
 We conclude that  for $p\not=-1/2$, the general solution  to (\ref{fgl}) on $[0,1]$, and continuous at $0,1$, is given by our poynomial
 $$q(x)=x+ \frac{p^2}{1+2p}.$$

 Thus the solution is completely established.\medskip
 
 The Schr\"oder type functional equations $f(x^m)=rf(x)$ are analyzed in detail in \cite{mr2024}.

\newpage
\nopagecolor
\begin{figure}[h!]
 
  \scalebox{0.5} 
  {\includegraphics{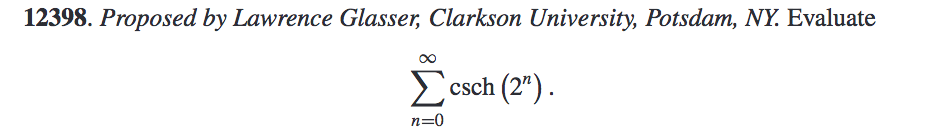}}
\end{figure}

\centerline{\bf Solution to problem 12398 in Amer. Math. Monthly 130 (2023), 587} \medskip 

\centerline{Raymond Mortini and Rudolf Rupp }

\bigskip

\centerline{- - - - - - - - - - - - - - - - - - - - - - - - - - - - - - - - - - - - - - - - - - - - - - - - - - - - - -}
  
  \medskip


We suppose that this agglomeration $csch$ of letters is nothing but $1/\sinh$. So let
$$S:=\sum_{n=0}^\infty \frac{1}{\sinh 2^n}.$$

We prove that $$\ovalbox{$S=\frac{2}{e-1}$}.$$
This is very simple though. Since $2\sinh x=e^x-e^{-x}$  and 
$$(e^{2^n}+1)(e^{2^n}-1)=e^{2^{n+1}}-1,$$
we obtain
\begin{eqnarray*}
S&=&2\sum_{n=0}^\infty  \frac{1}{e^{2^n}-e^{-2^n}}=2\sum_{n=0}^\infty \frac{e^{2^n}\co{+1-1}}{e^{2^{n+1}}-1}\\
&=&2\sum_{n=0}^\infty \left(\frac{1}{e^{2^n}-1}-\frac{1}{e^{2^{n+1}}-1}\right)=\frac{2}{e-1}.
\end{eqnarray*}

Another possibility would be to use the formula
$$\frac{1}{\sinh x}=\coth(x/2)-\coth x.$$
Then
\begin{eqnarray*}
S&=&\sum_{n=0}^\infty\left(\coth (2^{n-1})-\coth 2^n\right)=\coth(1/2)-1=\frac{2e^{-1/2}}{e^{1/2}-e^{1/2}}=\frac{2}{e-1}.
\end{eqnarray*}


\newpage

 \begin{figure}[h!]
 
  \scalebox{0.5} 
  {\includegraphics{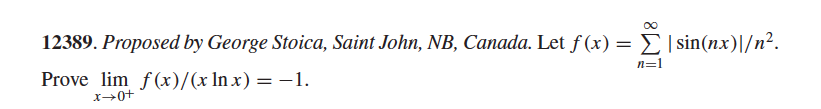}}
\end{figure}


\centerline{\bf Solution to problem 12389  in Amer. Math. Monthly 130 (2023), 386} \medskip 

\centerline{Raymond Mortini and Rudolf Rupp }

\bigskip

\centerline{- - - - - - - - - - - - - - - - - - - - - - - - - - - - - - - - - - - - - - - - - - - - - - - - - - - - - -}
  
  \medskip

Our tool will be the fact that for $\dis H_n:=\sum_{j=1}^n \frac{1}{j}$ we have $H_n-\log n \searrow \gamma$, where $\gamma$ is the Euler-Mascheroni constant.
First note that
\begin{eqnarray*}
\sum_{j=1}^\infty \frac{N}{(N+j)^2}&\leq &\sum_{j=1}^\infty \frac{N}{(N+j)(N+j-1)}\\
&=&N\sum_{j=1}^\infty\Big( \frac{1}{N+j-1}- \frac{1}{N+j}\Big)\\
&=&1.
\end{eqnarray*}

Fix $0<x<1$ and let $\dis N:=N(x):=\left\lfloor \frac{1}{x}\right\rfloor$. Let $\e\in \;]0,1/2]$. 
Since $|\sin y|\leq y$ for $y\geq 0$, 
 we obtain for 
\begin{eqnarray*}
\sum_{n=1}^\infty \frac{|\sin (nx)|}{n^2}&\leq& \sum_{n=1}^{N} \frac{nx}{n^2}+\sum_{n= N+1}^\infty \frac{1}{n^2}
\leq x\;\sum_{n=1}^{N} \frac{1}{n}+\frac{1}{N}\\
&\leq& x \Big( H_{N}-\log N-\gamma\Big) +x\gamma + x\,\log N +\frac{1}{N}.
\end{eqnarray*}

Hence, for $x$ small enough, $N$ is big, and so
\begin{eqnarray*}
H(x):=\frac{\sum_{n=1}^\infty \frac{|\sin (nx)|}{n^2}}{x\log (1/x)}&\leq \dis \frac{1}{\log(1/x)}+\frac{\gamma}{\log(1/x)} +
\frac{\log \left\lfloor \frac{1}{x}\right\rfloor}{\log(1/x)}+\frac{1}{x\log(1/x) \left\lfloor \frac{1}{x}\right\rfloor}.
\end{eqnarray*}

We conclude that 
$$0\leq \limsup_{x\to 0} H(x)\leq 0+0+1+0= 1.$$

Now we estimate $\liminf_{x\to 0} H(x)$.  Let $\e\in ]0,1/2]$.  Since $x\mapsto (\sin x)/x$ is decreasing on $[0,\pi/2]$, we see that
for $0<u\leq \e$
$$\frac{\sin u}{u}\geq \frac{\sin\e}{\e}.$$
For $0<x< \e$ put $\dis N:=N(x):=\left\lfloor\frac{\e}{x}\right\rfloor$.  Then, $N>0$ and for $n\leq N$ we have 
$$nx\leq Nx=\left\lfloor\frac{\e}{x}\right\rfloor  x\leq \frac{\e}{x} x=\e,$$
and so
$$\frac{\sin (nx)}{nx}\geq \frac{\sin\e}{\e}.$$

Hence
\begin{eqnarray*}
\sum_{n=1}^\infty \frac{|\sin (nx)|}{n^2}&\geq& \frac{\sin\e}{\e} \sum_{n=1}^N \frac{nx}{n^2}= 
x\;  \frac{\sin\e}{\e}\sum_{n=1}^N \frac{1}{n}\\
&\geq& x\;  \frac{\sin\e}{\e}\log N.
\end{eqnarray*}

We deduce that for $0<x<\e$
\begin{eqnarray*}
H(x)&\geq &\dis \frac{\sin\e}{\e}\; \frac{\log \left\lfloor\frac{\e}{x}\right\rfloor }{\log(1/x)}\geq \frac{\sin\e}{\e}\; 
\frac{\log(\frac{\e}{x}-1)}{\log(1/x)}\\
&=&\frac{\sin\e}{\e}\;  \frac{\log(\e-x)-\log x}{-\log x}.
\end{eqnarray*}

Since 
$$\lim_{x\to 0}  \frac{\log(\e-x)}{-\log x}=\log\e \cdot 0$$

we conclude that 
$$\liminf_{x\to 0} H(x)\geq \frac{\sin\e}{\e}.$$

Now $\e\to 0$ yields that $\dis \liminf_{x\to 0} H(x)\geq 1$. Consequently $\lim_{x\to 0}H(x)=1$.

\newpage

 \begin{figure}[h!]
 
  \scalebox{0.5} 
  {\includegraphics{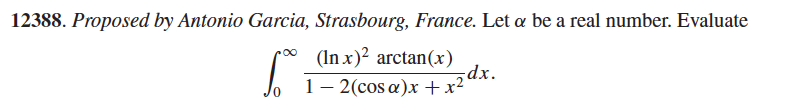}}
\end{figure}


\centerline{\bf Solution to problem 12388  in Amer. Math. Monthly 130 (2023), 385} \medskip 

\centerline{Raymond Mortini and Rudolf Rupp }

\bigskip

\centerline{- - - - - - - - - - - - - - - - - - - - - - - - - - - - - - - - - - - - - - - - - - - - - - - - - - - - - -}
  
  \medskip

For $a\in [0,2\pi]$, let
$$I(a):=\int_0^\infty \frac{(\log x)^2\,\arctan x}{1-2x\cos a +x^2}\;dx.$$
We prove that 
$$I(a)=\begin{cases}
\dis \pi\,\frac{a }{\sin a} \frac{(2\pi-a)(\pi-a)}{12}&\text{if $0<a<2\pi$, $a\not=\pi$}\\
\\
\dis\frac{\pi^3}{6}&\text{if $a=0$ or $a=2\pi$}\\
\\
\dis\frac{\pi^3}{12}&\text{if $a=\pi$}.
\end{cases}
$$
If $a$ is arbitrary, we replace $a$ by $a-2k\pi$, where $k\in \Z$ is chosen so that $2k\pi\leq a<2(k+1)\pi$.
\\

First we let  "disappear" the arctangent: the substitution $u=1/x$, $dx=-1/u^2$ and the formula $\arctan(1/ x)+\arctan x=\pi/2$ for $x>0$ yield
$$
I(a)=\int_0^\infty \frac{(\log u)^2\Big(\frac{\pi}{2}-\arctan u\Big)}{1-2 \frac{1}{u} \cos a +\frac{1}{u^2}}\; \frac{du}{u^2}
=-I+\frac{\pi}{2}\;\int_0^\infty \frac{(\log x)^2}{1-2x\cos a +x^2}\;dx,
$$
and so
$$I(a)=\frac{\pi}{4}\;\int_0^\infty  \frac{(\log x)^2}{1-2x\cos a +x^2}\;dx.$$

Using again the transformation $u=1/x$, we  obtain that
$$\int_0^1  \frac{(\log x)^2}{1-2x\cos a +x^2}\;dx= \int_1^\infty  \frac{(\log x)^2}{1-2x\cos a +x^2}\;dx,$$
and so
$$\ovalbox{$\dis I(a)=\frac{\pi}{2}\; \int_0^1  \frac{(\log x)^2}{1-2x\cos a +x^2}\;dx.$}$$

Next we use that for $a\notin \{k\pi: k\in \Z\}$
$$\frac{1}{1-2x\cos a +x^2}=\frac{A}{x- e^{ia}}- \frac{\ov A}{x- e^{-ia}},$$
where $\dis A= -\frac{i}{2\sin a}$. Hence, in that case,
\begin{eqnarray*}
I(a)&=&\dis 2 {\rm Re}\; \Big(\frac{\pi}{2}\; A \; \int_0^1 \frac{(\log x)^2}{x-e^{ia}}\;dx\Big)\\
&=& \frac{\pi}{2\sin a}\;{\rm Im}\; \int_0^1 \frac{(\log x)^2}{x-e^{ia}}\;dx\\
&=&\frac{\pi}{2\sin a}\;{\rm Im}\; (-e^{ia})\,\int_0^1\frac{(\log x)^2}{1-xe^{-ia}}\;dx
\end{eqnarray*}

Since $\sum_{n=0}^\infty x^n (\log x)^2$ is an $L^1(0,1)$-majorant, we have $\int\sum=\sum\int$. Thus
$$I(a)=-\frac{\pi}{2\sin a}\;{\rm Im}\;\Big(\sum_{n=0}^\infty e^{-ia(n+1)} \int_0^1 x^n(\log x)^2 \;dx\Big).$$
By twice partial integration, 
$$\int_0^1 x^n(\log x)^2 \;dx= \frac{2}{(n+1)^3}.$$
We conclude that
$$I(a)= -\frac{\pi}{2\sin a}\;{\rm Im}\;\Big(\sum_{n=0}^\infty e^{-ia(n+1)} \frac{2}{(n+1)^3}\Big)= 
\frac{\pi}{\sin a}\sum_{n=0}^\infty \frac{\sin(n+1)a}{(n+1)^3}.
$$

Let 
$$h(a):=\sum_{n=0}^\infty \frac{\sin(n+1)a}{(n+1)^3}.$$
Then 
$$h'(a)=\sum_{n=0}^\infty  \frac{\cos(n+1)a}{(n+1)^2}.
$$
Since $\frac{1}{3}\pi^2 + 4 \sum_{n=1}^\infty \frac{\cos nx}{n^2}$ is the  Fourier series of the  function $(x-\pi)^2$, $0\leq x<  2\pi$, 
extended $2\pi$-periodically,
we see that for $0<a<2\pi$,
$$h'(a)= \frac{(a-\pi)^2}{4}-\frac{\pi^2}{12}.$$

As $h(0)=0$, we deduce that for $0<a<2\pi$,
$$h(a)= \frac{(a-\pi)^3}{12}-\frac{\pi^2}{12} \,a+ \frac{\pi^3}{12}=\frac{a^3 -3\pi a^2+2 \pi^2 a}{12}.
$$

Consequently, for $0<a<2\pi$, $a\not=\pi$, 
$$I(a)=\pi\,\frac{a }{\sin a} \frac{(2\pi-a)(\pi-a)}{12}.
$$

Now let $a_n\searrow 0$ and $f_n(a):=\dis  \frac{(\log x)^2}{1-2x\cos a_n +x^2}$.
 As $f_n$ is positive and increases to $\frac{(\log x)^2}{(1-x)^2}$, we deduce from Beppo-Levi's monotone convergence theorem that 
 $I(a_n)\to I(0)$. Hence  $I(0)=\pi^3/6$.

Moreover, $I(b_n)\to \frac{\pi^3}{12}$ as $b_n\nearrow \pi$.  This is also the value of 
$$I(\pi)= \frac{\pi}{2}\; \int_0^1  \frac{(\log x)^2}{(1+x)^2}\;dx.$$
Just write
$$ \frac{(\log x)^2}{(1+x)^2}=\sum_{n=1}^\infty (-1)^{n-1} n x^{n-1}(\log x)^2,$$
and use again that $\int \sum=\sum\int$. Finally, $I(2\pi)=I(0)=\pi^3/6$.

\newpage

 \begin{figure}[h!]
 
  \scalebox{0.6} 
  {\includegraphics{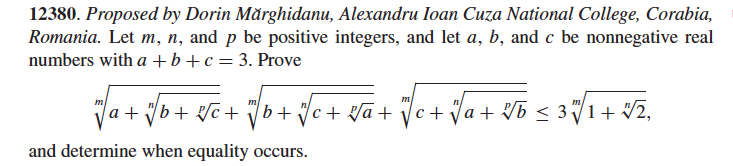}}
\end{figure}


\centerline{\bf Solution to problem 12380  in Amer. Math. Monthly 130 (2023), 285} \medskip 

\centerline{Raymond Mortini and Rudolf Rupp }

\bigskip

\centerline{- - - - - - - - - - - - - - - - - - - - - - - - - - - - - - - - - - - - - - - - - - - - - - - - - - - - - -}
  
  \medskip

{\bf Lemma} Let $f$, $g$, $h$ be  positive increasing  functions on $[0,\infty[$ satisfying for every $x_j\geq 0$ and $0\leq t_j\leq 1$
 with $\sum_{j=1}^n t_j=1$ the concavity inequality
\begin{equation}\label{ju}
f\Big(\sum_{j=1}^n t_j x_j\Big)\geq \sum_{j=1}^n t_j f(x_j)
\end{equation}
and similarily for $g,h$.
 Let $M_2,P_j\in [0,\infty[\times [0,\infty[$. Then the function $G$ given by 
$$G(M_2):=G(x,y):=f(x +g(y))$$
also satisfies
$$G\Big(\sum_{j=1}^nt_j P_j\Big)\geq \sum_{j=1}^nt_j G(P_j),$$
Similarily, if $M_3, Q_j\in [0,\infty[\times [0,\infty[\times [0,\infty[$, then 
 the  function  $H$ given by
 $$H(M_3):=H(x,y,z):=f(x+g(y+h(z)))$$ satisfies
 $$H\Big(\sum_{j=1}^nt_j Q_j\Big)\geq \sum_{j=1}^nt_j H(Q_j).$$

\medskip

{\bf Proof} 
\begin{eqnarray*}
G\Big(\sum_{j=1}^n t_j x_j, \sum_{j=1}^n t_jy_j\Big)&=&f\Big(\sum_{j=1}^n t_j x_j+ g\Big(\sum_{j=1}^n t_j y_j\Big)\Big)\\
&\geq &f\Big(\sum_{j=1}^n t_j x_j+\sum_{j=1}^n t_j g(y_j)\Big)\\
&=&f\Big(\sum_{j=1}^nt_j(x_j+g(y_j))\Big)\\
&\geq&\sum_{j=1}^n t_jf(x_j+g(y_j))\\
&=&\sum_{j=1}^n t_j G(P_j).
\end{eqnarray*}
Now applying this, we get
\begin{eqnarray*}
H(\sum_{j=1}^n t_j Q_j)&=&f\Big( \sum_{j=1}^n t_j x_j + g\Big(\sum_{j=1}^n t_jy_j+h(\sum_{j=1}^n t_jz_j)\Big)\Big)\\
&\geq&f\Big( \sum_{j=1}^n t_j x_j + \sum_{j=1}^n t_j g(y_j+h(z_j))\Big)\\
&=& f\Big(\sum_{j=1}^n t_j (x_j+ g(y_j+h(z_j))\Big)\\
&\geq&\sum_{j=1}^n t_j f(x_j+ g(y_j+h(z_j)))\\
&=&\sum_{j=1}^n t_jH(Q_j).
\end{eqnarray*}

Now we are ready to give the solution to the problem. Let
$$S(a,b,c):=\sqrt[m]{a+\sqrt[n]{b+\sqrt[p]c}}+\sqrt[m]{b+\sqrt[n]{c+\sqrt[p]a}}+\sqrt[m]{c+\sqrt[n]{a+\sqrt[p]b}}.$$

\medskip

For $M:=(x,y,z)\in \R^3, x,y,z\geq 0$, let 
$$f(M):=\dis f(x,y,z):=\sqrt[m]{x+\sqrt[n]{y+\sqrt[p]z}}.$$
By Lemma,  for $P:=(a,b,c)$, $Q=(b,c,a)$ and $R=(c,a,b)$ we have
\begin{equation}\label{str}
 \frac{1}{3}\big(f(P)+f(Q)+f(R)\big)\leq f\left(\frac{P+Q+R}{3}\right).
 \end{equation}

Since $a+b+c=3$, we deduce that

$$f(P)+f(Q)+f(R)\leq 3 \cdot f(1,1,1)=3 \sqrt[m]{1+\sqrt[n]2}.$$

In case $mnp>1$, at least one function $\sqrt[r]x$ for $r\in\{m,n,p\}$ is strictly concave and we have strict inequality in (\ref{ju}) whenever not all the $x_j$ are the same and $0<t_j<1$.  The proof of the Lemma in particular then yields that equality holds in (\ref{str}) only if $P=Q=R$, and so
$a=b=c$. Thus, due to $a+b+c=3$, we deduce that $a=b=c=1$. Hence
$$S(a,b,c)=3 \sqrt[m]{1+\sqrt[n]2}$$ if and only if $(a,b,c)=(1,1,1)$.

If $m=n=p=1$, then $f(M)$ is linear in $\R^3$, and so for all $a,b,c$ with $a+b+c=3$ we have equality:
$$S(a,b,c)=3(a+b+c)=9 =3 \sqrt[m]{1+\sqrt[n]2}.$$

\newpage

 \begin{figure}[h!]
 
  \scalebox{0.6} 
  {\includegraphics{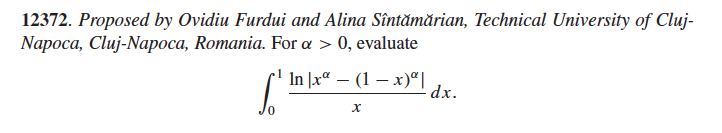}}
\end{figure}


\centerline{\bf Solution to problem 12372  in Amer. Math. Monthly 130 (2023), 187} \medskip 

\centerline{Raymond Mortini and Rudolf Rupp }

\bigskip

\centerline{- - - - - - - - - - - - - - - - - - - - - - - - - - - - - - - - - - - - - - - - - - - - - - - - - - - - - -}
  
  \medskip

For $a>0$, let $\dis I:=\int_0^1 \frac{\log|x^a-(1-x)^a|}{x}\;dx$, which is a double improper integral (with singularities at  $0,1/2$).
 We show that
$$\ovalbox{$\dis I=- \frac{a^2+2}{12 a} \;\pi^2$}.$$

To this  end, we first note that $x^a\leq (1-x)^a$ if and only if $0\leq x\leq 1/2$. Hence, by substituting $x\to 1-x$ in the second integral
\begin{eqnarray*}
I&=&\int_0^{1/2} \frac{\log\big((1-x)^a-x^a\big)}{x}\;dx + \int_0^{1/2} \frac{\log\big( x^a-(1-x)^a\big)}{x}\;dx\\
&=&\int_0^{1/2} \frac{\log\big((1-x)^a-x^a\big)}{x}\;dx +\int_0^{1/2} \frac{\log\big((1-x)^a-x^a\big)}{1-x}\;dx\\
&=& \int_0^{1/2} \frac{\log\big((1-x)^a-x^a\big)}{x(1-x)}\;dx\\
&=&\int_0^{1/2} \frac{\log\Big(1- \left(\frac{x}{1-x}\right)^a \Big) + a \log(1-x)}{x(1-x)}\;dx.
\end{eqnarray*}

Next we substitute $x/(1-x)=y$. Equivalently, $x=y/(1+y)$. Note that $0\to 0$ and $1/2\to 1$, $dx= \frac{1}{(1+y)^2}dy$ and
$1-x=\frac{1}{1+y}$. Hence
\begin{eqnarray*}
I&=& \int_0^1 \frac{\log(1-y^a)-a\;\log (1+y) }{\frac{y}{(1+y)^2}}\; \frac{1}{(1+y)^2}\; dy.
\end{eqnarray*}

Consequently,
$$I= \int_0^1 \frac{1}{x} \log\left(\frac{1-x^a}{(1+x)^a}\right)\;dx.
$$

Using partial integration for $\int_\e^{1-\eta}$ with $u':=1/x$ and $v= \log\big(\frac{1-x^a}{(1+x)^a}\big)$, and passing to the limits
 $\e,\eta\to 0$,
we obtain

\begin{eqnarray*}
I&=&0+ a\int_0^1  \left( \frac{x^{a-1}}{1-x^a} + \frac{1}{1+x}\right)\;\log x \; dx\\
&=& a\;\int_0^1\left( \sum_{n=0}^\infty x^{a-1} x^{na} + \sum_{n=0}^\infty (-1)^n x^n\right)\; \log x\; dx.
\end{eqnarray*}

\newpage

Note that $I$ has the form $I=\int\sum$. Now let us calculate $J:=\sum\int$.

To do this,  we apply for $\beta>-1$ the formula 
$$\int_0^1 x^\beta \log x\;dx= -\frac{1}{(\beta+1)^2},$$
(which can easily be obtained by partial integration $u=\log x, v'=x^\beta$).
 Hence
 
\begin{eqnarray}\label{werte1}
J/a&=&-\frac{1}{a^2} \sum_{n=0}^\infty \frac{1}{(n+1)^2} +\sum_{n=0}^\infty (-1)^{n +1} \frac{1}{(n+1)^2}\\
&=& -\frac{1}{a^2}\; \frac{\pi^2}{6} - \frac{\pi^2}{12}\\
&=& - \frac{a^2+2}{12 a^2} \;\pi^2.
\end{eqnarray}

To finish the proof,  we need to show that  $\int\sum=\sum\int$.   
As  the  summands in the first sum $\sum_{n=0}^\infty x^{a-1} x^{na} \log x$ do not change sign, we may use Beppo-Levi's theorem.
In the second sum, $ \sum_{n=0}^\infty (-1)^n x^n\; \log x$,  we have absolute convergence, in particular any rearrangement converges (to the same function),  and so we apply Beppo-Levi to the  sum over the odd integers and the sum over the even integers.
Thus $\int \sum_{even}=\sum_{even}\int$ and $\int \sum_{odd}= \sum_{odd} \int$.  Similarily to  (\ref{werte1}), it can be shown that the values of $ \sum_{odd} \int$ and $\sum_{even}\int$ are finite.  Hence
$$\int \sum=\int \sum_{even} -\int \sum_{odd}=\sum_{even}\int - \sum_{odd} \int=\sum\int.$$

\newpage

 \begin{figure}[h!]
 
  \scalebox{0.5} 
  {\includegraphics{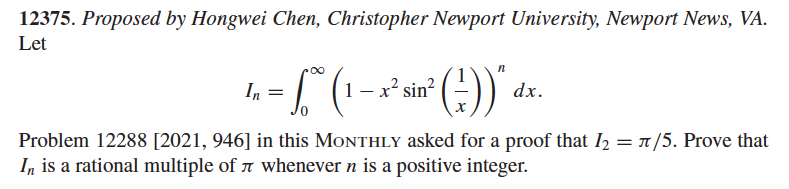}}
\end{figure}


\centerline{\bf Solution to problem 12375  in Amer. Math. Monthly 130 (2023), ??} \medskip 

\centerline{Raymond Mortini }

\bigskip

\centerline{- - - - - - - - - - - - - - - - - - - - - - - - - - - - - - - - - - - - - - - - - - - - - - - - - - - - - -}
  
  \medskip

A change of the variable $x\to 1/x$ yields that
$$J:=\int_0^\infty \left(1-x^2 \sin^2 \left(\frac{1}{x}\right)\right)^n\;dx=
\int_0^\infty \frac{(x^2-\sin^2 x)^n}{x^{2n+2}}\;dx.
$$

Now  we "linearize" the trigonometric powers: using
$\sin^2x=(1/2)(1-\cos 2x)$, we obtain
\begin{eqnarray*}
J&=&\frac{1}{2} \int_{-\infty}^\infty \frac{\big(x^2-\frac{1}{2}+ \frac{1}{2} \cos(2x)\big)^n}{x^{2n+2}}\;dx\\
&=& \frac{1}{2} \int_{-\infty}^\infty \frac{1}{x^{2n+2}}\; \sum_{j=0}^n  {n \choose j}\frac{1}{2^j}   \Big(x^2-\frac{1}{2}\Big)^{n-j}\;\cos^j (2x)\; dx.
\end{eqnarray*}
Noticing that 
$$\cos^ j(2x)=\frac{1}{2^j}\sum_{k=0}^j {j\choose k} \cos(2(j-2k) x),$$ we finally obtain that with $I:=2J$

$$I=\int_{-\infty}^\infty \sum_{\ell=0}^n \frac{1}{x^{2n+2}}\; \;p_\ell(x) \cos (2\ell x)\;dx$$
where $p_\ell$ is a polynomial of degree at most $2n$ and with rational coefficients.

Next we consider the functions 
$$f(z):=\sum_{\ell=0}^n \frac{1}{z^{2n+2}}\; \;p_\ell(z) e^{2i\ell z}$$ and

$$F(z):=f(z)- \frac{p(z)}{z^{2n+2}},$$
where $ \frac{p(z)}{z^{2n+2}}= \frac{q(z)}{z^{2n+2}}+ \frac{r}{z}$ 
is the principal part of the meromorphic function $f$. Note that $\deg p\leq 2n+1$, $\deg q\leq 2n$, and $r\in \Q+i\,\Q$.

In particular,  $F$ has  a holomorphic extension to the origin,  hence is an entire function.  Therefore
 $\int_\Gamma F(z)dz=0$, 
  where $\Gamma$ is the boundary of the half-disk $|z|\leq R$, ${\rm Im}\,z \geq 0$, consisting of the half circle 
  $\Gamma_R$  and  the interval $[-R,R]$.
  Hence, by letting $R\to \infty$ and taking real parts,
  $$0={\rm Re}\lim_{R\to\infty} \int_{\Gamma_R} F(z)dz+ I.$$
  
  By Jordan's Lemma, $\limsup_{R\to\infty} \left|\int_{\Gamma_R} e^{inz} dz\right|<\infty$. Hence, by noticing that the differences of the degrees of the polynomials in the denominator and numerator  if $f$ is bigger than 2,   
  $$\lim_{R\to\infty} \int_{\Gamma_R} F(z)dz=0+0+ \lim_{R\to \infty} \int_{\Gamma_R}\frac{r}{z} dz=  i\,r\,\pi.$$
  We conclude that the value of the original integral $J$ is rational.

\newpage

.\vspace{-1cm}
   \begin{figure}[h!]
 
  \scalebox{0.6} 
  {\includegraphics{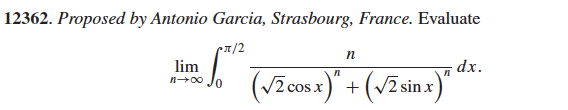}} 
\end{figure}


\centerline{\bf Solution to problem 12362  in Amer. Math. Monthly 129 (2022), 986} \medskip 

\centerline{Raymond Mortini }

\medskip

\centerline{- - - - - - - - - - - - - - - - - - - - - - - - - - - - - - - - - - - - - - - - - - - - - - - - - - - - - -}
  
  \medskip

We reduce the present problem to Problem 12340,  telling us that for each
 $f:[0,1]\to \R$  continuous, 
 \begin{equation}\label{old}
\lim_{n\to\infty}\frac{n}{2^n}\int_0^1 \frac{f(x)}{x^n+(1-x)^n}\;dx  =\frac{\pi}{4}\; f\left(\frac{1}{2}\right).
\end{equation}

First we note that one may replace of course $n$ by $t$, $t\to\infty$. Later we shall take $t=n/2$. 
As a result we obtain

\begin{equation}\label{mainsinus}
\ovalbox{$\dis \lim_{n\to\infty}\;\int_0^{\pi/2} \frac{n}{\Big(\sqrt 2 \cos x\Big)^n + \Big(\sqrt 2\sin x\Big)^n}\;dx=\frac{\pi}{2}$}
\end{equation}

\medskip

To see this, let $u:=\sin x$. Then, $dx=(1-u^2)^{-1/2}\;du$ and so
with
$$I_n:=\int_0^{\pi/2} \frac{n}{\Big(\sqrt 2 \cos x\Big)^n + \Big(\sqrt 2\sin x\Big)^n}\;dx,$$
we obtain
$$I_n=n\;\int_0^1 \frac{(1-u^2)^{-1/2}}{2^{n/2}\; (1-u^2)^{n/2}+ 2^{n/2}\; (u^2)^{n/2}}\;dx
$$

Now let $y:=u^2$.  Then $du=\frac{1}{2\sqrt y}\;dy$, and so
$$I_n= \frac{n}{2\cdot 2^{n/2}}\;\int_0^1 \frac{ y^{-1/2}(1-y)^{-1/2} }{(1-y)^{n/2} +y ^{n/2}}\; dy
$$

Let $g_{\e}(y)=(y+\e)^{-1/2}(1-y+\e)^{-1/2}$ and $g:=g_0$. Then
\begin{equation}\label{unten}
 \frac{n}{2\cdot 2^{n/2}}\;\int_0^1 \frac{g_\e(y)}{(1-y)^{n/2} +y ^{n/2}}\; dy\leq  I_n.
 \end{equation} 
Next we estimate from above.
Let $x\in [0,1]$ satisfy  $|x-1/2|\geq \delta$, where $\delta>0$ is small. Then, for $t\geq 1$, 
$$x^t +(1-x)^t \geq (1/2+\delta)^t +(1/2 -\delta)^t.$$
Hence 
$$\frac{t}{2^t}\;\frac{1}{x^t +(1-x)^t}\leq \frac{t}{(1+2\delta)^t+(1-2\delta)^t}=:m_t\to 0\; \text{as $t\to\infty$}$$

Now for $\e>0$, choose $\delta$ so small that $|g(x)-g(1/2)|<\e$ for $|x-1/2|\leq \delta $. Then
\begin{eqnarray*}
\frac{t}{2^t}\int_0^1  \frac{g(x)}{x^t+(1-x)^t}\;dx&\leq & \frac{t}{2^t}\int_{|x-1/2|\leq \delta} \frac{g(1/2)+\e}{x^t+(1-x)^t}\; dx+
m_t\int_{|x-1/2|>\delta}g(x)\; dx\\
&\leq& \frac{t}{2^t}\int_0^1 \frac{g(1/2)+\e}{x^t+(1-x)^t}\; dx +m_t ||g||_1\buildrel\longrightarrow_{t\to\infty}^{} \frac{\pi}{4} (g(1/2)+\e).
\end{eqnarray*}
Together with (\ref{unten}), we obtain that
$$\frac{\pi}{4}\;g_\e(1/2)\leq \liminf_n I_n\leq \limsup_n I_n\leq \frac{\pi}{4} (g(1/2)+\e).$$
Hence, by letting $\e\to 0$, 
$$\lim_n I_n=\frac{\pi}{4}\;g(1/2)=\pi/2.$$

\newpage

   \begin{figure}[h!]
 
  \scalebox{0.5} 
  {\includegraphics{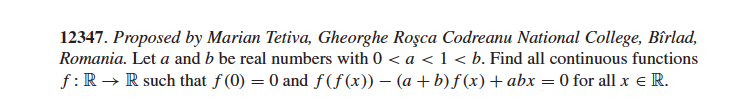}} 
\end{figure}

\centerline{\bf Solution to problem 12347  in Amer. Math. Monthly 129 (2022), 786} \medskip 

\centerline{Raymond Mortini }
\bigskip

\centerline{- - - - - - - - - - - - - - - - - - - - - - - - - - - - - - - - - - - - - - - - - - - - - - - - - - - - - -}
  
  \medskip

We show that on $\R$ there are exactly 4  continuous solutions to the functional equation
\begin{equation}\label{equ}
f(f(x))-(a+b)f(x) +abx=0, 
\end{equation}
whenever $f(0)=0$ and $0<a<1<b$. Namely

$$\mbox{$F_1(x)=ax$, $F_2(x)=bx$, $F_3(x)=\begin{cases} ax &\text{if $x\leq 0$}\\ bx &\text{if $x>0$}
\end{cases}$ and 
$F_3(x)=\begin{cases} bx &\text{if $x\leq 0$}\\ ax &\text{if $x>0$}.
\end{cases}$
}$$

It is easy to check that $F_j$ are solutions. Now suppose that $f$ is a solution.

\co {\Large i) } $f$ is injective:  let $f(x)=f(y)$. Then 
$$abx=-f(f(x))+(a+b)f(x)=-f(f(y))+(a+b)f(y)=aby$$
and so $x=y$.

\co {\Large ii)} $f$ is strictly increasing:  monotonicity implies that $M^\pm:=\lim_{x\to\pm\infty} f(x)$ exists in $[-\infty,\infty]$. Now $M^\pm$ cannot be finite, since (1) and continuity would imply that $f(M^\pm)-(a+b) M^\pm+ \pm\infty=0$, which is impossible.  
But $M^+\not=-\infty$, either,  since otherwise
$$\mbox{$f(f(x))+abx =(a+b)f(x)\to -\infty$ as $x\to\infty$},$$
and so $\dis \lim_{x\to\infty} f(f(x))=-\infty$. Hence, with $y:=f(x)\to-\infty$, we deduce that
$$\lim_{y\to-\infty} f(y)=-\infty=\lim_{x\to\infty} f(x),$$
 contradicting the monotonicity of $f$.  We conclude that $f$ is strictly increasing,  
 $f(x)\geq 0$ for $x\geq 0$, $f(x)\leq 0$ for $x\leq 0$, and 
$\lim_{x\to-\infty} f(x) =-\infty$, $\lim_{x\to\infty} f(x) =\infty$.\\

\co {\Large iii)} The inverse $h:=f^{-1}:\R\to \R$ satisfies the functional equation

\begin{equation}
h(h(y))-\left(\frac{1}{a}+\frac{1}{b}\right) h(y) +\frac{1}{ab}\; y =0.
\end{equation}

Just take $x:=h(h(y))$ in (\ref{equ}) and note that $h\circ f=f\circ h={\rm id}$. Then

$$y-(a+b)h(y)+ab\; h(h(y))=0.$$
Now divide by $ab$.  We also deduce the following  identity:
\begin{equation}
\big[f(y)-ay\big] +ab \big[ f^{-1}(y) - (1/a)\;y\big]=0.
\end{equation}
In particular $f^{-1}$ is increasing, too.\\

\co {\Large iv)}  The only fixed point of $f$ is $0$:  let $f(s)=s$. If $s\not=0$, then, by (\ref{equ})
$s-(a+b)s +ab s=0$. Thus $1+ab=a+b$, or equivalently, $b(a-1)=a-1$. That is, $b=1$ (since $a<1$).  A contradiction.
We conclude that for $x>0$ either $f(x)<x$ or $f(x)>x$ for every $x>0$.\\

\co {\Large v)} Let $f_n:=\underbrace{f\circ \cdots\circ f}_{n\text{-times}}$ be the $n$-th iterate of $f$ \footnote{ We never use the exponent $n$ to designate the $n$-th iterate when working with functions, as the risk to mix it up with the $n$-th power is too   big.}.

$\bullet$ Suppose that there is $x_0>0$ such that $f(x_0)<x_0$. Then $ f_n(x)\to 0$ for every $x\geq 0$.  Indeed,  by iv), $0<f(x)<x$ for $x>0$. Hence
$$f_{n+1}(x)=f(f_{n}(x))< f_{n}(x)$$
and so $M(x):=\lim_{n\to\infty} f_{n}(x) $ exists for every $x>0$. Plugging 
$f_{n}(x)$ into the functional equation (\ref{equ}), yields
$$M(x)-(a+b)M(x)+ab M(x)=0.$$
Consequently, $M(x)(1+ab-(a+b))=0$.  But $1+ab-(a+b)=(1-a)+b(a-1)=(1-a)(1-b)\not=0$. Hence $M(x)=0$.\\
$\bullet$ Suppose that there is $x_0>0$ such that $f(x_0)>x_0$.  Then, by iv) $f(x)>x$ for every $x>0$ and  the 
sequence $(f_n(x))$ of iterates is increasing for  each $x>0$.  As its limit $M(x)$ can't be finite, in particular not $0$, we see that 
$\lim_{n\to\infty} f_n(x)=\infty$ for every $x>0$.\\

\co {\Large vi)} For each $x\in \R$ we obtain the following three terms difference equations:
$$f_{n+2}(x)-(a+b)f_{n+1}(x) +ab f_{n}(x)=0,$$
with initial condition $f_0(x):=x$ and $f_1(x):=f(x)$.

The associated characteristic polynomial  is $p(z)=z^2-(a+b)z+ab$, which has as roots $a$ and $b$. Hence, there exist
real coefficients  $A_x$ and $B_x$ depending on the initial value $x$ such that
\begin{equation}\label{sol}
f_{n}(x)= A_x a^n +B_x b^n.
\end{equation}

If $f(x)<x$ for every $x>0$, then    $\lim_n f_{n}(x)=0$ implies that $B_x=0$, because  $b>1$ and $0<a<1$. Hence
$f(x)=A_x a$.  As the initial value $f_{0}(x)$ equals $x$, we deduce from (\ref{sol}) that $A_x=x$. Thus, for $x>0$, 
$f(x)=ax$ whenever there exists $x_0>0$ with $f(x_0)<x_0$.

If $f(x)>x$ for every $x>0$, then  $x>f^{-1}(x)$ (note that by iii) $h:=f^{-1}$ is increasing). Hence, the difference equations,
\begin{equation}\label{equa2}
h_{n+2} (x)-\left(\frac{1}{a}+\frac{1}{b}\right) h_{n+1}(x) +\frac{1}{ab}\; h_n(x) =0\end{equation}
with initial values $h_0(x)=x$
have  for $x>0$ the solutions
\begin{equation}\label{equ3}
h_n(x)=C_x \frac{1}{a^n}+D_x \frac{1}{b^n}
\end{equation}
for real coefficients $C_x$ and $D_x$.
Using (\ref{equa2}), we see as above that  $\lim_{n\to\infty} h_n(x)=0$. Hence $C_x=0$. Thus $h(x)=h_1(x)= D_x \frac{1}{b} $. As $h_0(x)=x$, we deduce  from (\ref{equ3}) that $D_x=x$. 

To some up, $f^{-1}(x)=h(x)= x/b$ and so $f(x)=bx$ for every $x>0$ whenever  there exists $x_0>0$ with $f(x_0)>x_0$.\\

\co {\Large vii)}
The case for negative arguments follows from the observation that if $f$ is a solution to (\ref{equ}), then the function
$g$ given by $g(x)=-f(-x)$ is a solution, too:
\begin{eqnarray*}
g(g(x))-(a+b)g(x)+ab x&=&-f(-f(-x))+(a+b) f(-x) +abx\\
& =&-\bigg(f(f(-x))-(a+b)f(-x)+ab (-x)\bigg)=0.
\end{eqnarray*}


\newpage

   \begin{figure}[h!]
 
  \scalebox{0.5} 
  {\includegraphics{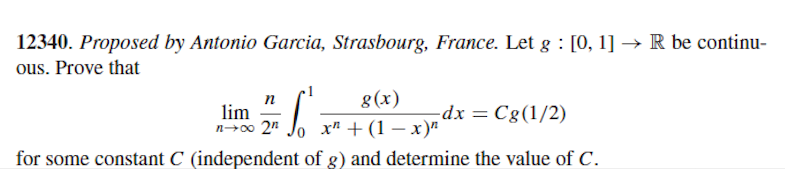}} 
\end{figure}

\medskip


\centerline{\bf Solution to problem 12340  in Amer. Math. Monthly  129 (2022),  686} \medskip 

\centerline{Raymond Mortini and  Rudolf Rupp}

\bigskip

\centerline{- - - - - - - - - - - - - - - - - - - - - - - - - - - - - - - - - - - - - - - - - - - - - - - - - - - - - -}
  
  \medskip

As $g:[0,1]\to \R$ is continuous, $||g||_\infty=\max\{|g(x)|: 0\leq x\leq 1\}<\infty$.
Let 
$$I_n:=\frac{n}{2^n}\int_0^1 \frac{g(x)}{x^n+(1-x)^n}\;dx.$$
We claim that $\lim_{n\to\infty} I_n= (\pi/4) g(1/2)$.

To see this, 
we split the integral into two parts and use two different change of variables:
\begin{eqnarray*}
I_n&=&\frac{n}{2^n}\int_0^{1/2} \underbrace{\frac{g(x)}{x^n+(1-x)^n}\;dx}_{x=:\frac{1}{2}-\frac{s}{2n}}
+ \frac{n}{2^n}\int_{1/2}^{1} \underbrace{\frac{g(x)}{x^n+(1-x)^n}\;dx}_{x=:\frac{1}{2}+\frac{s}{2n}}\\
&=&
\frac{n}{2^n}\int_{0}^n\frac{g(\frac{1}{2}-\frac{s}{2n})}{{(\frac{1}{2}-\frac{s}{2n})}^n+{(\frac{1}{2}+\frac{s}{2n})}^n}\;\frac{1}{2n}ds
+ \frac{n}{2^n}\int_{0}^n\frac{g(\frac{1}{2}+\frac{s}{2n})}{(\frac{1}{2}+\frac{s}{2n})^n+(\frac{1}{2}-\frac{s}{2n})^n}\;\frac{1}{2n}ds
\\
&=&\frac{1}{2} \int_0^n \frac{g(\frac{1}{2}-\frac{s}{2n})+g(\frac{1}{2}+\frac{s}{2n})}{(1-\frac{s}{n})^n +(1+\frac{s}{n})^n}\;ds.
\end{eqnarray*}

Note that $n\mapsto (1+\frac{s}{n})^n$ is increasing; so the integrand is dominated for $s\geq 1$ by
$$
\frac{||g||_\infty}{(1+\frac{s}{2})^2}\leq ||g||_\infty 4 s^{-2}.
$$

Hence, as $n\to\infty$,  
\begin{eqnarray*}
\lim_{n\to\infty} I_n&=&\frac{1}{2} 2 g(1/2)\int_0^\infty\frac{ds}{e^{-s}+e^s}\;ds\\
&=& g(1/2)\int_0^\infty \frac{e^s}{1+(e^s)^2}\;ds\\
&=& g(1/2) \big[\arctan e^s\big]^\infty_0=g(1/2)\left(\frac{\pi}{2}-\frac{\pi}{4}\right)\\
&=& \frac{\pi}{4}g(1/2).
\end{eqnarray*}

Generalizations appear in \cite{mo23}.

\newpage

   \begin{figure}[h!]
 
  \scalebox{0.6} 
  {\includegraphics{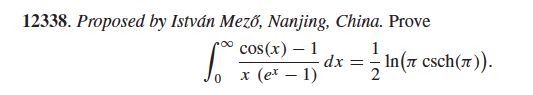}} 
\end{figure}

\centerline{\bf Solution to problem 12338  in Amer. Math. Monthly  129 (2022),  686} \medskip 

\centerline{Raymond Mortini and  Rudolf Rupp}

\bigskip

\centerline{- - - - - - - - - - - - - - - - - - - - - - - - - - - - - - - - - - - - - - - - - - - - - - - - - - - - - -}
  
  \medskip

Let $\dis f(x):=\frac{\cos x -1}{x(e^x-1)}$ and $I:=\int_{0}^\infty f(x)dx$. Note that $\lim_{x\to 0} f(x)=1/2$, that $f$ is bounded,  and that the integral converges (absolutely).
Using the Laplace transform of $f$, we are going to show that 
$$I=\frac{1}{2} \log\left(\frac{\pi}{\sinh \pi}\right).$$
So let 
$$F(s):=\int_0^\infty e^{-sx} f(x) dx.$$
Also this integral converges absolutely and uniformly in $s\geq 0$, as the integrand is dominated on $[1,\infty[$ by
$4 e^{-x}$. Moreover, $F$ is continuous on $[0,\infty[$ with $F(0)=I$.
Now, by a similar reason, 
$$G(s):=-\int_0^\infty xe^{-sx} f(x)dx$$
is absolutely convergent, as the integrand is dominated on $[1,\infty[$ by
$4 xe^{-x}$.
Hence $F'(s)=G(s)$. Moreover, by considering for $x>0$ the geometric series for $(1-e^{-x})^{-1}$,
\begin{eqnarray*}
G(s)&=& -\int_0^\infty \frac{e^{-(s+1)x}}{1-e^{-x}}(\cos x-1)\;dx=\int_0^\infty \sum_{k=0}^\infty e^{-(s+1+k)x} (1-\cos x)\;dx.
\end{eqnarray*}
As all the summands are positive,  Beppo Levi's monotone convergence theorem for Lebesgue integrals implies that
 $\int\sum=\sum\int$. Hence,  by using that  for $a>0$
$$\int_0^\infty e^{-ax}\cos xdx=\frac{a}{a^2+1},$$
we obtain
\begin{eqnarray*}
G(s)&=&\sum_{k=0}^\infty \int_0^\infty e^{-(s+1+k)x} (1-\cos x)\;dx=\sum_{k=0}^\infty\left( \frac{1}{s+1+k}-\frac{s+1+k}{(s+k+1)^2 +1}\right)\\
&\buildrel=_{}^{n=k+1}&\sum_{n=1}^\infty \frac{1}{(s+n)^3 +(s+n)}.
\end{eqnarray*}
The convergence being absolut and uniform on $[0,\infty[$ (a majorant is given by $\sum_{n=1}^\infty n^{-3}$), we can integrate termwise to re-obtain $F$. Note that a primitive $P$ of $G$ on $[0,\infty[$ is given by
\begin{eqnarray*}
P(s)&=&\sum_{n=1}^\infty \int \frac{1}{(s+n)^3 +(s+n)}ds =\sum_{n=1}^\infty \log \frac{s+n}{\sqrt{(s+n)^2+1}}\\
&=&-\frac{1}{2}\sum_{n=1}^\infty\log\frac{(s+n)^2+1}{(s+n)^2}
=-\frac{1}{2}\sum_{n=1}^\infty\log\left(1+\frac{1}{(s+n)^2}\right).
\end{eqnarray*}
Now $F=P+c$ for some constant $c$.  Since $P$ is uniformly convergent, it easily follows that $\lim_{s\to\infty} P(s)=0$
(just take a tail uniformly small, and use that the limit of the remaining finitely many summands is 0). But also
$\lim_{s\to\infty} F(s)=0$, because $|F(s)|\leq ||f||_\infty \int_0^\infty e^{-st}=||f||_\infty/s$.

Hence $c=0$ and so $\dis F(0)= -\frac{1}{2}\sum_{n=1}^\infty \log\left(1+\frac{1}{n^2}\right)$. Next we use that 
$$\sinh(\pi z)=\pi z \prod_{n=1}^\infty \left(1+\frac{z^2}{n^2}\right).$$
So, 
$$\log\sinh(\pi)=\log \pi +\sum_{n=1}^\infty\log \left(1+\frac{1}{n^2}\right).$$
Hence
$$F(0)=I=-\frac{1}{2}\log\sinh(\pi)+\frac{1}{2}\log \pi=\frac{1}{2} \log \frac{\pi}{\sinh\pi}.
$$

\newpage

  \begin{figure}[h!]
 
  \scalebox{0.6} 
  {\includegraphics{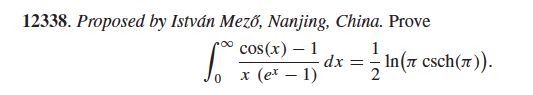}} 
\end{figure}

{\it  A different solution to problem 12338  in Amer. Math. Monthly  129 (2022),  686} \medskip 

\centerline{Raymond Mortini and  Rudolf Rupp}

Let $I:=\dis\int_0^\infty \frac{\cos x -1}{x(e^x-1)}\;dx$. We show that 
$$I=\frac{1}{2} \log\left(\frac{\pi}{\sinh \pi}\right).$$

Recall that $\N=\{0,1,2,\dots\}$.
First we develop for $x>0$ the integrand into a double, absolutely convergent series (so this is independent of the  arrangement):

\begin{eqnarray*}
g(x):= \frac{\cos x -1}{x(e^x-1)}=\frac{\cos x-1}{x} \frac{e^{-x}}{1-e^{-x}}&=&\sum_{n=1}^\infty\frac{(-1)^n}{(2n)!}x^{2n-1}\; 
 \sum_{k=1}^\infty e^{-kx}=\sum_{k=1}^\infty \sum_{n=1}^\infty\underbrace{\frac{(-1)^n}{(2n)!}x^{2n-1}\; e^{-kx}}_{:=a_{kn}}\\
 &=& \sum_{k=1}^\infty \sum_{n\;{\rm even}}^\infty\frac{1}{(2n)!}x^{2n-1}\; e^{-kx} -
  \sum_{k=1}^\infty \sum_{n\;{\rm odd}}^\infty\frac{1}{(2n)!}x^{2n-1}\; e^{-kx}.
\end{eqnarray*}
Note that $\lim_{x\to 0} g(x)=1/2$, but that both absolutely convergent double series at the right vanish at $0$. 
  Beppo Levi's monotone convergence theorem for Lebesgue integrals applied twice, 
 gives
\begin{equation}\label{beppo}
\mbox{$\dis\int\sum_{k\geq 2}\sum_{n\;\rm  even}=\sum_{k\geq 2}\sum_{n\;\rm  even}\int$\;\; and \;\;
$\dis\int\sum_{k\geq 2}\sum_{n\;\rm  odd} =\sum_{k\geq 2}\sum_{n\;\rm odd}\int .$}
\end{equation}
 
 As the calculations below show, the sums $\sum_{n\;\rm odd}\int |a_{kn}|$ and 
 $\sum_{n\;\rm  even}\int |a_{kn}|$ converge  for  $k\geq 2$, but diverge for $k=1$, though
 $\sum_n \int a_{1n}$ converges.  Moreover, $\sum_{k\geq 2}\sum_{n\;\rm odd}\int$ and  $\sum_{k\geq 2}\sum_{n\;\rm  even}\int$ are finite; hence (by (\ref{beppo})) ,
 \begin{equation}\label{ds}
\sum_{k\geq 2}\sum_n\int =\int \sum_{k\geq 2}\sum_n.
\end{equation}
 
 So, at the end, by adding in (\ref{ds}) the term $\sum_n \int a_{1,n}$, respectively $\int\sum_n  a_{1,n}$ (which coincide, too; see addendum) we see that
 $$\int \sum_k\sum_n =\sum_n\sum_k\int.$$
 
 To complete the calculations, we use  that for $m\in \N$ and $k\in \N\setminus\{0\}$,
$$\int_0^\infty x^me^{-kx}dx=m! / k^{m+1}.$$
Hence
\begin{eqnarray*}
&&\sum_{k=1}^\infty\sum_{n=1}^\infty  \frac{(-1)^n}{(2n)! }\int_0^\infty 
x^{2n-1}\; e^{-kx}\;dx=\sum_{k=1}^\infty\sum_{n=1}^\infty  \frac{(-1)^n}{(2n)! }\frac{(2n-1)!}{k^{2n}}=\\
&&\sum_{k=1}^\infty\left(\sum_{n=1}^\infty  \frac{(-1)^n}{2n}\frac{1}{k^{2n}}\right)=-\frac{1}{2}\sum_{k=1}^\infty\log\left(1+\frac{1}{k^2}\right),
\end{eqnarray*}
where the last identity comes from the fact that for $0\leq y\leq 1$
$$h(y):=\sum_{n=1}^\infty \frac{(-1)^n}{2n} y^{2n}=-\frac{1}{2} \log(1+y^2)$$
(note that for $y=1$ there is no absolute convergence).
Next we use that 
$$\sinh(\pi z)=\pi z \prod_{n=1}^\infty \left(1+\frac{z^2}{n^2}\right).$$
So, 
$$\log\sinh(\pi)=\log \pi +\sum_{n=1}^\infty\log \left(1+\frac{1}{n^2}\right).$$
Hence
$$I=-\frac{1}{2}\log\sinh(\pi)+\frac{1}{2}\log \pi=\frac{1}{2} \log \frac{\pi}{\sinh\pi}
$$

{\bf Addendum}\bigskip

\begin{equation}\label{Lpt}
J:=\int_0^\infty e^{-x} \frac{\cos x-1}{x}dx= \sum_{n=1}^\infty \frac{(-1)^n}{2n}=- \frac{1}{2}\log 2.
\end{equation}

First we note that, as above, $ \sum_n\int a_{1,n}= \sum_{n=1}^\infty \frac{(-1)^n}{2n}=- \frac{1}{2}\log 2$.
To show that $J=- \frac{1}{2}\log 2,$
 we interprete this as the Laplace transform
$L(q)(s)$ of the function $q(x)=(\cos x-1)/x$ evaluated at $s=1$. By a well-known formula, if  $L(F(t))(s)=f(s)$, then
$$L(q)(s)=L( \frac{F(t)}{t})(s)=\int_s^\infty f(u)du,$$
where 
$$f(s)=\int_0^\infty e^{-st} (\cos t -1)\;dt=\frac{1}{s^3+s}.$$
Hence $L(q)(s)=-\frac{1}{2} \log(1+s^{-2})$ and so
$J=L(q)(1))=- \frac{1}{2}\log 2$.

Another way to calculate the Laplace transform $J(s):=L(q)(s)$ of $q$ is to take derivatives: 
$$J'(s)=-\int_0^\infty e^{-st}(\cos t-1) dt=f(s).$$
Note that $\frac{d}{ds}\int=\int \frac{d}{ds}$, since both integrands are locally (in $s$) dominated by $L^1[0,\infty[$ functions.\\

{\bf Remark} This integral $J$ appears also on the web, see  \cite{yout}.

\newpage

   \begin{figure}[h!]
 
  \scalebox{0.5} 
  {\includegraphics{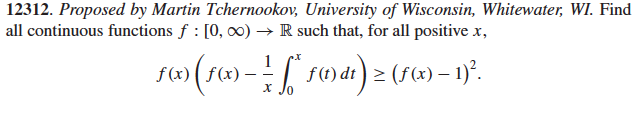}} 

\end{figure}

\centerline{\bf Solution to problem 12312, AMM 129 (3) (2022), p. 286} \medskip 

\centerline{Gerd Herzog, Raymond Mortini}

\bigskip

\centerline{- - - - - - - - - - - - - - - - - - - - - - - - - - - - - - - - - - - - - - - - - - - - - - - - - - - - - -}
  
  \medskip

\centerline{
\fbox{\parbox{9.6cm}{We show that  the constant function $1$ is the only solution}}
}

\medskip

Let $y=y(x):=\int_0^x f(t)dt$ and suppose that the continuous function $f:[0,\infty[\to\R$  satisfies on $]0,\infty[$ 
$$
f(x)\left( f(x)-\frac{1}{x}\int_0^x f(t)dt\right) \geq (f(x)-1)^2.
$$
Then 

\begin{equation}\label{dugl}
y'\left(2-\frac{y}{x}\right)\geq 1 \text{\;for $x>0$ and $y(0)=0$}.
\end{equation}
Note that this implies that $y'(0)=1$, because, by letting $x\to 0$, 
$$y'(0)(2-y'(0))\geq 1\iff (y'(0)-1)^2\leq 0$$

Let  the function $w:[0,\infty[\to \R$ be given by 
 $$w(x):=\begin{cases} \frac{y(x)}{x} &\text{if $x>0$}\\ y'(0)  &\text{if $x=0$}.
 \end{cases}
 $$
 Then $w\in C([0,\infty[)\inter C^1(\,]0,\infty[)$. We claim that 
 \begin{equation}\label{clai}
\mbox{$w(x)=1$ for every $x\geq 0$,}
\end{equation}
 from which we conclude that $y(x)=x$ and so $f(x)=y'(x)=1$  for $x\geq 0$.
 
 To see this, note that by (\ref{dugl}), $w(x)\not=2$. Since $w$ is continuous on $[0,\infty[$, $w(0)=1$,  and $w$ does 
 not take the value 2, we have that 
 $w(x)< 2$ for each $x>0$.  Hence,
for $x>0$,
\begin{eqnarray}\label{ungl}
w'(x)&=&\frac{x y'(x)-y(x)}{x^2}\geq \frac{1}{x}\left(\frac{1}{2-w(x)}-w(x)\right)\nonumber\\
&=&\frac{1}{x}\cdot\frac{(1-w(x))^2}{2-w(x)}
\end{eqnarray}

 Thus we may  deduce from  (\ref{ungl}) that  $w'\geq 0$; that is $w$ is increasing \footnote{ \;in the weak sense; or funnily called nondecreasing, a very ambiguous word.}.

 Now  suppose that (\ref{clai}) is not true.
 
 {\bf Case 1} There is $x_0>0$ with $w(x_0)<1$.   This is not possible, though, as $w$ is increasing, but $w(0)=1$.

{\bf Case 2}  There is $x_0>0$ with $w(x_0)>1$.  As $w$ is increasing, $w>1$ for $x\geq x_0$. Note that we already know that $w<2$.
Since the map $t\mapsto \frac{(1-t)^2}{2-t}$
is increasing on $[1,2[$, we deduce from  (\ref{ungl})  that for $x\geq x_0$
$$w'(x)\geq \frac{1}{x} \cdot\frac{(1-w(x_0))^2}{2-w(x_0)}=: c \frac{1}{x}.$$
Hence, by integration, for $x\geq x_0$,
$$w(x)\geq w(x_0)+ c\log(x/x_0)\to \infty ~~ (x\to \infty).$$
An obvious contradiction.\hfill $\square$

\newpage

 \begin{figure}[h!]
 
  \scalebox{0.45} 
  {\includegraphics{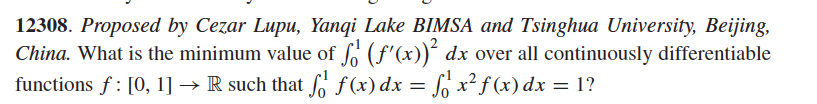}} 

\end{figure}


\centerline{\bf Solution to problem 12308, AMM 129 (3) (2022), p. 285 , by}\medskip

\centerline{Raymond Mortini}

\medskip

\centerline{- - - - - - - - - - - - - - - - - - - - - - - - - - - - - - - - - - - - - - - - - - - - - - - - - - - - - -}
  
  \bigskip\bigskip

\centerline{
\fbox{\parbox{14cm}{We show that  the minimal value is given by $105/2$ and is obtained by
 the polynomial $f(x)=-105/16 x^4+105/8 x^2 -33/16$}}
}

\medskip

Let $p$ be any polynomial. Then, by Cauchy-Schwarz,
$$\left(\int_0^1 f'p\; dx\right)^2 \leq  \left(\int_0^1 f'^2 dx\right)\, \left(\int_0^1 p^2 dx\right).$$

A  primitives of $f'p$ is given by $fp-\int fp' dx $.  Now choose $p$ so that $p(0)=p(1)=0$ and
$p'(x)=\alpha x^2+\beta $. To this end, put
$$p(x)=ax(x^2-1).$$ 
Then
\begin{eqnarray*}
I:=\int_0^1f'p\; dx&=& fp\big|^1_0-\int_0^1 f (3ax^2-a)dx=-3a+a=-2a
\end{eqnarray*}
Moreover,
$$\int_0^1 p^2 dx= a^2\int_0^1 (x^6+x^2-2x^4 )dx= a^2\left(\frac{1}{7}+\frac{1}{3}-\frac{2}{5}\right).$$
Hence
$$\int f'^2 dx\geq \frac{4a^2}{a^2\left(\frac{1}{7}+\frac{1}{3}-\frac{2}{5}\right)}=\frac{105}{2}.$$

Equality in the Cauchy-Schwarz inequality is given whenever $f'=p$. Thus
$$f(x)=\frac{a}{4} x^4-\frac{a}{2}x^2 +c.$$

Now $a$ and $c$ have to be chosen so that $\int f=\int x^2 f =1$. 
This yields the linear system
$$\begin{matrix} -7a+60c =60\\ -27a +140c=420
\end{matrix}
$$
whose solution is $a=-105/4$ and $c=-33/16$. Consequently
$$f(x)=-105/16 x^4+105/8 x^2 -33/16.$$
Note that
$$f'(x)^2=\big(-\frac{105}{4} x (x^2-1)\big)^2.$$

\newpage

   \begin{figure}[h!]
 
  \scalebox{0.5} 
  {\includegraphics{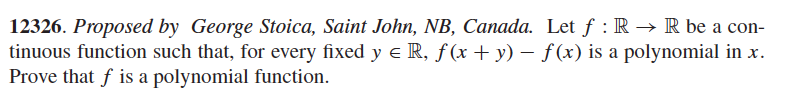}} 

\end{figure}


\centerline{\bf Solution to problem 12326, AMM 129 (5) (2022), p. 487} \medskip

\centerline{Raymond Mortini, Peter Pflug, Amol Sasane}

\bigskip

\centerline{- - - - - - - - - - - - - - - - - - - - - - - - - - - - - - - - - - - - - - - - - - - - - - - - - - - - - -}
  
  \medskip

By considering the symmetric function $p(x,y):=f(x+y)-f(x)-f(y)$ we get from the assumption that as well $p(\cdot,y)$ and $p(x,\cdot)$ are polynomials in their variables separately.  Hence, by \cite{ca}, $p(x,y)$ is  a polynomial.

{\bf Case 1} $f\in C^1(\R)$.
Write $p(x,y)=\sum a_{i,j}x^iy^j$ with  symmetrical coefficients and $a_{0,0}=-f(0)$ (the sum being finite of course)
If we take $y=0$, then for all $x$
$$-f(0)=f(x+0)-f(x)-f(0)=a_{0,0}+\sum  a_{i,0}x^i.$$
Hence $a_{i,0}=0$ for all $i\geq 1$. Due to symmetry, we also have $a_{0,j}=0$ for all $j\geq 1$. Thus we have only coefficients $a_{i,j}$ for $i,j\geq 1$.
Consequenlty
$$\frac{f(x+y)-f(x)-(f(y)-f(0)))}{y}=\sum_{i, j\geq 1}a_{i,j}x^iy^{j-1}.$$
As $f$ is assumed to be differentiable, we may take $y\to 0$ and get

$$f'(x)-f'(0)=\sum_{i\geq 1}a_{i,1}x^i.$$
Integration yields

$$f(x)-f(0)-xf'(0)=\sum_{i\geq 1}a_{i,1}\frac{x^{i+1}}{i+1}.$$
Thus $f$ is a polynomial.

{\bf Case 2}  $f\in C(\R)$. Let $F(x):=\int_0^x f(t)dt$ be  a primitive of $f$. Then with 
$$G(x,y):=F(x+y)-F(x)-F(y)$$
\begin{eqnarray*}
G(x,y)&=& \int_0^{x+y} f(t) dt -\int_0^x f(t) dt-\int_0^y f(t)dt\\
&\buildrel=_{t=y+s}^{}&\int_{-y}^x f(y+s) ds -\int_0^x f(t) dt-\int_0^y f(t)dt\\
&=& \int_{-y}^0  f(y+s) ds +\int_{0}^x\big( f(y+s)-f(s)\big)ds-\int_0^y f(t)dt\\
&\buildrel=_{t=y+s}^{}& \int_0^y f(t)dt+ \int_{0}^x\big( f(y+s)-f(s)\big)ds-\int_0^y f(t)dt\\
&=& \int_0^x p(y,s) ds + f(y) x
\end{eqnarray*}
which is a polynomial in $x$. Again, by symmetry, and the Carroll argument,   $G$ is a polynomial. Hence, by  Case 1, $F$ is a polynomial and so does $f=F'$.

\medskip

A detailed analysis of the functional equation $f(x+y)-f(x)-f(y)=p(x,y)$
(and based on these methods) appears in \cite{mps}.

\newpage

   \begin{figure}[h!]
 
  \scalebox{0.6} 
  {\includegraphics{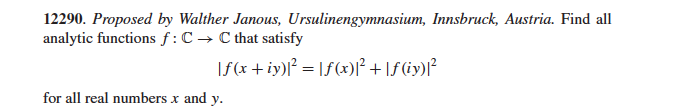}} 
\end{figure}
\centerline{\bf Solution to problem 12290  in Amer. Math. Monthly  128 (2021), 946}\medskip 

\centerline{Raymond Mortini and  Rudolf Rupp}

\medskip

\centerline{- - - - - - - - - - - - - - - - - - - - - - - - - - - - - - - - - - - - - - - - - - - - - - - - - - - - - -}
  
  \medskip

We show that all solutions are given by $az$, $b\sin(kz)$ and $c\sinh(kz)$ where $a,b,c\in \C$ and $k\in \R$. 

First we note that any solution $f$ necessarily satisfies $f(0)=0$.
Now let $h(z):=|f(z)|^2=(f \ov f)(z)$. Since  $f_x=f'$ and $ f_y=if_x=if'$, we see that $f_{xy}=(f')_y=i(f')_x=if''$. Moreover
$(\ov f)_x=\ov {f_x}$. Hence
\begin{eqnarray*}
h_{xy}&=&(f_x\ov f +f \ov f_x)_y= f_{xy} \ov f+f_x \ov f_{y} +f_y\ov f_x+f \ov f_{xy}\\
&=& 2{\rm Re}( f_{xy} \ov f)+0=2{\rm Re}(i f'' \ov f)= -2{\rm Im} (f'' \ov f).
\end{eqnarray*}
Now $|f(z)|^2=|f(x)|^2+|f(iy)|^2$ implies that the mixed derivative of the right hand side is $0$. We conclude that
${\rm Im} (f'' \ov f)=0$ in $\C$.  Let $U=\C\setminus Z(f)$, where $Z(f)=\{z\in\C: f(z)=0\}$. Then on $U$, this is equivalent to
$$0={\rm Im} \left(\frac{f''}{f} |f|^2\right)={\rm Im} \left(\frac{f''}{f}\right).$$

Thus, a necessary condition for $f\not\equiv 0$ being a solution is that $f''/f$ is a real constant $\lambda $. The differential 
equation $f''=\lambda  f$ in $\C$ has the solutions $az+d$ if $\lambda =0$, or $\alpha e^{\sqrt \lambda \, z}+\beta e^{-\sqrt\lambda \, z}$ if $\lambda >0$, and $\alpha e^{i\sqrt {|\lambda |}\, z}+\beta e^{-i\sqrt{|\lambda |}\,z}$  if $\lambda <0$. Since $f(0)=0$, we have $d=0$ and $\beta=-\alpha$. So, with $k:=\sqrt{|\lambda |}$, 
$$\mbox{$f(z)=az$,  $c\sinh kz$ if $\lambda >0$ and $c\sin kz$ if $\lambda <0$}.$$

It is now easy to check that these are  solutions indeed (wlog for $k=1$):

\begin{eqnarray*}
\sin(x+iy)&=&\cos(iy)\sin x +\cos x\sin(iy)\\
&=&\frac{e^{-y}+e^y}{2}\sin x -i  \cos x\frac{e^{-y}-e^y}{2}\\
&=&\cosh y\sin x+i \cos x \sinh y
\end{eqnarray*}

\begin{eqnarray*}
|\sin z|^2&=& \sin^2 x \cosh^2y +\cos^2 x\sinh^2y\\
&=& \sin^2 x \cosh^2y +(1-\sin^2 x)\sinh^2y\\
&=& \sin^2 x(\cosh^2 y-\sinh^2 y)+\sinh^2 y\\
&=& \sin ^2 x+\sinh^2 y\\
&=& \sin^2 x +|\sin^2(iy)|.
\end{eqnarray*}

as $\sin(iy)=i\sinh y$

\newpage
 
    \begin{figure}[h!]
 
  \scalebox{0.6} 
  {\includegraphics{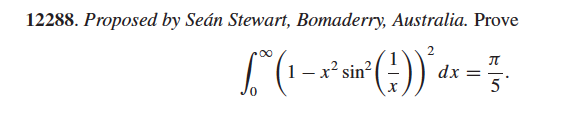}} 
\end{figure}


\centerline{\bf Solution to problem 12288  in Amer. Math. Monthly  128 (2021), 946}\medskip 

\centerline{Raymond Mortini and  Rudolf Rupp}

\bigskip

\centerline{- - - - - - - - - - - - - - - - - - - - - - - - - - - - - - - - - - - - - - - - - - - - - - - - - - - - - -}
  
  \medskip

A change of the variable $x\to 1/x$ yields that
$$J:=\int_0^\infty \left(1-x^2 \sin^2 \left(\frac{1}{x}\right)\right)^2\;dx=
\int_0^\infty \frac{(x^2-\sin^2 x)^2}{x^6}\;dx.
$$
Note that $$(x^2-\sin^2 x)^2=x^4-2x^2\sin^2 x+\sin^4 x.$$
Now  we "linearize" the trigonometric powers:
$\sin^2x=(1/2)(1-\cos 2x)$  and \\ $\sin^4 x=(3/8) -(1/2)\cos 2x +(1/8)\cos 4x$. 
Thus $J=I/2$, where

 $$I:= \int_{\mathbb R} \frac{\frac{3}{8}+x^4-x^2+(x^2-\frac{1}{2})\cos (2x) +\frac{1}{8}\cos(4x) }{x^6} \;dx.
  $$
  
  Next we  consider the meromorphic function
  $$f(z):=\frac{\frac{3}{8}+z^4-z^2+(z^2-\frac{1}{2})e^{2iz} +\frac{1}{8}e^{4iz} }{z^6}.$$
  Then we add in the numerator the polynomial
  $$p(z):=i\left( \frac{1}{2}z -\frac{4}{3}z^3 +\frac{2}{5}z^5\right),$$
  that is we consider the function
  $$F(z):= f(z)+\frac{p(z)}{z^6}.$$
  
  Note that this polynomial is chosen so that $F$ has a removable singularity at $z=0$ (in other words,  $-\frac{p(z)}{z^6}$ is the principal part in the Laurent expansion of $f$ around the origin). Hence $\int_\Gamma F(z)dz=0$, 
  where $\Gamma$ is the boundary of the half-disk $|z|\leq R$, ${\rm Im}\,z \geq 0$, consisting of the half circle 
  $\Gamma_R$  and  the interval $[-R,R]$.
  Hence, by letting $R\to \infty$ and taking real parts,
  $$0={\rm Re}\lim_{R\to\infty} \int_{\Gamma_R} F(z)dz+ I.$$
  
  By Jordan's Lemma, $\limsup_{R\to\infty} \int_{\Gamma_R} |e^{inz}| |dz|<\infty$. Hence,
  $$\lim_{R\to\infty} \int_{\Gamma_R} F(z)dz=0+0+ i\lim_{R\to \infty} \int_{\Gamma_R}\frac{\frac{2}{5}z^5}{z^6} dz=  -\frac{2\pi}{5}.$$
  We conclude that the value of the original integral $J$ is $\pi/5$.

 \newpage

   \begin{figure}[h!]
 
  \scalebox{0.6} 
  {\includegraphics{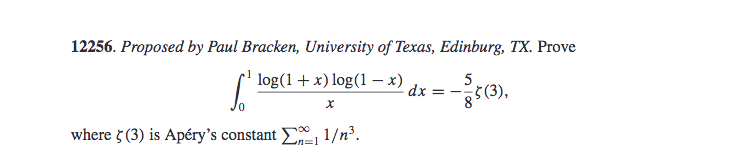}} 
\end{figure}

\centerline{\bf Solution to problem 12256  in Amer. Math. Monthly  128 (2021), 478}\medskip 

\centerline{Raymond Mortini and  Rudolf Rupp}

\medskip

\centerline{- - - - - - - - - - - - - - - - - - - - - - - - - - - - - - - - - - - - - - - - - - - - - - - - - - - - - -}
  
  \medskip

Using that $4ab=(a+b)^2-(a-b)^2$, we obtain

$$4\int_0^1 \frac{\log(1+x)\log(1-x) }{x}dx=\int_0^1 \frac{\log^2(1-x^2)}{x}dx -\int_0^1 \frac{\log^2\frac{1+x}{1-x}}{x}dx
=:I_1-I_2.$$

For $I_1$, we make the substitution $1-x^2=t^2$. Hence, due to $-xdx=t dt$,
$$I_1=\int_0^1 \frac{\log^2 t^2}{1-t^2} t\;dt$$

Using that $\int\sum=\sum\int$ (Lebesgue), and twice integration by parts,
$$I_1= 4\sum_{n=0}^\infty \int_0^1  t^{2n+1}\log^2 t  \;dt=
 8\sum_{n=0}^\infty \frac{1}{(2n+2)^3}=\xi(3).
$$
For the second one, $I_2$,  we make the substitution $t= \frac{1+x}{1-x}$. Then
$x=\frac{t-1}{t+1}$ and $dx= \frac{2}{(t+1)^2}dt$. Hence
$$I_2=2\int_1^\infty  \frac{\log^2 t}{1-t^2}\;dt\buildrel=_{}^{t=1/s}2\int_0^1 \frac{\log^2 s}{1-s^2}\;ds=
2\sum_{n=0}^\infty\int_0^1 s^{2n}\log^2 s\;ds= 4\sum_{n=0}^\infty \frac{1}{(2n+1)^3}= 4 \frac{7}{8}\xi(3).$$

Consequently, $4I= (1-\frac{7}{2})\xi(3)= -\frac{5}{2}\xi(3)$ and so 
$$\int_0^1 \frac{\log(1+x)\log(1-x) }{x}dx= -\frac{5}{8}\xi(3).$$

\newpage

\pagecolor{yellow}

 \begin{figure}[h!]
 
  \scalebox{0.5} 
  {\includegraphics{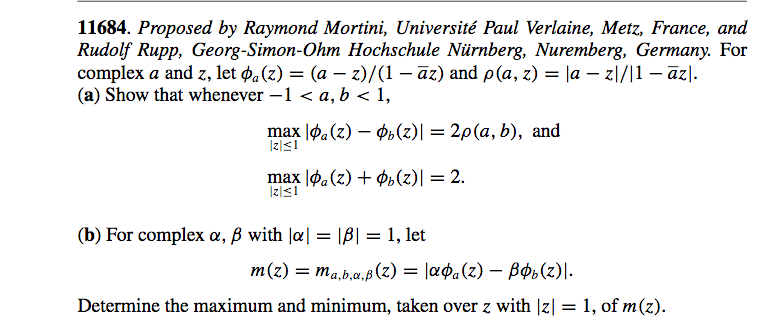} }

\end{figure}

{\sl \underline{original statement}}\\

Given  $a,b,\alpha,\beta \in \mathbb C$ with $|a|<1$,  $|b|<1$ and  $|\alpha|=|\beta|=1$,   let $\varphi_a(z)=(a-z)/(1-\overline a z)$ and $\rho(a,b)=|a-b|/|1-\ov a b|$ the pseudohyperbolic distance
between $a$ and $b$.\medskip

i) Show that whenever $a,b\in \;]-1,1[$, 
$$M^-:=\max_{|z|\leq 1} |\varphi_a(z)-\varphi_b(z)|=2\rho(a,b)  $$
and
$$M^+:=\max_{|z|\leq 1} |\varphi_a(z)+\varphi_b(z)|=2.$$

ii)  Determine
$$M:=\max_{|z|= 1} |\alpha \varphi_a(z)-\beta \varphi_b(z)|$$
and  
$$m:=\min_{|z|= 1} |\alpha \varphi_a(z)-\beta\varphi_b(z)|.$$
\vspace{1cm}

\centerline{\bf Solution to problem 11684  AMM  120 (2013),  76} \medskip 

\centerline{Raymond Mortini, Rudolf Rupp }

i) That $M^+=2$ is easy: just take $z=1$ and evaluate:
$$|\varphi_a(1)+\varphi_b(1)|=|-1-1|=2.$$
 Since $M^+\leq 2$, we are done. 
 \medskip
 
 ii) 
 We  first observe that $\phi_b$ is its own inverse.  
 Let $c= (b-a)/(1- a \ov b)$ and $\lambda= -(1-a\ov b)/(1-\ov a b)$.
   Since $\phi_b$ is a bijection of the unit circle onto itself,   
   $$\max_{|z|=1}|\alpha \varphi_a(z)-\beta\varphi_b(z)|=\max_{|z|=1}|\alpha\ov\beta\varphi_a(\varphi_b(z))-z|=
  \max_{|z|=1}|\alpha\ov \beta\lambda\varphi_c(z) - z|.$$
  The same identities hold when replacing the maximum with the minimum.\medskip

Put $\gamma:=\alpha\ov \beta\lambda$ and let  $-\pi <\arg\gamma\leq \pi$.
For $|z|=1$ we obtain
$$H(z):=|\gamma \phi_c(z)-z|=\left|\gamma \frac{z(c\ov z-1)}{1-\ov c z}-z \right|$$
$$=\left| \gamma \frac{1-c\ov z}{1-\ov c z}+1\right|=\left| \gamma \;\frac{w}{\;\ov w\;}+1\right|,$$ 
where $$w=1-c\ov z=1-c \frac{1}{z}.$$
If $z$ moves on the unit circle, then $w$ moves on the circle  $|w-1|=|c|$.
Let $w=|w|e^{i\theta}$. Then (see figure \ref{argument}) the domain of variation  of 
$\theta$ is the interval $[-\theta_m, \theta_m]$ with $|\theta_m|<\pi/2$ and $\sin\theta_m=|c|=\rho(a,b)$. Now

$$H(z)=|\gamma e^{2i\theta}+1|=2\mbox{$|\cos (\frac{\arg \gamma}{2} + \theta)$}|.
$$
Hence,  
$$M=\max_{|z|=1} H(z)=2\max\{\mbox{$|\cos (\frac{\arg \gamma}{2} + \theta)$}|: 
|\theta|\leq \arcsin(\rho(a,b))\}$$ and 
$$m=\min_{|z|=1} H(z)=2\min\{\mbox{$|\cos (\frac{\arg \gamma}{2} + \theta)$}|: 
|\theta|\leq \arcsin(\rho(a,b))\}.$$
In particular, if $a,b\in \;]-1,1[$ and $\alpha=\beta=1$, then $\gamma=-1$, and so
(using the maximum principle at $*$)
$$M^-\buildrel=_{}^{*}\max_{|z|=1} H(z)= 2\max \{|\sin \theta|:  |\theta|\leq \arcsin(\rho(a,b))\}=2\rho(a,b).$$
If $a,b\in \;]-1,1[$ and $\alpha=1, \beta=-1$, then $\gamma=1$, and so
$$M^+\buildrel=_{}^{*}\max_{|z|=1} H(z)= 2\max \{|\cos \theta|:  |\theta|\leq \arcsin(\rho(a,b))\}=2.$$

We note that $m=0$, that is $H(z_0)=0$ for some $z_0$ with $|z_0|=1$,
 if and only if $\gamma \phi_c$ has  a fixed point on the unit circle (namely $z_0$).
This is equivalent to the condition $|\cos (\frac{\arg \gamma}{2})|\leq |c|$.
Moreover, $M=2$ if and only if $|\sin (\frac{\arg\gamma}{2})|\leq |c|$.
\bigskip

 \begin{figure}[h]!
   \scalebox{.40} {\includegraphics{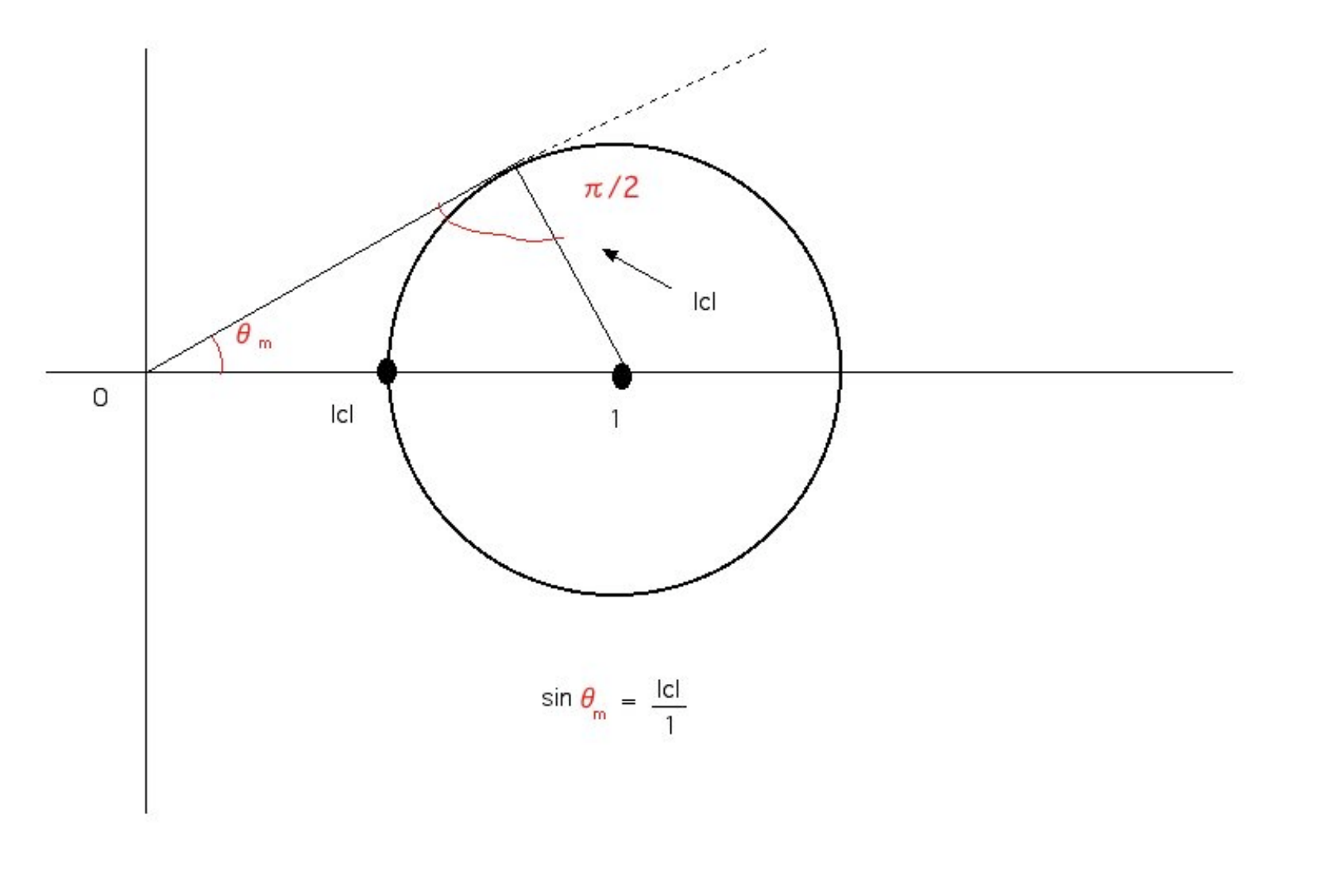}} 
\caption{\label{argument}The domain of variation of $\arg w$}
\end{figure}

\begin{figure}[h!]
 
  \scalebox{0.6} 
  {\includegraphics{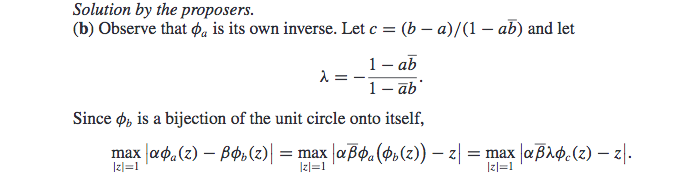}}
 \end{figure}

 \begin{figure}[h!]
 
  \scalebox{0.6} 
  {\includegraphics{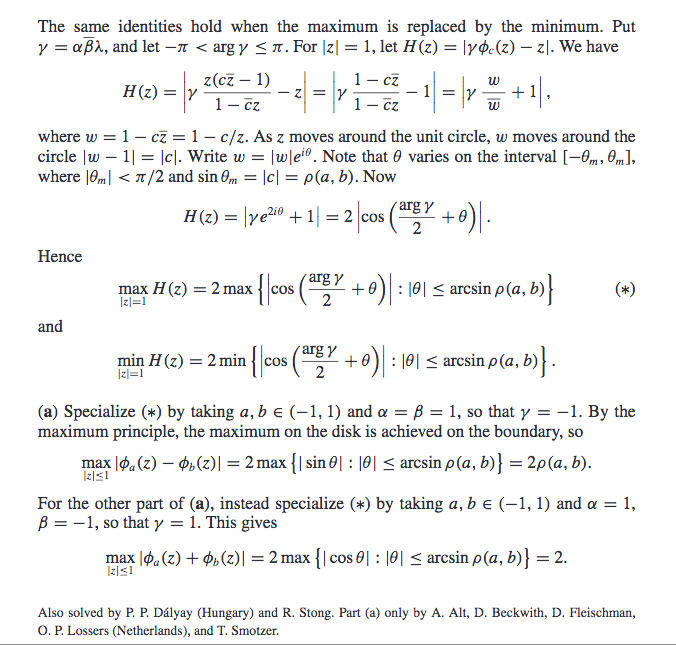} }

\end{figure}

\phantom{$Q$ obviously is continuous. 
If we suppose that $Q$ is constant $\kappa$, then this constant is necessarily 1 (just interchange $s$ with $s'$).
Now $x=(1-t)s+ts'$. Thus $x-s=t(s'-s)$ and $x-s'= (1-t)(s-s')$ and so $Q(s)=t/(1-t)$.
 Hence $1=\kappa=\frac{t}{1-t}$.
So $t=1/2$.  Now $x/||x||$ and $-x/||x||$ belong to $S$  and with $t=(1-||x||)/2$
we have $x= (1-t) \frac{x}{||x||} +t \frac{-x}{||x||}$.  So $t=1/2$ implies that $x=0$}

\vspace{2cm}
 \centerline{\copyright Mathematical Association of America, 2025. }

\newpage

\begin{figure}[h!]
 
  \scalebox{0.5} 
  {\includegraphics{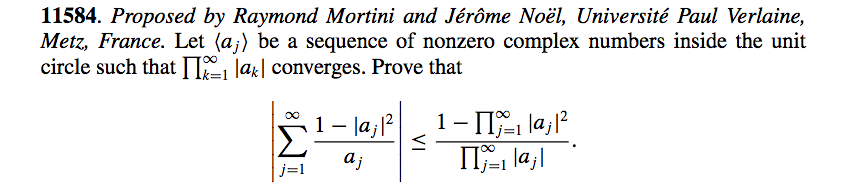}}
 \end{figure}

\centerline{\bf Solution to problem 11584 AMM 118 (2011), 558} \medskip 

\centerline{Raymond Mortini, J\'er\^ome No\"el }\medskip

\centerline{- - - - - - - - - - - - - - - - - - - - - - - - - - - - - - - - - - - - - - - - - - - - - - - - - - - - - -}
  
  \medskip

By the Schwarz-Pick inequality, $\frac{(1-|z|^2)|B'(z)|}{1-|B(z)|^2}\leq 1$
for any holomorphic self-map of the unit disk. Then, if we let $B$ be the Blaschke product
$$B(z)= \prod_{n=1}^\infty \frac{|a_n|}{a_n}\frac{a_n-z}{1-\overline a_n z}$$
associated with the zeros $(a_n)$, we get:

$$\frac{|B'(0)|}{1-|B(0)|^2}\leq 1.$$

But $$ \frac{B'(z)}{B(z)}=-\sum_{n=1}^\infty \frac{1-|a_j|^2}{(1-\overline a_j z) (a_j-z)}.$$
Hence
$$\left|\dis \sum_{j=1}^\infty \frac{1-|a_j|^2}{a_j}\right| = \frac{|B'(0)|}{ |B(0)|}\leq 
\frac{1-|B(0)|^2}{|B(0)|}=
\frac{1-\prod_{j=1}^\infty|a_j|^2}{\prod_{j=1}^\infty |a_j|}.$$

Motivation for posing this as a problem to AMM:  We are interested in a direct elementary proof.

\newpage
\nopagecolor

   \begin{figure}[h!]
 
  \scalebox{0.45} 
  {\includegraphics{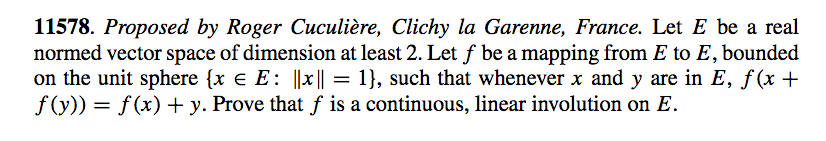}} 
\end{figure}

\centerline{\bf Solution to problem 11578  in Amer. Math. Monthly  118 (2011), 464} \medskip

\centerline{Raymond Mortini}

\medskip

\centerline{- - - - - - - - - - - - - - - - - - - - - - - - - - - - - - - - - - - - - - - - - - - - - - - - - - - - - -}
  
  \medskip

\begin{lemma}\label{center}

Let $0<||x||<1$ and  $s\in S$.  Let  $s'$ be the (second) uniquely determined intersection point of  the half-line starting at  $s$ and passing through $x$ with $S$. Then the 
map  $Q: S\to [0,\infty[, \; s\mapsto  ||x-s|| / ||x-s'||$ is a nonconstant continuous map.
 \end{lemma}
\begin{proof}
$Q$ obviously is continuous. 
If we suppose that $Q$ is constant $\kappa$, then this constant is necessarily 1 (just interchange $s$ with $s'$).
Now $x=(1-t)s+ts'$. Thus $x-s=t(s'-s)$ and $x-s'= (1-t)(s-s')$ and so $Q(s)=t/(1-t)$.
 Hence $1=\kappa=\frac{t}{1-t}$.
So $t=1/2$.  Now $x/||x||$ and $-x/||x||$ belong to $S$  and with $t=(1-||x||)/2$
we have $x= (1-t) \frac{x}{||x||} +t \frac{-x}{||x||}$.  So $t=1/2$ implies that $x=0$.
\end{proof}

\begin{lemma}\label{conn}
The unit sphere $S$ is connected whenever ${\rm dim}\; E\geq 2$.
\end{lemma}
\begin{proof}
Let $x,y\in S$, $x\not=y$. If $x$ is linear independent of $y$, then  the segment
$\{tx+(1-t)y: 0\leq t\leq 1\}$ does not pass through the origin; hence 
$$t\mapsto \frac{tx+(1-t)y}{||tx+(1-t)y||}$$
is a path joining $y$ with $x$ on $S$. 

If $y=\lambda  x$ for some $\lambda \in \mathbb R$, then 
we use the hypothesis that ${\rm dim}\; E\geq 2$
to guarantee the existence of a vector $u$ linear independent of $x$. Thus $v:=u/||u||\in S$.
By the first case, we may join $x$ with $v$ and then $v$ with $y$  by a path in $S$.
\end{proof}

The first  step is to show that $f(0)=0$. 

(1) Let $x=0$, $y=-f(0)$. Then $f(f(-f(0)))=f(0)-f(0)=0$;

(2) Let $x=y=0$. Then $f(f(0))=f(0)$;

(3) Let $x=-f(y)$. Then  $f(0)=f(-f(y))+y$. With $y=0$ this gives
$f(0)=f(-f(0))$.

(4) Applying $f$ yields $f(f(0))=f (f(-f(0)))\buildrel=_{(1)}^{}0$. 
Thus, by (2), $f(0)=0$. 

(5) Let $x=0$. Then $f(f(y))=f(0)+y=y$. Hence $f$ is an involution.

(6) $f$ is additive since 
$$f(x+y)\buildrel=_{(5)}^{}f(x+f(f(y)))=f(x)+ f(y).$$

(7) Next we show that $f$ is $\mathbb Q$-homogeneous by induction. 
Indeed, by (5), $$f((n+1)x)=f(nx + x)= f(\underbrace{nx}_X +f(\underbrace{f(x)}_Y))= f(nx)+f(x).$$
Thus $f(mx)=mf(x)$ for every $m\in \mathbb N$.

Now   $$0=f(0)=f(-x+x)\buildrel=_{(5)}^{}f(-x+f(f(x)))= f(-x)+f(x).$$
Thus $f(-x)=-f(x)$.
Hence, for $p\in\mathbb Z$, we have $f(px)=pf(x)$.

Next, if $n\in \mathbb N$, then 
$$n f(\frac{x}{n})= f(\frac{x}{n})+ (n-1) f(\frac{x}{n})= f(\frac{x}{n})+ f (\frac{n-1}{n}x)$$
$$=f\biggl(\underbrace{\frac{x}{n}}_X + f\Bigl(\underbrace{f\bigl(\frac{n-1}{n}x\bigr)}_Y\Bigr)\biggr)\buildrel=_{(5)}^{}
f\left(\frac{x}{n}+ \frac{n-1}{n}x\right)= f(x)$$
Hence $f\left(\frac{x}{n}\right)= \frac{1}{n} f(x)$. Therefore
$f\left(\frac{p}{n} \right)= \frac{p}{n} f(x)$ for $p\in \mathbb Z$ and $n\in \mathbb  N$.

\medskip
(8) By hypothesis, $||f(s)||\leq C$ for every $s\in S$.
Let $0<||x||<1$. Consider, as in Lemma \ref{center}, 
 the map $H: S\to [0,\infty[, s\mapsto ||x-s||/||x-s'||$. 
$H$ is continuous and non-constant. Since ${\rm dim}\; E\geq 2$,
 $S$ is connected by Lemma \ref{conn}. Hence $H(S)$ is an interval. In particular,
 there is $s\in S$ such that $r:=||x-s||/||x-s'||$ is rational. Thus, with $t= r/(1+r)$,
 $$x= (1-t)s  +ts'$$
 is  a rational convex-combination of two elements in the sphere.
  
Since $f$ is $\mathbb Q$-linear, we conclude that 
$$||f(x)||\leq (1- t) ||f(s)|| +t ||f(s')||\leq(1-t)C+tC=C.$$

Now let $x\in E$ be arbitrary. Choose a null-sequence  $\epsilon_n$ of positive numbers
so that $q_n:=||x||+\epsilon_n$ is rational. Then, $||x/q_n||\leq 1$. Since $f$ is
 $\mathbb Q$-linear, we obtain
$$||f(x)||=q_n ||f(x/q_n)||\leq q_nC.$$
Letting $n$ tend to infinity, we get 
$$||f(x)||\leq C ||x||.$$

Thus $f$ is continuous at the origin. Since $f$ is additive, we deduce that
$f$ is continuous everywhere; just use $f(x_0+x)=f(x_0)+f(x)\to f(x_0)$ if $x\to 0$.
\medskip

(9) It easily follows now that $f$ is homogeneous: if $\alpha\in \mathbb R$,
choose a sequence $(r_n)$ of rational numbers converging to $\alpha$.
Then, due to continuity,
$$f(\alpha x)=\lim_n r_n f(x)=\alpha f(x).$$

To sum up, we have shown that $f$ is a continuous linear involution.
\bigskip

{\bf Remarks}

If $n=1$, then the unit sphere $S$ is just a two point set, and so every function  is
automatically bounded on $S$. There exist, though, non-continuous  linear involutions
in $\mathbb R$. To this end, let $\mathcal B$ be  a Hamel basis of the $\mathbb Q$-vector space 
$\mathbb R$, endowed
with the usual Euclidean norm. We may assume that  $\mathcal B$ is dense in $\mathbb R$. 
Fix two elements $b_0$ and $b_1\in \mathcal B$. Let $f$ be defined
by $f(b_0)=b_1$, $f(b_1)=b_0$ and $f(b)=b$ if $b\in \mathcal B\setminus\{b_0,b_1\}$.
Linearly extend $f$ (in a unique  way). Then, obviously, $f$ is a linear involution.
But $f$ is not continuous at $b_0$. In fact, let $(b_k)_{n\geq 2}\in \mathcal B^{\mathbb N}$ converge to $b_0$. Then $f(b_k)=b_k\to b_0= f(b_1)\not=f(b_0)$.

\newpage

   \begin{figure}[h!]
 
  \scalebox{0.45} 
  {\includegraphics{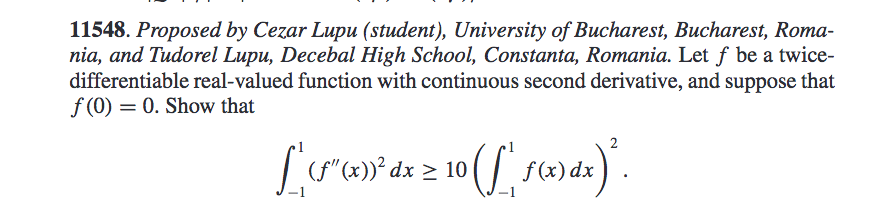}} 
\end{figure}

\centerline{\bf Solution to problem 11548  in Amer. Math. Monthly  118 (2011), 85}  \medskip

\centerline{Raymond Mortini and  J\'er\^ome No\"el}

\medskip

\centerline{- - - - - - - - - - - - - - - - - - - - - - - - - - - - - - - - - - - - - - - - - - - - - - - - - - - - - -}
  
  \medskip

Let $f\in C^2([-1,1])$, $f(0)=0$. Then 
$$\left(\int_{-1}^1 f(x) dx\right)^2\leq \frac{1}{10}\int_{-1}^1 (f''(x))^2 dx.$$
Moreover, the constant $1/10$ is best possible.\bigskip

{\bf Solution}
We consider the auxiliary integral 
$$I=\frac{1}{2} \left[  \int_0^1 (t-1)^2 f''(t) dt +\int_{-1}^0 (1+t)^2 f''(t)\right].$$

We first show that $I= \int_{-1}^1 f(t)dt$.  In fact, twice integration by parts yields:
$$\int_{0}^1 (t-1)^2 f''(t) dt= -f '(0) - 2\int_0^1(t-1) f '(t) dt=-f '(0) +2 \int_0^1 f(t)dt,
$$
as well as
$$\int_{-1}^0 (t+1)^2 f''(t) dt= f '(0) -2\int_{-1}^0 (t+1) f '(t) dt= f '(0) +2\int _{-1}^0 f (t) dt.$$
This proves the first claim. Now we use the Cauchy-Schwarz  inequality to estimate $I$:

$$\left(\int_0^1 (t-1)^2f''(t)dt\right)^2 \leq  \int_0^1 (t-1)^4 dt \int_0^1 (f''(t))^2 dt=\frac{1}{5}
\int _0^1 (f''(t))^2 dt,$$
and similarily for the second integral. Hence, by using that  $(A+B)^2\leq 2(A^2+B^2)$, we obtain

$$ I^2\leq 2 \frac{1}{4} \left(\frac{1}{5}\int_0^1  (f''(t))^2 dt + \frac{1}{5}\int_{-1}^0  (f''(t))^2 dt 
\right) = \frac{1}{10} \int_{-1}^1  (f''(t))^2 dt .
$$

The constant $1/10$ is obtained for the function
$$f(t)= \begin{cases} \frac{1}{12}t^4 + \frac{1}{3}t^3 + \frac{1}{2} t^2&\text{if $ -1\leq t\leq 0$}\\
\frac{1}{12}t^4 - \frac{1}{3}t^3 + \frac{1}{2} t^2&\text{if $ 0\leq t\leq 1$}.
\end{cases}$$
Indeed, this follows from the fact that in the Cauchy-Schwarz inequality we actually
have equality if the functions are colinear: $p ''(t)= (1+t)^2$ if $-1\leq t\leq 0$ and 
$p''(t) = (1-t)^2$ if $0\leq t \leq 1$.
A computation then shows that $$\left(\int_{-1}^1 p(x) dx\right)^2= \frac{1}{10}\int_{-1}^1 (p''(x))^2 dx=\frac{1}{25}.$$

{\bf Remark} If $f\in C^2([-1,1])$ satisfies  $f(1)=f(-1)=f '(1)=f '(-1)=0$,
then the inequality above holds, too. In fact.
$$ \int_{-1}^1 f(x) dx = \int_{-1}^1 1\cdot f(x) dx= xf(x)|^{1}_{-1}-\int_{-1}^1 xf '(x)dx =$$
$$=- \frac{1}{2}x^2 f '(x)|^{1}_{-1} + \frac{1}{2}\int_{-1}^1 x^2f ''(x) dx= \frac{1}{2}\int_{-1}^1 x^2f ''(x) dx
$$
Thus, by Cauchy-Schwarz,
$$\left(  \int_{-1}^1 f(x) dx\right)^2\leq \frac{1}{4}\int_{-1}^1 x^4 dx\int_{-1}^1 (f ''(x))^ 2 dx=
\frac{1}{4}\cdot \frac{2}{5}\int_{-1}^1 (f ''(x))^ 2 dx.$$

 \newpage
 
 \pagecolor{yellow}
 
 \begin{figure}[ht]
\scalebox{0.5}{\includegraphics{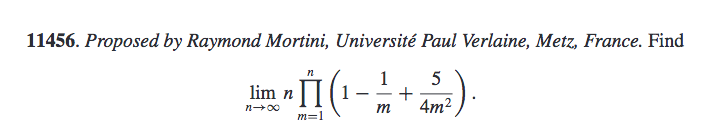}}
\end{figure}

\centerline {\bf Solution to problem 11456 AMM 116 (2009), 747}  \medskip
 
 \centerline{Raymond Mortini}

\medskip

\centerline{- - - - - - - - - - - - - - - - - - - - - - - - - - - - - - - - - - - - - - - - - - - - - - - - - - - - - -}
  
  \medskip

 $$a_m:=1-\frac{1}{m}+\frac{5}{4}\frac{1}{m^2}= \frac{1+\left(\frac{2m-1}{2}\right)^2}{m^2}$$

$$\prod_{m=1}^n  a_m=\frac{\prod_{m=1}^n \left(1+\left(\frac{2m-1}{2}\right)^2\right)}{\prod_{m=1}^n m^2}\;\frac{\prod_{m=1}^{2n}  m^2}{\prod_{m=1}^n (2m-1)^2\,\prod_{m=1}^n (2m)^2}$$

$$=\frac{\prod_{m=1}^n \left(\frac{1}{(2m-1)^2}+\frac{1}{4}\right)\; (2n)!^2}{4^n (n!)^4}=
\frac{\prod_{m=1}^n \left(\frac{4}{(2m-1)^2}+1\right)\; (2n)!^2}{16^n (n!)^4}.$$

Now, by Stirlings formula, 
$$\frac{(2n)!}{4^n n!^2}\sim \frac{(2n)^{2n}e^{-2n}\sqrt{4\pi n}}
{(n^n e^{-n} \sqrt{2\pi n})^2 2^{2n}}=\frac{1}{\sqrt{\pi n}}.$$

Since $\cos(\pi z)=\prod_{n=1}^\infty\left( 1-\frac{4z^2}{(2n-1)^2}\right)$, we have
 $$\lim_n n\prod_{m=1}^n a_m= \frac{\cos(\pi i)}{\pi}=\frac{\cosh \pi}{\pi}.$$

We note that 
$$\sqrt{\prod_{m=1}^n a_m} =\frac{1}{n!} \prod_{m=1}^n \left|i-\frac{2m-1}{2}\right|=
(n+1)\frac{2}{\sqrt 5}\frac{|f^{(n+1)}(0)|}{(n+1)!},$$
where $f(z)=(1-z)^{i+\frac{1}{2}}$, an interesting function in the Wiener algebra
(its Taylor coefficients  behave like $n^{-3/2}$ by the above calculations).

 \newpage
 \nopagecolor
 
 \begin{figure}[ht]
\scalebox{0.95}{\includegraphics{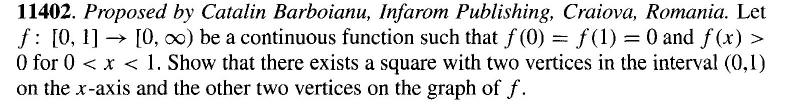}}
\end{figure}

\centerline  {\bf Solution to problem 11402, AMM 115 (10), (2008), p. 949}  \medskip
  
  \centerline{Raymond Mortini}
  
\medskip

\centerline{- - - - - - - - - - - - - - - - - - - - - - - - - - - - - - - - - - - - - - - - - - - - - - - - - - - - - -}
  
  \medskip

  The problem obviously is equivalent to show the existence of two points $0<a<b<1$
  with $f(a)=f(b)=b-a$, or in other words, find $0<a<b<1$ with $b-f(b)=a$ and $f(b)=f(a)$.
  
  To this end, consider the function $h(x):=f(x-f(x))-f(x)$, where we have continuously 
  extended $f$ by the value $0$ for $x<1$ and $x>1$.  Then $h$ is continuous.
  We have to show that $h$ admits a zero $b$  in $]0,1[$ with $f(b)< b$. 
  Then  $a:=b-f(b)\in \;]0,1[$ and $b-a=f(b)=f(a)$.
  
 To do this, we prove  that $h$ takes positive and negative values on $[0,1]$.  Since $h(0)=h(1)=0$,
  the continuity of $h$ implies that $h$ has  a zero $b$ in $]0,1[$. Our construction will guarantee
  that $f(b)<b$.\medskip
  
  Let $\xi_0$ be the largest fixed point  of $f$ (note hat $0\leq \xi_0<1$). For later
  purposes, we note that $f(x)\leq x$ whenever $\xi_0\leq x\leq 1$.
  If  $\xi_0=0$, we let $x_0$
  be the be the smallest point  for which 
  $f(x_0)=M:=\max_{x\in [0,1]} f(x)$. Note that $x_0\in \;]0,1[$.
  Finally,
   let  $x_1\in\; [\xi_0,1[$ be the largest point  with $f(x_1)=M_1:=\max_{x\in [\xi_0,1]} f(x)$.
 Then $0<x_0\leq x_1<1$. 
  Since the function   $x-f(x)$ is $0$ at $\xi_0$ and $1$ at $1$, the intermediate value theorem
  for continuous functions implies that there exists $y_1\in\; ]\xi_0,1[$ such that $y_1-f(y_1)=x_1$.
  Since $f>0$, $y_1>x_1$. Thus $$h(y_1)=f(y_1-f(y_1))-f(y_1)=f(x_1)-f(y_1)=M_1-f(y_1)>0.$$

  On the other hand,  $h(\xi_0)=f(\xi_0-f(\xi_0))-f(\xi_0)=0-f(\xi_0)< 0$ if $\xi_0>0$, and
  if $\xi_0=0$, then, $h(x_0)=f(x_0-f(x_0))-f(x_0)<0$ (since $x_0-f(x_0)$ is left from
  the smallest maximal point $x_0$ of $f$.)\medskip
  
  In both cases, there exists $b$ such that $h(b)=0$.   Since
  $\xi_0<b<y_1$ if $\xi_0>0$ and $0<x_0<b<y_1$ if $\xi_0=0$, we see that $f(b)<b$.

   \newpage

\begin{figure}[ht]
\scalebox{1.00}{\includegraphics{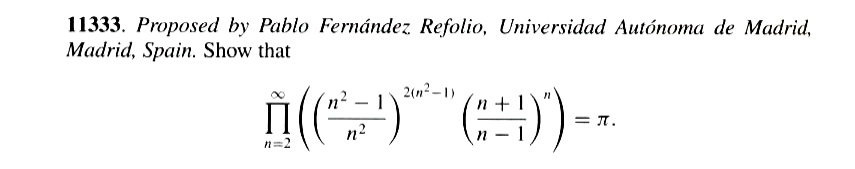}}
\end{figure}

\centerline  {\bf Solution to problem 11333, AMM 114 (10), (2007), p. 926}  \medskip
  
  \centerline{Raymond Mortini}
  
\medskip

\centerline{- - - - - - - - - - - - - - - - - - - - - - - - - - - - - - - - - - - - - - - - - - - - - - - - - - - - - -}
  
  \medskip
  
  Let 
  $$P_N=\prod_{n=2}^N \left( \left( \frac{n^2-1}{n^2}\right)^{2(n^2-1)}
  \left(  \frac{n+1}{n-1}\right)^n\right).$$
  \bigskip
  
  a)  We have the following  equalities:
  $$\prod_{n=2}^N \left( \frac{n^2-1}{n^2}\right)^{n^2-1}=
  \frac{(N+1)^{N^2-1}}{N^{N(N+2)}}\, (N!)^2,$$
  
  b) $$ \prod_{n=2}^N \left(  \frac{n+1}{n-1}\right)^n= \frac{(N+1)^N N^{N+1}}{2 (N!)^2}.$$

  Hence  
  $$\sqrt{P_N}= \frac{(N+1)^{N^2-1}}{N^{N(N+2)}}\, (N!)^2\,
  \frac{ (N+1)^{N/2} \, N^{(N+1)/2}}{\sqrt 2 N!}=$$
  $$ \left( \frac{N+1}{N}\right)^{N^2-1}\, \frac{ N^{N^2-1}}{N^{N(N+2)}}\, N!\;
   \frac{ (N+1)^{N/2} \, N^{(N+1)/2}}{\sqrt 2}=$$
   $$ \left( \frac{N+1}{N}\right)^{N^2-1} N! \;\frac{(N+1)^{N/2}}{\mathbf{ N^{N/2}}\sqrt 2}
   \,\frac{N^{(N+1)/2}\mathbf {N^{N/2}} }{N^{2N+1}}=$$
   $$\left( \frac{N+1}{N}\right)^{N^2-1} N!\; \frac{\left(1+\frac{1}{N}  \right)^{N/2}}{\sqrt 2}\,
   \frac{ \sqrt N}{N^{N+1}}.$$
   
   We are now using Stirling's formula telling us that $n! \sim e^{-n}n^n \sqrt{2\pi n}$.
  Hence
  $$\sqrt{P_N}\sim  \frac{\sqrt{e}}{\sqrt 2} N^N e^{-N} \sqrt{2\pi N} \left( \frac{N+1}{N}\right)^{N^2-1}
   \frac{ \sqrt N}{N^{N+1}}=$$
   
 $$\sqrt e \sqrt \pi e^{-N}  \left( \frac{N+1}{N}\right)^{N^2-1}.$$
 
 But $a_N:=\dis e^{-N} \left( \frac{N+1}{N}\right)^{N^2-1}\to \frac{1}{\sqrt e}$ as $N\to\infty$;
 in fact, by  taking logarithms we obtain
  
 $\log a_n=(N^2-1)\log (1+\frac{1}{N})-N\sim N^2 \log (1+\frac{1}{N})-N=
 N^2( \frac{1}{N} -\frac{1}{2N} \pm \cdots)-N\sim  -\frac{1}{2}.$
 \medskip
 
 Hence $\sqrt{P_N}\to \sqrt \pi$ and so $P_N\to \pi$.

  \newpage
  
  \begin{figure}[ht]
\scalebox{.45}{\includegraphics{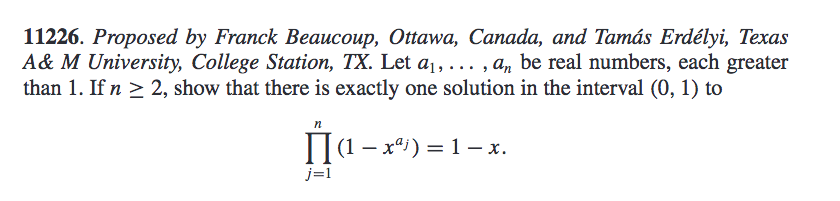}}
\end{figure}

 \centerline {\bf Solution to problem 11226, AMM 113 (5), (2006), p. 460}  \medskip
  
  \centerline{Raymond Mortini}
  
\medskip

\centerline{- - - - - - - - - - - - - - - - - - - - - - - - - - - - - - - - - - - - - - - - - - - - - - - - - - - - - -}
  
  \medskip

Let $h(x)=\prod_{j=1}^n (1-x^{a_j})$. Then 
$h'(x)/h(x)=-\sum_{j=1}^n \frac{a_jx^{a_j-1}}{1-x^{a_j}}$
and hence $$h'(x)=-\sum_{j=1}^n a_j x^{a_j-1}\prod_{k\not=j}(1-x^{a_k}).$$
 Clearly  $h'(0)=h'(1)=0$.
Let $$f(x)=(1-x)^{-1}\prod_{j=1}^n (1-x^{a_j})$$ if $0\leq x<1$.
Note that $f(0)=1$ and $\lim_{x\to 1}f(x)=-h'(1)=0$. Thus, if we show that $f'(0)>0$ and that
the derivative of $f$ has a unique zero in the open interval $]0,1[$,
we are done (that is we can then conclude by the intermediate value theorem
that there is a unique $x_0$ with $0<x_0<1$ so that  $f(x_0)=1$, and hence 
$h(x_0)= 1-x_0$.)

Now, $f'(x)/f(x)=\frac{1}{1-x}+h'(x)/h(x)$. In particular, $f'(0)=1$. Thus we have to look for  $x\in ]0,1[$
so that  $g(x):=\sum_{j=1}^n a_jx^{a_j-1}\frac{1-x}{1-x^{a_j}}=1$. But $g(0)=0$,
and, by de l'Hopital's rule, $\lim_{x\to 1} g(x)=n$.  The intermediate value theorem 
yields the existence of $x$. 
The uniqueness of such an $x$ follows from
the fact that $g$ is strictly increasing. This is due to the fact that the function
$\frac{x^{a-1}-x^a}{1-x^a}$ is strictly increasing on $]0,1[$ whenever $a>1$. 

The latter follows from the fact that 
$$\frac{d}{dx}\frac{x^{a-1}-x^a}{1-x^a}=\frac{x^{a-2}\big((a-1)+x^a-ax\bigr)}{(1-x^a)^2}$$

and that $k(x):=a-1 +x^a -ax\geq 0$ for $0\leq x\leq 1$, because $k(0)=a-1>0$, $k(1)=0$
and $k'(x)=a(x^{a-1}-1)\leq 0$.\vs1cm

  \newpage
  
     \begin{figure}[h!]
 
  \scalebox{0.5} 
  {\includegraphics{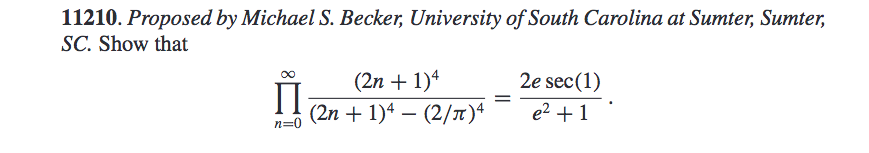}} 
\end{figure}

 \centerline {\bf Solution to problem 11210, AMM 113 (3), (2006), p. 267}  \medskip
  
  \centerline{Raymond Mortini}
  
\medskip

\centerline{- - - - - - - - - - - - - - - - - - - - - - - - - - - - - - - - - - - - - - - - - - - - - - - - - - - - - -}
  
  \medskip

  \vs0,5cm
  We note that
  $$ p_n:=\frac{(2n+1)^4-(2/\pi)^4}{(2n+1)^4}=
  \left( 1-\left[\frac{2}{\pi(2n+1)}\right]^4 \right)=$$
  $$
  \left( 1-\frac{4}{\pi^2(2n+1)^2}\right)\; \left( 1+\frac{4}{\pi^2(2n+1)^2}\right).$$
  \vs0,2cm
  Multiplying in the numerator and denominator (which is 1) with the "missing" factors
  $$\left(1-\frac{4}{\pi^2(2n)^2}\right)\; \left(1+\frac{4}{\pi^2(2n)^2}\right)
  $$
  
  we obtain
  $$P:=\prod_{n=0}^\infty p_n=\prod_{k=1}^\infty\frac{\left(1-\frac{4}{\pi^2k^2}\right)
  \left(1+\frac{4}{\pi^2k^2}\right)}{\left(1-\frac{1}{\pi^2k^2}\right)
  \left(1+\frac{1}{\pi^2k^2}\right)}.$$
  
  Using the standard infinite product representation of the sinus
  $$\frac{\sin z}{z}=\prod_{k=1}^\infty\left(1-\frac{z^2}{\pi^2k^2}\right),$$
  we obtain
  
  $$P=\frac{ \frac{\sin 2}{2}\; \frac{\sin(2i)}{2i}}{\frac{\sin 1}{1}\; \frac{\sin i}{i}}=
  \cos 1\cosh 1=(\cos 1)\, \frac{e^2+1}{2}.$$

\newpage
  
     \begin{figure}[h!]
 
  \scalebox{0.45} 
  {\includegraphics{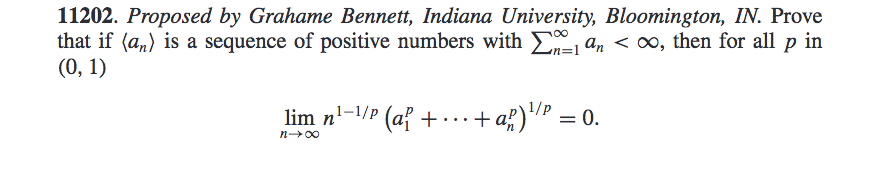}} 
\end{figure}

 \centerline  {\bf Solution to problem 11202, AMM 113 (2), (2006), p. 179}  \medskip
   
   \centerline{Raymond Mortini}
  
\medskip

\centerline{- - - - - - - - - - - - - - - - - - - - - - - - - - - - - - - - - - - - - - - - - - - - - - - - - - - - - -}
  
  \medskip

 \vs0,5cm

The assertion is an immediate consequence of H\"older's inequality:
Wlog let $0\leq a_j\leq 1$ and let $q\in ]0,1[$ be such that $p+q=1$ (note that $p\in\, ]0,1[$.)
$$n^{p-1}\sum_{j=1}^na_j^p=n^{p-1}\left(\sum_{j=1}^N a_j^p+\sum_{j=N+1}^n a_j^p\,\cdot 1\right)\leq$$
$$
\frac{N}{n^{1-p}}+\left(\sum_{j=N+1}^n (a_j^p)^{1/p}\right)^p
\left(\sum_{j=N+1}^n1^{1/q}\right)^q\; \frac{1}{n^{1-p}}\leq$$
$$ \frac{N}{n^{1-p}}+\left(\sum_{j=N+1}^\infty a_j\right)^p \;\frac{n^q}{n^{1-p}}
=  \frac{N}{n^{1-p}} + \left(\sum_{j=N+1}^\infty a_j\right)^p\leq \epsilon$$
\vs0,2cm
if $N$ and $n>N$ is sufficiently big.\vs1cm

\newpage 

\pagecolor{yellow}

     \begin{figure}[h!]
 
  \scalebox{0.5} 
  {\includegraphics{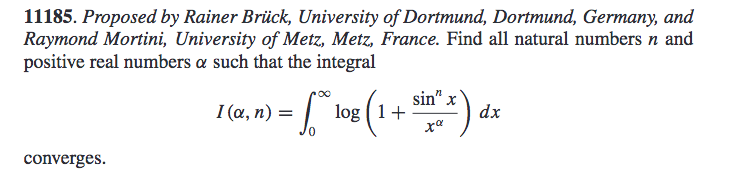}} 
\end{figure}

\centerline{\bf Solution to problem 11185 AMM 112 (2005), 840}  \medskip

\centerline{Rainer Br\"uck, Raymond Mortini }

\medskip

\centerline{- - - - - - - - - - - - - - - - - - - - - - - - - - - - - - - - - - - - - - - - - - - - - - - - - - - - - -}
  
  \medskip

We claim that
\medskip

\shadowbox{$I(\alpha,p)$ converges if and only if $(\alpha,p) \in
\left]1,\infty\right[ \times \mathbb{N}$ or $(\alpha,p) \in
\left]\frac{1}{2},1\right] \times (2\mathbb{N}+1)$.}
\bigskip

First we discuss the behaviour of the integrand at the origin. For
$\alpha>0$ we have
$\left|\log{\left(1+\frac{\sin^p{x}}{x^\alpha}\right)}\right| \leq
\log{(1+x^{-\alpha})}$. Substituting $\frac{1}{x}$ by $t$, we
obtain
\[
  \int_0^1 \log{\left(1+x^{-\alpha}\right)}\,dx =
  \int_1^\infty \frac{\log{(1+t^\alpha)}}{t^2}\,dt\,,
\]
and this integral is convergent. Hence, our integral $I(\alpha,p)$
converges at $0$ for every $\alpha>0$ and $p \in \mathbb{N}$.

Now we discuss the behaviour at infinity. Since $\lim\limits_{t
\to 0}{\frac{\log{(1+t)}}{t}}=1$, we see that at infinity
\[
  A(x) := \log{\left(1+\frac{\sin^p{x}}{x^\alpha}\right)} \sim
  \frac{\sin^p{x}}{x^\alpha} = :B(x)\,.
\]
Hence $\int_1^\infty A(x)\,dx$ converges absolutely if and only if
$\int_1^\infty B(x)\,dx$ does. Note that by Riemann's convergence
test $\int_1^\infty |B(x)|\,dx \leq \int_1^\infty
\frac{dx}{x^\alpha} < \infty$ whenever $\alpha>1$. Hence,
$\int_1^\infty A(x)\,dx$ is absolutely convergent for $\alpha>1$.

Now suppose that $0 < \alpha \leq 1$. On the intervals $J_k :=
\left[\frac{\pi}{6}+2k\pi,\frac{\pi}{2}+2k\pi\right]$, $k \geq 1$,
we have $|\sin{x}| \geq \frac{1}{2}$ and $x \geq 1$. Hence
$\frac{\sin^p{x}}{x^\alpha} \geq \frac{2^{-p}}{x} \geq
\frac{2^{-p}}{2\pi(k+1)}$. Therefore,
\[
  \int_{J_k}|B(x)|\,dx \geq \frac{1}{3}\cdot\frac{2^{-p-1}}{k+1}\,.
\]
Since $\int_1^\infty |B(x)|\,dx \geq \sum_{k=1}^\infty \int_{J_k}
|B(x)|\,dx$, we see that $\int_1^\infty |B(x)|\,dx$ and hence
$\int_1^\infty |A(x)|\,dx$ diverges (absolutely) for $0 < \alpha
\leq 1$. In particular, $\int_1^\infty A(x)\,dx$ diverges whenever
$p$ is even, since in that case $|A(x)|=A(x)$.
\medskip

To continue, we may thus assume that $p=2n+1$ is odd. We use that
for every $\alpha>0$ and $n \in \mathbb{N}$ the integral
$\int_1^\infty \frac{\sin^{2n+1}{x}}{x^\alpha}\,dx$ converges.
Indeed, let $I_m(x):=\int_1^x \frac{\sin^{m}{t}}{t^\alpha}\,dt$
and let $F_m$ be a primitive of $\sin^m{t}$ with $F_m(1)=0$. For
$m$ odd, $F_m$ is periodic, hence bounded. By partial integration
we obtain
\[
  I_{2n+1}(x) = \frac{F_{2n+1}(x)}{x^\alpha} + \alpha
  \int_1^x \frac{F_{2n+1}(t)}{t^{\alpha+1}}\,dt\,,
\]
and we conclude that $I_{2n+1}(x)$ converges as $x\to\infty$.

Now we use the Taylor development
\[
  \log{(1+u)} = \sum_{k=1}^{m-1} \frac{(-1)^{k-1}}{k}\,u^k +
  \frac{(-1)^{m-1}}{m}\,u^m\left(1+\varepsilon(u)\right)\,,
\]
where $\varepsilon$ is a continuous function of $u$ and
$\varepsilon(0)=0$. In particular, $|\varepsilon(u)|<1$ whenever
$|u| \leq \delta$ with $\delta>0$ sufficiently small. Now, we set
$u=u(x)=\frac{\sin^{2n+1}{x}}{x^\alpha}$, where $x>0$ is so large
that $|u| \leq \delta$. Then for sufficiently large real numbers
$M>N$, we have
\begin{align*}
  I &:= \int_{N}^M \log{\left(1+
  \frac{\sin^{2n+1}{x}}{x^\alpha}\right)}\,dx
  = \sum_{k=1}^{m-1} \frac{(-1)^{k-1}}{k} \int_{N}^M
  \left(\frac{\sin^{2n+1}{x}}{x^\alpha}\right)^k\,dx \\
  &\qquad + \frac{(-1)^{m-1}}{m} \int_{N}^M
  \left(\frac{\sin^{2n+1}{x}}{x^\alpha}\right)^m
  \left(1+\varepsilon(u(x))\right)\,dx
  =: \sum_{k=1}^{m-1} I_k + \widetilde{I}_m\,.
\end{align*}
Choosing $m \in \mathbb{N}$ such that $m\alpha>1$ and $(m-1)\alpha
\leq 1$, the boundedness of $\varepsilon(u)$ yields the absolute
convergence of the last integral $\widetilde{I}_m$. If
$\frac{1}{2} < \alpha \leq 1$, then $m=2$ and hence
$I=I_1+\widetilde{I}_2$. But $I_1$ and $\widetilde{I}_2$ converge,
and hence $I$ converges. If $0 < \alpha \leq \frac{1}{2}$, then $m
\geq 3$ and at least a third integral $I_2$ above appears. That
integral is divergent, since the exponent of the $\sin$ is an even
one (note that by the choice of $m$, the exponent of $x$ is still
at most $1$). Since all those divergent integrals $I_{2q}$ come up
with the same sign, we finally get the divergence of
$I_1+I_2+\dotsb+I_{m-1}$, and thus $I$ diverges.

Finally, we note that the example $p=1$ and $\alpha=\frac{1}{2}$
yields examples of functions $f$ and $g$ such that at infinity,
$f\sim g$, but for which $\int_0^\infty f(x)\,dx$ diverges and
$\int_0^\infty g(x)\,dx$ converges, namely $f(x) =
\log{\left(1+\frac{\sin{x}}{\sqrt{x}}\right)}$ and $g(x) =
\frac{\sin{x}}{\sqrt{x}}$.

\newpage

     \begin{figure}[h!]
 
  \scalebox{0.5} 
  {\includegraphics{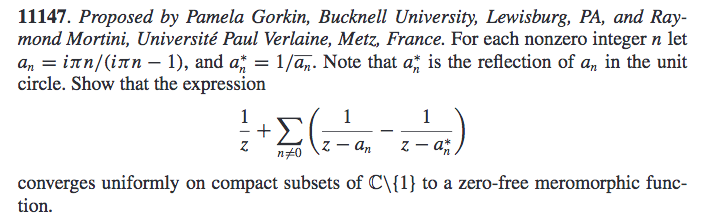}} 
\end{figure}

\centerline{\bf Solution to problem 11147 AMM 112 (2005), 366}  \medskip

\centerline{Pamela Gorkin, Raymond Mortini}
  
\medskip

\centerline{- - - - - - - - - - - - - - - - - - - - - - - - - - - - - - - - - - - - - - - - - - - - - - - - - - - - - -}
  
  \medskip

Let $S(z)=\exp\left(-\frac{1+z}{1-z}\right)$ be the atomic inner function. Put 
$$f=\frac{1/e-S}{1-(1/e)S}.$$
Then $f$ is an inner function (that is it has radial limts of modulus one almost everywhere). Since $f$ does not have radial limit zero, it must be a pure Blaschke product (see Garnett, p.76), that is

$$f(z)=e^{i\theta} z\prod_{n\in\Z\setminus \{0\}} \frac{|a_n|}{a_n}\;\frac{a_n-z}{1-\ov a_n z}.$$

Its zeros are exactly the numbers $a_n$ for $n\in\Z\setminus\{0\}$, including the the origin. Since the derivative of $S$ is $S'(z)=-S(z)\frac{2}{(1-z)^2}$, it follows that the derivative of $f$ does not vanish either. But
$$\frac{S'(z)}{S(z)}=\frac{1}{z} + \sum_{n\in\Z\setminus\{0\}} \left(\frac{1}{z-a_n}-\frac{1}{z-a_n^*}\right).$$

\newpage

     \begin{figure}[h!]
 
  \scalebox{0.5} 
  {\includegraphics{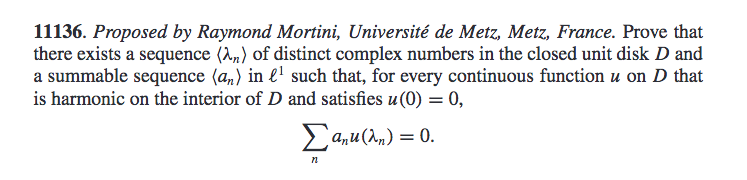}} 
\end{figure}

\centerline{\bf Solution to problem 11136 AMM 112 (2005), 181} \medskip

\centerline{Raymond Mortini}
  
\medskip

\centerline{- - - - - - - - - - - - - - - - - - - - - - - - - - - - - - - - - - - - - - - - - - - - - - - - - - - - - -}
  
  \medskip

Let $D_n=D(\lambda_n,r_n)$ be a sequence of pairwisw disjoint, closed disks contained in the open unit disk
$\D$ such that the area measure of $\D\setminus \Union D_n$ is zero. Noticing that by the mean-value area theorem
for harmonic functions
$$\int\int_{D(\lambda,r)}u(z) dA(z)=\pi r^2 u(\lambda),$$
we obtain the assertion
$$0=u(0)=\int\int_\D u(z)dA(z)=\sum_n \int\int_{D_n} u(z) dA(z)=\pi \sum_n r_n^2u(\lambda).$$

{\it Remark} The problem was motivated by the question, circulating in England, and communicated to me by
Joel F. Feinstein, whether the set of exponentials $\{e^{i\lambda z}: \lambda \in \C\}$  is countably
 linear independent!  The method for the proof above presumably appeared for the first
  time in a paper of J. Wolff [Comptes Rendus Acad. Sci. Paris 173 (1921), 1056-1058].

\newpage
\nopagecolor
  
     \begin{figure}[h!]
 
  \scalebox{0.45} 
  {\includegraphics{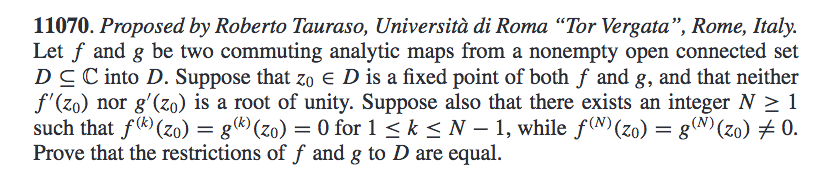}} 
\end{figure}

\centerline {\bf Solution to  problem 11070, AMM 111 (2004), p. 258}  \medskip

\centerline{Raymond Mortini}
  
\medskip

\centerline{- - - - - - - - - - - - - - - - - - - - - - - - - - - - - - - - - - - - - - - - - - - - - - - - - - - - - -}
  
  \medskip

Let $\N=\{1,2,\cdots\}$ and $f,g\in C^n(\Omega)$.
Then the  result follows from the following formula: 
$$(f\circ g)^{(n)}(z)=
\sum_{j=1}^n f^{(j)}(g(z))\biggl(\sum_{{k\in{\N}^j\atop |k|=n}}C_k^n\, g^{(k)}(z)\biggr),\eqno ({\rm Mo}_n)$$

where $k=(k_1,k_2,\dots,k_j)\in\N^j$ is an ordered multi-index with $k_1\leq k_2\leq\cdots\leq k_j$,
$|k|=\sum_{i=1}^jk_i,\sp g^{(k)}=g^{(k_1)}g^{(k_2)}\dots g^{(k_j)}$ and
$\dis C_k^n=\frac{1}{~\prod_i[A_k(i)!]~}{n\choose k}$. Here $A_k(i)$ denotes the cardinal of how often
$i$ appears within the ordered index $k$ and ${n\choose k}=\frac{n!}{ k_1!k_2!\dots k_j!}$.
\vs0,1cm

This formula has many advantages vis-\aa-vis the Faa di Bruno formula 
$$(f\circ g)^{(n)}=\sum{n\choose p}(f^{(p)}\circ g)\prod_{j=1}^n\left(\frac{g^{(j)}}{j!}\right)^{p_j},$$
where $p_j\in \{0,1,2,\cdots\}$, $p=p_1+p_2+\dots +p_n$ and  $p_1+2p_2+\dots +np_n=n$,
since one immediately can write
down all the factors that occur without solving  the above equations for $p_j$.\vs0,2cm

{\sl Case 1}: Let $f(z_0)=g(z_0)=z_0$, $A:=f'(z_0)=g'(z_0)\not=0$, $A^p\not= 1\;\forall p\in\N$
and $f\circ g=g\circ f$.

In order to show that $f\equiv g$ it is enough to prove that
 $f^{(n)}(z_0)=g^{(n)}(z_0)$ for all $n$. The proof is done inductively:

$n=2$:  Since $(f\circ g)''=(f''\circ g)g'^2 +(f'\circ g)g''$  and $f\circ g=g\circ f$ we get:
$f''(z_0)A^2+Ag''(z_0)=g''(z_0)A^2+Af''(z_0)$. Hence $f''(z_0)(A-1)=g''(z_0)(A-1)$.
Since $A\not=1$ we obtain that $f''(z_0)=g''(z_0)$.

$n\to n+1$:  
$$(f\circ g)^{(n+1)}=(f'\circ g) g^{(n+1)}+\sum_{j=2}^n(f^{(j)}\circ g)
\sum_{{k\in \N^j\atop |k|=n+1}} C^{n+1}_k g^{(k)}+(f^{(n+1)}\circ g) (g')^{n+1}$$
Evaluating at $z_0$ and noticing that, by induction hypotheses,  all derivatives appearing
 in the middle term 
coincide at $z_0$ with those  when $f$ is replaced by $g$, we get that 
$$A g^{(n+1)}(z_0)+f^{(n+1)}(z_0)A^{n+1}=A f^{(n+1)}(z_0)+g^{(n+1)}(z_0)A^{n+1}.$$
Hence $f^{(n+1)}(z_0)(A^n-1)=g^{(n+1)}(z_0)(A^n-1)$, from which we conclude that 
$f^{(n+1)}(z_0)= f^{(n+1)}(z_0)$, because $A^n\not=1$. \vs 0,3cm

{\sl Case 2}:    $f(z_0)=g(z_0)=z_0$, $f^{(j)}(z_0)=g^{(j)}(z_0)=0$ for $1\leq  j<n_0$, but
$f^{(n_0)}(z_0)=g^{(n_0)}(z_0)\not=0$ and $f\circ g=g\circ f$.

Suppose that $f^{(j)}(z_0)=g^{(j)}(z_0)$ has been shown to be true for $j<n$, where
 $n=pn_0+q$, with $0\leq q<n_0$ and $p\geq 1$. We show that this holds then for $j=n$.

Let $N= n_0^2+(p-1)n_0+q$ and consider $(f\circ g)^{(N)}(x_0)$.
 All the terms in (Mo)$_N$ with $j<n_0$ disappear, since
$f^{(j)}(g(z_0))=f^{(j)}(z_0)=0$. Moreover, as we are going to show, all other terms,
excepted the term for $j=n_0$ and the index $k=(n_0,\cdots, n_0, pn_0+q)\in \N^{n_0}$,
coincide for $f$ and $g$; hence can be thrown off when regarding the equality $(f\circ g)^{(N)}=(g\circ f)^{(N)}$.
Thus that equality is equivalent to
$$f^{(n_0)}(g(z_0))(g^{(n_0)})^{n_0-1}(z_0)g^{(pn_0+q)}(z_0)=
g^{(n_0)}(f(z_0))(f^{(n_0)})^{n_0-1}(z_0)f^{(pn_0+q)}(z_0)$$

But this implies of course that $f^{(pn_0+q)}(z_0)=g^{(pn_0+q)}(z_0)$, which is what we were after.
\vs0,2cm
That one can restrict to this single index $k=(n_0,\cdots, n_0, pn_0+q)\in \N^{n_0}$
is seen as follows:  Let $k'\in \N^{n_0}$, be an ordered index 
with $|k'|=|k|=(n_0-1)n_0+pn_0+q=N$. Suppose that the last coordinate of $k'$ (which is the maximum)
is strictly bigger than the last coordinate of $k$. Then at least one of the previous coordinates of $k'$
 must be strictly
smaller  than $n_0$. But the associated derivatives of $g$ (resp $f$) vanish at $z_0$.
Thus this term does not appear in the formula for $(f\circ g)^{(N)}(z_0)$. On the other hand,
if the last coordinate  of $k'$ is strictly less than $pn_0+q$ (hence all of the coordinates of $k'$),
then by induction all the associated derivatives of $g$ (in $(f\circ g)^{(N)}$) coincide with
those for $f$ (in $(g\circ f)^{(N)}$) at $z_0$. Thus these terms can be thrown away.

 Now let $k'\in \N^j$
with $n_0<j\leq N$ and $|k'|=N$. Then the maximum of the coordinates  of $k'$ is strictly less than
$pn_0+q$, since otherwise $|k'|\geq (j-1)n_0+pn_0+q\geq n_0^2+pn_0+q> N$, a contradiction.
Thus, as above, also these terms can be thrown away.

\newpage

\pagecolor{yellow}

    \begin{figure}[h!]
 
  \scalebox{0.5} 
  {\includegraphics{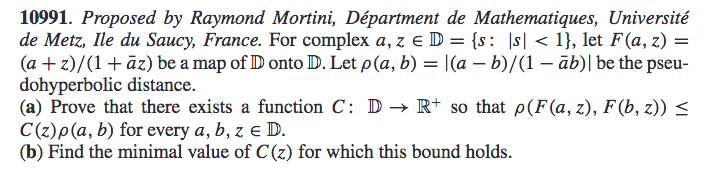}} 
\end{figure}

\centerline {\bf No own Solution to  problem 10991, AMM 110 (2003), p. 155}\vskip 1cm

\newpage

 
\centerline{\colorbox{white}{\parbox{13cm}{\phantom{\big|}{\bf 10890.} ~~{\sl Proposed by Raymond Mortini, Universit\'e de Metz, France.} ~~ Let $d_1$ and $d_2$
  be two metrics on a nonempty set $X$ with the property that every ball in $(X,d_1)$ contains a ball in $(X,d_2)$
   and vice versa.  Must $d_1$ and $d_2$ generate the same topology?\\
    \phantom{\big|}}}
   }\bigskip

\centerline {\bf Solution to  problem 10890, AMM 108 (2001), p. 668}  \medskip

\centerline{Raymond Mortini}
  
\medskip

\centerline{- - - - - - - - - - - - - - - - - - - - - - - - - - - - - - - - - - - - - - - - - - - - - - - - - - - - - -}
  
  \medskip

Let $d$ denote the Euclidean metric on $\R$ and
let $f$ be an injective real-valued function on $\R$. 
It is easy to see that the function $\rho(x,y)=|f(x)-f(y)|$
defines a second metric on  $\R$, i.e. satisfies
the axioms 
\begin{enumerate}

\item [(D1)]  $\rho(x,y)\geq 0, \rho(x,y)=0\Longleftrightarrow x=y$,

\item [(D2)] $\rho(x,y)=\rho(y,x)$

\item [(D3)]   $\rho(x,y)\leq  \rho(x,z)+ \rho(z,y)$ for all $x,y,z\in\R$.
\end{enumerate}

Let $B_d(x_0,\epsilon)$ resp. $B_{\rho}(x_0,\epsilon)$ denote the open  balls
of radius $\epsilon$ and center $x_0$ with respect to the distances $d$ and $\rho$.
\vskip 0,2cm

Let us now additionally assume that $f$ is   increasing,  one-sided
continuous but not continuous,   and has only a finite number of discontinuities.
 This guarantees that $I:=f(\R)$  is a union of non-degenerated
intervals, with pairwise disjoint closures. The inverse function $f^{-1}: I\to \R$
then is continuous on $I$. Fix $x_0$.  Hence for every $\epsilon>0$ there exists $\delta>0$
such that $B_{\rho}(x_0,\delta)\subseteq B_d(x_0,\epsilon)$. 

Let $x_0$ be a point at which $f$ is, say, left-continuous. Then for every
$\epsilon>0$ there exists $\delta>0$ such that for  all $x<x_0$, $d(x,x_0)<\delta$
implies $\rho(x,x_0)=|f(x)-f(x_0)|<\epsilon$. Let $x_1=x_0-{1\over 2}\delta$. Then
$B_d(x_1,\delta/2)\subseteq B_{\rho}(x_0,\epsilon)$.

Thus each ball in the $d$-metric contains a ball in the $\rho$-metric, and vice-versa. 

 It is  clear that the identity map id: $(\R,\rho)\to (\R,d)$, although
being continuous, has no continuous inverse. 
Note  that ${\rm id}: (\R,d)\to (\R,\rho)$ is continuous at $x_0$ if and only if $f$
is continuous at $x_0$. Thus the two topologies are distinct.

\medskip
{\it Remark}  If we additionally assume that $(X,d_j)$ are topological vector spaces, then the answer is yes.
This is due to the fact that these topologies can be generated by translation invariant metrics $d_1'$ and $d_2'$.
In fcat, $\forall \e>0 \;\exists\delta>0: B_{d_1'}(x_0,\delta)\ss B_{d_2'}(0,\e/2)$. In particular, $x_0$ and $-x_0$ are in $B_{d_2'}(0,\e/2)$. Hence
$$B_{d_1'}(0,\delta)=-x_0+B_{d_1'}(x_0,\delta)\ss B_{d_2'}(0,\e/2)+B_{d_2'}(0,\e/2)\ss B_{d_2'}(0,\e).$$
The problem was also solved by Matthias Bueger and Dietmar Voigt (Germany).

\newpage

\newpage
\nopagecolor

     \begin{figure}[h!]
 
  \scalebox{0.5} 
  {\includegraphics{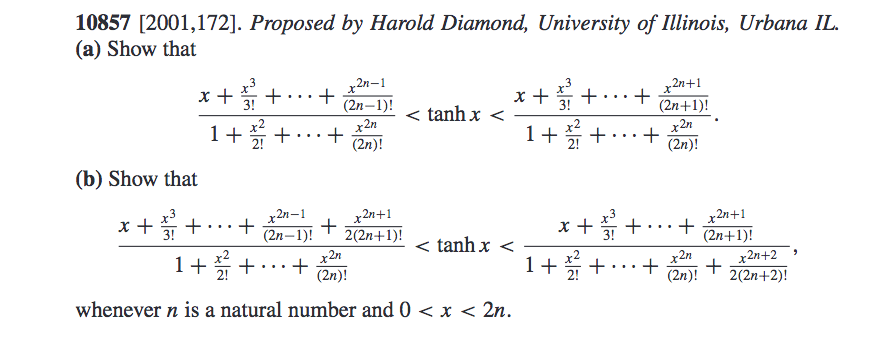}} 
\end{figure}

\centerline {\bf Solution to  problem 10857 (a), AMM 108 (2001), p. 172}  \medskip

\centerline{Raymond Mortini}
  
\medskip

\centerline{- - - - - - - - - - - - - - - - - - - - - - - - - - - - - - - - - - - - - - - - - - - - - - - - - - - - - -}
  
  \medskip

Let $\dis C_{2n}=\sum_{j=0}^{n}{x^{2j}\over (2j)!}$ and $\dis S_{2n+1}=\sum_{j=0}^{n}{x^{2j+1}\over
(2j+1)!}$. We show that, for every $x>0$, the sequence  $({S_{2n+1}\over C_{2n}})$ is strictly
decreasing, whereas $({S_{2n-1}\over C_{2n}})$ is stricly increasing. Since both sequences converge
to $\tanh x$ we get that ${S_{2n-1}\over C_{2n}}< \tanh<{S_{2n+1}\over C_{2n}}$.\vs0,2cm

i) We have the following equivalences:\vs 0,2cm

$\dis({S_{2n+1}\over C_{2n}})\searrow\ssi
{S_{2n+1}\over  S_{2n-1}}<    {C_{2n}\over C_{2n-2}}\ssi {S_{2n-1}+{x^{2n+1}\over (2n+1)!}\over S_{2n-1}}
< {C_{2n-2}+{x^{2n}\over (2n)!}\over C_{2n-2}}\ssi $\vs 0,2cm

$\dis \ssi 1+{{x^{2n+1}\over (2n+1)!}\over S_{2n-1}}<
1+{{x^{2n}\over (2n)!}\over C_{2n-2}}\ssi xC_{2n-2}< (2n+1)S_{2n-1}\ssi$\vs 0,2cm

$$\dis\sum_{j=0}^{n-1}{x^{2j+1}\over (2j)!}<(2n+1)\sum_{j=0}^{n-1}{x^{2j+1}\over (2j+1)!}\eqno (1)$$\vs 0,2cm

But ${1\over (2j)!}< (2n+1){1\over (2j+1)!}\ssi 2j+1< 2n+1$, which is true. Since $x>0$ we get (1).\vs 0,2cm

ii) That $({S_{2n-1}\over C_{2n}})$ is stricly increasing, is shown in exactly the same way.\vs 0,3cm

To sum up, we get

$${C_{2n+2}\over C_{2n}}<{S_{2n+1}\over S_{2n-1}}<{C_{2n}\over C_{2n-2}}.$$

\newpage

.\vspace{-8mm}
     \begin{figure}[h!]
 
  \scalebox{0.5} 
  {\includegraphics{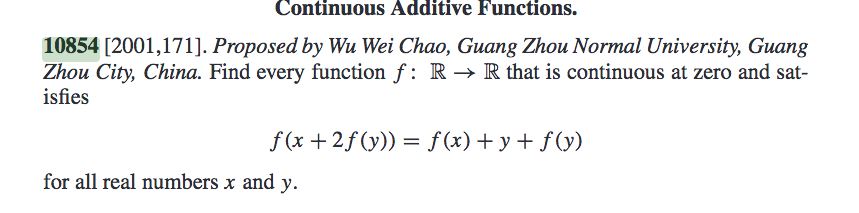}} 
\end{figure}

\centerline {\bf Solution to  problem 10854 AMM 108 (2001), p. 171}  \medskip

\centerline{Raymond Mortini}
  
\medskip

\centerline{- - - - - - - - - - - - - - - - - - - - - - - - - - - - - - - - - - - - - - - - - - - - - - - - - - - - - -}
  
  \medskip

Suppose that $f:{\bf R} \to{\bf R}$ is  a function, continuous  at the origin, and  satisfying 
$$f(x+2f(y))=f(x)+f(y)+y\eqno (1)$$
 for all $x,y\in{\bf R}$.
First, we shall show that $f$ is continuous everywhere. In fact,

$$f\bigl(x+2f(x+2f(y))\bigr)=f(x+2[f(x)+f(y)+y])=f\bigl([x+2y+2f(y)]+2f(x)\bigr)=$$
$$=f\bigl((x+2y)+2f(y)\bigr)+ f(x)+x=f(x+2y)+f(y)+y +f(x)+x.\eqno (2)$$

On the other hand:

$$f\bigl(x+2f(x+2f(y))\bigr)=f(x)+f(x+2f(y))+x+2f(y)=$$
$$=f(x)+[f(x)+f(y)+y]+x+2f(y)=2f(x)+3f(y)+y+x.\eqno(3)$$

By (2) and (3) we get that $f(x+2y)=f(x)+2 f(y)\; \forall (x, y)\in{\bf R}^2$.\vs0,2cm

In particular, by setting $x=y=0$, we see that $f(0)=0$. 

It easily follows that $f$ is continuous at every point $x\in{\bf R}$.  \vs 0,2cm

So, in order to continue,  we may assume that $f$ is a continuous solution of (1).

Let $x=y$. Then 
$$f(y+2f(y))=y+2f(y).\eqno  (4)$$

First we shall determine all continuous solutions of (4).
 Let $g(y)=y+2f(y)$. Since $g$ is continuous,  $g({\bf R})$
is either a singleton or a nondegenerate interval $I$.  If  $g$ is constant, say  $g\equiv c$, 
then $f(y)={c-y\over 2}$
and so $c=f(y+2f(y))=f(c)$, from which we conclude that $c=0$.  Hence $f(y)=-{y\over 2}$. If $g$
is not constant,
take $z\in I$; that is $y+2f(y)=g(y)=z$ for some $y$.
Then $f(z)=z$. Hence $f$ is the identity on $I$.  It follows that $3z=z+2f(z)=f(z+2f(z))=g(z)$. Therefore 
$3z\in I$ and so $I=\big<m,\infty[$ for some $m\in{\bf R}\cup\{-\infty\}$.
 Thus $f(x)=x$ for every $x>  m$. Since $g\geq m$, we have that $f\geq {m-y\over 2}$ 
on $]\;-\infty, m]$.\vs 0,2cm

To prove the converse, 
choose $m\in{\bf R}$. Let $f^*$ be any continuous function on $]-\infty,m]$ 
such that $f^*(y)\geq{m-y\over 2}$
for $y\leq m$ and so that $f^*(m)=m$. Then 
$$\widetilde f(y)=\begin{cases} f^*(y)&\text{if $y\leq m$}\\ y& \text{if $y\geq m$}\end{cases}\eqno (5)$$
is a continuous solution of $(4)$.  \vs 0,3cm

We deduce that any continuous solution of (1) necessarily has the form  (5) or equals $-{1\over 2}y$.
We shall now show that only for $f^*=id$, we really get a solution of (1).\vs0,2cm

So let $f$ be  a continuous solution of (1). Then $f=\widetilde f$ for
 some $f^*$. Fix $x<m$. Take $y> m$ so that $x+2y>m$. Then 

$f(x+2f(y))=f(x+2y)=x+2y$ and $f(x)+y+f(y)=f^*(x)+2y$. Hence (1) implies that $f^*(x)=x$.
\vs 0,2cm

We conclude that $f$ is a continuous solution of (1) if and only if $f(x)=x$ or $f(x)=-{x\over 2}$
on ${\bf R}$.

\newpage

     \begin{figure}[h!]
 
  \scalebox{0.5} 
  {\includegraphics{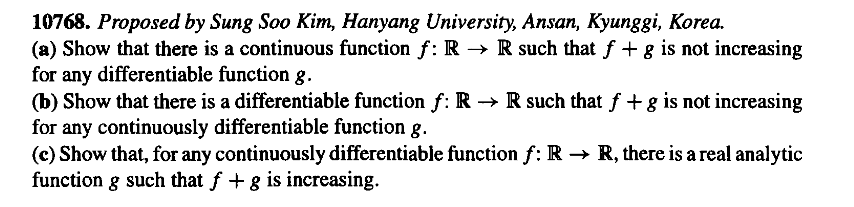}} 
\end{figure}

\centerline {\bf Solution to  problem 10768 AMM  106  (1999), 963}  \medskip

\centerline{Raymond Mortini}
  
\medskip

\centerline{- - - - - - - - - - - - - - - - - - - - - - - - - - - - - - - - - - - - - - - - - - - - - - - - - - - - - -}
  
  \medskip

a)\sp  Let $\dis f(x)=\sqrt{|x|}\sin{1\over x}$ for $x\not=0$ and $f(0)=0$.
Then $f$ is continuous on $\R$. Let $g$ be a differentiable function on $\R$.
Then,  in every neigborhood of $0$, $h:=f+g-g(0)$  takes  negative and positive
values. In fact,  suppose  that $h\geq 0$ on $[0,\e]$. Then ${h(x)\over x}\geq 0$
on $[0,\e]$. But $\liminf_{x\to 0^+}{h(x)\over x}= g'(0)+\liminf_{x\to 0^+}{1\over \sqrt  x}\sin{1\over x}
=-\infty$, a contradiction. Thus $f+g$ is not monotone on any interval centered at $0$. \vs 0,2cm

b) \sp Let $\dis f(x)=x^2\sin{1\over x^2}$ for $x\not=0$ and $f(0)=0$. Then $f$ is differentiable 
on $\R$, $f'(0)=0$, but
$f'(x)=2x\sin{1\over x^2}-{2\over x}\cos{1\over x^2}$ takes arbitrarily large negative and positive
values in any neighborhood $U$ of $0$. Let $g$ be any $C^1(\R)$ function. In particular, $g'$ is
bounded on every compact interval centered at $0$. Hence $f'+g'$ takes arbitrary large negative and
positive values in $U$. Thus $f+g$ is not monotone on any interval  centered at $0$.\vs 0,2cm

c) We show that for every function $f\in C^1(\R)$  there exists an entire function $g$ (that is a function
 holomorphic on the whole plane), real-valued on $\R$,  such that $f+g$ is increasing on $\R$.   In fact, $f'+2|f'|+2\e\geq
2\e>0$ on $\R$. Let $q=2 |f'|+2\e$. Then $q $ is continuous on $\R$. By Carleman's theorem 
(see \cite{carl} and \cite[p.125]{gai}.)
there exists an entire function $Q$ such that $\n q-Q\n_\infty\leq \e$, where ${\n\cdot\n}_\infty$ denotes
 the supremum norm on $\R$.  Let $G(x)={\rm Re}\; Q(x)$. Then  $\n q-G\n\leq\e$. Moreover,  the function
$H(z)={1\over 2}\bigl(Q(z)+\ov{Q(\ov z)}\bigr)$ is analytic in $\C$, and $H$ coincides on $\R$ with  $G$.

Now it is easy to check that  $f'+G\geq \e>0$.  Let $g$ be a primitive of $G$. Then $g$  is the trace of 
an entire function and $f+g$ is (strictly) increasing, since its derivative is  strictly positive. \vs 0,3cm

\newpage

     \begin{figure}[h!]
 
  \scalebox{0.5} 
  {\includegraphics{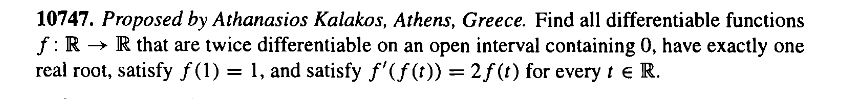}} 
\end{figure}

\centerline {\bf Solution to  problem 10747 AMM 106 (1999), p. 685}  \medskip

\centerline{Raymond Mortini}
  
\medskip

\centerline{- - - - - - - - - - - - - - - - - - - - - - - - - - - - - - - - - - - - - - - - - - - - - - - - - - - - - -}
  
  \medskip

 We claim that all differentiable solutions $f$  of  $f'(f(t))=2f(t),\; t\in\R, f(1)=1$, and having 
only one real root, have the form
$f(t)=t^2$ for $t\geq 0$ and $f(t)=g(t)$  for $t<0$, where $g$ is an arbitrary differentiable function, 
defined on $]-\infty,0]$ satisfying
$g(t)>0$ for $t<0$ and $g(0)=g'(0)=0$. The assumption, that $f$ should be twice differentiable
in a neighboorhood of $0$, is not important. \vs 0,2cm

{\sl Proof} Let $f$ be  a solution of the problem. Put $h=f\circ f-f^2$. Then $h'=(f'\circ f)f'-2f'f=f'(f'\circ f-2f)
\equiv 0$.  Hence $h$ is a constant, say $C$. Because $h(1)=0$, we see that $C=0$ and so $f\circ f=f^2$.
Let $y\in f(\R)$. Then $f(x)=y$ for some $x\in\R$. Therefore $f(y)=f(f(x))=f^2(x)=y^2$. By hypothesis,
$\{0,1\}\ss f(\R)$. By continuity we conclude that $[0,1]\ss f(\R)$. Since the left derivative at $x=1$
is $2$, the differentiability of $f$ now implies that there exists points $x_0$ greater than $1$ for 
which $f(x_0)>f(1)=1$. Since $f_{n+1}=f^{2^n}$, we obtain that $f_{n+1}(x_0)=[f(x_0)]^{2^n}\to\infty$.
Hence $f$ is unbounded. By the intermediate value theorem, we  then get that $[0,\infty]\ss f(\R)$.
Hence $f(x)=x^2$ for $x\geq 0$. \vs 0,2cm 
To determine the behaviour of $f$ for negative values, we use the hypothesis  that $f$ should have
only one zero. Since $f(0)=0$, by continuity, we conclude that either $f(x)<0$ for all $x<0$ or $f(x)>0$ 
for all $x<0$. But $f(x_0)<0$ for some (all) $x_0<0$ implies that  $f(f(x_0))=f^2(x_0)>0$, a contradiction.
Thus $f(x)>0$ for $x>0$. \vs 0,2cm

It is easy to check that every function of the form $f(x)=x^2$ for $x\geq 0$ and $f(x)=g(x)$ for $x<0$,
where $g>0$ is differentiable and satisfies $g(0)=g'(0)=0$, is a solution  of $f\circ f=f^2.$
Hence, by differentiating, $f'(f(x))f'(x)=2f'(x)f(x)$. If $f'(x)\not=0$, then we 
are done. If $f'(x_0)=g'(x_0)=0$ for some $x_0<0$, then  we use the fact that $y:=f(x_0)>0$ and that 
for these positive values $f(y)=y^2$. Hence, $f'(f(x_0))=2f(x_0)$. So we obtain a solution
of our functional equation. 
\newpage
 
      \begin{figure}[h!]
 
  \scalebox{0.5} 
  {\includegraphics{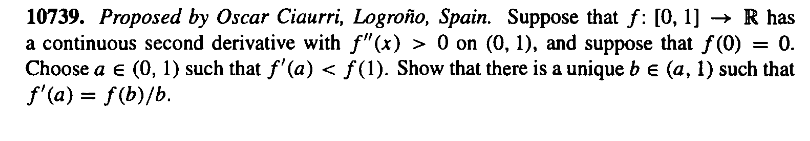}} 
\end{figure}

\centerline {\bf Solution to  problem 10739 AMM 106 (1999), p. 586}  \medskip

\centerline{Raymond Mortini}
  
\medskip

\centerline{- - - - - - - - - - - - - - - - - - - - - - - - - - - - - - - - - - - - - - - - - - - - - - - - - - - - - -}
  
  \medskip

Let $H(x)={f(x)-f(0)\over x-0}= {f(x)\over x}$. Since $f''(x)>0$,
the function  $f$ is strictly convex and both  its derivative and the quotient $H$
are  strictly increasing (see e.g. W. Walter, Analysis 1, Springer-Verlag, p. 303). 
Moreover, $H$ is
 continuous  on $]0,1]$. Note that $H(1)=f(1)$ and that
$H(0):=\lim_{x\to 0}f'(x)$ exists in $[-\infty, f(1)]$.   Hence, by the intermediate
value theorem, there exists for every value $w$ with $H(0)<w<H(1)$ a point
$b\in ]0,1[$
with $H(b)=w$. Now choose $a\in ]0,1[$ such that  
 $w:=f'(a)$ satisfies $H(0)<w<H(1)$ (such a choice obviously is possible).  
Thus there  exists
$b\in  ]0,1[$ so that   ${f(b)\over b}=H(b)=f'(a)$.  Choose $x_a\in ]0,a[$
so that $H(a)=f'(x_a)$. Due to the monotonicity of $f'$ we obtain:
$H(a)=f'(x_a)< f'(a)=H(b)$. Since $H$ is monotone, $b$ is unique and satisfies
 $a<b<1$.
 
 \newpage
 
      \begin{figure}[h!]
 
  \scalebox{0.5} 
  {\includegraphics{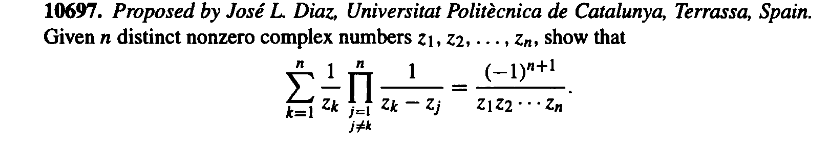}} 
\end{figure}

\centerline {\bf Solution to  problem 10697 AMM 105 (1998), p. 955}  \medskip

\centerline{Raymond Mortini}
  
\medskip

\centerline{- - - - - - - - - - - - - - - - - - - - - - - - - - - - - - - - - - - - - - - - - - - - - - - - - - - - - -}
  
  \medskip

This is nothing but a Lagrange interpolatory argument:\vs 0,3cm

In fact let $w_1,\cdots, w_n\in\C$. Then 

$$p(z)= \sum_{k=1}^nw_k{\prod_{j=1, j\ne k}^n(z-z_j)\over
\prod_{j=1, j\ne k}^n(z_k-z_j)} $$

is the unique polynomial of degree at most $n-1$ satisfying $p(z_k)=w_k,\;k=1,\cdots,n$.
Now choose $w_k=1$ for every $k$.  Since $q(z)\equiv 1$ satisfies the interpolation $q(z_k)=w_k$,
we obtain from uniqueness that $q=p$. Let $z=0$. Then

$$1=q(0)=\sum_{k=1}^n\prod_{{j=1\atop j\ne k}}^n{(-z_j)\over z_k-z_j}=(-1)^{n-1}
\sum_{k=1}^n{\prod_{j=1,j\ne k}^n z_j\over \prod_{j=1,j\ne k}^n(z_k-z_j)}.$$

Dividing by $\prod_{j=1}^nz_j$, yields the assertion

$$\sum_{k=1}^n{1\over z_k}\prod_{{j=1\atop j\ne k}}^n{1\over z_k-z_j}={(-1)^{n-1}\over \prod_{j=1}^nz_j}.$$

 \newpage
 
      \begin{figure}[h!]
 
  \scalebox{0.5} 
  {\includegraphics{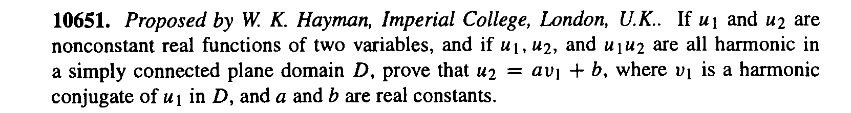}} 
\end{figure}

\centerline {\bf Solution to  problem 10651 AMM 105 (1998), p. 271}  \medskip

\centerline{Raymond Mortini}
  
\medskip

\centerline{- - - - - - - - - - - - - - - - - - - - - - - - - - - - - - - - - - - - - - - - - - - - - - - - - - - - - -}
  
  \medskip

We prove a stronger version than in the formulated problem (see \cite{mo2002}, which was based on  this).  \vs 0,2cm

{\bf Proposition 1} {\sl Let  $u$ and  $v$ be two non constant harmonic functions on a  domain $D\ss\C$.
Suppose that $uv$ is harmonic.  Then $u$ has an harmonic conjugate $\tilde u$
 on $D$ and  there are constants $a,b\in\R$ such that 
$$v=a\tilde u +b.\eqno (1)$$}\vs 0,3cm  
{\bf Remarks}. (1) If $u$ is a constant, then (1) is not true (because $v$ may be chosen to be any 
harmonic function). 

(2) If $v$ is a constant then (1) is true for $a=0$, provided a harmonic conjugate exists. A well known sufficient 
condition for the existence of a harmonic conjugate being that $D$ is simply connected. 

(3) Of course, if $v$ is any harmonic function satisfying (1), then $uv$ is harmonic.\vs 0,3cm

{\bf Solution} Let $\Delta$ be the Laplace operator. Because $\Delta u=\Delta v=0$ we obtain:
$$0=\Delta (uv)=\bigl(u_{xx}v+2u_xv_x+v_{xx}\bigr)+\bigl(u_{yy}v+2u_yv_y+v_{yy}\bigr)=
2\bigl(u_xv_x+u_yv_y\bigr).$$
Let $f=u_x-iu_y$ and $g=v_x-iv_y$.  The harmonicity of $u$ and $v$ imply that $f$ and $g$ satisfy the
Cauchy-Riemann differential equations; hence $f$ and $g$ are holomorphic. It is easy to see that
${\rm Re}\;{f\ov g}= u_xv_x+u_yv_y$. Thus ${\rm Re}\;{f\ov g}\equiv 0$ on $D$. \vs 0,2cm
Let $Z(g)=\{z\in D: g(z)=0\}$ denote the zero set of $g$.  It is a discrete subset of $D$  provided that
 $g\not\equiv 0$. Since $v$ is assumed not to be a constant, we see that $g\not\equiv 0$.
 Then on $D\setminus Z(g)$ we have
$\dis{\rm Re}\;{f\over g}={\rm Re}\;{f\ov g\over |g|^2}$. Thus $\dis{\rm Re}\;{f\over g}\equiv 0$ on $D\setminus 
Z(g)$.  This implies, in view of the analyticity,  that ${f\over g}$ is a pure imaginary constant, say 
${f\over g}\equiv i\lambda$ on $D\setminus Z(g)$. 
Hence $f=i\lambda g$ on $D$. 
 The definitions of $f$ and $g$ now
 yield that $u_x=\lambda v_y$ and $u_y=-\lambda v_x$.  Consequently, by the Cauchy-Riemann equations,  the function
$u+i\lambda v$ is holomorphic on $D$. In particular, $u$ has an harmonic conjugate on $D$. (Note that we do not 
have assumed  that $D$ is simply connected.) Thus, for any other harmonic conjugate $\tilde u$ of $u$,
we have  $\lambda v=\tilde u+c$ for some constant $c\in \R$. Note that $u$ not constant implies that $\lambda\ne 0$.
Thus $v$ has the desired form (1).

 \vs 0,5cm
A natural question now is the following. Let $u$ and $v$ be two harmonic functions on a domain $D\ss\C$. Then
$(u+iv)^2=u^2-v^2+2iuv$. Assume that $u^2-v^2$ is harmonic. What can be said for $v$? We have the 
following result:\\

{\bf Proposition 2} {\sl Assume that $u$, $v$ and $u^2-v^2$ are harmonic in a simply connected 
 domain $D\ss C$. Then there exists $a\in\R$ and $\theta\in [0,2\pi[$ such that 

$$v=\cos\theta\, u-\sin\theta\,\tilde u +a.\eqno (2)$$

Conversely, every function $v$ satisfying (2) for a harmonic function $u$ has the property that
$u^2-v^2$ is harmonic. }

\vs 0,2cm

{\bf Proof} Because $\Delta u=\Delta v=0$ we obtain:

$$0=\Delta (u^2-v^2)=2\Bigl(u_x^2+u_y^2-(v_x^2+v_y^2)\Bigr).$$

Hence $u_x^2+u_y^2=v_x^2+v_y^2$. Again, let $f=u_x-iu_y$ and $g=v_x-iv_y$.  As above, $f$ and $g$ 
are holomorphic on $D$. Moreover $|f|^2= |g|^2$. Thus $g$ is a rotation of $f$, say $g=e^{i\theta}f$. \vs 0,2cm

Let $z_0\in D$. Since $D$ is simply connected, $u$ and $v$  have harmonic conjugates $\tilde u$  and
 $\tilde v$ respectively, satisfying  $\tilde u(z_0)=\tilde v(z_0)=0$.  Let $F=u+i\tilde u$ and $G=v+i\tilde v$.
Then, by  Cauchy-Rieman, $F'=u_x+i\tilde u_x=u_x-iu_y=f$. Similiarly $G'=g$. Thus $G=e^{i\theta}F+c$ for some
constant $c\in\C$. Taking real parts yields

$$v=\cos\theta\, u-\sin\theta\,\tilde u +a$$

for some real constant $a$. The converse is easy to check.

The above results are related to the following more general result:\vs 0,2cm

{\bf Proposition 3.} {\sl  Let $h$ be an entire function and let $u:D\to\R$ and  $v:D\to\R$ be two
 nonconstant harmonic functions  in a simply connected domain $D$. Let $\tilde u$ be a harmonic 
conjugate of $u$ in $D$.
Then $h(u+iv):D\to \C$ is harmonic if and only if $v=\pm\tilde u +a$ for
a constant $a\in\R$.}\vs 0,2cm

{\bf Proof} Since $h$ is holomorphic, we have, by Cauchy-Riemann, $h_y=ih_x$ and $h_x=h'$. Hence
$h_{xx}=h'', h_{xy}=h_{yx}=ih''$ and $h_{yy}=-h''$. As above, let $f=u_x-iu_y$ and $g=v_x-iv_y$.Then 

$$\Delta [h\circ(u+iv)]= h''\circ(u+iv)\cdot[(|f|^2-|g|^2)+2i{\rm Re}\; f\ov g].$$

Obviously  $h''\circ q\not\equiv 0$ for any nonconstant continuous function $q$. 
Hence $h(u+iv)$ is harmonic if and only if  $|f|=|g|$ and ${\rm Re}\; f\ov g=0$. 
By the paragraphs
above we conclude that $f=i\lambda g$ for some $\lambda\in\R$. Hence $|\lambda|=1$. Thus $u_x=\pm v_y$ and $-u_y=
\pm v_x$. So $v$ or $-v$  is a harmonic conjugate of $u$ in $D$.  Therefore $v=\pm\tilde u +a$.
\vs 0,2cm 
To prove the converse, we have simply to note that the composition of a holomorphic
function with a holomorphic or anti-holomorphic function is harmonic.

\newpage

     \begin{figure}[h!]
 
  \scalebox{0.5} 
  {\includegraphics{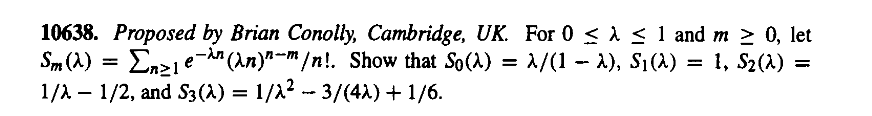}} 
\end{figure}

\centerline {\bf Solution to  problem 10638 AMM 105 (1998), p. 69}  \medskip

\centerline{Raymond Mortini}
  
\medskip

\centerline{- - - - - - - - - - - - - - - - - - - - - - - - - - - - - - - - - - - - - - - - - - - - - - - - - - - - - -}
  
  \medskip

In the following we present a solution  to problem number 10638. We shall not only compute the functions
$S_0,\cdots, S_3$, but we will give an explicit value for all $m\in \N$. To this end we need the following
Lemma.\vs 0,3cm

{\bf Lemma} {\sl  Let $f(z)= ze^z$. Then $f$ is invertible in a neighborhood of the origin in $\C$ and the inverse
function has the Taylor representation

$$f^{-1}(w)= \sum_{n=1}^\infty {n^{n-1}\over n!}(-1)^{n-1}w^n,$$

which converges for $|w|<{1\over e}$.}\vs 0,2cm

{\bf Proof}  By the residue theorem it is easy to see that whenever $f$ is holomorphic and injective in a disque
$D\ss \C$ (or even a simply connected domain),   then 

$$f^{-1}(w)n(\Gamma, f^{-1}(w))={1\over 2\pi i}\intop_{\Gamma}z\;{f'(z)\over f(z)-w}\;dz,$$

where $\Gamma$  is an arbitrary cycle (=finite union of closed, piecewise $C^1$-curves) in $D$. 

Applying this formula for $f(z)=ze^z$ and the disk $|z|<2\delta$, $\delta$ small enough,   we obtain :

$${d^n\over (dw)^n}f^{-1}(w)= {n!\over 2\pi i}\intop_{|z|=\delta} z\;{(z+1)e^z\over (ze^z-w)^{n+1}}\;dz.$$

Thus, for the power series $f^{-1}(w)=\sum_{n=0}^\infty a_nw^n$ we have $a_0=0$ and for $n\geq 1$:

$$a_n={1\over 2\pi i}\intop_{|z|=\delta} {z+1\over z^n} e^{-nz}\;dz= {1\over 2\pi i}\sum_{k=0}^\infty(-1)^k
\intop_{|z|=\delta}{z(nz)^k+(nz)^k\over z^nk!}\;dz= (-1)^{n-1}{n^{n-1}\over n!}.$$

By d'Alembert's rule it is easy to check that the radius of convergence is $1/e$.\hfill $\bigcirc$\vs 0,3cm

{\bf Proposition} {\sl For $0<\lambda<1$ and $m\in \Z$,  let $g_m(\lambda)=\lambda^mS_m(\lambda)$, where

$$S_m(\lambda)=\dis\sum_{n=1}^\infty e^{-\lambda n}(\lambda n)^{n-m}/n!.\eqno (1)$$

Then, for $m \in \{1,2,\cdots\}$, $g_m$ is a polynomial of degree $m$ vanishing at the origin, say $g_m(\lambda)=-\sum_{n=1}^m  b_{n,m}(-\lambda)^n$,
and the coefficients  $b_{n,m}$ are given by the recurrence relation 

$$b_{n,m}={1\over n}(b_{n,m-1}+b_{n-1,m-1}),\sp b_{1,1}=1.\eqno (3)$$

Solving these difference equations yields

$$b_{n,m}= \sum_{j=1}^n {1\over n!}{n\choose j }(-1)^{j-1} \left({1\over j}\right)^{m-n}.\eqno (4)$$}

{\bf Proof}  We note that, by Stirling's formula, the series $g_m(\lambda)$ converges locally uniformly in $0\leq\lambda<1$,
but does not converge whenever $\lambda=1$ and $m=0$.  Note that $g_m(0)=0$. 
Due to local uniform convergence, it is easy to see that, in order to obtain $ g'_m(\lambda)$, one can differentiate the
series for $g_m$ term by term. This yields that for $m\in \Z$

$$g'_m(\lambda)= {1-\lambda\over \lambda}g_{m-1}.\eqno (5)$$

Later we shall show that $g_1(\lambda)=\lambda$. Hence, by induction on (5), it is clear that for $m=1,2,\cdots$ the function 
 $g_m$ is a polynomial
vanishing at the origin, say $g_m(\lambda)=-\sum_{n=1}^m  b_{n,m}(-\lambda)^n$. If we let $x=-\lambda$, then we obtain
$\sum_{n=1}^m nb_{n,m}x^n=(1+x)\sum_{n=1}^m b_{n,m-1}x^n.$ Comparing coefficients, finally yields (3).

This difference equation can be solved by the usual methods. May be Maple or Mathematica gives the
solution. In any case, by the uniqueness of the solution,  it suffices to show that  (4) verifies the
 difference equation. Note also, that for $n>m$, the $b_{n,m}$ in (4) are $0$. This follows from the fact that
the $p$-th difference operator $D^p(a_n)=\sum_{j=0}^p {n\choose j}(-1)^j a_{n-j}$ vanishes identically
whenever $a_n$ is a polynomial (in $n$) of degree strictly less than $p$.\vs 0,2cm

For the readers convenience, here are the coefficients for
$m=1,\cdots, 5$: 

$$\begin{matrix}1&&&&\cr 1 &{1\over 2}&&&&\cr \noalign{\smallskip} 1 &{3\over 4}& {1\over 6}&&\cr\noalign{\smallskip}
1&{7\over 8}& {11\over 36}& {1\over 24}&\cr \noalign{\smallskip}1&{15\over 16}&{85\over 216}& {25\over 288}& {1\over 120}\end{matrix}$$
\vs 0,3cm
{\sl The case m=1}\sp In that case we have 
$$g_1(\lambda)= \lambda\sum_{n=1}^\infty e^{-\lambda n}(\lambda n)^{n-1}/n!=\sum_{n=1}^\infty {n^{n-1}\over n!}(\lambda e^{-\lambda})^n.$$

Let $w=-\lambda e^{-\lambda}$.  Now, for $w\in\C, |w|<{1\over e}$, the function $h(w)=\sum_{n=1}^\infty 
{n^{n-1}\over n!}(-1)^{n-1}w^n$ is, by Lemma 1, nothing but the inverse function of the holomorphic function $f(z)=ze^z$ ,
$|z|<\delta$ for sufficiently small $\delta>0$.  Thus $g_1(\lambda)=\lambda$. \vs 0,3cm

{\sl The case m=0}    \sp By (5) we see that $1=g_1'(\lambda)={1-\lambda\over\lambda}g_0(\lambda)$. Hence, $g_0(\lambda)={\lambda\over 1-\lambda}$. \vs 0,3cm

Using (5) it is also easy to derive, inductively, the values of $g_m$ for negative integers $m$.  For
example we get:\vs 0,2cm

$$g_{-1}(\lambda)={\lambda\over (1-\lambda)^3},\sp\sp g_{-2}(\lambda)={\lambda\over (1-\lambda)^5}(1+2\lambda),\sp\sp g_{-3}(\lambda)=
{\lambda\over (1-\lambda)^7}(1+8\lambda+6\lambda^2).$$

In general, one can convince oneself that for $m\in \Z, m<0$,  $g_m(\lambda)$ has the form
$g_m(\lambda)={\lambda\over (1-\lambda)^{-2m+1}}\;Q_m(\lambda)$, where $Q$ is a polynomial of degree $-m-1$ with value $1$
at the origin and satisfying the differential equations

$$Q_{m-1}(\lambda)=\lambda(1-\lambda)Q'_m(\lambda)+(1-2m\lambda)Q_m(\lambda).$$

Due to lack of time we were not able to solve this explicitely. May be Maple and Mathematica will be 
helpfull.

\newpage
     \begin{figure}[h!]
 
  \scalebox{0.5} 
  {\includegraphics{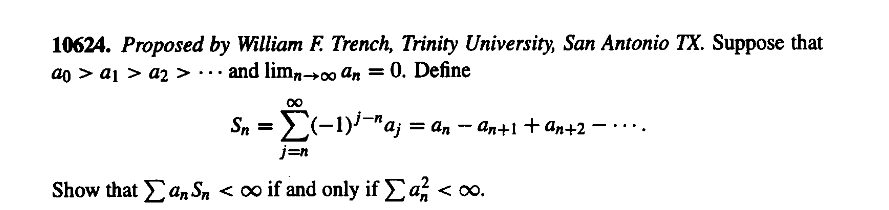}} 
\end{figure}

\centerline {\bf Solution to  problem 10624 AMM 104 (1997), p. 871}  \medskip

\centerline{Raymond Mortini}
  
\medskip

\centerline{- - - - - - - - - - - - - - - - - - - - - - - - - - - - - - - - - - - - - - - - - - - - - - - - - - - - - -}
  
  \medskip

By Leibniz's criteria, we know that $S_n$ actually converges and that $S_n\geq 0$ for every
$n\in \N$. Since $S_n=a_n-S_{n+1}$, we see that $S_n\leq a_n$ and so $\sum a_nS_n\leq \sum a_n^2$.
Thus the convergence of $\sum a_n^2$ implies the convergence of $\sum a_nS_n$.\vs 0,2cm

Now assume that $\sum a_nS_n= \sum(S_n+S_{n+1})S_n\sp (1)$ converges. Since all the terms are positive,
we deduce the convergence of the sums $\sum S_n^2 $ and $\sum S_{n+1}S_n$. A shift of the variable
yields that $\sum S_{n+1}^2$ converges. Hence $\sum S_{n+1}(S_{n+1}+S_n)\sp (2)$ converges. Summing
(1) and (2) yields that $\sum a_n^2 =\sum (S_{n+1}+S_n)^2= \sum(S_n+S_{n+1})S_n+\sum S_{n+1}(S_{n+1}+S_n)$
is convergent.\vs 2cm

 \newpage
 
      \begin{figure}[h!]
 
  \scalebox{0.5} 
  {\includegraphics{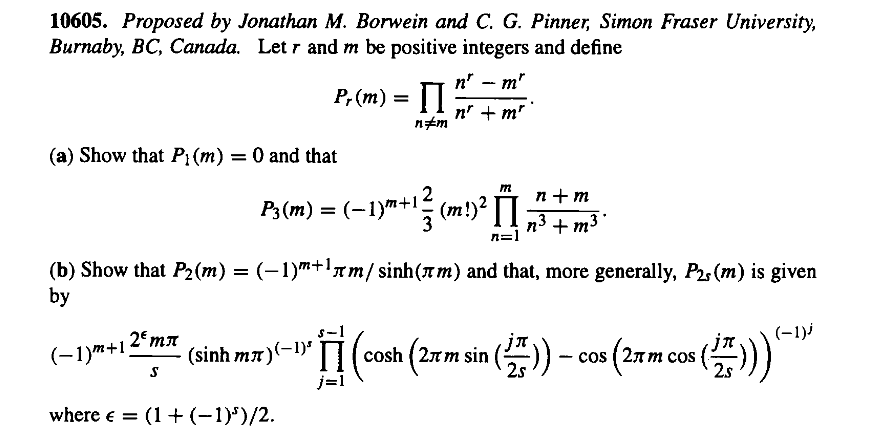}} 
\end{figure}

\centerline {\bf Solution to  problem 10605 (b) AMM 104 (1997), p. 567}  \medskip

\centerline{Raymond Mortini}
  
\medskip

\centerline{- - - - - - - - - - - - - - - - - - - - - - - - - - - - - - - - - - - - - - - - - - - - - - - - - - - - - -}
  
  \medskip

Write $\dis{n^{2s}-m^{2s}\over n^{2s}+m^{2s}}= {1-(m/n)^{2s}\over 1+(m/n)^{2s}} $. Let 
$\dis y=\left({m\over n}\right)^2$.
Then 
$${1-y^s\over 1+y^s}= {\prod_{j=0}^{s-1} \bigg(1-y\exp(-i{2\pi j\over s}) \bigg)\over 
\prod_{j=0}^{s-1} \bigg(1-y\exp(-i{\pi +2\pi j\over s})\bigg)}.$$

Since $\e\in\C$ is an $s$-root of $1$ [resp. (-1)] if and only if $\ov{\e}$ is an $s$-root, we obtain:

$${1-y^s\over 1+y^s} = {{(1-y)(1+y)\prod_{j=1}^{p-1}\left|1-y\exp(-i{2\pi j\over s})\right|^2\over
\prod_{j=1}^{p-1}\left|1-y\exp(-i{\pi (2j+1) \over s})\right|^2}}$$
 if $s=2p$ and

$${1-y^s\over 1+y^s} = {(1-y)\prod_{j=1}^{p}\left|1-y\exp(-i{2\pi j\over s})\right|^2\over (1+y)
\prod_{j=1}^{p-1}\left|1-y\exp(-i{\pi (2j+1) \over s})\right|^2}$$
 if $s=2p+1$.

This can be written by a single formula:

$${1-y^s\over 1+y^s} = (1-y)(1+y)^{(-1)^s}\prod_{k=1}^{s-1}\left|1-y\exp(-i{\pi k\over s})\right|
^{2(-1)^k}.\eqno (1)$$

In particular $$\prod_{k=1}^{s-1}\left|1-exp(-i{\pi k\over s})\right|^{2(-1)^k} = \lim_{y\to 1}
{\dis{1-y^s\over 1+y^s}\over (1-y)(1+y)^{(-1)^s}}={s\over 2\cdot 2^{(-1)^s}}.\eqno (2)$$

It is easy to check that 

$$ P:= \prod_{k=1}^{s-1}(2\pi^2m^2)^{(-1)^k}=\begin{cases}1, &\text{if $s$ is odd}\\
 {1\over 2\pi^2m^2},& \text{if $s$ is even.
}\end{cases}\eqno (3) $$

Now use the infinite product representation of the function $\sin \pi z$. This gives:

$${\sin \pi z\over \pi z} = \prod_{n=1}^\infty \left(1-{z^2\over n^2}\right),$$

and

$$ \prod_{n=1}^\infty \left(1+{z^2\over n^2}\right)={\sin i\pi z\over i\pi z}={\sinh \pi z\over \pi z}.$$

Moreover we have by de l'H\^opital's rule that

$$\prod_{n\not=m}^\infty \left(1-{m^2\over n^2}\right)= \lim_{z\to m} {\sin\pi z\over \pi z}
\bigg/1-\left({z\over m}\right)^2= {(-1)^{m+1}\over 2}.$$

Finally we need that $|\sin z|^2 = {1\over 2}(\cosh 2y -\cos 2x)$ for $z=x+iy$. 

Put all this together to get from (1)

$$2^{(-1)^s}\prod_{n\not=m}{n^{2s}-m^{2s}\over n^{2s}+m^{2s}}= {(-1)^{m+1}\over 2} 
\left({\sinh \pi m\over\pi m}\right)^{(-1)^s}
\cdot \prod_{k=1}^{s-1}\left|\prod_{n\not=m}\left(1-\left({m\over n}exp(-i{\pi k\over 2s})\right)^2\right)\right|^{2(-1)^k}=$$

$$={(-1)^{m+1}\over 2} \left({\sinh \pi m\over\pi m}\right)^{(-1)^s}\prod_{k=1}^{s-1}\left|
{\sin \bigl(\pi m \,\exp(-i{\pi k\over 2s})\bigr)
\over \pi m \left(1-\exp(-i{\pi k\over s})\right)}\right|^{2(-1)^k}=$$

$$= {(-1)^{m+1}\over 2} \left({\sinh \pi m\over\pi m}\right)^{(-1)^s}{\prod_{k=1}^{s-1}
\bigg[{1\over 2}\left(\cosh \bigg(2\pi m\sin{\pi k\over 2s}\bigg)-\cos\bigg(2\pi m
\cos{\pi k\over 2s}\bigg)\right)
\bigg]^{(-1)^k}\over \prod_{k=1}^{s-1}\left|1-\exp(-i{\pi k\over s})\right|^{2(-1)^k}\cdot
\prod_{k=1}^{s-1}(\pi m)^{2(-1)^k}}=$$

$$= {(-1)^{m+1}\over 2} \left({\sinh \pi m\over\pi m}\right)^{(-1)^s}{\prod_{k=1}^{s-1}
\bigg[\cosh\bigg( 2\pi m\sin{\pi k\over 2s}\bigg)-\cos\bigg(2\pi m\cos{\pi k\over 2s}\bigg)
\bigg]^{(-1)^k}\over \prod_{k=1}^{s-1}\left|1-\exp(-i{\pi k\over s})\right|^{2(-1)^k}\cdot
\prod_{k=1}^{s-1}(2\pi^2m^2)^{(-1)^k}}=$$

$$={ (-1)^{m+1}(\sinh \pi m)^{(-1)^s}  \prod_{k=1}^{s-1}
\bigg[\cosh \bigg(2\pi m\sin{\pi k\over 2s}\bigg)-\cos\bigg(2\pi m\cos{\pi k\over 2s}\bigg)\bigg]^{(-1)^k} \over
(\pi m)^{(-1)^s}   \dis{s\over  2^{(-1)^s}}\cdot P }$$

Clearly $$\dis {1\over (\pi m)^{(-1)^s}\,s\cdot P}= 2^{(1+(-1)^s)/2}\left({\pi m\over s}\right).$$

Putting $\e=(1+(-1)^s)/2$, we get the final equality:

$$P_{2s}=\prod_{n\not=m}{n^{2s}-m^{2s}\over n^{2s}+m^{2s}}=$$

$$= (-1)^{m+1}{2^\e m\pi\over  s}   (\sinh \pi m)^{(-1)^s}  \prod_{k=1}^{s-1}
\bigg[\cosh \bigg(2\pi m\sin{\pi k\over 2s}\bigg)-
\cos\bigg(2\pi m\cos{\pi k\over 2s}\bigg)\bigg]^{(-1)^k}. $$

If $s=1$ we interpret the empty product as $1$. This gives

$$P_2(m)=(-1)^{m+1}\pi m/\sinh(\pi m).$$\vs 1cm

\newpage
 
      \begin{figure}[h!]
 
  \scalebox{0.5} 
  {\includegraphics{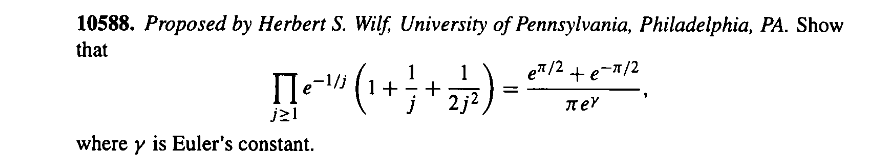}} 
\end{figure}

\centerline {\bf Solution to  problem 10588/10595 AMM 104 (1997), p. 456}  \medskip

\centerline{Raymond Mortini}
  
\medskip

\centerline{- - - - - - - - - - - - - - - - - - - - - - - - - - - - - - - - - - - - - - - - - - - - - - - - - - - - - -}
  
  \medskip

We show that 
$$P= \prod_{n=1}^\infty e^{-{1\over n}} \left( 1+{1\over n}+{1\over 2n^2}\right) = {e^{\pi/2}
+e^{-\pi/2}\over \pi e^\gamma}. $$\vs 2cm

Let $$\Gamma(z)= \left[e^{\gamma z}z \prod_{n=1}^\infty\left(1+{z\over n}\right)e^{z/n}\right]^{-1}$$

be the Gamma function and let $\e = {1\over 2}(1+i)$. Then $\ov{\e} = 1-\e$. Hence, as is well known, 

$$\Gamma(\e)\Gamma(\ov{\e}) = \Gamma(\e)\Gamma(1-\e) = {\pi \over \sin \pi \e}.$$

Therefore $${\sin\pi\e\over \pi} = e^{\gamma\e} \e \prod_{n=1}^\infty\left(1+{\e\over n}\right) 
e^{-\e/n}   \times
e^{\gamma\ov{\e}} \ov{\e} \prod_{n=1}^\infty\left(1+{\ov{\e}\over n}\right) e^{-\ov{\e}/n} =$$

$$=e^{\gamma} {1\over 2}\prod_{n=1}^\infty\left(1+{\e\over n}\right)\left( 1+{\ov{\e}\over n}\right) e^{-1/n} =$$
$$= e^{\gamma} {1\over 2}\prod_{n=1}^\infty\left(1+{1\over n}+{1\over 2n^2}\right)e^{-1/n}.$$

Hence $\dis P= {2\sin\pi\e\over \pi e^{\gamma}} = {2\cosh \pi/2\over \pi e^{\gamma}}$, which is the assertion.\vs 1cm

 \newpage

      \begin{figure}[h!]
 
  \scalebox{0.5} 
  {\includegraphics{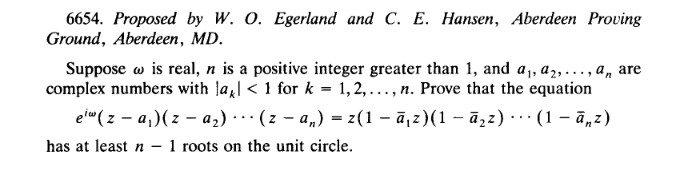}} 
\end{figure}
      \begin{figure}[h!]

\centerline {\bf Solution to  problem 6654 AMM 98 (1991), p. 273}  \medskip

\centerline{Raymond Mortini}
  
\medskip

\centerline{- - - - - - - - - - - - - - - - - - - - - - - - - - - - - - - - - - - - - - - - - - - - - - - - - - - - - -}
  
  \medskip

  \scalebox{0.5} 
  {\includegraphics{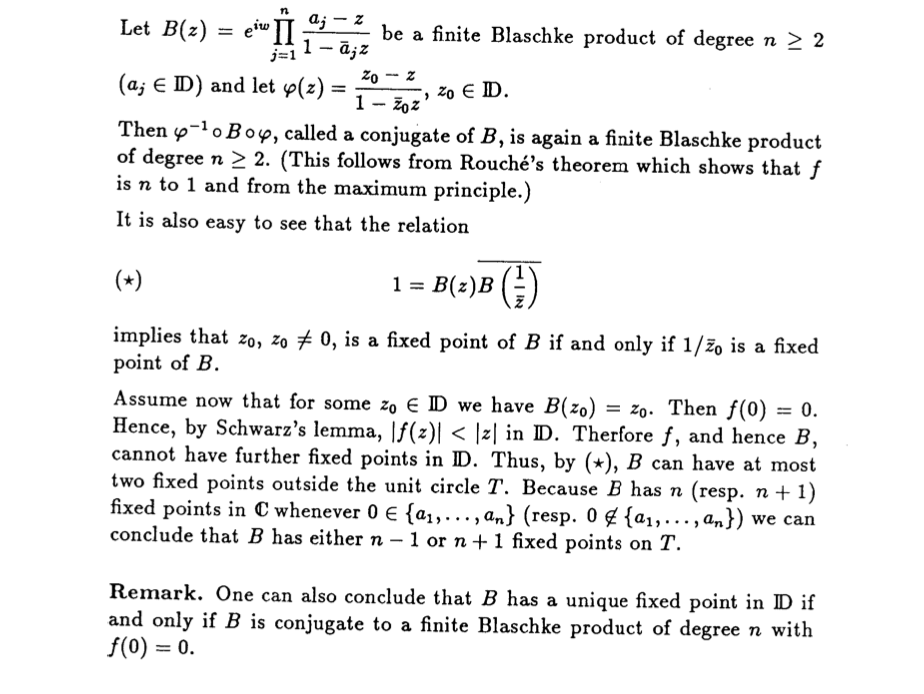}} 
\end{figure}

 \newpage

      \begin{figure}[h!]
 
  \scalebox{0.59} 
  {\includegraphics{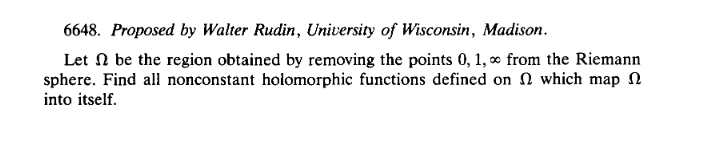}} 
\end{figure}

\centerline {\bf Solution to  problem 6648 AMM 98 (1991), p. 63}  \medskip

\centerline{Raymond Mortini}
  
\medskip

\centerline{- - - - - - - - - - - - - - - - - - - - - - - - - - - - - - - - - - - - - - - - - - - - - - - - - - - - - -}
  
  \medskip
      \begin{figure}[h!]
 
  \scalebox{0.6} 
  {\includegraphics{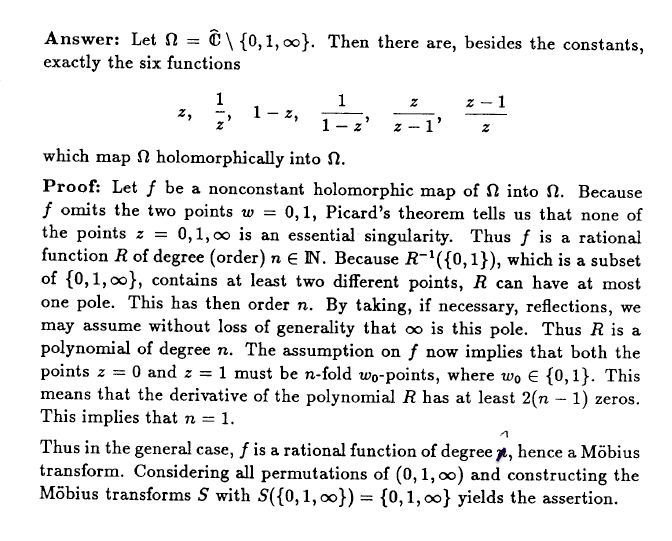}} 
\end{figure}

 \newpage

      \begin{figure}[h!]
 
  \scalebox{0.6} 
  {\includegraphics{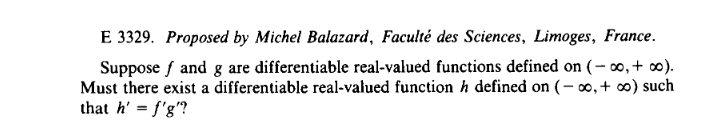}} 
\end{figure}

\centerline {\bf Solution to  problem E3329 AMM 96 (1989), p. 445}  \medskip

\centerline{Raymond Mortini}
  
\medskip

\centerline{- - - - - - - - - - - - - - - - - - - - - - - - - - - - - - - - - - - - - - - - - - - - - - - - - - - - - -}
  
  \medskip

      \begin{figure}[h!]
 
  \scalebox{0.60} 
  {\includegraphics{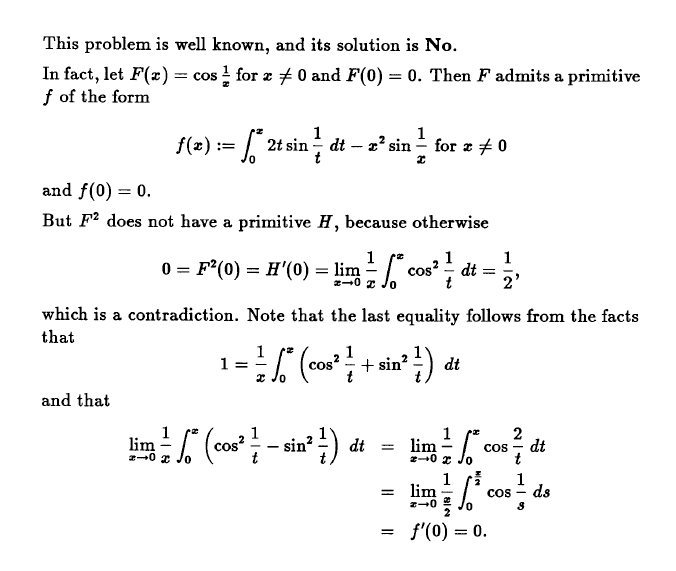}} 
\end{figure}

\newpage

      \begin{figure}[h!]
 
  \scalebox{0.45} 
  {\includegraphics{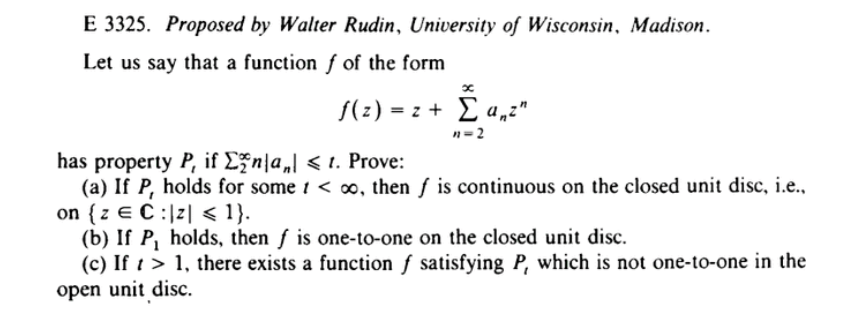}} 
\end{figure}

\centerline {\bf Solution to  problem E3325 AMM 96 (1989), p. 445}  \medskip

\centerline{Raymond Mortini}
  
\medskip

\centerline{- - - - - - - - - - - - - - - - - - - - - - - - - - - - - - - - - - - - - - - - - - - - - - - - - - - - - -}
  
  \medskip

      \begin{figure}[h!]
 
  \scalebox{0.5} 
  {\includegraphics{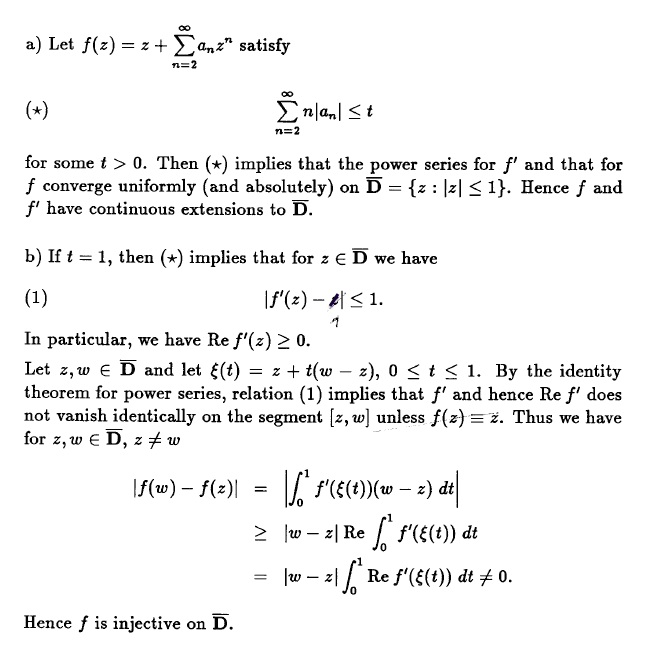}} 
\end{figure}

\vspace{-1cm}

      \begin{figure}[h!]
 \scalebox{0.38} 
 {\includegraphics{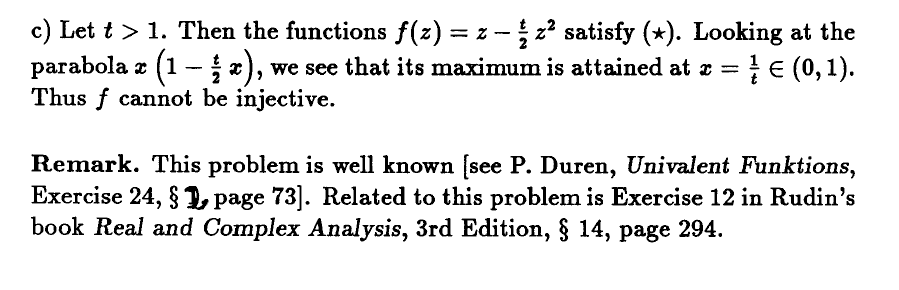}} 
\end{figure}


\newpage

\gr{\huge\section{Mathematics Magazine}}

\bigskip

\centerline{\gr{\copyright Mathematical Association of America, 2025.  }}


\bigskip


\begin{figure}[h!]
 
  \scalebox{0.45} 
  {\includegraphics{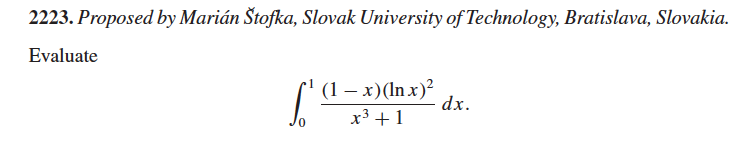}} 
\end{figure}

\centerline{\bf Solution to problem 2223 Math. Mag. 98 (3) 2025, p. 234}  \medskip

     \centerline{Raymond Mortini and Rudolf Rupp}

\medskip

\centerline{- - - - - - - - - - - - - - - - - - - - - - - - - - - - - - - - - - - - - - - - - - - - - - - - - - - - - -}
  
  \medskip
  

We show that
$$\ovalbox{$I:=\dis \int_0^1 \frac{(1-x)(\log x)^2}{x^3+1}\,dx= \frac{13}{9}\zeta(3)\sim 1.7363044156\dots$}.$$
\bigskip

First we note that for $j\in \N$ it can easily be proved via twice partial integration that
$$\int_0^1 x^j (\log x)^2 dx=  \frac{2}{(j+1)^3}.$$

Now, for $0\leq x<1$,  
\begin{eqnarray*}
\frac{(1-x)(\log x)^2}{1+x^3} &=&\sum_{j=0}^\infty (-1)^j (x^3)^j (\log x)^2-\sum_{j=0}^\infty (-1)^j x(x^3)^j (\log x)^2\\
&=& \sum_{j=0}^\infty (-1)^j x^{3j}(\log x)^2  -  \sum_{j=0}^\infty (-1)^j x^{3j+1}(\log x)^2.
\end{eqnarray*}

Note that we have $\int\sum=\sum\int$, since $\left|\sum_{n=1}^N (-1)^n x^{3n}\log^2 x\right|\leq \frac{\log^2x}{1+x^3}$ is an integrable majorant of both sums. Hence
\begin{eqnarray*}
I&=&2\sum_{n=0}^\infty   \frac{(-1)^n}{(3n+1)^3}-2\sum_{n=0}^\infty   \frac{(-1)^n}{(3n+2)^3}=: 2(T_1-T_2).
\end{eqnarray*}
Next we use that  $\dis \eta(3):=\sum_{n=0}^\infty  \frac{(-1)^n}{(n+1)^3}=\frac{3}{4}\zeta(3)$,
where $\eta$ is the Dirichlet $\eta$-function and $\zeta$ the Riemann $\zeta$-function. Since 
$$T_3:= \sum_{n=0}^\infty (-1)^n  \frac{1}{(3n+3)^3}= \frac{1}{27} \eta(3),$$
we deduce  from $T_1-T_2+T_3=\eta(3)$ that
$$I= 2(T_1-T_2)=2\left(\frac{3}{4}-\frac{1}{27}\cdot\frac{3}{4}\right)\zeta(3)=\frac{13}{9}\zeta(3).$$


\newpage


\begin{figure}[h!]
 
  \scalebox{0.45} 
  {\includegraphics{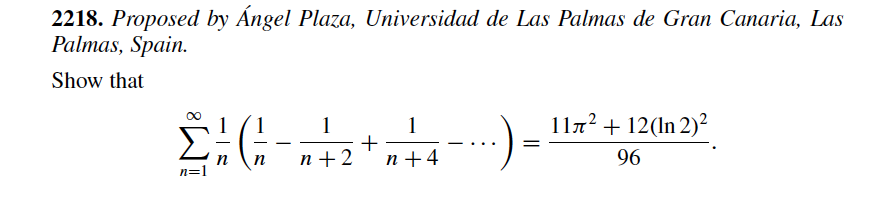}} 
\end{figure}

\centerline{\bf Solution to problem 2218 Math. Mag. 98 (2) 2025, p. 146}  \medskip

     \centerline{Raymond Mortini}

\medskip

\centerline{- - - - - - - - - - - - - - - - - - - - - - - - - - - - - - - - - - - - - - - - - - - - - - - - - - - - - -}
  
  \medskip
  

  {\bf Remark} This exercise appeared also  as CRUX 4965 (b) in \cite{cru}.\\

{\bf Step 1} {\it Reduction to the computation of an integral \footnote{ This stems from  sol. to CRUX 4965 by B. Chakraborty, R. Mortini and R. Rupp. }}.\\

We use that for $a>0$,
$$\frac{1}{a}=\int_0^\infty e^{-ax}dx.$$
Since the moduli of the  partial sums $\sum_{k=0}^N  e^{-nx} (-1)^k e^{-2kx} $ are bounded by the $L^1([0,\infty[)$-function $2e^{-nx}$,
 we have $\sum\int=\int\sum$, and so
\begin{eqnarray*}
S_n:=\frac{1}{n}-\frac{1}{n+2}+\dots &=& \sum_{k=0}^\infty \frac{(-1)^k}{n+2k}= \sum_{k=0}^\infty (-1)^k \int_0^\infty e^{-(n+2k)x} dx\\
&=&\int_0^\infty e^{-nx}\;\sum_{k=0}^\infty (-1)^k e^{-2kx} dx= \int_0^\infty \frac{e^{-nx}}{1+e^{-2x}}dx.
\end{eqnarray*}

Now let

$$B:=\sum_{n=1}^\infty\frac{1}{n}\; \left( \frac{1}{n}-\frac{1}{n+2}+\dots\right).$$

Since
$$0\leq S_n\leq  \int_0^\infty e^{-nx} dx= \frac{1}{n}.$$
 the series  $B$ converges absolutely as the general term is $\Oh(1/n^2)$.
Due to the boundedness of the partial sums of the  series $\sum_{n=1}^\infty \frac{1}{n} e^{-nx}$
 by the $L^1(]0,\infty[)$-function $|\log(1-e^{-x})|$, we have $\sum\int=\int\sum$. Hence

\begin{eqnarray*}
B&=& \sum_{n=1}^\infty\frac{1}{n}\;  \int_0^\infty \frac{e^{-nx}}{1+e^{-2x}}dx=
\int_0^\infty \frac{1}{1+e^{-2x}}\sum_{n=1}^\infty \frac{1}{n} e^{-nx} dx\\
&=&-\int_0^\infty \frac{\log(1-e^{-x})}{1+e^{-2x}}dx\buildrel=_{x\mapsto e^{-x}}^{} -\int_0^1 \frac{\log(1-x)}{x(1+x^2)}\; dx.
\end{eqnarray*}

\newpage

{\bf Step 2} { \it Evaluation of $ \int_0^1 \frac{\log(1-x)}{x(1+x^2)}\;dx$}.

The evaluation of this integral works along the same lines as that  where $\log(1-x)$ is replaced by $\log(1+x)$ (see \cite{11966} or \cite[p. 127-128]{chen}). For the readers convenience, we represent these  steps again, but adapted to the present minus sign. \\

The clue is to use that $\log(1-x)=-\int_0^1 \frac{x}{1-xt}\,dt$. Hence, by using Fubini's theorem for positive functions,

\begin{eqnarray*}
- \int_0^1 \frac{\log(1-x)}{x(1+x^2)}\;dx&=&-\int_0^1 \left( \frac{\log(1-x)}{x} - \frac{x\log(1-x)}{1+x^2}\right)dx\\
&=&-\int_0^1 \frac{\log s}{1-s}ds -\int_0^1\left( \int_0^1\frac{x ^2 }{(1-xt)(1+x^2)} dt\right)dx\\
&\buildrel=_{}^{{\rm Fubini}}&-\sum_{n=0}^\infty s^n \log s\; ds +\int_0^1 \frac{1}{1+t^2}\left[\int_0^1 \left(\frac{tx+1}{1+x^2} -\frac{1}{1-tx}\right)  dx\right] dt\\
&=&\sum_{n=0}^\infty \frac{1}{(n+1)^2}+\int_0^1  \frac{1}{1+t^2}\left(\frac{ t\log 2}{2}+\frac{\pi}{4}+\frac{\log(1-t)}{t}\right)\; dt\\
&=&\frac{\pi^2}{6} +\frac{\log^2 2}{4} +\left(\frac{\pi}{4}\right)^2+ \int_0^1 \frac{\log(1-t)}{t(1+t^2)}\;dt.
\end{eqnarray*}

Hence 
$$-\int_0^1  \frac{\log(1-t)}{t(1+t^2)}\;dt= \frac{\pi^2}{12}+\frac{\log^2}{8}+\frac{\pi^2}{32}
=\frac{11\pi^2+12 \log^2 2}{96}.$$

\newpage

\begin{figure}[h!]
 
  \scalebox{0.55} 
  {\includegraphics{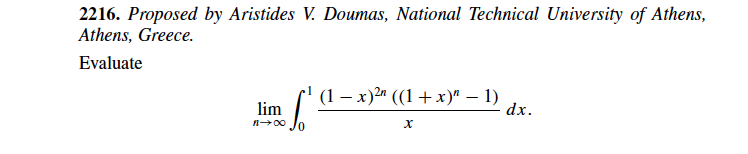}} 
\end{figure}

\centerline{\bf Solution to problem 2216 Math. Mag. 98 (2) 2025, p. 146}  \medskip

     \centerline{Raymond Mortini, Rudolf Rupp and Roberto Tauraso}

\medskip

\centerline{- - - - - - - - - - - - - - - - - - - - - - - - - - - - - - - - - - - - - - - - - - - - - - - - - - - - - -}
  
  \medskip
 
  We prove that
$$\ovalbox{$\dis I_n:=\int_0^1 \frac{(1-x)^{2n}\big((1+x)^n-1\big)}{x}dx\to \log 2$}.$$

To achieve this goal we use that the harmonic number 
$$H_n:=\sum_{j=1}^n \frac{1}{j}$$
can be written as
$$H_n=\int_0^1 \frac{1-x^n}{1-x}dx=\int_0^1\frac{1-(1-t)^n}{t}dt.$$
Now let 
$$J_n:= \int_0^1 \frac{(1-x)^n-(1-x)^{2n}}{x}dx.
$$
Then, using the addendum below,   $J_n=H_{2n}-H_n\to \log 2$.
But $I_n-J_n\to 0$. In fact, 

\begin{eqnarray*}
I_n-J_n&=&\int_0^1 \frac{(1-x)^{2n}(1+x)^n -(1-x)^n}{x}dx\\
&=&\int_0^1\frac{(1-x)^n\big( (1-x^2)^n-1\big)}{x} dx=:\int_0^1 L_n(x)dx.
\end{eqnarray*}
and, since $0\leq x^{2j-1}\leq x^j$, 
\begin{eqnarray*}
\frac{\left|(1-x^2)^n-1\right|}{x}&=&\left|\sum_{j=1}^n \binom{n}{j}(-1)^j x^{2j-1}\right|\leq\co{ \sum_{j=1}^n  \binom{n}{j} x^j\leq (1+x)^n.}
\end{eqnarray*}
Hence, for $0<x\leq 1$ 
$$|L_n(x)|\leq (1-x^2)^n\to 0.$$
Now just apply Lebesgue's dominated convergnce theorem to conclude that $\int_0^1 L_n(x)dx\to 0$.\\

{\bf Addendum}
It is well-known that $\lim_{N\to\infty} \sum_{n=N+1}^{2N} \frac{1}{n}=\log 2$, as it is a direct consequence to 
the fact that the Euler-Mascheroni constant $\gamma$ is given by 
$$\gamma=\lim(H_n-\log n),$$
and so 
$$H_{2N}-H_N=(H_{2N}-\log  (2N)- \gamma)+(\log N +\gamma -H_N) +\log 2\to \log 2.$$

\newpage

\begin{figure}[h!]
 
  \scalebox{0.55} 
  {\includegraphics{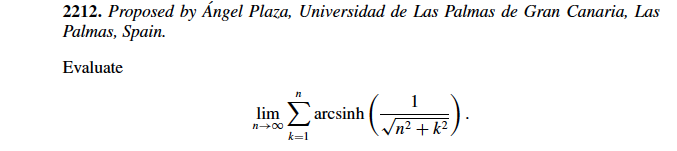}} 
\end{figure}

\centerline{\bf Solution to problem 2212 Math. Mag. 98 (1) 2025, p. 67}  \medskip

     \centerline{Raymond Mortini, Rudolf Rupp}

\medskip

\centerline{- - - - - - - - - - - - - - - - - - - - - - - - - - - - - - - - - - - - - - - - - - - - - - - - - - - - - -}
  
  \medskip
  
  
  Note that for $x\in \R$, we have $A(x):=\arsinh x=\log(x+\sqrt{1+x^2})$ and that $A'(x)=(1+x^2)^{-1/2}$.
Let 
$$S_n:=\sum_{k=1}^n \arsinh \left(\frac{1}{\sqrt{n^2+k^2}}\right).$$
We show that
$$\ovalbox{$\dis \lim_{n\to\infty} S_n=\arsinh 1=\log(1+\sqrt 2)$.}$$
To this end, we first use that with $a_{k,n}:= (n^2+k^2)^{-1/2}$ and the substitution $x=a_{k,n}u$
 
\begin{eqnarray*}
S_n&=& \sum_{k=1}^n \int_0^{a_{k,n}}\frac{dx}{\sqrt{1+x^2}}=  \sum_{k=1}^n \int_0^1 \frac{a_{k,n}}{\sqrt{1+a_{k,n}^2 u^2}\;}du=
\sum_{k=1}^n \int_0^1 \frac{du}{\sqrt{n^2+k^2+u^2}}\\
&=& \int_0^1 \left(\sum_{k=1}^n \frac{1}{n} \;\frac{1}{\sqrt{1+(k/n)^2+(u/n)^2}}\right)\;du=:\int_0^1 R_n(u)du.
\end{eqnarray*}
Next we use the estimates
\begin{eqnarray*}
1+(k/n)^2+(u/n)^2&\leq& 1+(k/n)^2+(1/n)^2\leq 1 +\big((k+1)/n\big)^2\\
1+(k/n)^2+(u/n)^2&\geq& 1 +(k/n)^2.
\end{eqnarray*}
Consequently, for every $u\in[0,1]$,
$$\sum_{k=1}^{n-1} \frac{1}{n}\;\frac{1}{\sqrt{ 1 +\big((k+1)/n\big)^2}}\leq R_n(u)\leq 
\sum_{k=1}^{n} \frac{1}{n}\;\frac{1}{\sqrt{ 1 +(k/n)^2}}.$$
These bounds for $R_n(u)$ are related to the lower, $L_n$, and  upper   Riemann sums,  $T_n$, for the monotone decreasing function
$f(x):=\frac{1}{\sqrt{1+x^2}}$. 
In fact, using the decomposition $\{0, 1/n, 2/n,\dots (n-1)/n, 1\}$ of $[0,1]$
$$\mbox{$\dis L_n:=\sum_{j=1}^{n} \frac{1}{n}\; \frac{1}{\sqrt{1+(j/n)^2}}$ and
$\dis T_n:=\sum_{j=0}^{n-1} \frac{1}{n}\; \frac{1}{\sqrt{1+(j/n)^2}}$},$$
we obtain
$$
\co{L_n\leq R_n(u)+\frac{1}{\sqrt{1+n^2}}}\;\leq \left(T_n-\frac{1}{n}+ \frac{1}{n\sqrt{2}}\right) +
\frac{1}{\sqrt{1+n^2}}\;\co{\leq T_n+\frac{1}{\sqrt{1+n^2}}}.
$$
Intregrating the red parts of the estimates with respect to $u$ yields (note that the $L_n$ and $R_n$ are independent of $u$)

$$L_n- \frac{1}{\sqrt{1+n^2}}\leq \int_0^1 R_n(u)\, du\leq T_n.$$

Since 
$$\lim_{n\to\infty} L_n=\lim_{n\to\infty}T_n= \int_0^1 \frac{1}{\sqrt{1+x^2}}dx= \arsinh 1,$$
we deduce that
$$S_n= \int_0^1 R_n(u)du \to \arsinh 1.$$
\bigskip

{\bf Remark} A function ``arcsinh x" does not exist; this wrong terminology should be avoided (used probably first by the software developers).  Also avoided should be the  ambiguous  notation, $\sinh^{-1} x$. The correct notation is  $\arsinh x$. This has to do with the area, and not the arcs.

\newpage

\begin{figure}[h!]
 
  \scalebox{0.45} 
  {\includegraphics{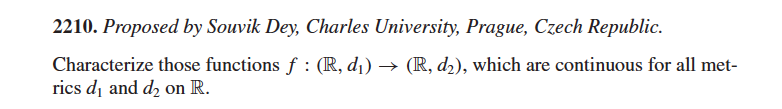}} 
\end{figure}

\centerline{\bf Solution to problem 2210 Math. Mag. 97 (2) 2024, p. 575}  \medskip

     \centerline{Raymond Mortini}

\medskip

\centerline{- - - - - - - - - - - - - - - - - - - - - - - - - - - - - - - - - - - - - - - - - - - - - - - - - - - - - -}
  
  \medskip

This does not seem to be  a serious problem proposal.  Let $X:=(\R,d_1)$, where $d_1$ is the usual Euclidean metric $d_1(x,y)=|x-y|$. 
Now take as $d_2$ the discrete metric defined by
$d_2(x,y)=1$ if $x\not=y$ and $d_2(x,x)=0$ (the triangle inequality is trivially satisfied).  Since continuous functions map connected sets to connected sets, and since the only non-empty connected sets in $Y=(\R, d_2)$ are the singletons,  any continuous map $f:X\to Y$ must be constant (since 
 $f(\R)$ is connected, hence a singleton). Conversely, every constant function $f:M\to \tilde M$, $x\mapsto c$, between any metric spaces $M,\tilde M$, is continuous since 
 $f^{-1}[V]=M$ for every  open set $V$ in $\tilde M$ with $c\in V$, and $f^{-1}[V]=\emp$ for any open set $V$ in
  $\tilde M$ with $c\notin V$.


\newpage


  \begin{figure}[h!]
 
  \scalebox{0.45} 
  {\includegraphics{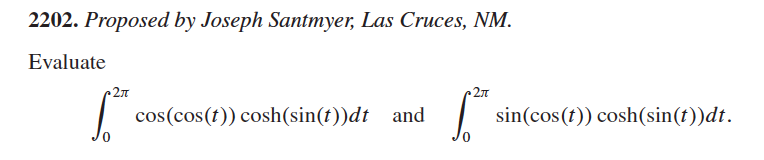}} 
\end{figure}

\centerline{\bf Solution to problem 2202 Math. Mag. 97 (2) 2024, p. 434}  \medskip

     \centerline{Raymond Mortini,  Rudolf Rupp}

\medskip

\centerline{- - - - - - - - - - - - - - - - - - - - - - - - - - - - - - - - - - - - - - - - - - - - - - - - - - - - - -}
  
  \medskip

Using complex analysis, we prove that
$$\ovalbox{$I_1:=\int_0^{2\pi} \cos(\cos t)\;\cosh(\sin t)\;dt=2\pi$}$$
and
$$\ovalbox{$I_2=\int_0^{2\pi} \sin(\cos t)\;\cosh(\sin t)\;dt=0$}.$$

First note that $\cosh z=\cos (iz)$. Hence, by using that  
$$\cos x+\cos y= 2 \cos\left(\frac{x+y}{2}\right)\;\cos\left(\frac{x-y}{2}\right),$$
and
$$\sin x+\sin y= 2 \sin\left(\frac{x+y}{2}\right)\;\cos\left(\frac{x-y}{2}\right),$$

we obtain
\begin{eqnarray*}
 \cos(\cos t)\;\cosh(\sin t)&=& \cos(\cos t)\;\cos(i\sin t)\\
 &=&\cos\left(\frac{e^{it}+e^{-it}}{2}\right)\; \cos\left(\frac{e^{it}-e^{-it}}{2}\right)\\
 &=& \frac{1}{2}\left(\cos(e^{it})+\cos(e^{-it})\right),
\end{eqnarray*}
and
\begin{eqnarray*}
 \sin(\cos t)\;\cosh(\sin t)&=& \sin(\cos t)\;\cos(i\sin t)\\
 &=&\sin\left(\frac{e^{it}+e^{-it}}{2}\right)\; \cos\left(\frac{e^{it}-e^{-it}}{2}\right)\\
 &=& \frac{1}{2}\left(\sin(e^{it})+\sin(e^{-it})\right).
\end{eqnarray*}

Now, by Cauchy's residue theorem in complex analysis,

\begin{eqnarray*}
I_1&=& \frac{1}{2}\;\int_0^{2\pi} \left(\cos(e^{it})+\cos(e^{-it})\right)\;dt\buildrel=_{}^{z=e^{it}} 
 \frac{1}{2i}\;\ointctrclockwise \frac{\cos z+\cos (1/z)}{z}\;dz\\
 &=&\frac{1}{2i} 2\pi i \left( {\rm Res}\left[\frac{\cos z}{z},0\right]+  {\rm Res}\left[\frac{\cos(1/z)}{z},0\right]\right)=\pi ( 1+1)=2\pi,
\end{eqnarray*}

and

\begin{eqnarray*}
I_2&=& \frac{1}{2}\;\int_0^{2\pi} \left(\sin(e^{it})+\sin(e^{-it})\right)\;dt\buildrel=_{}^{z=e^{it}} 
 \frac{1}{2i}\;\ointctrclockwise \frac{\sin z+\sin (1/z)}{z}\;dz\\
 &=&\frac{1}{2i} 2\pi i \left( {\rm Res}\left[\frac{\sin z}{z},0\right]+  {\rm Res}\left[\frac{\sin(1/z)}{z},0\right]\right)=\pi ( 0+0)=0.
\end{eqnarray*}

\newpage

\nopagecolor


  \begin{figure}[h!]
 
  \scalebox{0.5} 
  {\includegraphics{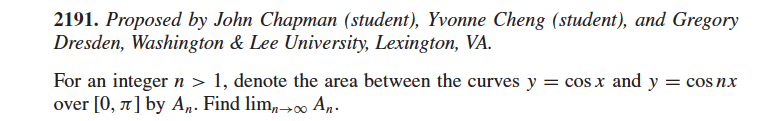}} 
\end{figure}

\centerline{\bf Solution to problem 2191 Math. Mag. 97 (2) 2024, p. 223}  \medskip

     \centerline{Raymond Mortini,  Rudolf Rupp}

\medskip

\centerline{- - - - - - - - - - - - - - - - - - - - - - - - - - - - - - - - - - - - - - - - - - - - - - - - - - - - - -}
  
  \medskip

We prove that this area $A_n$ tends to \ovalbox{$8/\pi\sim 2.5464790\dots$}. 

To this end, we use the Fourier series for 
$f(x):=|\sin x|$
\begin{equation}\label{fouriersin}
|\sin x|=\frac{2}{\pi}-\frac{4}{\pi}\;\sum_{k=1}^\infty \frac{\cos(2kx)}{4k^2 -1},~~ x\in \R,
\end{equation}
which can easily be obtained by noticing that  that $f$ is  even and $\pi$-periodic, and so 
$$f(x)\sim \frac{a_0}{2} + \sum_{n=1}^\infty a_k \cos(2kx)$$
 with
$$a_k=\frac{2}{\pi}\int_0^\pi f(x)\cos(2kx) dx=\frac{2}{\pi} \int_0^\pi \sin x \cos(2kx)dx\
=\frac{1}{\pi} \int_0^\pi \big(\sin(1+2k)x+\sin(1-2k)x\big)dx.$$

Note  that the desired area $A_n$ is given by $\int_0^\pi |\cos x-\cos (nx)|dx$. By the addition theorem
$$ |\cos x-\cos (nx)|= 2 |\sin \left((n-1)(x/2)\right)\; \sin \left((n+1)(x/2)\right)|.$$
Hence
\begin{eqnarray*}
q(x):=2 |\sin(n-1)(x/2)\sin(n+1)(x/2)|&=&2\left(\frac{2}{\pi}-\frac{4}{\pi}\;\sum_{k=1}^\infty \frac{\cos k(n-1)x}{4k^2 -1}\right)\;
\left(\frac{2}{\pi}-\frac{4}{\pi}\;\sum_{k=1}^\infty \frac{\cos k(n+1)x}{4k^2 -1}\right).
\end{eqnarray*}
Since for $m,n\in \N=\{0,1,2,\dots\}$, $m\not=n$,
$$\int_0^\pi \cos(mx)\cos(nx)dx=\frac{1}{2}\int_0^\pi\big(\cos(m-n)x +\cos (m+n)x\big)dx=0,$$ 
we deduce from the fact
that  the Fourier series  (\ref{fouriersin}) is  absolutely and uniformly convergent,  that for $n\geq 2$
\begin{eqnarray*}
A_n=\int_0^\pi q(x) dx&=&\frac{8}{\pi^2}\pi+\frac{32}{\pi^2} \sum_{j,k=1}^\infty \int_0^\pi \frac{\cos j(n-1)x\;\cos k(n+1)x}{(4k^2-1)(4j^2-1)}dx\\
&=& \frac{8}{\pi }+\frac{16\pi }{\pi^2}\sum_{j,k=1\atop j(n-1)=k(n+1)}^\infty \frac{1} {(4k^2-1)(4j^2-1)}.
\end{eqnarray*}

Now $j(n-1)=k(n+1)$ if $j=r (n+1)m$ and $k=r(n-1)m$ for some $m\in \N$ and 
$$r=\begin{cases} 1 &\text{if $n$ is even}\\
\frac{1}{2}&\text{if $n$ is odd}
\end{cases}$$
In fact, if $n$ is even, then the odd numbers $n-1$ and $n+1$ are relatively prime, since otherwise a joint divisor must also divide $2=(n+1)-(n-1)$.  Hence $k=m(n-1)$ and $j=m(n+1)$ for some $m\in \N$. And if $n$ is odd,  then $n-1$ and $n+1$ are even and $2$ then is the gcd of $n-1$ and $n+1$ since $(n+1)-(n-1)=2$. Now let $n\geq3$.
Due to $1\leq 2 (1/4) m^2(n\pm 1)^2\leq 2  r^2m^2(n+1)^2$, we finally conclude  that 
\begin{eqnarray*}
\frac{8}{\pi }\leq A_n&=&  \frac{8}{\pi }+\frac{16\pi }{\pi^2}\sum_{m=1}^\infty \frac{1}{(4r^2m^2(n-1)^2-1)(4r^2m^2(n+1)^2-1)}\\
&\leq & \frac{8}{\pi }+\frac{16\pi }{\pi^2}\sum_{m=1}^\infty \frac{1}{2r^2m^2(n+1)^2}= \frac{8}{\pi }+\frac{1}{(n+1)^2}\frac{16\pi }{\pi^2} \frac{1}{2r^2}
\sum_{m=1}^\infty \frac{1}{m^2}\\
&\to&  \frac{8}{\pi }.
\end{eqnarray*}

{\bf Remark.} 
Using that for $a\notin\Z$, $\dis \sum_{m=1}^\infty \frac{1}{m^2-a^2}=\frac{1-a\pi \cot(a\pi)}{2a^2}$, we obtain
\begin{eqnarray*}
A_n&=&  \frac{8}{\pi }+\frac{16 }{\pi} \frac{1}{16nr^2}\sum_{m=1}^\infty \left(\frac{1}{m^2-\frac{1}{4r^2(n-1)^2}}-
 \frac{1}{m^2-\frac{1}{4r^2(n+1)^2}}\right)\\
 &=& \frac{8}{\pi }+\frac{16 }{\pi} \frac{1}{16nr^2}\left(  \frac{1- \frac{\pi}{2r(n-1)} \cot (\frac{\pi}{2r(n-1)})}{\frac{2}{4r^2(n-1)^2}}
 - \frac{1- \frac{\pi}{2r(n+1)} \cot (\frac{\pi}{2r(n+1)})}{\frac{2}{4r^2(n+1)^2}}\right)\\
 &=& \frac{8}{\pi }+\frac{2}{\pi n r^2}\left(r^2(n-1)^2- \frac{r(n-1)\pi}{2}\cot \frac{\pi}{2r(n-1)}
 -r^2(n+1)^2+ \frac{r(n+1)\pi}{2}\cot \frac{\pi}{2r(n+1)}\right)\\
 &=&-\frac{n-1}{nr}\cot \frac{\pi}{2r(n-1)} +\frac{n+1}{nr}\cot \frac{\pi}{2r(n+1)}.
\end{eqnarray*}

For example, \\

  $\dis A_5=\frac{4}{5}(3\sqrt 3-2)\sim 2.5569219\dots$\\
  
  $\dis A_4= \frac{5}{4}\sqrt{5+2\sqrt 5}-\frac{3\sqrt 3}{4}\sim 2.54806631\dots$\\
  
  $\dis A_3=\frac{8}{3}\sim 2.6666666\dots$\\
  
  $\dis A_2=\frac{3\sqrt 3}{2}\sim 2.59807621\dots$.

 It is now very easy to determine  the limit directly (one may replace $r$ above even by any $x\not=0$):
  Let $a:=\frac{\pi}{2x(n+1)}$ und $b:=\frac{\pi}{2x(n-1)}$. Then
$$xA_n=(\cot a-\cot b) +\frac{1}{n}(\cot a+\cot b),$$
and so

$$xA_n= \frac{\sin(b-a)}{\sin a \sin b} + \frac{1}{n}\;\frac{\sin(a+b)}{\sin a \sin b}.$$

Note that
$$\mbox{$\dis b-a=\frac{\pi}{x}\;\frac{1}{n^2-1}$ and $\dis a+b= \frac{\pi}{x}\; \frac{n}{n^2-1}.$} $$
Since $\lim_{x\to 0}\frac{\sin x}{x}=1$,   $xA_n$ has the same asymptotic as
$$\frac{\frac{\pi}{x}\;\frac{1}{n^2-1}}{\frac{\pi}{2x(n+1)}\; \frac{\pi}{2x(n-1)}}+
 \frac{1}{n}\; \frac{\frac{\pi}{x}\; \frac{n}{n^2-1}}{\frac{\pi}{2x(n+1)}\; \frac{\pi}{2x(n-1)}}=\frac{4x}{\pi}+\frac{4x}{\pi}=\frac{8x}{\pi}.
 $$

\newpage


  \begin{figure}[h!]
 
  \scalebox{0.5} 
  {\includegraphics{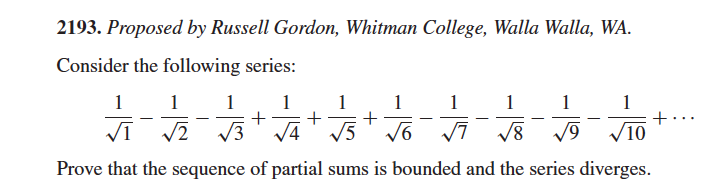}} 
\end{figure}

\centerline{\bf Solution to problem 2193 Math. Mag. 97 (2) 2024, p. 223}  \medskip

     \centerline{Raymond Mortini,  Rudolf Rupp}
            
\medskip

\centerline{- - - - - - - - - - - - - - - - - - - - - - - - - - - - - - - - - - - - - - - - - - - - - - - - - - - - - -}
  
  \medskip

Let  $\N=\{0,1,2,\dots\}$.
We assume that the terms of the series $\sum_{j=1}^\infty a_j$  are regrouped as follows (this regrouping does not change the partial sums):
 
 $$\frac{1}{\sqrt 1}-\left(\frac{1}{\sqrt 2} +\frac{1}{\sqrt 3}\right)+\left(\frac{1}{\sqrt4}+\frac{1}{\sqrt5}+\frac{1}{\sqrt 6}\right)-\left(\frac{1}{\sqrt 7}+\frac{1}{\sqrt8}+\frac{1}{\sqrt9}+\frac{1}{\sqrt{10}}\right) +-\cdots
 =:\sum_{k=1}^\infty (-1)^{k-1}P_k,$$
 where for $k\in\N$ 
 $$P_k:=\sum_{j=A_k}^{A_{k+1}-1} \frac{1}{\sqrt j},$$
and where  the number $A_k=\frac{k(k-1)}{2}+1$ is  the $j$-index of the first summand in the $k$-th  group.
The sum $P_k$ has $k$ summands. We now estimate each $P_k$:
 
\begin{equation}\label{pk-abs}
P_k\leq \sum_{j=A_k}^{A_{k+1}-1} \frac{1}{\sqrt j}\leq \frac{k}{\sqrt{A_k}}=\frac{k\sqrt 2 }{\sqrt{k(k-1)+2}}
 \leq 2\sqrt 2.
\end{equation}
 
 $$P_k\geq \frac{k}{\sqrt{A_{k+1}-1}}=\frac{\co{\sqrt 2}\,k}{\sqrt{k(k+1)}}\geq \frac{\co{\sqrt 2}\,k}{k+1}\geq \frac{\co{\sqrt 2}}{2}.
 $$
 
 By Cauchy's criterion, the series diverges as the blocks $P_k$ do not go to zero. Next we estimate 
  the partial sums
 $$
 \sum_{k=1}^N (-1)^{k-1}P_k.
$$
 
 To this end, we  use the useful inequality
 $$2(\sqrt{n+1}-\sqrt n)\leq\frac{1}{\sqrt n} \leq 2 (\sqrt n-\sqrt{n-1}).$$
 
 This yields (via telescoping property)
 $$ P_k\leq  2 (\sqrt{A_{k+1}-1}-\sqrt{A_k-1})= 2\sqrt 2\frac{\sqrt k}{\sqrt{k+1}+\sqrt{k-1}}=
 \sqrt 2 \left(1+ \frac{1}{8}\;\frac{1}{k^2}\right)+\oh(\frac{1}{k^2})$$
 $$P_k\geq 2( \sqrt{A_{k+1}}-\sqrt{A_k})\geq2\sqrt 2\frac{k}{\sqrt{k(k+1)+2}+\sqrt{k(k-1)+2}}=
 \sqrt 2\left(1-\frac{7}{8}\;\frac{1}{k^2}\right)+\oh(\frac{1}{k^2}).
 $$
 
Numerical computations let us guess that  the $P_k$ are {\it increasing}. 
Anyway,  the partial sums  formed with full blocks write as

 $$L_N:=\sum_{k=1}^N (-1)^{k-1}P_k=\begin{cases}(P_1-P_2)+(P_3-P_4)+\cdots+(P_{N-1}-P_N)&\text{if $N$ even}\\
 (P_1-P_2)+(P_3-P_4)+\cdots+(P_{N-2}-P_{N-1})+P_N &\text{if $N$ odd},
 \end{cases}
 $$
and the general one is given by

\begin{equation}\label{partisum}
S_n:=\sum_{j=1}^ n a_j=\sum_{k=1}^N (-1)^{k-1}P_k+(-1)^{N}\sum_{j=A_{N+1}}^n \frac{1}{\sqrt j},
\end{equation}

where $N$ is the unique number for which $A_{N+1}=\frac{N(N+1)}{2} +1\leq n< \frac{(N+2)(N+1)}{2}=A_{N+2}-1$.

 It remains to   estimate the differences $|P_k-P_{k+1}|$.

 \begin{eqnarray*}
 P_k-P_{k+1}&\leq & \left(\sqrt 2+\frac{\sqrt 2}{8}\;\frac{1}{k^2}+\oh(\frac{1}{k^2})\right)
  - \left(\sqrt 2 -\frac{7\sqrt 2}{8}\;\frac{1}{(k+1)^2}+\oh(\frac{1}{(k+1)^2})\right)\\
  &\leq&\frac{\sqrt 2}{k^2} +\oh(\frac{1}{k^2})
\end{eqnarray*}

\begin{eqnarray*}
P_k-P_{k+1}&\geq&  \left(\sqrt 2-\frac{7\sqrt 2}{8}\;\frac{1}{k^2}+\oh(\frac{1}{k^2})\right)
-\left(\sqrt 2 +\frac{\sqrt 2}{8}\;\frac{1}{(k+1)^2}+\oh(\frac{1}{(k+1)^2})\right)\\
&\geq &-\frac{\sqrt 2}{k^2} +\oh(\frac{1}{k^2}).
\end{eqnarray*}
 
 Hence, for all $k\geq k_0$,
 $$|P_k-P_{k+1}|\leq \frac{\sqrt 2}{k^2}+\oh(\frac{1}{k^2})\leq  \frac{\sqrt 2+1}{k^2}
 $$
  and so, for all $k$, 
 \begin{equation}\label{pkabs}
|P_k-P_{k+1}|\leq C {k^2}.
\end{equation}
 We deduce from (\ref{pk-abs}) and (\ref{partisum}) that for every $n$ and $N$ chosen as above
 
 \begin{eqnarray*}
|S_n|&\leq & |L_N| + \sum_{j=A_{N+1}}^n \frac{1}{\sqrt j}\\
&\leq& \sum_{j=1}^{\lfloor N/2\rfloor} |P_{2j-1}-P_{2j}|+\underbrace{P_N}_{\text{comes from the case $N$ odd}}
+ \sum_{j=A_{N+1}}^n \frac{1}{\sqrt j}\\
&\leq&\sum_{k=1}^\infty |P_k-P_{k+1}|+ P_N +P_{N+1}\\
&\buildrel\leq_{(\ref{pkabs})}^{(\ref{pk-abs})}&C \sum_{k=1}^\infty \frac{1}{k^2} + 4\sqrt 2=:\tilde C.
 \end{eqnarray*}
 
 {\bf Remark}  We also deduce that the associated parenthesized series converges:
 $$\left(\frac{1}{\sqrt 1}-\frac{1}{\sqrt 2} -\frac{1}{\sqrt 3}\right)+\left(\frac{1}{\sqrt4}+\frac{1}{\sqrt5}+\frac{1}{\sqrt 6}-\frac{1}{\sqrt 7}-\frac{1}{\sqrt8}-\frac{1}{\sqrt9}-\frac{1}{\sqrt{10}}\right) $$
 $$+\left(\frac{1}{\sqrt{11}}+\frac{1}{\sqrt{12}}+\frac{1}{\sqrt{13}}+\frac{1}{\sqrt{14}}+\frac{1}{\sqrt{15}}-\frac{1}{\sqrt{16}}-
 \frac{1}{\sqrt{17}}-\frac{1}{\sqrt{18}}-\frac{1}{\sqrt{19}}-\frac{1}{\sqrt{20}}-\frac{1}{\sqrt{21}}\right)+\cdots
 $$

\newpage

 \begin{figure}[h!]
 
  \scalebox{0.55} 
  {\includegraphics{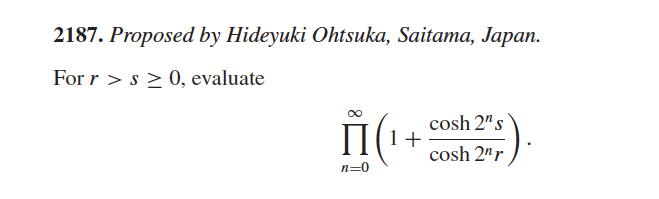}} 
\end{figure}

\centerline{\bf
Solution to problem 2187 Math. Mag. 97 (1) 2024, p. 81}  \medskip

     \centerline{Raymond Mortini,  Rudolf Rupp}
            
\medskip

\centerline{- - - - - - - - - - - - - - - - - - - - - - - - - - - - - - - - - - - - - - - - - - - - - - - - - - - - - -}
  
  \medskip


We claim that for $r>s\geq 0$,

$$\ovalbox{$\dis P:=\prod_{n=0}^\infty \left(1+\frac{\cosh(2^ns)}{\cosh(2^nr)}\right)= \frac{\sinh r}{\cosh r-\cosh s}$}.$$
\\

To this end we first show via induction that

\begin{equation}\label{hil1}
\prod_{n=0}^k \cosh(2^nx)=\frac{\sinh(2^{k+1}x)}{2^{k+1}\sinh x}.
\end{equation}

In fact, let $k=0$. Then $\frac{\sinh 2x}{2 \sinh x}=\cosh x$. If (\ref{hil1}) is correct for some $k$, then

\begin{eqnarray*}
\prod_{n=0}^{k+1} \cosh(2^nx)&=&\left(\prod_{n=0}^k \cosh(2^nx)\right)\cdot \cosh (2^{k+1}x)=
\frac{\sinh(2^{k+1}x)}{2^{k+1}\sinh x} \cdot  \cosh (2^{k+1}x)\\
&=& \frac{\sinh (2^{k+2}x)}{2^{k+2}\sinh x}.
\end{eqnarray*}

A way to come up with such a formula, is to use the well-known funny formula 
$$\dis \prod_{n=0}^{k} (1+w^{2^n})=\frac{1-w^{2^{k+1}}}{1-w}$$
 for $w=e^{-x}$
and by writing $ \cosh x=e^{x}(1+e^{-2x})/2$.\\

Now 
$$1+\frac{\cosh u}{\cosh v}=\frac{\cosh u+\cosh v}{\cosh v}=2\;\frac{\cosh(\frac{u+v}{2})\cosh(\frac{u-v}{2})}{\cosh v}.$$
Hence, with $u=2^ns$ and $v=2^nr$,
{\footnotesize \begin{eqnarray*}
P_k:= \prod_{n=0}^k\left(1+\frac{\cosh 2^ns}{\cosh 2^nr}\right)&=&2^{k+1}\,\prod_{n=0}^k \frac{\cosh(2^{n-1}(r+s))
 \cosh(2^{n-1}(r-s))}{\cosh 2^nr}\\
 &=&2^{k+1} \frac{ \cosh(\frac{r+s}{2})\, \cosh(\frac{r-s}{2})\;\prod_{j=0}^{k-1} \cosh 2^j(r+s) \prod_{j=0}^{k-1} \cosh 2^j(r-s)}{\prod_{n=0}^k \cosh 2^nr}\\
 &=&2^{k+1} \cosh\left(\frac{r+s}{2}\right)\, \cosh\left(\frac{r-s}{2}\right)\; \frac{\sinh 2^k(r+s)}{2^{k}\sinh (r+s)} \; 
  \frac{\sinh 2^k(r-s)}{2^{k}\sinh (r-s)} \; \frac{2^{k+1}\sinh r}{\sinh 2^{k+1}r}\\
  &=&4 \cosh\left(\frac{r+s}{2}\right)\, \cosh\left(\frac{r-s}{2}\right)\;\frac{\sinh r}{\sinh(r+s)\sinh(r-s)}\;
  \frac{\sinh 2^k(r+s)\sinh 2^k(r-s)}{\sinh 2^{k+1}r}
\end{eqnarray*}
}
Next we claim that for $r>s$,
$$\lim_{k\to\infty}\frac{\sinh 2^k(r+s)\sinh 2^k(r-s)}{\sinh 2^{k+1}r}=\frac{1}{2}.$$

In fact, using that
$\sinh x=\frac{e^x}{2}(1-e^{-2x})$
we obain
\begin{eqnarray*}
 \frac{\sinh 2^k(r+s)\sinh 2^k(r-s)}{\sinh 2^{k+1}r}&=&\frac{\frac{1}{4}e^{2^{k}(r+s)} e^{2^{k}(r-s)} (1-e^{-2^{k+1}(r+s)})(1-e^{-2^{k+1}(r-s)})}
 {\frac{e^{2^{k+1}r}}{2} (1-e^{2^{k+2}r})}\\
 &=&\frac{1}{2}\;\frac{(1-e^{-2^{k+1}(r+s)})(1-e^{-2^{k+1}(r-s)})}{ (1-e^{2^{k+2}r})}\to \frac{1}{2}.
\end{eqnarray*}

Now note that $\cosh^2 r-\sinh^2 s= \sinh(r+s)\sinh(r-s)$, and so 
\begin{eqnarray*}
4 \cosh\left(\frac{r+s}{2}\right)\, \cosh\left(\frac{r-s}{2}\right)\;\frac{\sinh r}{\sinh(r+s)\sinh(r-s)}&=& 
2 \frac{\cosh r+\cosh s}{\sinh(r+s)\sinh(r-s)}\sinh r\\
&=&\frac{2\sinh r}{\cosh r-\cosh s}.
\end{eqnarray*}
Thus
$$\lim_{k\to\infty} P_k=\frac{\sinh r}{\cosh r-\cosh s}.$$

\newpage

 \begin{figure}[h!]
 
  \scalebox{0.55} 
  {\includegraphics{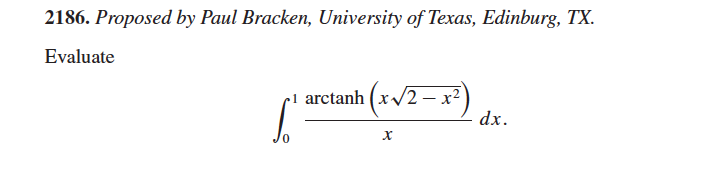}} 
\end{figure}

\centerline{\bf
Solution to problem 2186 Math. Mag. 97 (1) 2024, p. 81}  \medskip

     \centerline{Raymond Mortini,  Rudolf Rupp}

\medskip

\centerline{- - - - - - - - - - - - - - - - - - - - - - - - - - - - - - - - - - - - - - - - - - - - - - - - - - - - - -}
  
  \medskip

A function ``arctanh x" does not exist in the terminology we have learned  (see R. Burckel,  Classical Analysis in the plane, 2021,  p. 135). It is either $\arctan x$ or $\artanh x$. Not knowing whether the letter $c$  or the letter $h$ in your statement is superfluous, we consider both cases. So we will prove the following:\\

\begin{enumerate}
\item[(1)] $\dis \int_0^1 \frac{\artanh(x\sqrt{2-x^2}) }{x}\;dx=\frac{3}{16}\pi^2\sim 1.850550825204\cdots$,
\item[(2)] $\dis \int_0^1 \frac{\arctan(x\sqrt{2-x^2}) }{x}\;dx=\frac{1}{2} C +\frac{\pi}{4}\log(\sqrt 2+1)\sim 1.15021199360\cdots$.
\end{enumerate}
where $C$ is the Catalan constant.
\\

We need the following well-known integral:

\begin{lemma}\label{sinup}
Let $I_n:=\int_0^{\pi/2} (\sin x)^{n} dx$. Then  $I_0=\pi/2$ and $I_1=1$. For $n\in \N^*:=\{1,2,3\cdots\}$ we have  

$$ I_{2n}=\frac{1}{2}\cdot\frac{3}{4}\cdot\frac{5}{6}\cdots\frac{2n-1}{2n}\frac{\pi}{2}=\frac{(2n)!}{4^n (n!)^2}\cdot\frac{\pi}{2}=
\frac{{2n\choose n} }{4^n}\cdot \frac{\pi}{2}.$$

\end{lemma}
\begin{proof}
$I_{2n}= \frac{2n-1}{2n} I_{2n-2}$ for $n\in \N^*$ and $I_0=\frac{\pi}{2}$, because

\begin{eqnarray*}
2n I_{2n}- (2n-1)I_{2n-2}&=&\int_0^{\pi/2}(\sin x)^{2n-2}\left(2n \sin^2 x-(2n-1)\right)dx\\
&=&-\int_0^{\pi/2}(\sin x)^{2n-2}\left( (2n-1) \cos ^2x-\sin^2 x\right)dx\\
&=&-\left[ (\sin x)^{2n-1}\cos x   \right]^{\pi/2}_0=0.
\end{eqnarray*}

\end{proof}

Moreover,  we will use that $\dis\sum_{n=0}^\infty \frac{1}{(2n+1)^2}=\frac{\pi^2}{8}$ as well as $\dis\sum_{n=0}^\infty \frac{(-1)^n}{(2n+1)^2}=C$.
 Finally we need that for $|x|\leq 1$
$$\arcsin x=\sum_{n=0}^\infty   \frac{{2n\choose n}}{4^n}\frac{x^{2n+1}}{2n+1}\quad {\rm and}\quad
{\rm arsinh}\, x=\sum_{n=0}^\infty  (-1)^n \frac{{2n\choose n}}{4^n}\frac{x^{2n+1}}{2n+1}.
$$
Note that in view of Stirling's formula $\frac{{2n\choose n}}{4^n}\sim \frac{1}{\sqrt \pi\; \sqrt n}$, so the series converge absolutely for $x=\pm1$.

(1)  We make the substitution $x=\sqrt 2 \sin t$, $dx=\sqrt 2 \cos t\,dt$. Then, by using  $\int\sum=\sum\int$ (all terms are positive),
\begin{eqnarray*}
\dis \int_0^1 \frac{\artanh(x\sqrt{2-x^2}) }{x}\;dx&=&\int_0^{\pi/4} \frac{\artanh (\sin (2t))}{\sin t}\; \cos t\,dt\\
&=&\int_0^{\pi/4}\sum_{n=0}^\infty \frac{1}{2n+1}\frac{(\sin(2t))^{2n+1}}{\sin t}\cos t\;dt\\
&\buildrel=_{}^{\sin(2t)=2\sin t\cos t}&\int_0^{\pi/4}\sum_{n=0}^\infty \frac{1}{2n+1}(\sin(2t))^{2n}\;\underbrace{2 \cos^2 t}_{=1+\cos(2t)}\;dt\\
&=&\sum_{n=0}^\infty \frac{1}{2n+1}\left(\int_0^{\pi/4} (\sin (2t))^{2n} dt + \int_0^{\pi/4} (\sin(2t))^{2n}\cos(2t)dt\right)\\
&\buildrel=_{}^{2t=u}&\frac{1}{2}\sum_{n=0}^\infty\frac{1}{2n+1}\left(\int_0^{\pi/2} (\sin u)^{2n} du + \int_0^{\pi/2} (\sin u)^{2n}\cos u du\right)\\
&=&\frac{1}{2}\sum_{n=0}^\infty  \frac{1}{2n+1}\left(\frac{{2n\choose n} }{4^n}\cdot \frac{\pi}{2} + \frac{1}{2n+1}\right)\\
&=& \frac{\pi}{4} \arcsin 1 + \frac{\pi^2}{16} =\frac{3}{16}\pi^2.
\end{eqnarray*}

(2)  In this case case, the factor $\frac{1}{2n+1}$ is replaced by $\frac{(-1)^n}{2n+1}$.
Moreover, $\int \sum=\sum\int$, since
$$\left|\sum_{n=0}^N \frac{(-1)^n}{2n+1}(\sin u)^{2n}(1+\cos u)\right|\leq \sum_{n=0}^\infty \frac{1}{2n+1}(\sin u)^{2n}(1+\cos u),$$
which is an integrable majorant by (1).
Hence
$$ \int_0^1 \frac{\arctan(x\sqrt{2-x^2}) }{x}\;dx=\frac{\pi}{4} {\rm arsinh}\; 1 + \frac{1}{2} C.$$
Since ${\rm arsinh}\, x= \log(x+\sqrt{1+x^2})$, we are done.

\newpage

  \begin{figure}[h!]
 
  \scalebox{0.6} 
  {\includegraphics{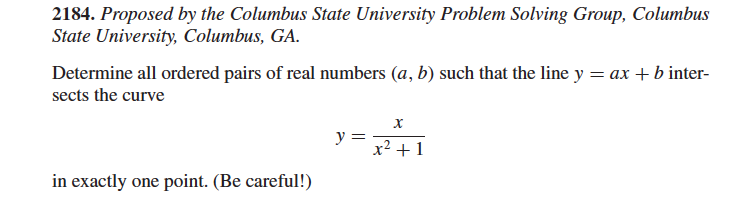}} 
\end{figure}

\centerline{\bf
Solution to problem 2184 Math. Mag. 96 (5) 2023, p. 567}  \medskip

     \centerline{Raymond Mortini, Peter Pflug  and Rudolf Rupp}
     \medskip

\centerline{- - - - - - - - - - - - - - - - - - - - - - - - - - - - - - - - - - - - - - - - - - - - - - - - - - - - - -}
  
  \medskip


We give two proofs (one geometric/intuitive informal one and one analytic one) of the  following result:\\

{\bf Proposition} {\sl
Let $f(x)=\frac{x}{1+x^2}$. Then the set of all those $(a,b)\in \R^2$ for which the line $y=ax+b$ cuts the
 graph $G:=\{(x,f(x)): x\in \R\}$ of $f$
in exactly one point is given by the "exterior"   $E$ \footnote{ This is the unbounded component of the complement of the curve.}   of the closed Jordan curve  (displayed in red below)

\begin{equation}\label{tangentenkurve}
\Gamma(t)= {a(t)\choose b(t)}= {\frac{1-t^2}{(1+t^2)^2}\choose \frac{2t^3}{(1+t^2)^2}},\;{\rm for}\;   t\in \R, \quad {\rm and}\quad
\Gamma(\pm \infty)={0\choose 0},
\end{equation} 

together with $ \{(0,0)\}\union\{(1,0)\}\union\{(0,\pm \frac{1}{2})\}\union \{(-\frac{1}{8}, \pm\frac{3\sqrt 3}{8}) \}$ and 
 deleted by the half lines 
 $$\mbox{$\{0\}\times \,]1/2,\infty[$ and $\{0\}\times ]-\infty,-1/2[$}.$$
 }
 
 \vspace{1cm}
   \begin{figure}[h!]
        {\scalebox{0.40} {\includegraphics{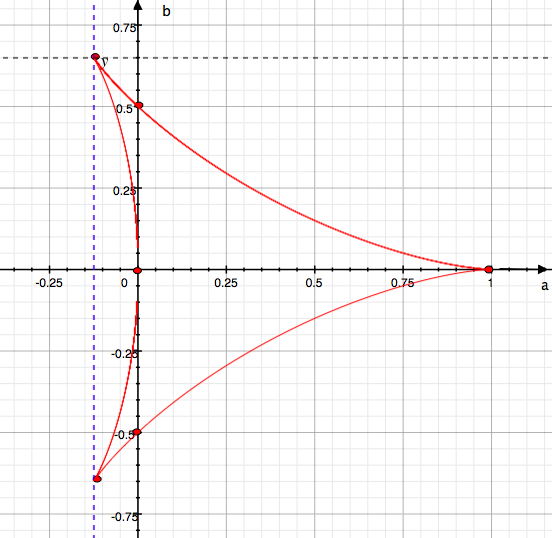}} }
\caption{\label{red} {The red curve}}
\end{figure}

 \begin{figure}[h!]
 
  \scalebox{0.45} 
  {\includegraphics{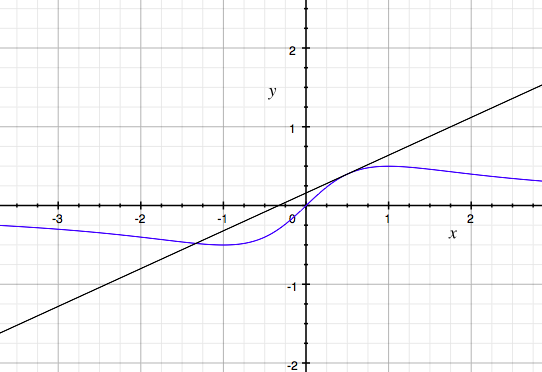}} 
 \caption{\label{bsp} {Graph of $f$ and one tangent}}

\end{figure}

\begin{proof}   We first discuss the geometry of the graph of $f$. 

 $\bullet$ Note that  $\dis f'(x)=\frac{1-x^2}{(1+x^2)^2}$. Hence
the red curve $\Gamma$, excepted the point $(0,0)$, is the set of $(a,b)=(a(x),b(x))$ such 
that $s\mapsto as+b$ is a tangent to the graph of $f$ at the point $(x,f(x))$ (since $a(x) x +b(x)= f(x)$ and $a(x)=f'(x)$).

Next,  $\dis f''(x)=\frac{2x(x^2-3)}{(1+x^2)^3}$.
Since $f''(0)=f''(\pm \sqrt 3)=0$, 
\begin{equation}\label{mwe}
\mbox{$\max f'=f(0)=1$ and $\min f'=f'(\pm\sqrt 3)=-\frac{1}{8}.$}
\end{equation}
Moreover $f(\pm \sqrt 3) =\pm\frac{\sqrt 3}{4}\sim \pm 0.433013\dots$ and $0$ and $\pm\sqrt 3$ are inflection points for $f$ 
and $\max f=f(1)=\frac{1}{2}$, respectively $\min f=f(-1)=-\frac{1}{2}$. If  in (\ref{tangentenkurve}) $t=\pm\sqrt 3$, then 
 $a(t)=-\frac{1}{8}$ and  $b(t)=\pm \frac{3\sqrt 3}{8} \sim \pm 0.64951\dots$.\\

   \begin{figure}[h!]
 
  \scalebox{0.35} 
  {\includegraphics{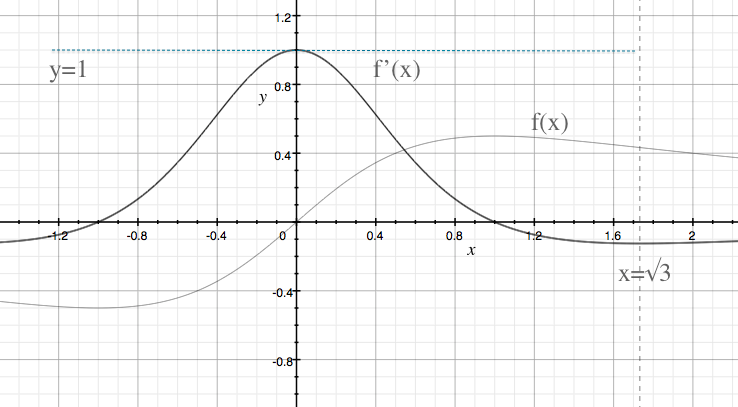}} 
  \caption{\label{graph} {$f$ and $f'$}}

\end{figure}

$\bullet$ Observe that if $s\mapsto as+b$ cuts the graph of $f$ in at least two different points, then $a\in [-1/8,1]$. In fact, by the mean value theorem, if $x_j$ are two intersection points, then
$$a=\frac{(ax_1+b)-(ax_2+b)}{x_1-x_2}=\frac{f(x_1)-f(x_2)}{x_1-x_2}=f'(\eta).$$
Now (\ref{mwe}) yields the assertion.
Consequently, if $a\notin [-1/8,1]$, then the line $s\mapsto as+b$ is either disjoint from the  graph of $f$ or cuts it in  a single point.

$\bullet$  The only lines $s\mapsto as+b$ which do not intersect the graph of $f$ are those that are parallel to the real axis (that is $a=0$) and for which $|b|>1/2$ (obviously clear by having  a glimpse at the figure \ref{graph} of the graph $G$ of $f$).   In fact, 
 any "oblique" line  $L$  (and any vertical line) has points in both domains determined by $G$ and so the connectedness of the line implies that $L\inter G\not=\emp$. \\
 Moreover, if $a=0$, then $b=0\cdot x+b=  \frac{x}{1+x^2}$ is equivalent to $bx^2-x+b=0$. So no solution exists if and only if the discriminant $1-4b^2$ is negative; that is if $|b|>1/2$.

$\bullet$ We will see below  that the only tangents meeting the graph of $f$ at a single point are the lines $y=\pm 1/2$ and $y=x$ and
$y= \frac{-1}{8} x \pm \frac{3\sqrt 3}{8}$ (those tangents associated with the extrema and the inflection points of $f$ ).
 All other tangents have another point of intersection: this is seen "geometrically" by looking at the graph  and by considering the three cases (and of course the associated opposites) : $0\leq x_0\leq 1$, $1< x_0\leq \sqrt 3$ and 
$x_0>\sqrt 3$, and by noticing that on the interior $I_j$ of these   three intervals  the tangents are on one side of the 
graph $\{(x,f(x): x\in I_j\}$, as we have no change of curvature ($f$ is either convex or concave on $I_j$.) 
See figure \ref{jordan1}. 

$\bullet$ The behavior  of the lines of the form $s\mapsto as+b$ with $a\in [-1/8,1]$ and $b>b(x)$ or $b<b(x)$ can be intuitively guessed by looking at the graph of $f$ (for a precise analytic proof, see next section).  
\end{proof}

  \begin{figure}[h!]
 
  \scalebox{0.35} 
  {\includegraphics{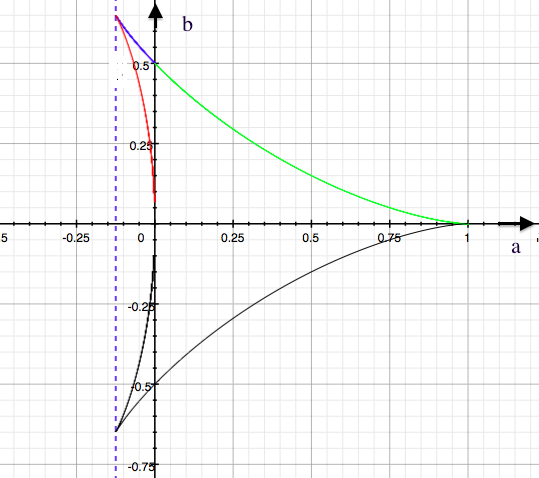}} 
  \caption{\label{jordan1} {The $a(x)$ and $b(x)$'s  for $x\in I_j$}}

\end{figure}

\subsection{An analytic proof}

The intersection condition is equivalent to  solving, for $a\not=0$,  the cubic polynomial equation
$$ax+b=\frac{x}{1+x^2}\iff  ax^3+bx^2+(a-1)x+b=0$$
and for $a=0$ the quadratic equation
$$b=\frac{x}{1+x^2}\iff bx^2-x+b=0.$$

 Put  
$$p(z):=p_{a,b}(z):=az^3+bz^2+(a-1)z+b.$$

Then several cases occur when discussing the equation $p_{a,b}(z)=0$, $a\not=0$:\\

i) one real solution and two complex ones (which are conjugated),

ii) three distinct real solutions,

iii) one double real solution and a second real solution, 

iv) a triple real solution.\\

A way to deal with this, is to use the discriminant.  For \ovalbox{$a\not=0$}, let
$$D:=a^4(z_1-z_2)^2(z_2-z_3)^2(z_1-z_3)^2$$
be the discriminant of this cubic equation. Here $z_1,z_2,z_3$ are the zeros. Then, 
$$D=-4a^4-8a^2b^2-4b^4+12a^3-20ab^2-12a^2+b^2+4a.$$
A lenghthier calculation (a posteriori verified by Maple and wolframalpha) gives
$$D=D(a,b)=-4\left[\left(b^2+a^2+\frac{5}{2}a-\frac{1}{8}\right)^2-8\left(a+\frac{1}{8}\right)^3\right].$$

It is well-known  that the cubic equation has a multiple zero if and only if the discriminant is zero. In other words, 
if and only if  $$\left(b^2+a^2+\frac{5}{2}a-\frac{1}{8}\right)^2=8\left(a+\frac{1}{8}\right)^3.$$
 Also,  $D>0$ if and only if the cubic equation (with real coefficients) has three distinct real zeros, and $D<0$ if and only if there 
is a unique real zero. In our situation here, 
 $D<0$ if and only if 
 $$-4\left[\left(b^2+a^2+\frac{5}{2}a-\frac{1}{8}\right)^2 -8\left(a+\frac{1}{8}\right)^3\right]<0,$$
 equivalently 
 $$\left(b^2+a^2+\frac{5}{2}a-\frac{1}{8}\right)^2>8\left(a+\frac{1}{8}\right)^3.$$

  Now we have  the following result:

\begin{lemma} \label{multi}
Let $(a,b)\in \R^2$. The following assertions  are equivalent:
\begin{enumerate}
\item[(1)]  $D(a,b)=0$ if $a\not=0$ or $(a,b)=(0,\pm 1/2)$ if $a=0$.
\item[(2)]  $p_{a,b}(z)=az^3+bz^2+(a-1)z+b$ has a multiple zero.

\item[(3)] The line $L:s\mapsto as+b$ is tangent to the graph $G$ of $f$ at the point $(x,f(x))$  for some \footnote{
Later we shall see that $x$ is uniquely determined; so a line $L$ can be tangent to  $G$ at at most one point.}
 $x\in \R$,  and
$$\mbox{$\dis a=a(x)= \frac{1-x^2}{(1+x^2)^2}$ and $\dis b=b(x)=\frac{2x^3}{(1+x^2)^2}$}.$$

\end{enumerate}

\end{lemma}
\begin{proof}
(1) $\iff $ (2): Discussed above for the case $a\not=0$. The case $a=0$ follows since the discriminant of the 
quadratic $bz^2-z+b$ is $1-4b^2$.

(2) $\imp$ (3): Suppose that $x\in \R$ is a multiple zero of $p$. Recall that $p'(z)= 3az^2+2bz+(a-1)$.
Then  $p(x)=p'(x)=0$ imply 
that $ax+b=\frac{x}{1+x^2}$ and 
 $$
 3a x^2+2x\left(\frac{x}{1+x^2}-ax\right)+(a-1)=0.
 $$
Thus
 $$a= \frac{1-x^2}{(1+x^2)^2}.$$
 (In case $a=0$, $x=\pm 1$). 
 Consequently, $s\mapsto as+b$  is a tangent to the graph of $f$ at $x$ (since  $ax+b=f(x)$ and $a=a(x)=f'(x)$). Moreover,
 $$b=\frac{x}{1+x^2}- \frac{1-x^2}{(1+x^2)^2}\,x= \frac{2x^3}{(1+x^2)^2}.$$
 (In case $a=0$, $b=\pm 1/2$). 
 
 (3) $\imp$ (1): Suppose that $s\mapsto as+b$ is a tangent at $(x,f(x))$ and that $a$ and $b$ have the form given in 
  the assumption (3).  If 
$a=a(x)\notin \{0,1\}$, then $x$ is (at least !)  a double zero of $p_{a,b}$, since $p_{a,b}(x)=0$ 
(equivalently $ax+b=f(x)$), and $p_{a,b}'(x)=0$ because
  $$3 \frac{1-x^2}{(1+x^2)^2}x^2 +2\frac{2x^3}{(1+x^2)^2}x +  \frac{1-x^2}{(1+x^2)^2}-1\equiv 0.$$
Moreover, if $a=a(x)=1$, then $x$  is a triple zero of $p_{a,b}$  and $b=b(x)=0$. Hence, as (2) $\imp$ (1), $D(a(x),b(x))=0$.
If $a=a(x)=0$, then $x=\pm 1$ and $b=\pm 1/2$ .  Thus (1) holds. 
\end{proof}

 Conclusion:
 The set $(a,b)\in\R^2$ of points where $p_{a,b}$ has a multiple zero  is in a one to one  correspondance with those lines $s\mapsto as+b$ which are tangent to the graph of $f$. It coincides with 
 $$\{(a,b)\in \R^2\setminus \big(\{0\}\times\R\big): D(a,b)=0\}\union \Big\{(0,-\frac{1}{2}),(0,\frac{1}{2})\Big\},$$
 and is the {\it Jordan arc} parametrized by 
 $$\Gamma(t)={a(t)\choose b(t)}={\frac{1-t^2}{(1+t^2)^2}\choose \frac{2t^3}{(1+t^2)^2}},\quad t\in \R.$$
 
 To see that $\Gamma$ is injective,  suppose that there exists $(a,b)\in \R^2$ such that  $a=a(t)=a(t')$ and $b=b(t)=b(t')$ 
 for $t\not=t'$. By Lemma \ref{multi} (and its proof), the line $s\mapsto as+b$ is tangent to the graph of $f$ at the points
 $(t,f(t))$ and $(t' ,f(t'))$,  and so $t$ and $t'$ are (at least ) double zeros of $p_{a,b}$. This would imply that the degree of $p_{a,b}$ is bigger than  $4$. A contradiction.\\

 The two components determined by the closure $J:=\Gamma(\R)\union \{(0,0)\}$ of this Jordan arc (which is a closed Jordan curve)  coincide with  \footnote{ Take  e.g. two points in the exterior complemented component of $J$, denoted by  $\Omega$.  Join those with an arc inside $\Omega$. Then 
 $\tilde D$ must have the same sign at both points; otherwise this arc would meet the set where  $\tilde D$ is zero. As this set coincides
 with the boundary of $\Omega$, that is the  Jordan curve $J$, we get  a contradiction.}
 $$\mbox{$\tilde D(a,b)^{-1}(\;]0,\infty[\,)$ and $\tilde D(a,b)^{-1}(\,]-\infty,0[\,)$},$$
 respectively, where \footnote{ In order to have continuity of $\tilde D$, we need to add the factor $b^2$.}
 $$\tilde D(a,b)=\begin{cases} D(a,b)&\text{if $a\not=0$}\\
 b^2(1-4b^2)&\text{if $a=0$}.
 \end{cases}
 $$
 
 The following observations now will show that the exterior of this Jordan domain is the set where $\tilde D(a,b)<0$. Always have in mind figure \ref{jordan1}. But attention: this is not yet the final set the problem is asking for.

Consider a tangent at $x$ with  $0<a(x)\leq 1$.
  Since $a(x)=f'(x)$ we deduce that $|x|<1<\sqrt 3$. This implies that
$$b(x)^2+a(x)^2+\frac{5}{2}a(x)-\frac{1}{8}= \frac{1}{8}\,\frac{(3-x^2)^3}{(x^2+1)^3}>0.$$
Hence, if $b>b(x)$, we get
$$ 8\left(a(x)+\frac{1}{8}\right)^3=\left(b(x)^2+a(x)^2+\frac{5}{2}a(x)-\frac{1}{8}\right)^2<
\left(b^2+a(x)^2+\frac{5}{2}a(x)-\frac{1}{8}\right)^2,$$
and so $D(a(x),b)<0$. This implies that there is a unique real zero of $p$ and so the line $s\mapsto a(x)s+b$ cuts the graph of $f$ at a single point.

Next, if $b=0$ and if $a\to -\infty$, then $D(a,b)\to -\infty$. So again $\tilde D<0$ in that part of the exterior of $J$ that is contained in the left-hand plane.

Finally, if $a=0$, the discriminant $1-4b^2$ of $bz^2-z+b$ is negative if and only if $p_{0,b}$ has no real zeros; so no intersection points of $ax+b$ exist whenever $a=0$ and $|b| >1/2$, but two if $0<|b|<1/2$ and one if $b=0$.  Consequently, the exterior of the Jordan curve is the set where $\tilde D<0$. \\

To  achieve the solution to the problem,  a last case has to be investigated: for which $(a,b)$ the polynomial $p_{a,b}$ has  triple zero (as this yields tangents which cut the graph $G$ of $f$ at a single point). 

So let $p_{a,b}(x)=p_{a,b}'(x)=p_{a,b}''(x)=0$.  By Lemma \ref{multi}, $s\mapsto as+b$ is tangent to the graph  $G$ of $f$. Hence 
$a= a(x)= \frac{1-x^2}{(1+x^2)^2}$ and $b=b(x)=\frac{2x}{(1+x^2)^2}$.  

Now $p''_{a,b}(z)=  6a z+2b$. Hence 
$$x= -\frac{b}{3a}=\frac{2x}{(1+x^2)^2} \bigg/  3\frac{1-x^2}{(1+x^2)^2}=- \frac{2}{3}\; \frac{x^3}{1-x^2}.
$$
Consequently, either $x=0$ or $x=\pm \sqrt 3$. This yields the values $(a,b)=(1,0)$ and $(a,b)=(-\frac{1}{8},  \pm\frac{3\sqrt 3}{ 8})$. \\

$\bullet$  We are now able to answer the question, which lines $s\mapsto as+b$ intersect the graph $G$
 of $f(x)=x/(1+x^2)$ in a single point:

i) All points $(a,b)\in \R^2$ for which  $a\not=0$ and $D(a,b)<0$.

ii) The 6 points $(a,b)\in  \{(0,0),\;(1,0),\;(0,\pm \frac{1}{2}),\; (-\frac{1}{8}, \pm\frac{3\sqrt 3}{8})\}$,  
which induce via the map $s\mapsto as+b$ tangents to $G$ whenever $(a,b)\not=(0,0)$.

\newpage

 \begin{figure}[h!]
 
  \scalebox{0.5} 
  {\includegraphics{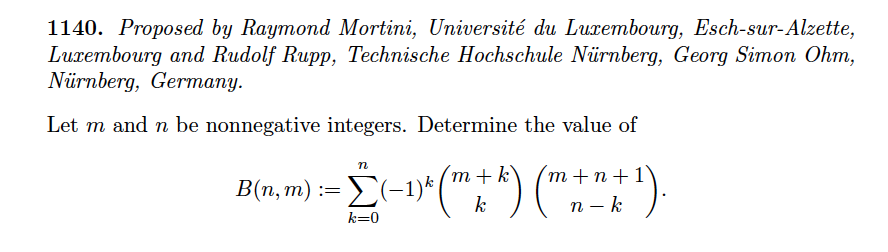}} 
\end{figure}

\centerline{\bf
Quicky  1140 Math. Mag. 97 (2024) by}
 \centerline{Raymond Mortini and Rudolf Rupp}
      \medskip

\centerline{- - - - - - - - - - - - - - - - - - - - - - - - - - - - - - - - - - - - - - - - - - - - - - - - - - - - - -}
  
  \medskip

\pagecolor{yellow}

     {\bf Submitted statement:}\\
     
(a) Let $m,n\in \N=\{0,1,2,\dots\}$. Determine the value of 
$$\fbox{\parbox{8cm}  {$\dis B(n,m):= \sum_{k=0}^n (-1)^k {m+k\choose k}\;{m+n+1\choose n-k}.$}}
$$

(b) Let $z,w\in \C\setminus \{-1,-2,-3\dots\}$. Suppose that  $z-w\notin \{-1-2,-3\dots\}$. Using that  for these parameters   
$ {z\choose w} :=\frac{\Gamma(z+1)}{\Gamma(w+1)\Gamma(z-w+1)}$ is well defined,  show that  for $a,b\in \C$ with 
${\rm Re}\; a>0$, ${\rm Re}\; b>1$, and  $b\notin \Z$,
the series 
$$\fbox{\parbox{8cm}  {$\dis S(a,b):= \sum_{k=0}^\infty (-1)^k {a+k-1\choose k}\;{a+b-1\choose b-k-1}.$}}
$$
converges absolutely and that $S(a,b)=1$.
\bigskip\bigskip

{\bf Solution}  (a) Note that
\begin{eqnarray*}
{m+k\choose k}\;{m+n+1\choose n-k}&=&\frac{(m+k)!}{m!k!}\; \frac{(m+n+1)!}{(m+k+1)! (n-k)!}\\
&=& \frac{(m+n+1)!}{m!} \frac{1}{k! (m+k+1)(n-k)!}= \frac{(m+n+1)!}{m!n!}{n\choose k}\frac{1}{m+k+1} .
\end{eqnarray*}
Hence
$$\sum_{k=0}^n(-1)^k\;{m+k\choose k}\;{m+n+1\choose n-k}=
 \frac{(m+n+1)!}{m!n!}\; \sum_{k=0}^n(-1)^k\; {n\choose k}\frac{1}{m+k+1}.$$
Put
$$f(x):= \sum_{k=0}^n (-1)^k\;{n\choose k}\frac{1}{m+k+1}\; x^{m+k+1}.$$
Then
$$f'(x)= \sum_{k=0}^n  (-1)^k\;{n\choose k} x^{m+k}=x^m (1-x)^n.$$

Consequently, as $\int_0^1 f'(x)dx=f(1)-f(0)$ and $f(0)=0$,

$$B(n,m)= \frac{(m+n+1)!}{m!n!} \int_0^1 x^m(1-x)^n\; dx.$$

The value of the integral is given by Euler's $\beta$ function 

$$\beta(m+1,n+1)=\frac{ \Gamma(m+1)\;\Gamma(n+1)}{\Gamma(m+n+2)}= \frac{m!\,n!}{(m+n+1)!}.$$

Hence
$$B(n,m)=  \frac{(m+n+1)!}{m!n!}\;  \frac{m!\,n!}{(m+n+1)!}=1.$$

\bigskip

(b) First we note that under the conditions on $a$ and $b$, the complex binomial coefficients are well defined.  Since 
$\Gamma(z+1)=z\Gamma(z)$  for $z\in \C\setminus (-\N)$, and since the $\Gamma$-function has no zeros, we have

\begin{eqnarray*}
 {a+k-1\choose k}\;{a+b-1\choose b-k-1}&=& \frac{\co{\Gamma(a+k)}}{\Gamma(a)\,\Gamma(k+1)}\; 
 \frac{\Gamma(a+b)}{\co{\Gamma(a+k+1)}\Gamma(b-k)}\\
 &=& \frac{\Gamma(a+b)}{\Gamma(a)\co{\Gamma(b)}}\; \frac{\co{\Gamma(b)}}{(a+k) \Gamma(k+1)\Gamma(b-k)}\\
 &=& \frac{\Gamma(a+b)}{\Gamma(a)\Gamma(b)}\; {b-1\choose k}\; \frac{1}{a+k}.
\end{eqnarray*}

Hence

$$S(a,b)=\sum_{k=0}^\infty (-1)^k{a+k-1\choose k}\;{a+b-1\choose b-k-1}= \frac{\Gamma(a+b)}{\Gamma(a)\Gamma(b)}\;
\sum_{k=0}^\infty (-1)^k{b-1\choose k}\; \frac{1}{a+k}.$$

It is known that the  binomial series $ \dis \sum_{k=0}^\infty {b-1\choose k}$ converges absolutely for 
${\rm Re}\; b>1$ (see \cite[p. 140]{kno}).
Hence $S(a,b)$ converges. Now consider for  $c\in \C$ and  $0< x\leq 1$ the functions $x^c:=\exp( c \log x)$, and
$$f(x):=\begin{cases} \dis x^a\;\sum_{k=0}^\infty (-1)^k{b-1\choose k}\; \frac{1}{a+k} x^{k} &\text{if $0<x\leq 1$}\\
0&\text{if $x=0$},\end{cases}
$$
 which is continuous  \footnote{ Note that ${\rm Re}\;a>0$ and so $0^a:=0$ is the correct value if one wants continuity:
 $|x^a|\leq \exp({{\rm Re}\, a }\;\log x)\to \exp(-\infty)=0$ as $x\to 0^+$. Also, as usual in the realm of power series, $0^0:=1$.} on $[0,1]$.
 Using  for $0<x<1$ the Newton-Abel formula for the binomial series with complex powers (see \cite[p. 158]{kno}), we obtain
$$f'(x)=\sum_{k=0}^\infty (-1)^k{b-1\choose k}\;x^{a+k-1}= x^{a-1} \sum_{k=0}^\infty {b-1\choose k} (-x)^k=x^{a-1}(1-x)^{b-1}.$$

Consequently, as $\int_0^1 f'(x)dx=f(1)-f(0)$ and $f(0)=0$,

$$S(a,b)= \frac{\Gamma(a+b)}{\Gamma(a)\Gamma(b)}\; \int_0^1 x^{a-1}(1-x)^{b-1}\;dx.$$

This integral is the $\beta$-function. Note that ${\rm Re}\,a>0$ and ${\rm Re}\,b>0$. Hence this integral is well defined and 
$\beta(a,b)=\frac{\Gamma(a)\Gamma(b)}{\Gamma(a+b)}$ (see \cite[p. 67\,ff]{rr}. Consequently, $S(a,b)=1$.

For this proposal, we were motivated by Problem 4862 Crux Math. 49 (7) 2023, 375. We hope that this sum has not been considered earlier.

\newpage

  \begin{figure}[h!]
 
  \scalebox{0.6} 
  {\includegraphics{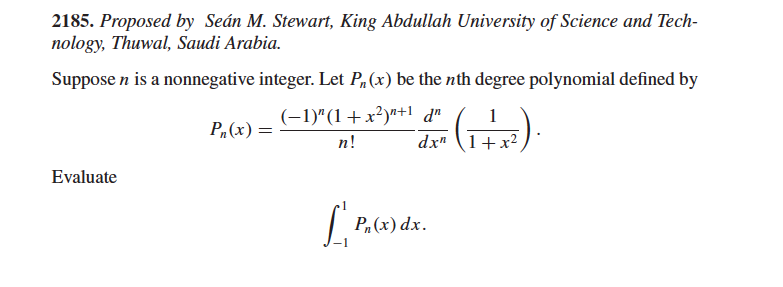}} 
\end{figure}

\centerline{\bf
Solution to problem 2185 Math. Mag. 96 (5) 2023, p. 567}  \medskip

     \centerline{Raymond Mortini  and Rudolf Rupp}

     \medskip

\centerline{- - - - - - - - - - - - - - - - - - - - - - - - - - - - - - - - - - - - - - - - - - - - - - - - - - - - - -}
  
  \medskip

\nopagecolor

We show that the value of the integral $I_n:=\int_{-1}^1P_n(x)dx$ is
$$\ovalbox{$\dis \frac{(-1)^n\, i^n}{n+2}\; \big(1+(-1)^n\big)\; 2^{\frac{n+2}{2}}\cos\left(n\frac{\pi}{4}\right).
$}$$
 Another representation is
$$I_n=\e_n\, \frac{2^\frac{n+4}{2}}{n+2},$$
where
$$\e_n=\begin{cases} 
1&\text{if $n\equiv 0 \mod 8$}\\
0&\text{if $n\equiv 1 \mod 8$}\\
0&\text{if $n\equiv 2 \mod 8$}\\
0&\text{if $n\equiv 3 \mod 8$}\\
-1&\text{if $n\equiv 4 \mod 8$}\\
0&\text{if $n\equiv 5 \mod 8$}\\
0&\text{if $n\equiv 6 \mod 8$}\\
0&\text{if $n\equiv 7 \mod 8$.}
\end{cases}
$$
A very strange result!  In fact,
$$\mbox{$\dis \frac{d^n}{dx^n}\left(\frac{1}{1+ix}\right)= \frac{i^n(-1)^n n!}{(1+ix)^{n+1}}$  \quad and \quad 
 $\dis \frac{d^n}{dx^n}\left(\frac{1}{1-ix}\right)= \frac{i^n n!}{(1-ix)^{n+1}}$}.
 $$
 
 Hence
 \begin{eqnarray*}
\frac{d^n}{dx^n}\left(\frac{1}{1+x^2}\right)&=& \frac{1}{2}\frac{d^n}{dx^n}\left(\frac{1}{1+ix}+\frac{1}{1-ix}\right)\\
&=& \frac{i^n}{2}n! \; \frac{(1+ix)^{n+1}+(-1)^n(1-ix)^{n+1}}{(1+x^2)^{n+1}}.
\end{eqnarray*}
 From this we get that $P_n$ is a polynomial of degree $n$ with $n+1$ as leading coefficient. We are now ready to calculate the integral:

 \begin{eqnarray*}
I_n= \int_{-1}^1 P_n(x)dx&=&\frac{(-1)^n}{2}i^n \int_{-1}^1 \left((1+ix)^{n+1}+(-1)^n(1-ix)^{n+1}\right)\; dx\\
 &=&\frac{(-1)^n}{2}i^n \left( \frac{(1+i)^{n+2}-(1-i)^{n+2}}{i(n+2)} +(-1)^n \frac{(1-i)^{n+2}-(1+i)^{n+2}}{(-i)(n+2)}\right)\\
 &=& \frac{(-1)^n i^{n-1}}{2(n+2)}\;\left(1+(-1)^n\right)  \;\left( (1+i)^{n+2}-(1-i)^{n+2}\right).
\end{eqnarray*}
 
 Since 
\begin{eqnarray*}
(1+i)^{n+2}-(1-i)^{n+2}&=& \sqrt 2^{\;n+1} \left( \left(\frac{1+i}{\sqrt 2}\right)^{n+2} -\left(\frac{1-i}{\sqrt 2}\right)^{n+2}\right)\\
&=&2^{\frac{n+1}{2}} \left( e^{(n+2)i\pi/4}-e^{-(n+2)i\pi/4}\right)\\
&=& 2 i \;2^{\frac{n+1}{2}} \sin\left( (n+2)\frac{\pi}{4}\right),
\end{eqnarray*}
 we conclude that
\begin{eqnarray*}
I_n&=&\frac{(-1)^n\, i^n}{n+2}\; \big(1+(-1)^n\big)\; 2^{\frac{n+2}{2}}\sin\left( (n+2)\frac{\pi}{4}\right)\\
&=&\frac{(-1)^n\, i^n}{n+2}\; \big(1+(-1)^n\big)\; 2^{\frac{n+2}{2}}\cos\left(n\frac{\pi}{4}\right).
 \end{eqnarray*}

\newpage

  \begin{figure}[h!]
 
  \scalebox{0.5} 
  {\includegraphics{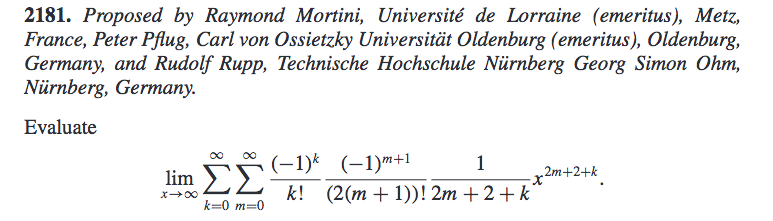}} 
\end{figure}

\centerline{\bf
Solution to problem 2181 Math. Mag. 96 (5) 2023, p. 566}  \medskip

     \centerline{Raymond Mortini, Peter Pflug  and Rudolf Rupp}
     
          \medskip

\centerline{- - - - - - - - - - - - - - - - - - - - - - - - - - - - - - - - - - - - - - - - - - - - - - - - - - - - - -}
  
  \medskip

\pagecolor{yellow}

a) The double series 
 $$S(x):=\sum_{k=0}^\infty \sum_{m=0}^\infty   \frac{1}{k!} \frac{1}{[2(m+1)]!}
 \frac{1}{2m+2+k} x^{2m+2+k}$$
 converges  since for every $j\in \N$ the partial sums can be estimated as follows:
 \begin{eqnarray*}
\sum_{n=0}^j \sum_{m=0}^j \frac{1}{k!} \frac{1}{[2(m+1)]!}\frac{1}{2m+2+k} x^{2m+2+k}&\leq&
 \sum_{n=0}^j  \sum_{m=0}^j  \frac{1}{k!} \frac{1}{[2(m+1)]!} x^{k} x^{2m+2}\\
 &=&
 \left(\sum_{n=0}^j  \frac{1}{k!}\; x^{k}\right)\; \left( \sum_{m=0}^j \frac{1}{[2(m+1)]!}\; x^{2m+2}\right)
\end{eqnarray*} 
Hence the series $P$  converges absolutely (and so does any re-arrangement) locally uniformly to some finite value $P(x)$. 

b) By the same reason the formal derivated series 
$$H(x):= \sum_{k=0}^\infty \sum_{m=0}^\infty   \frac{(-1)^k}{k!} \frac{(-1)^{m+1}}{[2(m+1)]!} x^{2m+1+k}
 $$
 converges  absolutely and locally uniformly for $x\geq 0$. Hence $P'=H$. Thus
 $$H(x)=\left(\sum_{k=0}^\infty  \frac{(-1)^k}{k!}\; x^{k}\right)\; \left( \sum_{m=0}^\infty\frac{(-1)^{m+1}}{[2(m+1)]!}\; x^{2m+1}\right)
 =e^{-x} \frac{\cos x -1}{x}.$$
 
 Consequently  $P$ is a primitive of $e^{-x} \frac{\cos x -1}{x}$ which vanishes at $0$.   Hence 
 $$J:=\lim_{x\to\infty} P(x)=\int_0^\infty e^{-x} \frac{\cos x -1}{x}\;dx.$$
 Next we show that $J=- \frac{1}{2}\log 2$ by
  interpreting this integral as the Laplace transform
$L(q)(s)$ of the function $q(x)=(\cos x-1)/x$ evaluated at $s=1$. By a well-known formula, if  $L(F(t))(s)=f(s)$, then
$$L(q)(s)=L( \frac{F(t)}{t})(s)=\int_s^\infty f(u)du,$$
where 
$$f(s)=\int_0^\infty e^{-st} (\cos t -1)\;dt=\frac{1}{s^3+s}.$$
Hence $L(q)(s)=-\frac{1}{2} \log(1+s^{-2})$ and so
$J=L(q)(1)=- \frac{1}{2}\log 2$.

\bigskip

{\bf Remark}
A formal  (but probably unjustifiable) way to calculate the value of $J$ would be the following:

$$J:=\int_0^\infty e^{-x} \frac{\cos x-1}{x}dx=\int_0^\infty e^{-x} \sum_{n=1}^\infty \frac{(-1)^n}{(2n)!} x^{2n-1}dx
$$
$$\buildrel=_{}^{!}\sum_{n=1}^\infty \frac{(-1)^n}{(2n)!}\int_0^\infty x^{2n-1}e^{-x}dx$$
Since for  $m\in \N$ and $k\in \N\setminus\{0\}$,
$$\int_0^\infty x^me^{-kx}dx=m! / k^{m+1},$$
we {\it would} obtain
$$J= \sum_{n=1}^\infty \frac{(-1)^n}{(2n)!} (2n-1)!=\sum_{n=1}^\infty \frac{(-1)^n}{2n}=- \frac{1}{2}\log 2.
$$

Note also that the softwares Wolframalpha/mathematica give the exact value of the integral, too.

The problem itself come to our mind when solving Problem number  12338 in Amer. Math. Soc..

\newpage

\nopagecolor

   \begin{figure}[h!]
 
  \scalebox{0.6} 
  {\includegraphics{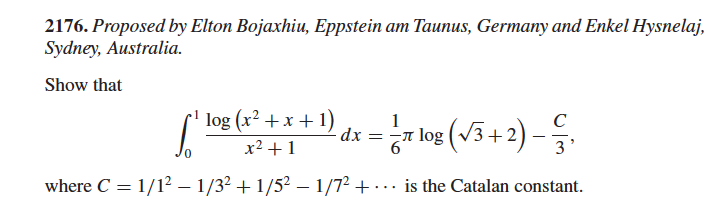}} 
\end{figure}

\centerline{\bf
Solution to problem 2176 Math. Mag. 96 (3) 2023, p. 468}  \medskip

     \centerline{Raymond Mortini and Rudolf Rupp}
     
          \medskip

\centerline{- - - - - - - - - - - - - - - - - - - - - - - - - - - - - - - - - - - - - - - - - - - - - - - - - - - - - -}
  
  \medskip

Let $I:=\dis \int_0^1\frac{\log(1+x+x^2)}{1+x^2}\;dx$.  We make the substitution $x=\tan u$, $dx/du=1+\tan^2 u=\frac{1}{\cos^2 u}$,
$x=0\to u=0$ and $x=1\to u=\pi/4$. Then
\begin{eqnarray*}
I&=& \int_0^{\pi/4} \frac{\log(\tan^2 u +\tan u +1)}{1+\tan^2 u}\;(1+\tan^2 u) du=\int_0^{\pi/4}\log\left(\frac{1+\cos u\sin u}{\cos^2 u}\right) du\\
&=&\int_0^{\pi/4} \log\left(1+\frac{1}{2} \sin(2u)\right) du -2\int_0^{\pi/4}\log(\cos u)du\\
&\buildrel=_{{\rm Lem.\; \ref{cata1}}}^{}&\frac{1}{2}\;\int_0^{\pi/2}  \log\left(1+\frac{1}{2} \sin(v)\right) dv -2\left( \frac{C}{2}-\frac{\pi}{4}\log 2\right)\\
&=&\frac{1}{2}\;\int_0^{\pi/2}\sum_{k=1}^\infty \frac{(-1)^{k+1}}{k}\frac{1}{2^k}(\sin x)^k dx \; +\frac{\pi}{2}\log 2 -C\\
&\buildrel=_{{\rm unif. abs. conv.\atop \int\sum=\sum\int}}^{{\rm Lem. \ref{sinup}}}&
\frac{1}{2}\;\sum_{n=0}^\infty \frac{1}{2n+1}\frac{1}{2^{2n+1}}\frac{4^n}{(2n+1){2n\choose n}} \;- 
\frac{1}{2}\;\sum_{n=1}^\infty \frac{1}{2n}\frac{1}{2^{2n}} \frac{{2n \choose n}}{4^n}\;\frac{\pi}{2} +\frac{\pi}{2}\log 2 -C\\
&=&\frac{1}{4}\;\sum_{n=0}^\infty \frac{1}{(2n+1)^2 {2n\choose n}}- \frac{\pi}{8}\sum_{n=1}^\infty \frac{1}{n}\;{2n\choose n} 16^{-n}
+\frac{\pi}{2}\log 2 -C\\
&\buildrel=_{{\rm Lemm.\;\ref{cata2}}}^{{\rm Lemm.\;\ref{cent}}}&
\frac{1}{4}\left( C-  \frac{1}{8} \pi \log(2+\sqrt 3) \right) \frac{8}{3}   -\frac{\pi}{8}\; 2 \log\left( \frac{1-\sqrt{1-4 (1/16)}}{2 (1/16)}\right)
+\frac{\pi}{2}\log 2 -C\\
&=&-\frac{1}{3} C +\frac{\pi}{6} \log(2+\sqrt 3).
\end{eqnarray*}

\subsection{Appendix}

Here we present for completeness the proofs of all those known results used above to derive the value of the integral.

\begin{lemma}\label{cata1}\cite[formula (8)]{bra}
$$C=2\int_0^{\pi/4}\log(2 \cos x) dx.$$
\end{lemma}

\begin{proof}
Since on $]0,1]$ the integrable function $|\log x|$ dominates the modulus of the  partial sums  $\sum_{n=0}^N(-1)^n x^{2n} \log x $, we have
 
$$\int_0^1 \frac{\log x}{1+x^2} dx=\sum_{n=0}^\infty \int_0^1 (-1)^n x^{2n} \log x \;dx =\sum_{n=0}^\infty (-1)^n \frac{-1}{(n+1)^2}=-C.$$

Hence, with $x=\tan u$ and $dx = (1+\tan^2 u) du$,
 
\begin{equation}\label{lo1}
C=-\int_0^{\pi/4} \log(\tan u)\, du=\int_0^{\pi/4} \log(\cos u) \,du-\int_0^{\pi/4}\log(\sin u)\, du=:L_c-L_s.
\end{equation}
Now, using the standard result that  $\int_0^{\pi/2}\log\sin x\,dx =-\frac{\pi}{2} \log 2$, we obtain
\begin{equation}\label{lo2}
L_c+L_s+\frac{\pi}{4}\,\log 4=\int_0^{\pi/4} \log(2\sin(2u))\;du=\int_0^{\pi/2} \log(2\sin x)\;dx=0.
\end{equation}
Adding (\ref{lo1}) and (\ref{lo2}), yields
$$C -\frac{\pi}{4}\log 4= 2 L_c.$$
In other words,
$$2\int_0^{\pi/4} \log(2\cos x) dx= C.
$$

\end{proof}

\begin{lemma}\label{sinu}
Let $I_n:=\int_0^{\pi/2} (\sin x)^{n} dx$. Then  $I_0=\pi/2$, $I_1=1$ and for $n\in \N^*$, 
\begin{enumerate}
\item[(1)] $ I_{2n}=\frac{1}{2}\cdot\frac{3}{4}\cdot\frac{5}{6}\cdots\frac{2n-1}{2n}\frac{\pi}{2}=\frac{(2n)!}{4^n (n!)^2}\cdot\frac{\pi}{2}=
\frac{{2n\choose n} }{4^n}\cdot \frac{\pi}{2}$.
\item[(2)] $ I_{2n+1}=\frac{2}{3}\cdot\frac{4}{5}\cdot\frac{6}{7}\cdots\frac{2n}{2n+1}= \frac{4^n (n!)^2}{(2n+1)!}= \frac{4^n}{(2n+1){2n \choose n}}$.
\end{enumerate}
\end{lemma}
\begin{proof}
(1) $I_{2n}= \frac{2n-1}{2n} I_{2n-2}$ for $n\in \N^*$ and $I_0=\frac{\pi}{2}$, because

\begin{eqnarray*}
2n I_{2n}- (2n-1)I_{2n-2}&=&\int_0^{\pi/2}(\sin x)^{2n-2}\left(2n \sin^2 x-(2n-1)\right)dx\\
&=&-\int_0^{\pi/2}(\sin x)^{2n-2}\left( (2n-1) \cos ^2x-\sin^2 x\right)dx\\
&=&-\left[ (\sin x)^{2n-1}\cos x   \right]^{\pi/2}_0=0.
\end{eqnarray*}

(2) $I_{2n+1}= \frac{2n}{2n+1} I_{2n-1}$ for $n\in \N^*$ and $I_1=1$, because

\begin{eqnarray*}
(2n+1)I_{2n+1}-2nI_{2n-1}&=& \int_0^{\pi/2} (\sin x)^{2n-1}\left((2n+1)\sin ^2 x -2n\right)dx\\
&=&-\int_0^{\pi/2}(\sin x)^{2n-1}\left(2n \cos^2 x-\sin^2x\right)dx\\
&=&-\left[(\sin x)^{2n} \cos x\right]^{\pi/2}_0=0.
\end{eqnarray*}
\end{proof}

\begin{lemma}\label{cent} \cite{leh}\hfill
\begin{enumerate}
\item[(1)] $\dis \sum_{n=0}^\infty {2n \choose n} x^n=\frac{1}{\sqrt{1-4x}}$ for $|x|<1/4$.
\item[(2)] $\dis \sum_{n=1}^\infty\frac{1}{n} {2n \choose  n} x^n=2 \log\left(\frac{1-\sqrt{1-4x} }{2x}\right)$ for $|x|<1/4$.
\end{enumerate}
\end{lemma}
\begin{proof}
(1) Note that 
\begin{eqnarray*}
{-1/2\choose n}&=& \frac{(-\frac{1}{2})\cdot(- \frac{3}{2})\cdot  \cdots(-\frac{1}{2}-n+1)}{n!}=
(-1)^n \frac{1\cdot3\cdot5\cdots (2n-1)}{2^n n!}\\
&=&(-1)^n \frac{(2n)!}{(2^n n!)^2}= (-1)^n \frac{{2n\choose n}}{4^n}.
\end{eqnarray*}
Hence, by Newton's binomial theorem
$$\sum_{n=0}^\infty {2n\choose n} x^n= \sum_{n=0}^\infty (-1)^n {-1/2\choose n} (4x)^n=(1-4x)^{-1/2}.$$
(2) Let $f(x):=\dis \sum_{n=1}^\infty\frac{1}{n} {2n \choose  n} x^n$. Then, for $0<x<1/4$,  
$$f'(x)=\frac{1}{x}\sum_{n=1}^\infty {2n \choose  n} x^n=\frac{1}{x\sqrt{1-4x}}-\frac{1}{x}.$$
To calculate the primitive, we make the transformation $u:=\sqrt{1-4x}$, or equivalently   $x=\frac{1-u^2}{4}$.
Since
$$
\int  \frac{4}{1-u^2} \frac{1}{u} \left(-\frac{u}{2}\right) du=\log\left(\frac{1-u}{1+u}\right),
$$
we deduce that
\begin{eqnarray*}
f(x)&=&  \log\left( \frac{1-\sqrt{1-4x}}{1+\sqrt{1-4x}}  \right)-\log x= \log\left( \frac{(1-\sqrt{1-4x})^2}{4x}\right)-\log x\\
&=& \log\left( \frac{(1-\sqrt{1-4x})^2}{(2x)^2}\right)=2 \log\left(\frac{1-\sqrt{1-4x} }{2x}\right).
\end{eqnarray*}

The following formula  is due to Ramanujan.

\begin{lemma}\label{cata2}\cite[formula (62)]{bra}
$$C=\frac{1}{8} \pi \log(2+\sqrt 3) +\frac{3}{8} \sum_{n=0}^\infty \frac{1}{(2n+1)^2 {2n\choose n}},$$
equivalently
$$ \sum_{n=0}^\infty \frac{1}{(2n+1)^2 {2n\choose n}}=\frac{8}{3}\, C+\frac{1}{3}\pi\log(2-\sqrt 3).$$
\end{lemma}

\begin{proof}
First we note that 
\begin{equation}\label{ar}
\artanh z=\frac{1}{2}\log\frac{1+z}{1-z}=\sum_{n=0}^\infty \frac{1}{2n+1} z^{2n+1},~~ |z|<1.
\end{equation}
Hence, by Lemma \ref{sinu}
\begin{eqnarray*}
J:=\int_0^{\pi/2}\log  \left(\frac{1+\frac{1}{2}\sin x}{1-\frac{1}{2}\sin x}\right)\;dx&=&\int_0^{\pi/2}2 \sum_{n=0}^\infty \frac{1}{2^{2n+1}} 
\frac{(\sin x)^{2n+1}}{2n+1}\;dx
\buildrel=_{ \int\sum=\sum\int}^{{\rm unif. abs. conv.}} \sum_{n=0}^\infty\frac{1}{2^{2n}} \frac{4^n}{(2n+1)^2 {2n\choose n}}\\
&=&\sum_{n=0}^\infty \frac{1}{(2n+1)^2 {2n\choose n}}.
\end{eqnarray*}
To calculate this integral we combine  calculations done in \cite{ms}  and \cite{ms5}, where it is  also shown that 
$$J=\int_0^1 \frac{\artanh \sqrt{u(1-u)}}{ \sqrt{u(1-u)}}du
$$
(just put $u=\sin^2x$). See \cite{ms4}, too.
Let us introduce the parametric integral
$$I(a):= \int_0^{\pi/2}\log  \left(\frac{1+\sin a\sin x}{1-\sin a\sin x}\right)\;dx.$$
Now $\frac{d}{da}\int=\int\frac{d}{da}$ (as all functions considered here  are continuously differentiable). Hence
\begin{eqnarray*}
I'(a)&=&\int_0^{\pi/2}\left( \frac{\cos a\sin x}{1+\sin a\sin x}+\frac{\cos a\sin x}{1-\sin a\sin x}\right)\;dx
=\int_0^{\pi/2}\frac{2\cos a\sin x}{1-\sin^2 a\sin^2 x}\;dx\\
&=&\frac{2\cos a}{\sin^2 a} \int_0^{\pi/2}  \frac{\sin x}{ \frac{1}{\sin^2 a}+\cos^2 x -1}\;dx=
\frac{2\cos a}{\sin^2 a} \int_0^{\pi/2}  \frac{\sin x}{\cot ^2 a +\cos^2 x}\;dx\\
&=&- \frac{2}{\sin a} \arctan( \cos x \tan a)\Big|_0^{\pi/2}= \frac{2a}{\sin a}.
\end{eqnarray*}
Thus, using partial integration,  and the fact that $\tan(x/2)=\frac{\sin x}{1+\cos x}$,
\begin{eqnarray*}
J&=&I(\pi/6)=I(\pi/6)-I(0)=\int_0^{\pi/6} I'(a) da=\int_0^{\pi/6} \frac{2a}{\sin a} \,da\\
&=&2 \int_0^{\pi/6} x \left( \log\Big(\tan \frac{x}{2}\Big)\right)' \;dx= 2x \log\Big(\tan \frac{x}{2}\Big)\Big|_0^{\pi/6}-
2\int_0^{\pi/6}\log\Big(\tan \frac{x}{2}\Big)\;dx\\
&\buildrel=_{\frac{x}{2}=t}^{}&\frac{\pi}{3}\log(2-\sqrt 3)-4\int_0^{\pi/12}\log(\tan t) dt.
\end{eqnarray*}

Now we follow \cite{ms5} \footnote{ We thank Roberto Tauraso for providing us this link.}.
Recall that on $]0,\pi[$  the Fourier series for $-\log\tan (t/2)$ is 
$$2\sum_{n=0}^\infty \frac{1}{2n+1} \cos(2n+1)t.$$
Since the Fourier series converges in the $L^2$-norm, hence $L^1$-norm on $]0,\pi[$, we have $\sum\int=\int\sum$. Hence
\begin{eqnarray*}
-\int_0^{\pi/12}\log(\tan x)dx &\buildrel=_{x=t/2}^{}& \int_0^{\pi/6} \sum_{n=0}^\infty \frac{1}{2n+1} \cos(2n+1)t \; dt=
 \sum_{n=0}^\infty \frac{1}{2n+1}\int_0^{\pi/6} \cos(2n+1)t \; dt\\
&=& \underbrace{\sum_{n=0}^\infty \frac{1}{(2n+1)^2} \sin\big(\frac{\pi}{6}(2n+1)\big)}_{:=S}\buildrel=_{}^{!}\frac{2}{3}\sum_{n=0}^\infty \frac{1}{(2n+1)^2}(-1)^n =\frac{2}{3}\;C.
\end{eqnarray*}
To show the penultimate identity, we follow \cite{ms3}. To this end we first note that
$$\sin\left(\frac{\pi}{6}(2k+1)\right)=\begin{cases} 
\frac{1}{2}&\text{if $k\equiv 0 \mod 6$}\\
1&\text{if $k\equiv 1 \mod 6$}\\ 
\frac{1}{2}&\text{if $k\equiv 2 \mod 6$}\\
-\frac{1}{2}&\text{if $k\equiv 3 \mod 6$}\\
-1&\text{if $k\equiv 4 \mod 6$}\\
-\frac{1}{2}&\text{if $k\equiv 5 \mod 6$.}
\end{cases}
$$
Hence
\begin{eqnarray*}
S&=& \frac{1}{2}\;\sum_{n=0}^\infty \frac{1}{(12n+1)^2} + \bl{ \sum_{n=0}^\infty \frac{1}{(12n+3)^2} }
+ \frac{1}{2}\; \sum_{n=0}^\infty \frac{1}{(12n+5)^2}\\
&&-\frac{1}{2}\;\ \sum_{n=0}^\infty \frac{1}{(12n+7)^2} \bl{- \sum_{n=0}^\infty \frac{1}{(12n+9)^2}}
-\frac{1}{2}\;\ \sum_{n=0}^\infty \frac{1}{(12n+11)^2}\\
&=&\bl{\frac{1}{9}\;\sum_{n=0}^\infty \left( \frac{1}{(4n+1)^2}-\frac{1}{(4n+3)^2}\right)}+
 \frac{1}{2} \left( \frac{1}{1^2}+\frac{1}{5^2}-\frac{1}{7^2}-\frac{1}{11^2}+\cdots\right)\\
 &=&\bl{\frac{1}{9}\;C} +\frac{1}{2}\; \left( \frac{1}{1^2} \co{-\frac{1}{3^2}}+\frac{1}{5^2}-\frac{1}{7^2} +\co{\frac{1}{9^2}}-\frac{1}{11^2}+\cdots\right)  +\frac{1}{2}\;\left(\co{\frac{1}{3^2} - \frac{1}{9^2} +  \frac{1}{15^2} -\cdots}\right) \\ 
 &=&\bl{\frac{1}{9}\;C} +\frac{1}{2} C +\frac{1}{2}\;\cdot \frac{1}{3^2}\co{\left( \frac{1}{1^2}  -  \frac{1}{3^2}  + \frac{1}{5^2}-\cdots \right)}\\
 &=&\bl{\frac{1}{9}\;C} +\frac{1}{2} C+\frac{1}{18}\;C=\frac{12}{18}\; C=\frac{2}{3}\;C.
\end{eqnarray*}

We conclude that 
$$J=\frac{\pi}{3}\log(2-\sqrt 3)+\frac{8}{3}\;C= \frac{8}{3}\;C-\frac{\pi}{3}\log(2+\sqrt 3).$$
\end{proof}

Here is a second proof to calculate the value of $\int_0^{\pi/12}\log(\tan t) dt.$

\begin{proof}

We follow \cite{ms3}. Consider for $a>0$  the integral
$$Q(a)=-\int_0^{\pi/12} \artanh\left(\frac{2\cos 2x}{a+a^{-1}}\right)\,dx=
-\int_0^{\pi/12} \artanh\left(\frac{2a\cos 2x}{a^2+1}\right)\,dx
.$$
(Note that $a+1/a\geq 2$, so this is well defined). Again $\frac{d}{da}\int=\int\frac{d}{da}$. Using that 
$\dis(\artanh z)' =\frac{1}{1-z^2}$, and that 
$$\frac{d}{da}\left( \frac{1}{a+a^{-1}}\right)= \frac{1-a^2}{(a^2+1)^2},$$
\begin{eqnarray*}
Q'(a)&=&-\int_0^{\pi/12}\frac{\frac{1-a^2}{(1+a^2)^2}\;2 \cos 2x}{1-\frac{4\cos^2 (2x)}{(a+a^{-1})^2}}\;dx=
\int_0^{\pi/12}\frac{(a^2-1)\; 2\cos 2x}{ (a^2+1)^2-4a^2\cos^2 2x}\;dx\\
&=&-\frac{1}{2a} \arctan\left( \frac{2a \sin 2x}{1-a^2}\right)\Big|_0^{\pi/12}=\frac{\arctan\frac{a}{a^2-1}}{2a}.
\end{eqnarray*}
Hence, by using that $Q(0)=0$, and that $\arctan u+\arctan v= \arctan(\frac{u+v}{1-uv})$,
$$\log(\tan x)=-\frac{1}{2} \log\left(\frac{1+\cos 2x}{1-\cos 2x}\right)=-\artanh (\cos 2x).$$
Consequently,  as $C=-\int_0^1 \frac{\arctan x}{x}dx$ (use the power series for $\arctan x$)
\begin{eqnarray*}
\int_0^{\pi/12}\log(\tan x)dx&=&Q(1)=\int_0^1 Q'(a) da=\int_0^1 \frac{\arctan\frac{a}{a^2-1}}{2a} da\\
&=&- \int_0^1\left( \frac{\arctan a}{2a}+ \frac{\arctan a^3}{2a}\right)\; da\\
&=\buildrel=_{a^3\to b}^{}&-\left(\frac{1}{2}+\frac{1}{6}\right)\int_0^1 \frac{\arctan a}{a}\;da=-\frac{2}{3} C.
\end{eqnarray*}
We conclude that 
$$J=\frac{\pi}{3}\log(2-\sqrt 3)+\frac{8}{3}\;C= \frac{8}{3}\;C-\frac{\pi}{3}\log(2+\sqrt 3).$$
\end{proof}

\subsection{Remarks}

(1) The integral $L:=\int_0^\infty \frac{\log(1+x+x^2)}{1+x^2}\;dx$ is mentioned on wikipedia \cite{wiki} (without a source) under the form
$$C=\frac{3}{4}\, L -\frac{\pi}{4}{\rm arcosh}\; 2.$$
We notice that $L=2I+2C$. In fact,
\begin{eqnarray*}
I&\buildrel=_{u=1/x}^{}&\int_1^\infty \frac{\log(u^2+u+1)-\log (u^2)}{1+u^2}\;du\\
&=&\int_1^\infty \frac{\log(u^2+u+1)}{u^2+1} -2C.
\end{eqnarray*}
Hence
$$2I= I+I=\int_0^\infty \frac{\log(u^2+u+1)}{u^2+1} -2C.$$
Using the assertion of the problem dealt with here, 
$$L= 2 \left(\frac{\pi}{6}\log(\sqrt 3+2) -\frac{C}{3}\right)+2C,
$$
and so
$$C=\frac{3}{4} L - \frac{\pi}{4}\log(\sqrt 3+2).
$$
Now just note that ${\rm arcosh}\; 2=\log(\sqrt 3+2)$.
\end{proof}

This  integral $L$  is also calculated in \cite{ms2}.

\newpage

  \begin{figure}[h!]
 
  \scalebox{0.6} 
  {\includegraphics{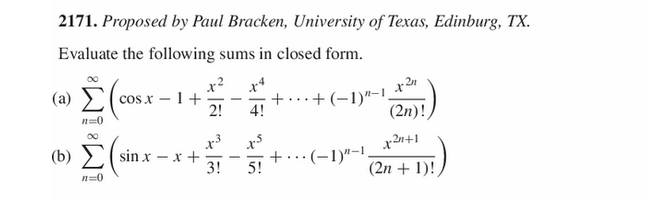}} 
\end{figure}

\centerline{\bf
Solution to problem 2171 Math. Mag. 96 (3) 2023, p. 359}  \medskip

     \centerline{Raymond Mortini and Rudolf Rupp}
     
               \medskip

\centerline{- - - - - - - - - - - - - - - - - - - - - - - - - - - - - - - - - - - - - - - - - - - - - - - - - - - - - -}
  
  \medskip

 For logical reasons, we think that the exponent of $(-1)$ in the two statements above has to be $n+1$, since one starts with $-(-1)^n$, where $n=0$.
Since the double series is absolutely convergent, we may arrange as we wish. \\

(a) Let  
$$C(x):=\sum_{n=0}^\infty \left(\cos x -1+\frac{x^2}{2!}-\frac{x^4}{4!}+\cdots+(-1)^{n+1}\frac{x^{2n}}{(2n)!}\right)$$
and let $T_n$ be the $2n$-th Taylor polynomial for $\cos x$, which is given by
$$T_n(x)=\sum_{j=0}^n (-1)^j \frac{x^{2j}}{(2j)!}.$$
 Then 
$$C(x)=\sum_{n=0}^\infty (\cos x-T_n(x)).$$
Hence

\begin{eqnarray*}
C(x)
&=&\sum_{n=0}^\infty \sum_{k=n+1}^\infty (-1)^{k}\frac{x^{2k}}{(2k)!}
=\sum_{k=1}^\infty\sum_{n=0}^{k-1} (-1)^{k}\frac{x^{2k}}{(2k)!}\\
&=&\sum_{k=1}^\infty k (-1)^{k+1} \frac{x^{2k}}{(2k)!}
= \frac{1}{2}\sum_{k=1}^\infty 2k (-1)^{k} \frac{x^{2k}}{(2k)!}\\
&=&\frac{x}{2}\sum_{k=1}^\infty  (-1)^{} \frac{x^{2k-1}}{(2k-1)!}\\
&=&-\frac{1}{2} x \sin x.
\end{eqnarray*}

For (b) we give two solutions.
\newpage 
Similarily to (a) , let
$$S(x):=\sum_{n=0}^\infty \left(\sin x -x+\frac{x^3}{3!}-\frac{x^5}{5!}+\cdots+(-1)^{n+1}\frac{x^{2n+1}}{(2n+1)!}\right).$$
Then
\begin{eqnarray*}
S(x)
&=&\sum_{n=0}^\infty \sum_{k=n+1}^\infty (-1)^{k}\frac{x^{2k+1}}{(2k+1)!}
=\sum_{k=1}^\infty\sum_{n=0}^{k-1} (-1)^{k}\frac{x^{2k+1}}{(2k+1)!}\\
&=&\sum_{k=1}^\infty k (-1)^{k} \frac{x^{2k+1}}{(2k+1)!}
=\frac{1}{2}\sum_{k=1}^\infty 2k (-1)^{k} \frac{x^{2k+1}}{(2k+1)!}\\
&=&\frac{1}{2}\sum_{k=1}^\infty (2k+1) (-1)^{k} \frac{x^{2k+1}}{(2k+1)!}- \frac{1}{2}\sum_{k=1}^\infty  (-1)^{k} 
\frac{x^{2k+1}}{(2k+1)!}\\
&=& \frac{x}{2} \sum_{k=1}^\infty (-1)^{k} \frac{x^{2k}}{(2k)!}-\frac{1}{2}(\sin x-x)\\
&=& \frac{x}{2}(\cos x-1) -\frac{1}{2}(\sin x-x)\\
&=&\frac{x\cos x-\sin x}{2}
\end{eqnarray*}

The second method is to integrate termwise and then  to interchange the sum with the integral (uniform convergence on compacta). 

\begin{eqnarray*}
S(x)&=&\sum_{n=0}^\infty \int_0^x\left( \cos t -1+\frac{t^2}{2!}-\frac{t^4}{4!}+\cdots+(-1)^{n+1} \frac{t^{2n}}{(2n)!}\right)\;dt\\
&=& \int_0^x  \sum_{n=0}^\infty \left( \cos t -1+\frac{t^2}{2!}-\frac{t^4}{4!}+\cdots+(-1)^{n+1} \frac{t^{2n}}{(2n)!}\right)\;dt\\
&\buildrel=_{(a)}^{}& -\frac{1}{2}\int_0^x (t\sin t)\;dt\\
&=&\frac{x\cos x-\sin x}{2}.
\end{eqnarray*}

\newpage

   \begin{figure}[h!]
 
  \scalebox{0.5} 
  {\includegraphics{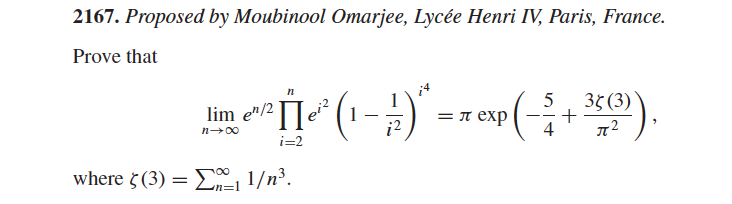}} 
\end{figure}

\nopagecolor

\centerline{\bf
Solution to problem 2167 Math. Mag. 96 (2) 2023, p. 190}  \medskip

     \centerline{Raymond Mortini and Rudolf Rupp}
     
                    \medskip

\centerline{- - - - - - - - - - - - - - - - - - - - - - - - - - - - - - - - - - - - - - - - - - - - - - - - - - - - - -}
  
  \medskip

Let 
$$L:=e^{n/2} \prod_{j=2}^n e^{j^2}\; \prod_{j=2}^n \left(1- \frac{1}{j^2}\right)^{j^4}.
$$
Then, by  using that  for $|x|<1$, $\dis-\log(1-x)=\sum_{k=1}^\infty \frac{1}{k} x^k$,
\begin{eqnarray*}
\log L&=&\frac{n}{2}+\sum_{j=2}^n j^2 +\sum_{j=2}^n j^4 \log\left(1- \frac{1}{j^2}\right)\\
&=&\frac{n}{2}+\sum_{j=2}^n j^2-\sum_{j=2}^n \sum_{k=1}^\infty \frac{j^4}{k}\frac{1}{j^{2k}}\\
&=&\frac{n}{2}+\sum_{j=2}^n j^2-\sum_{j=2}^n j^2 -\frac{1}{2}\sum_{j=2}^n 1-
\sum_{k=3}^\infty  \frac{1}{k}\sum_{j=2}^n\frac{1}{(j^2)^{k-2}}\\
&\buildrel \longrightarrow_{n\to\infty}^{m:=k-2}&\frac{1}{2}- \sum_{m=1}^{\infty} \frac{\zeta(2m)-1}{m+2}.
\end{eqnarray*}
So we need to show that
\begin{equation}\label{zeta1}
\ovalbox{$ \dis\sum_{m=1}^{\infty} \frac{\zeta(2m)-1}{m+2}=\frac{7}{4}-\log\pi-\frac{3}{\pi^2}\zeta(3)\sim 0.239888629\cdots,$}
\end{equation}
from which we conclude that
$$L=\pi\;\exp\left(\frac{3}{\pi^2}\zeta(3)-\frac{5}{4}\right)\sim 1.2970745345\cdots$$

To achieve our goal we use the partial fraction decomposition of 
$$\pi z \cot(\pi z)= 1+ \sum_{n=1}^\infty \frac{2z^2}{z^2-n^2}, ~~ z\in \C\setminus \Z,$$
a formula which implies   (see \cite[p. 182]{fl}) that 
$$ \sum_{n=1}^\infty(\zeta(2n)-1)x^{2n}= \frac{1}{2}\big(1-\pi x\cot(\pi x)\big) -\frac{x^2}{1-x^2} =
 -\sum_{n=2}^\infty\frac{x^2}{x^2-n^2}.$$
 
 Since $z=0$ and $z=\pm 1$ are removable singularities for  $\frac{1}{2}\big(1-\pi z\cot(\pi z)\big) -\frac{z^2}{1-z^2}$,
  the holomorphy in $|z|<2$ implies that we have uniform convergence of  $\sum_{n=1}^\infty(\zeta(2n)-1)x^{2n}$ on $[0,1]$.
Also, for $x\in\; ]0,1[$,
$$\sum_{n=1}^\infty(\zeta(2n)-1)x^{n+1}= \frac{x}{2}-\frac{\pi}{2} x^{3/2} \cot(\pi \sqrt x)-\frac{x^2}{1-x}.$$

A primitive on $]0,1[$ is then given by
$$\sum_{n=1}^\infty\frac{\zeta(2n)-1}{n+2}x^{n+2}= \frac{x^2}{4}-\int \left(\frac{\pi}{2}x^{3/2} \cot(\pi \sqrt x)+\frac{x^2}{1-x}\right)\;dx.
$$
Using  again uniform convergence on $[0,1]$, the substitution $s=\sqrt x$ and integration between 0 and 1 yields
\begin{eqnarray*}
\sum_{n=1}^\infty\frac{\zeta(2n)-1}{n+2}&=&\frac{1}{4}-\int_0^1 \left(\frac{\pi}{2}s^3 \cot(\pi s)+\frac{s^4}{1-s^2}\right)\;2sds\\
&=&\frac{1}{4}-\int_0^1 \left(\pi s^4 \cot(\pi s)+\frac{2s^5}{1-s^2}\right)\;ds
\end{eqnarray*}

Let 

$$\ovalbox{$\dis
I:= \int_0^1\left(\pi\;t^4\cot (\pi t)+\frac{2t^5}{1-t^2}\right)\;dt.
$}
$$
We claim that 
\begin{equation}\label{zeta2}
I= \frac{3}{\pi^2} \zeta(3) -\frac{3}{2}+\log\pi\sim 0.0101113705\cdots
\end{equation}

To determine the value of $I$, we first calculate a primitive of 
$$f(x):=\pi\;x^4\cot (\pi x)$$
on $]0,1[$. 
This is done by  using  partial integration
and  the 1-periodic Fourier series
$$\sum_{k=1}^\infty \frac{\cos(2k\pi x)}{k}= -\log(2\sin(\pi x))=-\log 2 -\log(\sin(\pi x)),$$
where the convergence is considered in  the $L^2$-norm on $]0,1[$, which  also guarantees that $\int\sum=\sum\int$ below. 
\begin{eqnarray*}
\int f(x) dx&=& x^4 \log (\sin(\pi x))- 4 \int x^3\log (\sin(\pi x))\;dx\\
&=&  x^4 \log (\sin(\pi x))+4 \int \sum_{k=1}^\infty  x^3  \frac{\cos(2k\pi x)}{k} \;dx  +4\int x^3 \log 2 \;dx\\
&=&  x^4 \log (\sin(\pi x))+4  \sum_{k=1}^\infty  \int x^3  \frac{\cos(2k\pi x)}{k} \;dx  +4 \int x^3 \log 2 \;dx
\end{eqnarray*}
Before evaluating at the boundary points, we need to add 
$$\frac{2x^5}{1-x^2}= -2x^3-2x -\frac{1}{1+x}+\frac{1}{1-x},$$
 since the integral 
$\int_0^1 f(x) dx $ is divergent (at $1$). Defining  the symbol $[h(x)]_0^1$ below as 
$$[h(x)]_0^1:=\lim_{x\to 1-} h(x) -\lim_{x\to 0} h(x),$$
 we obtain

\begin{eqnarray*}
I=\int_0^1\left(\pi\;x^4\cot (\pi x)+\frac{2x^5}{1-x^2}\right)\;dx&=& \Bigg[ x^4 \log (\sin(\pi x))-\log(1-x)  -\frac{x^4}{2}
-x^2 -\log(1+x)\Bigg]^1_0\\
&&+4  \sum_{k=1}^\infty \int_0^1 x^3  \frac{\cos(2k\pi x)}{k} \;dx+\log 2. 
\end{eqnarray*}

Note that 
$$x^4 \log (\sin(\pi x))-\log(1-x)=x^4\; \log \frac{\sin(\pi x)}{1-x}+(x^4-1)\log(1-x)\to \log\pi \text{\;\;as $x\to 1$}.$$

Also, three times partial integration yields
$$ \int_0^1 x^3  \frac{\cos(2k\pi x)}{k} \;dx=\frac{3}{4k^3\pi^2}.
$$
Hence
\begin{eqnarray*}
I&=&\log\pi -\frac{3}{2}-\log 2 +4 \sum_{k=1}^\infty \frac{3}{4k^3\pi^2}+\log 2 \\
&=& \log\pi -\frac{3}{2}+\frac{3}{\pi^2}\,\zeta(3),
\end{eqnarray*}
yielding (\ref{zeta2}).
We conclude that  (\ref{zeta1}) is satisfied, that is
$$\sum_{n=1}^\infty \frac{\zeta(2n)-1}{n+2}= \frac{1}{4}-I=\frac{7}{4}-\log\pi-\frac{3}{\pi^2}\zeta(3)\sim  0.23988862\cdots$$

 {\bf Remarks} 
 (1) The value for $I$ is also given directly by Maple

   \begin{figure}[h!]
 \scalebox{0.3} 
  {\includegraphics{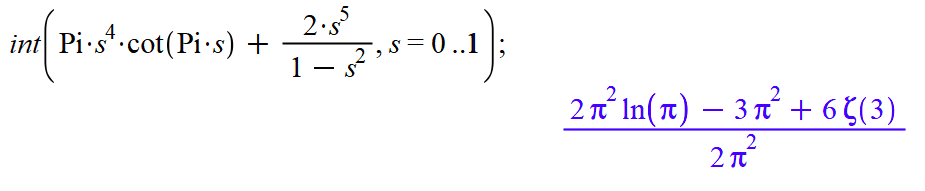}} 
\end{figure}

 (2) Using  Wolframalpha's representation below of  a primitive  of $\pi\;t^4\cot (\pi t)+\frac{2t^5}{1-t^2}$ and evaluating at the boundary points, we also obtain the value of $I$. Just note  that $\zeta(2)= \pi^2/6$ and $\zeta(4)=\pi^4/90$:

\begin{eqnarray*}
I&=& 2i \frac{\zeta(2) }{\pi}+ 3\frac{\zeta(3)}{\pi^2}-3i \frac{\zeta(4)}{\pi^3}-3 \frac{\zeta(5)}{2\pi^4} +\frac{i\pi}{5}-\frac{1}{2}-1 
+\co{\lim_{s\to 1} \big(s^4 \log(1-e^{-2\pi i s})-\log(1-s^2)\big)} \\
&& +3 \frac{\zeta(5)}{2\pi^4} -\bl{\lim_{s\to 0} \big(s^4 \log(1-e^{-2\pi i s})-\log(1-s^2)\big)}\\
&=& 3\frac{\zeta(3)}{\pi^2}  + i\left(\frac{1}{3}-\frac{1}{30} \right)\pi  +\frac{i\pi}{5}-\frac{3}{2}+(\co{ -i \frac{\pi}{2} +\log\pi} ) -(\bl 0)\\
&=&3\frac{\zeta(3)}{\pi^2}-\frac{3}{2}+\log\pi.
\end{eqnarray*}

   \begin{figure}[h!]
 \scalebox{0.5} 
  {\includegraphics{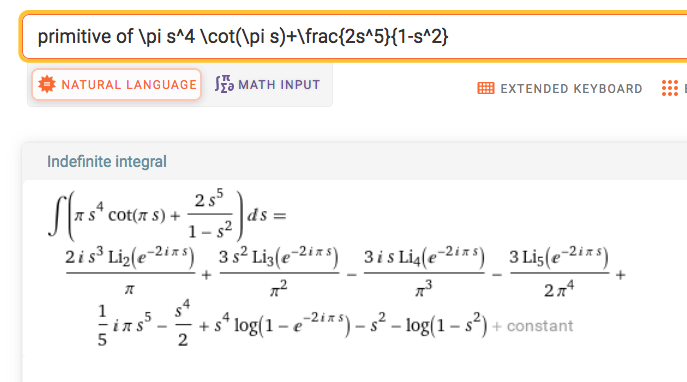}} 
\end{figure}

(3) Generalizations of formula \ref{zeta1} are given in \cite{mr24}.
\newpage

   \begin{figure}[h!]
 
  \scalebox{0.6} 
  {\includegraphics{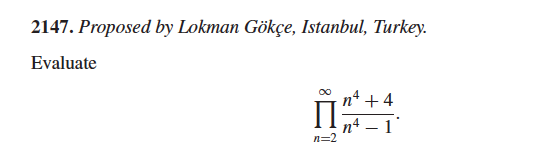}} 

\end{figure}

\centerline{\bf
Solution to problem 2147 Math. Mag. 95 (2) 2022, p. 242}  \medskip

     \centerline{Raymond Mortini and Rudolf Rupp}
          
               \medskip

\centerline{- - - - - - - - - - - - - - - - - - - - - - - - - - - - - - - - - - - - - - - - - - - - - - - - - - - - - -}
  
  \medskip

We show that, in accordance with WolframAlpha,
$$P:=\prod_{n=2}^\infty \frac{n^4+4}{n^4-1}=\frac{2\sinh\pi}{5\pi}.$$

Due to $\sin (iz)=i\sinh z$, we have
$$P(z):=\prod_{n=1}^\infty \left(1-\frac{z^4}{n^4}\right)=\prod_{n=1}^\infty \left(1-\frac{z^2}{n^2}\right)\;\prod_{n=1}^\infty 
\left(1+\frac{z^2}{n^2}\right)=\frac{\sin \pi z\;\sinh \pi z}{\pi^2z^2}$$
and so
$$Q(z):=\prod_{n=2}^\infty \left(1-\frac{z^4}{n^4}\right)= \frac{P(z)}{ 1-z^4}$$

Note that
$$P=\prod_{n=2}^\infty \frac{1+\frac{4}{n^4}}{1-\frac{1}{n^4}}.$$

We put either $z=1$ or $z=1+i$. Note that $(1+i)^4=-4$ and that 
$$\lim_{z\to 1} Q(z)=\lim_{z\to 1} \frac{\sin\pi z}{1-z} \;\lim_{z\to 1}\frac{1}{(1+z)(1+z^2)}\; \frac{\sinh \pi}{\pi^2}=
\frac{1}{4}\; \frac{\sinh\pi}{\pi}.$$

Hence
$$P= \frac{1}{1-(1+i)^4}\frac{\sin (\pi (1+i))\;\sinh ( \pi (1+i))}{\pi^2(1+i)^2}\;\bigg/  \frac{1}{4}\; \frac{\sinh\pi}{\pi}=\frac{2}{5}\frac{\sinh\pi}{\pi}
$$

\newpage

   \begin{figure}[h!]
 
  \scalebox{0.5} 
  {\includegraphics{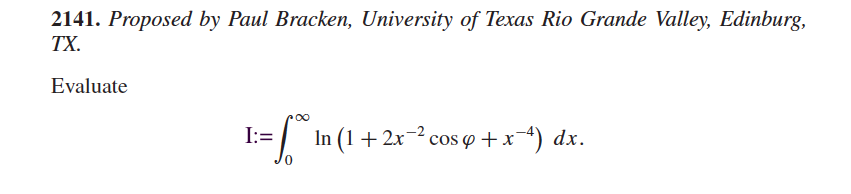}} 

\end{figure}

   \centerline {\bf Solution to problem 2141 Math. Mag. 95 (2) 2022, p. 157}  \medskip
   
     \centerline{Raymond Mortini and Rudolf Rupp}
  
                   \medskip

\centerline{- - - - - - - - - - - - - - - - - - - - - - - - - - - - - - - - - - - - - - - - - - - - - - - - - - - - - -}
  
  \medskip

The value of the integral $I$ is $2\pi \cos(\varphi/2)$.\\

First we note that $(u+e^{i\varphi})(u+e^{-i\varphi})= 1+2u \cos\varphi +u^2$.

{\bf Case 1} $\cos\varphi\not=0$ (or equivantly $\varphi\notin\{\pi+2k\pi:k\in \mathbb Z\}$).

Let  $\log z=\log |z|+i\arg z$ be the main branch of the complex logarithm  (that is $-\pi<\arg z<\pi$). 
Put $H:=\mathbb C\setminus \;]-\infty, 0]$.
Note that for $z\in S:=\mathbb C\setminus\{it: |t|>1\}$ we have 
\begin{equation}\label{arct}
\arctan z=:\frac{1}{2i}\log\frac{1+iz}{1-iz}
\end{equation}
is a primitive of $1/(1+z^2)$.

Now for $x\in \mathbb R$, $1/x^2+e^{\pm i\varphi}\in H$ and so
$f(x):=\log(1/x^2+e^{\pm i\varphi})$ is well defined.  A primitive is given by

$$F(x)=x \log(1/x^2+e^{\pm i\varphi})-\int x \frac{d}{dx} \log(1/x^2+e^{\pm i\varphi})=
x \log(1/x^2+e^{\pm i\varphi})+\int \frac{2}{1+x^2 e^{\pm i\varphi}} dx.
$$
Since $\varphi\not=\pm \pi$,  $z:=x e^{\pm i\varphi/2}\in S$ and so, by using (\ref{arct}) and the fact that
$\frac{1+iz}{1-iz}$ maps the right-half plane onto the upper-half-plane,
$$\int \frac{2}{1+x^2 e^{i\varphi}} dx=2e^{-i\varphi/2}\arctan (e^{i\varphi/2} x)\buildrel\longrightarrow_{x\to\infty}^{} 2e^{-i\varphi/2} \pi/2$$

Now  $\arg z+\arg \ov z=0$ and so  $\log z+\log \ov z=\log |z|^2$, $z\in H$. Hence, with $z=1/x^2+e^{ i\varphi}\in H$,

$$x \log(1/x^2+e^{ i\varphi})+x \log(1/x^2+e^{- i\varphi})=x\log(1+2x^{-2}\cos\varphi+x^{-4})
\buildrel\longrightarrow_{x\to 0}^{x\to\infty} 0
$$

Hence $\lim_{x\to\infty} I(x)=0 +\pi (e^{-i\varphi/2}+e^{i\varphi/2})=2\pi \cos\varphi/2$. 

{\bf Case 2} $\cos\varphi=0$. In other words
$I=\int_0^\infty \log((1-1/x^2)^2) dx$, which is improper at $0$ and $1$.  In this case $I=0$. In fact, for $x>0$ and $x\not=1$,

\begin{eqnarray*}
h_1(x)&:=&x\log \left(\frac{1}{x^2}-1\right)^2+\log \left(\frac{1+x}{1-x}\right)^2\\
&=& (x-1)\log(x-1)^2 +2(x+1)\log(1+x)-4x\log x
\end{eqnarray*}
is a primitive of $\log((1-1/x^2)^2)$. Hence
$$\int_0^1 \log((1-1/x^2)^2) dx=\log 16$$ and 
$$\int_1^\infty \log((1-1/x^2)^2) dx=\log (1/16).$$

\bigskip

\hrule

\bigskip

A related method (for $\cos\varphi\not=0$) is to apply  partial integration directly to $B(x):=\log(1+2x^{-2}\cos\varphi+x^{-4})$ and which gives
$$\int B(x) dx= x B(x)-\int x B'(x) dx$$
with
$$xB'(x)=4  \frac{1+\cos\varphi \;x^2}{x^4+2\cos\varphi\; x^2 +1}=:4 R(x).$$

This rational function writes as
$$R(x)=\frac{A}{x^2+e^{i\varphi}} + \frac{B}{x^2+e^{-i\varphi}} $$
with 
$$A=\frac{e^{i\varphi} \cos\varphi -1}{2i\sin\varphi},~~ B=\ov A.$$

By using (\ref{arct}), we obtain
$$\int_0^\infty \frac{dx}{x^2+b^2}=\lim_{x\to\infty} \frac{1}{b}\arctan(x/b)=\begin{cases}\phantom{+} \frac{\pi}{2b}
&\text{if Re $b>0$}\\ -\frac{\pi}{2b}&\text{if Re $b<0$.}\end{cases}
$$
Since $\displaystyle \lim_{x\to 0\atop x\to \infty} xB(x)=0$, we   deduce  (with $b:=e^{i\varphi/2}$) that $\int_0^\infty B(x)dx=2\pi \cos(\varphi/2)$.

\newpage

   \begin{figure}[h!]
 
  \scalebox{0.6} 
  {\includegraphics{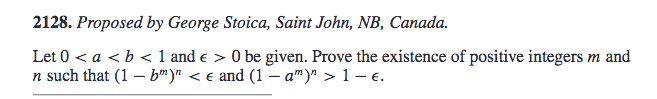}} 

\end{figure}

\centerline {\bf Solution to  problem 2128 Math. Mag. 94 (2021), p. 308}  \medskip

\centerline{Raymond Mortini}

                   \medskip

\centerline{- - - - - - - - - - - - - - - - - - - - - - - - - - - - - - - - - - - - - - - - - - - - - - - - - - - - - -}
  
  \medskip

We first show that
$$\lim_{m\to\infty}(1-b^m)^{\frac{1\phantom{|}}{ma^m}}=0 \text{\;\;and $\dis \lim_{m\to\infty}(1-a^m)^{\frac{1\phantom{|}}{ma^m}}=1$}.$$

In fact, taking logartithms, this is equivalent to show that
$$L:=\lim_{m\to\infty} \frac{\log(1-b^m)}{ma^m}=-\infty \text{\;\; and $R:=\dis\lim_{m\to\infty} \frac{\log(1-a^m)}{ma^m}=0$}.$$

Since $\sum_m ma^m$ converges for $|a|<1$, $ma^m\to 0$. Hence we have an indeterminate form $0/0$ and may use l'Hospital's rule.
Using that for $x>1$, we have $\lim x^m/m=\infty$, we obtain

\begin{eqnarray*}
L&=&\lim_{m\to\infty}\frac{\dis\frac{1}{1-b^m}(-b^m)\log b}{a^m(1+m\log a)}=-\log b\lim_{m\to\infty}\left(\frac{b}{a}\right)^m\frac{1}{1+m\log a}\\
&=&-\frac{\log b}{\log a}\lim_{m\to\infty}\left(\frac{b}{a}\right)^m\frac{1}{m}=-\infty.
\end{eqnarray*}

Moreover

\begin{eqnarray*}
R&=&\lim_{m\to\infty}\frac{\dis\frac{1}{1-a^m}(-a^m)\log a}{a^m(1+m\log a)}=-\log a\lim_{m\to\infty} \frac{1}{m\log a}=0.
\end{eqnarray*}

Next we use that  for  $s_m:=\frac{1}{ma^m}$ and $x_m:=\log(1-b^m)\to 0$
the inequalities 

$$(s_m-1)x_m<\lfloor s_m\rfloor x_m\leq s_mx_m$$
imply that the limits $R$ and $L$  do not change if we replace $s_m$ by $\lfloor s_m\rfloor$.

Consequently, 
$$     \lim_{m\to\infty}(1-b^m)^{\left\lfloor\frac{1}{ma^m}\right\rfloor}=0 \text{\;\;and 
$\dis \lim_{m\to\infty}(1-a^m)^{\left\lfloor\frac{1}{ma^m}\right\rfloor}=1$}$$

Hence, for each $\e\in ]0,1[$ we obtain $m,n\in \N$ such that 
$$(1-b^m)^n<\e \;\; {\rm and}\;\; (1-a^m)^n>1-\e.$$

\newpage

   \begin{figure}[h!]
 
  \scalebox{0.6} 
  {\includegraphics{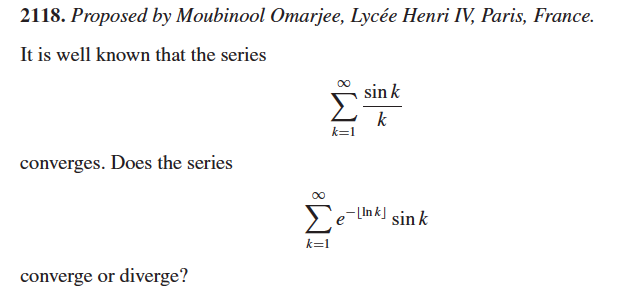}} 

\end{figure}

\centerline {\bf Solution to  problem 2118 Math. Mag.  94 (2021), p. 150}  \medskip

\centerline{Raymond Mortini}

                   \medskip

\centerline{- - - - - - - - - - - - - - - - - - - - - - - - - - - - - - - - - - - - - - - - - - - - - - - - - - - - - -}
  
  \medskip

The series converges.
This is an immediate consequence to Abel's theorem telling us that if $(a_n)$ is a sequence of positive numbers with
 $a_n\searrow 0$, then the trigonometric series
$S(t):=\sum_{n=0}^\infty a_ne^{int}$ converges for all $t\not\in\{2k\pi: k\in \mathbb Z\}$
(see i.e. Appendix 4 in my encyclopedic monograph: R. Mortini, R. Rupp,  Extension Problems  and Stable Ranks, A Space Odyssey, Birkh\"auser 2021, ca 2150 pages):

Just take $a_n=e^{-\lfloor \log n\rfloor}$, $t=1$,  and the imaginary part of $S(t)$. The proof is based on the Abel-Dirichlet rule, telling us that  with $b_n=e^{int}$,  and
$$|b_0+b_1+\cdots+b_m|=|1+e^{it}+\cdots+ e^{imt}|=$$
$$\mbox{$\displaystyle =\left|\frac{1-e^{(m+1)it}}{1-e^{it}}\right| $ if $e^{it}\not=1$}.$$
we obtain for  $t\notin 2\pi\mathbb Z$ that
\begin{equation}\label{sin-esti}
|b_0+b_1+\dots+b_m|\leq \frac{2}{|1-e^{it}|}=:M.
\end{equation}
Hence the series  $\sum_{n=0}^\infty a_nb_n$ is  convergent.

\newpage

  \begin{figure}[h!]
 
  \scalebox{0.6} 
  {\includegraphics{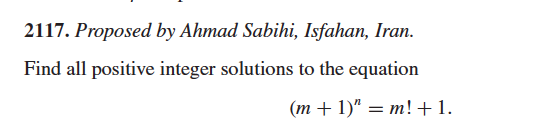}} 
\end{figure}

\centerline{\bf Solution to problem 2117  in Math. Mag. 94 (2021), p. 150}  \medskip

\centerline{Raymond Mortini, Rudolf Rupp and Amol Sasane}
                   \medskip

\centerline{- - - - - - - - - - - - - - - - - - - - - - - - - - - - - - - - - - - - - - - - - - - - - - - - - - - - - -}
  
  \medskip

There are only the three solutions $(n,m)\in\{(1,1), (1,2), (2,4)\}$.

It is easy to check that these are solutions. %

Now suppose that $n\geq m\geq 2$. Then $(n,m)$ cannot be a solution
since
$$\mbox{$(m+1)^n\geq (m+1)^m>m^m>m!$, so $(m+1)^n >m!+1$}.$$
Now, if $2=n<m$, then
$$(m+1)^2=m!+1\iff m+2=(m-1)!$$
which is obviously only satisfied for $m=4$. 

Next let 
 $2<n<m$. Then we see that if $(n,m)$ is a solution to $(m+1)^n=m!+1$,  then $m$ must be even.  
 (Actually, by  Wilson's theorem,  $m+1$ divides $m!+1$ if and only $m+1$ is prime; but we do not need this result).
 In particular, $m\geq 4$.
 Note that the equation $(m+1)^n-1=m!$ under discussion is equivalent to
\begin{equation}\label{summi}
\sum_{k=0}^{n-1} (m+1)^k=(m-1)!.
 \end{equation}
 
 \underline{$1^\circ$} ~ $m=4$. Then, due to  \zit{summi}, $6=3! =1+5+\cdots$ implying that $n=2$. A contradiction to the assumption $2<n<m$..
 
  \underline{$2^\circ$} ~ $m\geq 6$.  Then $2< m/2<m-1$. Hence  the integer $m/2$  divides  $(m-1)!$. Since $m/2>2$,
  additionally the number $2$ divides $(m-1)!$. Thus  $m=2\cdot (m/2)$ divides $(m-1)!$.

Now, \zit{summi} yields  $n\equiv 0$, $\mod m$.  That is, $m$ divides $n$ and so $m\leq n$.  This is again a contradiction to the assumption $2<n<m$.

\newpage

  \begin{figure}[h!]
 
  \scalebox{0.5} 
  {\includegraphics{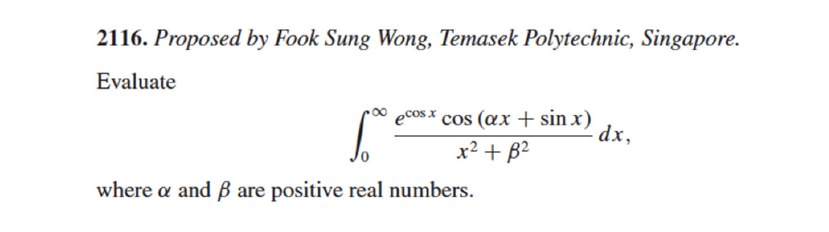}} 
  
\end{figure}

\centerline{\bf Solution to problem 2116  Math. Mag.  94 (2021), 150}  \medskip

\centerline{Raymond Mortini, Rudolf Rupp }
                   \medskip

\centerline{- - - - - - - - - - - - - - - - - - - - - - - - - - - - - - - - - - - - - - - - - - - - - - - - - - - - - -}
  
  \medskip

We show that

  \begin{figure}[h!]
  \scalebox{0.5} 
  {\includegraphics{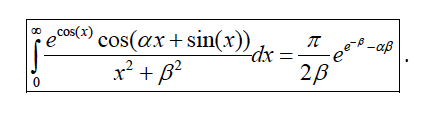}} 
\end{figure}

\vspace{-1cm}

 \begin{figure}[h!]
 
  \scalebox{0.5} 
  {\includegraphics{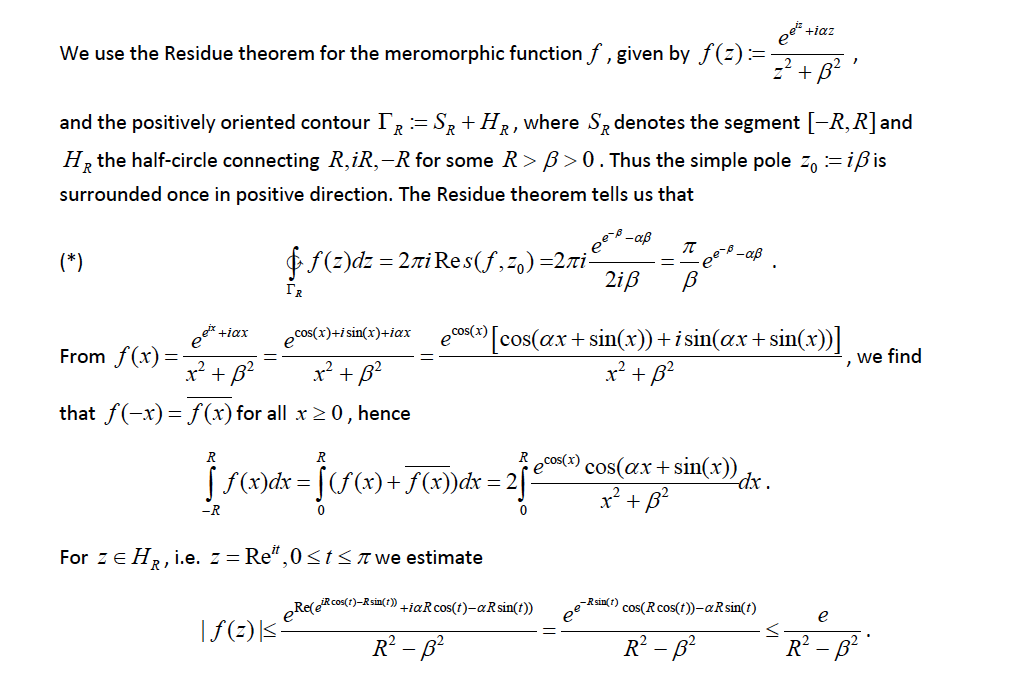}} 

\end{figure}

 
 \newpage
 
 \begin{figure}[h!]
 
  \scalebox{0.48} 
  {\includegraphics{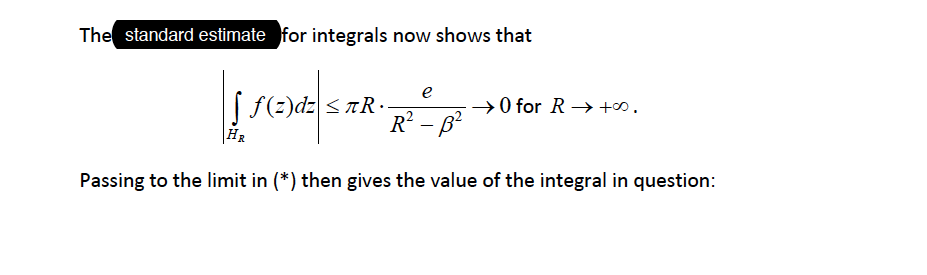}} 
\end{figure}

\newpage

\pagecolor{yellow}

 \begin{figure}[h!]
 
  \scalebox{0.6} 
  {\includegraphics{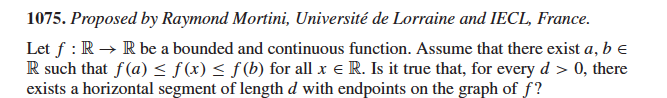}} 

\end{figure}

\centerline{\bf Solution to Quicky 1075  in Math. Mag. 90 (2017), 384}  \medskip

\centerline{Raymond Mortini}

                   \medskip

\centerline{- - - - - - - - - - - - - - - - - - - - - - - - - - - - - - - - - - - - - - - - - - - - - - - - - - - - - -}
  
  \medskip
 
  \begin{figure}[h]
  \scalebox{.40} 
  {\includegraphics{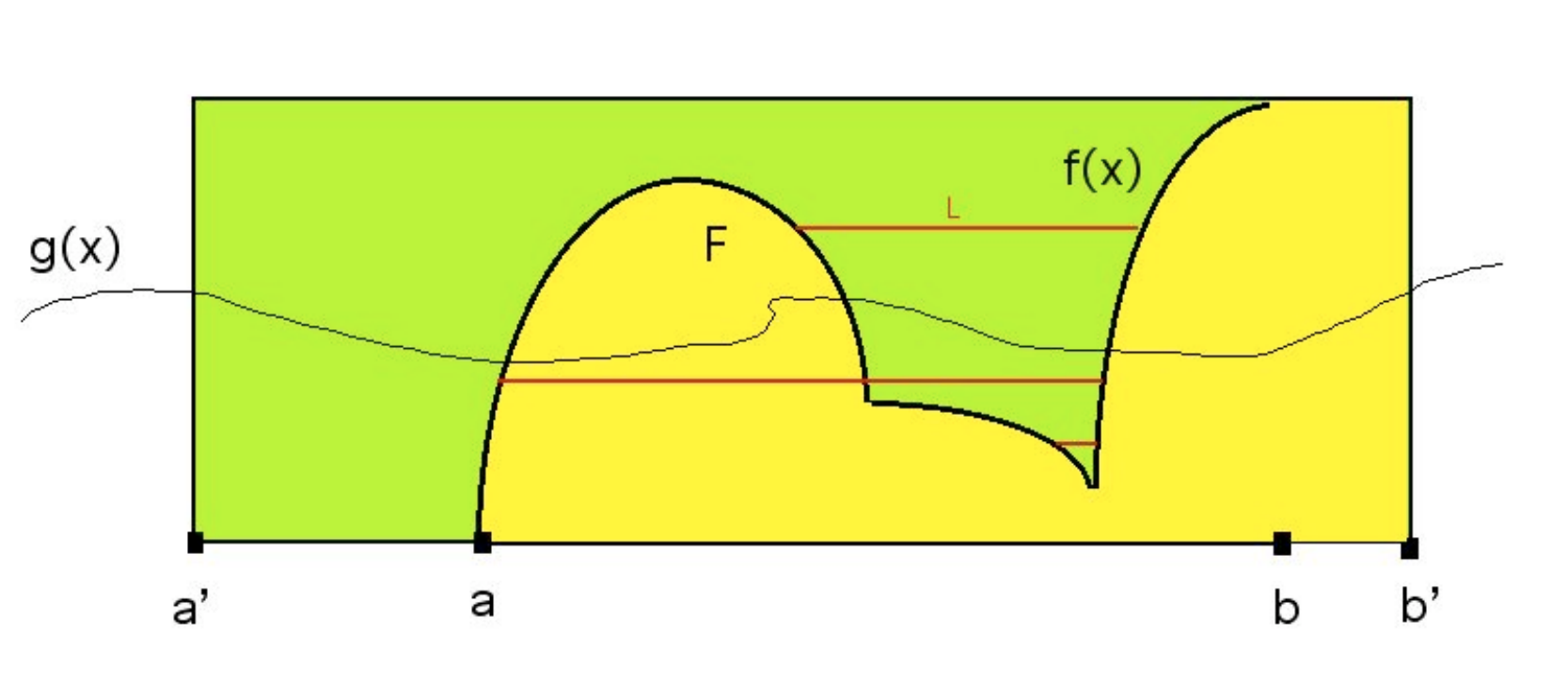}} 
\caption{\label{} {Intersecting curves}}
\end{figure}

  Yes. We have to show that for every $d\in \R$, 
  there is $x_0\in \R$ such that $f(x_0)=f(x_0-d)$.\\
  
  {\it 1) Non-elementary geometric approach.~} 

   Put $g(x):=f(x-d)$ and choose $a,b\in \R$ such that $m=f(a)$,  $M=f(b)$ and  $m<f(x)<M$
  for $x\in \,]a,b[$. We may assume that $a<b$. 
  Of course, $m\leq g(x)\leq M$.  Let $a'<a$ and $b>b'$. Then $x\mapsto (x, g(x)),  a'\leq x\leq b'$ is a  curve
  in the rectangle $R:=[a',b']\times [m,M]$ starting at the left  of the graph $F:=\{(x,f(x)): a\leq x\leq b\}$ of  $f$
  and ending at the right (here we need that the Jordan arc $F$ is a cross-cut of $R$). Thus this curve meets the graph: that is there is $a'\leq x_0\leq b'$ such that
  $(x_0,g(x_0))\in F$. Hence,  there is $a\leq x_1\leq b$ such that $(x_0,g(x_0))=(x_1,f(x_1))$. Consequently, $x_0=x_1$
  and so $f(x_0)=f(x_0-d)$.\\
  
  {\it 2) Analytic approach.~} Let $H:=f-g$. Then $H(a)=m-g(a)\leq 0$ and $H(b)=M-g(b)\geq 0$. If $g(a)=m$ or $g(b)=M$, then we are done. So we may assume that $H(a)<0$ and $H(b)>0$. Hence, by the intermediate value theorem, there is $x_0\in\, ]a,b[$
  such that $H(x_0)=0$. We conclude that $f(x_0)=g(x_0)=f(x_0-d)$.\\
 
 \hrule\vspace{0,2cm}
 
  {\small Let us point out that the assertion does not hold whenever merely $ \inf_\R f$ and $\sup_\R f$ exist: just look 
    at $f(x)=\arctan x$.  Motivation for the problem came from the paper:
 Peter   Horak, Partitioning $\R^n$  into connected components.
Am. Math. Mon. 122, No. 3, 280-283 (2015),
where periodic functions were considered.}

 \newpage
 
 \begin{figure}[h!]
 
 \scalebox{0.5} 
  {\includegraphics{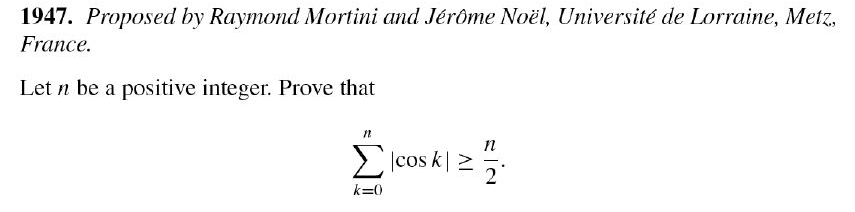}} 

\end{figure}

\centerline {\bf Solution to problem 1947 Math. Mag. 87 (2014), 230}  \medskip
 
\centerline{Raymond Mortini, J\'er\^ome No\"el }
                   \medskip

\centerline{- - - - - - - - - - - - - - - - - - - - - - - - - - - - - - - - - - - - - - - - - - - - - - - - - - - - - -}
  
  \medskip

 $$\sum_{k=0}^n |\cos k| \geq1+\sum_{k=1}^n (\cos k)^2=1+\sum_{k=1}^n \frac{\cos(2 k)+1}{2}= 
1+\frac{n}{2} + \frac{1}{2}{\rm Re}\left(\sum_{k=1}^n e^{2ik}\right).
$$
Now 
$$\sum_{k=1}^n e^{2ik}=e^{2i}\frac{1-e^{2in}}{1-e^{2i}}= e^{2i}\frac{e^{in}}{e^i}
\frac{\sin n}{\sin 1}=e^{i(n+1)}\frac{\sin (n)}{\sin 1}.$$
Hence
$$\sum_{k=0}^n |\cos k| \geq1+\frac{n}{2}+ \frac{\cos (n+1)\sin n}{2\sin 1}\geq
1+\frac{n}{2}-\frac{1}{2\sin 1}
$$
$$=\frac{n}{2}+ \underbrace{\left(1-\frac{1}{2\sin 1}\right)}_{ >0}\geq \frac{n}{2},$$
because $2\sin 1>1$ (note that $\pi/4<1<\pi/3$ implies $1<\sqrt 2 <2\sin 1 < \sqrt 3$).

Let us remark that in the very first step it was important to begin the sum at $k=1$
in order to have the summand 1. Otherwise we would have obtained
 $$\sum_{k=0}^n |\cos k|\geq \frac{n+1}{2} + \frac{\cos n \sin (n+1)}{2\sin 1}
\geq \frac{n+1}{2}- \frac{1}{2\sin 1}$$
$$=\frac{n}{2}+ \frac{1}{2}\left(1- \frac{1}{\sin 1}\right),
$$
an estimate that is less than $n/2$.

 \newpage

\nopagecolor

 \begin{figure}[h!]
 
 \scalebox{0.4} 
  {\includegraphics{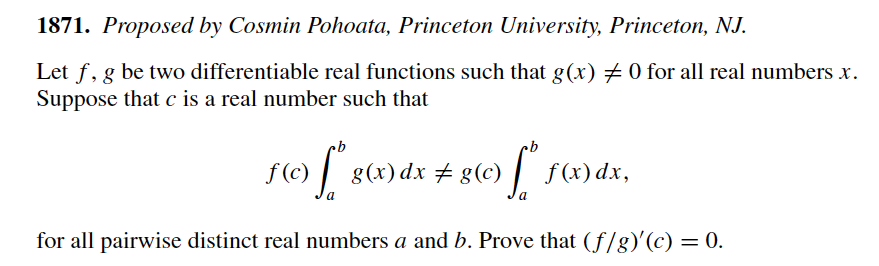}} 

\end{figure}

\centerline {\bf Solution to  problem 1871 Math. Mag. 84 (2011), p. 229}  \medskip

\centerline{Raymond Mortini}

                   \medskip

\centerline{- - - - - - - - - - - - - - - - - - - - - - - - - - - - - - - - - - - - - - - - - - - - - - - - - - - - - -}
  
  \medskip

The solution is a based on the following Lemma:

\begin{lemma}\label{dia}
 Let $F$ be  a continuous, real-valued function on $\R\times \R$. Suppose that
$F$ is not zero outside the diagonal  $D$ and not constant $0$  on $D$. Then either $F\geq 0$ or $F\leq 0$ everywhere.
\end{lemma}

\begin{proof}
Let $P^+=\{(x,y)\in \R^2; x<y\}$ and $P^-=\{(x,y)\in \R^2; x>y\}$.\\
Case 1:   if $F(x_0,y_0)<0$ and $F(x_1,y_1)>0$ for some 
points $P_0=(x_0,y_0)$ and $P_1=(x_1,y_1)$ in $P^+$, then $F$ must have a zero on the
segment  $S$ joining $P_0$ and $P_1$ in $P^+$ (since the image of $S$ under $F$ is an interval).

Case 2: if $F(P_0)<0$ and $F(P_1)>0$ for some $P_0\in P^+$ and $P_1\in P^-$, then we may choose
an arc  $A$ (piecewise parallel to the axis) such that $F\not=0$ on $A\inter D$, which is  a singleton.
By the   the intermediate value theorem, there is a zero of $F$ on the arc $A$, but outside $D$.

Case 3: if $F(Q_0)<0$ and $F(Q_1)>0$ for some $Q_0, Q_1\in D$, then there are 
$P_0\in P^+$ and $P_1\in P^-$ such that $F(P_0)<0$ and $F(P_1)>0$. Hence we  are in the second case.

Thus, all cases yield a contradiction to the assumption. Hence, in the image space, $0$
is a global  extremum.
\end{proof}

{\bf Solution to the problem}
Without loss of generality, we may assume that $g>0$.
Let $$H(a,b)=\begin{cases}\dis\frac{\int_a^b f(x)dx}{\int_a^b g(x) dx}& \text {if $a\not=b$}\vspace{10pt}\\

\dis\frac{f(a)}{g(a)}&\text {if $a=b$}.\end{cases}
$$

We claim  that
 $H$ is continuous on $\R\times \R$. In fact, it suffices to prove continuity at the diagonal.
 So let $(a_0,a_0)\in D$. Then, for $(a,b)\in\R^2\setminus D$, there is $\xi\in \;]a,b[$ such
 that $\int_a^b f(x)dx/(b-a)=f(\xi)\to f(a_0)$ if $(a,b)\to (a_0,a_0)$.
 Thus $\lim H(a,b)=H(a_0)$.

  By assumption,
 $H(a,b)\not= f(c)/g(c)$ whenever $(a,b)$ is outside the diagonal in $\R^2$.
 
 Case 1: $H\equiv f(c)/g(c)$ on the diagonal $D$. Then the function $x\mapsto f(x)/g(x)$ has
 derivative $0$ everywhere, and so satisfies the assertion of the problem.
 
 Case 2: $H$ not constant $f(c)/g(c)$ on $D$. Then, by Lemma \ref{dia} applied
 to $F=H-f(c)/g(c)$, we see that $H\geq  f(c)/g(c)$
 on $\R\times \R$ or $H\leq  f(c)/g(c)$ on $\R\times \R$.  In particular, $c$ is an extrema of 
 the function $x\mapsto f(x)/g(x)$ and so the differentiability of $f/g$ implies that
 $(f/g)'(c)=0$.
  
\newpage

\begin{figure}[h!]
 
 \scalebox{0.5} 
  {\includegraphics{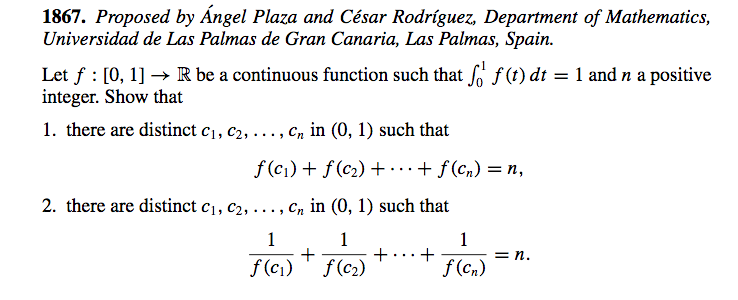}} 

\end{figure}

\centerline {\bf Solution to  problem 1867 Math. Mag. 84 (2011), p. 150}  \medskip

\centerline{Raymond Mortini}

\centerline{- - - - - - - - - - - - - - - - - - - - - - - - - - - - - - - - - - - - - - - - - - - - - - - - - - - - - -}
  
  \medskip

{\small If $f\equiv 1$, then the assertions are  trivially true (just take any $n$ points in $]0,1[$~~).
If $f\not\equiv 1$, then there exist points at which $f$ is strictly less than 1 and
points where $f$ is strictly bigger than one (note that this is the only occasion where we have used
the hypothesis that $\int_0^1f(t)dt=1$).
 Hence, due to  intermediate value theorem,
there is at least one point at which $f$ takes the value 1.  In particular, if $h=f$ or $h=1/f$,
and noticing that the image of $[0,1]$ under $f$ is an interval containing the point $1$ in its interior, 
 there exist $b\in [0,1]$ with $M:=h(b)>1$ and a sequence $(a_i)$ with  $h(a_i)<1$
 and $\lim h(a_i)=1$. By compactness, we may assume that $(a_i)$ is converging to some
 $a\in [0,1]$.  Hence $h(a)=1$ and 
 $$\mbox{$m(\delta):=\min\{ h(x): x\in [a-\delta, a+\delta]\inter [0,1]\}  \to1$ if $\delta\to 0$. }$$
 For later purposes,  we note  that  $m(\delta)<1$.  
   Choose $\delta$ so small that 
$$(n-1)(1-m(\delta))\leq M-1.$$
Then 
$$n-M=(n-1)-(M-1)\leq (n-1)m(\delta).$$
Now choose $n-1$ distinct points $x_1,\dots,x_{n-1}$ in $ [a-\delta, a+\delta]\;\inter \;]0,1[$ such that
$$ m(\delta) <h(x_j)< 1.$$
Then $A:=\sum_{j=1}^{n-1} h(x_j)$ satisfies
$$(n-1)m(\delta)\leq A < n-1.$$
Thus  $n-M\leq A$ and so $1< n-A\leq M$. Again, by the  intermediate value theorem,
there is $x_n\in ]0,1[$ such that $h(x_n)= n-A$. Hence
$$\sum_{j=1}^n h(x_j)=n.$$
Note that $x_n\notin\{x_1,\dots,x_{n-1}\}$.
\bigskip

{\bf Alternate proof concerning the existence of the $c_j$}\medskip

Let $F(x)=\int_0^x f(t)dt$ be the primitive of $f$ vanishing at the origin. 
Let $x_j=j/n, j=0,1,\dots, n$. Then, by the mean-value theorem of  differential calculus, 
there exist $c_j\in\; ]x_{j-1},x_j[\;\ss\; ]0,1[$ such that
$$ 1= F(1)-F(0)=\sum_{j=1}^n (F(x_j)-F(x_{j-1}))=
\sum_{j=1}^n F'(c_j) (x_j-x_{j-1})= \frac{1}{n}\sum_{j=1}^n f(c_j).$$
}

\newpage

\begin{figure}[h!]
 
 \scalebox{0.5} 
  {\includegraphics{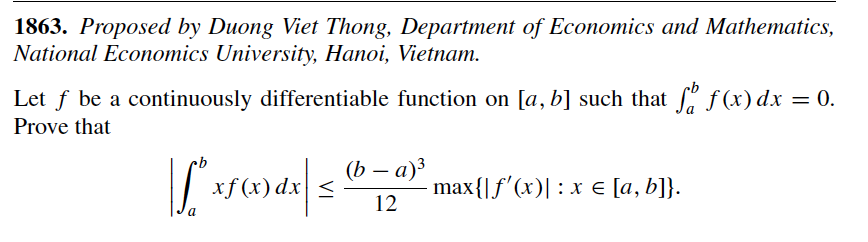}} 

\end{figure}

\centerline {\bf Solution to  problem 1863,   Math. Mag. 84 (2011), 64.}

\centerline{Raymond Mortini}

\centerline{- - - - - - - - - - - - - - - - - - - - - - - - - - - - - - - - - - - - - - - - - - - - - - - - - - - - - -}
  
  \medskip

We use Carath\'eodory's definition of differentiability:
A function $f:I\to \R$ is differentiable at a point $x_0\in I$,  $I\ss\R$ an  interval,
if there exists a function $g=g_{x_0}:I\to\R$ continuous at $x_0$ such that
 $$f(x)=f(x_0)+(x-x_0)g(x);$$
just define $g_{x_0}(x)=\begin{cases}\frac{f(x)-f(x_0)}{x-x_0}& \text{ if $x\not =x_0$}\\
 f '(x_0) & \text{ if $x\not =x_0$}\end{cases}.$
 
 Now if $f\in C^1[a,b]$, then $g_{x_0}$ is continuous and, by Rolle's theorem, 
  $g_{x_0}(x)=f ' (\xi)$ for some $\xi \in \;]a,b[$, 
 $\xi$ depending on $ x_0$ and $x$.
 Hence $$\sup_{a\leq s\leq b}|g_{x_0}(s)| \leq \max_{a\leq t\leq b} |f '(t)|=:M.$$
 
 \medskip
Let $c=(a+b)/2$. Then, using the hypotheses that $\int_a^b f(x)dx =0$ and
the fact that $\int_a^b (x-c)dx =0$ we obtain the following equalities:

$$J:=\int_a^b xf(x)dx= \int_a^b (x-c)f(x)dx=$$ 
$$\int_a^b(x-c)(f(x)-f(c)) dx=
 \int_a^b (x-c)^2 g_c(x)dx.$$

Thus

$$|J|\leq \int_a^b(x-c)^2 Mdx =\frac{1}{3}\left[{(x-c)^3}\right]^b_a M=$$
$$\frac{2}{3}\left( \frac{b-a}{2}\right)^3M
=\frac{1}{12}(b-a)^3 M.$$
 
If $f(x)=x$ and $a=-1, b=1$ then $\int_{-1}^1 f(x)dx=0$  and $\int_{-1}^1 xf(x)dx= 1/3=
(b-a)^3/12$.

\newpage

\begin{figure}[h!]
 
 \scalebox{0.5} 
  {\includegraphics{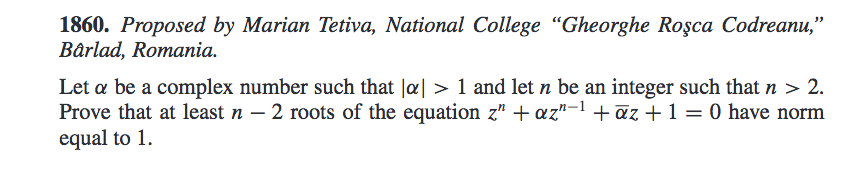}} 

\end{figure}


\centerline {\bf Solution to  problem 1860,   Math. Mag. 83 (2010), 392}   \medskip

\centerline{Raymond Mortini}

\centerline{- - - - - - - - - - - - - - - - - - - - - - - - - - - - - - - - - - - - - - - - - - - - - - - - - - - - - -}
  
  \medskip

We use the Schwarz-Pick Lemma telling us that holomorphic selfmaps of the unit
disk are contractions with respect to the (pseudo)-hyperbolic metric $\rho$
and that $\rho(f(z), f(w))=\rho(z,w)$ for some pair $(z,w)\in \D^2$, $z\not=w$ implies
that $f$ is  a conformal selfmap of $\D$ (hence of the form $e^{i\theta}\frac{b-z}{1-\ov b z}$)
and so a (pseudo)-hyperbolic isometry.

Note that $z^n+\alpha z^{n-1}+\ov a z+1=0$ for some $z\in \D$ if and only if 
$z^{n-1}=- \frac{\ov a z+1}{a+z}$. Now suppose that there are two solutions $z,w$
in $\D$.  Let $f(z)= - \frac{\ov a z+1}{a+z}$. Then
$$\rho(z,w)=\rho(f(z), f(w))=\rho(z^{n-1}, w^{n-1}).$$
But this would imply that $z^{n-1}$ is a bijection of $\D$ onto itself; a contradiction
since $n\geq 3$.  

Thus the equation $z^n+\alpha z^{n-1}+\ov \alpha z+1=0$ has at most one solution in $\D$.
Since $z$ is a solution if and only if $\frac {1}{~~\ov z~~}$ is a solution, we see
that this polynomial of degree $n$ must have at least $n-2$ solutions (multiplicities counting) on the unit circle. 

Next we note that $u\in \T$ s a solution of modulus one of $z^n+\alpha z^{n-1}+\ov \alpha z+1=0$
if and only if $u$ is a fixed point on $\T$ of the selfmap $\varphi(z)=f^{-1}(z^{n-1})$ of $\D$.
Since  the derivative of $\varphi$ does not vanish
at boundary fixed points, we conclude that  there are at least $n-2$ 
{\sl distinct} solutions of unit modulus.

\newpage

\gr{\huge\section{College Math. J.}}
\bigskip

\centerline{\gr{\copyright Mathematical Association of America, 2025.  }}


\bigskip

 \begin{figure}[h!]
  \scalebox{0.45} 
  {\includegraphics{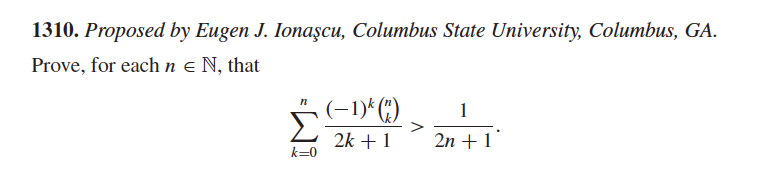}} 
\end{figure}

\centerline{\bf Solution to problem 1310 College Math. J. 56 (2025), 327}

\bigskip
 

 \centerline {Raymond Mortini, Rudolf Rupp} \medskip

\centerline{- - - - - - - - - - - - - - - - - - - - - - - - - - - - - - - - - - - - - - - - - - - - - - - - - - - - - -}
  
  \medskip

\newpage

 \begin{figure}[h!]
  \scalebox{0.45} 
  {\includegraphics{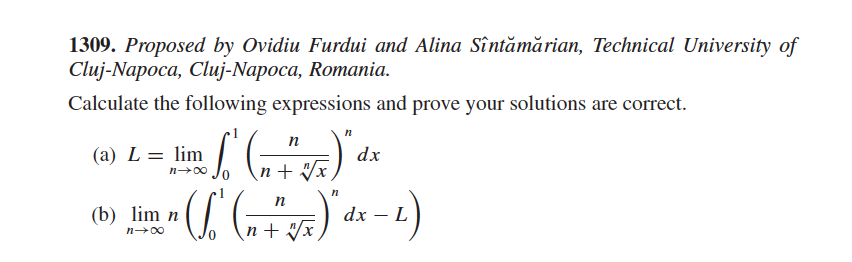}} 
\end{figure}

\centerline{\bf Solution to problem 1309 College Math. J. 56 (2025), 327}

\bigskip
 

 \centerline {Raymond Mortini, Rudolf Rupp} \medskip

\centerline{- - - - - - - - - - - - - - - - - - - - - - - - - - - - - - - - - - - - - - - - - - - - - - - - - - - - - -}
  
  \medskip

\newpage

 \begin{figure}[h!]
  \scalebox{0.45} 
  {\includegraphics{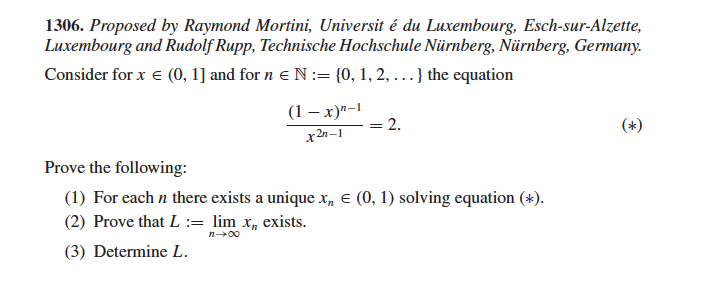}} 
\end{figure}

\centerline{\bf Solution to problem 1306 College Math. J. 56 (2025), 326}

\bigskip
 
\pagecolor{yellow}

 \centerline {Raymond Mortini, Rudolf Rupp} \medskip

\centerline{- - - - - - - - - - - - - - - - - - - - - - - - - - - - - - - - - - - - - - - - - - - - - - - - - - - - - -}
  
  \medskip
  
  \newpage
  
  \nopagecolor

   \begin{figure}[h!]
  \scalebox{0.45} 
  {\includegraphics{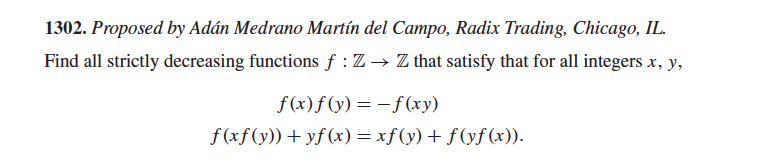}} 
\end{figure}

\centerline{\bf Solution to problem 1302 College Math. J. 56 (2025), 244}

\bigskip

 \centerline {Raymond Mortini, Rudolf Rupp} \medskip

\centerline{- - - - - - - - - - - - - - - - - - - - - - - - - - - - - - - - - - - - - - - - - - - - - - - - - - - - - -}
  
  \medskip

\newpage

\newpage

   \begin{figure}[h!]
  \scalebox{0.55} 
  {\includegraphics{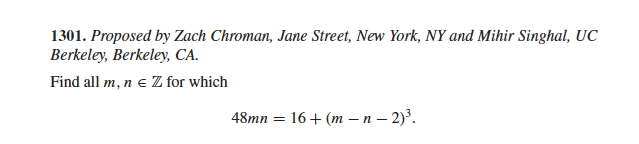}} 
\end{figure}

\centerline{\bf Solution to problem 1301 College Math. J. 56 (2025), 244}

\bigskip

 \centerline {Raymond Mortini, Rudolf Rupp} \medskip

\centerline{- - - - - - - - - - - - - - - - - - - - - - - - - - - - - - - - - - - - - - - - - - - - - - - - - - - - - -}
  
  \medskip

\newpage

   \begin{figure}[h!]
  \scalebox{0.45} 
  {\includegraphics{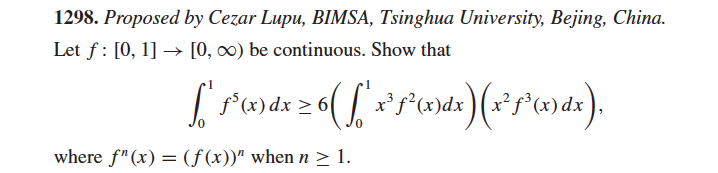}} 
\end{figure}

\centerline{\bf Solution to problem 1298 College Math. J. 56 (2025), 159}

\bigskip
 

 \centerline {Raymond Mortini} \medskip

\centerline{- - - - - - - - - - - - - - - - - - - - - - - - - - - - - - - - - - - - - - - - - - - - - - - - - - - - - -}
  
  \medskip

\newpage

   \begin{figure}[h!]
  \scalebox{0.45} 
  {\includegraphics{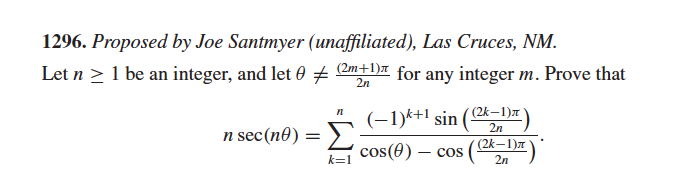}} 
\end{figure}

\centerline{\bf Solution to problem 1296 College Math. J. 56 (2025), 159}

\bigskip
 

 \centerline {Raymond Mortini, Rudolf Rupp } \medskip

\centerline{- - - - - - - - - - - - - - - - - - - - - - - - - - - - - - - - - - - - - - - - - - - - - - - - - - - - - -}
  
  \medskip

  This is straightforward. 
  Let $T_n(x)$ be the Chebychev polynomial (of the fist kind). This has degree $n$ and is defined as $T_n(\cos t)=\cos(nt)$.
  Its zeros are $x_k:=\cos(\frac{(2k-1)\pi}{2n})$, $k=1,\dots, n$. They are all simple. Thus, the partial fraction decomposition of
   $n/T_n(x)$ is given by
   $$\frac{n}{T_n(x)}=\sum_{k=1}^n \frac{c_k}{x-x_k},$$
   where the residues/coefficients $c_k$ are given by
\begin{eqnarray*}
c_k&=&\lim_{x\to c_k} n\frac{x-x_k}{T_n(x)}\buildrel=_{x=\cos t}^{}\lim_{t\to (2k-1)\pi/2n} n\frac{\cos t-x_k}{T_n(\cos t)}=
\lim_{t\to (2k-1)\pi/2n} n\; \frac{\cos t-x_k}{\cos(n t)}\\
&\buildrel=_{}^{{\rm L'Hospital}}&\frac{-n\sin \left(\frac{(2k-1)\pi}{2n}\right)}{-n\sin\left(n \frac{(2k-1)\pi}{2n}\right))}=
\frac{-\sin \left(\frac{(2k-1)\pi}{2n}\right)}{\sin(k\pi-\pi/2)}=-(-1)^{k} \sin \left(\frac{(2k-1)\pi}{2n}\right).
\end{eqnarray*}  
Hence
$$\frac{n}{\cos(nt)}= \sum_{k=1}^n \frac{c_k}{\cos t-x_k}=\sum_{k=1}^n \frac{(-1)^{k+1}\sin \left(\frac{(2k-1)\pi}{2n}\right)}
{\cos t- \cos(\frac{(2k-1)\pi}{2n})}.
$$


\newpage

 \begin{figure}[h!]
  \scalebox{0.45} 
  {\includegraphics{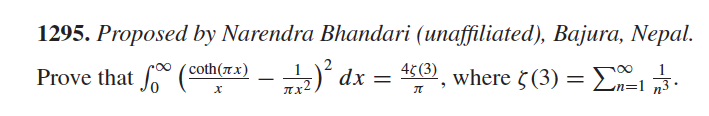}} 
\end{figure}

\centerline{\bf Solution to problem 1295 College Math. J. 56 (2025), 69 }

\bigskip

 \centerline {Raymond Mortini, Rudolf Rupp } \medskip

\centerline{- - - - - - - - - - - - - - - - - - - - - - - - - - - - - - - - - - - - - - - - - - - - - - - - - - - - - -}
  
  \medskip

  {\bf First method.}
  The partial fraction decomposition of $\coth (\pi x)$ is given by
  $$\coth (\pi x)=  \frac{x}{\pi}\; \sum_{k\in \Z} \frac{1}{k^2+x^2}.
    $$
  Hence
  $$\frac{\coth(\pi x)}{x}=\frac{2}{\pi}\sum_{k=1}^\infty \frac{1}{k^2+x^2}+ \frac{1}{\pi x^2}.
  $$
  Consequently, since $\int\sum\sum=\sum\sum\int$ (note that all terms are positive)
  \begin{eqnarray*}
  I:=\int_0^\infty \left( \frac{\coth(\pi x)}{x}-\frac{1}{\pi x^2}\right)^2 \;dx&=& \frac{4}{\pi^2}\int_0^\infty \left(\sum_{k=1}^\infty\frac{1}{k^2+x^2}\right)^2\;dx\\
  &=&\frac{4}{\pi^2}\int_0^\infty \sum_{n=1}^\infty \sum_{k=1}^\infty \frac{1}{k^2+x^2}\;\frac{1}{n^2+x^2}\;dx\\
  &=& \frac{4}{\pi^2}\sum_{n=1}^\infty \sum_{k=1}^\infty\int_0^\infty  \frac{1}{k^2+x^2}\;\frac{1}{n^2+x^2}\;dx.
\end{eqnarray*}
  
  Now, for $n\not =k$,  
  $$ \frac{1}{k^2+x^2}\;\frac{1}{n^2+x^2} =\frac{1}{n^2-k^2}\left(\frac{1}{k^2+x^2} -\frac{1}{n^2+x^2}\right).
  $$
  Using that $\dis\int_0^\infty \frac{1}{k^2+x^2}\;dx= \frac{\pi}{2k}$ and $\dis \int_0^\infty \frac{1}{(n^2+x^2)^2}\;dx=\frac{\pi}{4n^3}$, we obtain
  \begin{eqnarray*}
  I&=&\frac{4}{\pi^2}\;\sum_{n=1}^\infty  \left(\sum_{k=1\atop k\not=n}^\infty  \frac{\pi}{2} \frac{1}{n^2-k^2} 
   \left(\frac{1}{k}-\frac{1}{n}\right) +\frac{\pi}{4n^3}\right)\\
   &=&\frac{2}{\pi}\sum_{n=1}^\infty\frac{1}{n}\; \left(\sum_{k=1\atop k\not=n}^\infty \frac{1}{(n+k) k}+\frac{1}{2n^2}\right)\\
   &=& \frac{2}{\pi}\sum_{n=1}^\infty\sum_{k=1}^\infty \frac{1}{(k+n)kn}\\
   &=&\frac{2}{\pi}\sum_{n=1}^\infty\frac{1}{n^2}\; \sum_{k=1}^\infty\left(\frac{1}{k}-\frac{1}{n+k}\right)\\
   &\buildrel=_{}^{(*)}&\frac{2}{\pi}\sum_{n=1}^\infty\frac{H_n}{n^2},
\end{eqnarray*}
  where $H_n=1+(1/2)+\cdots+(1/n)$ is the $n$-th harmonic  number. It is now well-known that 
  $$\sum_{n=1}^\infty\frac{H_n}{n^2}=2 \zeta(3).$$
  (see for instance \cite[p. 206]{Fu}). Hence $I=\frac{4}{\pi}\zeta(3).$\medskip
  
  {\it Remark}  Note that $(*)$ can be verified as follows: put $a_k:=1/k$, $H_n=\sum_{k=1}^n a_k$  and let $N>n$. We use that
  $H_n-\log n\to \gamma$, where $\gamma$ is the Euler-Mascheroni constant. Then
  \begin{eqnarray*}
  \sum_{k=1}^N (a_k-a_{n+k})&=& H_N-\sum_{k=n+1}^{n+N}a_k =\left(H_N-\sum_{k=1}^{n+N} a_k \right) +H_n\\
  &=& H_N-H_{n+N}+H_n=(H_N-\log N)- (H_{n+N}-\log(n+N))+\log\frac{N}{N+n}+H_n\\
  &\buildrel\longrightarrow_{N\to\infty}^{}& \gamma-\gamma + 0+ H_n=H_n.
 \end{eqnarray*}
 \bigskip
  
  {\bf Second method}  We use the residue theorem.  Let
  $$f(z):=\left(\frac{e^{2\pi z}+1}{(e^{2\pi z}-1)z}-\frac{1}{\pi z^2}\right)^2.
  $$
  Note that $f$ is meromorphic in $\C$ with double poles in $z_k=i k$ for $k=1,2,3,\dots$
  Moreover, as above,  
  $$f(z)= \frac{4}{\pi^2}\ \left(\sum_{k=1}^\infty\frac{1}{k^2+z^2}\right)^2.$$
This shows that $0$ is a removable singularity of $f$. Finally, $f$ is continuous on $\R$ and 
 $$I=\frac{1}{2}\int_{-\infty}^\infty  f(x) dx.
  $$
  Consider the rectangle $R_n:=[-n,n]\times [0, n+\frac{1}{2}]$. Then
 \begin{eqnarray*}
\int_{\partial R_n} f(z) dz&=&\int_{-n}^n f(x) dx +\int_0^{n+\frac{1}{2} }f(n+it) idt- \int_{-n}^n f(t+i(n+\frac{1}{2})) dt- 
  \int_0^{n+\frac{1}{2}} f(-n+it) i dt \\
  &=&I_1(n)+I_2(n)+I_3(n)+I_4(n).
  \end{eqnarray*}
  We show that $\lim_{n\to\infty}I_j(n)=0$ for $j=1,2,3$. 
  In view of the residue theorem, we then have
  $$I=i\pi \sum_{k=1}^\infty {\rm Res}\,(f, z_k).$$
  
  Here are the estimates:
  \begin{eqnarray*}
  |I_2(n)|&\leq &(n+\frac{1}{2}) \left(\frac{e^{2\pi n}+1}{e^{2\pi n}-1} \frac{1}{n} +\frac{1}{n^2}\right)^2\leq
  2n \left(\frac{3}{n}+\frac{1}{n^2}\right)^2\leq \frac{32}{n}\\
|I_4(n)|&\leq &(n+\frac{1}{2}) \left(\frac{1+e^{-2\pi n}}{1-e^{-2\pi n}} \frac{1}{n} +\frac{1}{n^2}\right)^2\leq
2n\left(\frac{4}{n} +\frac{1}{n^2}\right)^2\leq \frac{50}{n}.
\end{eqnarray*}
  Moreover, by using that $|\tanh x|\leq 1$,
  \begin{eqnarray*}
  |I_3(n)|&=&\left|\int_{-n}^n \left(\frac{e^{2\pi (t+i(n+\frac{1}{2})}+1}{(e^{2\pi (t+i(n+\frac{1}{2}))}-1) (t+i(n+\frac{1}{2})}-
  \frac{1}{\pi (t+i(n+\frac{1}{2})^2}\right)^2 dt\right|\\
  &=&\left|\int_{-n}^n \left(\frac{e^{2\pi t}-1}{e^{2\pi t}+1}\;\frac{1}{t+i(n+\frac{1}{2})}-\frac{1}{\pi (t+i(n+\frac{1}{2}))^2}\right)^2 dt \right|\\
  &\leq& 2n \;\left( \frac{1}{n}+\frac{1}{n^2}\right)^2\leq \frac{8}{n}.
\end{eqnarray*}
  Finally, we show that  $\dis {\rm Res}\,(f, z_k)=-\frac{4i}{\pi^2 k^3}$. Fix $k\in \N^*$ and consider the variable $w:=z-ki$. 
  The auxiliary function $g$ is defined as follows:
\begin{eqnarray*}
g(w)=f(w+ki)&=& \frac{e^{2\pi w}+1}{e^{2\pi w}-1}\frac{1}{w+ki}-\frac{1}{\pi(w+ki)^2}\\
&=& \frac{\coth (\pi w)}{w+ki}-\frac{1}{\pi(w+ki)^2}.
\end{eqnarray*}
  Then $g$ is meromorphic in $\C$ and has a simple  pole at the origin with Laurent expansion (near the origin)
  $$g(w)=\frac{b_{-1}}{w} +b_0+ \Oh(w).$$
  Here 
  $$b_{-1}=\lim_{w\to 0} w g(w)=\lim_{w\to 0} \frac{1}{\pi}\cosh (\pi w) \frac{\pi w}{\sinh (\pi w)} \frac{1}{w+ki} +0= \frac{1}{\pi ki}
  =-\frac{i}{\pi k},
  $$
  and
  \begin{eqnarray*}
  b_0&=&\lim_{w\to 0}\left( \frac{\coth(\pi w)}{w+ki}-\frac{1}{\pi(w+ki)^2}-\frac{b_{-1}}{w}\right)\\
  &=&\lim_{w\to 0} \frac{w\cosh(\pi w)-\sinh(\pi w)b_{-1}(w+ki)}{\sinh(\pi w)(w+ki)w}+\frac{1}{\pi k^2}\\
  &=& \lim_{w\to 0}\frac{w (1+ \frac{\pi^2}{2} w^2+\cdots)-(\pi w+\frac{\pi^3}{6}w^3+\cdots )(w+ki)(-\frac{i}{\pi k})}{(\pi w+\frac{\pi^3}{6}w^3+\cdots)(w+ki)w}+\frac{1}{\pi k^2}\\
  &=&\lim_{w\to 0}\frac{1+\Oh(w^2)+\cdots -(\pi ki+\pi w)(-\frac{i}{\pi k})+\Oh(w^2)}{(\pi+\cdots)(w+ki)w}+\frac{1}{\pi k^2}\\
  &=& \frac{\frac{i}{k}}{\pi ki}+\frac{1}{\pi k^2}= \frac{2}{\pi k^2}.
\end{eqnarray*}

Now
\begin{eqnarray*}
(g(w))^2&=&\left(\frac{b_{-1}}{w} +b_0+ \oh(w)\right)^2=\frac{(b_{-1}+b_0w+\Oh(w^2))^2}{w^2}\\
&=&\frac{b_{-1}^2+2b_0b_{-1}w +\Oh(w^2)}{w^2}.
\end{eqnarray*}

Thus 
$${\rm Res}(f, ik)={\rm Res}(g, 0)= 2b_0b_{-1}=2 \left(-\frac{i}{\pi k}\right)\;\frac{2}{\pi k^2}=-\frac{4i}{\pi^2k^3}.$$
We conclude that

\begin{eqnarray*}
&I&=i\pi \sum_{k=1}^\infty {\rm Res}\,(f, z_k)=i\pi \sum_{k=1}^\infty \left(-\frac{4i}{\pi^2k^3}\right) 
=\frac{4}{\pi}\sum_{k=1}^\infty \frac{1}{k^3}=\frac{4 \zeta(3)}{\pi}.
\end{eqnarray*}


\newpage

 \begin{figure}[h!]
  \scalebox{0.45} 
  {\includegraphics{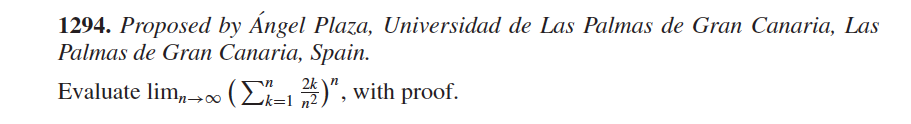}} 
\end{figure}

\centerline{\bf Solution to problem 1294 College Math. J. 56 (2025), 69 }

\bigskip

 \centerline {Raymond Mortini, Rudolf Rupp } \medskip

\centerline{- - - - - - - - - - - - - - - - - - - - - - - - - - - - - - - - - - - - - - - - - - - - - - - - - - - - - -}
  
  \medskip

Let $\dis L_n:= \left(\sum_{k=1}^n \frac{2k}{n^2}\right)^n$. We show that 
$$\ovalbox{$\dis \lim_{n\to\infty} L_n=e$.}$$
\medskip

In fact, using Gauss's classroom formula $\sum_{k=1}^n k =n(n+1)/2$, we immediately obtain  that 
\begin{eqnarray*}
L_n&=&\left(\frac{2}{n^2}\sum_{k=1}^n k\right)^n= \left(\frac{2}{n^2} \frac{n(n+1)}{2}\right)^n=\left(\frac{n+1}{n}\right)^n =
\left(1+\frac{1}{n}\right)^n.
\end{eqnarray*}
Hence $\lim L_n=e$.

\newpage

  \begin{figure}[h!]
  \scalebox{0.45} 
  {\includegraphics{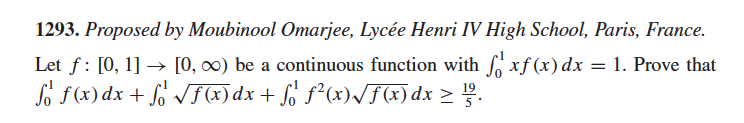}} 
\end{figure}

\centerline{\bf Solution to problem 1293 College Math. J. 56 (2025), 69 }

\bigskip

 \centerline {Raymond Mortini, Peter Pflug } \medskip

\centerline{- - - - - - - - - - - - - - - - - - - - - - - - - - - - - - - - - - - - - - - - - - - - - - - - - - - - - -}
  
  \medskip

We think that the statement was not carefully thought out, as already the third term is bigger than  $19/5$: in fact,
   by H\"older's inequality
  
  $$1=\int_0^1 xf(x) dx \leq \left(\int_0^1 x^{5/3}dx\right)^{3/5}\;\left( \int_0^1 f(x)^{5/2}dx\right)^{2/5}$$
Thus
$$ \int_0^1 f(x)^{5/2}dx\geq \frac{1}{\left((8/3)^{-3/5}\right)^{5/2}}=\left(\frac{8}{3}\right)^{3/2}\sim 4.3546\cdots>19/5= 3.80.
$$

\bigskip

\centerline{- - - - - - - - - - - - - - - - - - - - - - - - - - - - - - - - - - - - - - - - - - - - - - - - - - - - - -}
  \bigskip

Just for curiosity:

$$0\leq \int_0^1 (\sqrt f -x)^2 \sqrt f dx=\int_0^1( f\sqrt f +x^2 \sqrt f -2x  f )dx\leq \int_0^1 ( f\sqrt f +\sqrt f )dx  -2$$
Hence  $$\int f+\int \sqrt f + \int f\sqrt f \geq 3.$$

Since
$$1=\int xf dx \leq \left(\int x^3dx\right)^{1/3}\;\left( \int f^{3/2}dx\right)^{2/3}$$
we also have that 
$$\int f^{3/2} dx \geq \frac{1}{\left(4^{-1/3}\right)^{3/2}}= 2. $$

\newpage

   \begin{figure}[h!]
  \scalebox{0.45} 
  {\includegraphics{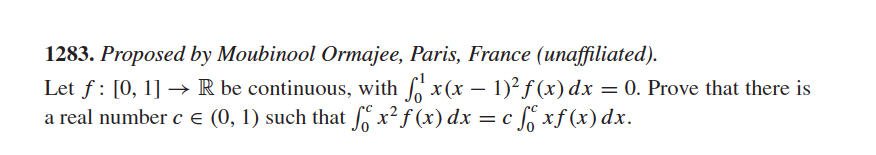}} 
\end{figure}

\centerline{\bf Solution to problem 1283 College Math. J. 55 (2024), 353 }

\bigskip
 

 \centerline {Raymond Mortini, Peter Pflug } \medskip

\centerline{- - - - - - - - - - - - - - - - - - - - - - - - - - - - - - - - - - - - - - - - - - - - - - - - - - - - - -}
  
  \medskip

We shall use the mean value theorem (MVT) for integrals.   Let
$$h(x):= x\int_0^x t f(t)dt-\int_0^x t^2f(t)dt.$$
Then this $c\in ]0,1[$ with $h(c)=0$ exists once we can show that
$I:=\int_0^1 h(x)dx=0$.  This is true, though, in view of the assumption
$$\int_0^1 (1-t^2)t f(t)=0.$$
In fact, using Fubini's theorem, and noticing that $0\leq t\leq x\leq 1$,
\begin{eqnarray*}
\int_0^1 h(x) dx&=& \int_0^1 \int_0^x \left(xt f(t)- t^2f(t)\right)dt dx
=\int_0^1 \int_t^1 \left(xt f(t)- t^2f(t)\right)dxdt\\
&=&\int_0^1\int_t^1 (x-t) tf(t) dx dt=
\int_0^1 \left(\frac{1-t^2}{2} -(1-t)t\right) tf(t)\; dt\\
&=&\frac{1}{2}\;\int_0^1 (1-t)^2  t f(t)\; dt=0.
\end{eqnarray*}

\newpage


\gr{\huge\section{Elemente der Mathematik}}
\bigskip


  \begin{figure}[h!]
  \scalebox{0.55} 
  {\includegraphics{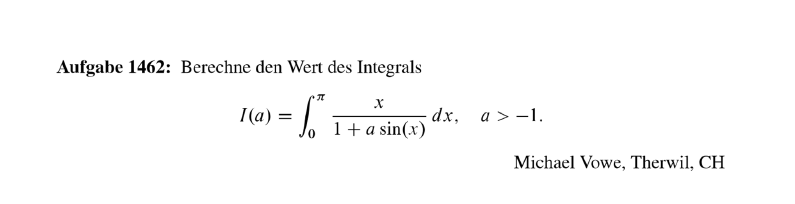}} 
\end{figure}



\centerline{\bf Solution to problem 1462 Elem. Math. 80 (2025), 132}

\bigskip

 \centerline {Raymond Mortini, Rudolf Rupp } \medskip

\centerline{- - - - - - - - - - - - - - - - - - - - - - - - - - - - - - - - - - - - - - - - - - - - - - - - - - - - - -}
  
  \medskip

\newpage

   \begin{figure}[h!]
  \scalebox{0.35} 
  {\includegraphics{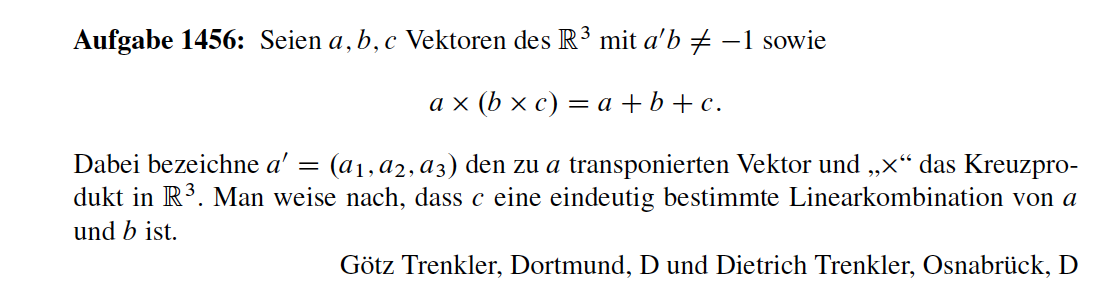}} 
\end{figure}



\centerline{\bf Solution to problem 1456 Elem. Math. 80 (2025), 38}

\bigskip

 \centerline {Raymond Mortini, Rudolf Rupp } \medskip

\centerline{- - - - - - - - - - - - - - - - - - - - - - - - - - - - - - - - - - - - - - - - - - - - - - - - - - - - - -}
  
  \medskip


Benutzt man die Grassman Identit\"at aus dem Physikunterricht, welche aussagt dass
  $$a\times(b\times c)= b (a\cdot c) -c(a\cdot b)$$
  gilt (hier ist das Skalarprodukt mit $\cdot{}$ gekennzeichnet und (zur besseren Memorisierung)
   die  "\"ubliche" Schreibweise   $\lambda\, \vec x$ der Multiplikation eines Vektors mit einem Skalar durch
   $\vec x\, \lambda$ ersetzt)), so erh\"alt man:
   \begin{eqnarray*}
   a+b+c=a\times (b\times c)& \iff & a+b+c=b(a\cdot c)-c(a\cdot b)\\
   &\iff& a+ b(1-a\cdot c) +c\underbrace{(1+a\cdot b)}_{\not=0}=0\\
   &\iff& c=- a\;\frac{1}{1+a\cdot b} - b\;\frac{1-a\cdot c}{1+a\cdot b}.
\end{eqnarray*}

\bigskip
{\bf Bemerkung} Wozu den transponierten Vektor einf\"uhren? Man nehme anstatt Spaltenvektoren  $\underline a$ im $\R^3$ die entsprechenden Zeilenvektoren $a$ und dann ist $ \underline a' \;\underline b=a\cdot b$.

\newpage

   \begin{figure}[h!]
  \scalebox{0.45} 
  {\includegraphics{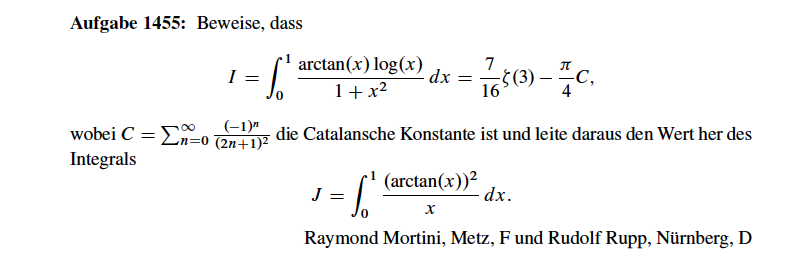}} 
\end{figure}



\centerline{\bf Solution to problem 1455 Elem. Math. 79 (2024), 176}

\bigskip

 \centerline {Raymond Mortini, Rudolf Rupp } \medskip

\centerline{- - - - - - - - - - - - - - - - - - - - - - - - - - - - - - - - - - - - - - - - - - - - - - - - - - - - - -}
  
  \medskip


 We use the substitution $t=\arctan x$, $dt=\frac{dx}{1+x^2}$. Thus
   
   $$I=\int_0^{\pi/4} t\log(\tan t) dt.$$
     
     Now we compute  the Fourier series for $\log\tan t$:
     
   \begin{eqnarray*}
\log\tan t&=&\log \cos t -\log \sin t =\log(\co{2} \cos t)-\log(\co{2}\sin t)\\
&=&\left(\cos(2t)-\frac{\cos(4t)}{2}+\frac{\cos(6t)}{3}-\cdots\right) - \left(-\cos(2t)-\frac{\cos(4t)}{2}-\frac{\cos(6t)}{3}-\cdots\right)\\
&=&-2 \left(\cos(2t)+\frac{\cos(6t)}{3} + \frac{\cos(10t)}{5}+\cdots\right)\\
&=& -2 \sum_{n=1}^\infty\frac{1}{2n-1} \cos((4n-2)t).
\end{eqnarray*}
   
  Hence
  $$t\log \tan t=-2\sum_{n=1}^\infty \frac{t}{2n-1} \cos((4n-2)t).
  $$
  This  converges uniformly on $[0,\pi/4]$ \footnote{ Abel-Dirichlet rule \cite[p. 2002]{moru} for $\sum a_n \cdot b_n(t)$
   with $|\sum_{n=0}^m b_n(t)|\leq M$, 
  $a_n=1/(2n-1)$ and $b_n(t)=t e^{it(4n-2)}$. Here $M=\pi/2$  as $0\leq x/\sin x \leq \pi/2$ for $0\leq x\leq \pi/2$.}.
  Hence $\int\sum=\sum\int$. Since a primitive of $t\cos (at)$ is given by
  $$
  \frac{t\sin(at)}{a}+\frac{\cos(at)}{a^2},
  $$
  we obtain
  $$I=-2\sum_{n=1}^\infty \frac{1}{2n-1}\left(-\frac{1}{4}\,\frac{1}{(2n-1)^2}- \frac{\pi}{8}\frac{(-1)^n}{2n-1}\right).
  $$ 
  
  Hence 
  $$I=\frac{1}{2}\sum_{n=1}^\infty \frac{1}{(2n-1)^3}-\frac{\pi}{4}\sum_{n=1}^\infty\frac{(-1)^{n-1}}{(2n-1)^2}
  =\frac{7}{16}\zeta(3)-\frac{\pi}{4}\; C.$$

   As  a corollary we obtain  the following known result (\cite{ko}) \footnote{ Which we stumbled upon accidentally  when we googled to check whether our result was new or not. It turned out that, coincidentally, the proofs are similar,}:
   
   $$J=\int_0^1 \frac{(\arctan x)^2}{x}dx= \frac{\pi C}{2}-\frac{7}{8}\zeta(3).$$
   
   In fact, using partial integration $\int uv'=uv-\int u'v$ with $u=\log x$ and $v'=\frac{\arctan x}{1+x^2}$ we obtain
   \begin{eqnarray*}
   \int_0^1 \frac{(\arctan x)(\log x)}{1+x^2}dx&=&\frac{1}{2}\, \log x( \arctan x)^2)\Big|_0^1- \frac{1}{2}\, \int_0^1 \frac{(\arctan x)^2}{x} dx\\
   &=&- \frac{1}{2}\, \int_0^1 \frac{(\arctan x)^2}{x} dx.
\end{eqnarray*}

   Or vice versa, using partial integration with  $u=(\arctan x)^2$ and $v'=1/x$,
   
   \begin{eqnarray*}
    \int_0^1 \frac{(\arctan x)^2}{x} dx&=& (\arctan x)^2\log x\Big|_0^1-\int_0^1 2\frac{\arctan x}{1+x^2}\log x dx\\
    &=& -2   \int_0^1 \frac{(\arctan x)(\log x)}{1+x^2}dx.
\end{eqnarray*}

\newpage
\nopagecolor

   \begin{figure}[h!]
  \scalebox{0.45} 
  {\includegraphics{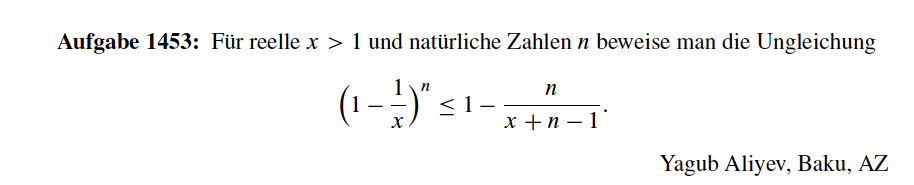}} 
\end{figure}

\centerline{\bf Solution to problem 1453 Elem. Math. 79 (2024), 176}

\bigskip

 \centerline {Raymond Mortini, Rudolf Rupp } \medskip

\centerline{- - - - - - - - - - - - - - - - - - - - - - - - - - - - - - - - - - - - - - - - - - - - - - - - - - - - - -}
  
  \medskip


This is entirely trivial.  First we have that
\begin{eqnarray*}
\left(1-\frac{1}{x}\right)^n\leq 1-\frac{n}{x+n-1}=\frac{x-1}{x+n-1}=x\;\frac{1-\frac{1}{x}}{x+n-1}&\iff&
\ovalbox{$\dis\left(1-\frac{1}{x}\right)^{n-1}\leq \frac{x}{x+n-1}$}.
\end{eqnarray*}

Now the right inequality is shown by induction ($x>1$):\\

$n=0$:  $\left(1-\frac{1}{x}\right)^{-1}=\frac{x}{x-1}$.

$n\to n+1$:

\begin{eqnarray*}
\left(1-\frac{1}{x}\right)^n&\buildrel\leq_{}^{{\rm  ind. hyp}}& \frac{x}{x+n-1}\left(1-\frac{1}{x}\right)= \frac{x}{x+n-1}\frac{x-1}{x}=
\frac{x\co{+n}-1 \co{-n}}{x+n-1}\\
&=& 1-\frac{n}{x+n-1}\leq 1-\frac{n}{x+n}= \frac{x}{x+n}.
\end{eqnarray*}

\newpage

\newpage

   \begin{figure}[h!]
  \scalebox{0.45} 
  {\includegraphics{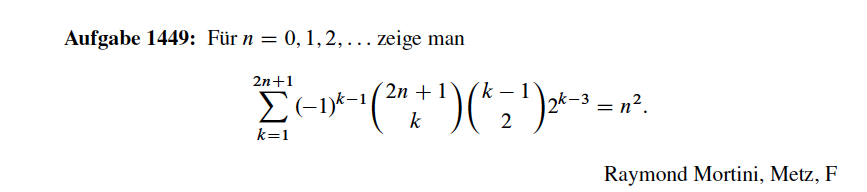}} 
\end{figure}

\nopagecolor

\centerline{\bf Solution to problem 1449 Elem. Math. 79 (2024), 131}

\bigskip
 

 \centerline {Raymond Mortini} \medskip

\centerline{- - - - - - - - - - - - - - - - - - - - - - - - - - - - - - - - - - - - - - - - - - - - - - - - - - - - - -}
  
  \medskip

 Es sei $\dis f(x)=\frac{(x+1)^{2n+1}}{x}$. Dann gilt
 $$f''(x)= \frac{(x+1)^{2n-1}\Big(2n(2n-1)x^2-(4n-2)x+2\Big)}{x^3}$$
 sowie
 $$f''(-2)=2n^2-\frac{1}{4}$$
 Nun berechnen wir die Laurentreihe und deren zweite Ableitung:
 $$f(x)=\sum_{k=0}^{2n+1} {2n+1\choose k} x^{k-1}.$$
\begin{eqnarray*}
f''(x)&=&\sum_{k=0}^{2n+1}{2n+1\choose k}  (k-1)(k-2)  x^{k-3}\\
&=&2\;\sum_{k=0}^{2n+1}{2n+1\choose k}  {k-1 \choose 2}  x^{k-3}
\end{eqnarray*}
 Auswertung an $x=-2$ ergibt:
\begin{eqnarray*}
f''(-2)&=&2\;\sum_{k=0}^{2n+1} (-1)^{k-1}{2n+1\choose k}  {k-1 \choose 2}  2^{k-3}\\
&=& 2 (-1)^1 {2n+1\choose 0}{-1 \choose 2} 2^{-3} + 2A\\
&=&-\frac{1}{4}+2A.
 \end{eqnarray*}
Folglich gilt  $A=n^2$.\\

This was a very special case of problems treated in \cite[p.8]{chmo}.


\newpage

\nopagecolor

   \begin{figure}[h!]
  \scalebox{0.45} 
  {\includegraphics{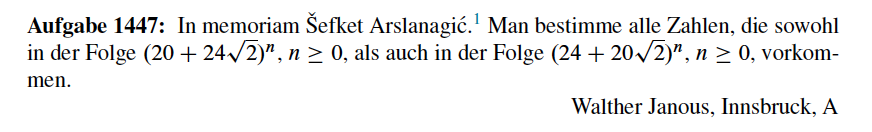}} 
\end{figure}

\centerline{\bf Solution to problem 1447 Elem. Math. 79 (2024), 84}

\bigskip
 

 \centerline {Raymond Mortini, Rudolf Rupp } \medskip

\centerline{- - - - - - - - - - - - - - - - - - - - - - - - - - - - - - - - - - - - - - - - - - - - - - - - - - - - - -}
  
  \medskip

\vspace{1cm}

Zu untersuchen ist f\"ur welche $n,m$  die Gleichung $(20+24 \sqrt 2)^n=(24+20\sqrt 2)^m$ gilt. Wir zeigen, dass
nur die $1$ als gemeinsame Zahl vorkommt (dies f\"ur $n=m=0$).  In der Tat, bezeichnet $N(a+b\sqrt 2):=a^2-2b^2$ 
die sogenannte Norm des Elements $a+b\sqrt 2$ im quadratischen Zahlk\"orper $\mathbb Q[\sqrt 2]$, so erhalten wir wegen deren Multiplikativit\"at $N(xy)=N(x)N(y)$  folgende Bedingung:

$$N(4^n(5+6\sqrt 2)^n)=N( 4^m(6+5\sqrt 2)^m).$$
\"Aquivalent:
$$4^{2n}(25-72)^n=4^{2m}(36-50)^m$$
also
$$16^n (-47)^n=16^m(-14)^m.$$

Da $47$ eine Primzahl ist,  die rechte Seite aber nicht durch diese Zahl teilbar ist falls $n, m\geq 1$, bleibt nur $n=m=0$ \"ubrig.


\newpage

   \begin{figure}[h!]
  \scalebox{0.45} 
  {\includegraphics{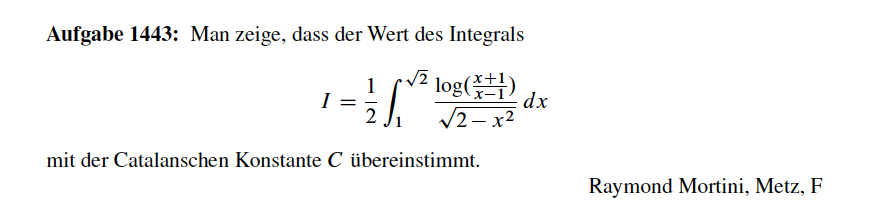}} 
\end{figure}

\nopagecolor

\centerline{\bf Solution to problem 1443 Elem. Math. 79 (2024), 38}

\bigskip
 

 \centerline {Raymond Mortini} \medskip

\centerline{- - - - - - - - - - - - - - - - - - - - - - - - - - - - - - - - - - - - - - - - - - - - - - - - - - - - - -}
  
  \medskip

We use that  (see \cite[p. 453]{ste})
$$\int_0^{\pi/4} \log\left( \frac{\sqrt 2 \cos x+1}{\sqrt 2 \cos x -1}\right) dx =2C.$$

Now put  $v:=\sqrt 2 \cos x$,  $1\leq v\leq \sqrt 2$, 
$$dv= -\sqrt 2 \sin x\; dx= -\sqrt 2 \sqrt{1-\cos^2 x}\;dx=-\sqrt 2( \sqrt{1-v^2/2}\; dx= -\sqrt{2-v^2}\; dx.
$$

Hence

$$2C= \int_1^{\sqrt 2} \frac{\log\left( \frac{v+1}{v-1}\right)}{ \sqrt{2-v^2}}\; dv.
$$


\newpage

   \begin{figure}[h!]
  \scalebox{0.45} 
  {\includegraphics{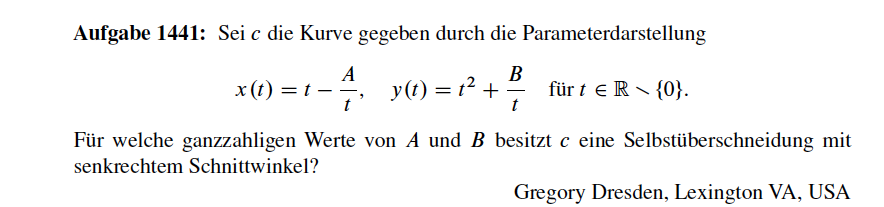}} 
\end{figure}

\nopagecolor

\centerline{\bf Solution to problem 1441 Elem. Math. 78 (2023), 180}

\bigskip
 

 \centerline {Raymond Mortini, Rudolf Rupp } \medskip

\centerline{- - - - - - - - - - - - - - - - - - - - - - - - - - - - - - - - - - - - - - - - - - - - - - - - - - - - - -}
  
  \medskip

Diese Aufgabe ist unseren Ermessens nach nicht korrekt formuliert,  da per Definition eine Kurve das stetige Bild eines {\it Intervals}
in einen topologischen Raum  ist. Deshalb liegen hier zwei Kurven vor und wir werden folgendes zeigen
\footnote{Die Schreibweise mit $-A$ und $B$, anstatt $A$ und $B$, l\"asst uns vermuten, dass bloss nichtnegative Zahlen gemeint waren mit der Bezeichnung "ganzzahlig". Im zweiten Abschnitt  betrachten wir auch den Fall wo $a,b$ beliebig sind.}:

\subsection{Nichtnegative Parameter}

\begin{proposition}
Es seien $a,b\in \R$, $a,b\geq 0$ und $\Gamma^\pm$ die Kurven, welche gegeben sind durch die Parameterdarstellung
$$\Gamma^+: z(t)=(x(t), y(t))=\left(t-\frac{a}{t}, t^2+\frac{b}{t}\right), t>0,$$
beziehungsweise
$$\Gamma^-: z(t)=(x(t), y(t))=\left(t-\frac{a}{t}, t^2+\frac{b}{t}\right), t<0.$$
Dann gilt:
\begin{enumerate}
\item[(1)] $\Gamma^\pm$ besitzen keine Selbst\"uberschneidungen; sind also Jordanb\"ogen.
\item[(2)] F\"ur $a>0$  schneidet $\Gamma^+$ die Kurve $\Gamma^-$ in genau einem Punkt.
\item[(3)] F\"ur $a=0$ liegt der Graph einer Funktion auf $\R\setminus\{0\}$ vor, und folglich sind keine Selbst\"uberschneidungen vorhanden.
\item[(4)] $\Gamma^+$ schneidet $\Gamma^-$   f\"ur $a,b\in \N=\{0,1,2\dots\}$  mit senkrechten Schnittwinkel genau dann wenn 
$(a,b)=(2r^2+1,r(2r^2+1))$ f\"ur ein $r\in \N$.
\end{enumerate}
\end{proposition}

{\bf L\"osung}  (1) (2)
Zu betrachten ist das folgende System von Gleichungen f\"ur $s,t\in \R\setminus\{0\}$:
$$\left\{\begin{matrix} t-a/t &=&s-a/s\\\\ t^2+b/t&=&s^2+b/s.
\end{matrix}\right.
\iff \left\{\begin{matrix} t-s&=&a\left(\frac{1}{t}-\frac{1}{s}\right)=a(s-t)\frac{1}{st}\\\\ t^2-s^2&=&b\left(\frac{1}{s}-\frac{1}{t}\right)
=b(t-s)\frac{1}{st}.
\end{matrix}\right.
$$
F\"ur $s\not=t$  und $a\not=0$ ist dies \"aquivalent zu
$$
\left\{\begin{matrix} st&=&-a\\ st(t+s)&=&b
\end{matrix}\right. \iff 
\left\{\begin{matrix} st&=&-a\\ t+s&=&-\frac{b}{a}.
\end{matrix}\right. 
$$
Dies f\"uhrt  auf die L\"osung der quadratischen Gleichung 
$$0= x^2-(s+t)x+st= x^2+\frac{b}{a}x-a.$$

Die L\"osungen hierzu sind 
\begin{equation}\label{werte}
\mbox{$\dis s= \frac{-b+\sqrt{b^2+4a^3}}{2a}$ und $\dis t= \frac{-b-\sqrt{b^2+4a^3}}{2a}$}.
\end{equation}

Es liegen also ein negativer und ein positiver Wert der Kurven-Parameter
 $s$ und $t$ vor. Folglich schneidet $\Gamma^-$ die Kurve $\Gamma^+$ in genau einem Punkt und es liegen keine Selbst\"uberschneidungen vor. \\

(3) ist klar.\\

(4) Es sei $M:=(z(s),z(t))$ dieser eindeutige Schnittpunkt mit $s<0$ und $t>0$.
Diese Kurven $\Gamma^-$ und $\Gamma^+$ schneiden sich nun senkrecht in $M$ genau dann wenn  gilt
\begin{eqnarray*}
0&=&\left<\dot z(s), \dot z(t)\right> = \dot x(s)\; \dot x(t)+ \dot y(s)\; \dot y(t)\\
&\iff&\left(1+\frac{a}{s^2}\right)\; \left(1+\frac{a}{t^2}\right)=- \left( 2s-\frac{b}{s^2}\right)\; \left( 2t-\frac{b}{t^2}\right).
\end{eqnarray*}

Einsetzen von $st=-a$ und $st(s+t)=b$ ergibt
$$-\left(1-\frac{t}{s}\right)\;\left(1-\frac{s}{t}\right)=\left(2s-\frac{t(t+s)}{s}\right)\; \left(2t-\frac{s(t+s)}{t}\right).$$
Multiplikation mit $st$ liefert
\begin{eqnarray*}
-(s-t)(t-s)&=&\big(2s^2-t(t+s)\big)\;\big(2t^2-s(t+s)\big)\\
&=&(s-t)(t+2s)\; (t-s)(s+2t)\\
&=& (s-t)(t-s)(t+2s)(s+2t).
\end{eqnarray*}
Da $s\not=t$, erh\"alt man schliesslich
\begin{equation}\label{ts}
(t+2s)(s+2t)=-1.
\end{equation}
Umformen ergibt

$$
-1=ts+2s^2+2t^2+4st=5st+2\big((s+t)^2-2st\big)=st+2(s+t)^2.
$$
Durch Einsetzen von $st=-a$ und $s+t=-b/a$ ergibt das
\begin{equation}\label{ab}
-1=-a+2\left(\frac{b}{a}\right)^2\iff -a^2=-a^3+2b^2\iff 2b^2+a^2-a^3=0.
\end{equation}

Wir m\"ussen nun alle L\"osungspaare $(a,b)\in \N^2$ dieser diophantischen Gleichung bestimmen.
Ein Umschreiben ergibt
\begin{equation}\label{dio}
2\frac{b^2}{a^2}-a=-1.
\end{equation}
Es sei $r:=\frac{b}{a}$. Gem\"ass der Voraussetzung ist $r\in \Q$. Also hat $a$ die Form  $a=2r^2+1$ und $b=ra=r(1+2r^2)$.
  Dies ist jedoch nur m\"oglich wenn $r$ selbst in $\N$ liegt, was man wie folgt einsehen kann. Ist $r=0$ so ist das evident. Sei 
  also $r\not=0$.  
  Die Voraussetzungen $a\in \N$ und
  $a=2r^2+1$ implizieren  $m:=2r^2\in \N$.  Wir zeigen dass $m$
  von der Form $m=2n^2$ ist f\"ur ein $n\in \N$ und damit ist $r\in \N$.
  Es sei $r=p/q$, mit ${\rm ggt}\;(p,q)=1$. Sodann $ mq^2=2p^2$. Ist $m=2i$ gerade, so erhalten wir  $iq^2=p^2$. Da jeder Primfaktor von $q$ nun $p^2$ teilt, also auch $p$, muss wegen $ {\rm ggt}\;(p,q)=1$ nun $q=1$ sein.  D.h. $m$ hat die gew\"unschte Form.
  Ist $m=2i+1$ ungerade, so muss wegen $2p^2=mq^2$ die Zahl $q$ auch gerade sein. Sagen wir $q=2^j u$, wobei $j\in \N, j\not=0$, 
  und $u$ ungerade. Folglich ist $p^2=(mu^2) 2^{2j-1}$. Weil $2$ kein gemeinsamer Faktor von $p$ und $q$ ist,  erhalten wir den Widerspruch, da die linke Seite von $p^2=(mu^2) 2^{2j-1}$ ungerade ist, die rechte aber gerade.

Damit haben alle L\"osungen von (\ref{dio}) notwendigerweise die Form $(a,b)=(1+2r^2, r (1+2r^2))$, $r\in \N$. 
Umgekehrt, ist auch jedes solche Paar
Lo\"sung der Gleichung (\ref{dio}):
$$ 2 r^2-(1+2r^2)=-1.$$

\vspace{1cm}

Beispiele f\"ur $(a,b)$: $(1,0)$, $(3,3)$, $(9, 18)$, $(19, 57)$, $(33, 132)$.

Damit hat man mit $(a,b)=(1,0)$ den einzigen Schnittpunkt $M=(0,1)$ von $\Gamma^+$ mit 
$\Gamma^-$ im 90 Grad Winkel f\"ur $(s,t)=(1,-1)$ bei
$$\Gamma^+(s)=\left(s-\frac{1}{s}, s^2\right), s>0,\quad \Gamma^-(t)=\left(t-\frac{1}{t},t^2\right), t<0,$$

 \begin{figure}[h!]
  \scalebox{0.45} 
  {\includegraphics{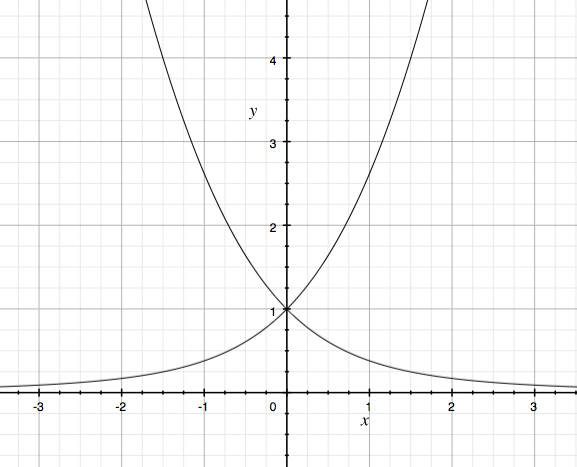}} 
  \caption{\label{$r=0$} $r=0$, $a=1$, $b=0$}
\end{figure}

oder mit $(a,b)=(3,3)$ den Schnittpunkt $M=(-1, 4)$
von $\Gamma^+$ mit 
$\Gamma^-$ im 90 Grad Winkel f\"ur $(s,t)=(\frac{-1+\sqrt {13}}{2},\frac{-1-\sqrt {13}}{2})$ bei
$$\Gamma^+(s)=\left(s-\frac{3}{s}, s^2+\frac{3}{s}\right), s>0,\quad \Gamma^-(t)=\left(t-\frac{3}{t},t^2+\frac{3}{t}\right), t<0,$$

oder mit $(a,b)=(9,18)$ den Schnittpunkt $M=(-2,13)$
von $\Gamma^+$ mit 
$\Gamma^-$ im 90 Grad Winkel f\"ur $(s,t)=(\sqrt{10}-1,-1-\sqrt{10})$ bei
$$\Gamma^+(s)=\left(s-\frac{9}{s}, s^2+\frac{18}{s}\right), s>0,\quad \Gamma^-(t)=\left(t-\frac{9}{t},t^2+\frac{18}{t}\right), t<0,$$

Das \"Uberraschendste f\"ur uns bei dieser proposition: der Schnittpunkt $M$ hat auch ganzzahlige Komponenten:
$$M=(-r, 1+3r^2).$$

\begin{figure}[h!]
  \scalebox{0.35} 
  {\includegraphics{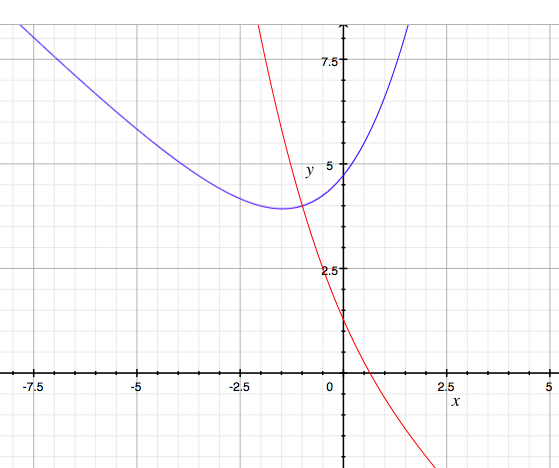}} 
   \caption{\label{$r=1$} $r=1$, $a=3$, $b=3$}
\end{figure}

\newpage
\subsection{Die restlichen F\"alle}

{\it Fall 2.1}: $b^2+4a^3<0$, \"aquivalent $a< -\left(\frac{b^2}{4}\right)^{1/3}$. In diesem Fall hat die quadratische Gleichung (\ref{werte}) keine reellen L\"osungen, und folglich sind 
$\Gamma^+$ und $\Gamma^-$ Jordanb\"ogen die sich nicht schneiden (siehe Grafik \ref{a-neg}).

\begin{figure}[h!]
  \scalebox{0.33} 
  {\includegraphics{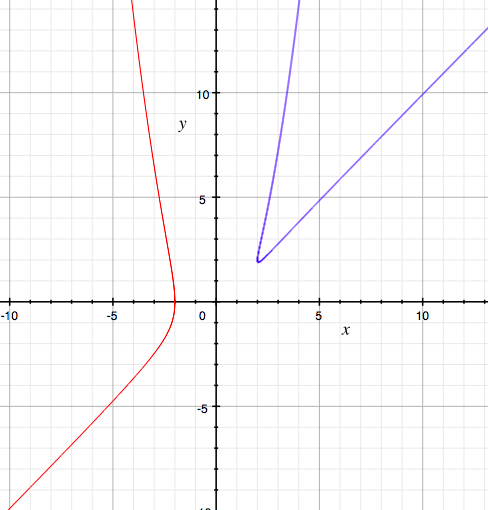}} 
   \caption{\label{a-neg} $a=-1$, $b=1$}
\end{figure}

{\it Fall 2.2}: $-\left(\frac{b^2}{4}\right)^{1/3}<a<0$. In diesem Fall ist $b^2+4a^3>0$ und es liegen zwei verschiedene  L\"osungen $s,t$ der 
quadratische Gleichung (\ref{werte}) vor, welche aber wegen $st=-a>0$ dasselbe Vorzeichen haben. Damit schneiden 
sich die Kurven $\Gamma^+$ und $\Gamma^-$ nicht, aber genau eine von denen hat einen Selbst\"uberschneidungspunkt (siehe Grafik \ref{a-neg2}). N\"amlich $\Gamma^+$ falls $b>0$ und $\Gamma^-$ falls $b<0$.  Anmerken m\"ochten wir noch, dass die Gleichung
(\ref{ab}), $2b^2+a^2-a^3=-1$, keine L\"osung hat falls $a<0$. Folglich ist diese Selbst\"uberschneidung nie senkrecht.

\begin{figure}[h!]
  \scalebox{0.37} 
  {\includegraphics{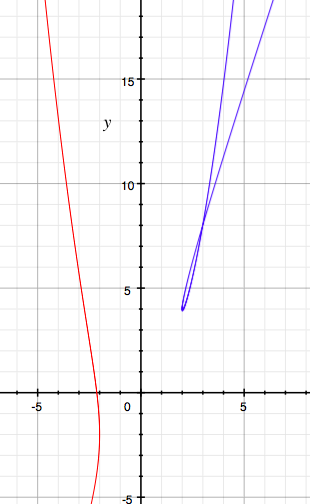}} 
   \caption{\label{a-neg2} $a=-1$, $b=3$}
\end{figure}

\newpage

{\it Fall 2.3}: $b^2+4a^3=0$.

Auch hier schneiden sich die Kurven $\Gamma^+$ und $\Gamma^-$ nicht, und beide sind wieder Jordanb\"ogen.%

Siehe Grafik (\ref{a-neg3}) zum Beispiel
$$\Gamma^+(s)=\left(s+\frac{1}{s}, s^2+\frac{2}{s}\right), s>0,\quad \Gamma^-(t)=\left(t+\frac{1}{t},t^2+\frac{2}{t}\right), t<0.$$

\begin{figure}[h!]
  \scalebox{0.33} 
  {\includegraphics{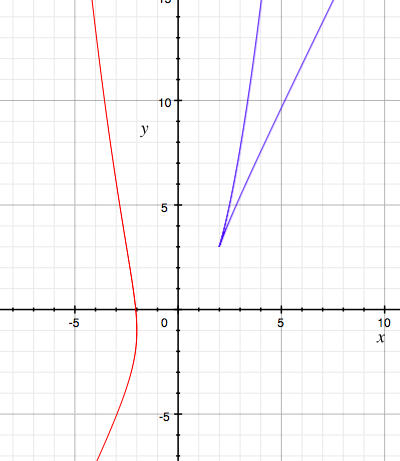}} 
   \caption{\label{a-neg3} $a=-1$, $b=2$}
\end{figure}

{\it Fall 3}: $a>0$, $b\leq 0$. Da ist prinzipiell kein Unterschied zum Fall $a>0,b\geq 0$; es liegt nur eine Spiegelung der Kurven an der
 $y$-Achse vor. Z.B. hat f\"ur $b<0$ die Kurve
 
 $$(x(t), y(t))=\left(t-\frac{a}{t}, t^2+\frac{b}{t}\right), t>0,$$
mit der  Transformation $t\to -t$  auch die Parameterdarstellung

$$ \left(-(t-\frac{a}{t}), t^2-\frac{b}{t}\right), t<0.
$$

Die  Spiegelung an der $y$ Achse ist dann gegeben durch

$$ \left(t-\frac{a}{t}, t^2+\frac{(-b)}{t}\right), t<0$$

\begin{figure}[h!]
  \scalebox{0.33} 
  {\includegraphics{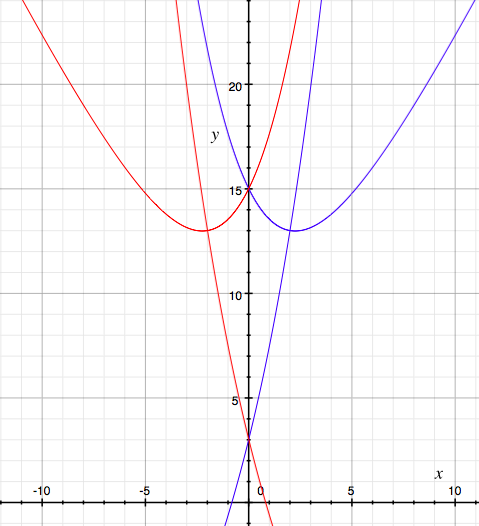}} 
   \caption{\label{a-neg4} $a=9$, $b=18$ blau, $b=-18$ rot}
\end{figure}

{\bf Fazit}: Die Kurve $\Gamma^+=\Gamma^+(a,b)$ schneidet  f\"ur $a,b\in \Z$ die Kurve $\Gamma^-(a,b)$ in einem rechten Winkel genau dann wenn
$$\mbox{$(a,b)=\big((1+2r^2), \pm r(1+2r^2)\big)$ mit $r\in \N=\{0,1,2,3\dots\}$.}$$
Oder in der Formulierung der proposition:
Die unstetige "Kurve" $c=c(a,b)$ besitzt eine "Selbst\"uberschneidung" mit senkrechtem Schmittwinkel genau dann wenn
$$\mbox{$(a,b)=\big((1+2r^2), r(1+2r^2)\big)$ mit $r\in \Z$.}$$

\newpage

   \begin{figure}[h!]
  \scalebox{0.45} 
  {\includegraphics{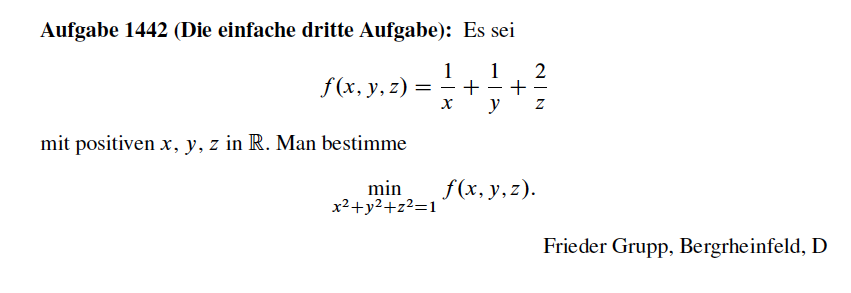}} 
\end{figure}

\centerline{\bf Solution to problem 1442 Elem. Math. 78 (2023), 180}

\bigskip
 

 \centerline {Raymond Mortini, Rudolf Rupp } \medskip

\centerline{- - - - - - - - - - - - - - - - - - - - - - - - - - - - - - - - - - - - - - - - - - - - - - - - - - - - - -}
  
  \medskip

Let $\dis f(x,y,z):=\frac{1}{x}+\frac{1}{y}+\frac{2}{z}$.
Using Lagrange multipliers, we show that

$$\ovalbox{$\dis\inf_{x^2+y^2+z^2=1\atop x>0, y>0, z>0} f(x,y,z)=
\min_{x^2+y^2+z^2=1\atop x,y,z \geq \frac{1}{4\sqrt 3}}f(x,y,z)=
(2+2^{2/3})^{3/2}\sim 6.794693902\cdots$}$$
\bigskip

First we recall the  version of Lagrange's theorem, we will use:

\begin{theorem}
Let $G:=\{\xi=(x,y,z)\in \R^3: x>0,y>0, z>0\}$ be the first octant and $g(x,y,z):=x^2+y^2+z^2-1$. Then $f$ and $g$ belong to $C^1(G)$. 
If $\zeta\in G$ is a local extremum of  $f$ on the set $N:=\{(x,y,z)\in G: g(x,y,z)=0\}$ for which $\frac{\partial}{\partial z}  g(\zeta)\not=0$, then there exists $\lambda\in \R$ such that $(\zeta,\lambda)$ is  a stationary point of Lagrange's function
$$L(x,y,z,\lambda):=f(x,y,z)+\lambda g(x,y,z).$$
\end{theorem}

Next, we prove the existence of such a local extremum. Let $\eta:=(3^{-1/2},3^{-1/2},3^{-1/2})$. Then $\eta\in N$. Moreover,
$f(\eta)=4\sqrt 3$.  Also, if $(x,y,z)\in N$ is such that at least one of its coordinates is strictly bigger than $(4\sqrt 3)^{-1}$, then
 $f(x,y,z)> 4\sqrt 3$. Hence
 $$\inf_N f =\min\left\{f(x,y,z): x^2+y^2+z^2=1,~~ x,y,z \geq \frac{1}{4\sqrt 3}\right\}.$$
 
 Finally we solve Lagrange's equations for $(x,y,z)\in G$ and $\lambda\in \R$:
 $$(1)\hspace{1cm}\frac{\partial}{\partial x} L(x,y,z,\lambda)= -\frac{1}{x^2} +\lambda (2x)\buildrel=_{}^{!}0$$
$$(2)\hspace{1cm}\frac{\partial}{\partial y} L(x,y,z,\lambda)= -\frac{1}{y^2} +\lambda (2y)\buildrel=_{}^{!}0$$
$$(3)\hspace{1cm}\frac{\partial}{\partial z} L(x,y,z,\lambda)= -\frac{2}{z^2} +\lambda (2z)\buildrel=_{}^{!}0$$
$$\hspace{1.3cm}{(4)}\quad\quad\quad\quad\frac{\partial}{\partial \lambda} L(x,y,z,\lambda)= x^2+y^2+z^2-1\buildrel=_{}^{!}0$$

(1) and (2) yield that $x=y$  and (2) and (3) yield that $\frac{2}{z^3}=\frac{1}{y^3}$, equivalently $z=2^{1/3}y$. Due to (4),
$$1=x^2+x^2+2^{2/3}x^2,$$
hence
$$x=y=(2+2^{2/3})^{-1/2},\quad\quad z=2^{1/3}(2+2^{2/3})^{-1/2}.$$

Consequently, the unique stationary point of $L$ on $G\times\R$ is 
$$P=\left(\frac{1}{\sqrt{2+2^{2/3}}},\frac{1}{\sqrt{2+2^{2/3}}}, \frac{2^{1/3}}{\sqrt{2+2^{2/3}}}, \frac{(2+2^{2/3})^{3/2}}{2}\right)
$$
Let $\zeta$ be the point formed with the first three coordinates of $P$, which are of course bigger than $(4\sqrt 3)^{-1}$. Then 
$$f(\zeta)=2\,\sqrt{2+2^{2/3}}+ 2\; \frac{\sqrt{2+2^{2/3}}}{2^{1/3}}= (2+2^{2/3})\;\sqrt{2+2^{2/3}}= (2+2^{2/3})^{3/2}.
$$
Of course, this point $\zeta$ must  now be that unique point on $N$ where $\inf_N$ is taken (note that $\sup f_N=\infty)$.

\newpage

\pagecolor{yellow}
   \begin{figure}[h!]
 
  \scalebox{0.45} 
  {\includegraphics{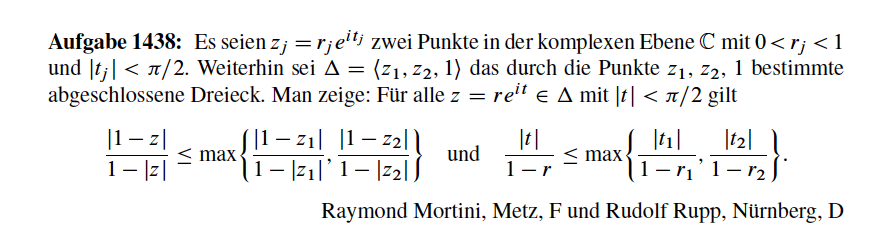}} 
\end{figure}

\centerline{\bf Solution to problem 1438 Elem. Math. 78 (2023), 135}

\bigskip
 

 \centerline {Raymond Mortini, Rudolf Rupp } \medskip

\centerline{- - - - - - - - - - - - - - - - - - - - - - - - - - - - - - - - - - - - - - - - - - - - - - - - - - - - - -}
  
  \medskip

Three solutions for (1) in \cite[Appendix 27]{moru}, one for (2) in  \cite[Appendix 31]{moru}.\\

(1)  Man beachte zun\"achst
$$\Delta= \{t_1z_1+t_2z_2+t_3:0\leq  t_j\leq 1, t_1+t_2+t_3=1\}.$$
Ist nun $z\in \Delta$, so gilt
$$z=t_1z_1+t_2z_2+(1-t_1-t_2)=t_1(z_1-1)+t_2(z_2-1)+1.$$
Daher
$$|1-z|=|t_1(z_1-1)+t_2(z_2-1)|\leq t_1|z_1-1|+t_2|z_2-1|.$$
Andererseits
$$|z|\leq t_1|z_1|+t_2|z_2|+1-t_1-t_2.$$
Somit
$$1-|z|\geq t_1(1-|z_1|)+t_2(1-|z_2|).$$
Folglich gilt
$$\frac{|1-z|}{1-|z|}\leq \frac{t_1|z_1-1|+t_2|z_2-1|}{t_1(1-|z_1|)+t_2(1-|z_2|)}\leq \max\{Z_1,Z_2\}:=\kappa,$$
wobei
$$Z_j=\frac{|1-z_j|}{1-|z_j|}.$$

(2) Ist f\"ur $\kappa>0$, $D_\kappa^*=\{re^{it}\in \bs D: |t|\leq\kappa(1-r)\}$,  so ist der Teil der  Randkurve welche im 1-ten Quadraten liegt, in Polarkoordinaten gegeben durch
$$r(t)=1-\frac{t}{\kappa},  0\leq t\leq \min\{\pi/2,\kappa\}.$$

Diese Kurve ist Graph einer konkaven Funktion (Bilder und Details wie im Buch \cite[Appendix 31]{moru} unter Beachtung dass dort alles auch f\"ur $0<\kappa<1$ gilt). Folglich ist die Menge 
$$\{(x,y):  0\leq x=x(t)\leq 1, 0\leq y\leq y(t)\}$$ konvex,
und somit auch  $S(\kappa):=D_\kappa^*\inter \{z\in \C: {\rm Re} \,z\geq 0\}$.  Unter Beachtung dass 
$D_{\kappa_1}^*\ss  D_{\kappa_2}^*$ f\"ur $0< \kappa_1\leq \kappa_2$, schliessen wir,
dass mit
 $z_1,z_2, 1\in S(\kappa)$ auch $\Delta$ in $S(\kappa)$ liegt falls
 $$\kappa= \max \left\{ \frac{|t_1|}{1-r_1}, \frac{|t_2|}{1-r_2}\right\}.$$

\newpage
\nopagecolor

   \begin{figure}[h!]
 
  \scalebox{0.45} 
  {\includegraphics{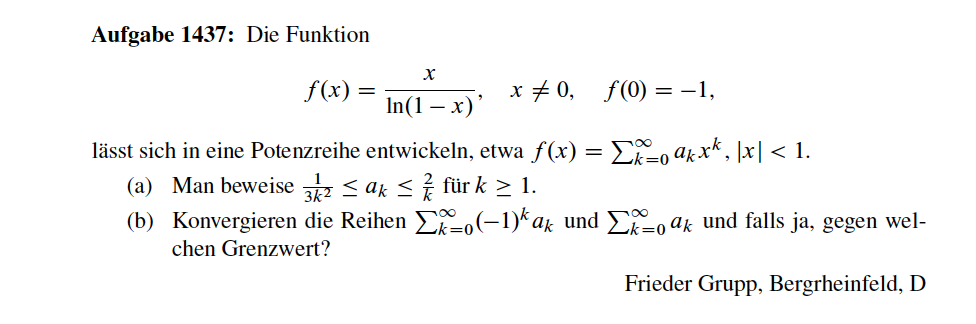}} 

\end{figure}

\centerline{\bf Solution to problem 1437 Elem. Math. 78 (2023), 135}

\bigskip
 

 \centerline {Raymond Mortini, Rudolf Rupp } \medskip

\centerline{- - - - - - - - - - - - - - - - - - - - - - - - - - - - - - - - - - - - - - - - - - - - - - - - - - - - - -}
  
  \medskip

This is well known since "hundreds of years" (see \cite{greg}).\\

(a)
First note that $f(x)=-\int_0^1 (1-x)^s\;ds$. In fact
\begin{eqnarray*}
-\int_0^1 (1-x)^s\;ds&=& -\int_0^1e^{s\log(1-x)}\;ds=-\left[ \frac{e^{s\log(1-x)}}{\log(1-x)}\right]^1_0\\
&=&-\left( \frac{1-x}{\log(1-x)}-\frac{1}{\log(1-x)}\right)=\frac{x}{\log(1-x)}.
\end{eqnarray*}
Now 
$$(1-x)^s=\sum_{k=0}^\infty (-1)^k {s\choose k} x^k.$$
Since $\int\sum=\sum\int$ (due to uniform convergence in $s$ for every fixed $x\in\; ]-1,1[$; note that $ |(-1)^k {s\choose k}|\leq 1$ ),
 we obtain
$$
f(x)=\sum_{k=0}^\infty \left[(-1)^{k+1}\int_0^1  {s\choose k} \;ds \right] x^k.
$$

Hence
\begin{eqnarray*}
a_k&=&(-1)^{k+1}\int_0^1  {s\choose k} \;ds= (-1)^{k+1} \int_0^1 \frac{s(s-1)(s-2)\cdots(s-k+1)}{k!}\;ds\\
&=&\int_0^1 \frac{s(1-s)(2-s)\cdots(k-1-s)}{k!}\;ds\\
&=&\frac{1}{k}\int_0^1 s \left(1- \frac{s}{1}\right)\left(1-\frac{s}{2}\right)\cdots\left(1-\frac{s}{k-1}\right)\;ds.
\end{eqnarray*}
Since every factor is less than $1$, we obtain
$$0\leq a_k\leq \frac{1}{k}.$$
Moreover, as $0\leq s\leq 1$, and for $k\geq 2$,
$$a_k\geq \int_0^1 s(1-s)  \frac{1\cdot 2\cdot 3\cdot 4\dots (k-2)}{k!} \;ds= \left(\frac{1}{2}-\frac{1}{3}\right) \frac{1}{k(k-1)}
= \frac{1}{6k(k-1)}.$$

The right-hand side, though,  is  smaller than $(3k^2)^{-1}$. So we need a more careful estimate. Instead of "breaking"  after the second factor, we brake after the fifth factor.  Noticing that
$$\int_0^1 s(1-s)(2-s)(3-s)(4-s)\;d s=\frac{9}{4},
$$
we obtain for $k\geq 5$,
$$a_k\geq \frac{9}{4}\;   \frac{ 4\dots 5\dots (k-2)}{k!} = \frac{9}{4} \frac{1}{3!}\frac{1}{k(k-1)}=
\frac{3}{8}\;\frac{1}{k^2}\geq \frac{1}{3}\;\frac{1}{k^2}.
$$

The estimate $a_k\geq \frac{1}{3k^2}$ now holds also for $k=1,\dots,4$, due to the following explicit representation of the Taylor sums for $f(x)$:

  \begin{figure}[h!]
 
  \scalebox{0.45} 
  {\includegraphics{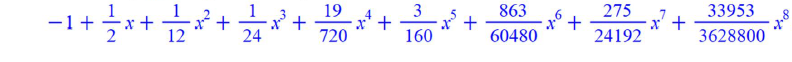}} 

\end{figure}

(b) We first show that $(a_k)$ is decreasing (to $0$):

\begin{eqnarray*}
a_{k+1}&=& \frac{1}{k+1}\int_0^1 s \left(1- \frac{s}{1}\right)\left(1-\frac{s}{2}\right)\cdots\left(1-\frac{s}{k-1}\right)
\underbrace{\left(1-\frac{s}{k}\right)}_{\leq 1}\;ds\\
&\leq& \frac{k}{1+k} a_k\leq a_k.
\end{eqnarray*}
The alternating series test of Leibniz  now yields the convergence of $\sum_{k=0}^\infty (-1)^k a_k$. Finally, by Abel's rule (see
 \cite[p. 1415]{moru}), 
$$\sum_{k=0}^\infty (-1)^k a_k=\lim_{x\to -1^+} f(x)=-\frac{1}{\log 2}.$$

Next we show that $ka_k\to 0 $. In fact, due to $1-x\leq e^{-x}$ for $x\geq 0$, 
\begin{eqnarray*}
ka_k&=& \int_0^1 s \left(1- \frac{s}{1}\right)\left(1-\frac{s}{2}\right)\cdots\left(1-\frac{s}{k-1}\right)\;ds\\
&\leq&\int_0^1 \exp\left(-s \sum_{j=1}^{k-1} \frac{1}{j}\right)\; ds= -\left[\frac{\dis\exp\left(-s \sum_{j=1}^{k-1} \frac{1}{j}\right)}
{\dis\sum_{j=1}^{k-1}\frac{1}{j}}\right]_0^1\\
&\leq& \frac{1}{\phantom{x}\dis \sum_{j=1}^{k-1} \frac{1}{j}\phantom{x}}\leq \frac{1}{\phantom{x}\dis\int_1^k dx/x\phantom{x}}=
\frac{1}{\log k}.
\end{eqnarray*}

Note that $\lim_{x\to 1^-} f(x)=0$.
Hence, by Tauber's Theorem (\cite[p. 52]{la-ga}), $F:=\sum_{k=0}^\infty a_k$ is convergent and $\sum_{k=0}^\infty a_k=0$. 
Very funny!  By the way, $F$ is called Fontana's series (Gregorio Fontana 1735--1803), see \cite[(formula 20)]{bl}, and the $a_k$ are the (moduli) of the Gregory coefficients (James Gregory 1638--1675), see
\cite{bl} and \cite{greg}.

\newpage


  \begin{figure}[h!]
 
  \scalebox{0.45} 
  {\includegraphics{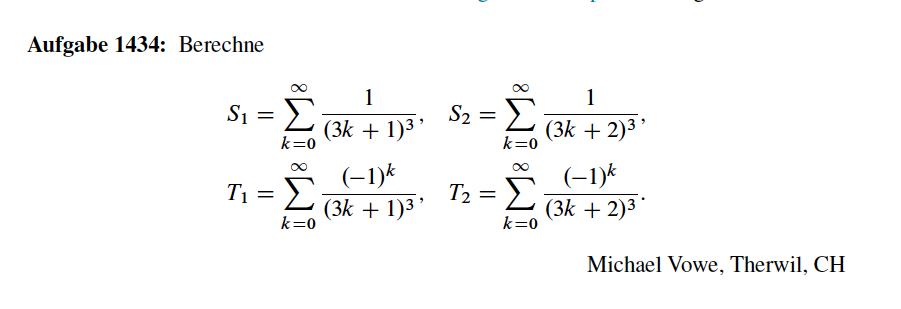}} 

\end{figure}

\centerline{\bf Solution to problem 1434 Elem. Math. 77 (2022), 85}

\bigskip
 

 \centerline {Raymond Mortini, Rudolf Rupp } \medskip

\centerline{- - - - - - - - - - - - - - - - - - - - - - - - - - - - - - - - - - - - - - - - - - - - - - - - - - - - - -}
  
  \medskip

 Let 
$$\mbox{$\dis S_3:=\sum_{n=0}^\infty \frac{1}{(3n+3)^3}$ and $\dis T_3:=\sum_{n=0}^\infty \frac{(-1)^n}{(3n+3)^3}$}.$$
Then 
$$\mbox{$\dis S_3=\frac{1}{27} \zeta(3)$ and $\dis T_3=\frac{1}{27} \eta(3)$}$$
where $\eta(3)$ is the Dirichlet $\eta$-function.  It is well-known that $\eta(3)=(3/4)\zeta(3)$. Thus
$$\mbox{$S_1+S_2+S_3=\zeta(3)$ and $T_1-T_2+T_3=\eta(3)$}$$
imply that
$$\mbox{$S_1+S_2=\frac{26}{27}\zeta(3)$ and $T_1-T_2= \frac{13}{18} \zeta(3)$}.$$

On the other hand,
\begin{eqnarray*}
S_1-T_1&=&\sum_{n=0}^\infty \frac{1-(-1)^n}{(3n+1)^3}\rel=_{n=2k+1}^{}2\; \sum_{k=0}^\infty\frac{2}{(3(2k+1)+1)^3}\\
&=&2\;\sum_{k=0}^\infty \frac{1}{(6k+4)^3}=\frac{1}{4}\;\sum_{k=0}^\infty \frac{1}{(3k+2)^3}\\
&=& \frac{1}{4} S_2.
\end{eqnarray*}
Moreover,
\begin{eqnarray*}
S_2+T_2&=&\sum_{n=0}^\infty \frac{1+(-1)^n}{(3n+2)^3}\rel=_{n=2k}^{}2\;\sum_{k=0}^\infty \frac{1}{(6k+2)^3} \\
&=&=\frac{1}{4}\;\sum_{k=0}^\infty \frac{1}{(3k+1)^3}\\
&=& \frac{1}{4} S_1.
\end{eqnarray*}

This yields the linear system
$$\left(\begin{matrix} 1 &1&0&0\\ 1&-\frac{1}{4}&-1&0\\ -\frac{1}{4}&1&0&1\\ 0&0&1&-1 
\end{matrix}\right)\; \cdot \left(\begin{matrix} S_1\\S_2\\T_1\\T_2  \end{matrix}\right)=
\left(\begin{matrix} \frac{26}{27}\zeta(3)\\0\\0\\\frac{13}{18}\zeta(3)\end{matrix}\right).
$$

The determinant is, unfortunately, zero. The null-space is the one dimensional vector space generated by
$(4,-4,5,5)^\bot$.

We may write  the system as

$$\left(\begin{matrix} 1 &1&0&0\\ 1&-\frac{1}{4}&-1&0\\ -\frac{1}{4}&1&0&1\\ 0&0&1&1 
\end{matrix}\right)\; \cdot \left(\begin{matrix} S_1\\S_2\\T_1\\T_2  \end{matrix}\right)=
\left(\begin{matrix} \frac{26}{27}\zeta(3)\\0\\0\\\frac{13}{18}\zeta(3)+2T_2\end{matrix}\right).
$$
In other words,
$$\left(\begin{matrix} S_1\\S_2\\T_1\\T_2  \end{matrix}\right)=
\left(\begin{matrix} \frac{1}{2}&\frac{2}{5}&\frac{-2}{5}&\frac{2}{5}\\\\
\frac{1}{2}&\frac{-2}{5}&\frac{2}{5}&\frac{-2}{5}\\\\
\frac{3}{8}&\frac{-1}{2}&\frac{-1}{2}&\frac{1}{2}\\\\
\frac{-3}{8}&\frac{1}{2}&\frac{1}{2}&\frac{1}{2}
 \end{matrix}\right)\cdot
\left(\begin{matrix} \frac{26}{27}\zeta(3)\\0\\0\\\frac{13}{18}\zeta(3)+2 T_2\end{matrix}\right).
 $$

That is
\begin{equation}\label{system}
\begin{cases}
S_1=&R_1+\frac{4}{5}T_2\\
S_2=&R_2-\frac{4}{5}T_2\\
T_1=&R_3+T_2\\
T_2=&\phantom{~~~R_3+}T_2
\end{cases}
\end{equation}
where the $R_j$ are rational multiples of $\eta(3)$.  More precisely,
$$R_1=\frac{1}{2}\cdot\frac{26}{27}\zeta(3) +\frac{2}{5}\cdot\frac{13}{18}\zeta(3)=\frac{104}{5\cdot 27}\cdot\zeta(3)
,$$
$$R_2=\frac{1}{2}\cdot\frac{26}{27}\zeta(3) -\frac{2}{5}\cdot\frac{13}{18}\zeta(3)=\frac{26}{5\cdot 27}\cdot\zeta(3)
,$$
$$R_3=\frac{3}{8}\cdot\frac{26}{27} \zeta(3) +\frac{1}{2}\cdot \frac{13}{18}\zeta(3)=\frac{13}{18}\cdot\zeta(3).
$$
Next we use that  for $0<a<2\pi$ (see below)
\begin{equation}\label{altesh}
h(a):=\sum_{n=0}^\infty \frac{\sin(n+1)a}{(n+1)^3}=\frac{a^3 -3\pi a^2+2 \pi^2 a}{12},
\end{equation}
and put $a=2\pi/3$.
Then, since $\sin(2\pi/3)=\sqrt 3/2$, 
$$S_1-S_2=\frac{2}{\sqrt 3}\;h(2\pi/3)= \frac{2}{\sqrt 3}\cdot \frac{2\pi^3}{81}= \frac{4\pi^3}{81 \sqrt 3}.$$
By the formula (\ref{system}) above,
$$S_1-S_2=\frac{4}{5} \cdot\frac{13}{18}\zeta(3) + \frac{8}{5} T_2.$$
Hence
$$T_2= \frac{5}{8}\left(  \frac{4\pi^3}{81 \sqrt 3}-\frac{4}{5}  \cdot\frac{13}{18}\zeta(3)\right)=
\ovalbox{$\dis \frac{5\pi^3}{2\cdot 81\;\sqrt 3 }- \frac{13}{36}\zeta(3)$}\sim 0.11843\cdots.$$
Finally, by  (\ref{system}) again,
$$
\ovalbox{$\dis S_1= \frac{13}{27}\;\zeta(3) + \frac{2\pi^3}{81 \sqrt 3}$}= 3^{-5}( 117\, \zeta(3) +2 \sqrt 3\; \pi^3)\sim 1.02078\cdots,
$$
$$
\ovalbox{$\dis S_2= \frac{13}{27}\;\zeta(3)- \frac{2\pi^3}{81 \sqrt 3}$}= 3^{-5}( 117\, \zeta(3) -2 \sqrt 3\; \pi^3) \sim 0.13675\cdots,
$$
$$
\ovalbox{$\dis T_1=\frac{13}{36}\;\zeta(3)+ \frac{5\pi^3}{2\cdot 81\;\sqrt 3 }$}=
\frac{1}{1944}\left( 702\,\zeta(3) +20\sqrt 3\;\pi^3\right)  \sim 0.98659\cdots.
$$

\vspace{1cm}

{\bf Addendum}

The value in (\ref{altesh}) for $h(a)$ is given as follows (we had developed this in  solving the problem 12388 in AMM).
Note that 
$$h'(a)=\sum_{n=0}^\infty  \frac{\cos(n+1)a}{(n+1)^2}.
$$
Since $\frac{1}{3}\pi^2 + 4 \sum_{n=1}^\infty \frac{\cos nx}{n^2}$ is the  Fourier series of the  function $(x-\pi)^2$, $0\leq x<  2\pi$, 
extended $2\pi$-periodically,
we see that for $0<a<2\pi$,
$$h'(a)= \frac{(a-\pi)^2}{4}-\frac{\pi^2}{12}.$$

As $h(0)=0$, we deduce that for $0<a<2\pi$,
$$h(a)= \frac{(a-\pi)^3}{12}-\frac{\pi^2}{12} \,a+ \frac{\pi^3}{12}=\frac{a^3 -3\pi a^2+2 \pi^2 a}{12}.
$$

 \newpage
 
  \begin{figure}[h!]
 
  \scalebox{0.35} 
  {\includegraphics{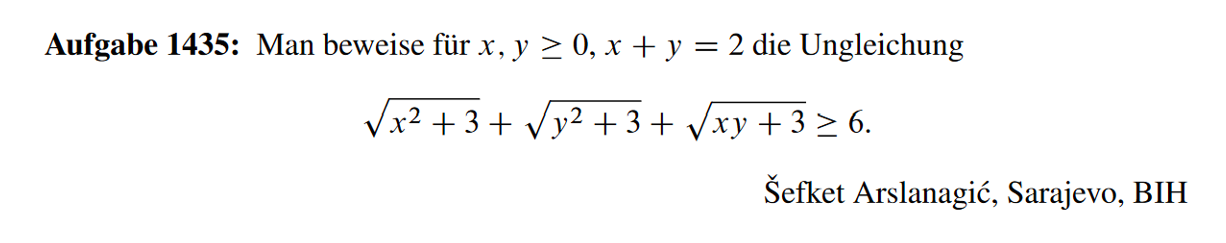}} 

\end{figure}

\centerline{\bf Solution to problem 1435 Elem. Math. 77 (2022), 85}

\bigskip
 

 \centerline {Raymond Mortini, Rudolf Rupp } \medskip

\centerline{- - - - - - - - - - - - - - - - - - - - - - - - - - - - - - - - - - - - - - - - - - - - - - - - - - - - - -}
  
  \medskip

Let 
$$A(x,y):=\sqrt{x^2+3}+\sqrt{y^2+3}+\sqrt{xy+3}.$$

Since by assumption $x+y=2$, we just have to prove that $6$ is the minimal value of the function
$$f(x):=\sqrt{x^2+3}+\sqrt{(2-x)^2+3}+\sqrt{x(2-x)+3}, \quad 0\leq x\leq 2,$$
which is attained at $x=1$.

  \begin{figure}[h!]
 
  \scalebox{0.25} 
  {\includegraphics{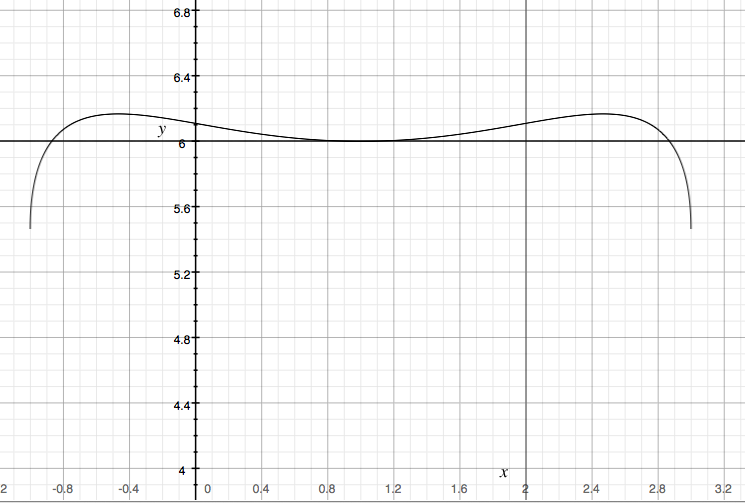}} 
\end{figure}

First we note that $f$ is symmetric with respect to $x=1$; that is 
$$\mbox{$f(1+x)=f(1-x)$ for $0\leq x\leq 1$.}$$

We show that $f$ decreases on $[0,1]$.  Note that $f(0)=2\sqrt 3+ \sqrt 7> 6 =f(1)$. It is sufficient to prove that $f'\leq 0$ on $[0,1]$:
$$f'(x)= \frac{x}{\sqrt{x^2+3}}- \frac{2-x}{\sqrt{(2-x)^2+3}}+ \frac{1}{2}\,\frac{2-2x}{\sqrt{x(2-x)+3}.}
$$
But  $f'\leq 0$ for $0\leq x\leq 1$  if and only if 
\begin{equation}\label{deri}
L(x):= \frac{x}{\sqrt{x^2+3}}+ \frac{1-x}{\sqrt{x(2-x)+3}}\leq \frac{2-x}{\sqrt{(2-x)^2+3}}.
\end{equation}
Next note that, due to $2-x\geq  x$,
$$
\frac{1-x}{\sqrt{x(2-x)+3}}\leq \frac{1-x}{\sqrt{x^2+3}}.
$$
Hence $L(x)\leq \frac{1}{\sqrt{x^2+3}}$. Thus (\ref{deri}) holds for $0\leq x\leq 1$ if 
\begin{equation}
\frac{1}{\sqrt{x^2+3}}\leq  \frac{2-x}{\sqrt{(2-x)^2+3}},
\end{equation}
or equivalently
\begin{equation}\label{deri2}
\frac{\sqrt{(2-x)^2+3}}{\sqrt{x^2+3}}\leq 2-x.
\end{equation}
This holds, though, due to the following equivalences for $0\leq x\leq1$:
$$
(2-x)^2+3 \leq (2-x)^2 (x^2+3)\iff 3\leq (2-x)^2 (x^2+2)=:R(x).
$$
The latter is true, since $\min_{0\leq x\leq 1} R(x)=R(1)=3$ ( note that the derivative of $R$ equals $R'(x)=-4(2-x)(x^2-x+1)$, so $R'\leq 0$ on $[0,1]$.)
\newpage

\pagecolor{yellow}

\begin{figure}[h!]
 
  \scalebox{0.45} 
  {\includegraphics{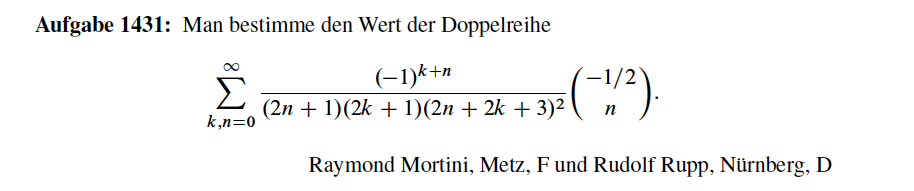}} 

\end{figure}

\centerline{\bf Partial solution to problem 1431,  Elem. Math. 78 (2023), 44}

\bigskip

 \centerline {Raymond Mortini, Rudolf Rupp } \medskip

\centerline{- - - - - - - - - - - - - - - - - - - - - - - - - - - - - - - - - - - - - - - - - - - - - - - - - - - - - -}
  
  \medskip

Let
  
  $$S:=\sum_{k,n=0}^\infty \frac{(-1)^{k+n+1}}{(2n+1)(2k+1)(2n+2k+3)^2}{-1/2\choose n}.$$
  Then 
$S$ is the value of the integral
  
  $$\int_0^1 \arcsin x\arctan x\log x\; dx.$$
  
Just take   the series
  $$\arcsin x =\sum_{n=0}^\infty \frac{(-1)^n}{2n+1}{-1/2\choose n} x^{2n+1}$$
  $$\arctan x= \sum_{n=0}^\infty \frac{(-1)^n}{2n+1}  x^{2n+1}$$ 
  and use that 
  $$\int_0^1 x^m\log x\; dx= -(m+1)^{-2}.$$
  

For a solution see \cite{Stad}.

\newpage

\nopagecolor

  \begin{figure}[h!]
 
  \scalebox{0.25} 
  {\includegraphics{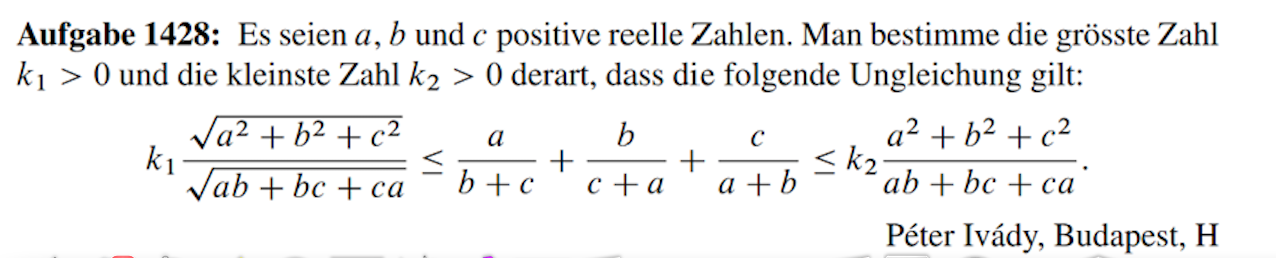}} 

\end{figure}

\centerline{\bf Partial solution to problem 1428 Elem. Math. 77 (2022), 196}

\bigskip

 \centerline {Raymond Mortini, Rudolf Rupp } \medskip

\centerline{- - - - - - - - - - - - - - - - - - - - - - - - - - - - - - - - - - - - - - - - - - - - - - - - - - - - - -}
  
  \medskip

Let 
$$f(a,b,c):= \left( \frac{a}{b+c}+ \frac{b}{c+a}+ \frac{c}{a+b}\right) \;  \frac{ab+bc+ca}{a^2+b^2+c^2}.
$$
Then, by using that $\frac{2xy}{x+y}\leq \sqrt{xy}\leq \frac{x+y}{2}$, 
\begin{eqnarray*}
f(a,b,c)&=&\frac{1}{a^2+b^2+c^2}\left( a\;  \frac{bc}{b+c}+ b\;  \frac{ca}{c+a} + c\;  \frac{ab}{a+b}\right)+1\\
&\leq& \frac{1}{a^2+b^2+c^2} \left(a\; \frac{b+c}{4}+  b\; \frac{c+a}{4}+ c\; \frac{a+b}{4}\right)+1\\
&=& \frac{1}{a^2+b^2+c^2} \frac{2ab+2bc+2ca}{4}+1\\
&=&\frac{1}{a^2+b^2+c^2} \frac{(a+b+c)^2-(a^2+b^2+c^2)}{4}+1\\
&\leq&\frac{1}{a^2+b^2+c^2}\frac{(a^2 +b^2+c^2) (1+1+1)-(a^2+b^2+c^2)}{4}+1\\
&\leq& \frac{3}{2}.
\end{eqnarray*}
If we let $a=b=c$, then $f(a,a,a)= 3/2$ and so \ovalbox{$k_2=3/2$}. 
To determine $k_1$, let
$$ g(a,b,c):=  \left( \frac{a}{b+c}+ \frac{b}{c+a}+ \frac{c}{a+b}\right)^2 \;    \frac{ab+bc+ca}{a^2+b^2+c^2}.
$$
Then, by using two of the estimates above, namley $f\geq 1$, and Cauchy-Schwarz,
$$g(a,b,c)=\frac{a^2+b^2+c^2}{ab+bc+ca}\; f(a,b,c)^2 > \frac{a^2+b^2+c^2}{ab+bc+ca}\geq 1.$$

We guess $k_1=\sqrt 2$. In fact, we may restrict to triples $(x,1,c)$ (homogeniety). Then it remains to prove that 

$$f_c(x):=\left(\frac{1}{x+c}+\frac{x}{c+1}+\frac{c}{1+x}\right)^2 \left(\frac{x+cx+c}{1+x^2+c^2}\right)\geq 2.$$

Now $$\lim_{c\to 0} f_c(x)= x+\frac{1}{x}=f_0(x)\geq 2.$$
Graphical evidence seems to indicate that
$m_c:=\min_{x>0}f_c(x)\geq 2$ and $\lim_{c\to 0}   m_c=2$.

As it is customn with this type of questions, the infimum of the two-variable function $f(x,c):=f_c(x)$ is taken on the boundary of the first quadrant; that is when $c=0$. We have no proof though of this last claim.

\newpage

\pagecolor{yellow}
\begin{figure}[h!]
 
 \scalebox{0.4} 
  {\includegraphics{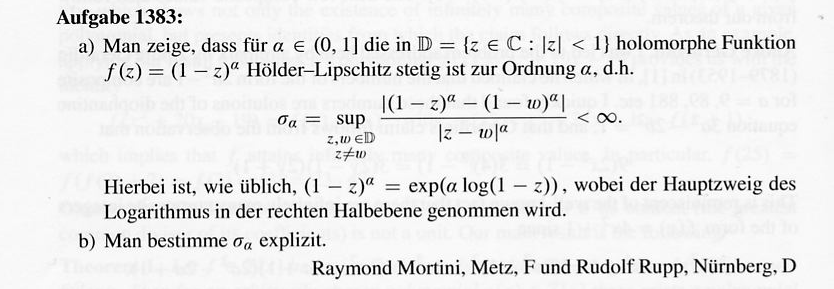}} 

\end{figure}

\centerline{\bf  problem 1383  in Elem. Math 74 (2019), 38, by}\medskip

\centerline{Raymond Mortini and  Rudolf Rupp} \medskip

\centerline{- - - - - - - - - - - - - - - - - - - - - - - - - - - - - - - - - - - - - - - - - - - - - - - - - - - - - -}
  
  \medskip

   \begin{theorem}\label{lippi}
    Let $0<\alpha\leq 1$. Then 
 
\begin{eqnarray*}
\sigma(\alpha)&:=& \dis\sup\left\{\frac{|(1-z)^\alpha-(1-w)^\alpha|}{|z-w|^\alpha}: |z|,|w|\leq 1, z\not=w\right\}\\
&=& \max\{1, 2^{1-\alpha}\sin(\alpha\pi/2)\}\\
&=&\begin{cases}
   1 & \text {if $0<\alpha\leq 1/2$}\\
  2^{1-\alpha}\sin(\alpha\pi/2) & \text{if $1/2\leq \alpha\leq1$.}
   \end{cases}
   \end{eqnarray*}
 
  Moreover, 
 
  $$\dis\max_{0<\alpha\leq 1} \log \sigma(\alpha)= 
   \left(1-\frac{2}{\pi}\arctan\left(\frac{\pi}{2\log 2}\right)\right)\log 2+
  \log \left(    \frac{\pi}{\sqrt{\pi^2+4(\log2)^2}}\right).$$
  
\end{theorem}

 \begin{figure}[h!]
  \scalebox{.45} 
  {\includegraphics{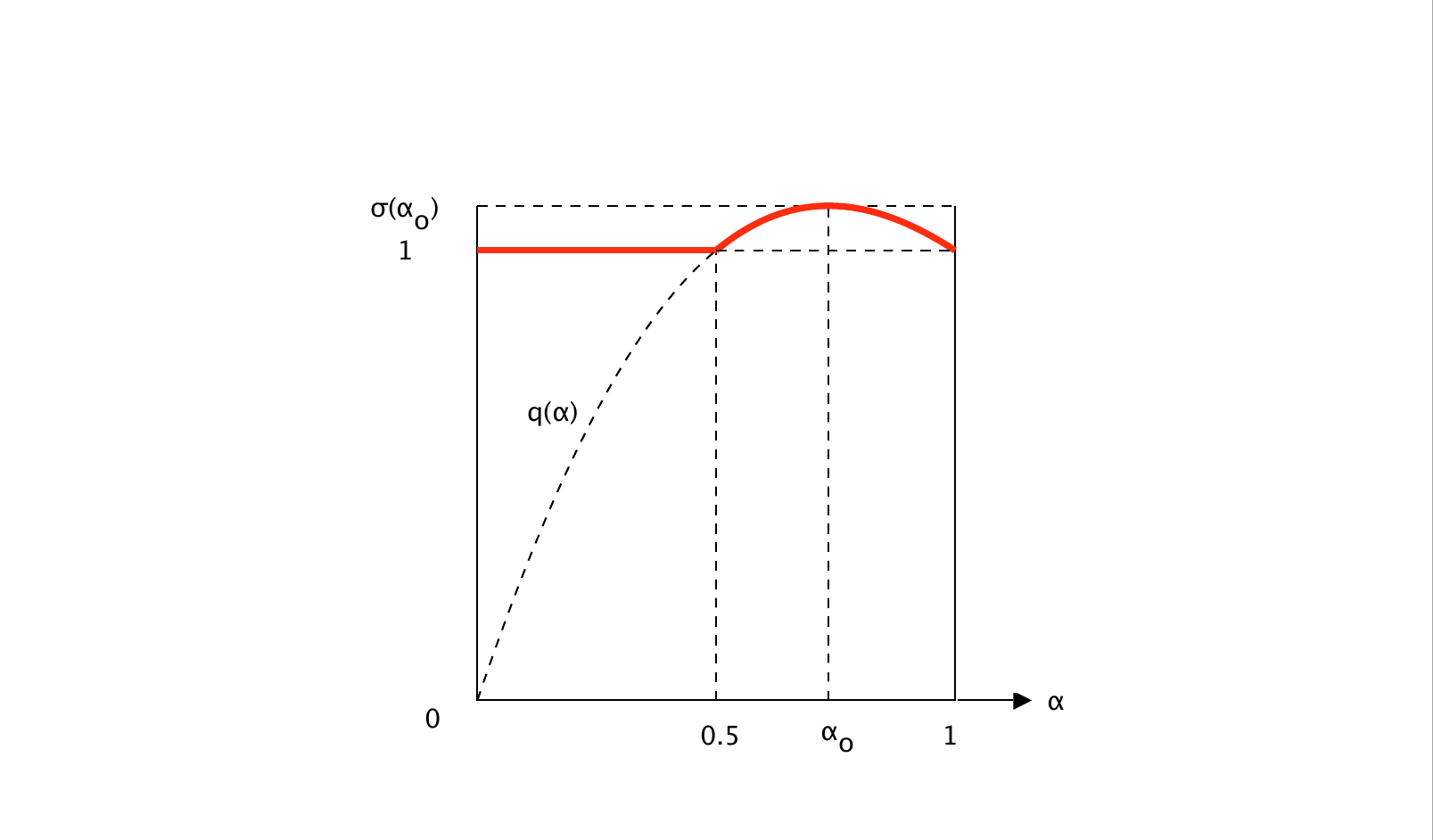}} 
\caption{\label{max-sigma} The H\"older-Lipschitz constant $\sigma(\alpha)$}
\end{figure}

See  \cite{mr20} for  this best H\"older-Lipschitz constant associated with $(1-z)^\alpha$.

\newpage

\begin{figure}[h!]
 
 \scalebox{0.85} 
  {\includegraphics{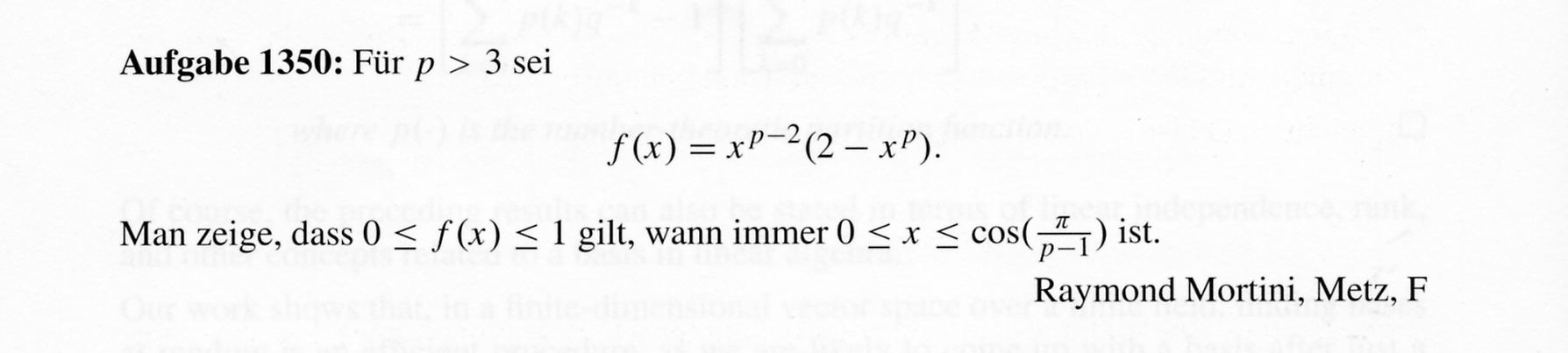}} 

\end{figure}

\centerline{\bf Solution to problem 1350  in Elem. Math 71 (2016), 84}\medskip

\centerline{Raymond Mortini} \medskip

\centerline{- - - - - - - - - - - - - - - - - - - - - - - - - - - - - - - - - - - - - - - - - - - - - - - - - - - - - -}
  
  \medskip

\begin{figure}[h!]
 
 \scalebox{0.5} 
  {\includegraphics{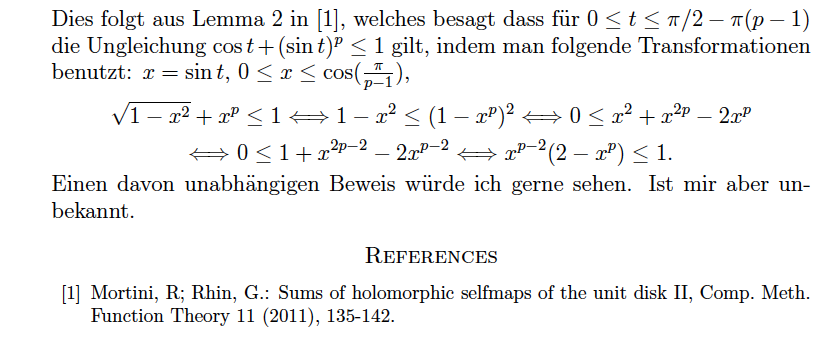}} 

\end{figure}

\newpage
\nopagecolor

\begin{figure}[h!]
 
 \scalebox{0.5} 
  {\includegraphics{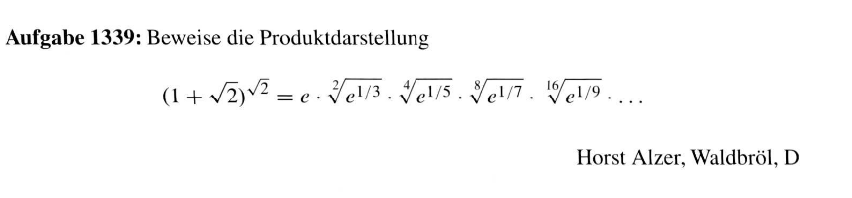}} 

\end{figure}

\centerline {\bf Solution to  problem 1339  Elem. Math.  70 (2015), 82.}\medskip

\centerline{Raymond Mortini}
\medskip

\centerline{- - - - - - - - - - - - - - - - - - - - - - - - - - - - - - - - - - - - - - - - - - - - - - - - - - - - - -}
  
  \medskip

\begin{figure}[h!]
 
 \scalebox{0.5} 
  {\includegraphics{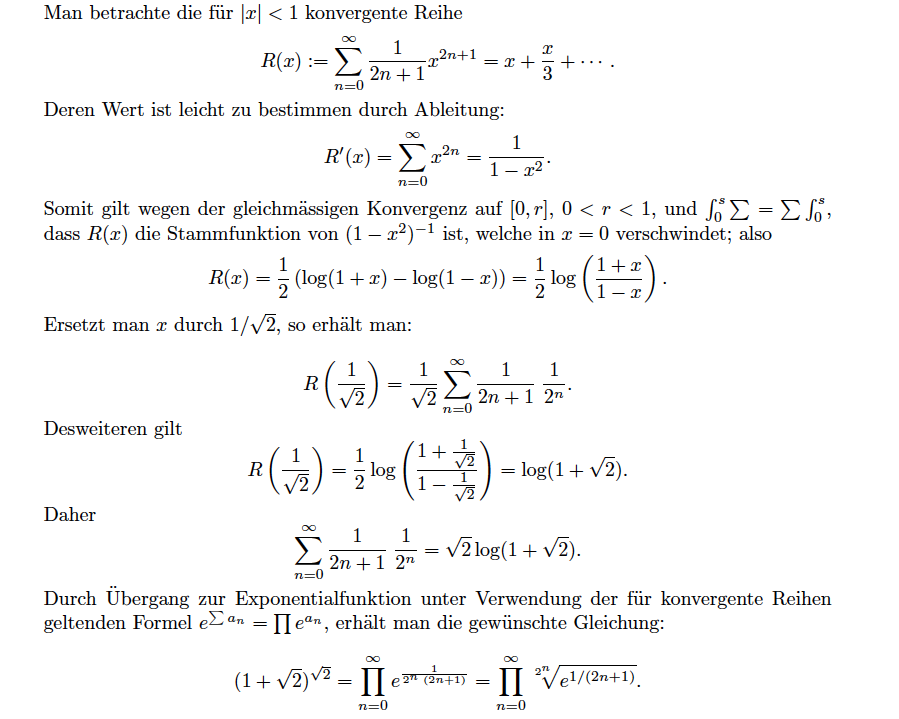}} 

\end{figure}

\newpage

\pagecolor{yellow}
\begin{figure}[h!]
 
 \scalebox{0.4} 
  {\includegraphics{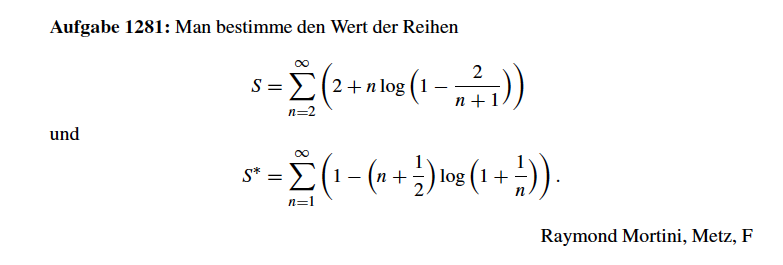}} 

\end{figure}

\centerline{\bf Solution to problem 1281 Elem. Math.  65 (2010), 127, by} \medskip

\centerline{Raymond Mortini, J\'er\^ome No\"el }

\medskip

\centerline{- - - - - - - - - - - - - - - - - - - - - - - - - - - - - - - - - - - - - - - - - - - - - - - - - - - - - -}
  
  \medskip

 a)  Exponentiating, we  have to  calculate the  value of the infinite product
  $$P=\prod_{n=2}^\infty \left( e^2\left( \frac{n-1}{n+1}\right)^n\right).$$
  We claim that $P= \frac{4\pi}{e^3}$; so $S= \log(4\pi) -3$.
  \medskip
  
  Let $P_N=\prod_{n=2}^N \left( e^2\left( \frac{n-1}{n+1}\right)^n\right)$.
  Then by Stirlings formula,  telling us that $n! \sim e^{-n}n^n \sqrt{2\pi n}$, we obtain

  $$P_N= \frac{1}{e^2} e^{2N}\frac{\prod_{n=2}^N (n-1)^n}{\prod_{n=2}^N (n+1)^n}=$$
 $$ \frac{1}{e^2} e^{2N}\frac{\prod_{n=2}^N (n-1)^n}{\mathbf {\prod_{n=2}^N n^{n+1}}}
 \; \frac{\mathbf {\prod_{n=2}^N n^{n+1}}}{\prod_{n=2}^N (n+1)^n}=$$
 $$ \frac{2}{e^2} \frac{e^{2N} \,(N!)^2}{ N^{N+1} (N+1)^N}=
 \frac{2}{e^2} \left( \frac{e^N N!}{N^N}\right)^2 \,\frac{N^N}{(N+1)^N} \; \frac{1}{N}\sim$$
 $$ \frac{2}{e^2}\, \frac{(\sqrt{2\pi N})^2}{N}\, \frac{1}{(1+\frac{1}{N})^N}\to
  \frac{4\pi}{e^3}.
  $$
  \vs 0,2cm
  
  b) To determine $S^*$, we use the same method and calculate the value of
  $$P^*=\prod_{n=1}^\infty e\left(\frac{n}{n+1} \right)^{n+\frac{1}{2}}$$
  We claim that $P^*= \frac{\sqrt {2\pi}}{e}$ and so $S^*= \frac{1}{2}\log (2\pi) - 1$.
  \medskip
  
  In fact 
  
  $$P_N^*=\prod_{n=1}^N e\left(\frac{n}{n+1} \right)^{n+\frac{1}{2}}=
  \frac{e^N N!}{(N+1)^N}\,\frac{1}{\sqrt{N+1}}.$$
  Using Stirling's formula we obtain
  $$
  P_N\sim \frac{N^N}{(N+1)^N }\sqrt{2\pi N}\,\frac{1}{\sqrt{N+1}}=
  \sqrt{2\pi} \frac{1}{(1+\frac{1}{N})^N}\, \frac{\sqrt N}{\sqrt{N+1}}
  \to  \frac{\sqrt{2\pi}}{e}.
  $$

\newpage
.\vspace{-12mm}
\nopagecolor

\begin{figure}[h!]
 
 \scalebox{0.38} 
  {\includegraphics{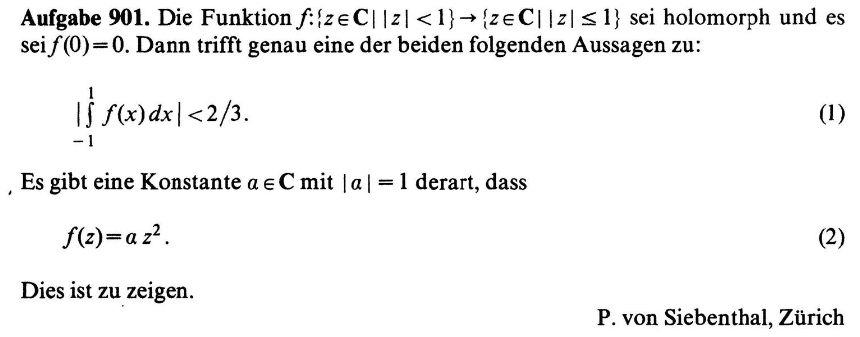}} 

\end{figure}

\centerline {\bf Solution to  problem 901  Elem. Math.  38 (1983), 128.}\medskip

\centerline{Raymond Mortini}

  


\begin{figure}[h!]
 
 \vspace{-3.3mm}
 \scalebox{0.6} 
  {\includegraphics{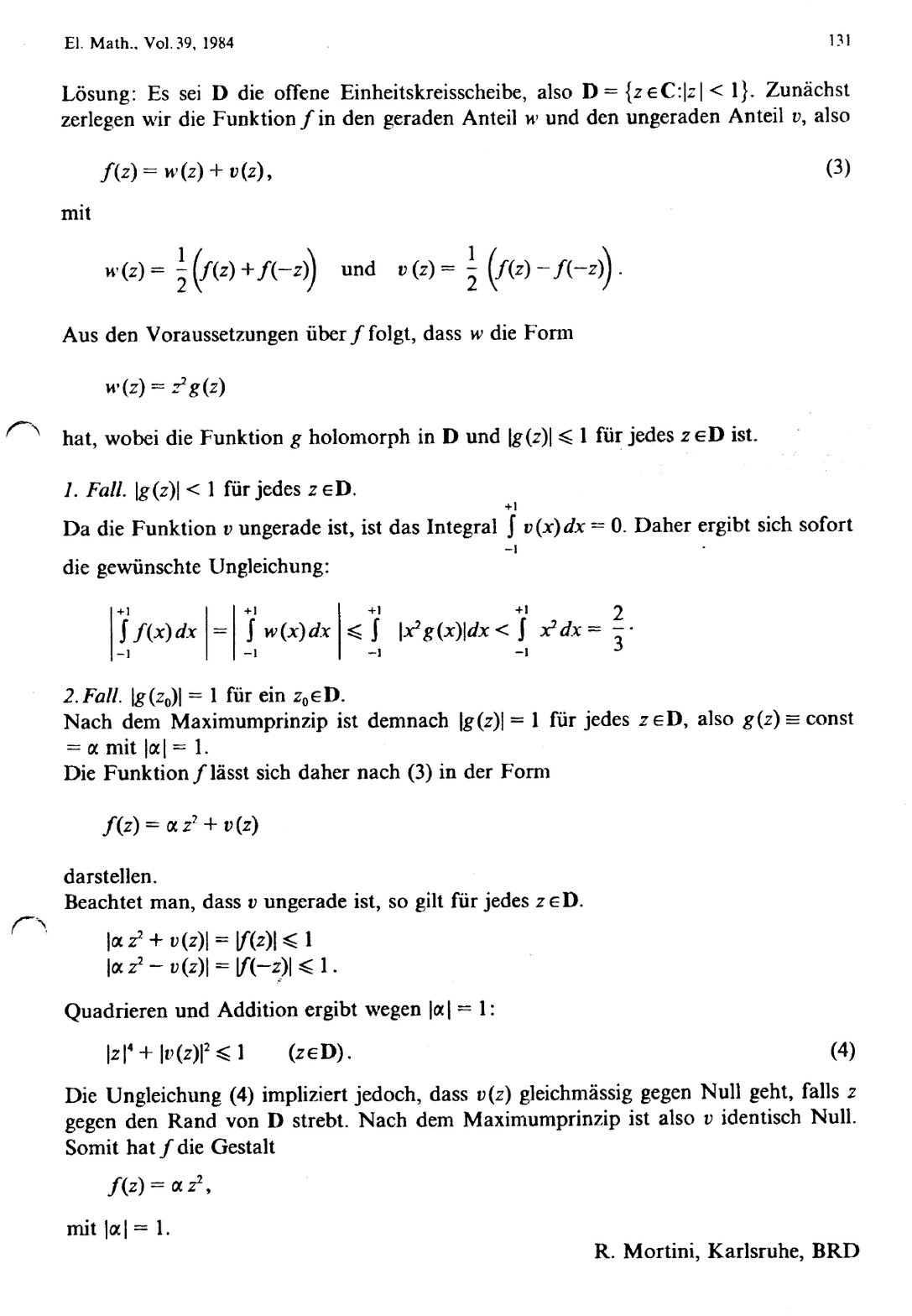}} 

\end{figure}

\newpage

.
\vspace{-2cm}
\gr{
{\huge{\section{Crux Mathematicorum}}}
}

\begin{figure}[h!]
 
  \scalebox{0.45} 
  {\includegraphics{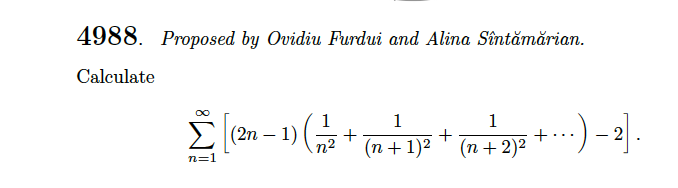}} 
\end{figure}


\centerline{\bf Solution to problem 4988 Crux Math. 50 (9) 2024, 466}  \medskip

 \centerline {Raymond Mortini, Rudolf Rupp}
 
 \medskip

\centerline{- - - - - - - - - - - - - - - - - - - - - - - - - - - - - - - - - - - - - - - - - - - - - - - - - - - - - -}
  
  \medskip

We show that
$$S:=\sum_{n=1}^\infty\left[ (2n-1)\sum_{j=0}^\infty \frac{1}{(j+n)^2} -2\right]=-\frac{1}{2}.$$

The series $\dis \sum_{j=0}^\infty \frac{1}{(j+x)^2}$, denoted by $\psi^{(1)}(x)$ or $\zeta(2,x)$, is sometimes called the trigamma function. It is a special Hurwitz-zeta function. It can be represented as an integral: if $x>0$, then
$$\psi^{(1)}(x)= \int_0^\infty \frac{t e^{(1-x)t}}{e^t-1}\;dt.$$
Now, via induction (thanks to the software wolframalpha; attention: that software says that the series diverges, sic!)
$$\sum_{n=1}^N [(2n-1)\psi^{(1)}(n)-2]= -N+N^2\psi^{(1)}(N+1).$$
Thus it remains to show that
$$I_N:= \left( -N+N^2 \int_0^\infty  \frac{t }{e^t-1}\;e^{ -Nt}dt\right) \to -\frac{1}{2}.$$
This is done using twice partial integration:

\begin{eqnarray*}
I_N&=&-N+N^2\left[-\frac{e^{-Nt}}{N} \frac{t }{e^t-1}\Bigg|_0^\infty+\frac{1}{N}\int_0^\infty e^{-Nt} \;\frac{e^t-1-te^t}{(e^t-1)^2}\;dt\right]\\
&=&-N+N^2\left[ \frac{1}{N}+\frac{1}{N}\int_0^\infty e^{-Nt}  \;\frac{e^t-1-te^t}{(e^t-1)^2}\;dt\right] =N\int_0^\infty e^{-Nt} 
\;\frac{e^t-1-te^t}{(e^t-1)^2}\;dt\\
&=&N\left[ -\frac{e^{-Nt}}{N}\;\frac{e^t-1-te^t}{(e^t-1)^2}\Bigg|_0^\infty +\frac{1}{N}\int_0^\infty e^{-Nt}\;
\underbrace{\frac{-(e^t-1) te^t -(e^t-1-te^t)2 e^t}{(e^t-1)^3}}_{:=b(t)} \;dt\right]\\
&\buildrel=_{{\rm l'Hospital}}^{t\to 0}&-\frac{1}{2} +\int_0^\infty  e^{-Nt} b(t) dt=-\frac{1}{2} +\e_N.
\end{eqnarray*}

Since $b$ is bounded on $]0,\infty[$ and $\int_0^\infty  e^{-Nt}dt =\frac{1}{N}$, we have that $\e_N\to 0$.

\newpage

\begin{figure}[h!]
 
  \scalebox{0.4} 
  {\includegraphics{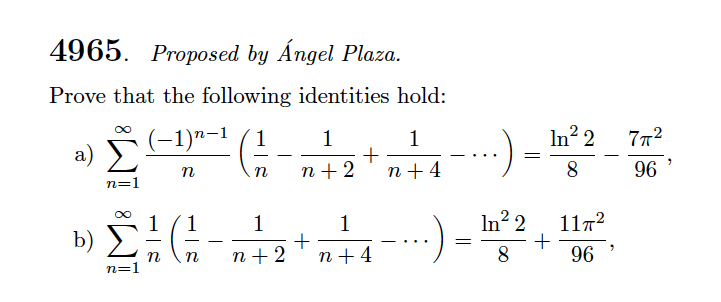}} 
\end{figure}

\nopagecolor

\centerline{\bf Solution to problem 4965 Crux Math. 50 (7) 2024, 368}\bigskip

 \centerline {Bikash Chakraborty, Raymond Mortini, Rudolf Rupp}
  \medskip

\centerline{- - - - - - - - - - - - - - - - - - - - - - - - - - - - - - - - - - - - - - - - - - - - - - - - - - - - - -}
  
  \medskip
 
We observe that the sign in a) is false.\\

{\bf Step 1} {\it Reduction to the computation of an integral}\\

We use that for $a>0$,
$$\frac{1}{a}=\int_0^\infty e^{-ax}dx.$$
Since the moduli of the  partial sums $\sum_{k=0}^N  e^{-nx} (-1)^k e^{-2kx} $ are bounded by the $L^1([0,\infty[)$-function $2e^{-nx}$,
 we have $\sum\int=\int\sum$, and so
\begin{eqnarray*}
S_n:=\frac{1}{n}-\frac{1}{n+2}+\dots &=& \sum_{k=0}^\infty \frac{(-1)^k}{n+2k}= \sum_{k=0}^\infty (-1)^k \int_0^\infty e^{-(n+2k)x} dx\\
&=&\int_0^\infty e^{-nx}\;\sum_{k=0}^\infty (-1)^k e^{-2kx} dx= \int_0^\infty \frac{e^{-nx}}{1+e^{-2x}}dx.
\end{eqnarray*}

Now let
$$A:=\sum_{n=1}^\infty\frac{(-1)^{n-1}}{n}\; \left (\frac{1}{n}-\frac{1}{n+2}+\dots\right)$$
and 
$$B:=\sum_{n=1}^\infty\frac{1}{n}\; \left( \frac{1}{n}-\frac{1}{n+2}+\dots\right).$$

Note that $S_n$ is decreasing, so the series $A$ converges (Leibniz rule). But actually,
$$0\leq S_n\leq  \int_0^\infty e^{-nx} dx= \frac{1}{n}.$$
So the series $A$ and $B$ converge absolutely as the general term is $\Oh(1/n^2)$.

Due to the boundedness of the partial sums of the  series $\sum_{n=1}^\infty \frac{1}{n} e^{-nx}$
 by the $L^1(]0,\infty[)$-function $|\log(1-e^{-x})|$, we have $\sum\int=\int\sum$ in both cases. Hence

\begin{eqnarray*}
B&=& \sum_{n=1}^\infty\frac{1}{n}\;  \int_0^\infty \frac{e^{-nx}}{1+e^{-2x}}dx=
\int_0^\infty \frac{1}{1+e^{-2x}}\sum_{n=1}^\infty \frac{1}{n} e^{-nx} dx\\
&=&-\int_0^\infty \frac{\log(1-e^{-x})}{1+e^{-2x}}dx\buildrel=_{x\mapsto e^{-x}}^{} -\int_0^1 \frac{\log(1-x)}{x(1+x^2)}\; dx.
\end{eqnarray*}

\begin{eqnarray*}
A&=& \sum_{n=1}^\infty\frac{(-1)^{n-1}}{n}\;  \int_0^\infty \frac{e^{-nx}}{1+e^{-2x}}dx=
\int_0^\infty \frac{1}{1+e^{-2x}}\sum_{n=1}^\infty \frac{(-1)^{n-1}}{n} e^{-nx} dx\\
&=&\int_0^\infty \frac{\log(1+e^{-x})}{1+e^{-2x}}dx\buildrel=_{x\mapsto e^{-x}}^{} \int_0^1 \frac{\log(1+x)}{x(1+x^2)}\; dx.
\end{eqnarray*}

{\bf Step 2} {\it Calculating the integrals}\\

To evaluate the integrals, we calculate $I:=A+B$ and $J:=A-B$ giving $A=(I+J)/2$ and $B=(I-J)/2$.

\begin{eqnarray*}
I&=& \bigintss_0^1 \frac{\log\left(\frac{1+x}{1-x}\right)}{x(1+x^2)} dx \buildrel=_{x=\frac{1-t}{1+t}} ^{ \frac{1-x}{1+x}=t} 
\bigintss_0^1\frac{-\log t}{\dis \left(\frac{1-t}{1+t} \right)\Big(1+ \left(\frac{1-t}{1+t}\right)^2\Big)} \frac{2dt}{(1+t)^2}\\
&=&-2\bigintss_0^1 \frac{(1+t)\log t}{(1-t) \big((1+t)^2+(1-t)^2\big)} dt= -2\bigintss_0^1\frac{(1+t)\log t}{(1-t)( 2+2t^2)}dt\\
&=&-\int_0^1 \frac{(1+t)\log t}{(1-t)(1+t^2)}dt=-\left(\int_0^1 \frac{\log t}{1-t}dt +\int_0^1\frac{t\log t}{1+t^2}dt\right)\\
&\buildrel=_{}^{s=t^2}&-\int_0^1 \frac{\log t}{1-t}dt  -\frac{1}{4}\int_0^1 \frac{\log s}{1+s} ds.
\end{eqnarray*}

\begin{eqnarray*}
J&= &\int_0^1 \frac{\log(1-x^2)}{x(1+x^2)}dx\buildrel=_{}^{x^2=s} \frac{1}{2}\int_0^1\frac{\log(1-s)}{s(1+s)}ds\\
&=&\frac{1}{2} \int_0^1 \frac{\log (1-s)}{s} ds -\frac{1}{2} \int_0^1 \frac{\log (1-s)}{1+s}ds\\
&=& \frac{1}{2} \int_0^1 \frac{\log t}{1-t} dt - \frac{1}{2} \int_0^1 \frac{\log t}{2-t}dt.
\end{eqnarray*}

Now
\begin{eqnarray*}
\int_0^1 \frac{\log t}{2-t}dt&\buildrel=_{}^{t=2u}&\int_0^{1/2} \frac{\log(2u)}{2-2u} 2du=\int_0^{1/2} \frac{\log 2+\log u}{1-u} du\\
&=& -(\log 2 )\Big[\log(1-u)\Big]_0^{1/2}+\int_0^{1/2} \frac{\log u}{1-u}\,du\\
&=& \log^2 2 +\int_0^{1/2} \frac{\log u}{1-u}\,du.
\end{eqnarray*}

Hence, by using that  for $0<u\leq 1$, $\dis -\int_0^u \frac{\log (1-x)}{x}dx={\rm Li}_2(u)=\sum_{n=1}^\infty \frac{u^n}{n^2}$ (di-logarithm),
we obtain with $0< v<1$,
\begin{eqnarray*}
\int_0^v \frac{\log x}{1-x} dx&\buildrel=_{x=1-s}^{}&\int_{1-v}^1 \frac{\log(1- s) }{s}ds=\int_0^1  \frac{\log(1- s) }{s}ds- 
\int_0^{1-v}\frac{\log(1- s) }{s}ds\\
&=&-{\rm Li}_2(1) +{\rm Li}_2(1-v).
\end{eqnarray*}
Consequently, by additionally using that 
 $$ \int_0^1 \frac{\log x}{1+x}dx=\sum_{n=0}^\infty \int_0^1 (-1)^n(\log x)x^n dx=\sum_{n=0}^\infty (-1)^{n+1}\frac{1}{(n+1)^2}=
 -\frac{\pi^2}{12},$$

\begin{eqnarray*}
2A= I+J&=&\left[ -\int_0^1 \frac{\log x}{1-x}dx  -\frac{1}{4}\int_0^1 \frac{\log x}{1+x} dx \right]+
\left[ \frac{1}{2} \int_0^1 \frac{\log x}{1-x} dx -\frac{1}{2} \left( \log^2 2 +\int_0^{1/2} \frac{\log x}{1-x}\right)\right]\\
&=&-\frac{1}{2}\log^2 2+\frac{1}{2}\;{\rm Li}_2 (1)- \frac{1}{2}\;\big( -{\rm Li}_2(1) +{\rm Li}_2 (\frac{1}{2})\big)
-\frac{1}{4} \sum_{n=0}^\infty \frac{(-1)^{n+1} }{(n+1)^2}\\
&=& -\frac{1}{2}\log^2 2 +\frac{\pi^2}{6}- \frac{1}{2}\left(\frac{ \pi^2}{12} -\frac{\log^2 2}{2}\right)+\frac{\pi^2}{48}\\
&=& -\frac{1}{4}\log^2 2 +\frac{7}{48}\pi^2,
\end{eqnarray*}
and

\begin{eqnarray*}
2B=I-J&=&\left[ -\int_0^1 \frac{\log x}{1-x}dx  -\frac{1}{4}\int_0^1 \frac{\log x}{1+x} dx \right]-
\left[ \frac{1}{2} \int_0^1 \frac{\log x}{1-x} dx -\frac{1}{2} \left( \log^2 2 +\int_0^{1/2} \frac{\log x}{1-x}\right)\right]\\
&=&\frac{1}{2}\log^2 2+\frac{3}{2}\;{\rm Li}_2 (1)+ \frac{1}{2}\;\big(- {\rm Li}_2(1) +{\rm Li}_2 (\frac{1}{2})\big)
-\frac{1}{4} \sum_{n=0}^\infty \frac{(-1)^{n+1} }{(n+1)^2}\\
&=& \frac{1}{2}\log^2 2 +\frac{\pi^2}{6}+ \frac{1}{2}\left(\frac{ \pi^2}{12} -\frac{\log^2 2}{2}\right)+\frac{\pi^2}{48}\\
&=& \frac{1}{4}\log^2 2 +\frac{11}{48}\pi^2.
\end{eqnarray*}

Consequently
$$A= -\frac{1}{8}\log^2 2 +\frac{7\pi^2}{96} \sim0.659602\dots$$
and

$$B= \frac{1}{8}\log^2 2 +\frac{11\pi^2}{96}\sim 1.19095\dots$$

\newpage

\begin{figure}[h!]
 
  \scalebox{0.5} 
  {\includegraphics{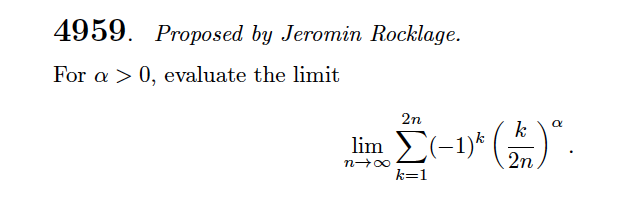}} 
\end{figure}

\centerline{\bf Solution to problem 4959 Crux Math. 50 (6) 2024, 314}\bigskip

 \centerline {Raymond Mortini, Rudolf Rupp} 
 
 \medskip

\centerline{- - - - - - - - - - - - - - - - - - - - - - - - - - - - - - - - - - - - - - - - - - - - - - - - - - - - - -}
  
  \medskip

For $\alpha>0$, let $\dis S_n:=\sum_{k=1}^{2n}(-1)^k\left(\frac{k}{2n}\right)^\alpha$. We show  that 
$\ovalbox{$\lim_{n\to\infty} S_n=\frac{1}{2}$.}$

We will use the higher order Euler-MacLaurin formula applied to the
  function $f(x)=x^{\alpha}$. This formula tells us that for $n\geq m$,
 
\begin{eqnarray*}
\sum_{i=m}^n f(i) -\int_m^n f(x) dx
&=& \frac{f(m)+f(n)}{2}+\frac{1}{6} \frac{f'(n)-f'(m)}{2!}-\int_m^n f''(x) \frac{P_2(x)}{2!}dx
\end{eqnarray*}
 where  $P_2$ is the  periodized Bernoulli function 
 $$P_2(x)=B_2(x-\lfloor x\rfloor),$$
 and 
 $B_2(x)=x^2-x+\frac{1}{6}$,\quad
 the  second  Bernoulli polynomial.\\

Let us first rewrite $S_n$, splitting the sum into odd and even indices:

\begin{eqnarray*}
S_n=\sum_{k=1}^{2n}(-1)^k\left(\frac{k}{2n}\right)^\alpha&=&\sum_{j=1}^n \left(\frac{j}{n}\right)^\alpha-\sum_{j=1}^n \left(\frac{2j-1}{2n}\right)^\alpha\\
&=&\frac{1}{n^\alpha} \sum_{j=1}^n j^\alpha-\frac{1}{(2n)^{\alpha}}\sum_{j=1}^n (2j-1)^\alpha.
\end{eqnarray*}

{\bf Case 1} $\alpha=1$.  Then
\begin{eqnarray*}
S_n&=&\sum_{k=1}^{2n}(-1)^k \frac{k}{2n}=\frac{1}{n}\sum_{j=1} ^n j -\frac{1}{2n}\sum_{j=1}^n(2j-1)\\
&=&\frac{1}{n}\sum_{j=1} ^n j -\frac{1}{n}\sum_{j=1} ^n j + \frac{1}{2n}\cdot n=\frac{1}{2}.
\end{eqnarray*}

{\bf Case 2} $\alpha\not=1$.

By choosing  $m=1$ in the Euler-MacLaurin formula, and estimating
$$\left|\int_1^n f''(x) \frac{P_2(x)}{2!}dx\right|\leq C \int_{1}^n x^{\alpha-2} dx=
\begin{cases}\frac{C}{\alpha-1}\left(n^{\alpha-1}-1\right)&\text{if $\alpha>1$}\\ 
C\log n&\text{if $\alpha=1$}\\
\frac{C}{1-\alpha}\left(1-n^{\alpha-1}\right)&\text{if $0<\alpha<1$},
\end{cases}
$$
we obtain 
$$\sum_{j=1}^n j^\alpha=\left(\frac{n^{\alpha+1}}{\alpha+1}-\frac{1}{\alpha +1}\right)+\frac{n^\alpha+1}{2}+
\frac{1}{12} \left(\alpha n^{\alpha-1}-\alpha\right)+
\co{\begin{cases}\Oh(n^{\alpha-1}) &\text{if $\alpha>1$}\\ \Oh(1)&
\text{if $0<\alpha<1$}.\end{cases}
}
$$
Hence
\begin{eqnarray*}
\frac{1}{n^\alpha} \sum_{j=1}^n j^\alpha&=&\left(\frac{n}{\alpha+1}+\frac{1}{2}+\frac{\alpha}{12} \frac{1}{n} \right)+\Oh(n^{-\alpha})+\co{\begin{cases}\Oh(n^{-1})&\text{if $\alpha>1$}\\ \Oh(n^{-\alpha})&\text{if $0<\alpha<1$}\end{cases}}
\\
&=&\frac{n}{\alpha+1}+\frac{1}{2}+\oh(1).
\end{eqnarray*}

Using, additionally, the formula
\begin{eqnarray*}
\frac{1}{(2n)^\alpha}\sum_{j=1}^n (2j-1)^\alpha&=&\frac{1}{(2n)^\alpha}\left(\sum_{j=1}^{2n} j^\alpha - 
2^\alpha\sum_{j=1}^n j^\alpha\right)\\
&=& \frac{1}{(2n)^\alpha}\sum_{j=1}^{2n} j^\alpha - \frac{1}{n^\alpha} \sum_{j=1}^n j^\alpha
\end{eqnarray*}
we conclude that 

\begin{eqnarray*}
S_n&=&\left(\frac{n}{\alpha+1}+\frac{1}{2} \right) -\left(\frac{2n}{\alpha+1}+\frac{1}{2}-\left(\frac{n}{\alpha+1}+\frac{1}{2}\right)
\right)+\oh(1)\\
&=&	\frac{1}{2}+\oh(1).
\end{eqnarray*}


\newpage

\begin{figure}[h!]
 
  \scalebox{0.5} 
  {\includegraphics{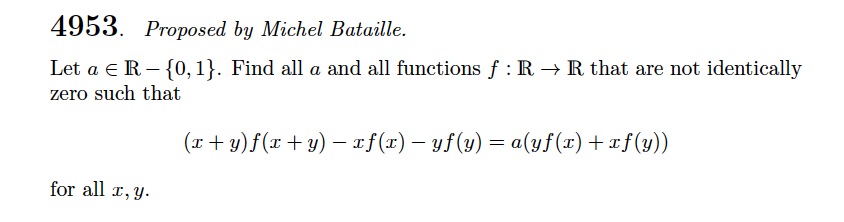}} 
\end{figure}

\centerline{\bf Solution to problem 4953 Crux Math. 50 (6) 2024, 313}\bigskip

 \centerline {Raymond Mortini, Rudolf Rupp}

  \medskip

\centerline{- - - - - - - - - - - - - - - - - - - - - - - - - - - - - - - - - - - - - - - - - - - - - - - - - - - - - -}
  
  \medskip

We show that under the assumption $a\in \R\setminus\{0,1\}$, the number \ovalbox{$a=3$} is the only one for which the functional equation
\begin{equation}\label{eigenartig}
(x+y)f(x+y)-xf(x)-yf(y)=a(yf(x)+xf(y))
\end{equation}
equivalently
\begin{equation}\label{form22}
(x+y)f(x+y)= f(x)(x+ay)+f(y)(y+ax)
\end{equation}
has non-trivial solutions. These are given by \ovalbox{$f(x)= cx^2$} for $c\not=0$.\\

\hrule \bigskip

It is obvious that the set of solutions of (\ref{eigenartig}) is a vector space; may be the trivial one $\{0\}$.  Moreover, if $f$ is a solution then $f(0)=0$: put $y=0$, and  so,  $x f(x)= xf(x) +f(0) a x$. As $a\not=0$, $f(0)=0$.

Note that if $a=3$, then $f(x)=x^2$ is  a solution since
\begin{eqnarray*}
(x+y)^3-x^3-y^3 -3 yx^2-3x y^2=0.
\end{eqnarray*}

Now suppose that $f$ is a solution whenever $a=3$. 
Then  $f$ is even. In fact, putting  $y=-x$,  and using that  $f(0)=0$, yields 
$0=f(x)(-2x)+f(-x)(2x)$ and so $f(-x)=f(x)$.

Next assume that $f(1)=1$ and let us introduce the auxiliary function $h(x):=f(x)-x^2$.  Then $h$ is a solution, too, satisfying $h(0)=h(1)=0$. We show that $h\equiv 0$  (yielding the assertion that $f(x)=x^2$, and so, if $f(1)=c\not=0$, $f(x)=cx^2$, and if $f(1)=0$, $f\equiv 0$).

To see this, take $y=1$ and $x\to-x$. Then
$$(-x+1) h(-x+1)+ xh(-x) = 3 h(-x),$$
hence, by using that $h$ is even,
\begin{equation}\label{equat1}
h(1-x)= \frac{3-x}{1-x}  h(x),\quad x\not=1.
\end{equation}
On the other hand, if $x+y=1$, then by (\ref{form22}),
$$ 0=h(x)(x+ 3(1-x))+h(1-x)(1-x+3x).$$
That is
\begin{equation}
h(1-x)=\frac{2x-3}{2x+1}\; h(x), \quad x\not=-\frac{1}{2}.
\end{equation}
Consequently,
$$ \frac{2x-3}{2x+1} h(x) = \frac{3-x}{1-x}\;  h(x).$$
Since  
$\frac{2x-3}{2x+1} = \frac{3-x}{1-x}\iff -3=3$ is impossible,  $h(x)=0$ for $x\not\in\{1, -1/2\}$. By (\ref{equat1}), $0=h(1/2)=h(-1/2)$.
As $h(1)=0$ we are done.\\

It remains to see that $a$ necessarily is $3$ for the existence of nontrivial solutions.  So let $f$ be  a solution to (\ref{eigenartig}), $f\not\equiv 0$.

$\bullet$  Put $y=x$. Then $2x f(2x)-2xf(x)=2a x f(x)$, implying that 
$$f(2x)=(1+a)f(x),$$
 valid also for $x=0$.

$\bullet$  Put $y=2x$. Then $3x f(3x)-xf(x)-2x f(2x)=a(2x f(x)+x f(2x))$. Hence
$$3 f(3x)=(1+2a) f(x)+ (2+a) f(2x),$$
valid also for $x=0$.

Pulling in $f(2x)$ gives
$$
3f(3x)= (2a+1) f(x)+ (2+a) (1+a) f(x)= (3+5a+a^2)f(x).
$$

$\bullet$ Put $y=3x$.  Then $4x f(4x)-xf(x)-3xf(3x)=a(3x f(x)+xf(3x))$. Hence
$$4f(4x)=(1+3a)f(x)+(3+a)f(3x),$$
valid also for $x=0$.
Since $f(4x)=f(2 (2x))=(1+a)f(2x)=(1+a)^2 f(x)$, we obtain with the formula for $f(3x)$ that
$$
4(1+a)^2f(x)=(1+3a)f(x)+\frac{1}{3}\;(3+a)(3+5a+a^2)f(x).
$$
Consequently, as $f\not\equiv 0$, 
$$
3(3+5a +4a^2) -(3+a)(3+5a+a^2)=0,
$$
or equivalently
$$a(a-1)(a-3)=0.$$
Thus $a=3$.
\\

{\bf Remark} What happens for $a=0$ or $a=1$? If $a=0$, then  the solutions $f$ are those functions for which $xf$ is additive.  If $a=1$ then the solutions are exactly the additive functions: in fact, by (\ref{form22}), $(x+y)f(x+y)= f(x)(x+y)+f(y)(y+x)$. Hence, if $x+y\not=0$,
$f(x+y)=f(x)+f(y)$. Now let $x+y=0$. Since $f(2x)=f(x+x)=f(x)+f(x)=2f(x)$ we conclude that for $x\not=0$,
\begin{eqnarray*}
f(x)+f(-x)&=&f((x-1)+1)+ f((-1-x)+1)=f(x-1) +f(1) +f(-1-x) +f(1)\\
& =&f((x-1)+(-1-x)) +2f(1)= f(-2) + 2f(1)\\
&=&2 (f(-1)+f(1)).
\end{eqnarray*}
In particular, if $x=1$, we deduce that  $f(-1)+f(1)=0$, and so
$$f(x)+f(-x)=0=f(0)=f(x+(-x)).$$ 

{\bf Remark} It is quite astonishing for us that in the these two exceptional cases there are so many solutions, as 
the set of additive functions is in a one to one correspondance with the $\Q$-linear functions on the $\Q$-vector
 space $V$ determined  by the real numbers.


\newpage

\begin{figure}[h!]
 
  \scalebox{0.5} 
  {\includegraphics{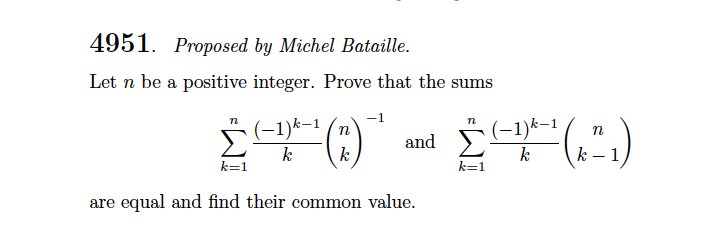}} 
\end{figure}

\centerline{\bf Solution to problem 4951 Crux Math. 50 (6) 2024, 313}\bigskip

 \centerline {Raymond Mortini, Rudolf Rupp}
 
  \medskip

\centerline{- - - - - - - - - - - - - - - - - - - - - - - - - - - - - - - - - - - - - - - - - - - - - - - - - - - - - -}
  
  \medskip

We show that the common value is \quad \ovalbox{$\dis \frac{1+(-1)^{n+1}}{n+1}$.}

Let $$P_1(x):=\sum_{k=1}^n \frac{(-1)^{k-1}}{k}\binom{n}{k-1}x^{k-1}.$$
To be calculated is $P_1(1)$.  To this end, note that
\begin{eqnarray*}
(xP_1(x))'&=&\sum_{k=1}^n(-1)^{k-1}\binom{n}{k-1}x^{k-1}=\sum_{j=0}^{n-1}(-1)^j\binom{n}{j}x^j\\
&=&-(-x)^n+\sum_{j=0}^n \binom{n}{j} (-x)^j=(1-x)^n-(-x)^n.
\end{eqnarray*}
Hence
$$xP_1(x)=\int_0^x \left((1-t)^n-(-t)^n \right) dt,$$
from which we conclude that 
$$P_1(x)= \frac{1}{x(n+1)} \left(-(1-x)^{n+1}+(-1)^{n+1}x^{n+1}+1\right).$$
Consequently $\dis P_1(1)= \frac{1+(-1)^{n+1}}{n+1}$.\\

To calculate $\dis P_2:=\sum_{k=1}^n \frac{(-1)^{k-1}}{k}{\binom{n}{k}}^{-1}$, we use that (see below)
$$\frac{1}{k}{\binom{n}{k}}^{-1}=\beta(k,n-k+1)=\int_0^1 (1-x)^{k-1}x^{n-k}dx=:I.$$
Hence
\begin{eqnarray*}
P_2&=& \int_0^1 \sum_{k=1}^n (x-1)^{k-1}x^{n-k}dx=\int_0^1 \sum_{j=0}^{n-1}(x-1)^j x^{n-j-1}dx\\
&=& -\int_0^1\big((x-1)^n-x^n\big)dx=\frac{1+(-1)^{n+1}}{n+1}.
\end{eqnarray*}

{\bf Addendum} A classical formula tells us that for $\mu,\nu\in \N$ one has
$$\frac{\Gamma(\mu+1)\;\Gamma(\nu+1)}{\Gamma\big((\mu+1)+(\nu+1)\big)}=\beta(\mu+1, \nu+1):=\dis \int_0^1 (1-x)^\mu x^\nu dx= \frac{\mu! \;\nu!}{(\mu+\nu+1)!}.$$
Hence, for $\mu=k-1$ and  $\nu=n-k$,
\begin{eqnarray*}
I=\beta(k, n-k+1)&=& \frac{(k-1)!\; (n-k)!}{n!}=\frac{1}{k}\frac{k!\;(n-k)!}{n!}=\frac{1}{k}{\binom{n}{k}}^{-1}.
\end{eqnarray*}

\newpage

\nopagecolor
\begin{figure}[h!]
 
  \scalebox{0.5} 
  {\includegraphics{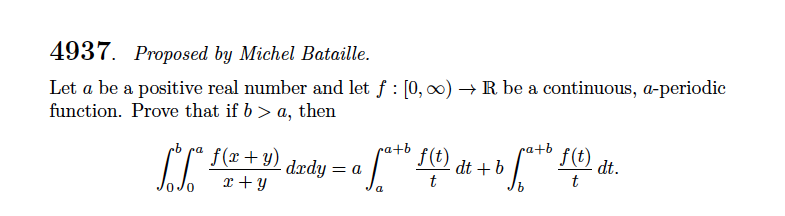}} 
\end{figure}

\centerline{\bf Solution to problem 4937 Crux Math. 50 (4) 2024, 200}\bigskip

 \centerline {Raymond Mortini, Rudolf Rupp}
 
  \medskip

\centerline{- - - - - - - - - - - - - - - - - - - - - - - - - - - - - - - - - - - - - - - - - - - - - - - - - - - - - -}
  
  \medskip


Let $\dis F(y)=\int_a^y \frac{f(t)}{t}dt$ be  a primitive of $f(t)/t$ on $]0,\infty[$. Note that $F(y)$ may be unbounded near $0$
 (for $f\equiv 1$ for instance). But one can control its behaviour near $0$: 
$$R(y):= y( F(a+y)-F(y))\to 0\; \text{as $y\to 0$}.$$

In fact,  as $f$ is continuous, $|f|\leq M$ on $[0,a]$. Hence, for $0<y\leq a$, 
$$y|F(y)|\leq My\int_y^a \frac{1}{t}dt= My\log a-My\log y\to 0,
 $$
 and so $R(y)\to 0$ as $F$ is continuous at $a$.
 
Using the transformation $x+y=t$ with $y\leq t\leq a+y$, we obtain for $\delta>0$ close to $0$,
\begin{eqnarray*}
\int_\delta^b \left(\int_0^a \frac{f(x+y)}{x+y}dx\right)dy&=&\int_\delta^b\left(\int_y^{a+y} \frac{f(t)}{t}dt\right)dy
=\int_\delta^b \left(F(a+y)-F(y)\right) dy\\
&=&\Big[y(F(a+y)-F(y))\Big]^b_\delta- \int_\delta^b y( F'(y+a)-F'(y))dy\\
&=&  b(F(a+b)-F(b)) -\underbrace{\delta( F(a+\delta)-F(\delta))}_{=R(\delta)}-
\int_\delta^b y\left( \frac{f(a+y)}{a+y}-\frac{f(y)}{y}\right)dy\\
&\buildrel\longrightarrow_{\delta\to 0}^{}& b(F(a+b)-F(b)) -\int_0^b \frac{\co{a}+y\co{-a}}{a+y}f(a+y)dy + \int_0^b f(y) dy\\
&=& b(F(a+b)-F(b)) -\int_0^b f(a+y)dy + a\int_0^b\frac{f(a+y)}{a+y}dy+ \int_0^b f(y) dy\\
&\buildrel=_{s=a+y}^{f\; \text{$a$-periodic}}& b(F(a+b)-F(b))+a \int_a^{a+b} \frac{f(s)}{s}ds\\
&=& b\left(\int_a^{a+b} \frac{f(t)}{t} dt - \int_a^b  \frac{f(t)}{t} dt\right)+a \int_a^{a+b} \frac{f(s)}{s}ds\\
&=&b\int_b^{a+b} \frac{f(t)}{t} dt+a \int_a^{a+b} \frac{f(s)}{s}ds.
\end{eqnarray*}


\newpage

\begin{figure}[h!]
 
  \scalebox{0.55} 
  {\includegraphics{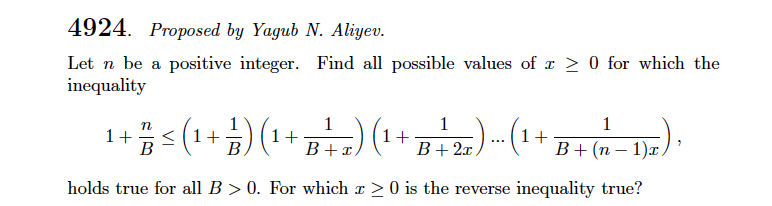}} 
\end{figure}

\centerline{\bf Solution to problem 4924 Crux Math. 50 (3) 2024, 148}\bigskip

 \centerline {Raymond Mortini, Rudolf Rupp}
  \medskip

\centerline{- - - - - - - - - - - - - - - - - - - - - - - - - - - - - - - - - - - - - - - - - - - - - - - - - - - - - -}
  
  \medskip


We claim that for all $B>0$ and $n\geq 2$
$$\ovalbox{$\dis 1+\frac{n}{B}\leq \prod_{j=0}^{n-1}\left(1+ \frac{1}{B+jx}\right)\;\text{if \;$0\leq x\leq 1$}
$}$$
and
$$\ovalbox{$\dis 1+\frac{n}{B}\geq \prod_{j=0}^{n-1}\left(1+ \frac{1}{B+jx}\right)\;\text{if \;$x\geq 1$}.
$}$$
In particular, we have equality for $x=1$. Moreover, equality holds for all $B>0$ and all $x\geq 0$ if $n=1$.\bigskip

Let $L_n$ be the left hand side and $R_n$ the right hand side.

$\bullet$ $n=1$.
 Then
$$L_1-R_1=1+\frac{1}{B}- (1+\frac{1}{B})=0.$$

$\bullet$~ $n\geq 2$. We show the assertion above  via induction on $n$. So let $n=2$.

\begin{eqnarray*}
L_2-R_2&=&  1+\frac{2}{B}-\left(1+\frac{1}{B}\right)\left(1+\frac{1}{B+x}\right)\\
&=&\frac{x-1}{B(B+x)}.
\end{eqnarray*}
Hence the assertion is true for $n=2$.

$\bullet $ $n\to n+1$.  Assume that the assertion is correct for some $n\in \N$. Then, for $x\geq 1$, $L_n\geq R_n$, and so
\begin{eqnarray*}
L_{n+1}-R_{n+1}&=&  1+\frac{n+1}{B}-\prod_{j=0}^{n}\left(1+ \frac{1}{B+jx}\right)\\
&\geq& 1+\frac{n+1}{B}- \left(1+\frac{n}{B}\right) \left(1+\frac{1}{B+nx}\right)\\
&=& \frac{1}{B}\left( 1- \frac{B+n}{B+nx}\right)=\frac{1}{B}\;\frac{n(x-1)}{B+nx}\geq 0.
\end{eqnarray*}
The same estimates replacing $\geq $ by $\leq$ show that  we also get the assertion for $0\leq x\leq 1$.

\newpage

\begin{figure}[h!]
 
  \scalebox{0.5} 
  {\includegraphics{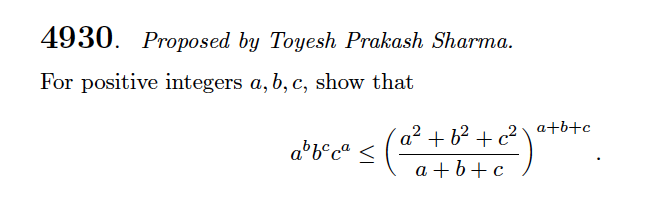}} 
\end{figure}

\centerline{\bf Solution to problem 4930 Crux Math. 50 (3) 2024, 149}\bigskip

 \centerline {Raymond Mortini, Rudolf Rupp}
 
  \medskip

\centerline{- - - - - - - - - - - - - - - - - - - - - - - - - - - - - - - - - - - - - - - - - - - - - - - - - - - - - -}
  
  \medskip


First we note that for $x,y\geq 0$,  $xy\leq \frac{1}{2}(x^2+y^2)$. Hence
$$ab+bc+ca\leq \frac{1}{2}(a^2+b^2)+\frac{1}{2}(b^2+c^2)+\frac{1}{2}(c^2+a^2)=a^2+b^2+c^2.$$

As $\log x$ is concave, we know that $\log (\sum_{j=1}^n \e_j x_j)\geq \sum_{j=1}^n \e_j \log x_j$ whenever
$\sum_{j=1}^n\e_j=1$, $\e_j\geq 0$.
Hence
\begin{eqnarray*}
\frac{b}{a+b+c}\log a+\frac{c}{a+b+c}\log b+\frac{a}{a+b+c}\log c
&\leq & \log\left( \frac{b}{a+b+c} a + \frac{c}{a+b+c}b +\frac{a}{a+b+c}c\right)\\
&\leq& \log\left(\frac{a^2+b^2+c^2}{a+b+c}\right).
\end{eqnarray*}

Hence
$$a^bb^cc^a\leq  \left(\frac{a^2+b^2+c^2}{a+b+c}\right)^{a+b+c}.$$

\newpage
\begin{figure}[h!]
 
  \scalebox{0.55} 
  {\includegraphics{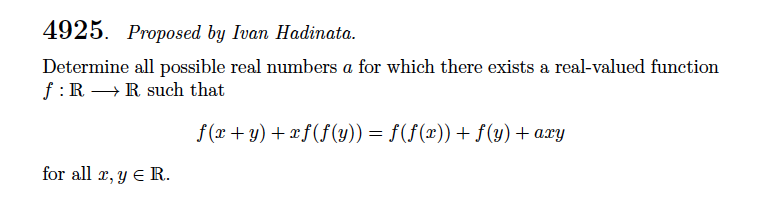}} 
\end{figure}

\centerline{\bf Solution to problem 4925 Crux Math. 50 (3) 2024, 148}\bigskip

 \centerline {Raymond Mortini, Rudolf Rupp}
 
  \medskip

\centerline{- - - - - - - - - - - - - - - - - - - - - - - - - - - - - - - - - - - - - - - - - - - - - - - - - - - - - -}
  
  \medskip
 

{\it We show that the functional equation
\begin{equation}\label{feqa}
f(x+y)+xf(f(y))=f(f(x))+f(y)+axy
\end{equation} 
admits a solution if and only if $a\in \{0,1\}$.
If $a=1$, then  the identity is the only  solution and if $a=0$, then the zero function is the only solution.}\\

Put $x=y=0$. Then $f(0)+0=f(f(0))+f(0)$ implies that $f(f(0))=0$. Now put $y=0$ in (\ref{feqa}). Then
\begin{equation}\label{feqa1}
f(x)=f(f(x))+f(0).
\end{equation}
This yields the new equation
\begin{equation}\label{feqa2}
f(x+y)+x\big(f(y)-f(0)\big)=f(x)-f(0)+f(y)+axy.
\end{equation}
Now put $x=1$. Then
\begin{equation}\label{feqa3}
f(1+y)+f(y)-f(0)=f(1)-f(0)+f(y)+ay,
\end{equation}
from which we conclude that
$f(y+1)=f(1)+ay$, or in other words,
\begin{equation}\label{feqa4}
f(u)=f(1)+a(u-1)=:au+b,
\end{equation}
that is, $f$ is linear-affine.

Since $f(f(0))=0$, we have $a(a\cdot 0+b)+b=0$, and so $ab+b=0$. Thus $b=0$ or $a=-1$.
Due to (\ref{feqa1}), 
$$ax+b=a(ax+b)+b+b.$$
That is
$$ax=a^2x +(ab+b)=a^2x,$$
from which we deduce $a\in \{0,1\}$.  Hence, as $a\not=-1$, $b=0$. Consequently, $f(x)=x$ (if $a=1$) or $f\equiv 0$ if $a=0$.
It is straightforward to check that these are actually solutions.

\newpage
\begin{figure}[h!]
 
  \scalebox{0.5} 
  {\includegraphics{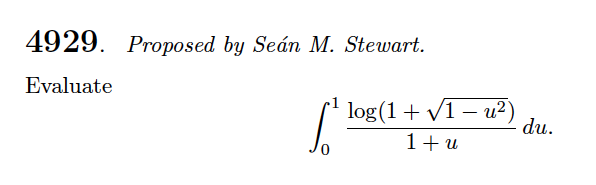}} 
\end{figure}

\centerline{\bf Solution to problem 4929 Crux Math. 50 (3) 2024, 149}\bigskip

 \centerline {Raymond Mortini, Rudolf Rupp}
 
  \medskip

\centerline{- - - - - - - - - - - - - - - - - - - - - - - - - - - - - - - - - - - - - - - - - - - - - - - - - - - - - -}
  
  \medskip


We show that  
$$\ovalbox{$\dis I:= \int_0^1 \frac{\log(1+\sqrt{1-u^2})}{1+u}\; du=\frac{\pi^2}{24}\sim 0.4112335167\cdots$}.$$
In fact, 
\begin{eqnarray*}
I&=& \int_0^1 \log\left(1+\sqrt{1-u^2}\right)\sum_{n=0}^\infty (-1)^n u^n\; du=\sum_{n=0}^\infty (-1)^n\int_0^1 \log\left(1+\sqrt{1-u^2}\right) u^n\; du,
\end{eqnarray*}
where the interchanging $\lim_n \int S_n=\int\lim S_n$ is allowed as the partial sums
$$S_n= \sum_{j=0}^n\log\left(1+\sqrt{1-u^2}\right)(-1)^j u^j$$ 
are bounded by the $L^1[0,1]$-function $\frac{\log\left(1+\sqrt{1-u^2}\right)}{1+u}$.
Now put 
$$I_n:=\int_0^1 \log\left(1+\sqrt{1-u^2}\right) u^n\; du,$$
and use 
partial integration with $f(u)= \log\left(1+\sqrt{1-u^2}\right)$ and $g'(u)=u^n$. Note that $g(u)=\frac{u^{n+1}}{n+1}$ and 
$$f'(u)= \frac{\frac{-u }{\sqrt{1-u^2}}}{1+\sqrt{1-u^2}}=-\frac{u (1-\sqrt{1-u^2}) }{\sqrt{1-u^2}\;u^2}=
-\frac{1}{u\,\sqrt{1-u^2}}+\frac{1}{u}.
$$
Hence
\begin{eqnarray*}
I_n&=& -\frac{1}{n+1}\int_0^1  \left(u^n-\frac{u^n}{\sqrt{1-u^2}}\right)\;du\\
&\buildrel=_{u=\sin t}^{}& -\frac{1}{(n+1)^2}+\frac{1}{n+1}\int_0^{\pi/2} \sin^n t \,dt:=  -\frac{1}{(n+1)^2}+J_n.
\end{eqnarray*}

{\bf Method 1}.

By Lemma \ref{intsinu} below,

\begin{eqnarray*}
I=-\sum_{n=0}^\infty (-1)^n \frac{1}{(n+1)^2}+\sum_{n=0}^\infty (-1)^n J_n&=&-\frac{\pi^2}{12}+\sum_{m=0}^\infty J_{2m} -\sum_{m=0}^\infty J_{2m+1}\\
&=&-\frac{\pi^2}{12} +\sum_{m=0}^\infty \frac{{2m\choose m} }{(2m+1)4^m}\cdot \frac{\pi}{2}-
 \sum_{m=0}^\infty\frac{4^m}{(2m+1)(2m+2){2m \choose m}}\\
 &=& -\frac{\pi^2}{12}+ \frac{\pi^2}{4} - \frac{\pi^2}{8}=\frac{\pi^2}{24}.
\end{eqnarray*}

Here we have used the Taylor series (see e.g. \cite{549028} and \cite{1448822}):

$$\arcsin x=\sum_{n=0}^\infty   \frac{{2n\choose n}}{4^n}\frac{x^{2n+1}}{2n+1}$$
\begin{eqnarray*}
\arcsin^2 x&=&2 \sum_{n=0}^\infty \frac{4^n x^{2n+2}}{(2n+1)(2n+2){2n\choose n}}\\
&=& \sum_{n=0}^\infty \frac{2^{2n+1}(n!)^2}{(2n+2)!} x^{2n+2}= \sum_{n=1}^\infty \frac{2^{2n-1}}{n^2 {2n\choose n}}x^{2n}.
\end{eqnarray*}
evaluated at $x=1$.

\begin{lemma}\label{intsinu}
Let $I_n:=\int_0^{\pi/2} (\sin x)^{n} dx$. Then  $I_0=\pi/2$, $I_1=1$ and for $n\in \N^*$, 
\begin{enumerate}
\item[(1)] $ I_{2n}=\frac{1}{2}\cdot\frac{3}{4}\cdot\frac{5}{6}\cdots\frac{2n-1}{2n}\frac{\pi}{2}=\frac{(2n)!}{4^n (n!)^2}\cdot\frac{\pi}{2}=
\frac{{2n\choose n} }{4^n}\cdot \frac{\pi}{2}$.
\item[(2)] $ I_{2n+1}=\frac{2}{3}\cdot\frac{4}{5}\cdot\frac{6}{7}\cdots\frac{2n}{2n+1}= \frac{4^n (n!)^2}{(2n+1)!}= \frac{4^n}{(2n+1){2n \choose n}}$.
\end{enumerate}
\end{lemma}
\begin{proof}
(1) $I_{2n}= \frac{2n-1}{2n} I_{2n-2}$ for $n\in \N^*$ and $I_0=\frac{\pi}{2}$, because

\begin{eqnarray*}
2n I_{2n}- (2n-1)I_{2n-2}&=&\int_0^{\pi/2}(\sin x)^{2n-2}\left(2n \sin^2 x-(2n-1)\right)dx\\
&=&-\int_0^{\pi/2}(\sin x)^{2n-2}\left( (2n-1) \cos ^2x-\sin^2 x\right)dx\\
&=&-\left[ (\sin x)^{2n-1}\cos x   \right]^{\pi/2}_0=0.
\end{eqnarray*}

(2) $I_{2n+1}= \frac{2n}{2n+1} I_{2n-1}$ for $n\in \N^*$ and $I_1=1$, because

\begin{eqnarray*}
(2n+1)I_{2n+1}-2nI_{2n-1}&=& \int_0^{\pi/2} (\sin x)^{2n-1}\left((2n+1)\sin ^2 x -2n\right)dx\\
&=&-\int_0^{\pi/2}(\sin x)^{2n-1}\left(2n \cos^2 x-\sin^2x\right)dx\\
&=&-\left[(\sin x)^{2n} \cos x\right]^{\pi/2}_0=0.
\end{eqnarray*}
\end{proof}

Next we note that for $0\leq x<1$, $\arcsin x=\arctan\left(\frac{x}{\sqrt{1-x^2}}\right)$. In fact

\begin{lemma}
For $0\leq x<1$ we have
\begin{enumerate}
\item[(1)] $\dis \arcsin x=\arctan\left(\frac{x}{\sqrt{1-x^2}}\right)$.
\item[(2)] $\dis \frac{\arcsin x}{\sqrt{1-x^2}}=\sum_{n=0}^\infty x^{2n+1}\int_0^{\pi/2} (\sin t)^{2n+1}\,dt$.
\item[(3)] $\dis(\arcsin x)^2= \sum_{n=0}^\infty \left(\frac{\int_0^{\pi/2} (\sin t)^{2n+1}dt}{n+1} \right)x^{2n+2}$.
\end{enumerate}
\end{lemma}

\begin{proof}
(1) Obvious, as $\tan(\arcsin x)=\sin(\arcsin x)/(\cos\arcsin x)$.

(2) As in \cite{549028}): we shall use that for $0\leq y<1$
$$\frac{y}{1-y^2}= \frac{1}{2}\left(\frac{1}{1-y}-\frac{1}{1+y}\right)=\frac{1}{2}\sum_{n=0}^\infty  (y^n-(-1)^n y^n)
=\sum_{n=0}^\infty y^{2n+1}.$$
Now let 
$$I(t):=\frac{1}{\sqrt{1-x^2}}\;\arctan \frac{x\sin t}{\sqrt{1-x^2}}. 
$$
Then, with $0<x<1$,
\begin{eqnarray*}
 \frac{\arcsin x}{\sqrt{1-x^2}}&=& I\left(\frac{\pi}{2}\right)-I(0)=\int_0^{\frac{\pi}{2}} \frac{\partial I}{\partial t} dt=
 \int_0^{\frac{\pi}{2}} \frac{x\cos t}{1-x^2\cos^2 t}\;dt\\
 &=&  \int_0^{\frac{\pi}{2}}\sum_{n=0}^\infty (x\cos t)^{2n+1}\; dt=\sum_{n=0}^\infty x^{2n+1} \int_0^{\frac{\pi}{2}} (\cos t)^{2n+1}\;dt.
\end{eqnarray*}
Here $\sum\int=\int\sum$, as all the terms are positive.

(3) Just use that  $\dis \frac{d}{dx}(\arcsin x)^2= 2  \frac{\arcsin x}{\sqrt{1-x^2}}$.
\end{proof}

{\bf Method 2} 

Note that

\begin{eqnarray*}
I=-\sum_{n=0}^\infty (-1)^n \frac{1}{(n+1)^2}+\sum_{n=0}^\infty (-1)^n J_n.
\end{eqnarray*}
Consider the auxiliary function
$$h(x)=\sum_{n=0}^\infty \left( \int_0^{\frac{\pi}{2}} \frac{(\sin t)^n}{n+1}\,dt \right)(-1)^nx^{n+1}.$$
Then 
\begin{eqnarray*}
h'(x)&=&\sum_{n=0}^\infty \left( \int_0^{\frac{\pi}{2}} (\sin t)^n\,dt \right) (-1)^n x^n\\
&=&\int_0^{\frac{\pi}{2}} \sum_{n=0}^\infty (- x\sin t )^n\; dt =
\int_0^{\frac{\pi}{2}} \frac{1}{1+x\sin t}\;dt
\buildrel=_{}^{(*)}  \frac{\arccos x}{\sqrt{1-x^2}}
\end{eqnarray*}

Hence, 
$$h(1)=h(1)-h(0)= \int_0^1  \frac{\arccos x}{\sqrt{1-x^2}} dx= \left[ -\frac{1}{2}\arccos^2 x\right]^1_0=\frac{\pi^2}{8}.$$

We conclude that
 $\dis I=-\frac{\pi^2}{12} +\frac{\pi^2}{8}= \frac{\pi^2}{24}$.
 \\
 
 To prove (*), just use the transformation $\tan(t/2)=y$ to get
 
 \begin{eqnarray*}
 \int_0^{\frac{\pi}{2}} \frac{1}{1+x\sin t}\;dt&=&2 \int_0^1 \frac{1}{1+y^2+2xy}dy=
 \frac{2}{1-x^2}\int_0^1 \frac{1}{1+\frac{(y+x)^2}{1-x^2}}dy\\
 &=& 2\left[\frac{\arctan\frac{x+y}{\sqrt{1-x^2}}}{\sqrt{1-x^2}}\right]^1_0=
 2\frac{\arctan\frac{x+1}{\sqrt{1-x^2}}-\arctan\frac{x}{\sqrt{1-x^2}} }{\sqrt{1-x^2}}\\
 &=& \frac{\arccos x}{\sqrt{1-x^2}}.
\end{eqnarray*}

The latter is verified by calculating the derivatives of the numerators and by using  that for $x=0$, $2\arctan 1=\pi/2=\arccos 0$.\\

{\bf Method 3}

\begin{eqnarray*}
 f(1)&=&\int_0^{1}\int_0^{\frac{\pi}{2}} \frac{1}{1+\xi\sin t}\;dt d\xi= \int_0^{\frac{\pi}{2}}\int_0^1  \frac{1}{1+\xi\sin t} d\xi dt\\
 &=&\int_0^{\frac{\pi}{2}} \frac{\log (1+\sin t)}{\sin t}dt\\
 &=&\int_0^1 \frac{\log(1+y)}{y\sqrt{1-y^2}}dy\\
 &=& \frac{\pi^2}{8}.
\end{eqnarray*}

From \cite{549028}

  \begin{figure}[h!]
 
  \scalebox{0.45} 
  {\includegraphics{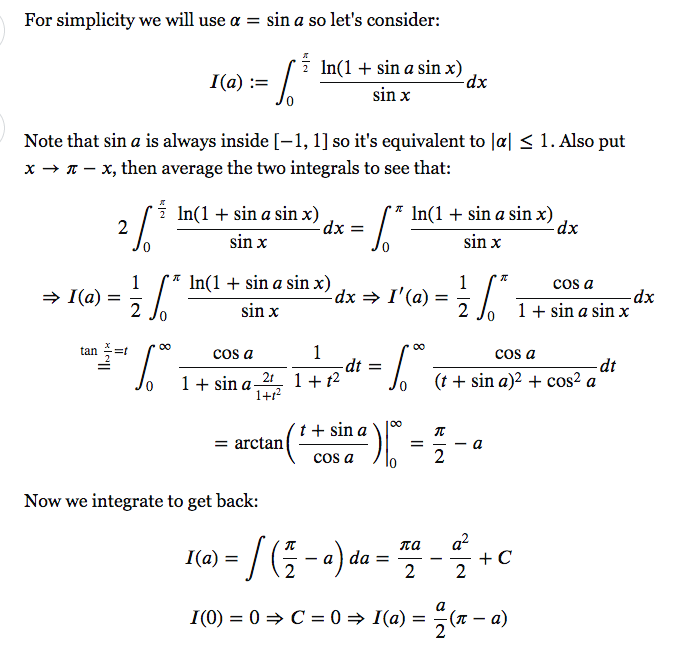}} 
\end{figure}


\newpage
\begin{figure}[h!]
 
  \scalebox{0.55} 
  {\includegraphics{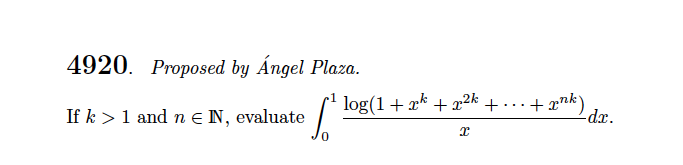}} 
\end{figure}

\centerline{\bf Solution to problem 4920 Crux Math. 50 (2) 2024, 84}\bigskip

 \centerline {Raymond Mortini, Rudolf Rupp}
 
  \medskip

\centerline{- - - - - - - - - - - - - - - - - - - - - - - - - - - - - - - - - - - - - - - - - - - - - - - - - - - - - -}
  
  \medskip


Let 
$$I:= \int_0^1 \frac{\log(1+x^k+x^{2k}+\cdots+x^{nk})}{x}\;dx.
$$

We show that

$$\ovalbox{$\dis I= \frac{\pi^2}{6} \frac{n}{k(n+1)}$}.$$\\

For the proof, we use the power series representation
$$\log(1-x)=-\sum_{j=1}^\infty \frac{x^j}{j},\quad |x|<1.$$
So, if $0<x<1$ we have
\begin{eqnarray*}
f(x):=\frac{\log(1+x^k+x^{2k}+\cdots+x^{nk})}{x}&=&\frac{1}{x}\;\log\left(  \frac{1-x^{k(n+1)}}{1-x^k}        \right)\\
&=&-\sum_{j=1}^\infty \frac{x^{k(n+1)j-1}}{j} + \sum_{j=1}^\infty \frac{x^{kj-1}}{j}.
\end{eqnarray*}
Hence, a primitive is given by
$$
-\sum_{j=1}^\infty \frac{x^{k(n+1)j}}{j^2k(n+1)}+\sum_{j=1}^\infty \frac{x^{kj}}{j^2k}.
$$
We conclude that
\begin{eqnarray*}
I=\int_0^1 f(x) dx&=&\sum_{j=1}^\infty \left(-\frac{1}{k(n+1)} \frac{1}{j^2}+ \frac{1}{k} \frac{1}{j^2}\right)\\
&=& \frac{\pi^2}{6} \left(\frac{1}{k}-\frac{1}{k(n+1)} \right)\\
&=& \frac{\pi^2}{6}  \frac{n}{k(n+1)}.
\end{eqnarray*}

\newpage

\begin{figure}[h!]
 
  \scalebox{0.55} 
  {\includegraphics{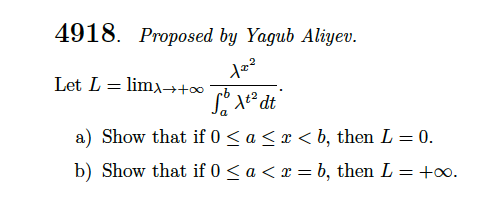}} 
\end{figure}

\centerline{\bf Solution to problem 4918 Crux Math. 50 (2) 2024, 83}\bigskip

 \centerline {Raymond Mortini, Rudolf Rupp}
 
  \medskip

\centerline{- - - - - - - - - - - - - - - - - - - - - - - - - - - - - - - - - - - - - - - - - - - - - - - - - - - - - -}
  
  \medskip

Let $0\leq a <b$. 

a) Suppose that $a\leq x<b$. Then for $a\leq t\leq b$ we have  $1\geq t/b$, and so, for $\lambda>1$, 
\begin{eqnarray*}
0\leq I(\lambda):=\frac{\lambda^{x^2}}{\int_a^b \lambda^{t^2}dt}&\leq & \frac{b\lambda^{x^2}}{\int_a^b t\lambda^{t^2}dt}
=\frac{2b(\log\lambda)\, \lambda^{x^2}}{\lambda^{b^2}-\lambda^{a^2} }=
\frac{2b(\log\lambda)\, \lambda^{x^2-b^2}}{1-\lambda^{a^2-b^2} }.
\end{eqnarray*}
Since 
$$(\log\lambda) e^{(\log\lambda) (x^2-b^2)}= \frac{\log\lambda}{e^{(\log \lambda) (b^2-x^2)}}\to 0 \text{\;as $\lambda\to\infty$},$$
we have that $\lim_{\lambda\to\infty} I(\lambda)=0$.\\

b) Suppose that $x=b$. Then for $a\leq t\leq b$ we have  $1\leq t/a$ and so 
\begin{eqnarray*}I(\lambda)=\frac{\lambda^{b^2}}{\int_a^b \lambda^{t^2}dt}&\geq &
 \frac{a\lambda^{b^2}}{\int_a^b t\lambda^{t^2}dt}
=\frac{2a(\log\lambda)\, \lambda^{b^2}}{\lambda^{b^2}-\lambda^{a^2} }\\
&\geq&\frac{2a(\log\lambda)\, \lambda^{b^2}}{\lambda^{b^2} }=2 a\log\lambda \\
&\to& \infty \text{\;as $\lambda\to\infty$}.
\end{eqnarray*}

\newpage
\begin{figure}[h!]
 
  \scalebox{0.55} 
  {\includegraphics{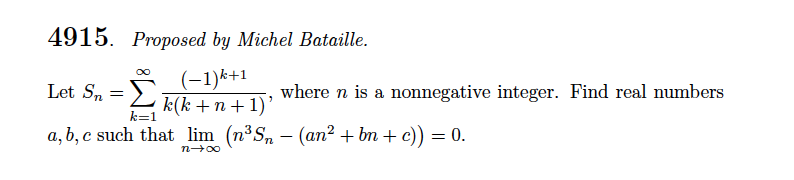}} 
\end{figure}

\centerline{\bf Solution to problem 4915 Crux Math. 50 (2) 2024, 82}\medskip

 \centerline {Raymond Mortini, Rudolf Rupp}
 
  \medskip

\centerline{- - - - - - - - - - - - - - - - - - - - - - - - - - - - - - - - - - - - - - - - - - - - - - - - - - - - - -}
  
  \medskip

We claim that $a=\log 2, \; b=-\frac{1}{2}-\log 2$ and $c=\log 2+\frac{5}{4}$; that is

$$\ovalbox{$\dis\lim_{n\to\infty}\left(n^3S_n-\Big( (\log 2) {{n}^{2}}+\Big( -\frac{1}{2}-\log 2 \Big)n +\log 2+\frac{5}{4}\Big)\right)=0$.}
$$

  \begin{eqnarray*}
 {{S}_{n}}:=\sum\limits_{k=1}^{\infty }{\frac{{{(-1)}^{k+1}}}{k(k+n+1)}}&=&\frac{1}{n+1}\sum\limits_{k=1}^{\infty }{{{(-1)}^{k+1}}\frac{n+1+k-k}{k(k+n+1)}}\\
 &=&\frac{1}{n+1}\sum\limits_{k=1}^{\infty }{{{(-1)}^{k+1}}\left( \frac{1}{k}-\frac{1}{n+k+1} \right)} \\
& =&\frac{1}{n+1}\sum\limits_{k=1}^{\infty }{\left( {{(-1)}^{k+1}}\int\limits_{0}^{1}{{{x}^{k-1}}}dx-{{(-1)}^{k+1}}\int\limits_{0}^{1}{{{x}^{n+k}}}dx \right)}\\
&\buildrel=_{}^{(*)}&\frac{1}{n+1}\int\limits_{0}^{1}\Big({\sum\limits_{k=1}^{\infty }{{{(-1)}^{k+1}}{{x}^{k-1}}-}}\sum\limits_{k=1}^{\infty }{{{(-1)}^{k+1}}{{x}^{n+k}}}\Big)dx \\
& =&\frac{1}{n+1}\int\limits_{0}^{1}{\left( \frac{1}{1+x}-{{x}^{n+1}}\sum\limits_{k=1}^{\infty }{{{(-1)}^{k+1}}{{x}^{k-1}}} \right)}dx=\frac{1}{n+1}\int\limits_{0}^{1}{\frac{1-{{x}^{n+1}}}{1+x}}dx\\
&=&\frac{\log (2)}{n+1}-\frac{1}{n+1}\int\limits_{0}^{1}{\frac{{{x}^{n+1}}}{1+x}}dx \\
& =&\frac{\log (2)}{n+1}-\frac{1}{n+1}\left[ \left. \frac{{{x}^{n+2}}}{n+2}\frac{1}{1+x} \right|_{0}^{1}+\frac{1}{n+2}\int\limits_{0}^{1}{\frac{{{x}^{n+2}}}{{{(1+x)}^{2}}}}dx \right] \\
& =&\frac{\log (2)}{n+1}-\frac{1}{2(n+1)(n+2)}-\frac{1}{(n+2)(n+1)}\int\limits_{0}^{1}{\frac{{{x}^{n+2}}}{{{(1+x)}^{2}}}}dx \\
& =&\frac{\log (2)}{n+1}-\frac{1}{2(n+1)(n+2)}-\frac{1}{(n+2)(n+1)}\left[ \left. \frac{{{x}^{n+3}}}{n+3}\frac{1}{{{(1+x)}^{2}}} \right|_{0}^{1}+\frac{2}{n+3}\int\limits_{0}^{1}{\frac{{{x}^{n+2}}}{{{(1+x)}^{3}}}}dx \right] \\ 
& =&\frac{\log (2)}{n+1}-\frac{2n+7}{4(n+1)(n+2)(n+3)}-\frac{2}{(n+1)(n+2)(n+3)}\int\limits_{0}^{1}{\frac{{{x}^{n+2}}}{{{(1+x)}^{3}}}}dx.
 \end{eqnarray*}\\
 
 Here the interchanging of $\int\sum=\sum\int$ in $(*)$ is possible since the partial sums
 $$\sum_{k=1}^N (-1)^{k+1}( x^{k-1}-x^{n+k})$$
   are bounded.
Hence \footnote{ Here $\Oh$ and $\oh$ denote the Landau symbols: $\Oh (1/n)$ is a function 
$n\mapsto h(n)$ satisfying $\frac{|h(n)|}{ 1/n}\leq C$  and $\oh(1)$ is a function $n\mapsto g(n)$ with $\lim_{n\to\infty} g(n)=0$. In particular, $\Oh(1/n)$ implies $\oh(1)$.}
\begin{eqnarray*}
 {{n}^{3}}{{S}_{n}}&=&\frac{{{n}^{3}}}{n+1}\log (2)-\frac{{{n}^{3}}(2n+7)}{4(n+1)(n+2)(n+3)}-\underbrace{\frac{2{{n}^{3}}}{(n+1)(n+2)(n+3)}}_{\le 2}\underbrace{\int\limits_{0}^{1}{\frac{{{x}^{n+2}}}{{{(1+x)}^{3}}}}dx}_{\le 1/(n+3)} \\
& =&({{n}^{2}}-n+1+\Oh(1/n))\log (2)-\frac{n}{2}+\frac{5}{4}+\Oh(1/n)+\Oh(1/n) \\
& =&{{n}^{2}}\log (2)+n\left( -\frac{1}{2}-\log (2) \right)+\log (2)+\frac{5}{4}+\Oh(1/n) \\
&=& an^2+bn+c+\oh(1),
\end{eqnarray*}
where the asymptotics are obtained by calculating the partial fraction decomposition of the rational functions in $n$.

\newpage

\begin{figure}[h!]
 
  \scalebox{0.55} 
  {\includegraphics{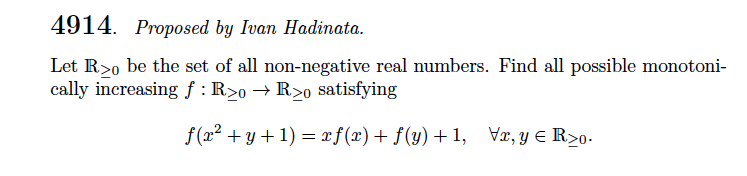}} 
\end{figure}

\centerline{\bf Solution to problem 4914 Crux Math. 50 (2) 2024, 82}\bigskip

 \centerline {Raymond Mortini, Rudolf Rupp}
 
  \medskip

\centerline{- - - - - - - - - - - - - - - - - - - - - - - - - - - - - - - - - - - - - - - - - - - - - - - - - - - - - -}
  
  \medskip

\ovalbox{We show that the {\it identity} is the only monotonically increasing solution on $[0,\infty[$. }\\

First we note that $f(x)=x$ trivially satisfies the functional equation on $[0,\infty[$. Now suppose that $f$ is a solution
and that $f$ is  increasing on $[0,\infty[$.\\

$\bullet $ Put $y=0$, respectively $ x=0$. Then 
$$\mbox{$f(x^2+1)=xf(x)+f(0)+1$ and $f(y+1)=f(y)+1$.}$$
Hence, with $y=x^2$, respectively $y=0$,
$$\mbox{$f(x^2+1)=f(x^2)+1$ and $f(1)=f(0)+1$}$$
 and so 
\begin{equation}\label{4914-1}
f(x^2)+1=f(x^2+1)=x f(x)+f(1).
\end{equation}
This implies that, with $x=1$,
$$f(1)+1 = 1 \cdot f(1)+f(1)=2 f(1).$$
Hence  $f(1)=1$ and so $f(0)=0$. Consequently

\begin{equation}\label{newfeq}
f(x^2) =f(x^2+1)-1=xf(x)+f(1)-1=xf(x).
\end{equation}

 $\bullet$ Next we note that $f$ is right-continuous at $x_1=1$. In fact, since 
$f$ is increasing, $f$ is bounded in a right-neighborhood of $x_0=0$ \footnote{For this,  it is  important that $f$ is defined at $0$.}. Hence, if $x\to 0$, we have that $xf(x)\to 0$.
Consequently, by  the second identity in (\ref{4914-1})
$$ \lim_{x\to 0} f(x^2+1)=f(1)=1.$$

$\bullet$  Via induction we obtain from $f(x^2)=xf(x)$, or  equivalently   $f(x)=\sqrt x\, f(\sqrt x)$ (just  replace $x^2$ by $x$), that  
$$\frac{f(\sqrt[2^n]{x})}{\sqrt[2^n]{x}}=\frac{f(x)}{x}.$$

Now $\sqrt[2^n]{x}\to 1$ for $x>0$ and $\sqrt[2^n]{x}\geq 1$ for $x\geq 1$. Thus, the right-continuity of $f$ at $x_1=1$ yields that 
$$\frac{f(1)}{1}=\frac{f(x)}{x},$$
 from which we conclude that $f(x)=x$ for $x\geq 1$.  

Now let $0\leq x\leq 1$. Then $x+1\geq 1$ and so, due to $f(y+1)=f(y)+1$, 
$$x+1=f(x+1)=f(x)+1.$$
Consequently $f(x)=x$, too. \\


{\bf Remark 1}  Our proof shows that instead of $f$ being increasing, we may have assumed merely that $f$ is  bounded in a right neighborhood of the origin. \\

{\bf Remark 2} A small modification of the proof (see below) shows that any solution to the functional equation 
$$(E) \hspace{2cm}f(x^2+y+1)=xf(x)+f(y)+1, \quad (x,y\geq 0)$$
\\
 actually is 
additive on $[0,\infty[$; that is satisfies $f(u+v)=f(u)+f(v)$ with $f(0)=0$ and $f(1)=1$. Thus we obtain from the well known fact on the Cauchy functional equation (restricted to the non-negative reals) that actually every measurable solution of (E) coincides with the identity.\\

\footnotesize{ In fact, $f(x^2)\buildrel=_{}^{(\ref{newfeq})} xf(x)$ and  $f(y+1)=f(y)+1$ imply that (E) becomes

\begin{equation}\label{zw}
f(x^2+y+1)=f(x^2)+f(y)+1=f(x^2)+ f(y+1), \quad (x,y\geq 0), 
\end{equation}

and so
$$\mbox{$f(u+v)=f(u)+f(v)$ for $u,v\geq 0$},$$

due to the following reason: since for $y\geq 1$, $f(y-1)=f(y)-1$, 

\begin{eqnarray*}
f(u) +f(v)&= &f(u)+f \big((v+1)-1\big)=f(u) + f(v+1) -1\\
&\buildrel=_{(\ref{zw})}^{}&f(u+(v+1))  -1=f(u+v)+1-1\\
&=&f(u+v).
\end{eqnarray*}
}

\newpage

 \begin{figure}[h!]
 
  \scalebox{0.55} 
  {\includegraphics{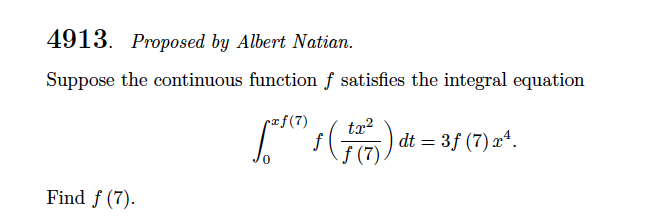}} 
\end{figure}

\centerline{\bf Solution to problem 4913 Crux Math. 50 (2) 2024, 82}\bigskip

 \centerline {Raymond Mortini, Rudolf Rupp}
 
  \medskip

\centerline{- - - - - - - - - - - - - - - - - - - - - - - - - - - - - - - - - - - - - - - - - - - - - - - - - - - - - -}
  
  \medskip

 Let $a\in \R$  and, for $f\in C(\R)$, let
$$I_f(x,a):=\int_0^{xa} f(tx^2 a^{-1}) dt.$$
We show that \ovalbox{$f(7)=42$ whenever $I_f(x,a)=3a x^4$.}\\

{\bf Proof.}
Using for $x\not=0$ the linear substitution $t\to u$ with  $u:=tx^2a^{-1}$ and $dt= a x^{-2} du$, we obtain
\begin{eqnarray*}
I_f(x,a)&=&a x^{-2}\;\int_0^{x^3} f(u) du.
\end{eqnarray*}
Now $I_f(x,a)=3a x^4$ if and only if  
$$\int_0^{x^3} f(u) du= 3 x^6.$$
Differentiating yields 
$$3 x^2\,f(x^3) =18 x^5,$$
equivalently
$$f(x^3)=6x^3.$$
As $x\mapsto x^3$ is a bijection from $\R$ onto $\R$, we obtain $f(u)=6u$.  Conversely, it is straightforward to check that this $f$ satisfies for every $a$ the given integral equation. So $f(7)=42$ independently of $a$.

\newpage

 \begin{figure}[h!]
 
  \scalebox{0.55} 
  {\includegraphics{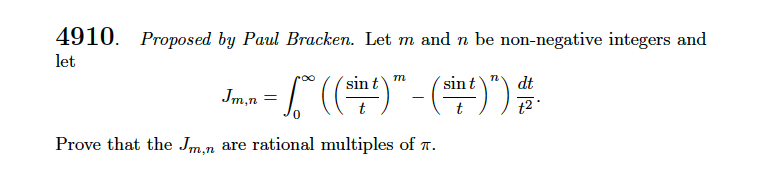}} 
\end{figure}

\centerline{\bf Solution to problem 4910 Crux Math. 50 (1) 2024, 38}\bigskip

 \centerline {Raymond Mortini, Rudolf Rupp}
 
  \medskip

\centerline{- - - - - - - - - - - - - - - - - - - - - - - - - - - - - - - - - - - - - - - - - - - - - - - - - - - - - -}
  
  \medskip


It is sufficient to consider the case $n=0$, otherwise write
$$\int \frac{S^m-S^n}{x^2} dx=\int \frac{S^m-1}{x^2}dx + \int\frac{1-S^n}{x^2}dx,$$
where $\dis S=\frac{\sin x}{x}.$
Since $|S|\leq 1$, we see that $\int_1^\infty \frac{S^m-1}{x^2}dx $ converges. Now use that
$$1-S=\frac{x^2}{3!}-\frac{x^4}{5!}\pm\cdots=x^2\left(\frac{1}{6}+\oh(x)\right)\quad {\rm as}\quad x\to 0,$$
and
$$|S^m-1|=|S-1|\; \Big|\sum_{j=0}^{m-1}S^j\Big|\leq m |S-1|, $$
to conclude that  $\dis \int_0^1 \frac{S^m-1}{x^2}dx $ converges, too. Hence
$$I(m):=\dis \int_0^\infty  \frac{S^m-1}{x^2}dx $$
 converges. Next we write
$$J:=\frac{S^m-1}{x^2}= \frac{(\sin x)^m-x^m}{x^{m+2}}=: \frac{f(x)}{x^{m+2}}.$$
Now we apply Apostol's method (see \cite{apo}). Integration by parts 
$\int uv' =uv-\int u'v$ with $u=f$ and $v'= x^{-m-2}$ yields:

\begin{eqnarray*}
I(m)=\frac{1}{m+1}\int_0^\infty \frac{f'(x)}{x^{m+1}}dx,
\end{eqnarray*}
since  $\dis\lim_{x\to 0} \frac{f(x)}{x^{m+1}}=0$ (and $\dis\lim_{x\to \infty} \frac{f(x)}{x^{m+1}}=0$) because
\begin{eqnarray*}
\Big| \frac{f(x)}{x^{m+1}}\Big|&\leq &m \frac{|S-1|}{x}= mx\left(\frac{1}{6}+\oh(x)\right) \quad {\rm as}\quad x\to 0.
\end{eqnarray*}

Similarily, since $0$ is a zero of order $1$ of the analytic   
\footnote{ Note that $J(z)= z\frac{\phantom{\Big|}\left( \frac{\sin z}{z}\right)^{m}-1}{\phantom {\big|}z^2}= z
\left(-\frac{1}{3!}+ \frac{z^2}{5!}+\cdots\right) R(z)$, where
$\dis\lim_{z\to 0} R(z)=\lim_{z\to 0 }\sum_{j=0}^{m-1}\left( \frac{\sin z}{z}\right)^j=m$.
}
function $\dis J(z):=\frac{(\sin z)^m -z^m}{z^{m+1}}$,
we have that for all $j=0,1,\dots, m$
$$\lim_{x\to 0} \frac{f^{(j)}(x)}{x^{m+1-j}}=0.$$
Hence, by repeating this procedure another $m$-times, we obtain
$$I(m)=\frac{1}{(m+1)!}\int_0^\infty \frac{f^{(m+1)}(x)}{x} dx.$$

Now $f^{(m+1)}(x)=\frac{d^{m+1}}{dx} (\sin x)^m-0$. Next we "linearize" the sinus-power:

\begin{eqnarray*}
(\sin x)^m&=&\left(\frac{e^{ix}-e^{-ix}}{2i}\right)^m=\left(\frac{1}{2i}\right)^m\,\sum_{j=0}^m (-1)^{m-j}{m\choose j}e^{ijx} e^{-i(m-j)x}\\
&=& (-1)^m \frac{1}{(2i)^m} \sum_{j=0}^m (-1)^{-j}{m\choose j} e^{ix(2j-m)}.
\end{eqnarray*}
Since the "constant" term (appearing for $j=m/2$ when $m$ is even) is annihilated by the derivative, we find
\begin{eqnarray*}
\frac{d^{m+1}}{dx} (\sin x)^m&=&(-1)^m  (-i)^{m+1}\frac{1}{(2i)^m} \sum_{0\leq j<\frac{m}{2}}  (m-2j)^{m+1}(-1)^{j}{m\choose j} e^{-ix(m-2j)}\\
&&+(-1)^m i^{m+1} \frac{1}{(2i)^m} \sum_{\frac{m}{2}< j\leq m}(2j-m)^{m+1} (-1)^{j}{m\choose j} e^{ix(2j-m)}.
\end{eqnarray*}
As the left hand side is real, we may take the real part on the right hand side and get (by observing ${\rm Re}\, iz=-{\rm Im}\, z)$
\begin{eqnarray*}
\frac{d^{m+1}}{dx} (\sin x)^m&=&-\frac{1}{2^m} \sum_{0\leq j<\frac{m}{2}}  (m-2j)^{m+1}(-1)^{j}{m\choose j} \sin((m-2j)x)\\
&&+ (-1)^{m+1}\frac{1}{2^m} \sum_{\frac{m}{2}< j\leq m}  (2j-m)^{m+1}(-1)^{j}{m\choose j} \sin((2j-m)x).
\end{eqnarray*}

Finally,  as $$\int_0^\infty \frac{\sin (px)}{x}dx= \int_0^\infty \frac{\sin (x)}{x}dx=\frac{\pi}{2}$$
whenever $p>0$,  we deduce that
\begin{eqnarray*}
I(m)=\frac{1}{(m+1)!} \int_0^\infty \frac{f^{(m+1)}(x)}{x}dx&=
&\frac{\pi}{2} \frac{1}{(m+1)!}\Bigg[-\frac{1}{2^m} \sum_{0\leq j<\frac{m}{2}}  (m-2j)^{m+1}(-1)^{j}{m\choose j} \\
&&+ (-1)^{m+1}\frac{1}{2^m} \sum_{\frac{m}{2}< j\leq m}  (2j-m)^{m+1}(-1)^{j}{m\choose j} \Bigg]\\
\end{eqnarray*}
which surely is a rational multiple of $\pi$.
Making in the second summand the substitution $k=m-j$, then we obtain
$$I(m)=-\frac{\pi}{2^m (m+1)!} \sum_{0\leq j<\frac{m}{2}}  (m-2j)^{m+1}(-1)^{j}{m\choose j}. $$

For instance $I(1)=-\frac{\pi}{4}, I(2)=-\frac{\pi}{3}, I(3)=-\frac{13 \pi}{32}$.

\newpage

 \begin{figure}[h!]
 
  \scalebox{0.55} 
  {\includegraphics{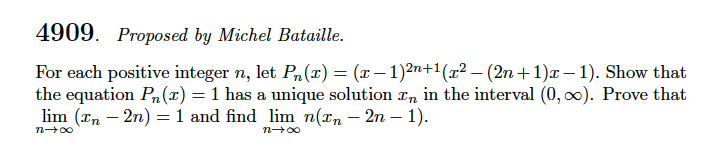}} 
\end{figure}

\centerline{\bf Solution to problem 4909 Crux Math. 50 (1) 2024, 38}\bigskip

 \centerline {Raymond Mortini, Rudolf Rupp}
 
  \medskip

\centerline{- - - - - - - - - - - - - - - - - - - - - - - - - - - - - - - - - - - - - - - - - - - - - - - - - - - - - -}
  
  \medskip


First we note that $P_n$ is continuous on $[0,\infty[$, $P_n(0)=1$ and that $\lim_{n\to\infty} P_n(x)=\infty$. Now
\begin{eqnarray*}
P_n'(x)&=& x(x-1)^{2n} \Big((2n+3) x- (4n^2+6n+4)\Big)\\
&=& x(x-1)^{2n}(2n+3) \Big( x-\big(2n+\frac{4}{2n+3}\big)\Big).
\end{eqnarray*}
Hence $P_n$ is strictly decreasing on $\dis\left[0, 2n+\frac{4}{2n+3}\right]$ and strictly increasing on 
$\dis\left[2n+\frac{4}{2n+3},\infty\right]$. Combining all this, and thanks to the intermediate value theorem, 
we deduce there exists a unique $x_n\in  ]0,\infty[$ with $P_n(x_n)=1$.
Next we discuss the assymptotics of the sequence $(x_n)$. Since  $x_n> 2n+\frac{4}{2n+3}$, we see that $x_n\to\infty$.

$\bullet$~  As
\begin{eqnarray*}
1=P_n(x_n)&=& (x_n-1)^{2n+1}(x_n^2-(2n+1)x_n-1)\\
&=&x_n(x_n-1)^{2n+1} (x_n-2n-1 -x_n^{-1})
\end{eqnarray*}
we get
\begin{equation}\label{rela}
 \frac{1}{x_n(x_n-1)^{2n+1}} +\frac{1}{x_n}=x_n-2n-1.
\end{equation}
But $x_n\to\infty$. Thus $x_n-2n-1\to 0$, from which we conclude that \ovalbox{$x_n-2n\to 1$}. In particular,
$$
\frac{x_n}{n} -2=\frac{x_n-2n}{n}\to 0.
$$

$\bullet$~ By (\ref{rela}),
\begin{eqnarray*}
n\,(x_n-2n-1)&=&\frac{n}{x_n}+ \frac{n}{x_n(x_n-1)^{2n+1}}\\
&=& \frac{n}{x_n}\left( 1+ \frac{1}{(x_n-1)^{2n+1}}\right)\to \frac{1}{2} ( 1+ 0)=\frac{1}{2}.
\end{eqnarray*}

\newpage

 \begin{figure}[h!]
 
  \scalebox{0.55} 
  {\includegraphics{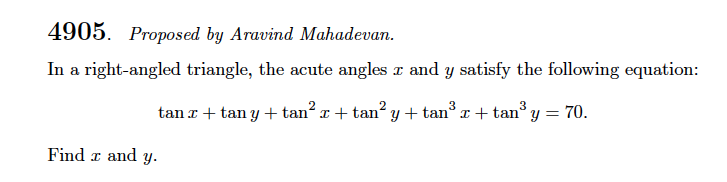}} 
\end{figure}

\centerline{\bf Solution to problem 4905 Crux Math. 50 (1) 2024, 37}\bigskip

 \centerline {Raymond Mortini, Rudolf Rupp}
  \medskip

\centerline{- - - - - - - - - - - - - - - - - - - - - - - - - - - - - - - - - - - - - - - - - - - - - - - - - - - - - -}
  
  \medskip


We show that \ovalbox{$\dis(x,y)=\left(\frac{\pi}{12},\frac{5\pi}{12}\right)$}\;, or in terms of degrees \ovalbox{$15^\circ$ and $75^\circ$}\;.
\bigskip

We may assume that  $0\leq x\leq y\leq \pi/2$ and $x+y=\pi/2$. Now 
$$\tan y=\tan\left(\frac{\pi}{2}-x\right) =\cot(x)=\frac{1}{\tan x}.$$
So we have to solve for $t=\tan x$  the equation
\begin{equation}\label{tanequ}
t+t^2+t^3+\frac{1}{t}+\frac{1}{t^2}+\frac{1}{t^3}=70
\end{equation}
(or equivalently $1+t+t^2-70 t^3 +t^4+t^5+t^6=0$). Such symmetric equations are solved using the substitution $s:=t+1/t$. 
Now $s^2=(t+\frac{1}{t})^2=t^2+\frac{1}{t^2}+2$; hence $t^2+\frac{1}{t^2}=s^2-2$. Moreover, 
$$s^3=\left(t+\frac{1}{t}\right)^3=t^3+3t+\frac{3}{t}+\frac{1}{t^3}$$
and so $t^3+\frac{1}{t^3}= s^3-3s$. This yields the equation
$s+s^2-2+s^3-3s=70$, or equivalently $s^3+s^2-2s-72=0$.
 As $s=4$ is a solution, we obtain the factorization
 $$0=(s-4)(s^2+5s+18)=(s-4) \left( \Big(s+\frac{5}{2}\Big)^2 +\frac{47}{4}\right).$$
 So $s=4$ is the only real solution. The equation $4=t+\frac{1}{t}$ now is equivalent to $t^2-4t+1=0$, which has
 $2\pm \sqrt 3$ as solutions. Now we have to calculate the values $x$ for which $\tan x=2\pm \sqrt 3$. 
 As is well known, $\arctan (2-\sqrt 3)=\pi/12$ and $\arctan(2+\sqrt 3)=5\pi/12$.  This can be verified by using the  formulas
 $$\sin x= \sqrt{\frac{1-\cos 2x}{2}}~~ {\rm and}~~\cos x=\sqrt{\frac{1+\cos 2x}{2}}.$$

\newpage

 \begin{figure}[h!]
 
  \scalebox{0.55} 
  {\includegraphics{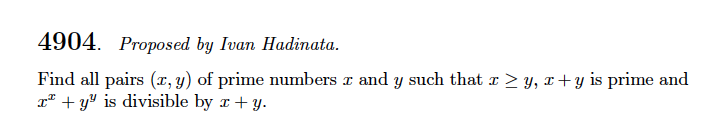}} 
\end{figure}

\centerline{\bf Solution to problem 4904 Crux Math. 50 (1) 2024, 37}\bigskip

 \centerline {Raymond Mortini, Rudolf Rupp} \medskip

\centerline{- - - - - - - - - - - - - - - - - - - - - - - - - - - - - - - - - - - - - - - - - - - - - - - - - - - - - -}
  
  \medskip

 This problem has no solution \footnote{ Under the usual assumption that the number 1 is not considered as a prime number.}. In fact,
 suppose that $x+y$ is prime whenever $x$ and $y $ are prime, $y\leq x$.  Since the sum of two odd prime numbers is even, 
it cannot be prime. Hence $y=2$. Let $n:=x$ \footnote{ The symbol $x$ for a natural number hurts my eyes \Smiley.}. Then we have to discuss the property $4+n^n$ is divisible by $2+n$.
\begin{eqnarray*}
\frac{4+n^n}{2+n}&=&\frac{4 +((n+2)-2)^n}{2+n}=\frac{4+\dis \sum_{j=0}^n (n+2)^j {n\choose j} (-1)^{n-j} 2^{n-j}}{n+2}\\
&=&\frac{4+(-1)^n 2^n + \dis \sum_{j=1}^n (n+2)^j {n\choose j} (-1)^{n-j} 2^{n-j}}{n+2}\\
&=&\frac{4+(-1)^n 2^n}{n+2} + \dis \sum_{j=1}^n (n+2)^{j-1} {n\choose j} (-1)^{n-j} 2^{n-j}\\
&=:&\frac{4+(-1)^n 2^n}{n+2}+m,
\end{eqnarray*}
where $m\in \Z$ (note that the binomial coefficients belong to $\N$).
If $n\geq 3$ is odd, then $2+n$ is odd and therefore $2+n$ cannot divide (in $\Z$) the even number $4+(-1)^n 2^n$. Hence the
 primeness of $n$ implies that $n=2$.  Since  $x+y=2+2=4$ is not prime, the pair $(2,2)$ is not  a solution either.

\newpage

 \begin{figure}[h!]
 
  \scalebox{0.55} 
  {\includegraphics{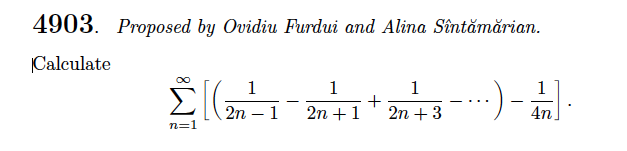}} 
\end{figure}

\centerline{\bf Solution to problem 4903 Crux Math. 50 (1) 2024, 37}\bigskip

 \centerline {Raymond Mortini, Rudolf Rupp} \medskip

\centerline{- - - - - - - - - - - - - - - - - - - - - - - - - - - - - - - - - - - - - - - - - - - - - - - - - - - - - -}
  
  \medskip


Let
$$S:= \sum_{n=1}^\infty \left(-\frac{1}{4n}+\sum_{k=0}^\infty \frac{(-1)^k}{2n+2k-1}\right)$$
which is the concise form of the sum in the problem. We claim that

$$\ovalbox{$\dis S= \frac{\log 2}{2}+\frac{\pi}{8}$}.$$

{\bf Solution}\\

For $n\geq 1$, let
$$F_n(x):=\sum_{k=0}^\infty \frac{(-1)^k}{2n+2k-1}x^{2n+2k-1}$$
be the generating function.
Then $F_n(x)$ converges for $0\leq x\leq 1$ (Leibniz rule for the alternating series at  $x=1$), and by Abel's rule, 
$F_n$ is continuous on $[0,1]$.
Now
$$F_n'(x)=\sum_{k=0}^\infty (-1)^k x^{2n+2k-2}=\frac{x^{2n-2}}{1+x^2}.$$
Since $F_n(0)=0$, we obtain
$$F_n(1)= \sum_{k=0}^\infty \frac{(-1)^k}{2n+2k-1}=\int_0^1 \frac{x^{2n-2}}{1+x^2}\;dx.$$
Note that  $S=\sum_{n=1}^\infty \left(F_n(1)-\frac{1}{4n}\right)$.  
Partial integration  $\int u'v =uv -\int uv' $with
$$\mbox{$\dis u'= x^{2n-2}$ and $v=(1+x^2)^{-1}$}$$
 yields
\begin{eqnarray*}
\int_0^1 \frac{x^{2n-2}}{1+x^2}\;dx&=& \frac{x^{2n-1}}{2n-1}\;\frac{1}{1+x^2}\Big|^1_0+\frac{1}{2n-1}\;\int_0^1 x^{2n-1} \frac{2x}{(1+x^2)^2}\;dx\\
&=&\frac{1}{2(2n-1)}+\frac{2}{2n-1}\int_0^1 \frac{x^{2n}}{(1+x^2)^2}\;dx.
\end{eqnarray*}

Hence, by using that $\sum\int=\int\sum$ as all factors are positive, and the fact that
$$\sum_{n=1}^\infty \left(\frac{1}{2(2n-1)}-\frac{1}{4n}\right)=\lim_{N\to\infty}\sum_{n=1}^N \left(\frac{1}{2(2n-1)}-\frac{1}{4n}\right)
=\frac{1}{2} \lim_{N\to\infty} \sum_{j=1}^{2N}\frac{(-1)^{j-1}}{j}=\frac{\log 2}{2},$$
we obtain
\begin{eqnarray*}
S&=& \sum_{n=1}^\infty \left(\frac{1}{2(2n-1)}-\frac{1}{4n}\right) + 2\;\int_0^1 \left( \sum_{n=1}^\infty \frac{x^{2n-1}}{2n-1}\right)
 \frac{x}{(1+x^2)^2}\;dx\\
 &=& \frac{\log 2}{2}+\int_0^1 \left(\log\left(\frac{1+x}{1-x}\right)\; \frac{x}{(1+x^2)^2}\right)dx\\
 &=:&\frac{\log 2}{2}+\int_0^1I(x) dx.
\end{eqnarray*}

To calculate a primitive of $I(x)$, we use partial integration with $u=\log\left(\frac{1+x}{1-x}\right)$ and $v'= \frac{x}{(1+x^2)^2}$. Hence

\begin{eqnarray*}
\int I(x) dx&=&-\frac{1}{2}\frac{1}{1+x^2}\log\left(\frac{1+x}{1-x}\right) + \int \frac{1}{1-x^4}dx\\
&=&-\frac{1}{2}\frac{1}{1+x^2}\log\left(\frac{1+x}{1-x}\right) + \frac{1}{4}\log\left(\frac{1+x}{1-x}\right)+ \frac{1}{2}\arctan x\\
&=& \frac{1}{4}\;\frac{x^2-1}{x^2+1}\; \log\left(\frac{1+x}{1-x}\right)+ \frac{1}{2}\arctan x\\
&=:& R(x).
\end{eqnarray*}

Hence 
$$\int_0^1 I(x)dx=\lim_{x\to 1}R(x) - R(0)=\frac{\pi}{8}.$$

\newpage

 \begin{figure}[h!]
 
  \scalebox{0.45} 
  {\includegraphics{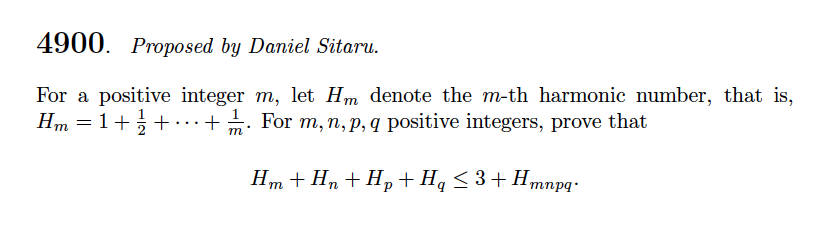}} 
\end{figure}

\centerline{\bf Solution to problem 4900 Crux Math. 49 (10) 2023, 541}\bigskip

 \centerline {Raymond Mortini, Rudolf Rupp} \medskip

\centerline{- - - - - - - - - - - - - - - - - - - - - - - - - - - - - - - - - - - - - - - - - - - - - - - - - - - - - -}
  
  \medskip

 
We first show that

\begin{equation}\label{harmo2}
\mbox{$H_p+H_q\leq 1+H_{pq}$ for $1\leq p\leq q$.}
\end{equation}

In fact, if $1=p=q$, then nothing has to be shown. So let us assume that $q\geq 2$. Then
\begin{eqnarray*}
1+H_{pq}&=&1+H_p+\bl{\left(\frac{1}{p+1}+\cdots+ \frac{1}{2p}\right)} + \left(\frac{1}{2p+1}+\cdots+\frac{1}{3p}\right)
+\cdots+ \bl{\left( \frac{1}{(q-1)p+1}+\cdots+ \frac{1}{qp}\right)}\\
&\geq& 1+H_p+\bl{p \cdot \frac{1}{2p}} +p \cdot \frac{1}{3p}+\cdots+ \bl{p \cdot \frac{1}{qp}}\\
&=&H_p+ H_q.
\end{eqnarray*}

Now let $1\leq m\leq n\leq p\leq q$ (of course this is without loss of generality). Then by (\ref{harmo2}),

\begin{eqnarray*}
(H_m+H_n)+(H_p+H_q)&\leq& (1+ H_{mn})+(1+H_{pq})= 2+ H_{mn} + H_{pq} \\
&\leq& 2 + 1+ H_{(mn)(pq)}=3 + H_{mnpq}.
\end{eqnarray*}

{\bf Remark 1} More generally, one can show that
$$\sum_{j=1}^n H_{n_j} \leq (n-1)+ H_{\prod _{j=1}^n n_j}.
$$

{\bf Remark 2} Solutions to the special case (\ref{harmo2}) above also appeared in Amer. Math. Monthly 56 (2) 1949, 109-110, Problem E819 Euler's constant.\\

{\bf Remark 3} A different proof of  $H_p+H_q\leq 1+H_{pq}$, $p,q \in \N=\{1,2,\dots\}$,   can be given via induction on $q$: for $q=1$, 
$H_p+H_1=H_p+1\leq 1 + H_{p\cdot 1}$. Now for $q\to q+1$, we use that
\begin{eqnarray*}
H_{p(q+1)}&=&H_{pq}+\frac{1}{pq+1}+\cdots +\frac{1}{pq+p}\geq H_{pq}+ \frac{p}{pq+p}= H_{pq}+\frac{1}{1+q}.
\end{eqnarray*}
Hence
$$H_p+H_{q+1}=H_p+H_q+\frac{1}{q+1}\leq 1+H_{pq}+\frac{1}{q+1}\leq 1+H_{p(q+1)}.$$

\newpage

 \begin{figure}[h!]
 
  \scalebox{0.45} 
  {\includegraphics{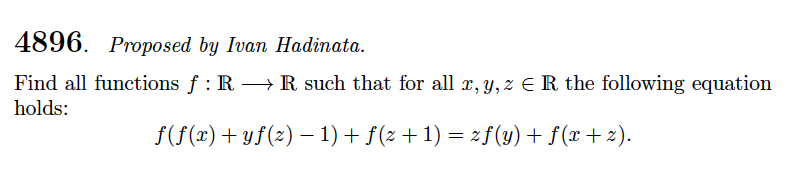}} 
\end{figure}

\centerline{\bf Solution to problem 4896 Crux Math. 49 (10) 2023, 540}\bigskip

 \centerline {Raymond Mortini, Rudolf Rupp} \medskip

\centerline{- - - - - - - - - - - - - - - - - - - - - - - - - - - - - - - - - - - - - - - - - - - - - - - - - - - - - -}
  
  \medskip

We show that all solutions $f:\R\to\R$ of the functional equation 
\begin{equation}\label{fequ3}
f\Big( f(x)+y f(z)-1\Big)+f(z+1)=zf(y)+f(x+z) \quad (x,y,z\in \R)
\end{equation}

are given by 
$$\ovalbox{$f\equiv 0$ or $f(x)=x$ for $x\in \R$.}$$

\bigskip

Obviously $f\equiv 0$  and the identity are solutions. Now let $f$ be  a solution to (\ref{fequ3}) with $f(y_0)\not=0$ for some $y_0\in \R$. 
 We claim that $f$ is surjective.\\
 
 In fact, put $x=1$ and $y=y_0$ in  (\ref{fequ3}). Then, for all $z\in \R$.
 \begin{equation}\label{fequ4}
 f\Big( f(1)+y_0 f(z)-1\Big)+f(z+1)=zf(y_0)+f(1+z) \iff f\Big(f(1)+y_0 f(z)-1\Big)= f(y_0)\, z.
 \end{equation}
 As the function $z\mapsto f(y_0)\,z$ is surjective, the function $z\mapsto f\Big(f(1)+y_0 f(z)-1\Big)$ is surjective, too. Hence
 $x\mapsto f(x)$ is surjective.\\
 
 Next put $y=0$ and $z=0$ in (\ref{fequ3}). Then, for all $x\in \R$
 \begin{equation}\label{fequ5}
 f\big(f(x)-1\big) +f(1)=f(x).
\end{equation}
Now put $u:=f(x)$. Note that if $x$ runs through $\R$, the surjectivity of $f$ implies that $u$ runs through $\R$, too. In particular, $f(u-1)=u-f(1)$ for every $u\in \R$ and so, with $v:=u-1$,  $f(v)=v+1-f(1)$. Now $v=1$ yields that $f(1)=2-f(1)$ and so $f(1)=1$.  Hence $f(v)=v$ for every 
$v\in \R$.


\newpage

 \begin{figure}[h!]
 
  \scalebox{0.45} 
  {\includegraphics{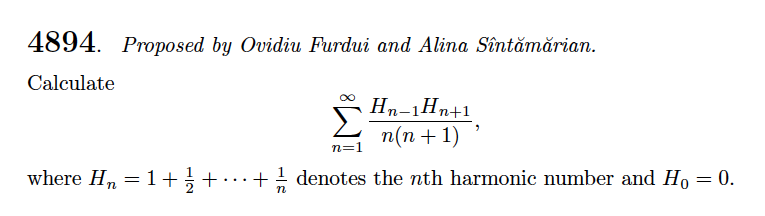}} 
\end{figure}

\centerline{\bf Solution to problem 4894 Crux Math. 49 (10) 2023, 539}\bigskip

 \centerline {Raymond Mortini, Rudolf Rupp} \medskip

\centerline{- - - - - - - - - - - - - - - - - - - - - - - - - - - - - - - - - - - - - - - - - - - - - - - - - - - - - -}
  
  \medskip


We show that
$$\ovalbox{$ \dis\sum_{n=1}^\infty \frac{H_{n-1} H_{n+1}}{n(n+1)}=3$}.$$

For the proof we shall decompose the series into several telescoping series. 

\begin{eqnarray*}
\frac{H_{n-1} H_{n+1}}{n(n+1)}&=&\frac{H_{n-1} H_{n+1}}{n}-\frac{H_{n-1} H_{n+1}}{n+1}=\frac{H_{n-1} H_{n+1}}{n}-
\frac{(H_n-\frac{1}{n}) (H_{n+2}-\frac{1}{n+2})}{n+1}\\
&=&\frac{H_{n-1} H_{n+1}}{n}-\frac{H_nH_{n+2}-\frac{1}{n}H_{n+2} - \frac{1}{n+2}H_n+ \frac{1}{n(n+2)}}{n+1}\\ 
&=&\left(\frac{H_{n-1} H_{n+1}}{n} -\frac{H_nH_{n+2}}{n+1}\right) + \frac{H_{n+2}}{n(n+1)}+ \frac{H_n}{(n+1)(n+2)}-\frac{1}{n(n+1)(n+2)}.
\end{eqnarray*}
Now 
\begin{eqnarray*}
\frac{H_{n+2}}{n(n+1)}&=&\left(\frac{H_{n+1}}{n} +\frac{\frac{1}{n+2}}{n}\right) -\frac{H_{n+2}}{n+1}= 
\left(\frac{H_{n+1}}{n}-\frac{H_{n+2}}{n+1}\right) +\frac{1}{n(n+2)}
\end{eqnarray*}
and
\begin{eqnarray*}
\frac{H_{n}}{(n+1)(n+2)}&=&\left(\frac{H_{n-1}}{n+1} +\frac{\frac{1}{n}}{n+1}\right) -\frac{H_{n}}{n+2}= \left(\frac{H_{n-1}}{n+1}-
\frac{H_{n}}{n+2}\right) +\frac{1}{n(n+1)}.
\end{eqnarray*}
Moreover,
\begin{eqnarray*}
\frac{2}{n(n+2)}&=&\frac{1}{n}-\frac{1}{n+2}=\left(\frac{1}{n}-\frac{1}{n+1}\right)+\left(\frac{1}{n+1}-\frac{1}{n+2}\right)\\
\frac{2}{n(n+1)(n+2)}&=&\frac{1}{n(n+1)}-\frac{1}{(n+1)(n+2)}.
\end{eqnarray*}
Hence 
\begin{eqnarray*}
\frac{H_{n-1} H_{n+1}}{n(n+1)}&=&\left(\frac{H_{n-1} H_{n+1}}{n} -\frac{H_nH_{n+2}}{n+1}\right) +
\bl{\left(\frac{H_{n+1}}{n}-\frac{H_{n+2}}{n+1}\right) }+\left(\frac{H_{n-1}}{n+1}-\frac{H_{n}}{n+2}\right)\\
&&+\bl{\frac{3}{2} \,\left(\frac{1}{n}-\frac{1}{n+1}\right)+\frac{1}{2}\,\left(\frac{1}{n+1}-\frac{1}{n+2}\right)}
 - \frac{1}{2}\left(\frac{1}{n(n+1)}-\frac{1}{(n+1)(n+2)}\right).
\end{eqnarray*}
Consequently

\begin{eqnarray*}
\lim_{N\to\infty}\sum_{n=1}^N \frac{H_{n-1} H_{n+1}}{n(n+1)}&=&\left(H_0H_2- \lim_{N\to\infty}\frac{H_N H_{N+2}}{N+1}\right)
+\bl{\left(H_2-\lim_{N\to\infty}\frac{H_{N+2}}{N+1}\right)}
+ \left(\frac{H_{0}}{2}-\lim_{N\to\infty}\frac{H_{N}}{N+2}\right)\\
&&+\bl{\frac{3}{2}+\frac{1}{4}}- \frac{1}{4}.
\end{eqnarray*}
Since $\gamma=\lim_{n\to\infty}(H_n-\log n)$, we have that  $\lim_{N\to\infty}\frac{H_N H_{N+2}}{N+1}=0$ as well as 
$\lim_{N\to\infty}\frac{H_{N+2}}{N+1}=0$ and $\lim_{N\to\infty}\frac{H_{N}}{N+2}=0$.
Consequently, by noticing that $H_0=0$, 
$$\sum_{n=1}^\infty \frac{H_{n-1} H_{n+1}}{n(n+1)}=H_2+ \frac{3}{2}=3.
$$

We thank Roberto Tauraso for confirming the result via Maple and wolframalpha.com, the latter though using a different representation:

  \begin{figure}[h!]
 
  \scalebox{0.35} 
  {\includegraphics{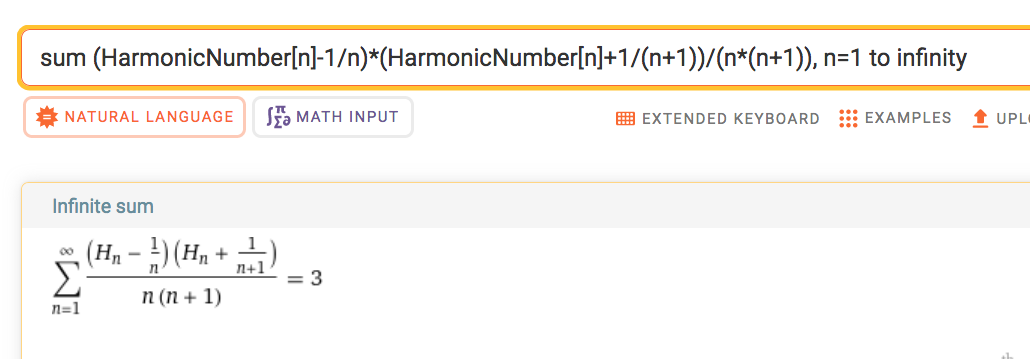}} 
\end{figure}

Calculating the sum beginning with index $n=2$, this software obtains the wrong result (the actual sum is of course 3, 
too as the first summand $\frac{(H_1-1)(H_1+0.5)}{1* 2}=0$):

\begin{figure}[h!]
 
  \scalebox{0.35} 
  {\includegraphics{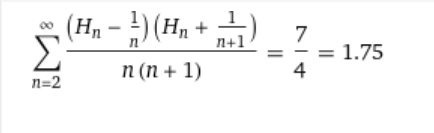}} 
\end{figure}

Very strange, too, is that the software does not give the correct value of the original sum but only very rough approximations:

\begin{figure}[h!]
 
  \scalebox{0.35} 
  {\includegraphics{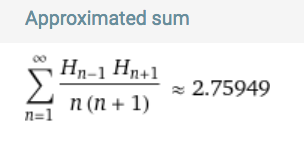}} 
\end{figure}

\begin{figure}[h!]
 
  \scalebox{0.35} 
  {\includegraphics{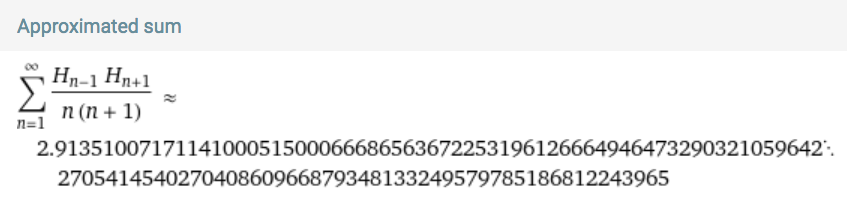}} 
\end{figure}

\newpage

 \begin{figure}[h!]
 
  \scalebox{0.55} 
  {\includegraphics{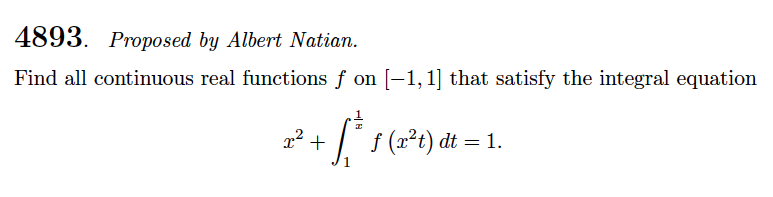}} 
\end{figure}

\centerline{\bf Solution to problem 4893 Crux Math. 49 (10) 2023, 539}\bigskip

 \centerline {Raymond Mortini, Rudolf Rupp} \medskip

\centerline{- - - - - - - - - - - - - - - - - - - - - - - - - - - - - - - - - - - - - - - - - - - - - - - - - - - - - -}
  
  \medskip

\bl{The statement of the problem is a bit ambiguous, as problems arise for $x=0$. 
Note that for $0<x\leq 1$ and $1\leq t\leq 1/x$ one has
$$0\leq x^2 t\leq x^2\frac{1}{x}=x\leq 1,$$
so that the integral  $\int_1^{1/x} f(x^2t)\,dt$  is well defined for $0<x\leq 1$.
Moreover,  for $-1\leq x<0$ and $1/x\leq t\leq 0$, one has
$$-1 \leq x= x^2\frac{1}{x}\leq x^2 t\leq 0, $$
and so the integral $\int_1^{1/x} f(x^2t)\,dt= -\int_0^1 f(x^2t)\,dt -\int_{1/x}^0 f(x^2t)\,dt$   is well defined for $-1\leq x<0$, too. \\
If $x=0$, though,  then the symbol $\int_1^{1/x}$ is not well defined as $1/0^+=\infty$ and $1/0^-=-\infty$. Actually
 no function can be a solution to  $x^2+\int_1^{1/x} f(x^2t) dt=1$  also at this point, as $\int_1^{\pm \infty} f(0) dt$ is divergent if $f(0)\not=0$,
and if $f(0)=0$, then $\int_1^{\pm \infty} 0\,dt=0$ but  $0+0\not=1$.\\
Thus we need to interprete at $x=0$ this functional equation as
$$\lim_{x\to 0} \left(x^2+\int_1^{1/x} f(x^2t) dt\right)=1.$$
}

We show that 
$$\ovalbox{$f(x)=2x$}$$
 is   the only continuous function $f:[-1,1]\to\R$  satisfying  for 
\co{$x\in [-1,1]\setminus \{0\}=:X$}  the integral equation
\begin{equation}\label{int-feq}
x^2+\int_1^{1/x} f(x^2t)\,dt=1,
\end{equation}
and
\begin{equation}\label{int-feq0}
\lim_{x\to 0} \left(x^2+\int_1^{1/x} f(x^2t) dt\right)=1.
\end{equation}
{\bf Proof}. For $x\in X$ and $f\in C[-1,1]$,  let $F(x):=x^2+\int_1^{1/x} f(x^2t) dt$. 
By the change of variable  $u:=x^2t$ we obtain
$$x^2 F(x)-x^4=\int_{x^2} ^x f(u) du.$$

So $F\equiv 1$ on $X$ if and only if $\int_{x^2} ^x f(u) du=x^2-x^4$ on $X$, hence also on $[-1,1]$. Hence, if $f\in C[-1,1]$ is a solution on $X$ to  (\ref{int-feq}) 
then, by taking derivatives,  $f(x)-f(x^2)=2x - 4x^3$ on $[-1,1]$. From this, we guess that $f(x)=2x$.  To this end, let 
$g(x):=f(x)-2x$, $x\in [-1,1]$.  Then $g\in C[-1,1]$ and  $g(x)=2x g(x^2)$ for $x\in [-1,1]$.   Since $g(-x)=-2x g(x^2)$, 
it suffices to determine $g$ for $x\in [0,1]$. 

By induction, for each $x\in [0,1]$,
$$g(x)=2^n x^{2^n-1}g(x^{2^n}).$$
Now, for $0\leq x<1$ we may let $n\to \infty$ and conclude (due to the continuity of $g$ at $0$ and $my^m\to 0$ for $0\leq y<1$) 
that $g(x)=\lim_{n\to\infty}2^nx^{2^n-1} \; g(0)= 0$. As $g$ is continuous at $1$, we deduce that $g\equiv 0$  on $[0,1]$, hence on $[-1,1]$, and so $f(x)=2x$ for $x\in [-1,1]$ whenever $f$ satisfies (\ref{int-feq}) on $X$.  Now it is straightforward to show that $2x$ also satisfies
(\ref{int-feq0}), that is
$$\lim_{x\to0} \left(x^2+\int_1^{1/x} f(x^2t) dt\right)=1.$$

\newpage

 \begin{figure}[h!]
 
  \scalebox{0.55} 
  {\includegraphics{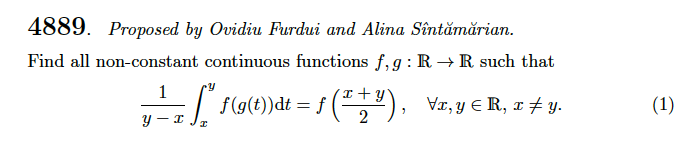}} 
\end{figure}

\centerline{\bf Solution to problem 4889 Crux Math. 49 (9) 2023, 491}\bigskip

 \centerline {Raymond Mortini, Rudolf Rupp} \medskip

\centerline{- - - - - - - - - - - - - - - - - - - - - - - - - - - - - - - - - - - - - - - - - - - - - - - - - - - - - -}
  
  \medskip


Consider the functional equation
 \begin{equation}\label{1n}
\frac{1}{y-x}\int_x^y f(g(t))dt= f\left(\frac{x+y}{2}\right) , x\not=y.
\end{equation}

We show that all nonconstant continuous functions $f:\R\to \R$ and $g:\R\to\R$ satisfying (1)
are given by 
$$\ovalbox{$f(x)=ax+b$ for $a\not=0$, $b\in \R$  and $g(x)=x$.}$$
\medskip

In fact, for fixed $x$, let $y\to x$. Then (1) implies (for instance via l'Hospital's rule) that 

\begin{equation}\label{fg}
f(g(x))=f(x)
\end{equation}
 for every $x\in R$. Hence
\begin{equation}\label{imv}
\frac{1}{y-x}\int_x^y f(t)dt= f\left(\frac{x+y}{2}\right) , x\not=y.
\end{equation}

Next we observe that any continuous $f$ satisfying (\ref{imv}), necessarily is $C^\infty$.
In fact, for all $x$,
$$ \int_{x-1}^{x+1} f(t) dt = 2 f(x).$$
As the function on the left obviously is differentiable by the fundamental theorem of calculus,
 we do have the same for the function on the right.  A calculation gives $f(x+1)-f(x-1)=2f'(x)$. Hence, 
  the continuity of $f$ implies that $f$ is  continuously differentiable.  Inductively, we now conclude that $f$ is $C^\infty$.
  
Now let
$$H(x,y):= \int_{y-x}^{y+x} f(t)dt.$$
Then, by assumption, 
$$H(x,y)= 2x f\left(\frac{y+x+(y-x)}{2}\right)=2x\;f(y).$$
Therefore,
$$H_y=f(y+x)-f(y-x)=2x f'(y),$$
and
$$H_x=f(y+x)+f(y-x)=2 f(y).$$
Addition yields
$$f(y+x)=xf'(y)+ f(y).$$
Now
\begin{eqnarray*}
f'(y+x)&=&\frac{\partial }{\partial y} f(y+x)= xf''(y)+f'(y)\\
f'(y+x)&=&\frac{\partial }{\partial x} f(y+x) =f'(y).
\end{eqnarray*}

Hence, for all $x$, we must have $xf''(y)=0$. As $f''$ is continuous,  $f''\equiv 0$, and so $f(x)=ax+b$ with $a\not=0$ 
(since $f$ is not constant) and $b\in \R$. Moreover, as we know  from (\ref{fg}) that $f(g(x))=f(x)$, the injectivity 
of the linear function $f$ implies that  $g$ is the identity.
\bigskip

We note that an equivalent for (\ref{imv}), the mid-point mean value theorem for derivatives, was dealt with in \cite{ca-lo}.

\newpage

  \begin{figure}[h!]
 
  \scalebox{0.55} 
  {\includegraphics{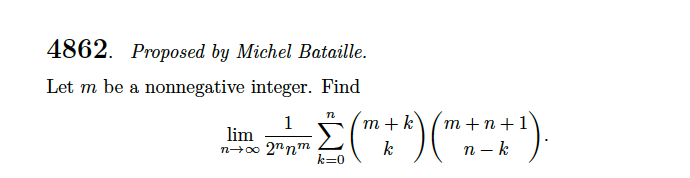}} 
\end{figure}

\centerline{\bf Solution to problem 4862 Crux Math. 49 (7) 2023, 375}\bigskip

 \centerline {Raymond Mortini, Rudolf Rupp} \medskip

\centerline{- - - - - - - - - - - - - - - - - - - - - - - - - - - - - - - - - - - - - - - - - - - - - - - - - - - - - -}
  
  \medskip


Let 
$$L_m(n):=\frac{1}{2^n}\;\frac{1}{n^m}\; \sum_{k=0}^n {m+k\choose k}\;{m+n+1\choose n-k}.
$$
We prove that 
$$ \ovalbox{$\dis\lim_{n\to \infty} L_m(n)= \frac{2}{m!}$}.$$

To this end, note that
\begin{eqnarray*}
{m+k\choose k}\;{m+n+1\choose n-k}&=&\frac{(m+k)!}{m!k!}\; \frac{(m+n+1)!}{(m+k+1)! (n-k)!}\\
&=& \frac{(m+n+1)!}{m!} \frac{1}{k! (m+k+1)(n-k)!}= \frac{(m+n+1)!}{m!n!}{n\choose k}\frac{1}{m+k+1} .
\end{eqnarray*}
Hence
$$\sum_{k=0}^n{m+k\choose k}\;{m+n+1\choose n-k}= \frac{(m+n+1)!}{m!n!}\; \sum_{k=0}^n {n\choose k}\frac{1}{m+k+1}.$$
Put
$$f(x):= \sum_{k=0}^n {n\choose k}\frac{1}{m+k+1}\; x^{m+k+1}.$$
Then
$$f'(x)= \sum_{k=0}^n  {n\choose k} x^{m+k}=x^m (1+x)^n.$$
Consequently, as $\int_0^1 f'(x)dx=f(1)-f(0)$ and $f(0)=0$,
$$ \sum_{k=0}^n{m+k\choose k}\;{m+n+1\choose n-k}=\frac{(m+n+1)!}{m!n!}\;\int_0^1 x^m(1+x)^n\;dx,$$
and so
\begin{equation}\label{lux}
L_m(n)=\frac{1}{n^m}\frac{(m+n+1)!}{m!n!}\;\int_0^1 x^m\left(\frac{1+x}{2}\right)^n\;dx.
\end{equation}

{\it Case 1} If $m=0$, then
$$L_0(n)=\frac{(n+1)!}{n!}\;\int_0^1 \left(\frac{1+x}{2}\right)^n\;dx=
(n+1) \left[\frac{2}{n+1} \left(\frac{1+x}{2}\right)^{n+1}\right]^1_0=2\left(1-\frac{1}{2^{n+1}}\right)\to 2.$$

{\it Case 2} $m\geq 1$.
We claim that
$$ R_n:=(n+1)\;\int_0^1 x^m \left(\frac{1+x}{2}\right)^{n}\;dx\to 2 \text{\quad as $n\to\infty$}.$$
In fact, partial integration yields

\begin{eqnarray*}
R_n&=& \left[ x^m 2 \left(\frac{1+x}{2}\right)^{n+1}\right]^1_0-
\int_0^1 mx^{m-1} 2  \left(\frac{1+x}{2}\right)^{n+1}\;dx\\
&=& 2- m \underbrace{\int_0^1x^{m-1} 2  \left(\frac{1+x}{2}\right)^{n+1}\;dx}_{:=I_n}.
\end{eqnarray*}

Since
$$0\leq  I_n\leq  2\int_0^1 \left(\frac{1+x}{2}\right)^{n+1}\;dx= \frac{1}{n+2}\left[\left( \frac{1+x}{2}\right)^{n+2}\right]^1_0\leq  \frac{1}{n+2},
$$
we conclude that $I_n\to $ and so $R_n\to 2$.

Together with (\ref{lux}), this finally yields  that
\begin{eqnarray*}
L_m(n)&=& \frac{1}{n^m}\frac{(m+n+1)!}{m!n!}\;\int_0^1 x^m\left(\frac{1+x}{2}\right)^n\;dx\\
&=& \frac{1}{m!}\; \frac{(m+n+1) (m+n)\dots(n+2)}{n^m} \; (n+1) \;\int_0^1 x^m \left(\frac{1+x}{2}\right)^{n}\;dx\\
&=&  \frac{1}{m!}\; \frac{m+n+1}{n}\frac{m+n}{n}\dots \frac{n+2}{n} \; R_n\\
&\to& 1^m\cdot  \frac{2}{m!}= \frac{2}{m!}.
\end{eqnarray*}

\newpage

  \begin{figure}[h!]
 
  \scalebox{0.45} 
  {\includegraphics{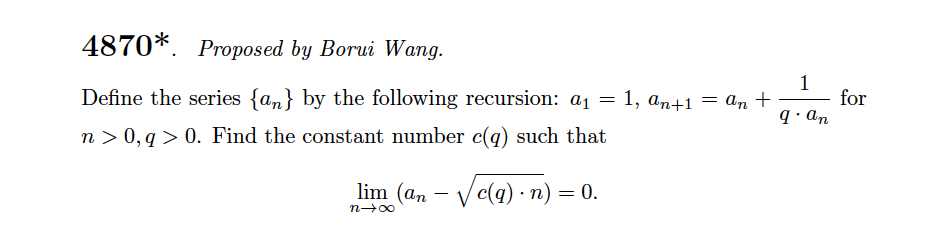}} 
\end{figure}

\centerline{\bf Solution to problem 4870 Crux Math. 49 (7) 2023, 377}\bigskip

 \centerline {Raymond Mortini, Rudolf Rupp} \medskip

\centerline{- - - - - - - - - - - - - - - - - - - - - - - - - - - - - - - - - - - - - - - - - - - - - - - - - - - - - -}
  
  \medskip


We will show the following:

\fbox{\parbox{13cm}{Let $q, c>0$ and define the sequence  (\footnotesize{not series like it is mentioned in the problem statement}) $(a_n)$ by $a_1=c$ and 
$\dis a_{n+1}=a_n+\frac{1}{q\cdot a_n}$ for $n\geq 1$. Then, with $\dis c(q)=2/q$
$$ \lim_{n\to\infty}\big|a_n- \sqrt{c(q) n}\big|=0.$$}}
\bigskip

{\bf Remark} If one starts with $c<0$, then all the $a_n$ are negative too, and one obtains $\big |a_n+ \sqrt{c(q) n}\big|\to 0$. 
(This is done by considering $b_n:=-a_n)$.

\bigskip

{\bf Solution}
Let $\dis h(x):= x+ \frac{1}{qx}$. Then  $h>0$ on $]0,\infty[$ and so $a_{n+1}=h(a_n)$ is well defined. Taking squares
$$a_{n+1}^2=a_n^2+\frac{2}{q} +\frac{1}{q^2 a_n^2},$$
or equivalently
$$a_{n+1}^2-a_n^2=\frac{2}{q} +\frac{1}{q^2 a_n^2},$$
we obtain  the finite  telescoping series:

$$a_{n+1}^2 - a_1^2=\sum_{k=1}^n (a_{k+1}^2 -a_k^2)= \frac{2}{q}\, n + \frac{1}{q^2} \sum_{k=1}^n \frac{1}{a_k^2}.$$

Hence
\begin{equation}\label{folg1}
a_{n+1}^2=c^2 +\frac{2}{q}\, n + \frac{1}{q^2} \sum_{k=1}^n \frac{1}{a_k^2}.
\end{equation}

This allows us to estimate $a_{n+1}$:

$$a_{n+1}^2\geq c^2 +\frac{2}{q}\, n,$$
and so by using this,
\begin{equation}\label{folg2}
a_{n+1}^2\leq   c^2 +\frac{2}{q}\, n + \frac{1}{q^2}\,\sum_{k=1}^n \frac{1}{\dis c^2+\frac{2}{q}(k-1)}.
\end{equation}

Next consider the decreasing function $\dis f(x)= \frac{1}{\dis c^2+ \frac{2}{q} x}$, $x\geq 0$. Then

\begin{eqnarray*}
\sum_{k=1}^n \frac{1}{\dis c^2+\frac{2}{q}(k-1)}&=& \frac{1}{c^2} +\sum_{k=1}^{n-1} f(k)\leq 
\frac{1}{c^2} +\int_0^{n-1} f(x)dx\\
&=&  \frac{1}{c^2}+\left[ \frac{q}{2}\log\left (c^2+ \frac{2}{q}x\right)\right]^{n-1}_0\\
&=& \frac{1}{c^2} + \frac{q}{2}\log\left( 1+ \frac{2}{qc^2}(n-1)\right).
\end{eqnarray*}

Thus we have arrived at the following estimates:

\begin{equation}
c^2 +\frac{2}{q}\, n\leq a_{n+1}^2\leq  c^2 +\frac{2}{q}\, n + \frac{1}{q^2}\;\frac{1}{c^2}+ 
\frac{1}{2q}\log\left( 1+ \frac{2}{qc^2}(n-1)\right).
\end{equation}

Hence
\begin{eqnarray*}
\Delta_n:=\left|a_{n+1}- \sqrt{\frac{2}{q} n}\right|&=&\frac{\Big|a_{n+1}^2- \frac{2}{q}n\Big|}{a_{n+1}+ \sqrt{\frac{2}{q} n}}\\
&\leq&\frac{c^2+\frac{1}{q^2}\;\frac{1}{c^2}+\frac{1}{2q}\log\left( 1+ \frac{2}{qc^2}(n-1)\right)}{\sqrt{c^2 +\frac{2}{q}\, n}+
\sqrt{\frac{2}{q} n} }\\
&\buildrel\longrightarrow_{n\to\infty}^{} &0.
\end{eqnarray*}

Finally
\begin{eqnarray*}
\left|a_{n}- \sqrt{\frac{2}{q} n}\right|&\leq& \left|a_{n}- \sqrt{\frac{2}{q} (n-1)}\right|+ \left|\sqrt{\frac{2}{q} (n-1)}-\sqrt{\frac{2}{q} n}\right| \\
&=&\Delta_{n-1}+\sqrt{\frac{2}{q} }\cdot \frac{1}{\sqrt{n-1}+\sqrt n}\\
&\buildrel\longrightarrow_{n\to\infty}^{} &0.
\end{eqnarray*}

\newpage

   \begin{figure}[h!]
 
  \scalebox{0.55} 
  {\includegraphics{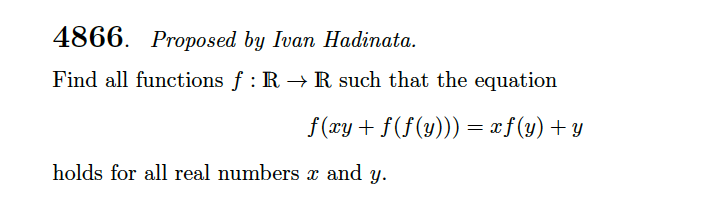}} 
\end{figure}

\centerline{\bf Solution to problem 4866 Crux Math. 49 (7) 2023, 376}\bigskip

 \centerline {Raymond Mortini, Rudolf Rupp} \medskip

\centerline{- - - - - - - - - - - - - - - - - - - - - - - - - - - - - - - - - - - - - - - - - - - - - - - - - - - - - -}
  
  \medskip


This is entirely trivial.
We claim that ${\ovalbox{\rm the identity}}$ is the only solution $f:\R\to \R$ to

\begin{equation}\label{sol1}
f(xy+f(f(y)))=xf(y)+y.
\end{equation}
\bigskip

We first show that  $f(0)=0$.   In fact, if $x=0$, then $f(f(f(y)))=y$. Now take $y=0$ in (\ref{sol1}). Then
$0=f(f(f(0)))=xf(0)$. Hence $f(0)=0$.

Next, we take $y=1$ in (\ref{sol1}). Then

\begin{equation}\label{sol2}
f(x+f(f(1)))=xf(1)+1.
\end{equation}

Put $u:=x+f(f(1))$. This yields

\begin{equation}\label{rud1}
f(u)=(u-f(f(1)))\; f(1) +1=:au+b
\end{equation}

In other words, $f$ necessarily is an affine function. As we already know, $f(0)=0$. Thus $b=0$. Now (\ref{sol1}) yields
$$a(xy+ a^2 y)=x ay+y.$$
Hence $a=1$.

\newpage

\nopagecolor

  \begin{figure}[h!]
 
  \scalebox{0.5} 
  {\includegraphics{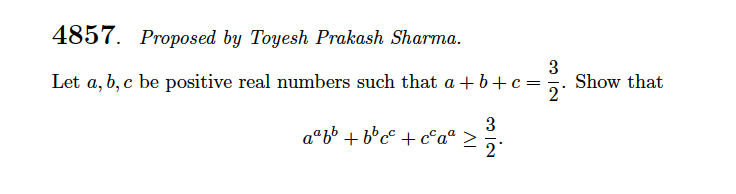}} 
\end{figure}


\centerline {\bf Solution to problem 4857 Crux Math. 49 (5) 2023, 323} \bigskip
 
  \centerline {Raymond Mortini} \medskip

\centerline{- - - - - - - - - - - - - - - - - - - - - - - - - - - - - - - - - - - - - - - - - - - - - - - - - - - - - -}
  
  \medskip

Since the function $\log x$ is concave on $ ]0,\infty[$, we have 
$$\log\left(\frac{A+B+C}{3}\right)\geq \frac{\log A+\log B+\log C}{3}=:R.$$
Here we take
$$A:=a^a b^b, B:=b^bc^c, C=c^ca^a.$$
Now the function $f(x):=2x\log x$ is convex on $]0,\infty[$, since $f''(x)=2/x\geq 0$.
Hence
\begin{eqnarray*}
2a\log a +2b\log b+2c\log c &=&f(a)+f(b)+f(c) \geq 3 f\left(\frac{a+b+c}{3}\right)\\
&=&3 f\left(\frac{1}{2}\right)=-3\log2.
\end{eqnarray*}
Hence $3 R\geq -3\log 2$, equivalently $R\geq \log(1/2)$ from which we deduce that
$A+B+C\geq 3/2$.
In other words
$$a^ab^b+b^bc^c+c^ca^a\geq \frac{3}{2}.$$

\newpage
  \begin{figure}[h!]
 
  \scalebox{0.5} 
  {\includegraphics{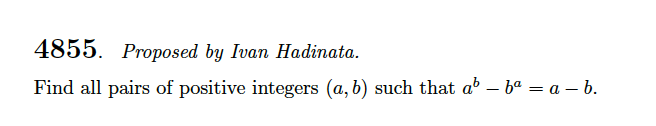}} 
\end{figure}


\centerline {\bf Solution to problem 4855 Crux Math. 49 (5) 2023, 323} \bigskip
 
  \centerline {Raymond Mortini, Rudolf Rupp } \medskip

\centerline{- - - - - - - - - - - - - - - - - - - - - - - - - - - - - - - - - - - - - - - - - - - - - - - - - - - - - -}
  
  \medskip

We claim that all solutions $(a,b)\in \N\times \N$ are given by
$$\ovalbox{$(1, v), (u,1), (t,t),  (2,3), (3,2)$}$$

where $u,v,t\in \N:=\{1,2,\cdots\}$ can be  arbitrarily chosen.\\

It is easily seen that these are solutions. Now let $(a,b)$ be a solution.  Then $(b,a)$ is a solution, too. If $b=a$, or if $b=1$, then nothing remains to be shown. So we may assume that
$a> b>1$.  Let $\log x$ be the natural logarithm. Now the function $f:x \mapsto x/\log x$ is strictly increasing for $x\geq e$ and strictly decreasing for $1<x\leq e$ with $ \min_{x>0} f(x)= e$. So if $a>b\geq 3> e$, 
$$\frac{a}{\log a}>\frac{b}{\log b}$$
or equivalently,

$$b^a> a^b.$$

Hence $a^b-b^a<0$, but $a-b>0$. So this case, where $a>b\geq 3$,  does not occur.  So it remains to consider the case  
$a>b=2$.  If $a=3$, then we actually have the solution $(3,2)$.  If $a\geq 4$, then 
$$\frac{a}{\log a}\geq \frac{4}{\log 4}= \frac{2}{\log 2},$$
and so
$$a^2-2^a\leq 0<a-2.$$
Thus this case $a\geq 4 >2=b$ does not occur, either. As all cases have been considered, we obtain the assertion.

\newpage
  \begin{figure}[h!]
 
  \scalebox{0.5} 
  {\includegraphics{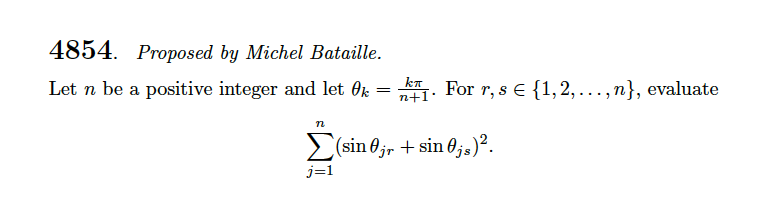}} 
\end{figure}


\centerline {\bf Solution to problem 4854 Crux Math. 49 (5) 2023, 323} \bigskip
 
  \centerline {Raymond Mortini, Rudolf Rupp } \medskip

\centerline{- - - - - - - - - - - - - - - - - - - - - - - - - - - - - - - - - - - - - - - - - - - - - - - - - - - - - -}
  
  \medskip

We prove that for $1\leq r,s\leq n$,

$$\ovalbox{$S:=\dis\sum_{j=1}^n \left(\sin\left( j \frac{r\pi}{n+1}\right)+ \sin\left( j \frac{s\pi}{n+1}\right)\right)^2=\begin{cases}
n+1&\text{if $r\not=s$}\\ 2(n+1)&\text{if $r=s$}\end{cases}
$}
$$
We first show that

\begin{equation}\label{cos1}
\sum_{j=1}^n \cos \left(j \dis \frac{2 \rho\pi }{n+1}\right)=\begin{cases} -1&\text{if $\rho\in \Z\setminus (n+1)\Z$}\\
n&\text{if $\rho\in (n+1)\Z$}
\end{cases}
\end{equation}

and that for odd $\rho\in \Z$
\begin{equation}\label{cos11}
\sum_{j=1}^n \cos \left(j \dis \frac{ \rho\pi }{n+1}\right)=0.
\end{equation}

To see this, we will use that  $\cos x={\rm Re} (e^{ix})$, and that

\begin{equation}\label{hauptung}
\sum_{j=1}^n e^{i j t}=e^{it} \sum_{j=0}^{n-1}e^{i jt}= e^{it} \frac{1-e^{int}}{1-e^{it}}=\frac{e^{it}-e^{i(n+1)t}}{1-e^{it}}.
\end{equation}

Now put  $t=2\rho\pi/(n+1)$ whenever  $\rho\in \Z\setminus (n+1)\Z$.  The latter guarantees that the denominator does not vanish. Hence
$$
\sum_{j=1}^n e^{i j \frac{2\rho\pi}{n+1}}= \frac{e^{i \frac{2\rho\pi}{n+1}}-1}{1-e^{i \frac{2\rho\pi}{n+1}}}=-1.
$$
Now if $\rho\in (n+1)\Z$, then, 
$$\sum_{j=1}^n e^{i j \frac{2\rho\pi}{n+1}}=n.$$
Thus (\ref{cos1}) holds.  If  $\rho$ is odd,  then, by putting $t=\rho\pi/(n+1)$ in (\ref{hauptung}), we obtain

\begin{eqnarray*}
\sum_{j=1}^n e^{i j \frac{\rho\pi}{n+1}}&=&\frac{e^{i \frac{\rho \pi}{n+1}}+1}{1-e^{i \frac{\rho\pi}{n+1}}}=
i\; \cot\left( \frac{1}{2}\;\frac{\rho \pi}{n+1}\right).
\end{eqnarray*}

This is a purely imaginary number, so its real part is 0. This yields (\ref{cos11}).

From  (\ref{cos1}) we easily deduce that for $r\in \{1,2,\dots,n\}$

\begin{equation}\label{cos2}
\sum_{j=1}^n \sin^2 \left(j \frac{r\pi}{n+1}\right)=\frac{n+1}{2}.
\end{equation}

In fact, using  that $\sin^2 x=\frac{1-\cos 2x}{2}$, we obtain from (\ref{cos1})
\begin{eqnarray*}
\sum_{j=1}^n \sin^2 \left(j \frac{r\pi}{n+1}\right)&=& \sum_{j=1}^n \frac{1-\cos \left(j \dis \frac{2 r\pi }{n+1}\right)}{2}\\
&=&\frac{n}{2} -\frac{1}{2}\;\sum_{j=1}^n \cos \left(j \dis \frac{2 r\pi }{n+1}\right)\\
&=&\frac{n+1}{2}.
\end{eqnarray*}

We are now ready to calculate the value of $S$.\\

$\bullet$ {\it Case 1}  \ovalbox{$r=s$.}  Then 
\begin{eqnarray*}
S=\sum_{j=1}^n \left(\sin\left( j \frac{r\pi}{n+1}\right)+ \sin\left( j \frac{r\pi}{n+1}\right)\right)^2&=&
4 \sum_{j=1}^n \sin^2\left( j \frac{r\pi}{n+1}\right)\\
&\buildrel=_{(\ref{cos2})}^{}& 4\frac{n+1}{2}=2 (n+1).
\end{eqnarray*}

$\bullet$ {\it Case 2} \ovalbox{$r\not=s$.}
Since  $r,s\in \{1,2,\dots,n\}$,  $r$ and $s$ do not belong to $(n+1)\Z$. 
Note that  due to $\sin x\sin y=\frac{1}{2}(\cos(x-y)-\cos(x+y))$,
\begin{equation}\label{cos3}
(\sin x+\sin y)^2=\sin^2 x+\sin^2 y + \cos(x-y)-\cos(x+y).
\end{equation}

Hence

\begin{eqnarray*}
S&=& \sum_{j=1}^n \sin^2\left( j  \frac{r\pi}{n+1}\right)+ \sum_{j=1}^n \sin^2\left( j  \frac{s\pi}{n+1}\right)+
\sum_{j=1}^n \cos\left(j \pi\frac{r-s}{n+1}\right)-\sum_{j=1}^n \cos\left(j \pi\frac{r+s}{n+1}\right)\\
&\buildrel=_{(\ref{cos2})}^{}& n+1+\sum_{j=1}^n \cos\left(j \pi \frac{r-s}{n+1}\right)-\sum_{j=1}^n \cos\left(j\pi \frac{r+s}{n+1}\right)\\
&=& n+1+ S_1-S_2.
\end{eqnarray*}

Several cases have to be analyzed now:\\

a) $r-s$ is even, say $r-s=2\rho$, where $\rho\in \Z$.  Then $r+s$ is even, too. Since $0<|r-s|\leq n-1$ and $0<r+s\leq 2n<2(n+1)$,
we again have two subcases: \\

a1) $r+s\not\in \Z(n+1)$ (equivalently $r+s\not=n+1)$: Then by (\ref{cos1}),
$$S=n+1+(-1)-(-1)=n+1.$$

a2) $r+s=n+1\in \Z(n+1)$:  Then $n$ is odd,  say $n=2m+1$ for some $m\in\{0,1,2,\dots\}$, and so  
$$S_2=\sum_{j=1}^{2m+1} \cos (j\pi)=\underline{(-1)+(+1)} +\cdots +\underline{(-1)+(+1)} +(-1)=-1.$$
Hence 
$$S=n+1 +(-1) -(-1)=n+1.$$

b) $r-s$ is odd. Then $r+s$ is odd, too. Again we have two subcases: \\

b1)  $r+s\not=n+1$: Then by (\ref{cos11}),
$$S=n+1 +0-0=n+1.
$$

b2) $r+s=n+1$. Then $n$  is even, say $n=2m$ with $m\in\{1,2,\dots\}$, and so
$$S_2=\sum_{j=1}^{2m} \cos (j\pi)=\underline{(-1)+(+1)} +\cdots +\underline{(-1)+(+1)} =0.$$
Hence
$$S=n+1+0-0=n+1.$$

\newpage

.\vspace{-5mm}
  \begin{figure}[h!]
 
  \scalebox{0.45} 
  {\includegraphics{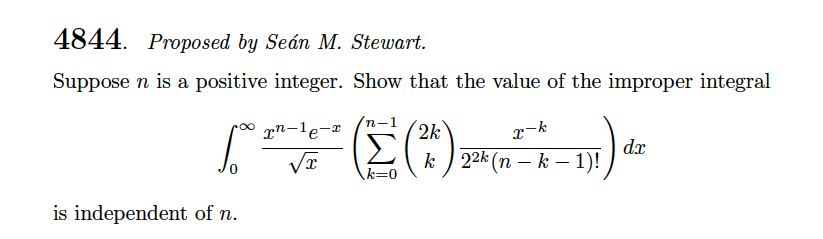}} 
\end{figure}

 \nopagecolor

\centerline {\bf Solution to problem 4844 Crux Math. 49 (5) 2023, 273} \bigskip
 
  \centerline {Raymond Mortini, Rudolf Rupp } \medskip

\centerline{- - - - - - - - - - - - - - - - - - - - - - - - - - - - - - - - - - - - - - - - - - - - - - - - - - - - - -}
  
  \medskip

For $n\geq 1$, let
$$I_n:=\int_0^\infty \frac{x^{n-1}e^{-x}}{\sqrt x}\left( \sum_{k=0}^{n-1} {2k\choose k} \frac{x^{-k}}{2^{2k}(n-k-1)!}\right)\;dx.$$
We show that
$$\ovalbox{$\dis I_n=\sqrt \pi$}.$$
\vspace{0.5cm}

We use the following well-known formulas, where $\Gamma$ is the Gamma function:
\begin{equation}
\int_0^\infty x^{s-1} e^{-x}\;dx=\Gamma (s),\quad \Gamma(s+1)=s\;\Gamma(s),  s>0
\end{equation}
\begin{equation}
\Gamma\left(m+\frac{1}{2}\right)=\Gamma \left(\frac{1}{2}\right) \frac{1}{2}\;\frac{3}{2}\;\frac{5}{2}\cdots\frac{2m-1}{2}=
\sqrt \pi \; \frac{\prod_{k=1}^m (2k-1)}{2^{m}}
=\frac{(2m)!}{m! 4^m}\sqrt \pi
\end{equation}
So, with $m=n-k-1$,
\begin{eqnarray*}
I_n&=&\sum_{k=0}^{n-1}  {2k\choose k} \frac{1}{2^{2k}(n-k-1)!}\;\int_0^\infty x^{(n-k-1/2)-1} e^{-x}\; dx\\
&=&\sum_{k=0}^{n-1}  {2k\choose k} \frac{1}{2^{2k}(n-k-1)!}\; \Gamma(n-k-1/2)\\
&=&\sum_{k=0}^{n-1}  {2k\choose k} \frac{1}{2^{2k}(n-k-1)!}\; \frac{(2(n-k-1))!}{(n-k-1)! 4^{n-k-1}}\;\sqrt \pi\\
&=&\frac{1}{4^{n-1}}\sum_{k=0}^{n-1} \frac{(2k)!}{(k!)^2}  \frac{(2(n-k-1))!}{((n-k-1)!)^2} \sqrt \pi=
\frac{1}{4^{n-1}}\sum_{k=0}^{n-1} {2k\choose k} {2(n-k-1)\choose n-k-1}\sqrt \pi.
\end{eqnarray*}

This is related to  the coefficient in the Cauchy product of 
$$\sum_{n=0}^\infty {2n \choose n} x^n= \sum_{n=0}^\infty {-1/2\choose n} (-1)^n 4^nx^n=\frac{1}{\sqrt{1-4x }},$$
with itself  and which converges for $|x|<1/4$, or if we take $x=y/4$,
$$\sum_{n=0}^\infty \frac{1}{4^n} {2n \choose n} y^n=\frac{1}{\sqrt{1-y}}.$$
In fact, for $|y|<1$,
$$\sum_{m=0}^\infty y^m =\frac{1}{1-y}=\frac{1}{\sqrt{1-y}} \frac{1}{\sqrt{1-y}}=
\sum_{m=0}^\infty \left(\sum_{k=0}^m\frac{{2k\choose k}}{4^k} \frac{ {2(m-k) \choose m-k }}{4^{m-k}}\right) y^m
$$
The coefficients being unique, we deduce that for every $m=0,1,\cdots$ 
$$\sum_{k=0}^m\frac{{2k\choose k}}{4^k} \frac{ {2(m-k) \choose m-k }}{4^{m-k}}=1.$$
Hence, with $m=n-1$, we conclude that $I_n=\sqrt \pi$.

\newpage

 \begin{figure}[h!]
 
  \scalebox{0.5} 
  {\includegraphics{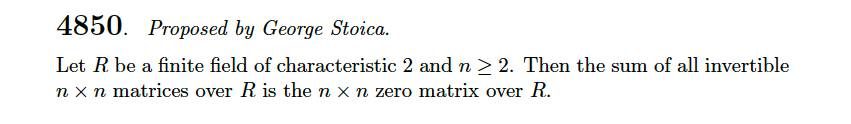}} 

\end{figure}
 
 \nopagecolor

\centerline {\bf Solution to problem 4850  Crux Math. 49 (5) 2023, 274, first version \footnote{ This was  tacitly replaced by another problem later on.}} \bigskip
 
  \centerline {Raymond Mortini, Rudolf Rupp } \medskip

\centerline{- - - - - - - - - - - - - - - - - - - - - - - - - - - - - - - - - - - - - - - - - - - - - - - - - - - - - -}
  
  \medskip

Let $n\geq 2$. We show that for {\it any} finite field the sum $S$ of all invertible $n\times n$ matrices is the $n\times n$ zero matrix $O_n$.\\

For $n\geq 1$, let $\mathcal M_n$ be the set of all $n\times n$ matrices and let $\mathcal U_n$ be the set of all  invertible $n\times n$ matrices. 
Since the field has only a finite number of elements, $\mathcal  U_n$ has  only a finite number of elements. So 
$S:=\sum_{U\in \mathcal U_n} U$ is  a well defined element in $\mathcal M_n$.
We will show that for every $\tilde U\in \mathcal U_n$,
$$S\cdot \tilde U=S.$$

Fix an invertible matrix $\tilde U\in \mathcal  U_n$ and consider the map
$$\iota:  \begin{cases}\mathcal  M_n&\to  \mathcal  M_n\\ X&\mapsto X\cdot \tilde U\end{cases}.$$
Then $\iota$ is a bijection of $\mathcal M_n$ onto itself. The inverse is given by $\iota^{-1}(Y)=Y \cdot \tilde U^{-1}$, since 
$$\iota\circ \iota^{-1} (Y)= \iota(Y \cdot \tilde U^{-1})= (Y \cdot \tilde U^{-1}) \cdot \tilde U=Y$$
and
$$\iota^{-1}\circ \iota (X)=\iota^{-1}( X\cdot \tilde U)= (X\cdot \tilde U)\cdot \tilde U^{-1}=X.$$
Moreover, and this is the main point here,  $\iota$ maps $\mathcal U_n$ bijectively onto itself.
Thus  (and here we have not yet used that $n\not =1$)

\begin{equation}\label{S=SU}
S=\sum_{U\in \mathcal U_n} \iota(U)= \iota (\sum_{U\in\mathcal U_n} U)=\iota(S)=S \cdot \tilde U.
\end{equation}

Now we use that $n\geq 2$. Take for $\tilde U$ and $1\leq i<j\leq n$ the elementary matrices 
$$E_{ij}=(\vec e_1,\dots,\underbrace{\vec e_j}_{\text{$i$-th col}},\dots, \underbrace{\vec e_i}_{\text{$j$-th col}},\dots, \vec e_n),$$
which interchange for $X\cdot  E_{ij}$ the $i$-th and $j$-th column of $X$. Thus $S \cdot E_{ij}=S$ implies that all the columns of $S$ are the same.
Say $S=(\vec s,\dots, \vec s)$. Next we consider the matrix 
$$E=\left(\begin{matrix} 1&1& &0\\ 0&1&&0\\ &&\ddots&\vdots\\ 0&&\cdots&1
\end{matrix}\right).
$$
Note that the action $X\cdot E$ of $E$ on a matrix $X$ is to replace the second column of $X$ by the sum of the first and second column. Since $E$ is invertible, we obtain from  (\ref{S=SU}) that $S\cdot E=S$   and so
$$\vec s+\vec s=\vec s.$$
Hence $\vec s= \vec 0$. Consequently $S=O_n$.
\vspace{0.3cm}

{\bf Remark} We may also consider the case  $n=1$. Note that the smallest field is given by $\mathbb F_2:=\{0,1\}$, with $1\not=0$, where $0$ is the neutral element for addition and $1$ the one for multiplication. This necessarily has characteristic 2. 
Here $S=1$.  If the finite field is not field-isomorphic to $\mathbb F_2$, it has more than two elements, and so   there is an (invertible) element $u$  different from $1$.  Now by (\ref{S=SU}), $S=Su$, hence $S(1-u)=0$.
Since $1-u\not=0$, hence invertible,we conclude that    $S=0$.

\newpage

 \begin{figure}[h!]
 
  \scalebox{0.6} 
  {\includegraphics{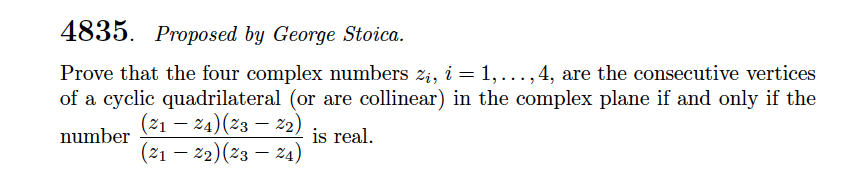}} 

\end{figure}
 
 
 \nopagecolor
 
\centerline {\bf Solution to problem 4835 Crux Math. 49 (4) 2023, 213} \bigskip
 
  \centerline {Raymond Mortini, Rudolf Rupp } \medskip

\centerline{- - - - - - - - - - - - - - - - - - - - - - - - - - - - - - - - - - - - - - - - - - - - - - - - - - - - - -}
  
  \medskip

This is a standard result/exercise in old monographs on function theory/complex analysis and is for instance in \cite[p. 70]{kn}

Using a not so sophisticated wording, we will   show that  four distinct points  $z_j$ ($j=1,\dots,4$) in the plane belong to a circle or a line if and only if their cross-ratio (bi-rapport, Doppelverh\"altnis)
$$DV(z_1,z_2,z_3,z_4):= \frac{z_1-z_2}{z_1-z_4}\Big/ \frac{z_3-z_2}{z_3-z_4}$$
is a real number. \\

In particular, being real, will be independent of the "order" of the points on the circle, respectively line. 

Our  proof will be done in the extended complex plane, $\widehat{\C}:=\C\union\{\infty\}$ (also called the one-point compactification of $\C$). Let us recall some terminology here. If $L$ is a line in $\C$, then  $L\union\{\infty\}$ is called an extended line.
As usual we call the elements of the  set of circles and  extended lines in $\widehat{\C}$  "generalized circles".

We also use an  extension of  the definition of the cross-ratio to points in $\widehat{\C}$. This is done by taking limits.  For instance
\begin{equation}\label{dvinf}
D(z_1,z_2,z_3,\infty)= \frac{z_1-z_2}{z_3-z_2}.
\end{equation}
 Finally, let us recall the following results:

i) There is a unique  linear-fractional map (or in modern terminology, a M\"obius transform)  $T(z):=(az+b)/(cz+d)$, $ad-bc\not=0$, 
 viewed as map  from $\widehat{\C}$ to $\widehat{\C}$  
mapping three distinct points $z_2,z_3,z_4$ in $\widehat{\C}$ to $0,1,\infty$, namely
$T(z)=DV(z,z_2,z_3,z_4)$.
 
 ii) The cross ratio  is invariant under linear-fractional maps: 
 $$DV(T(z_1), T(z_2), T(z_3), T(z_4))=DV(z_1,z_2,z_3,z_4).$$

 Note that the latter is an immediate consequence of i). \\
 
 iii) The class of generalized circles is invariant under M\"obius transforms.\\
 
 
 Now we are ready to confirm the statement above:\\
 
 Given  four distinct points $z_1,z_2,z_3,z_4\in \C$, consider the map $S(z):=DV(z,z_2,z_3,z_4)$. Suppose that these $z_j$ belong to a generalized circle $E$.  Now $S$ maps $E$ to the extended real line $\R\union \{\infty\}$, since $z_2\to 0$, $z_3\to 1$ and $z_4\to\infty$. In particular,  
 $w_j:=S(z_j)\in \R\union \{\infty\}$ for $j=1,\dots,4$.
Since $DV(w_1,w_2,w_3,w_4)$ is real, the invariance result shows that $DV(z_1,z_2,z_3,z_4)$ is real.

Conversely, suppose that  $DV(z_1,z_2,z_3,z_4)$ is real.  Note that $S(z_j)\in \{0,1,\infty\}\ss \R\union \{\infty\}$ for $j=2,3,4$. 
Now the  image of the extended real line by the inverse M\"obius transform $S^{-1}$ is a generalized  circle, $E$. Of course $E$ contains the points $z_2,z_3$ and $z_4$.     But, by (\ref{dvinf}), and the assumption,  we have
$$S(z_1)=DV(S(z_1), S(z_2),S(z_3),S(z_4))=DV(z_1, z_2,z_3,z_4)\in \R.$$
Hence $z_1=S^{-1}(S(z_1))\in E$. In other words, all the $z_j$ belong either to a circle or a line.
\vspace{0.5cm}

\hrule
\vspace{0.5cm}

This can be shortened, without the explicit use of the cross ratio. Actually, just iii) is relevant here:
Consider the M\"obius transform
$$M(z):=\frac{z-z_4}{z-z_2}\; \frac{z_3-z_2}{z_3-z_4}.$$
Then $z_4,z_3,z_2$ are mapped to $0,1,\infty$, and so the (unique) generalized circle $E$ determined by $z_4,z_3,z_2$ is mapped to the extended real line. Thus the point $z_1$ belongs to $E$ if and only if $M(z_1)\in \R$.  In other words, all the $z_j$ belong either to a circle or a line if and only if $\frac{z_1-z_4}{z_1-z_2}\; \frac{z_3-z_2}{z_3-z_4}\in \R$.

\newpage

  \begin{figure}[h!]
 
  \scalebox{0.6} 
  {\includegraphics{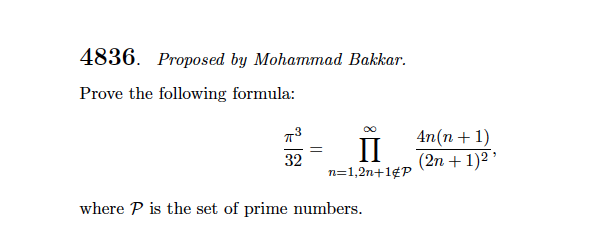}} 

\end{figure}
 
 
 \nopagecolor
 
\centerline {\bf Solution to problem 4836 Crux Math. 49 (4) 2023, 214} \bigskip
 
  \centerline {Raymond Mortini, Rudolf Rupp } \medskip

\centerline{- - - - - - - - - - - - - - - - - - - - - - - - - - - - - - - - - - - - - - - - - - - - - - - - - - - - - -}
  
  \medskip

We first calculate the missing part

$$P:=\prod_{n=1\atop 2n+1 \in \mathcal P}^\infty \frac{4n(n+1)}{(2n+1)^2}.$$
Put $p:=2n+1$. Then $n=(p-1)/2$ and so, in view of the Euler formula 
$$\frac{\pi^2}{6}=\sum_{n=1}^\infty \frac{1}{n^2}=\prod_{p\in\mathcal  P} \frac{1}{1-p^{-2}},$$
we have
$$P=\prod_{p\in \mathcal P\atop p\not=2}  \frac{p^2-1}{p^2}=\frac{4}{3}\; \frac{6}{\pi^2}=\frac{8}{\pi^2}.
$$
To calculate
$$R:=\prod_{n=1}^\infty \frac{4n(n+1)}{(2n+1)^2},$$
we use partial products and Stirling's formula
$\dis \lim_{n\to\infty}   \frac{n^n e^{-n} \sqrt{2\pi n}}{n!}=1.$

\begin{eqnarray*}
P_N:= \prod_{n=1}^N \frac{4n(n+1)}{(2n+1)^2}&=& \frac{4^N\; N! (N+1)!}{\left(\frac{(2N+1)!}{ \prod_{n=1}^N (2n) }\right)^2}
=\frac{4^N N! (N+1)!}{(2N+1)!^2}\; \frac{\left(2^N N!\right)^2}{1}\\
&=& \frac{4^{2N} N!^4 (N+1)}{(2N+1)!^2}\\
&\sim&\frac{4^{2N} \;N^{4N}e^{-4N}4\pi^2 N^2 (N+1)}{(2N+1)^{4N+2} e^{-4N-2}\;2\pi (2N+1)}\\
&=&\pi e^{2}\; \frac{(2N)^{4N}N^2(N+1)}{(2N+1)^{4N} (2N+1)^3}\\
&=& 2\pi e^{2}\; \frac{1}{\left[\left(1+\frac{1}{2N}\right)^{2N}\right]^2}\frac{N^2(N+1)}{(2N+1)^3}\\
&\to& 2\pi e^{2} \; \frac{1}{e^2}\;\frac{1}{8}=\frac{\pi}{4}.
\end{eqnarray*}

Hence $\dis \prod_{n=1\atop 2n+1\notin\mathcal P}^\infty \frac{4n(n+1)}{(2n+1)^2}= \frac{\pi/4}{8/\pi^2}=\frac{\pi^3}{32}$.
\vspace{0.5cm}

\hrule

\vspace{0.5cm}

A second way to derive the value of $P$ is as follows:

For $z\in \mathbb C$ we have
\begin{eqnarray*}
\sin(\pi z)&=&\pi z \prod_{n=1}^\infty \left(1-\frac{z^2}{n^2}\right)=\pi z\prod_{n=1}^\infty  \left(1-\frac{z}{n}\right) e^{\frac{z}{n}} \;\prod_{n=1}^\infty  \left(1+\frac{z}{n}\right) e^{-\frac{z}{n}} \\
&=& \pi z(1-z) e^z \prod_{n=1}^\infty  \left(1-\frac{z}{n+1}\right) e^{\frac{z}{n+1}}\;\prod_{n=1}^\infty  \left(1+\frac{z}{n}\right) e^{-\frac{z}{n}}\\
&=&\pi z(1-z)e^z\prod_{n=1}^\infty  \left(1-\frac{z}{n+1}\right) \left(1+\frac{z}{n}\right) e^{ \frac{z}{n+1}-\frac{z}{n}}\\
&=&\pi z(1-z) e^z\; e^{\sum_{n=1}^\infty \left(\frac{z}{n+1}-\frac{z}{n}\right)} \prod_{n=1}^\infty   \left(1-\frac{z}{n+1}\right)
 \left(1+\frac{z}{n}\right) \\
 &=& \pi z(1-z)  \prod_{n=1}^\infty   \left(1-\frac{z}{n+1}\right)
 \left(1+\frac{z}{n}\right) .
\end{eqnarray*}

Hence
\begin{eqnarray*}
P&=&\prod_{n=1}^\infty \frac{4n(n+1)}{(2n+1)^2}= \prod_{n=1}^\infty\frac{2n}{2n+1}   \frac{2n+2}{2n+1}\\
&=&\prod_{n=1}^\infty  \frac{1}{1+\frac{1}{2n}}\; \frac{1}{1-\frac{1}{2(n+1)}}
=\frac{1}{\dis\prod_{n=1}^\infty \left(1+\frac{1}{2n}\right)\left(1-\frac{1}{2(n+1)}\right)}
\\
&=&\frac{\pi z(1-z)}{\sin(\pi z)}\Big|_{z=1/2}\\
& =&\frac{\pi}{4}.
\end{eqnarray*}

\newpage
   \begin{figure}[h!]
 
  \scalebox{0.5} 
  {\includegraphics{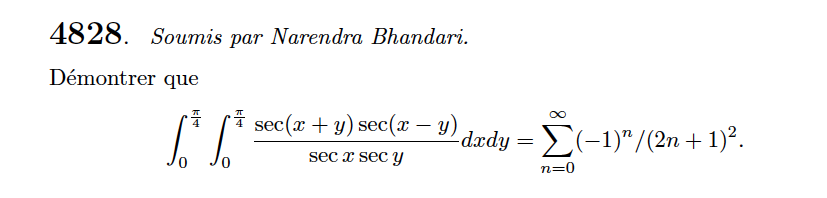}} 

\end{figure}
 

\centerline {\bf Solution to problem 4828 Crux Math. 49 (3) 2023, 157} \bigskip
 
  \centerline {Raymond Mortini, Rudolf Rupp } \medskip

\centerline{- - - - - - - - - - - - - - - - - - - - - - - - - - - - - - - - - - - - - - - - - - - - - - - - - - - - - -}
  
  \medskip

Let 
$$I:=\int_0^{\pi/4}\underbrace{\int_0^{\pi/4} \frac{\cos x\; \cos y}{ \cos(x+y)\;\cos(x-y)}\;dy}_{:=I(x)}dx.$$
Now fix the variable $x$. Since 
$$\cos(x+y)\;\cos(x-y)=\cos^2y-\sin^2 x,$$
we obtain
\begin{eqnarray*}
I(x)&=&\cos x \int_0^{\pi/4}\frac{\cos y}{(1-\sin^2 x)-\sin^2 y}\;dy\\
\tiny{\text{$u:=\sin y$}}&=& \cos x \int_0^{\sqrt 2/2} \frac{du}{\cos^2 x-u^2}\\
&=& \frac{1}{2} \left(\log(\cos x+u)-\log(\cos x -u)\right)\Big|_{u=0}^{\sqrt 2/2}\\
&=&\frac{1}{2} \log\left(\frac{\cos x+1/\sqrt 2}{\cos x-1/\sqrt 2}
\right).
\end{eqnarray*}
Hence (using Fubini),
\begin{equation}\label{wert}
I=\frac{1}{2}\int_0^{\pi/4} \log\left(\frac{\sqrt 2\cos x+1}{\sqrt 2 \cos x-1}     \right)\;dx.
\end{equation}

The value of this integral is known to be  the Catalan number $C$ (see  formula (18) in  \cite{sms}). 
An independent proof is below: using that $\cos a+\cos b= 2 \cos(\frac{a+b}{2}) \cos(\frac{a-b}{2})$ and
$\cos(a-b)=-2 \sin\frac{a+b}{2} \sin \frac{a-b}{2}$,  we obtain

\begin{eqnarray*}
\log\left(\frac{\sqrt 2\cos x+1}{\sqrt 2 \cos x-1}\right) &=&\log\left(\frac{\cos x+\cos\pi/4}{\cos x-\cos\pi/4}\right) \\
 &=&-\log \tan\left( \frac{x+\pi/4}{2}\right)-\log \tan\left( \frac{-x+\pi/4}{2}\right).
\end{eqnarray*}

A change of the variable $x+\pi/4=2y$, respectively $-x+\pi/4=2y$, and a standard integral representation of $C$ yields
$$I=-\frac{1}{2}\int_{\pi/8}^{\pi/4} \log \tan y\; (2dy)-\frac{1}{2}\int_0^{\pi/8} \log\tan y\;(2dy)=-\int_0^{\pi/4} \log\tan y \;dy=C.$$

A proof of this standard representation  can be given for instance by using power series or Fourier series:

$$h(z):=\frac{1}{2} \log\left(\frac{1+z}{1-z}\right)=\sum_{n=0}^\infty \frac{1}{2n+1} z^{2n+1}.$$
Its Taylor coefficients belong to $\ell^2$ and so the associated Fourier series 
$$h^*(e^{it}):=\sum_{n=0}^\infty  \frac{1}{2n+1} e^{i(2n+1)t}$$ 
converges in the $L^2(]0,\pi[)$-norm to 
$$h(e^{it})=\frac{1}{2}\log(i \cot(t/2))=i\,\frac{\pi}{4} -\frac{1}{2}\log \tan(t/2)$$
(Actually the series $h^*(e^{it})$ converges pointwise for $z=e^{it}$ with $0<t<\pi$  by the Abel-Dirichlet rule, but we do not need this.)

Taking real parts, and using that  $\int\sum=\sum\int$ (note that Fourier series converge in the $L^2$-norm, hence in the $L^1$ norm), we may conclude that
$$-\int_0^{\pi/4} \log\tan y \;dy=\int_0^{\pi/2} \sum_{n=0}^\infty \frac{1}{2n+1} \cos(2n+1)t\;dt=\sum_{n=0}^\infty (-1)^n \frac{1}{(2n+1)^2}.$$

 \newpage
   \begin{figure}[h!]
 
  \scalebox{0.5} 
  {\includegraphics{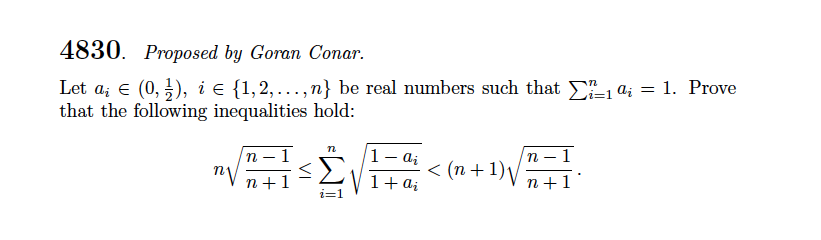}} 

\end{figure}
 
 
 \nopagecolor
 
\centerline {\bf Solution to problem 4830 Crux Math. 49 (3) 2023, 158} \bigskip
 
  \centerline {Raymond Mortini, Rudolf Rupp } \medskip

\centerline{- - - - - - - - - - - - - - - - - - - - - - - - - - - - - - - - - - - - - - - - - - - - - - - - - - - - - -}
  
  \medskip

First we claim that on  $[0,1/2]$ the function $\dis f(x)=\sqrt{\frac{1-x}{1+x}}$ is convex. In fact,
$$f'(x)=-\frac{1}{\sqrt{\frac{1-x}{1+x}}{(1+x)}^2}
$$
and 
$$f''(x)= \frac{1-2x}{(1-x)(x+1)^3\sqrt{\frac{1-x}{1+x}}}\geq 0.
$$
 Since the graph of  a convex function lies below the secant determined by $(a, f(a)), (b,f(b))$, we obtain that 
 $f(x)\leq 1-2(1-3^{-1/2})x$,  where $a=0$ and $b=1/2$.  Since $1-3^{-1/2}\geq 1/3$, we deduce that for $0\leq x\leq 1/2$
 $$f(x)\leq 1-(2/3)x$$,
  and so
 $$\sum_{i=1}^n f(a_i)\leq n-(2/3)\sum_{i=1}^n a_i= n-2/3.$$
 But  for $n\geq 2$, we have
 $$n-2/3<(n+1) \sqrt{\frac{n-1}{n+1}}= \sqrt{n^2-1},$$
 since
 $$n^2-1 -(n-2/3)^2= 4/3 n -13/9\geq 8/3 -13/9= 11/9>0. $$
 This upper bound in the problem appears to be artificial.  We did not see a way to derive this in a natural way.
To prove the reverse inequality, we use Jensen's inequality and obtain
 $$\frac{1}{n}\sum_{i=1}^n f(a_i)\geq f\left(\frac{\sum_{i=1}^n a_i }{n}\right)= f(1/n).$$
 Hence
$$ \sum_{i=1}^n \sqrt{\frac{1-a_i}{1+a_i}}\geq n\; \sqrt{\frac{1-\frac{1}{n}}{1+\frac{1}{n}}}= n\; \sqrt{ \frac{n-1}{n+1}} .$$

 \newpage
 

\newpage
 
   \begin{figure}[h!]
 
  \scalebox{0.5} 
  {\includegraphics{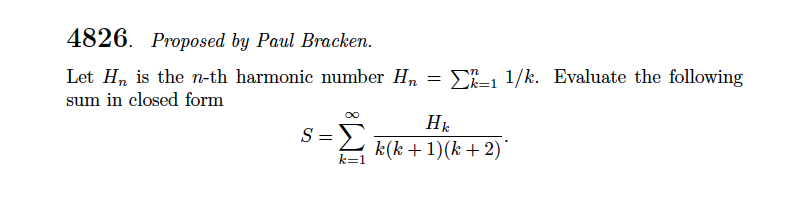}} 

\end{figure}
 

\centerline {\bf Solution to problem 4826 Crux Math. 49 (3) 2023, 157}
 
 \bigskip

 \centerline {Raymond Mortini, Rudolf Rupp } \medskip

\centerline{- - - - - - - - - - - - - - - - - - - - - - - - - - - - - - - - - - - - - - - - - - - - - - - - - - - - - -}
  
  \medskip

 We claim that 
$$\ovalbox{$\dis S=\frac{\pi^2}{12}-\frac{1}{2}$}.$$  
Just write
\begin{eqnarray*}
\frac{H_k}{k(k+1)(k+2)}&=& \frac{1}{2} \left(H_k\Big(\frac{1}{k(k+1)}-\frac{1}{(k+1)(k+2)}\Big)\right)\\
&=&\frac{1}{2} \left( \frac{H_k}{k(k+1)}- \frac{H_{k+1}}{(k+1)(k+2)} + \frac{1}{(k+1)^2 (k+2)}\right).
\end{eqnarray*}
Now 
$$ \frac{1}{(k+1)^2 (k+2)}= \frac{1}{k+2}-\frac{k\co{+1-1}}{(k+1)^2} =\left(\frac{1}{k+2}-\frac{1}{k+1}\right) +\frac{1}{(k+1)^2}.
$$
Since the Cesaro means of the sequences $(1/k)$  converge to 0, that is $H_k/k\to 0$, we conclude that
\begin{eqnarray*}
S&=&\frac{1}{2}\frac{H_1}{2} -\frac{1}{2} \frac{1}{1+1}+\frac{1}{2}\sum_{k=1}^\infty \frac{1}{(k+1)^2}=\frac{\pi^2}{12}-\frac{1}{2}.
\end{eqnarray*}

\newpage
 
   \begin{figure}[h!]
 
  \scalebox{0.5} 
  {\includegraphics{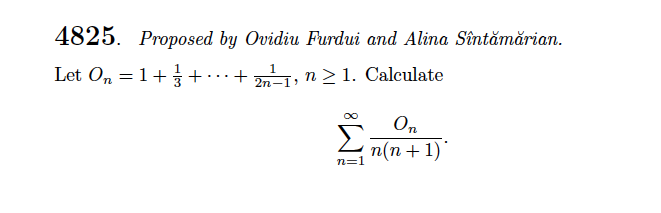}} 

\end{figure}
 

\centerline {\bf Solution to problem 4825 Crux Math. 49 (3) 2023, 157}

 \bigskip

 \centerline {Raymond Mortini, Rudolf Rupp } \medskip

\centerline{- - - - - - - - - - - - - - - - - - - - - - - - - - - - - - - - - - - - - - - - - - - - - - - - - - - - - -}
  
  \medskip

We prove that 
$$\ovalbox{$\dis I:=\sum_{n=1}^\infty \frac{O_n}{n(n+1)}=\log 4$.}$$
First we note that
$$ \frac{O_n}{n(n+1)}=O_n \left(\frac{1}{n}-\frac{1}{n+1}\right)= \frac{O_n}{n}-\frac{O_{n+1}}{n+1} +\frac{1}{(2n+1)(n+1)}.$$
Since the Cesaro means of the null sequence $(1/(2n+1))$ converge to $0$, we obtain
$$I= \frac{O_1}{1} +\sum_{n=1}^\infty \frac{1}{(2n+1)(n+1)}=1 + 2\log 2-1= \log 4.$$
The value of the series 
$\dis S:=\sum_{n=1}^\infty \frac{1}{(2n+1)(n+1)}$ can be determined as follows:

\begin{eqnarray*}
\sum_{n=1}^N \frac{1}{(2n+1)(n+1)}&=&\sum_{n=1}^N \Big(\frac{2}{2n+1}-\frac{1}{n+1}\Big)\\
\text{\footnotesize  splitting into even an odd} &=& \sum_{n=1}^N \Big(\frac{1}{2n+1}-\frac{1}{2n+1}\Big) + \sum_{n=1}^{N}\Big(\frac{1}{2n+1}-\frac{1}{2n}\Big) +\sum_{n=N+1}^{2N+1} \frac{1}{n}\\
&=&-1+ \sum_{n=1}^{2N+1} (-1)^{n+1} \frac{1}{n}   +\sum_{n=N+1}^{2N+1} \frac{1}{n}\\
&\rel\longrightarrow_{N\to\infty}^{}&-1+\log 2+\log 2.
\end{eqnarray*}

Note that the well-known assertion $\lim_{N\to\infty} \sum_{n=N+1}^{2N} \frac{1}{n}=\log 2$ is a direct consequence of the fact that the Euler-Mascheroni constant $\gamma$ is given by 
$$\gamma=\lim(H_n-\log n),$$
where $H_n:=\sum_{i=1}^n \frac{1}{i}$, since 
$$H_{2N}-H_N=(H_{2N}-\log  (2N)- \gamma)+(\log N +\gamma -H_N) +\log 2\to \log 2.$$

\newpage
 
   \begin{figure}[h!]
 
  \scalebox{0.5} 
  {\includegraphics{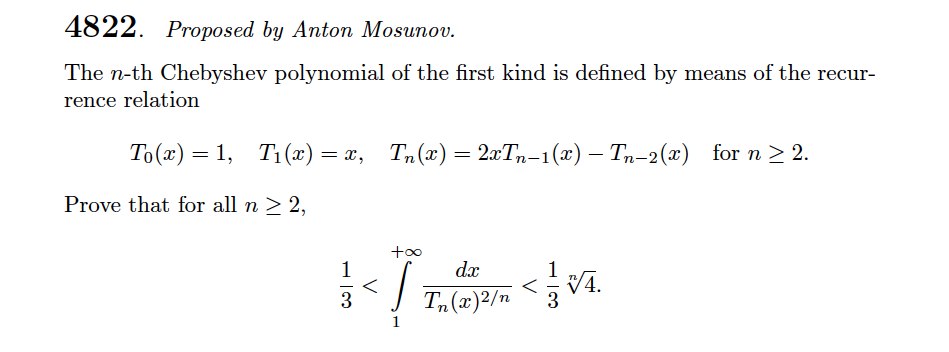}} 

\end{figure}
 

\centerline {\bf Solution to problem 4822 Crux Math. 49 (3) 2023, 156}

 \bigskip

 \centerline {Raymond Mortini, Rudolf Rupp } \medskip

\centerline{- - - - - - - - - - - - - - - - - - - - - - - - - - - - - - - - - - - - - - - - - - - - - - - - - - - - - -}
  
  \medskip

Substituting  $x=\cosh t $ we obtain $T_n(\cosh t)=\cosh(nt)$. In particular, $T_n$ has no zeros on $ [1,\infty[$. Hence

\begin{eqnarray*}
I:=\int_1^\infty \frac{dx}{T_n(x)^{2/n}}&=&\int_0^\infty \frac{\sinh t}{\big(\cosh (nt)\big)^{2/n}}\; dt= \int_0^\infty 
\frac{e^t-e^{-t}}{2 \dis \left(\frac{e^{nt}+e^{-nt}}{2}\right)^{2/n}}\;dt\\
&=&2^{-1+2/n}\int_0^\infty \frac{1-e^{-2t}}{e^t\big(1+e^{-2nt}\big)^{2/n}}\;dt.
\end{eqnarray*}

Hence
\begin{eqnarray*}
I&<&2^{-1+2/n}\int_0^\infty \frac{1-e^{-2t}}{e^t}\;dt =2^{-1+2/n}\left[ -e^{-t}+\frac{1}{3}e^{-3t}\right]^\infty_0\\
&=& 2^{-1+2/n} \frac{2}{3}= \frac{1}{3} \sqrt[n]{4}.
\end{eqnarray*}
Moreover
\begin{eqnarray*}
I&>& 2^{-1+2/n} \int_0^\infty \frac{1-e^{-2t}}{e^t(1+1)^{2/n}}\;dt= 2^{-1} \int_0^\infty \frac{1-e^{-2t}}{e^t}\;dt= 2^{-1} \frac{2}{3}=\frac{1}{3}.
\end{eqnarray*}

\newpage
 
   \begin{figure}[h!]
 
  \scalebox{0.5} 
  {\includegraphics{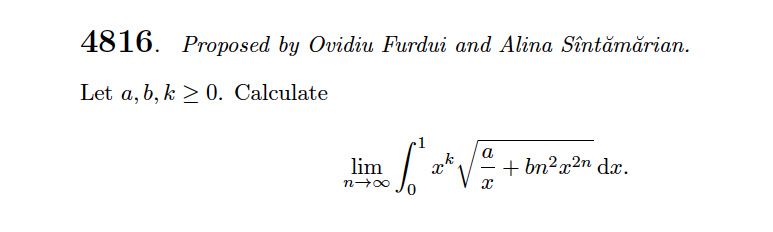}} 

\end{figure}
 
 \nopagecolor

\centerline {\bf Solution to problem 4816 Crux Math. 49 (2) 2023, 101}

 \bigskip

 \centerline {Raymond Mortini, Rudolf Rupp } \medskip

\centerline{- - - - - - - - - - - - - - - - - - - - - - - - - - - - - - - - - - - - - - - - - - - - - - - - - - - - - -}
  
  \medskip

 We show that for $a,b, k\geq 0$ ($k$ not necessary an integer)
 
$$\ovalbox{$\dis I_n:=\int_0^1 x^k\sqrt{\frac{a}{x}+ bn^2 x^{2n}}\;dx\;\rel\longrightarrow_{}^{n\to\infty}\; \sqrt b+ \frac{\sqrt a}{k+1/2}.$}$$
\bigskip

Write 
$$f_n(x)= x^{k-1/2} \sqrt{a+bn^2x^{2n+1}}.$$
If $a=0$, then
$$I_n=\int_0^1  \sqrt bn x^{n+k}dx= \frac{ n\;\sqrt b}{n+k+1}\to \sqrt b.
$$
For $a>0$, let 
 $$d_n(x):=x^{k-1/2} \Big(\sqrt{a+ bn^2 x^{2n+1}}-\sqrt{bn^2x^{2n+1}}\Big).$$
  Then
  $$0\leq d_n(x)= x^{k-1/2}\;\frac{a}{\sqrt{a+ bn^2 x^{2n+1}}+\sqrt b n x^{n+1/2}}\leq \frac{a}{\sqrt a}x^{k-1/2}.
  $$
  Hence $d_n$ is dominated by an $L^1[0,1]$ function and so, by using that $nx^n \to 0$ for $0<x<1$,
  $$\lim_n \int_0^1 d_n(x) dx= \int_0^1 \lim_n d_n(x)dx=\int_0^1\sqrt a x^{k-1/2}= \frac{\sqrt a}{k+1/2}.
  $$
  Consequently, 
\begin{eqnarray*}
\int_0^1 f_n(x)dx&= &\int_0^1 d_n(x)dx + \sqrt b \;\int_0^1 n x^{k-1/2}x^{n+1/2}dx\\
&=& \int_0^1 d_n(x)dx+ \sqrt b\,\frac{n}{k+n+1}\\
&\rel\longrightarrow_{n\to\infty}^{}& \frac{\sqrt a}{k+1/2}+\sqrt b.
\end{eqnarray*}

 \newpage
 
   \begin{figure}[h!]
 
  \scalebox{0.5} 
  {\includegraphics{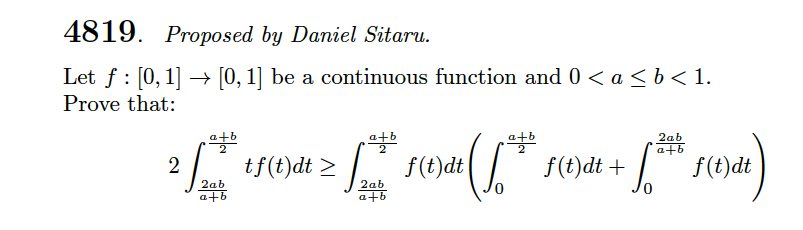}} 

\end{figure}
 

\centerline {\bf Solution to problem 4819 Crux Math. 49 (2) 2023, 102}
  
 \bigskip

 \centerline {Raymond Mortini, Rudolf Rupp } \medskip

\centerline{- - - - - - - - - - - - - - - - - - - - - - - - - - - - - - - - - - - - - - - - - - - - - - - - - - - - - -}
  
  \medskip

 Note that the harmonic mean $x_0:=2ab/(a+b)$ is less than or equal to the arithmetic mean $y_0:=(a+b)/2$. We show that the inequality holds for arbitrary $x_0, y_0$ with $0<x_0<y_0<1$.   So let $F$ be  that primitive of $f$ on $[0,1]$ with $F(0)=0$. We shall prove that
$$\ovalbox{$\displaystyle2 \int_{x_0}^{y_0} t f(t)\, dt\geq (F(y_0)-F(x_0))( F(y_0)+F(x_0))$},
$$
from which the desired inequality immediately follows. 
By partial integration,
\begin{equation}
2\,\int_{x_0}^{y_0}t\, f(t) \,dt= 2\,\int_{x_0}^{y_0} t F'(t)\,dt= 2 (y_0F(y_0)-x_0F(x_0))-2\int_{x_0}^{y_0} F(t)dt. 
\end{equation}

For $0\leq x,y\leq 1$, put
$$H(x,y):= 2 yF(y)-2xF(x)-2 \int_x^y F(t)dt -(F(y)^2-F(x)^2).$$
We have to show that $H(x_0,y_0)\geq 0$.
Since $0\leq f\leq 1$,  $F(x)\leq \int_0^x 1 \;dt = x$. Hence
\begin{eqnarray*}
\frac{\partial H}{\partial x}(x,y)=-2\big(F(x)+x f(x)\big) +2 F(x)+2 F(x)f(x)= 2 \big(F(x)-x\big) f(x)\leq 0.
\end{eqnarray*}
Consequently, by using that $H(y,y)=0$, we obtain $\xi\in ]x_0,y_0[$ with 
$$H(x_0,y_0)=H(x_0,y_0)-H(y_0,y_0)= \underbrace{\frac{\partial H}{\partial x} (\xi,y_0)}_{\leq 0}\underbrace{(x_0-y_0)}_{\leq 0} \geq 0.
$$

\newpage

   \begin{figure}[h!]
 
  \scalebox{0.5} 
  {\includegraphics{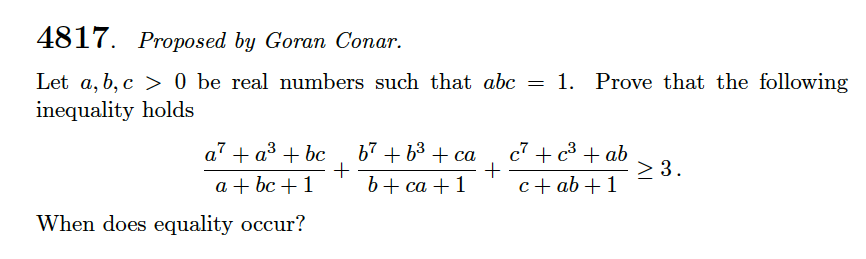}} 

\end{figure}

\centerline{\bf Solution to problem 4817 Crux Math. 49 (2) 2023, 102}

  \bigskip

 \centerline {Raymond Mortini, Rudolf Rupp } \medskip

\centerline{- - - - - - - - - - - - - - - - - - - - - - - - - - - - - - - - - - - - - - - - - - - - - - - - - - - - - -}
  
  \medskip

Let $E:=]0,\infty[ \,\times\, ]0,\infty[\,\times \,]0,\infty[$ and 
let $H: E\to \;]0,\infty[$ be given by

$$
H(a,b,c)=\frac{a^7+a^3+bc}{a+bc+1} +\frac{b^7+b^3+ca}{b+ca+1} +\frac{c^7+c^3+ab}{c+ab+1}.
$$
\medskip

Put $L:=\{(a,b,c)\in E: abc=1\}$.
To be shown is that $\inf_L H =3$ and that this lower bound is obtained exactly at $(1,1,1)$. To this end, consider for $x>0$ the function

$$
f(x):= \frac{x^7+x^3+x^{-1}}{x+x^{-1}+1}= \frac{x^8+x^4+1}{x^2+x+1}= x^6-x^5+x^3-x+1.
$$

Then $f$ is convex on $[0,\infty[$. In fact, 
$$\mbox{$f'(x)=6x^5-5x^4+3x^2-1$ and $f''(x)=30x^4-20x^3+6x=2x(15x^3-10x^2+3)$.}$$
Now $f''(x)= 2x\big(5x^2(3x-2)+3\big)$. Then, clearly, $f''(x)\geq 0$ if $x\geq 2/3$. Since  
$$\max_{[0,2/3]} x^2(2-3x)=32/3^5\leq 3/5,$$
 we deduce that $f''(x)\geq 0$ on $[0,2/3]$, too. 
Due to Jensen's inequality, for $(a,b,c)\in L$
\begin{eqnarray*}
H(a,b,c)&=& f(a) +f(b)+f(c) = 3 \;\frac{f(a)+f(b)+f(c)}{3}\geq 3\;f\Big(\frac{a+b+c}{3}\Big)
\end{eqnarray*} 

Since $f$ is convex for $x\geq 0$, $f(x)\geq f(1)+f'(1)(x-1)=1+3(x-1)=-2+3x$. Why we take evaluation at $1$? Because it works!  It is an a posteriori choice, since the minimal value is taken at $(a,b,c)=(1,1,1)$. 
Thus we obtain the estimate
$$H(a,b,c)\geq 3 \left( -2+ (a+b+c)\right)=-6+3\left(a+b+\frac{1}{ab}\right).
$$

 {We can even avoid Jensen's inequality:
$$H(a,b,c)=f(a)+f(b)+f(c)\geq (-2+3a)+(-2+3b)+(-2+3c)=-6 + 3(a+b+c).$$

Since $a+b+\frac{1}{ab}\geq 3$ (see below) we deduce that for $abc=1$ we have $H(a,b,c)\geq -6+9=3$. As $H(1,1,1)=3$, we are done.

The inequality $\dis g(a,b):=a+b+\frac{1}{ab} \geq 3$ is well known. It can for instance be shown by using differential calculus:
$$\mbox{$\dis g_a(a,b)=1-\frac{1}{ab^2}=0\iff ab^2=1$ and $\dis g_b(a,b)=1-\frac{1}{ba^2}=0\iff ba^2=1$}.$$
In other words, $ab(a-b)=0$. Hence $a=b=1$ is the only stationary point. Thus  $g(1,1)=3$ is the minimum, since the limit of $g$ at the boundary $ab=0$ is $\infty$.

\newpage
   \begin{figure}[h!]
 
  \scalebox{0.5} 
  {\includegraphics{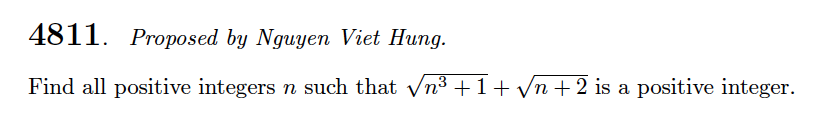}} 

\end{figure}

\centerline {\bf Solution to problem 4811 Crux Math. 49 (2) 2023, 101}
  
   \bigskip

 \centerline {Raymond Mortini, Rudolf Rupp } \medskip

\centerline{- - - - - - - - - - - - - - - - - - - - - - - - - - - - - - - - - - - - - - - - - - - - - - - - - - - - - -}
  
  \medskip

We show that $n=2$ is the only solution. In fact $\sqrt{2^3+1}+\sqrt{2+2}=3+2=5$. Now, for $x,y\geq 0$, one has
$\sqrt x+\sqrt y\in \N$ if and only if $x$ and $y$ are perfect squares.  To see this, just note that
$$\sqrt x+\sqrt y= \frac{x-y}{\sqrt x-\sqrt y}$$
implies that $\sqrt x+\sqrt y\in \Q$ if and only $\sqrt x-\sqrt y\in \Q$ and so, by adding (respectively substracting), $\sqrt x$ and
 $\sqrt y$ are rational.  Thus  $\sqrt x=p/q$ for some $p,q\in \N$ with no common divisor. Hence $x^2=p^2/q^2\in \N$, and so  $q=1$.
 
 Due to a classical result by L.  Euler, the Diophantine equation $n^3+1=m^2$ has in $\N=\{0,1,2,\dots\}$  only the solutions $(m,n)=(1,0)$ and $(m,n)=(3,2)$ (see for instance \cite{mathover}, a reference provided to the first author by Amol Sasane). Thus $n=2$ is the only positive integer also satisfying $\sqrt{n+2}\in \N$.

  \newpage

   \begin{figure}[h!]
 
  \scalebox{0.5} 
  {\includegraphics{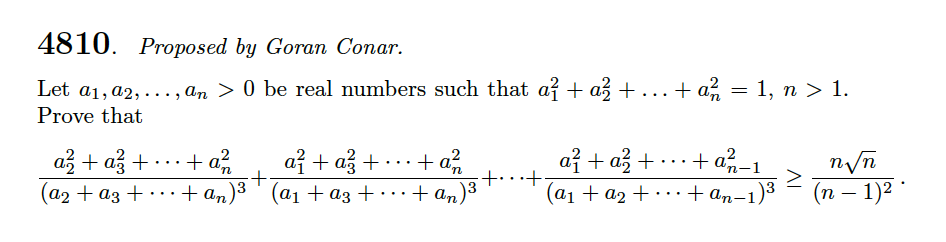}} 

\end{figure}

\centerline{\bf Solution to problem 4810 Crux Math. 49 (1) 2023, 45}
  \bigskip

 \centerline {Raymond Mortini, Rudolf Rupp } \medskip

\centerline{- - - - - - - - - - - - - - - - - - - - - - - - - - - - - - - - - - - - - - - - - - - - - - - - - - - - - -}
  
  \medskip

We first show that whenever $\sum_{j=1}^n a_j^2=1$, then
\begin{equation}\label{zwis}
\frac{\dis\sum^n_{j=1 \atop j\not=i}a_j^2}{\dis\Big(\sum^n_{j=1 \atop j\not=i} a_j\Big)^3} \geq \frac{1}{\sqrt{1-a_i^2}}\; \frac{\sqrt{n-1}}{(n-1)^2}.
\end{equation}

In fact, using Cauchy-Schwarz, we immediately obtain

\begin{eqnarray*}
\frac{\dis\sum^n_{j=1 \atop j\not=i}a_j^2}{\dis\Big(\sum^n_{j=1 \atop j\not=i} a_j\Big)^3}\geq 
\frac{\dis\sum^n_{j=1 \atop j\not=i}a_j^2}{\dis\Big(\big(\sum^n_{j=1 \atop j\not=i}a_j^2\big) (n-1)\Big)^{3/2}}
= \frac{1}{\sqrt{1-a_i^2}}\; \frac{\sqrt{n-1}}{(n-1)^2}.
\end{eqnarray*}

Next we prove that whenever $\sum_{j=1}^n a_j^2=1$, then

\begin{equation}\label{zwis2}
\sum_{i=1}^n \frac{1}{\sqrt{1-a_i^2}} \geq n\;\sqrt{\frac{n}{n-1}}.
\end{equation}

In fact, consider the convex function $f(x)= \frac{1}{\sqrt{1-x}}$.
By Jensen's inequality (or one of the possible defintions of convexity), if $\sum_{j=1}^n t_j=1$  where $(0\leq t_j\leq 1$), then
$$f\Big(\sum_{j=1}^n t_jx_j\Big)\leq \sum_{j=1}^n t_jf(x_j).$$
Here  we choose   $x_i=a_i^2$, and $t_j=1/n$. Note that  $\frac{\sum_{i=1}^n a_i^2}{n}=1/n$.
Hence
$$\sum_{i=1}^n \frac{1}{\sqrt{1-a_i^2}}=n\sum_{i=1}^n \frac{1}{n}\,f(x_i)\geq n f\big(\frac{1}{n}\sum_{i=1}^n x_j\big)=\frac{n}{\sqrt{1-\frac{1}{n}}}= n\; \sqrt{\frac{n}{n-1}}.$$

Now putting (\ref{zwis}) and (\ref{zwis2}) together yields

$$
\sum_{i=1}^n\frac{\dis\sum^n_{j=1 \atop j\not=i}a_j^2}{\dis\Big(\sum^n_{j=1 \atop j\not=i} a_j\Big)^3} \geq 
\frac{\sqrt{n-1}}{(n-1)^2}\; n\; \sqrt{\frac{n}{n-1}}=\frac{n \sqrt n}{(n-1)^2}.
$$

\newpage
   
   \begin{figure}[h!]
 
  \scalebox{0.5} 
  {\includegraphics{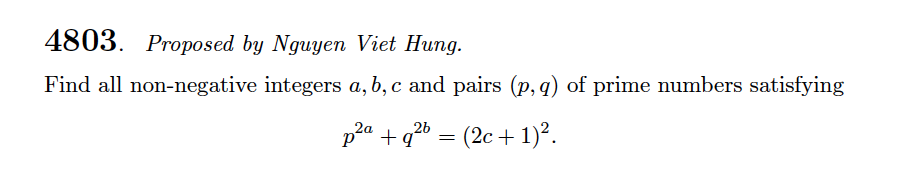}} 

\end{figure}


\centerline{\bf Solution to problem 4803 Crux Math. 49 (1) 2023, 44}
  \bigskip

 \centerline {Raymond Mortini, Rudolf Rupp } \medskip

\centerline{- - - - - - - - - - - - - - - - - - - - - - - - - - - - - - - - - - - - - - - - - - - - - - - - - - - - - -}
  
  \medskip

It turns out that  the triple $(3,4,5)$ satisfying $3^2+4^2=5^2$ is relevant here. Only one solution to the problem with $p\leq q $  exists:
$p=2, q=3$ and $a=2, b= 1, c= 2$. To sum up:

$$\ovalbox{$ 2^{2\cdot 2} + 3^{2\cdot 1} =(2 \cdot 2+1)^2$}$$

To see this, we use of course the well known parametrizations of the solutions to $A^2+B^2=C^2$, which are given by
$$(*)~~~~\mbox{$A=2mn, B=m^2-n^2$ and $C=m^2+n^2, ~~m,n\in \N$}.$$

The conditions to be dealt with are  

$$\mbox{$i)~~~ 2mn=p^a$, $ii)~~~ m^2-n^2=q^b$ and $iii)~~~ m^2+n^2=2c+1$.}$$

$\bullet$ First we note that  $(a,b)=(0,0)$ is not admissible as  $1+1=2$ is even. 
Now if $b=0$ and  $a\not=0$,  then by i) $p$ necessarily must be an even prime, 
that is $p=2$.  Hence 
$$2^{2a}+1=(2c+1)^2.$$
  By (*),  $1=m^2-n^2$ and $2^{2a}=2mn$. Consequently $m$ and $n$ are powers of $2$. Hence $m^2-n^2$ is an even number; and not $1$. Thus $ab\not=0$.

$\bullet$  So let $ab>0$. 
Since $p$ is prime,  $m$ and $n$ can only be powers of $2$ by (i). Due to  iii), telling us that $m^2+n^2$ is an odd number, 
not both $m$ and $n$ can be proper powers of $2$.  Since $m\geq n$ (by ii)), we necessarily have $n=1$ and $m=2^x$ with $x\not=0$.
By ii),  
$$q^b=m^2-1= (2^x)^2-1=(2^x-1)(2^x+1).$$
This implies that $q\not=2$ (as the right hand side is odd). Since the difference of the factors is 2, $q\geq 3$ cannot divide both factors. Thus  we can only have that the factor $2^x-1$  equals 1.

 Hence $x=1$ and $q^b=3$, yielding $b=1$ and $q=3$.  Finally by i), $p^a=2 mn=2 \cdot 2^1\cdot 1=2^2$. So $p=2$ and $a=2$. Finally, $c=2$ as $3^2+4^2=5^2$.

\newpage

   \begin{figure}[h!]
 
  \scalebox{0.5} 
  {\includegraphics{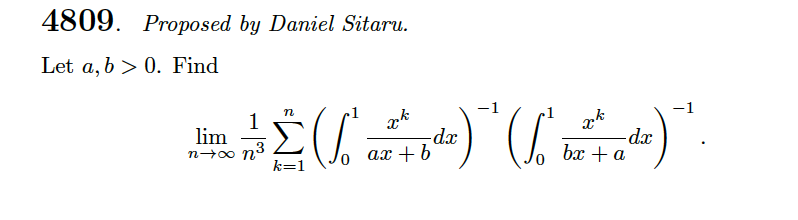}} 

\end{figure}

\centerline{\bf Solution to problem 4809 Crux Math. 49 (1) 2023, 45}\medskip

  \centerline{Raymond Mortini, Rudolf Rupp } \medskip

\centerline{- - - - - - - - - - - - - - - - - - - - - - - - - - - - - - - - - - - - - - - - - - - - - - - - - - - - - -}
  
  \medskip

  We show, more generally, that whenever $f,g:[0,1]\to [0,\infty[$ are continuous and $f(1)g(1)\not=0$, then
$$\ovalbox{$\dis{ \lim_{n\to\infty} \frac{1}{n^3}  \sum_{n=1}^n \left(\int_0^1 x^k f(x)\,dx\right) ^{-1}\;  \left(\int_0^1 x^k g(x)\,dx\right) ^{-1}=
\frac{1}{3}\;\frac{1}{f(1)}\frac{1}{g(1)}}$.
}$$
Hence the limit in the problem is $(a+b)^2/3$.\\

{\bf Proof}
Let $M:=\max\{|f(x)|: 0\leq x\leq 1\}$. Given $\e$ with $0<\e<\frac{1}{2}\min\{f(1),g(1)\}$, choose $\delta>0$ so that $|f(x)-f(1)|\leq \e$ for $\delta\leq x\leq 1$. Moreover, let $n_0$ be so large that $\delta^{k+1}\leq \e/(2M)$ for $k\geq n_0$.  Then

\begin{eqnarray*}
\left|\int_0^1 x^k f(x)\,dx-\frac{1}{k+1} f(1)\right|&=& \left|\int_0^1 x^k (f(x)-f(1))\,dx\right|\\
& \leq& 2M\;\int_0^\delta x^k +\int_\delta^1 x^k |f(x)-f(1)|\,dx  \\
&\leq & 2M \frac{\delta^{k+1}}{k+1} +\e \int_0^1 x^kdx \\
&\leq& \frac{\e}{k+1}+  \frac{\e}{k+1}.
\end{eqnarray*}

Therefore
$$ \left( \frac{1}{k+1} f(1)+\frac{2\e}{k+1}   \right)^{-1}\leq  \left(\int_0^1 x^k f(x)\,dx\right) ^{-1}\leq  \left( \frac{1}{k+1} f(1)-\frac{2\e}{k+1}   \right)^{-1}.$$

We conclude that
\small\begin{eqnarray*}
 \sum_{k=n_0}^n \frac{k+1}{f(1)+2\e}\;\frac{k+1}{g(1)+2\e}
\leq \sum_{k=n_0}^n   \left(\int_0^1 x^k f(x)\,dx\right) ^{-1}\;  \left(\int_0^1 x^k g(x)\,dx\right) ^{-1} &\leq&
 \sum_{k=n_0}^n \frac{k+1}{f(1)-2\e}\;\frac{k+1}{g(1)-2\e}.
\end{eqnarray*} 

Hence, by using that $\sum_{j=1}^n j^2= \frac{n(n+1)(2n+1)}{6}$, we deduce that
$$ \limsup_{n\to\infty} \frac{1}{n^3} \sum_{n=1}^n \left(\int_0^1 x^k f(x)\,dx\right) ^{-1}\;  \left(\int_0^1 x^k g(x)\,dx\right) ^{-1}\leq \frac{1}{3}\;
\frac{1}{g(1)-2\e}\;\frac{1}{f(1)-2\e}
$$
and
$$ \liminf_{n\to\infty} \frac{1}{n^3} \sum_{n=1}^n \left(\int_0^1 x^k f(x)\,dx\right) ^{-1}\;  \left(\int_0^1 x^k g(x)\,dx\right) ^{-1}\geq \frac{1}{3}\;
\frac{1}{f(1)+2\e}\frac{1}{g(1)+2\e}.
$$
from which we conclude that
$$ \lim_{n\to\infty} \frac{1}{n^3}  \sum_{n=1}^n \left(\int_0^1 x^k f(x)\,dx\right) ^{-1}\;  \left(\int_0^1 x^k g(x)\,dx\right) ^{-1}=
\frac{1}{3}\;\frac{1}{f(1)}\frac{1}{g(1)}.
$$
{\bf Remark} The lower estimate show that the limit is infinite if $f(1)g(1)=0$.

\newpage

   \begin{figure}[h!]
 
  \scalebox{0.5} 
  {\includegraphics{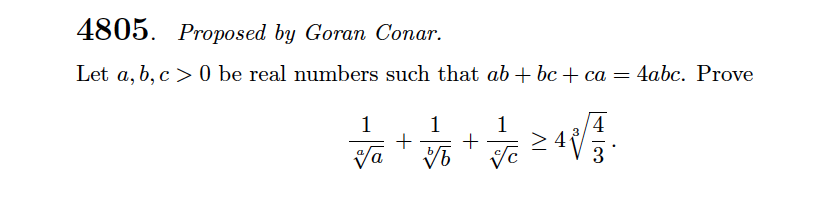}} 

\end{figure}

\centerline{\bf Solution to problem 4805 Crux Math. 49 (1) 2023, 44}\medskip

  \centerline{Raymond Mortini, Rudolf Rupp } \medskip

\centerline{- - - - - - - - - - - - - - - - - - - - - - - - - - - - - - - - - - - - - - - - - - - - - - - - - - - - - -}
  
  \medskip

First we note that $ab+bc+ca=4abc$ is equivalent to 
$$(*)\hspace{3cm}\dis  \ell(a,b,c):=\frac{1}{a}+ \frac{1}{b}+ \frac{1}{c}=4.$$
If $a=b=c$, then this condition is satisfied if $a=3/4$. 
Let 
$$g(a,b,c)= a^{-1/a}+b^{-1/b}+c^{-1/c}.$$
It suffices to show that the minimum $M$ of $g$ under  condition (*) is obtained for $a=b=c$.
Note that $M:=g(3/4,3/4,3/4)= 3 (3/4)^{-4/3}=4 (4/3)^{1/3}\sim 4.402$.

The gradient of the Lagrange function 
$$H(a,b,c,\lambda)= g(a,b,c)+\lambda(\ell(a,b,c)-4)$$
 is zero if 
$$\lambda= a^{-1/a} \;(1-\log a)=b^{-1/b} \;(1-\log b)=c^{-1/c} \;(1-\log c).$$

Since the function $x\mapsto x^{-1/x} \;(1-\log x)$ is strictly decreasing on $]0,\infty[$, the only solution is where $a=b=c$. 
The existence of the minimum is shown as follows (note that 
$$E:=\{(a,b,c): a,b,c>0, \ell(a,b,c)=4\}$$ is not compact.
Condition (*) implies that $a,b,c\geq 1/4$. Let $L:=\inf_E g$. Then 
$$L\geq 3 \min_{[1/4, \infty[} x^{-1/x}=3  e^{-1/e}\geq 3 \times  0.692=2.076.$$
If this infimum is not taken on $E$, then there is $a_n\to \infty$ (or $b_n\to \infty,$ or $c_n\to\infty$)  such that $(a_n,b_n,c_n)\in E$ and
$g(a_n,b_n,c_n)\to L$.  In particular $a_n^{-1/a_n}\to 1$. We may assume that $b_n\to b_0$ and $c_n\to c_0$ (since otherwise $b_n\to\infty$ and so $c_n\to 1/4$, as well as $L=1+1+ 4^4>M$, a contradiction).
 Hence
$L=\inf_{E'} (1+ b^{-1/b}+c^{-1/c})$, where 
$$E'=\{(b,c):  b, c>0, 1/b+1/c=3\}.$$
 In particular, $b\geq 1/3$.
Thus (by using Lagrange again, yielding $x=2/3$)
$$L=\inf_{]1/3,\infty[} 1+ x^{-1/x} + \left(\frac{3x-1}{x}\right)^{\frac{3x-1}{x}}\rel=_{}^{x=2/3} 1+2(3/2)^{3/2}\sim 4.674 >M.$$
A contradiction.
Consequently $(a_n,b_n,c_n)\to (\alpha,\beta,\gamma)\in E$ and so the infimum is a minimum.  Hence 
$$g(a,b,c)\geq g(\alpha,\beta,\gamma)=L=M.$$
\bigskip

Here is a second proof, based on the article  \cite{ahm} (which unfortunately contains many typos (poor proofreading? Poor referee job?). The function $f(x):=x^x$ is convex. 
Let $T_u(x):=f'(u)(x-u)+f(u)$ be the tangent to the graph of $f$ at the point $(u, f(u))$. Then $f(x)\geq T_u(x)$. 
Next, let  $x_1=1/a$, $x_2=1/b$ and  $x_3=1/c$. Then with $u:=S=(x_1+x_2+x_3)/3$, 
$$ \sum_{j=1}^3 f(x_j)\geq \sum_{j=1}^3 T_S(x_j)=\sum_{j=1}^3 \left( f'(S)(x_j-S)+f(S)\right)= f'(S) \sum_{j=1}^3(x_j-S) + 3 f(S)
$$
$$=f'(S)\sum_{j=1}^3 x_j - 3 S f'(S)+3 f(S)=3 f(S).$$
Since $S=(1/a+1/b+1/c)/3=4/3$,  we obtain with $1/a+1/b+1/c=4$ that
$$
(1/a)^{1/a}+(1/b)^{1/b} +(1/c)^{1/c}= \sum_{j=1}^3 f(x_j)\geq 3 f(4/3)= 3 (4/3)^{4/3}=4 (4/3)^{1/3}.
$$


\newpage

   \begin{figure}[h!]
 
  \scalebox{0.5} 
  {\includegraphics{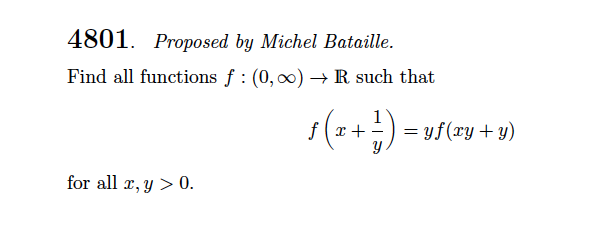}} 

\end{figure}

\centerline{\bf Solution to problem 4801 Crux Math. 49 (1) 2023, 44}\medskip

 \centerline{Raymond Mortini, Rudolf Rupp } \medskip

\centerline{- - - - - - - - - - - - - - - - - - - - - - - - - - - - - - - - - - - - - - - - - - - - - - - - - - - - - -}
  
  \medskip


\ovalbox{We show that all solutions are given by $\dis f(x)= \frac{C}{1+x}$ for $C\in \R$.}\\

$\bullet$ It is straightforward to check that these are solutions:
\begin{eqnarray*}
f\left(x+\frac{1}{y}\right)= \frac{c}{1+ x+\frac{1}{y}}=\frac{cy}{y+yx+1}= y f(xy+y).
\end{eqnarray*}

$\bullet$ Suppose that $f:\,]0,\infty[\to \R$ is a solution. Let $y= \frac{1}{1+x}$. Then
\begin{equation}\label{main4801}
f(2x+1)=f(x+\frac{1}{y})=  y f(y(1+x))=\frac{1}{1+x} f(1).
\end{equation}
Next, let $y=\frac{1}{x}$. Hence, by using  (\ref{main4801}), 
$$
f(2x)=  f(x+\frac{1}{y})=  y f(y(1+x))= \frac{1}{x} f(1+\frac{1}{x})=\frac{1}{x} f(1+2\frac{1}{2x})\rel=_{}^{(\ref{main4801})} \frac{1}{x}\frac{f(1)}{1+ \frac{1}{2x}}=\frac{2 f(1)}{2x+1}.
$$
Now let $X:=2x$ and $C:=2f(1)$. Then $\dis f(X)= \frac{2 f(1)}{ X+1}= \frac{C}{1+X}$.

\newpage

\begin{figure}[h!]
  
   \scalebox{0.5} 
  {\includegraphics{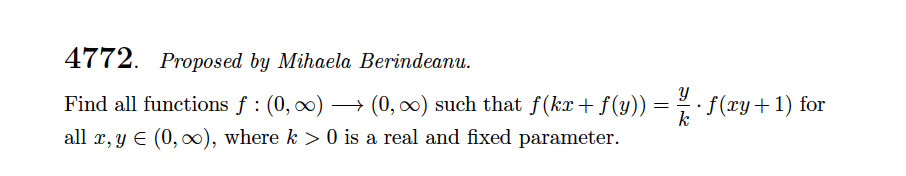}} 
\end{figure}

\centerline{\bf {Solution to problem 4772 Crux Math. 48 (8) 2022, 483}}\medskip

   \centerline{Raymond Mortini, Rudolf Rupp}  \medskip

\centerline{- - - - - - - - - - - - - - - - - - - - - - - - - - - - - - - - - - - - - - - - - - - - - - - - - - - - - -}
  
  \medskip

   For $k=1$, this problem  was given for instance in the Middle European Mathematical Olympiad (MEMO) in 2012 in Switzerland
(see \cite{memo} and \cite{mhr}) and we follow those published solutions.  \\

We claim that for $a>0$, all solutions $f:\R^+\to\R^+$ of \footnote{ We prefer to use the letter $a$ instead of $k$, as for us $k$ always belongs to $\mathbb N$.}
\begin{equation}\label{fgle}
f(ax+f(y))=\frac{y}{a}\;f(xy+1)
\end{equation}
are given by $f(x)=a/x$.  First, it is straightforward to see that this is a solution.
Now we proceed as in  \cite{memo, mhr}. 
Let $f$ be  a solution. 

{\bf Step 1}
Consider  for $y>0$, $y\not=a$, the auxiliary function 
$$
g(y):= \frac{a-yf(y)}{a-y}
$$
(this function is formally  obtained by 
 solving in $\R\times \R^+$ the equation  $ax+f(y)=xy+1$, which gives 
$x=x_y=\frac{1-f(y)}{a-y}$ for $y\not=a$, and so $ax+f(y)= \frac{a-yf(y)}{a-y}=g(y)$. It will turn out that $x=-1/y$ and $g\equiv 0$).\\

Now for every $y>0$ with  $y\not=a$ and $x_y>0$, we have that $g(y)\leq 0$, since otherwise $f$ is well-defined at $g(y)>0$ and so 
$f(g(y))=\frac{y}{a} f(g(y))$, yielding that   $y=a$, a contradiction. \\

{\bf Step 2} 

{\it Case 1} If there would exist $y_0>1$ such that $f(y_0)< a/y_0$,  then with $x_0:=1-\frac{1}{y_0}>0$ we have $x_0y_0+1=y_0$, 
$$u_0:=ax_0+f(y_0)=a-\frac{a}{y_0} +f(y_0)<a,
$$
and
$$f(u_0)=f(ax_0+f(y_0))=\frac{y_0}{a} \;f(y_0)<1.$$
Then $x_{u_0}:=\frac{1-f(u_0)}{a-u_0}>0$ and  so
$$g(u_0)=a x_{u_0}+f(u_0)=x_{u_0}u_0+1>0.$$
But by  Step 1, $g(u_0)\leq 0$, a contradiction.

{\it Case 2} 
  If there would exist $y_1>1$ such that $f(y_1)> a/y_1$, then  by the same reasoning as above, 
  with $x_1:=1-\frac{1}{y_1}$ and 
$$u_1:=a x_1+f(y_1)>a,$$
we have $f(u_1)>1$ and so $g(u_1)>0$, again.  A contradiction.\\

We conclude that $f(y)=a/y$ for every $y>1$. To deal with the remaining case,  take $x=1/a$ and $0<y\leq 1$. Then by (\ref{fgle}),
 
\begin{equation}\label{inter}
f(1+f(y))= \frac{y}{a} f\left(\frac{y}{a}+1\right).
\end{equation}

As both $1+f(y)$ and $\frac{y}{a}+1$ are bigger than $1$, we deduce from (\ref{inter}) that
$$
\frac{a}{1+f(y)}=\frac{y}{a}\frac{a}{\frac{y}{a}+1}=\frac{ay}{y+a}.
$$
Hence $f(y)= a/y$.

\newpage

 \begin{figure}[h!]
  
   \scalebox{0.5} 
  {\includegraphics{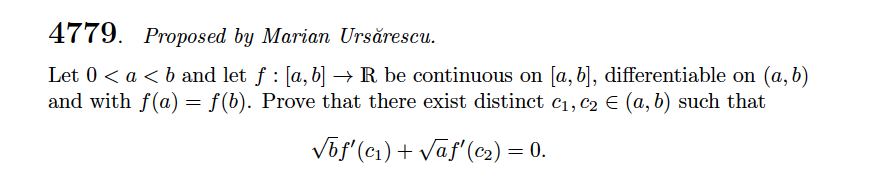}} 
\end{figure}


\centerline{\bf {Solution to problem 4779 Crux Math. 48 (8) 2022, 483}}\medskip

   \centerline{Raymond Mortini} \medskip

\centerline{- - - - - - - - - - - - - - - - - - - - - - - - - - - - - - - - - - - - - - - - - - - - - - - - - - - - - -}
  
  \medskip

 If $f$ is constant, then $f'\equiv 0$ and we may choose any numbers $a<c_1<c_2<b$ to satisfy

\begin{equation}
\sqrt b f'(c_1)_+\sqrt a f'(c_2)=0.
\end{equation}
Otherwise,  $f$ takes its distinct extremal values on $[a,b]$.  We may assume that  $M:=\max_{[a,b]} f >f(a)$
(if not, $M=f(a)$ and so $\min_{[a,b]} f< f(a)$ and we consider $-f$).  Say
$M=f(x_0)$ for some $x_0\in\,]a,b[$. Then $f'(x_0)=0$, and due to continuity of $f'$, there are $a<x_1<x_2\leq x_0$ with $f'(x)>0$ for $x\in\; ]x_1,x_2[$, but $f'(x_2)=0$; we may choose 
$$\mbox{$x_2= \inf \{t \leq x_0: f'\equiv 0$ on $[t,x_0]\}$}.$$
By a similar argument, there are $x_0\leq y_2< y_1$ such that $f'(y_2)=0$, but $f'(x)<0$ for $x\in \;]y_2, y_1[$. By the intermediate value theorem for continuous functions, here for  $f'$, there exists a small $\e>0$  such that $f'$ takes every value from  $[0,\e]$ on $]x_1,x_2]$ and every value from  $[-\e,0]$  on $[y_2,y_1[$.
Now choose $c_1\in\; ]x_1,x_2[$ so that $\frac{\sqrt b}{\sqrt a} f'(c_1) \in\; ]0,\e[$ (this is  possible since $\lim_{x\nearrow x_2} f'(x)=0$).
Hence there exists $c_2\in\; ]y_2,y_1[$ with 
$$f'(c_2)=- \frac{\sqrt b}{\sqrt a} f'(c_1).$$
Thus $\sqrt b f'(c_1)_+\sqrt a f'(c_2)=0$ and $c_1<c_2$. \\

{\bf Remark}  I do not see the role played by the special coefficients $\sqrt a$ and $\sqrt b$. The whole works for any 
$0<s_1<s_2<\infty.$

\newpage

 \begin{figure}[h!]
  
   \scalebox{0.5} 
  {\includegraphics{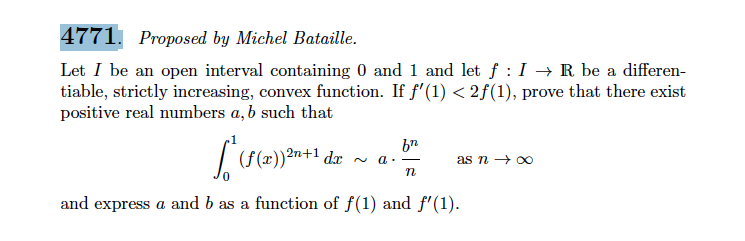}} 
\end{figure}

\centerline{\bf {Solution to problem 4771 Crux Math. 48 (8) 2022, 483}}\medskip

   \centerline{Raymond Mortini, Rudolf Rupp } \medskip

\centerline{- - - - - - - - - - - - - - - - - - - - - - - - - - - - - - - - - - - - - - - - - - - - - - - - - - - - - -}
  
  \medskip

   The problem is a bit ambiguous, due to an undefined $\sim$ symbol.  Let
$$\mbox{$\dis L_n:=\int_0^1 f(x)^{2n+1}\,dx\quad$ and $\quad\dis R_n=a\cdot \frac{b^n}{n}$}$$

Is it $L_n-R_n\to 0$? Or $L_n/R_n\to 1$? Or $c L_n\leq R_n\leq CL_n$ for almost every $n$ and some positive constants
 $c,C$? Note that, a priori, it is not even clear that  $L_n>0$.

We are going to show the following:

\begin{exemple}
$$\ovalbox{$\dis \lim_{n\to\infty} \frac{n+1}{f(1)^{2n+1}}\;\int_0^1 f(x)^{2n+1}\,dx= \frac{f(1)}{2 f'(1)}.$}$$
\end{exemple}

Hence, with $\dis a:= \frac{f(1)^2}{2 f'(1)}$ and $b=f(1)^2$ we get that $L_n/R_n\to 1$. 

\begin{proof}

 Since $f$ is assumed to be  increasing, we see that $f'(x)\geq 0$ for $0\leq x\leq 1$. To exclude that for some points
  $x_0\in\; ]0,1]$, $f'(x_0)=0$, we need the convexity \footnote{ Note that $f$ merely being strictly increasing, does not exclude the existence of zeros of $f'$: $f(x)=(x-1/2)^3$.}  of $f$: in fact, let $T$ be the tangent to the graph of $f$ at $(x_0,f(x_0))$; then $T(x)=f(x_0)+f'(x_0)(x-x_0)$. The convexity of $f$ implies that the graph of $f$ lies above $T$. In particular, if $f'(x_0)=0$,
  then, due to $f$ being strictly increasing,  $f(x_0-\e)< f(x_0)<f(x_0+\e)$ would contradict this fact.  We conclude that $f'(x)>0$ for every $x\in\;]0,1]$.

To calculate our limit, we let $0<s<1$ and write the integral $\dis \frac{n+1}{f(1)^{2n+1}} L_n$ as $I_n(s)+J_n(s)$, where 
$$\mbox{$\dis I_n(s)= \frac{n+1}{f(1)^{2n+1}}\int_0^s f(x)^{2n+1}\;dx\quad$ and 
$\dis\quad J_n(s)= \frac{n+1}{f(1)^{2n+1}}\int_s^1 f(x)^{2n+1}\;dx$}.$$

{\bf Claim 1}  There is a function $h(s)$ with $0<h(s)<1$, such that
\begin{equation}\label{eq1}
\frac{f(1)}{2 f'(1)}\Big(1-h^{2n+2}(s)\Big)\leq J_n(s)\leq \frac{f(1)}{2 f'(s)}.
\end{equation}

To see this, note that $f$ convex and $C^1$ imply that $f'$ is increasing (by the way, a fact equivalent to $f$  being convex). 
By the mean-value theorem,  and  for $s<x\leq 1$, there is $c_x\in\,]s,1[$ with $ f'(c_x)=\frac{f(1)-f(x)}{1-x}$. Hence
$$f'(s) \leq f'(c_x)\leq f'(1)$$
and so
$$f'(s)\leq \frac{f(1)-f(x)}{1-x}\leq f'(1).$$
In other words
\begin{equation}\label{equ2}
f(1)-f'(1) + f'(1)x \leq f(x)\leq f(1)-f'(s)+ f'(s) x.
\end{equation}
Now for $f(x)=Ax+B$ with $A\not=0$  we have
$$\int_s^1 (Ax+B)^{2n+1}\;dx=\frac{(A+B)^{2n+2}-(As+B)^{2n+2}}{A(2n+2)}.$$
Applying this to (\ref{equ2}) yields

\begin{eqnarray*}
\frac{n+1}{f(1)^{2n+1}} \int_s^1 f(x)^{2n+1}\;dx&\leq& \frac{n+1}{f(1)^{2n+1}} \frac{(f'(s)+f(1)-f'(s))^{2n+2}-
(f'(s)s+f(1)-f'(s))^{2n+2}}{f'(s)(2n+2)}\\
&=&\frac{1}{2 f'(s)}\; \frac{f(1)^{2n+2}-(f'(s)s +f(1)-f'(s))^{2n+2}}{f(1)^{2n+1}}\\
&=& \frac{f(1)}{2f'(s)}\left( 1-\left( 1-\frac{f'(s)}{f(1)}(1-s)\right)^{2n+2}\right)\\
&\leq&  \frac{f(1)}{2f'(s)}
\end{eqnarray*}
because $0\leq  1-\frac{f'(s)}{f(1)}(1-s)<1$ for $s\in [s_1,1]$. Similarily,

\begin{eqnarray*}
\frac{n+1}{f(1)^{2n+1}} \int_s^1 f(x)^{2n+1}\;dx&\geq& \frac{n+1}{f(1)^{2n+1}} \frac{(f'(1)+f(1)-f'(1))^{2n+2}-
(f'(1)s+f(1)-f'(1))^{2n+2}}{f'(1)(2n+2)}\\
&=&\frac{1}{2 f'(1)}\; \frac{f(1)^{2n+2}-(f'(1)s +f(1)-f'(1))^{2n+2}}{f(1)^{2n+1}}\\
&=& \frac{f(1)}{2f'(1)}\left( 1-\left( 1-\frac{f'(1)}{f(1)}(1-s)\right)^{2n+2}\right)\\
&=:&  \frac{f(1)}{2f'(1)}\left(1-h(s)^{2n+2}\right),
\end{eqnarray*}
with $h(s):=1-\frac{f'(1)}{f(1)}(1-s)$. Note that $0<h(s)<1$ for $s\in [s_2,1]$.

This finishes the proof of Claim 1.\\

{\bf Claim 2} $\lim_{n\to\infty} I_n(s)=0$ for every $0<s<1$.

To this end, we need to show that $\max_{[0,1]} |f| =f(1)$ and that the maximum is {\it only} obtained at $1$ (note that  $f$ may take negative values). In fact, since $f$ is  increasing, $f(0)\leq f(x)\leq f(1)$ for every $x\in [0,1]$. If $f(0)\geq 0$, nothing has to be proven. So let $f(0)<0$. 
Then, by the mean value theorem on $[0,1]$ there is $0<c_x<1$ such that
$$f(x)=f(0)+f'(c_x) x\leq f(0)+f'(1) x\leq f(0)+f'(1)$$
(note that $f'$ is increasing). Using that $0\leq f'(1)< 2 f(1)$ \footnote{ It is only here that we use this assumption.}, we obtain $f(1)< f(0)+2 f(1)$. Hence $f(0)> -f(1)$. As $f$ is strictly increasing, we also have $f(0)<f(1)$, and so $|f(0)|<f(1)$. Moreover,
 $|f(x)|\not= f(1)$ for any $x\in \;[0,1[$.
 
 We conclude that 
$$|I_n(s)|=\left|\frac{n+1}{f(1)^{2n+1}}\; \int_0^s f(x)^{2n+1}\;\right| \leq (n+1)  s\; \left( \frac{\max_{[0,s]}|f(x)|}{f(1)}\right)^{2n+1}=:
 (n+1)M^{2n+1}, $$
 where $0<M=M(s)<1$. As $\sum_{n=1}^\infty (n+1)M^{2n+1}$ converges, $I_n(s)\to 0$ as $n\to\infty$.\\

We are now ready to determine the limit of  $\frac{n+1}{f(1)^{2n+1}}\;\int_0^1 f(x)^{2n+1}\,dx$. To this end, fix $\e>0$ and choose $s_3=s_3(\e)\in ]0,1[$ so that for all $s\in\; [s_3,1]$
$$\left|\frac{f(1)}{2 f'(s)}-\frac{f(1)}{2f'(1)}\right|<\e.$$

Now for $s_0:=\max \{s_1,s_2,s_3\}$, depending on $\e$,  we obtain from Claim 1 that
$$\frac{f(1)}{2 f'(1)}\Big(1-h^{2n+2}(s_0)\Big)\leq J_n(s_0)\leq \frac{f(1)}{2 f'(s_0)}\leq \frac{f(1)}{2 f'(1)}+\e .$$
Since $0<h(s_0)<1$, there is $n_0=n_0(\e,s_0)$ such that 
$$\mbox{$\dis 0<h(s_0)^{2n+2}<\e$ for all $n\geq n_0$}.$$
Thus, for $n\geq n_0$

$$
\frac{f(1)}{2 f'(1)} (1-\e)\leq  J_n(s_0)\leq  \frac{f(1)}{2 f'(1)}+\e.
$$

By Claim 2, there is $n_1\geq n_0$ (depending on $\e$) such that $|I_n(s_0)|<\e$ for $n\geq n_1$. We conclude that for these $n\geq n_1$
$$
\frac{n+1}{f(1)^{2n+1}} L_n =I_n(s_0)+J_n(s_0)\begin{cases} 
\leq& \e+  \frac{f(1)}{2 f'(1)}+\e\\
\geq& -\e +\frac{f(1)}{2 f'(1)} (1-\e).
\end{cases}
$$
Hence
$$\left| \frac{n+1}{f(1)^{2n+1}} L_n-\frac{f(1)}{2 f'(1)}\right|\leq \max\left\{ 2\e, \e\left(1+ \frac{f(1)}{2f'(1)}\right)\right\}.
$$

\end{proof}

{\bf Remark}  The function $f(x)=x-1/2$ shows that the assertion may fail if $f'(1)=2 f(1)$, since in this case
$L_n=0$. On the other hand, it may hold, too if $f'(1)=2f(1)$. In fact,  if $f(x)=e^{2x}$, then $f'(1)=2 f(1)$ and 
$$\mbox{$\dis L_n= \frac{e^{4n+2}-1}{4n+2}\quad $ and $\dis \quad R_n=\frac{e^4}{4e^2}\cdot \frac{e^{4n}}{n}=
\frac{e^{4n+2}}{4n}$,}
$$
nevertheless $ L_n/R_n\to 1$. What is the reason for this? Well, an analysis of the proof shows that the condition
$f'(1)<2f(1)$ can  be replaced by the assumption that
 the maximum of $|f|$ is {\it only} obtained at $1$. This makes the class of functions with the wished assymptotic behavior of the integrals $\int_0^1 f(x)^{2n+1}\,dx$ much larger.

\newpage

 \begin{figure}[h!]
  
   \scalebox{0.5} 
  {\includegraphics{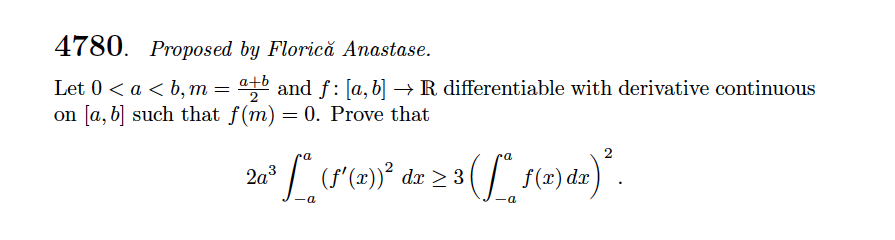}} 
\end{figure}

\centerline{\bf {Solution to problem 4780 Crux Math. 48 (8) 2022, 484}}\medskip

   \centerline{Raymond Mortini, Rudolf Rupp } \medskip

\centerline{- - - - - - - - - - - - - - - - - - - - - - - - - - - - - - - - - - - - - - - - - - - - - - - - - - - - - -}
  
  \medskip

The assertion is not compatible with the hypotheses. So we prove the following two results:

\begin{exemple}
Let $a>0$ and $f\in C^1[-a,a]$. If $f(0)=0$, then 
$$\int_{-a}^a (f'(x))^2\;dx \geq \frac{3}{2a^3}\left(\int_{-a}^af(x)\,dx\right)^2.$$
\end{exemple}

\begin{exemple}
Let $0<a,b<\infty$ and $f\in C^1[a,b]$. If $f((a+b)/2)=0$, then, with $C=\frac{12}{(b-a)^3}$,
$$\int_{a}^b (f'(x))^2\;dx \geq C\left(\int_{a}^bf(x)\,dx\right)^2.$$
\end{exemple}

\begin{proof}[Proof of Example 1]
Let $p$ be  a polynomial. Then, using The Cauchy-Schwarz inequality
$$I:=\left(\int_0^a (f' p)(x)\,dx\right)^2\leq \left(\int_0^a (f'(x))^2\, dx \right)\;\left( \int_0^a p(x)^2\,dx\right)$$
Using partial integration,
$$I=\left((f(x)p(x)\Big|_0^a -\int_0^a f(x)p'(x)\,dx\right)^2$$
Now choose $p(x)=x-a$. Then $ \int_0^a p(x)^2\,dx=\frac{1}{3}(x-a)^3\Big|_0^a=\frac{1}{3}a^3$. Hence, by noticing that 
$p(a)= f(0)=0$, 
$$I=\left(\int_0^a f(x)\,dx\right)^2\leq  \left(\int_0^a (f'(x))^2\, dx \right) \frac{1}{3}a^3
$$
If we choose $p(x)=x+a$, then $p(-a)=0$, and we similarily obtain the appropriate estimation for   $\int_{-a}^0 f(x)dx$.
Hence, using that $(x+y)^2\leq 2(x^2+y^2)$, 
$$\left(\int_{-a}^af(x)\,dx\right)^2 \leq \frac{2}{3} a^3 \int_{-a}^a (f'(x))^2\;dx$$
\end{proof}

\begin{proof}[Proof of Example 2]
Just use the affine transformation $\phi$ given by  $\phi(x)= x+\frac{a+b}{2}$. 
Then $\phi(-\frac{b-a}{2})=a$ and $\phi(\frac{b-a}{2})=b$, as well as $\phi(0)=\frac{a+b}{2}$. Let $c:= (b-a)/2$. Hence, with
$F(t):= f(\phi(t))$ for $-c\leq t\leq c$ we obtain
$$\int_a^b (f'(x))^2\,dx=\int_{-c}^c (F'(t))^2 \,dt\geq  \frac{3}{2 c^3} \left(\int_{-c}^c F(t)\,dt\right)^2= 
\frac{12}{(b-a)^3}\left(\int_a^b f(x)\, dx\right)^2.$$
\end{proof}

Of course Example 1 is  a special case of Example 2.  Is $C$ best possible? Let
$$q(x)=\begin{cases} \displaystyle \frac{(x-a)^2}{2}-\frac{(b-a)^2}{8} & \text{if $a\leq x\leq (a+b)/2$}\\
\displaystyle \frac{(x-b)^2}{2} -\frac{(a-b)^2}{8}&\text{if $(a+b)/2\leq x\leq b$}.
\end{cases}$$
Then $q$ is continuous on $[a,b]$,  $q((a+b)/2)=0$ and 
$$\int_a^b (q'(x))^2\,dx=\frac{12}{(b-a)^3}\left(\int_a^b q(x)\, dx\right)^2.$$
Unfortunately,  $q$ is not $C^1$. How to modify?

\newpage

  \begin{figure}[h!]
  
   \scalebox{0.5} 
  {\includegraphics{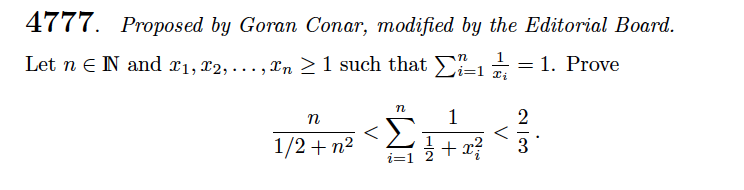}} 
\end{figure}


\centerline{\bf Solution to problem 4777 Crux Math. 48 (8) 2022, 484}\medskip

   \centerline{Raymond Mortini, Rudolf Rupp } \medskip

\centerline{- - - - - - - - - - - - - - - - - - - - - - - - - - - - - - - - - - - - - - - - - - - - - - - - - - - - - -}
  
  \medskip

The assertion is not correct. In fact, let $\mathbf x:=(x_1,\dots,x_n)$, $\mathbb R^+=\{x\in \mathbb R: x>0\}$,
$$S:=\Big\{(x_1,\dots,x_n)\in (\mathbb R^+)^n: \sum_{i=1}^n \frac{1}{x_i}=1\Big\},$$ 
and 
$$f(\mathbf x):= \sum_{i=1}^n \frac{1}{\frac{1}{2}+x_i^2}.$$
We prove that for $n\geq 2$,\\

\centerline{
\ovalbox{$\displaystyle
\frac{n}{1/2+n^2}=\min_{\mathbf x\in S} f(\mathbf x )<\sup_{\mathbf x\in S} f(\mathbf x )   =\frac{2}{3},$
}}
\medskip

and that for $n=1$, $S=\{1\}$ and  so\\

\centerline{
\ovalbox{$\displaystyle f(x)=\frac{1}{\frac{1}{2}+x^2}= f(1)=\frac{2}{3}.$}
}

\medskip

{\sl Proof} Wlog $n\geq 2$.
First we note that $\sum_{i=1}^n 1/x_i=1$ for $x_i\in \mathbb R^+$  implies that $x_i\geq 1$ for every $i$. Now
$\displaystyle \max_{1\leq x<\infty}\frac{x}{1+2x^2}=\frac{1}{3}$, since the function is decreasing on $[1,\infty[$. Hence, 
for $\mathbf x\in S$,
$$\sum_{i=1}^n \frac{1}{\frac{1}{2}+x_i^2}= \sum_{i=1}^n \frac{x_i}{1+2x_i^2}\frac{2}{x_i}< \frac{2}{3}
\sum_{i=1}^n \frac{1}{x_i}=\frac{2}{3},$$
since  for $n\geq 2$,  no $x_i$ can be $1$.  If for $k>n$
$$\mathbf x_k=\Big(x_1^{(k)},\dots,x_n^{(k)}\Big):= \Big(\frac{1}{1-(n-1)/k}, k,\dots,k\Big),$$
 then $\sum_{i=1}^n (1/x_i^{(k)})=1$, $\mathbf x_k\to (1, \infty,\dots,\infty)$ and $f(\mathbf x_k)\to 2/3$. Hence $\sup_S f=2/3$.

To prove the assertion on the minimum, we use Lagrange. It is preferable to work with the new variable $y_j:=1/x_j$ (to get a compact definition set, guarantying the existence of the global extrema). 
So let 
$$S'=\Big\{(y_1,\dots,y_n)\in \mathbb R^n,\, y_j\geq 0: \sum_{j=1}^n y_j=1\Big\}$$
and $$g(y_1,\dots,y_n):= \sum_{i=1}^n \frac{y_i^2}{1+\frac{1}{2}y_i^2}.$$
Then  $S'$ is compact and $\inf f_S=\inf g_{S'}=\min g_{S'}=:m$. 
Say $g(\mathbf  x')=m$ for some $\mathbf x'\in S'$. In order to apply Lagrange, we need to show that  $\mathbf x'$ is an interior point of $S'$ (in symbols,  $\mathbf x'\in (S')^\circ$).  Let $\mathbf y':=(1/n,\dots,1/n)$.  Then $\mathbf y'\in (S')^\circ$.
Now on $\partial S'$ at least one of the coordinates of these points  $\mathbf y:=(y_1,\dots, y_n)\in \partial S'$ is $0$. Say, $y_n=0$. But then
$\sum_{i=1}^{n-1} y_i=1$ and  (via induction on $n$, starting with the trivial case of one-tuples)
$$g(\mathbf y) \geq \frac{n-1}{1/2+(n-1)^2}>  \frac{n}{1/2+ n^2}=g(\mathbf y').$$
Hence the absolute minimum of $g$  on $S'$ does not belong to the boundary.

By Lagrange's theorem, there exists $\lambda \in \mathbb R$ and $(y_1,\dots,y_n)\in S'$ such that
$$\nabla\Big( g(y_1,\dots,y_n)+\lambda (1-\sum_{i=1}^n y_i)\Big)=\mathbf 0.$$
That is, for every $i\in\{1,\dots,n\}$,
\begin{equation}\label{main4777}
\lambda=\frac{2y_i}{(1+\frac{1}{2}y_i^2)^2}.
\end{equation}

Unfortunately,  the function $y\mapsto q(y):=\frac{2y}{(1+\frac{1}{2}y^2)^2}$ is not  injective on $[0,1]$ (note that the derivative vanishes at $y=\pm \sqrt{2/3}$).  So we must discuss several cases (see figure \ref{no-inj}): \\

  \begin{figure}[h!]
 
  \scalebox{0.35} 
  {\includegraphics{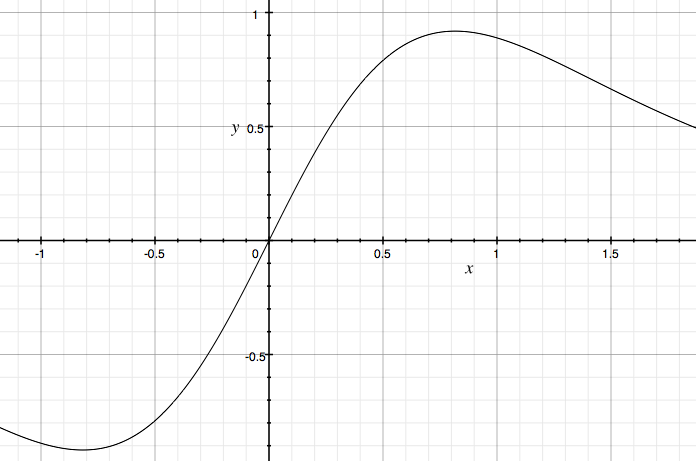}} 
\caption{\label{no-inj} Non injectivity of $q$ on $[0,1]$}

\end{figure}

(i) If $8/9=q(1) \leq \lambda < \max_{[0,1]} q$, then the equation $q(y)=\lambda$ has two solutions $0<y_1,y_2\leq 1$. 

(ii) If $\lambda= \max_{[0,1]} q$  or if $0\leq \lambda <q(1)=8/9$,
then the equation $q(y)=\lambda$ has exactly one  solution $0\leq y_0\leq 1$.

(iii) In all other cases, there is no solution with $y\geq 0$.\\

We first show that  the case (i)  does not yield minimal solutions.
 In fact, for fixed $\lambda\in [q(1), \max_{[0,1]} q[$, 
equation  (\ref{main4777})
has  $2^n$ solutions  of the form
 $P:=(\underbrace{a,\dots, a}_{k\text{-times}}, \underbrace{b,\dots,b}_{(n-k) \text{-times}})$ and their permutations, where 
$k=0,\dots, n$ and $0\leq a\leq b\leq 1$.  
Note that
\begin{equation}\label{abs}
q(1/n)= \frac{8n^3}{(1+2n^2)^2}\leq q(1/2)<  q(1) =\frac{8}{9}<q(\sqrt{2/3}).
\end{equation}
Hence $1/n\leq 1/2<\min\{a,b\}$ (see figure \ref{no-inj}).

Let $A:=(1/n,\dots, 1/n)$. Then $A\in S'$. Since the function $y\mapsto y^2/(1+\frac{1}{2} y^2)$ is 
increasing on $[0,\infty[$, we deduce that
$$g(P)=k \frac{a^2}{1+\frac{1}{2}a^2}+ (n-k) \frac{b^2}{1+\frac{1}{2}b^2} > g(A),$$
so $P$ does not yield a minimum.  Thus only the second case occurs.
That is, we need to consider only a solution of (\ref{main4777}) of the form $(y_1,\dots, y_n)=(a,\dots,a)$ with $0<a\leq 1$.
  Using the constraint condition 
$\sum_{i=1}^n y_i=1$, we obtain that $a=1/n$, hence $(y_1,\dots, y_n)=(1/n,\dots, 1/n)$. 
 Consequently, $\mathbf x'=(1/n,\dots, 1/n)$ is the unique point
where $g$ takes its absolute minimum on $S'$. We conclude that 
$$\min g_{S'}=\frac{n}{\frac{1}{2}+n^2}.$$

For completeness, we observe that $M:=\max_{S'} g$ necessarily is obtained on the boundary of $S'$ (for instance, $M=g(1,0,\dots, 0)=2/3$),
 as Lagrange only yields a single stationary point of the Lagrange function in $(S')^\circ$. \\

A second way to see that case (i) does not occur goes as follows:

We first show that  the case (i)  does not yield minimal solutions.
 In fact, for fixed $\lambda\in [q(1), \max_{[0,1]} q[$, 
equation  (\ref{main4777})
has  $2^n$ solutions  of the form
 $P:=(\underbrace{a,\dots, a}_{k\text{-times}}, \underbrace{b,\dots,b}_{(n-k) \text{-times}})$ and their permutations, where 
$k=0,\dots, n$ and $0\leq a\leq b\leq 1$.  
Note that $q(1/2)= (8/9)^2$ and that $n\geq 2$. Thus
\begin{equation}\label{absi}
q(1/n) \leq q(1/2)<  q(1)\leq \lambda <q(\sqrt{2/3}).
\end{equation}
Hence $1/n\leq 1/2<\min\{a,b\}=a$ (see figure \ref{no-inj}). Since for such a point $P=(y_1,\dots,y_n)$ we have
$$\sum_{i=1}^n y_i= ka +(n-k)b> k \frac{1}{2}+ (n-k)\frac{1}{2}=\frac{n}{2}\geq 1,$$
$P$ does not belong to $S'$; that is such a solution of the system (\ref{main4777}) of equations does not satisfy the constraint
 $P\in S'$.

\newpage

  \begin{figure}[h!]
  
   \scalebox{0.5} 
  {\includegraphics{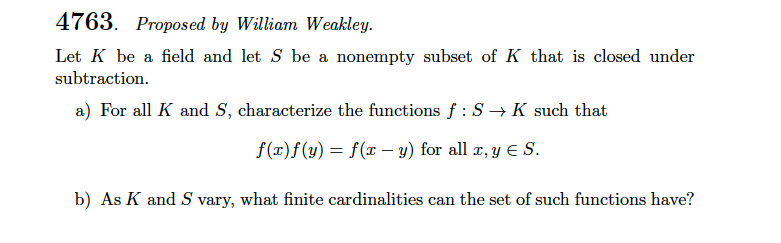}} 
\end{figure}


\centerline {\bf {Partial Solution to problem 4763 Crux Math. 48 (7) 2022, 421}}\medskip
  
 \centerline{Raymond Mortini, Rudolf Rupp } \medskip

\centerline{- - - - - - - - - - - - - - - - - - - - - - - - - - - - - - - - - - - - - - - - - - - - - - - - - - - - - -}
  
  \medskip

Here we give our thoughts on this not  very precisely formulated problem.
  
 First we note that $S\ss K$ necessarily is an additive subgroup of the field $K$. Note that $\{0,1\}\ss K$. In particular $0=x-x\in S$ and with $x\in S$ we have $-x=0-x\in S$.
 
 If 
 \begin{equation*}(FE)\hspace{3cm}\label{fgll}
 \text{$f(x-y)=f(x)f(y)$  for all $ x,y\in S$}
\end{equation*}
then we get the following:

(1) $y=x\imp f(0)=f(x)^2$

(2) $y=0\imp f(x)=f(x)f(0)\imp f(x)(1-f(0))=0$

{\it Case 1} There exists $x_0\in S$ with $f(x_0)=0$. Then, by (1), $f(0)=0$ and so $f(x)=0$ for all $x\in S$.

{\it Case 2} $f$ has no zeros. Then (2) implies that $f(0)=1$.

We claim that $f(2 x)=1$ for every $x\in S$  (note that $\Z S\ss S)$.

In fact, $f(x)=f(2x-x)=f(2x)f(x)$, hence $f(2x)=1$.

We conclude that  for $S=\R$ e.g., the constant function $f(y)=1$ is the only solution, as every $y\in \R$ writes as $y=2x$ for some $x$.

Next we show that $f$ is even and that $f(x)\in \{-1,1\}$. In fact, by (FE), for $x=0$, 
$$\text{$f(-y)=f(0)f(y)=f(y)$ for every $y\in S$}.$$
Hence $1=f(2u) =f(u-(-u))=f(u)f(-u)=f(u)^2$ for any $u\in S$.

If $S=\Z$, then we have  three solutions: $f\equiv 0$, $f\equiv 1$ but also 
$$f(n)=\begin{cases} 1 &\text{if $n$ even}\\  -1&\text{if $n$ odd}.
\end{cases} $$

In fact by the claim above, $f(2m)=1$ for every $m\in \Z$.  Now let $\sigma:=f(1)$. We already know that $\sigma=\pm 1$. 
Now for every $m\in \Z$, 
$$\sigma=f(1)=f((2m+1)-2m)=f(2m+1) f(2m)=f(2m+1).$$

Let $P:=P_f:=\{x\in S: f(x)=1\}$ and $R:=\{x\in S: f(x)=-1\}$. Then $P$ is a subgroup of $S$ since $x,y\in P$ implies that
$x-y\in P$, because $f(x-y)=f(x)f(y)=1\cdot 1=1$. 

As shown above,   $2S\ss P\ss S$ and $2S$ is a subgroup of $S$.  Here $S=2S$  if and only if all the translation operators 
$\tau_x: S\to S, y\mapsto x-y$ have a fixed point.

Also note that $R$ has the following property:
 
 \begin{equation*}(PR)\hspace{2cm}\label{PR}
 \text{$(R-R)\ss P$ and $(R-P) \union (P-R)  \ss R$}.
\end{equation*}
 
 Conversely, if $P$ is a proper subgroup of $S$ and $R:=S\setminus P$ such that (PR) holds, then the function $g$ given by
 $$g(x)=\begin{cases} 1 &\text{if $x\in P$}\\ -1&\text{if $x\in R$}\end{cases},
 $$
 satisfies the functional equation (FE) $g(x-y)=g(x)g(y)$ for $x,y\in S$.
 
 Note that $P$ may be  strictly bigger than $2S$:  in fact, let  $K=\C$, $S:=\Z+i\Z$,  $P=2\Z+i\Z$ and $R= S\setminus P$. Then $S,P,R$ satisfy (PR), but $P:=2S$  does not satisfy (PR).

 If $S=K$ is a field of characteristic 2, then $P_f=R=S$ (note that $1=-1$),  and so only the constant functions 1 and 0 satisfy (FE).

\newpage

  \nopagecolor
  \begin{figure}[h!]
  
   \scalebox{0.5} 
  {\includegraphics{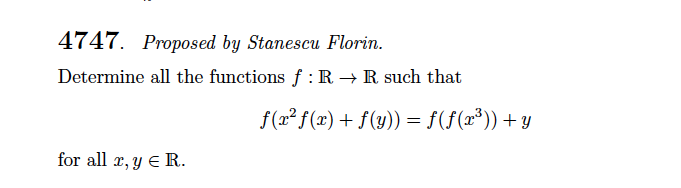}} 
\end{figure}

\centerline  {\bf Solution to problem 4747 Crux Math. 48 (2022), 282 by} \medskip
  
\centerline{Raymond Mortini, Rudolf Rupp } \medskip

\centerline{- - - - - - - - - - - - - - - - - - - - - - - - - - - - - - - - - - - - - - - - - - - - - - - - - - - - - -}
  
  \medskip

  We claim that all solutions of the functional equation
$$f(x^2f(x)+f(y))=f(f(x^3)))+y, ~~x,y\in \mathbb R$$
are given by
$f(x)=x$ and $f(x)=-x$.   

\bigskip

{\bf Claim1} $f$ is injective:

Put $x=0$.  Then
\begin{equation}\label{1}
f(f(y))=f(f(0))+y.
\end{equation}
Now if $f(y_1)=f(y_2)$, then by (\ref{1})
$$\mbox{$f(f(y_1))=f(f(0))+y_1$ and $f(f(y_2))=f(f(0))+y_2$}$$
Hence $y_1=y_2$.\\

{\bf Claim 2} $f$ is surjective:

Let $w\in \mathbb R$. Then, by (\ref{1}),
$$w=f(f(0))+(w-f(f(0)))=f(f(w-f(f(0))).$$\\

{\bf Claim 3} $f(0)=0$:

Take $y=0$: then $f(x^2f(x)+f(0))=f(f(x^3))$. Since  $f$ is bijective, we conclude that  $x^2f(x)+f(0)=f(x^3)$.
Now put $x=1$: then $1^2\; f(1) +f(0)=f(1)$. Hence $f(0)=0$.\\

{\bf Claim 4} $f\circ f=id$ (that is, $f$  is an involution).

This follows from  (\ref{1}).

Hence our equation becomes

\begin{equation}\label{2}
f(x^2f(x)+f(y))=x^3+y  \;\;  (x,y\in \mathbb R).
\end{equation}

In particular, for $y=0$, 
\begin{equation}\label{3}
\mbox{$f(x^2f(x))=x^3$ or equivalenty  $x^2f(x)=f(x^3)$}.
\end{equation}

{\bf Claim 5} $f$ is additive:

In fact, the surjectivity of $f$ and $x\mapsto x^3$ now imply that $x\mapsto x^2f(x)$ is surjective, too.
Hence
$$f(\underbrace{x^2f(x)}_{=a}+\underbrace{f(y)}_{=b})=x^3+y=f(x^2f(x))+y=f(a)+f(b)$$ yields the additivity of $f$.

{\bf Claim 6}  $f(-x)=-f(x)$.

Just use that with  $f(0)=0$ and $f$ additive, 
$$0=f(0)=f(x+(-x))=f(x)+f(-x).$$

{\bf Claim 7} Let $f(a)=1$. Then $a=\pm 1$.

 Recall that  by (5) and (6), $f(a+b)=f(a)+f(b)$  for $a,b\in\mathbb R$ and $f(mx)=mf(x)$ for every 
$m\in \mathbb Z$.  Hence, by (\ref{3}), for $x=a+b$,
$$ (a+b)^2 f(a+b)=f((a+b)^3).$$
Expansion yields:

\begin{eqnarray*}\label{haupt}
(a^2+b^2 +2ab)(f(a) +f(b))&= &f(a^3)+3 f(a^2b)+3 f(ab^2)+f(b^3)\iff\nonumber\\
\underline{a^2f(a)}+b^2 f(a)+2ab f(a)+ a^2f(b) +\underline{b^2 f(b)}+2ab f(b)&=&  \underline{f(a^3)}+3 f(a^2b)+3 f(ab^2)+\underline{f(b^3)}\iff\nonumber\\
b^2f(a)+2ab f(a)+a^2f(b)+2ab f(b)&=&3f(a^2b)+3f(ab^2).
\end{eqnarray*}

$\bullet$ Let  $b=1$ and note that $a=f(1)$. Then
$$ 1+2a +a^3+2a^2= 3 f(a^2) +3 =3a^3+3\iff
$$
$$2 a^3-2a^2 -2a+2=0\iff a^2(a-1)-(a-1)=0\iff (a-1)(a^2-1)=0\iff a\in \{-1,1\}.$$\\

{\bf Claim 8} If the additive function $f$  satisfies $x^2f(x))=f(x^3)$, then  $f(x)= f(1)\, x$. To see this, we consider four
 cases:

$\bullet$\; Let $a=2$, $f(1)=\pm 1$ and $b=x$. Then

\begin{equation}\label{a}
\ovalbox{$\pm x^2 \pm 4x-4f(x)+2xf(x)-3f(x^2)=0$}.
\end{equation}

$\bullet$\; Let $a=1$, $f(1)=\pm 1$ and $b=x$. Then,

\begin{equation}\label{b}
\ovalbox{$\pm x^2\pm  2x-2f(x)+2 xf(x)- 3f(x^2)=0$}.
\end{equation}

Calculating  (\ref{a})-(\ref{b}), yields $\pm 2x -2 f(x)=0$. Hence $f(x)=\pm  x=f(1) x$.\\

\hrule

\medskip

One can also prove Claim 8 without using Claim 7, and then deducing Claim 7 from Claim 8 if additionally we assume  that
$f$ is an involution. \\

In 
\begin{equation}
b^2f(a)+2ab f(a)+a^2f(b)+2ab f(b)=3f(a^2b)+3f(ab^2).
\end{equation}
choose $a=1$, resp. $a=2$ and $b=x$. Then $f(2)=f(2\cdot 1)=2 f(1)$ and so
\begin{equation}\label{n1}
 x^2 f(1)+2 f(1) x+ f(x)+2xf(x)- 3 f(x)-3 f(x^2)=0
 \end{equation}
\begin{equation}\label{n2}
2 x^2  f(1)+ 8 f(1) x+ 4f(x)+4xf(x)-12 f(x)-6f(x^2)=0
\end{equation}
Hence, by calculating (\ref{n1})-$\frac{1}{2}$(\ref{n2}), we obtain
\begin{equation}
-2 f(1)x - f(1) x +3 f(x)=0.
\end{equation}
Hence $f(x)=f(1) x$.  Using (\ref{3}), that is $f(x^2 f(x))=x^3$, we have
$$f(1) x^2 f(1)x =x^3.$$
Hence $ f(1)^2=1$ and so $f(1)=\pm 1$.

  \newpage

 \begin{figure}
  \scalebox{0.5} 
  {\includegraphics{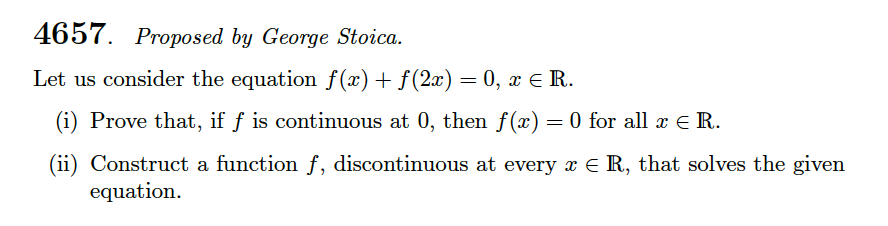}} 
\end{figure}

\centerline{\bf Solution to problem 4657 Crux Math. 47 (2021), 301 by} \medskip

\centerline{Raymond Mortini, Rudolf Rupp }\medskip

\centerline{- - - - - - - - - - - - - - - - - - - - - - - - - - - - - - - - - - - - - - - - - - - - - - - - - - - - - -}
  
  \medskip

}

a) Suppose that the  function $f$ satisfies  $f(x)+f(2x)\equiv 0$ on $\mathbb R$. Then the continuity of $f$ at $x=0$ implies that
 $f\equiv 0$. In fact, fix $x\in \mathbb R\setminus \{0\}$. By induction, $f(x/2^n)=(-1)^n f(x)$. By taking  limits,  the continuity at $0$ implies that for $n$ even we get $f(0)=f(x)$ and for $n$ odd, we get $f(0)=-f(x)$. Hence $2f(x)=f(0)-f(0)=0$, and so $f\equiv 0$.
 
 b) Define the function $f:\mathbb R^+\to \mathbb R$ by $f(0)=0$, $f(x)=1$ if $1\leq x<2$ and $x$ rational, $f(x)=-1$  if $1<x<2$ and $x$ irrational.
  If $n\in \mathbb N$ and $2^n\leq x<2^{n+1}$, put $f(x)=(-1)^nf(x/2^n)$. If $\frac{1}{2^{n+1}}\leq x< \frac{1}{2^n}$, put $f(x)=(-1)^nf(2^nx)$.
 If $x<0$, then let $f(x)=f(-x)$. Then $f$ is discontinuous everywhere and, by construction, $f(x)+f(2x)=0$.
\bigskip

c) \underline{All solutions to $f(x)+f(2x)\equiv 0$ on $\mathbb R$:}

Let $g:[-2,-1[\cup [1,2[\to \mathbb R$ be an arbitrary function.  Put  

$$f(x)=\begin{cases}  0 &\text{ if $x=0$}\\ 
(-1)^n g(x/2^n) &\text{if  $2^n \leq |x|< 2^{n+1}$}\\
 (-1)^ng(2^nx) &\text{if $\frac{1}{2^{n+1}}\leq |x|< \frac{1}{2^n}$}.
\end{cases}
$$

This functional equation and  its companion $f(x)=f(2x)$ appear multiple times, see 
\cite{1,2,3,4,5,6,7,8,9,10}

\newpage
  \begin{figure}[h!]
 
  \scalebox{0.5} 
  {\includegraphics{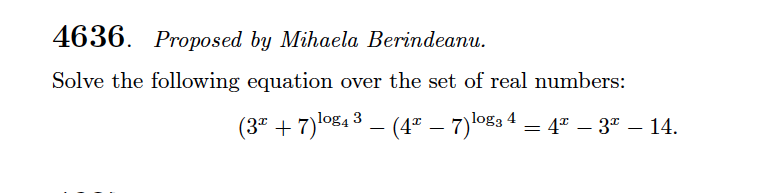} }
\end{figure}

\centerline{\bf Solution to problem 4636   Crux Math. 47 (2021), 200 by} \medskip

\centerline{Raymond Mortini, Rudolf Rupp } \medskip

\centerline{- - - - - - - - - - - - - - - - - - - - - - - - - - - - - - - - - - - - - - - - - - - - - - - - - - - - - -}
  
  \medskip

The equation 
$$(3^x+7)^{\log_4 3}-(4^x-7)^{\log_3 4}=4^x-3^x-14$$

has  on $\mathbb R$ the unique solution $x=2$. In fact, first note that $a:=\log_4 3=\frac{\log 3}{\log 4}>0$ 
and $\log_3 4=1/a$. Then with
$A:=3^x+7$ and $B:=4^x-7$ we have to solve $A^a-B^{1/a}=B-A$ or equivalently,
$$A^a+A=  (B^{1/a})^a +B^{1/a}.$$
Since the function $x\mapsto x^a+x$ is strictly increasing, we deduce that $A=B^{1/a}$. In other words,
$3^x+7= (4^x-7)^{1/a}$, or equivalently
\begin{equation}\label{eqoo}
\log 4 \log(4^x-7)=\log 3 \log (3^x+7).
\end{equation}

The curve $y(x)=\log 4 \log(4^x-7)-\log 3 \log(3^x+7)$ is defined for $x> \log 7/\log 4:=x_0$ with  $\lim_{x\to x_0} y(x)=\infty$ and its derivative 
$$y'(x)=\log^2 4 \frac{1}{1-7x^{-4}}-\log^2 3 \frac{1}{1+7x^{-3}}$$
is strictly  decreasing
with $\lim_{y\to x_0} y'(x)=\infty$ and $\lim_{x\to\infty} y'(x)=(\log^2 4-\log^2 3)$. Note that   the asymptote at infinite is the line $y=(\log^2 4-\log^2 3) x$. In particular, $y'>0$ and so the curve  is strictly increasing and its unique zero is $x_1=2$ (observe that
$\log(4) \log(9)=4\log(2)\log(3)=\log(3)\log(16)$, so (\ref{eqoo}) holds).

\newpage
  \begin{figure}[h!]
 
  \scalebox{0.5} 
  {\includegraphics{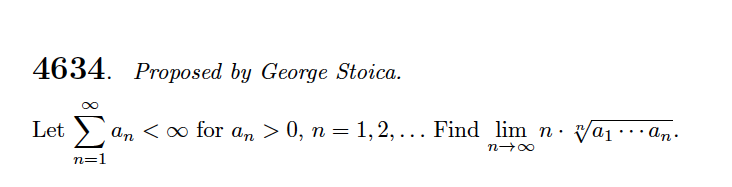}} 
\end{figure}

\centerline{\bf Solution to problem 4634   Crux Math. 47 (2021), 200 by} \medskip

\centerline{Raymond Mortini, Rudolf Rupp } \medskip

\centerline{- - - - - - - - - - - - - - - - - - - - - - - - - - - - - - - - - - - - - - - - - - - - - - - - - - - - - -}
  
  \medskip

For $a_n>0$, 
let $G_n:=n(a_1\cdots a_n)^{1/n}$. Then $\lim_{n\to\infty} G_n=0$ whenever $\sum_{n=1}^\infty a_n$ is convergent. In fact,
given $\e>0$, choose $N$ so big that $\sum_{n=N}^\infty a_n<\e$.
Due to the arithmetic-geometric inequality, for $n> N$, 
$$G_n=(a_1\cdots a_N)^{1/n} \; \frac{n}{n-N} \bigg( (n-N)(a_{N+1}\cdots a_n)^{1/(n-N}\bigg)^{\frac{n-N}{n}}
(n-N)^{1-\frac{n-N}{n}}
$$
$$\leq \sigma_n \bigg(\sum_{j=N+1}^n a_j\bigg)^{\frac{n-N}{n}},$$
where 
$$\sigma_n:=(a_1\cdots a_N)^{1/n} \; \frac{n}{n-N}(n-N)^{N/n}.$$
Since $\lim_n\sigma_n=1$, we have $\limsup_n G_n\leq \limsup_n\e^{\frac{n-N}{n}}=\e$, from which we deduce that $G_n\to 0$.

\newpage

  \begin{figure}[h!]
 
  \scalebox{0.5} 
  {\includegraphics{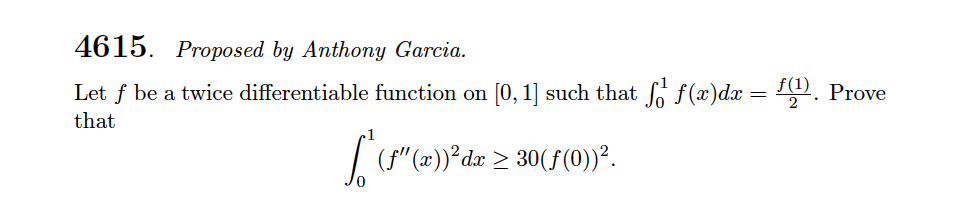}} 
\end{figure}

\centerline{\bf Solution to problem 4615  Crux Math.  47 (2021), 301 by} \medskip

\centerline{Raymond Mortini, Rudolf Rupp } \medskip

\centerline{- - - - - - - - - - - - - - - - - - - - - - - - - - - - - - - - - - - - - - - - - - - - - - - - - - - - - -}
  
  \medskip

If $p$ is a polynomial, we have (due to Cauchy-Schwarz)
$$\left|\int_0^1 f''p dx\right|^2\leq \left(\int_0^1 (f'')^2dx\right) \left(\int_0^1 p^2dx\right).$$
Now, by using twice integration by parts, 
$$\int f''p dx= (f'+c)p- \Big( (f+cx+c')p'- \int (f+cx+c')p'' dx\Big)$$
Now let $p(x)=x(x-1)$. Evaluation at the end-points and using the hypothesis that $\int_0^1 fdx=f(1)/2$, yields
$$\int_0^1 f''p dx= -f(0).$$
Since $\int_0^1p^2dx= \int_0^1 (x^4+x^2-2x^3)dx=1/30$, 
we deduce that
$$\int_0^1 (f'')^2dx\geq 30 f(0)^2.$$
Equality is given if $f''=p$ and $f(1)=2\int_0^1 fdx$; for instance if
$$f(x)=\frac{1}{12}x^4 -\frac{1}{6}x^3-\frac{1}{30}.$$
Here $f(1)=-7/60$.

\medskip  Generalizations appear in \cite{mr21}.
\newpage

{\huge \section{\gr{EMS Newsletter}}}

\vspace{1cm}


\begin{figure}[h!]

\scalebox{0.5} 
  {\includegraphics{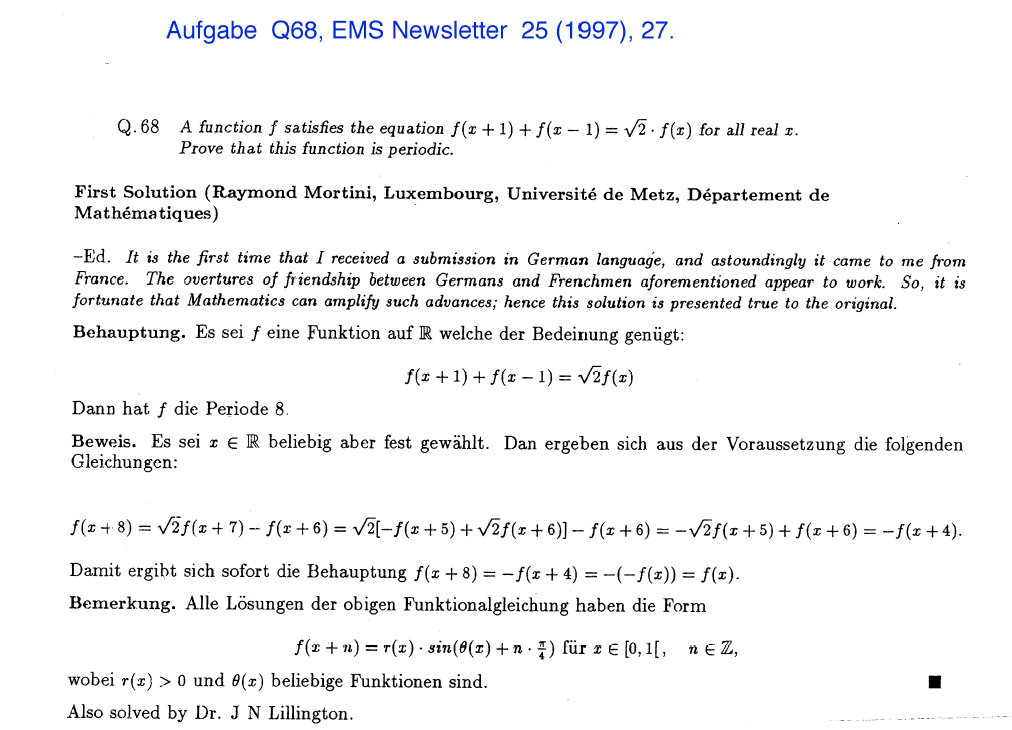}} 

\end{figure}

\newpage

{\huge\gr{\section{Math. Gazette}}}

\bigskip


\begin{figure}[h!]

\scalebox{0.45} 
  {\includegraphics{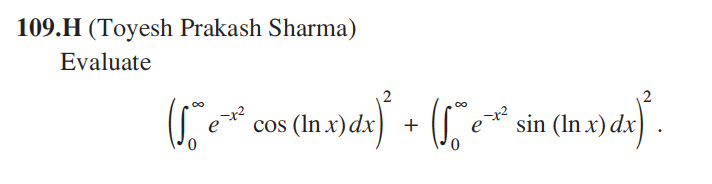}} 

\end{figure}


\centerline{\bf Solution to problem 109.H, Math. Gazette 109, Issue 575 (2024), p. 354}\medskip

\centerline{by Raymond Mortini and Rudolf Rupp}

\medskip

\centerline{- - - - - - - - - - - - - - - - - - - - - - - - - - - - - - - - - - - - - - - - - - - - - - - - - - - - - -}
  
  \medskip

\newpage

\begin{figure}[h!]

\scalebox{0.45} 
  {\includegraphics{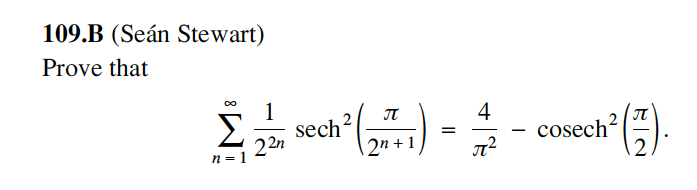}} 

\end{figure}


\centerline{\bf Solution to problem 109.B, Math. Gazette 109, Issue 574 (2024), p. 169}\medskip

\centerline{by Raymond Mortini and Rudolf Rupp}

\medskip

\centerline{- - - - - - - - - - - - - - - - - - - - - - - - - - - - - - - - - - - - - - - - - - - - - - - - - - - - - -}
  
  \medskip

\newpage

\begin{figure}[h!]

\scalebox{0.35} 
  {\includegraphics{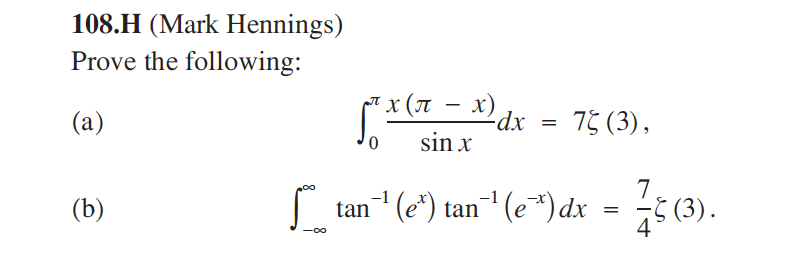}} 

\end{figure}


\centerline{\bf Solution to problem 108.H, Math. Gazette 108, Issue 572 (2024), p. 364}\medskip

\centerline{by Raymond Mortini and Rudolf Rupp}

\medskip

\centerline{- - - - - - - - - - - - - - - - - - - - - - - - - - - - - - - - - - - - - - - - - - - - - - - - - - - - - -}
  
  \medskip

(a)  We first consider the Fourier series of the  function
$$g(x)=\begin{cases}
x(\pi-x)&\text{if $0\leq x\leq \pi$}\\
x(\pi+x)&\text{if $-\pi\leq x\leq 0$}
\end{cases}
$$
(the odd extension of $x(\pi-x)$), and extend it $2\pi$-periodically.
Then 
$$g(x)=\frac{8}{\pi}\sum_{n=1\atop n\;{\rm odd} }^\infty \frac{\sin(nx)}{n^3}.$$
Hence
\begin{eqnarray*}
I_1:=\int_0^\pi \frac{x(\pi-x)}{\sin x}dx&=&\frac{8}{\pi} \sum_{n=1\atop n\;{\rm odd} }^\infty \frac{1}{n^3}\;\int_0^{\pi} \frac{\sin (nx)}{\sin x}dx\\
&=&\frac{8}{\pi} \sum_{n=1\atop n\;{\rm odd} }^\infty \frac{1}{n^3} \pi =8 \frac{7}{8}\xi(3) =7\zeta(3).
\end{eqnarray*}
Note that  the integrand coincides with $U_{n-1}(\cos x)$, where $U_n$ is the Chebyshev polynomial of the second kind. 
To see that for odd $n$, 
$$J_n:=\int_0^{\pi} \frac{\sin (nx)}{\sin x}dx=\pi$$
 (a well known result),  it suffices to show that $J_{n}=J_{n+2}$. This holds, though,  since
\begin{eqnarray*}
J_{n+2} =\int_0^{\pi} \frac{\sin (nx) \cos 2x+\cos (nx) \sin 2x}{\sin x}dx&=&\int_0^{\pi} \left(\frac{\sin (nx)}{\sin x} (1-2\sin^2 x) +2\cos nx  \cos x \right)dx\\
&=& J_n + 2\int_0^\pi  \left(-\sin(nx) \sin x + \cos( nx ) \cos x \right)dx\\
&=&J_n +2\int_0^\pi \cos(n+1)x\,dx\\&=&J_n.
\end{eqnarray*}

\newpage

(b) We shall work with suitable variable substitutions to regain the integral in (a). We use $\arctan x$ instead of  the ambiguous notation $\tan^{-1}x$. Note that $\arctan(1/y)+\arctan y=\pi/2$ for $y>0$.

\begin{eqnarray*}
I_2:= \int_{-\infty}^\infty \arctan(e^x)\arctan(e^{-x}) dx&\buildrel=_{x=\log y}^{e^x=y}&\int_0^\infty \frac{\arctan y \;\arctan \frac{1}{y}}{y} dy\\
&\buildrel=_{\arctan y =s}^{y=\tan s}&\int_0^{\pi/2} \frac{s (\frac{\pi}{2}-s)}{\tan s}\; \frac{1}{\cos^2 s}ds\\
&=&2\int_0^{\pi/2} \frac{s(\frac{\pi}{2}-s)}{\sin (2s)} ds\\
&\buildrel=_{}^{2s=t}&\int_0^\pi \frac{\frac{t}{2}\left(\frac{\pi}{2}-\frac{t}{2}\right)}{\sin t} dt\\
&=&\frac{1}{4}  I_1.
\end{eqnarray*}

\newpage

\begin{figure}[h!]

\scalebox{0.35} 
  {\includegraphics{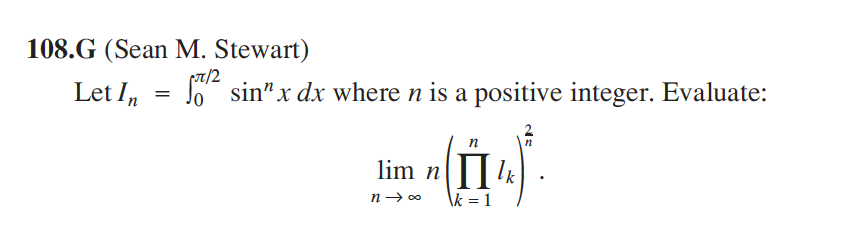}} 

\end{figure}

\centerline{\bf Solution to problem 108.G, Math. Gazette 108, Issue 572 (2024), p. 364}\medskip

\centerline{by Raymond Mortini and Rudolf Rupp}

\medskip

\centerline{- - - - - - - - - - - - - - - - - - - - - - - - - - - - - - - - - - - - - - - - - - - - - - - - - - - - - -}
  
  \medskip

We suppose that the $l_k$ in the integrand should be the $I_k$, that is
$$I_k:=\int_0^{\pi/2} (\sin x)^k dx.$$ 

So let
$$S_n:=n \left(\prod_{k=1}^n I_k\right)^{\hspace{-3pt}\frac{2}{n}}.$$
We claim that 
$$\ovalbox{$\dis  \lim_n S_n=\frac{\pi}{2}e$}.
$$
We shall use that 
$$I_k= \frac{\sqrt \pi}{2}\;\frac{\Gamma(\frac{k+1}{2})}{\Gamma(\frac{k+2}{2})}.$$

Now $S_n$ is a telescopic product; hence
$$S_n=n \left(\frac{\sqrt\pi}{2}\right)^2\;\frac{\Gamma(1)}{\left(\Gamma(\frac{n}{2}+1)\right)^{2/n}}.$$

Using (the non-discrete) Stirling formula, which tells us that
$$\lim_{x\to\infty}\frac{\Gamma(x) }{ \sqrt{\frac{2\pi}{x}} \left(\frac{x}{e}\right)^x}=1,$$
we obtain with $\Gamma\left(\frac{n}{2}+1\right)=\frac{n}{2}\,\Gamma\left(\frac{n}{2}\right)$ that
\begin{eqnarray*}
S_n&\sim&\frac{\pi}{4} \frac{n}{\left(\frac{n}{2}\; \sqrt{\frac{2\pi}{n/2}} \left(\frac{n/2}{e}\right)^{n/2}           \right)^{2/n}}
= \frac{\pi}{4}  \frac{n}{\left(\frac{n}{2}\right)^{2/n}\left(\frac{4\pi}{n}\right)^{1/n}\frac{n}{2e}}\to\frac{\pi}{2} e.
\end{eqnarray*}

\newpage

  \begin{figure}[h!]
 
  \scalebox{0.45} 
  {\includegraphics{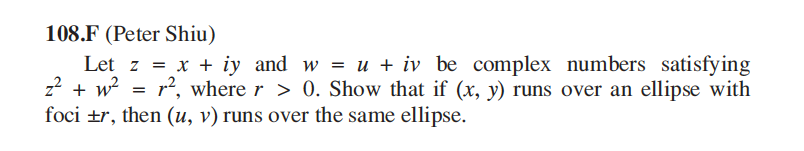} }
\end{figure}

\centerline{\bf Solution to problem 108.F, Math. Gazette 108, Issue 572 (2024), p. 364}\medskip

\centerline{by Raymond Mortini and Rudolf Rupp}

\medskip

\centerline{- - - - - - - - - - - - - - - - - - - - - - - - - - - - - - - - - - - - - - - - - - - - - - - - - - - - - -}
  
  \medskip

This is entirely trivial: Note that  an ellipse with focii $\pm r$, $r>0$,  in the $z$-plane is given by
$|z-r|+|z+r|=2c$ for some positive constant $c$. Hence, as $ |z^2-r^2|= |w^2|$, respectively   $|w^2-r^2|= |z^2|$, and  through squaring,

\begin{eqnarray*}
|z-r|^2+|z+r|^2 +2 |z-r|\;|z+r|&=& 2|z|^2 +2r^2+2 |z^2-r^2|\\
&=&2|z|^2 +2r^2+2|w|^2\\
\end{eqnarray*}
As the  latter is symmetric in $z$ and $w$, we obtain
$$(|z-r|+|z+r|)^2=(|w-r|+|w+r|)^2.$$
As the terms are positive, it follows that $|z-r|+|z+r|=|w-r|+|w+r|$.

\newpage

  \begin{figure}[h!]
 
  \scalebox{0.45} 
  {\includegraphics{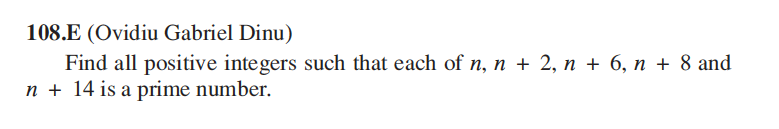} }
\end{figure}

\centerline{\bf Solution to problem 108.E, Math. Gazette 108, Issue 572 (2024), p. 364}\medskip

\centerline{by Raymond Mortini and Rudolf Rupp}

\medskip

\centerline{- - - - - - - - - - - - - - - - - - - - - - - - - - - - - - - - - - - - - - - - - - - - - - - - - - - - - -}
  
  \medskip

We claim that only $n=5$ yields the desired prime quintuplets.\\

Suppose that $n,n+2,n+6, n+8,n+14$ are prime numbers. Say $n=a_0+\sum_{k=1}^m a_k 10^k$, with $0\leq a_k\leq 9$.
Obviously $a_0\in \{1,3,7,9\}$ or $a_0=5$ and all $a_k=0$ for $k\geq 1$.

\underline{{\it Case 1}} ~ $a_0=1$. Then $n+14$ is not prime since $n+14=5+ (a_1+1)10+\sum_{k=1}^m a_k10^k$ is divisible by $5$
\footnote{Note that $a_1+1$ may be equal to $10$; that does not alter the divisibility property, though.
}.

\underline{{\it Case 2}} ~ $a_0=3$. Then $n+2$ is not prime since $n+2=5+\sum_{k=1}^m a_k10^k$ is divisible by $5$.

\underline{{\it Case 3}} ~ $a_0=7$. Then $n+8$ is not prime since $n+8=5+(a_1+1)10+\sum_{k=1}^m a_k10^k$ is divisible by $5$.

\underline{{\it Case 4}} ~ $a_0=9$. Then $n+6$ is not prime since $n+6=5+(a_1+1)10+\sum_{k=1}^m a_k10^k$ is divisible by $5$.

So it remains $a_0=5$: $(5,7,11,13,19)$, the only prime quintuple of this form.

\newpage

\nopagecolor

   \begin{figure}[h!]
 
  \scalebox{0.30} 
  {\includegraphics{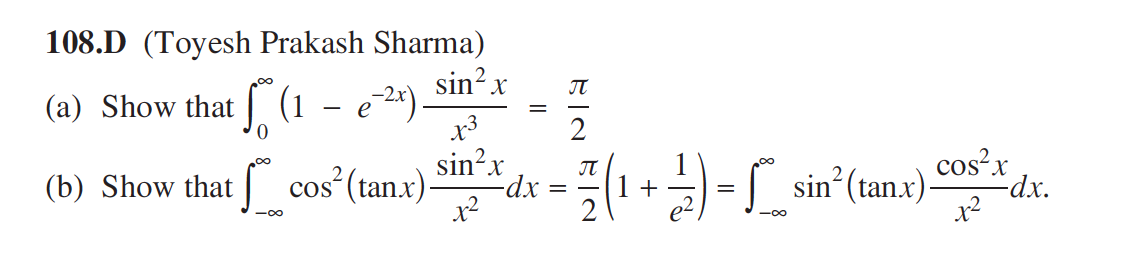} }
\end{figure}

\centerline{\bf Solution to problem 108.D, Math. Gazette 108, Issue 571 (2024), p. 167}\medskip

\centerline{Raymond Mortini, Rudolf Rupp}
\medskip

\centerline{- - - - - - - - - - - - - - - - - - - - - - - - - - - - - - - - - - - - - - - - - - - - - - - - - - - - - -}
  
  \medskip

We give for (a) two proofs. \\

{\bf (a1)} This is done via the Laplace transform $F(s):=\mathcal L[f](s):=\int_0^\infty f(t)e^{-st}dt$ for certain admissible functions $f$.
First note that $\dis\frac{1-e^{-2x}}{x}=\int_{s=0}^2  e^{-sx} ds.$ Then, due to the convergence of the integrals, and Fubini's theorem,
\begin{eqnarray*}
I:=\int_0^\infty(1-e^{-2x})\frac{\sin^2 x}{x^3}dx&=&\int_0^\infty \frac{1-e^{-2x}}{x}\; \frac{\sin^2 x}{x^2}dx=
\int_{x=0}^\infty \int_{s=0}^2 e^{-sx} \frac{\sin^2 x}{x^2}ds dx\\
&=&\int_{s=0}^2\left(\int_{x=0}^\infty e^{-sx}\;\frac{\sin^2 x}{x^2}\;dx\right)ds.
\end{eqnarray*}
Starting with $\mathcal L[\sin^2 t](s)=\mathcal L[\frac{1-\cos (2t)}{2}](s)=  \frac{1}{2s}- \frac{1}{2} \frac{s}{s^2+4}$
and using  twice the formula
$$\mathcal L [f(t)/t](s)= \int_s^\infty F(\sigma) d\sigma,$$
we see that
the Laplace transform of the function $S(x):=\frac{\sin^2 x}{x^2}$  is given by
$$\frac{s\log s}{2}-\frac{s\log(s^2+4)}{4}+\underbrace{\arctan (\frac{2}{s})}_{=\frac{\pi}{2}-\arctan \frac{s}{2}}.$$
Hence 
$$I= \left[ \frac{s\pi}{2} + \frac{s^2\log s}{4}- \frac{s^2+4}{8}\log(s^2+4)+ \frac{1}{2} -s\arctan\frac{s}{2} +
\log\left(\frac{s^2}{4}+1\right)\right]^2_0= \frac{\pi}{2}.$$
\medskip

{\bf(a2)}   Note that 
$$I=\frac{1}{2}\int_0^\infty (1-e^{-2x}) \frac{1-\cos(2x)}{x^3} dx.$$
So we need to apply the residue theorem to the function 
$$f(z):=\frac{(1-e^{-2z})(1-e^{2iz})}{z^3}
$$
and the contour

 \begin{figure}[h!]
 \hspace{-4cm}
  \scalebox{0.30} 
  {\includegraphics{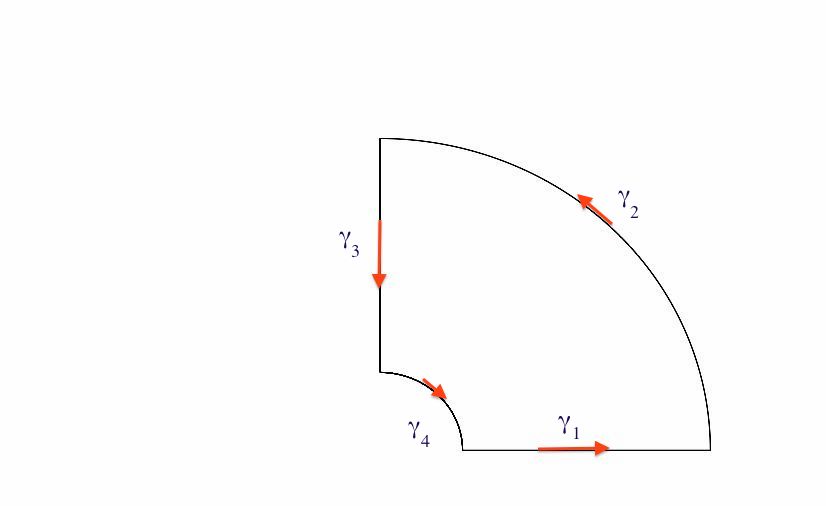} }
\end{figure}

$$\Gamma(r,R)=\gamma_1(r,R)\oplus \gamma_2(R)\oplus \gamma_3(r,R)\oplus \gamma_4(r,R),$$
where $\gamma_1(n)=[r, R]$, $\gamma_2(R)=Re^{it}, 0\leq t\leq \pi/2$, $\gamma^-_3(r)=it$, $r\leq t\leq R$, $\gamma^-_4(r,R)=re^{it}$,
$0\leq t\leq \pi/2$.

Observe that $\int_{\Gamma(r,R)} f(z)dz =0$.
Since
\begin{eqnarray*}
\int_{\gamma^-_3(r,R)} f(z) dz&=&\int_r^R f(it) i dt=\int_r^R\frac{(1-e^{-2it})(1-e^{-2t})}{(-i)t^3} idt\\
&=&-\overline{\int_{\gamma_1(r,R)} f(z)dz},
\end{eqnarray*}
we have
$$\int_{\gamma_1(r,R)} f(z)dz +\int_{\gamma_3(r,R)} f(z)dz= 2\;  \int_{\gamma_1(r,R)} \hspace{-5mm}{\rm Re} \;f(z)dz=
4\int_r^R (1-e^{-2x})\frac{\sin^2 x}{x^3}dx.
$$

Now $I(R):=\dis\int_{\gamma_2(R)} f(z)dz\to 0$ as $R\to\infty$. In fact,

\begin{eqnarray*}
|I(R)|&=&\left| \int_0^{\pi/2} f(Re^{it}) iR e^{it}dt\right|\leq \int_0^{\pi/2} \frac{(1+e^{-2R\cos t})(1+e^{-2R\sin t})}{R^2}dt\leq \frac{4 \;\dis\frac{\pi}{2}}{R^2}.
\end{eqnarray*}

Moreover, $J(r):=\dis\int_{\gamma_4(r)} f(z)dz\to  -2\pi$ as $r\to 0$.  In fact, since
the Taylor expansion of $zf(z)=\frac{(1-e^{-2z})(1-e^{2iz})}{z^2}$ at the origin equals
$$ -4i +(4+4i)z-4z^2+\Oh(z^3),$$
we have
\begin{eqnarray*}
J(r)&=& -\int_0^{\pi/2} \frac{(1-e^{-2 re^{it}})(1-e^{2ire^{it}})}{r^2e^{3it}} i e^{it} dt \;\buildrel\longrightarrow_{}^{r\to 0}\; -(-4i) i  \; \int_0^{\pi/2} 1\,dt =-2\pi
\end{eqnarray*}
(note that $\lim_{r\to 0} \int=\int\lim_{r\to 0}$ since the integrand is uniformly continuous for $(r,t)\in \;]0,1]\times [0,2\pi]$.)
Now
\begin{eqnarray*}
0&=&\lim_{R\to\infty\atop r\to 0} \int_{\Gamma(r,R)} f(z) dz\\
&=& \lim_{R\to\infty\atop r\to 0} \left(\int_{\gamma_1(r,R)} f(z)dz+ \int_{\gamma_3(r,R)} f(z)dz\right) +  \lim_{R\to\infty} \int_{\gamma_2(R)}  f(z)dz+  \lim_{ r\to 0} \int_{\gamma_4(r)} f(z)dz\\
&=& 4 I + 0- 2\pi.
\end{eqnarray*}
Consequently, $I=\pi/2 $.\\

%

(b) We first observe that  due to the majorant $1/x^2$ for $|x|\geq 1$,  the integral  exists as Lebesgue integal as well as improper Riemann integral (note the discontinuity points at 
$-\pi/2+k\pi, k\in \Z$).

Since $f(x):=\cos^2(\tan x)$ is $\pi$-periodic and even,  we may use the Lobashevski integral formula 
$$\int_0^\infty f(x)\frac{\sin^2 x}{x^2}dx=\int_0^{\pi/2} f(x) dx$$
(see e.g. \cite{jo} \footnote{ The proof there is also valid for the case of all even  Riemann-integrable $\pi$-periodic functions.}) to conclude that
\begin{eqnarray*}
\int_{{-\infty}}^\infty \cos^2 (\tan x)\frac{\sin^2 x}{x^2}dx&=&2\int_0^{\pi/2} \cos^2 (\tan x)dx\\
&\buildrel=_{dx=\frac{ds}{1+s^2}}^{s:=\tan x}&2\int_{s=0}^\infty \frac{\cos^2 s}{1+s^2}ds= \int_{s=0}^\infty \frac{1+\cos(2s)}{1+s^2}ds\\
&=& \arctan s \Big|^\infty_0 + \frac{1}{2}\int_{{-\infty}}^\infty \frac{\cos(2s)}{1+s^2}ds\\
&=&\frac{\pi}{2}+\pi i {\rm Res}\;\left(\frac{e^{2iz}}{1+z^2}; z=i\right)\\
&=& \frac{\pi}{2}+\pi i \frac{e^{-2}}{2i}=\frac{\pi}{2}\left(1+ \frac{1}{e^2}\right).
\end{eqnarray*}

Here we have applied the residue theorem to evaluate the classical integral $\int_{{-\infty}}^\infty \frac{\cos(2s)}{1+s^2}ds$,
which used to be an exercise in many complex analysis courses (see e.g. \cite[p. 210]{ch}).\\

Now we have the following identities:

$$
\int_{-\infty}^\infty \sin^2 (\tan x)\frac{\cos^2 x}{x^2}dx -\int_{{-\infty}}^\infty \cos^2 (\tan x)\frac{\sin^2 x}{x^2}dx$$
$$=
\int_{-\infty}^\infty\left( \sin^2 (\tan x)\frac{\cos^2 x}{x^2}+\sin^2 (\tan x)\frac{\sin^2 x}{x^2} \right)dx-
\int_{-\infty}^\infty\left(\sin^2 (\tan x)\frac{\sin^2 x}{x^2} +\cos^2 (\tan x)\frac{\sin^2 x}{x^2}\right)dx
$$
$$=\int_{-\infty}^\infty\frac{\sin^2 (\tan x)}{x^2} dx - \int_{-\infty}^\infty\frac{\sin^2 x}{x^2} dx.
$$
The latter difference, though, vanishes. In fact, since $\sum\int=\int\sum$ (note that  all terms are positive),we obtain in view of the classical formula (see e.g. \cite{jo}),
$$\frac{1}{\sin^2 x}=\sum_{k=-\infty}^\infty \frac{1}{(k\pi +x)^2}$$
that
\begin{eqnarray*}
\int_{{-\infty}}^\infty\frac{\sin^2 (\tan x)}{x^2} dx&=&\sum_{k=-\infty}^\infty \int_{-\frac{\pi}{2}+k\pi}^{\frac{\pi}{2} +k\pi} \frac{\sin^2 (\tan x)}{x^2} dx\\
&\buildrel=_{}^{x=-\frac{\pi}{2}+k\pi +u}&\sum_{k=-\infty}^\infty \int_0^\pi \frac{\sin^2(\tan(u-\frac{\pi}{2}))}{(-\frac{\pi}{2}+k\pi +u)^2} du\\
&=&\int_0^\pi \sin^2(-\cot u)\sum_{k=-\infty}^\infty \frac{1}{(k\pi +(u-\frac{\pi}{2}))^2}\;du\\
&=&\int_0^\pi \sin^2(-\cot u) \frac{1}{\sin^2(u-\frac{\pi}{2})}du=\int_0^\pi \frac{\sin^2(-\cot u) }{\cos^2 u} \;du\\
&\buildrel=_{du=\sin^2 u ds}^{-\cot u=s}&\int_{-\infty}^\infty \frac{\sin^2 s}{\cos^2 u}\; \sin^2 u \;ds=\int_{-\infty}^\infty \frac{\sin^2 s}{s^2}\;ds.
\end{eqnarray*}
Amazing! Hence the second formula 

$$\int_{{-\infty}}^\infty\frac{\sin^2 (\tan x) \;\cos^2 x}{x^2}dx=\frac{\pi}{2}\left(1+ \frac{1}{e^2}\right)$$
holds, too.

\newpage

\begin{figure}[h!]

\scalebox{0.4} 
  {\includegraphics{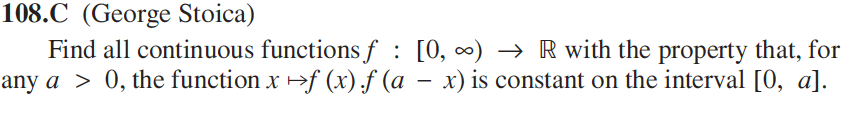}} 

\end{figure}

\centerline{\bf Solution to problem 108.C, Math. Gazette 108, Issue 571 (2024), p. 167}\medskip

\centerline{by Raymond Mortini}

\medskip

\centerline{- - - - - - - - - - - - - - - - - - - - - - - - - - - - - - - - - - - - - - - - - - - - - - - - - - - - - -}
  
  \medskip


We claim that all solutions $f$ continuous on $[0,\infty[$ are given by 
$$\ovalbox{$\lambda\, e^{\beta x}$, where $\lambda,\beta\in \R$ },$$
or equivalently
$$\ovalbox{$f\equiv 0$ or $\pm  e^{\alpha+\beta x}$, where $\alpha,\beta\in \R$. }$$

\bigskip
To verify this, we first note that, trivially, all these functions are solutions, since for $0\leq a\leq x$
$$\lambda e^{\beta x}\cdot \lambda e^{\beta (a-x)}= \lambda^2 e^{\beta a},$$

Conversely, let $f$ be  a solution to the problem,  that is,  $f(x)\cdot f(a-x)=:c(a)$ is independent of  $x$ whenever $0\leq x\leq a$, and this  for any $a>0$.\\

$\bullet$ Let $x=0$. Then $f(0) f(a)=c(a)$. Now let $x=a/2$. Then $f(a/2) f(a/2)=c(a)$, too. Hence, for every $a>0$, 
 
\begin{equation}\label{form1}
f(a/2)^2= f(0) f(a).
\end{equation}
Thus $f$ has everywhere the sign of $f(0)$ or $f\equiv 0$ if $f(0)=0$.
Since $f$ is a solution if and only $\lambda f$ is a solution, we may assume wlog that $f(0)=1$. Hence, by (\ref{form1}), $f>0$ 
on $[0,\infty[$.\\

$\bullet$ Let $F(x):=\log f(x)$. Then, for any $a>0$,  $F$ satisfies on $[0,a]$ the functional equation

\begin{equation}\label{form2}
F(x)+F(a-x)=C(a)
\end{equation}
for some (continuous) function $C(a)$.

Now, by $(\ref{form1})$, $C(a)=2F(a/2)=F(0)+ F(a)=F(a)$. In particular, for $b=a/2$, 

\begin{equation}\label{form3}
2F(b)=F(2b)
\end{equation}

 Moreover, if we take $x=a/3$, 
$$F(a/3)+F((2/3) a)=C(a)=F(a).$$
Thus, for $b=a/3$,

$$0=F(b)+ F(2b)-F(3b)=F(b)+2F(b)-F(3b),$$

and so
\begin{equation}\label{form4}
3 F(b)=F(3b).
\end{equation}
We conclude that  for every $x>0$, and $n,k\in \N:=\{0,1,2,\dots\}$,
$$F\left(\frac{2^n}{3^k} x\right)=\frac{2^n}{3^k} F(x),$$
since, by induction, $F(2^nx)=2^nF(x)$ and $F(x/3^k)= \frac{1}{3^k} F(x)$. Now replace $x$ by $2^nx$.
\\

Since the set  $\{\frac{2^n}{3^k}:n,k\in \N\}$ is dense in $[0,\infty[$ (see e.g. \cite[p. 1879]{moru}), 
continuity of $F$ in $[0,\infty[$  yields that $F(\mu x)=\mu F(x)$ for every  $\mu>0$ and $x>0$. Therefore, for $x=1$, $F(\mu)= \mu F(1)$. Consequently,
with $\beta:= F(1)$,
$$f(\mu)= e^{\beta \mu}.$$

\newpage

The following list is ordered according to the "chronological" appearance in this text.

\end{document}